\documentclass[11pt, openany]{book}
\usepackage{amsmath, amsfonts,amsthm,amssymb,amscd, verbatim, graphicx,color,multirow,booktabs, caption,tikz,tikz-cd, mathdots, todonotes}
\usetikzlibrary{matrix, calc, arrows, graphs}
\usepackage{colortbl, xcolor}% http://ctan.org/pkg/xcolor
\usepackage{chngcntr, arydshln, soul, float,subcaption}
\counterwithin{table}{section}
\def\classification#1{\def\@class{#1}}
\classification{\null}
\usepackage[margin=0.75in]{geometry}
\usepackage[normalem]{ulem}

\usepackage{mathtools}

\usepackage{hyperref}

\newenvironment{dedication}
  {%\clearpage           % we want a new page          %% I commented this
   \thispagestyle{empty}% no header and footer
   \vspace*{\stretch{1}}% some space at the top
   \itshape             % the text is in italics
   \raggedleft          % flush to the right margin
  }
  {\par % end the paragraph
   \vspace{\stretch{3}} % space at bottom is three times that at the top
   \clearpage           % finish off the page
  }

\newcommand{\Alt}{\mathop{\mathrm{Alt}}}
\newcommand{\Sym}{\mathop{\mathrm{Sym}}}
\newcommand{\Z}{\mathbb{Z}}
\newcommand{\floor}[1]{\lfloor#1\rfloor}
\newcommand{\Aut}{\mathop{\mathrm{Aut}}}
\newcommand{\ASL}{\mathop{\mathrm{ASL}}}
\renewcommand{\wr}{\mathop{\mathrm{wr}}}
\def\nor#1#2{{N}_{{#1}}{{(#2)}}}
\def\cent#1#2{{C}_{{#1}}{{(#2)}}}
\def\Zent#1{{Z}{{(#1)}}}
\DeclareFontFamily{OT1}{rsfs}{}
\DeclareFontShape{OT1}{rsfs}{n}{it}{<-> rsfs10}{}
\DeclareMathAlphabet{\mathscr}{OT1}{rsfs}{n}{it}
 
 \newcommand{\e}{\epsilon}
\renewcommand{\a}{\alpha}
\renewcommand{\b}{\beta}
\def\d{\delta}
\def\g{\gamma}

\def\D{\Delta}
\def\O{\Omega}
\def\s{\sigma}
\def\l{\lambda}
\def\la{\langle}
\def\ra{\rangle}

\def\no{\noindent}
 
\renewcommand{\l}{\lambda} \renewcommand{\O}{\Omega}

\newcommand{\fix}{{\rm fix}}
\newcommand{\srel}{{\rm SRC}}
\newcommand{\trel}{{\rm TRC}}
\newcommand{\rrel}{{\rm RC}}
\newcommand{\relR}{\mathcal{R}}
\newcommand{\widesim}{
  \mathrel{{\scalebox{1.5}[1]{$\sim$}}}
}

\newcommand{\reducesize}[2]{%
  \mathbin{% This will be a binary math symbol (in terms of spacing around it)
    \ooalign{% Overlay a number of symbols
      \raisebox% Adjust vertical positioning of <object>
       {.4ex}% Move it down relative to current font     
          {$#1\widesim$}% First symbol (\sim in correct math style)
      \cr % Move to next symbol
      \hidewidth% Move symbol to right (~\hfill)
      \raisebox% Adjust vertical positioning of <object>
        {-.6ex}% Move it down relative to current font
        {\scalebox% Change the "font size"
          {.75}% to 50% of current font size
          {$#1#2$}% <object> in current math style
        }% \raisebox
      \hidewidth% Move symbol to left (~\hfill)
    }% \ooalign
  }% \mathbin
}% \reducesize

\newcommand{\nreducesize}[2]{%
  \mathbin{% This will be a binary math symbol (in terms of spacing around it)
    \ooalign{% Overlay a number of symbols
      \raisebox% Adjust vertical positioning of <object>
       {.4ex}% Move it down relative to current font     
          {$#1\not\widesim$}% First symbol (\sim in correct math style)
      \cr % Move to next symbol
      \hidewidth% Move symbol to right (~\hfill)
      \raisebox% Adjust vertical positioning of <object>
        {-.6ex}% Move it down relative to current font
        {\scalebox% Change the "font size"
          {.75}% to 50% of current font size
          {$#1#2$}% <object> in current math style
        }% \raisebox
      \hidewidth% Move symbol to left (~\hfill)
    }% \ooalign
  }% \mathbin
}% \reducesize

\newcommand{\stb}[1]{\mathpalette\reducesize{#1}}
\newcommand{\nstb}[1]{\mathpalette\nreducesize{#1}}
\newcommand{\stc}[2]{{\genfrac{}{}{-2pt}{}{\widesim}{#1,#2}}}
\newcommand{\magma}{{\tt magma}\ }

\newcommand{\bbu}{\mathbf{u}}
\newcommand{\F}{\mathbb{F}}
\newcommand{\Fq}{\mathbb{F}_q}
\newcommand{\Fqz}{\mathbb{F}_{q_0}}
\newcommand{\K}{\mathbb{K}}
\newcommand{\Fqt}{\mathbb{F}_{q^2}}
\newcommand{\bF}{\mathbb{F}}

\newcommand{\C}{\mathcal{C}}
\newcommand{\B}{\mathcal{B}}

\newcommand{\stab}{\mathrm{Stab}}
\newcommand{\SL}{\mathrm{SL}}
\newcommand{\SU}{\mathrm{SU}}
\newcommand{\Cl}{\mathrm{Cl}}
\newcommand{\PGU}{\mathrm{PGU}}
\newcommand{\SOr}{\mathrm{SO}}
\newcommand{\AGL}{\mathrm{AGL}}
\newcommand{\Or}{\mathrm{O}}

\newcommand{\Sp}{\mathrm{Sp}}
\newcommand{\POmega}{\mathrm{P\Omega}}
\newcommand{\PSL}{\mathrm{PSL}}
\newcommand{\PSU}{\mathrm{PSU}}
\newcommand{\PSp}{\mathrm{PSp}}
\newcommand{\PSO}{\mathrm{PSO}}
\newcommand{\PO}{\mathrm{PO}}
\newcommand{\PGL}{\mathrm{PGL}}
\newcommand{\PGammaL}{\mathrm{P\Gamma L}}
\newcommand{\GammaL}{\mathrm{\Gamma L}}
\newcommand{\GammaSp}{\mathrm{\Gamma Sp}}
\newcommand{\GL}{\mathrm{GL}}
\newcommand{\GU}{\mathrm{GU}}
\newcommand{\GSp}{\mathrm{GSp}}
\newcommand{\symme}{\mathrm{Sym}}
\newcommand{\alter}{\mathrm{Alt}}
\newcommand{\spann}{\mathrm{span}}

\newcommand{\Fix}{\mathrm{Fix}}

\newcommand{\RC}{\mathrm{RC}}
\newcommand{\base}{\mathrm{b}}
\newcommand{\Base}{\mathrm{B}}
\newcommand{\Height}{\mathrm{H}}
\newcommand{\Irred}{\mathrm{I}}

\newtheorem{prop}{Proposition}[section]
\newtheorem{theo}{Theorem}[chapter]
\newtheorem{thm}[prop]{Theorem}

\newtheorem{conj}[theo]{Conjecture}
\newtheorem{example}[prop]{Example}

\newtheorem{cor}[prop]{Corollary}
\newtheorem{coroll}[theo]{Corollary}
\newtheorem{lem}[prop]{Lemma}
\newtheorem{defn}[prop]{Definition}
\theoremstyle{definition}

\numberwithin{equation}{section}

\begin{document}

\title{Cherlin's conjecture on finite primitive binary permutation groups}

\author{Nick Gill \\Department of Mathematics, \\ University of South Wales,  \\ Treforest CF37 1DL, U.K.\\
{\tt nick.gill@southwales.ac.uk}
\and
Martin W. Liebeck \\Department of Mathematics, \\Imperial College London, \\London SW7 2AZ, UK \\
{\tt m.liebeck@imperial.ac.uk }
\and 
Pablo Spiga \\
Dipartimento di Matematica e Applicazioni, \\University of Milano-Bicocca, \\Via Cozzi 55, \\ 20125 Milano, Italy \\
{\tt pablo.spiga@unimib.it}
}

\maketitle

\begin{dedication}
In memory of Jan Saxl.
\end{dedication}

\tableofcontents

\chapter{Introduction}

In this monograph, we are concerned with the problem of classifying the finite primitive binary permutation groups. Let $G$ be a permutation group on the set $\Omega$. Given a positive integer $n$, given $I:=(\omega_1,\omega_2,\ldots,\omega_n)$ in the Cartesian product $\Omega^n$ and given $g\in G$, we write
$$I^g:=(\omega_1^g,\omega_2^g,\ldots,\omega_n^g).$$Moreover, for every $1\le i<j\le n$, we let $I_{ij}:=(\omega_i,\omega_j)$ be the $2$-subtuple of $I$ corresponding to the $i^{\mathrm{th}}$ and to the $j^{\mathrm{th}}$ coordinate. Now, the permutation group $G$ on  $\Omega$ is called {\it binary} if, for all positive integers $n$, and for all 
$I$ and 
$J$ in $\Omega^n$, there exists $g\in G$ such that $I^g=J$
 if and only if for all $2$-subtuples, $I_{ij}$, of $I$, there exists an element $g_{ij}$ such that $I_{ij}^{g_{ij}} = J_{ij}$.

Cherlin has proposed a conjecture listing the finite primitive binary permutation groups \cite{cherlin1}. The conjecture is as follows, and our task is to complete the proof of this conjecture.

\begin{conj}\label{conj: cherlin original}
A finite primitive binary permutation group must be one of the following:
\begin{enumerate}
 \item a symmetric group $\Sym(n)$ acting naturally on $n$ elements;
\item  a cyclic group of prime order acting regularly on itself;
\item an affine orthogonal group $V\rtimes \Or(V)$ with $V$ a vector space over a finite field equipped with a non-degenerate anisotropic quadratic form, acting on itself by translation, with complement
the full orthogonal group $\Or(V)$.
\end{enumerate}
\end{conj}

The terminology of Conjecture~\ref{conj: cherlin original} is fully explained in subsequent sections. In particular, we give two equivalent definitions of the adjective ``binary'' in \S\ref{s: rc}, and all three families listed in Conjecture~\ref{conj: cherlin original} are fully discussed in \S\ref{s: examples}.

The O'Nan--Scott--Aschbacher theorem describes the structure of finite primitive permutation groups: there are five families of these. Thus, to prove Conjecture~\ref{conj: cherlin original}, it is sufficient to prove it for each of these families.

Cherlin himself gave a proof of the conjecture for the family of affine permutation groups, i.e.\ when $G$ has an abelian socle \cite{cherlin2}.  Wiscons then studied the remaining cases and showed that Conjecture~\ref{conj: cherlin original} reduces to the following statement concerning almost simple groups \cite{wiscons}.

\begin{conj}\label{conj: cherlin}
 If $G$ is a finite binary almost simple primitive group on $\Omega$, then $G=\symme(\Omega)$. 
\end{conj}

We recall that an \emph{almost simple group} $G$ is a finite group that has a unique minimal normal subgroup $S$ and, moreover, the group $S$ is non-abelian and simple. Note that $S$ is the socle of $G$.

We now invoke the Classification of Finite Simple Groups which says that a non-abelian simple group is either an alternating group, $\Alt(n)$ with $n\geq 5$; a simple group of Lie type; or one of 26 sporadic groups.

In \cite{gs_binary}, Conjecture~\ref{conj: cherlin} was proved for groups with socle a simple alternating group; in \cite{dgs_binary}, Conjecture~\ref{conj: cherlin} was proved for groups with socle a sporadic simple group. In this monograph we deal with the remaining family.

\begin{theo}\label{t: main}
Let $G$ be an almost simple group with socle a finite group of Lie type and assume that $G$ has a primitive and binary action on a set $\Omega$. Then $|\Omega|\in \{5,6,8\}$ and $G\cong \Sym(\Omega)$. 
\end{theo}
The examples in Theorem~\ref{t: main} arise via the isomorphisms
\begin{enumerate}
 \item $G\cong \SL_2(4).2 \cong \PGL_2(5) \cong \symme(5)$ and $|\Omega|=5$;
 \item $G\cong \Sp_4(2) \cong\PSL_2(9).2 \cong \symme(6)$ and $|\Omega|=6$;
 \item $G\cong \SL_4(2).2\cong \symme(8)$ and $|\Omega|=8$.
\end{enumerate}
Note that, here, we have not tried to list all isomorphisms between classical groups and the symmetric groups listed in Theorem~\ref{t: main}. The listed isomorphisms are the ones that crop up in the proof that follows; there are many further isomorphisms with classical groups not listed in the theorem (for example $\SOr_4^-(2) \cong\Gamma\Or_3(4) \cong \symme(5)$).

A special case of Theorem~\ref{t: main} has already appeared in the literature; in \cite{dgs_binary}, the theorem is proved for the case where $G$ is almost simple with socle a finite group of Lie type of rank 1.

Theorem~\ref{t: main} is the final piece in the jigsaw. We can now assert that Cherlin's conjecture is true:\footnote{Wiscons informed us of a small gap in his proof of \cite[Proposition~4.1]{wiscons}. The next paragraph consists of his comments on this, including a patch. For notation and terminology, we refer to the rest of this chapter.

\cite[Proposition~4.1]{wiscons} is devoted to showing that primitive groups of diagonal type are not binary. The gap in the proof stems from an implicit (and accidental) assumption in the first sentence of the proof of \cite[Lemma~4.2]{wiscons} that the socle is a product of at least \emph{three} isomorphic nonabelian simple groups. This leaves open the case of two factors, for which it suffices to consider the following setting: let $G$ be a group acting on a nonabelian group $T$ in such a way that the stabilizer of $1\in T$ satisfies $\operatorname{Inn}(T)\le G_1\le \operatorname{Aut}(T)\times\langle i\rangle$ for $i:T\rightarrow T$  the inversion map. In this context, we show the action of $G$ on $T$ is not binary. To see this, choose noncommuting $a,b\in T$, not both of order $2$, and observe that $(1,a,b,ab)$ and $(1,a,b,ba)$ are $2$-subtuple complete (witnessed by conjugating by $1$, $a$, or $b^{-1}$). However, since one of $a$ or $b$ is not of order $2$, $G_{1,a,b} \le \operatorname{Aut}(T)$, so $G_{1,a,b}$ fixes $ab \neq ba$. Thus, $(1,a,b,ab)$ and $(1,a,b,ba)$ are not $4$-subtuple complete, so such an action is not binary.}

\begin{coroll}\label{c: main}
Conjecture~$\ref{conj: cherlin original}$ is true.
\end{coroll}

As will become clear, once the various equivalent definitions of the word ``binary'' have been introduced, a proof of Conjecture~\ref{conj: cherlin original} is equivalent to a classification of the finite primitive binary relational structures. In particular we have the following (the definition of homogeneous relational structure can be found in Definitions~\ref{d: relational structure} and~\ref{d: homogeneous}):

\begin{coroll}\label{c: pbrs}
  Let $\mathcal{R}$ be a homogeneous binary relational structure with vertex set $\Omega$, such that $G=\Aut(\mathcal{R})$ acts primitively on $\Omega$. Then the action of $G$ on $\Omega$ is one of the actions listed in Conjecture~$\ref{conj: cherlin original}$.
\end{coroll}

We have not completely described the relational structure $\mathcal{R}$ in our statement of Corollary~\ref{c: pbrs} -- to do this, we would need to specify the relations in $\mathcal{R}$. We will not do this here, but we can at least start the task, making use of the fact that all relations of $\mathcal{R}$ must be unions of orbits of $G$ on $\Omega^2$. 

Consider the first family listed in Conjecture~\ref{conj: cherlin original}, where $G=\Sym(\Omega)$. In this case $G$ has two orbits on $\Omega^2$: the set $D$, of distinct pairs, and the set $R$, of repeated pairs. Thus the binary relational structures with all relations some union of $D$ and $R$ are:
\[
 (\Omega),\, (\Omega, D),\, (\Omega, R),\, (\Omega, D, R),\, (\Omega, R, D) \textrm{ and }(\Omega, D\cup R).
\]
One can check directly that every one of these is homogeneous and has automorphism group isomorphic to $\Sym(\Omega)$. One needs to repeat this analysis for the other two families; in these cases enumerating orbits and ascertaining which of the resulting structures are homogeneous is much more difficult.

For the remainder of this chapter we have three basic aims: first we seek to give the basic theory of relational complexity for permutation groups including, in particular, the definition of a binary action, and of a binary permutation group. We will also describe some of the key examples.

Second, we will give some motivation for interest in our result -- thus we will survey some related results in the study of relational structures, and in group theory. We will also briefly discuss Cherlin's original motivation for studying binary permutation groups, which arises from model theoretic considerations. 

In neither of these first two aspects do we make any claim for originality -- instead we seek to draw the key definitions and examples together into one place. Much of the material of this kind that we present below was worked out by Cherlin in his papers \cite{cherlin1, cherlin2, cherlin_martin}. 

Our third aim in this chapter is to present some of the results and methods concerning binary permutation groups that we consider to be most essential. These will be used in subsequent chapters when we commence our proof of Theorem~\ref{t: main}. 

The remainder of this monograph is occupied with a proof of Theorem~\ref{t: main}. In Chapter~\ref{ch: prelim} we give a number of general background results concerning groups of Lie type; in Chapter~\ref{ch: exceptional} we prove the theorem for the exceptional groups of Lie type; in Chapter~\ref{ch: classical} we prove the theorem for the classical groups of Lie type.

\subsection*{Acknowledgements}

All three authors were supported in this work by the Engineering and Physical Sciences Research Council grant number EP/R028702/1.

All three of us wish to express our thanks to Gregory Cherlin for his help and encouragement of our work, and for his creation of the beautiful mathematics that inspired our research in the first place. Thanks are also due to Joshua Wiscons for a number of helpful discussions.

NG and PS would like to thank their PhD students, Scott Hudson and Bianca Lod\'a; their research into the relational complexity of permutation groups has shed a great deal of light.

\section{Basics: The definition of relational complexity}\label{s: rc}

The notion of relational complexity can be defined in two different ways. Our job in this section is to present these definitions, and to show that they are equivalent. Throughout this section $G$ is a permutation group on a set $\Omega$ of size $t<\infty$. Note that when we write ``permutation group'' we are assuming that the associated action of $G$ on $\Omega$ is faithful -- in other words we can think of $G$ as a subgroup of $\Sym(\Omega)$.

\subsection{Relational structures}

The first approach towards relational complexity is via the concept of a relational structure \cite{cherlin2}. Recall that, for a positive integer $\ell$, $\Omega^\ell$ denotes the set of $\ell$-tuples with entries in $\Omega$.

\begin{defn}\label{d: relational structure}{\rm
 A \emph{relational structure} $\mathcal{R}$ is a tuple $(\Omega, R_1,\dots, R_k)$, where $\Omega$ is a set, $k$ is a non-negative integer and, for each $i\in\{1,\dots, k\}$, there exists an integer $\ell_i\geq 2$ such that $R_i\subseteq \Omega^{\ell_i}$.

The set $\Omega$ is called the \emph{vertex set} of the structure, while the sets $R_1,\dots, R_k$ are referred to as \emph{relations}; in addition, for each $i$, the integer $\ell_i$ is the \emph{arity} of relation $R_i$. We say that the relational structure $\mathcal{R}$ \emph{is of arity $\ell$}, where $\ell=\max\{\ell_1,\dots, \ell_k\}$. }
\end{defn}

\begin{example}{\rm 
 If a relation, or a relational structure is of arity $2$ (resp. $3$), then it is commonly called \emph{binary} (resp. \emph{ternary}). Binary relational structures which contain a single relation are nothing more nor less than directed graphs: if $\mathcal{R}=(\Omega, R_1)$ is one such, then the elements of the vertex set $\Omega$ are of course the vertices, and each pair in $R_1$ can be thought of as a directed edge between two elements of $\Omega$. (Note that by ``graph'' here we implicitly mean a graph with no multiple edges.)
 
When considering a binary relational structure with more than one relation, it is sometimes helpful to think of it as a directed graph in which there are several different ``edge colours'' -- each relation corresponding to a different ``colour''. 
} 
\end{example}

The notions of isomorphism and automorphism are generalizations of the corresponding definitions for graphs.

\begin{defn}\label{d: rs aut}{\rm
 Let $\mathcal{R}=(\Omega, R_1,\dots, R_k)$ and $\mathcal{S}=(\Lambda, S_1,\dots, S_k)$ be relational structures. An isomorphism $h:\mathcal{R}\to \mathcal{S}$ is a bijection $h: \Omega \to \Lambda$ such that
  \[
  (\omega_1,\dots, \omega_{\ell_i}) \in R_i \Longleftrightarrow (\omega_1^h,\dots, \omega_{\ell_i}^h) \in S_i.
 \]
An \emph{automorphism} $g$ of $\mathcal{R}$ is an element of $\Sym(\Omega)$ that is also an isomorphism $g:\mathcal{R}\to \mathcal{R}$. It is clear that the set of all automorphisms of $\mathcal{R}$ forms a group under composition of bijections; we denote this group by $\Aut(\mathcal{R})$, and note that it is a subgroup of $\Sym(\Omega)$.
}
\end{defn}

Note that we have only defined isomorphisms between relational structures that have the same number of relations; the definition also implies that the (ordered) list of relation-arities must be the same for isomorphic relational structures.\footnote{One can imagine a slight weakening of Definition~\ref{d: rs aut} where one allows an automorphism of $\mathcal{R}$ to map a set of tuples corresponding to one relation to the set of tuples corresponding to a different relation -- for certain relational structures, this would yield a larger automorphism group (which would contain $\Aut(\mathcal{R})$ as defined above, as a normal subgroup). We will not need this extension in what follows.}

Our focus will be on those relational structures that exhibit the maximum possible level of symmetry -- this requires the notion of \emph{homogeneity}. To state this definition we must first explain what is meant by ``an induced substructure'' -- once again this notion is a direct analogue of the same idea for graphs.

\begin{defn}{\rm 
 Let $\mathcal{R}=(\Omega, R_1,\dots, R_k)$ be a relational structure, with $R_i$ a relation of arity $\ell_i$ for each $i=1,\dots, k$. Let $\Gamma$ be a subset of $\Omega$. The \emph{induced substructure on $\Gamma$} is the relational structure $\mathcal{R}_\Gamma = (\Gamma, R_1', \dots, R_k')$ where $R_i' = \Gamma^{\ell_i} \cap R_i$.}
\end{defn}

So, to clarify what we said above: if $\mathcal{R}=(\Omega, R_1)$ is a binary structure with a single relation (i.e.\ a directed graph), and $\Gamma$ is a subset of the vertex set $\Omega$, then $\mathcal{R}_\Gamma$ is precisely the induced subgraph on $\Gamma$.

\begin{defn}\label{d: homogeneous}{\rm 
 A relational structure $\mathcal{R}=(\Omega, R_1,\dots, R_k)$ is called \emph{homogeneous} if, for all $\Gamma,\Gamma'\subset \Omega$ and for all isomorphisms $h:\mathcal{R}_{\Gamma}\to \mathcal{R}_{\Gamma'}$, there exists $g\in \Aut(\mathcal{R})$ such that $g|_{\Gamma} = h$.}
\end{defn}

The following example will be important shortly.

\begin{example}\label{e: src}{\rm 
Given a permutation group $G$ on a set $\Omega$ of size $t$, we define a relational structure $\mathcal{R}_G=(\Omega, R_1,\dots, R_k)$, where the relations $R_1,\dots, R_k$ are precisely the orbits of the group $G$ on the sets $\Omega^2,\dots, \Omega^{t-1}$.
 
Observe, first, that by definition any element of $G$ maps an element of relation $R_i$ to an element of relation $R_i$, for all $i\in\{1,\dots, k\}$; we conclude that $G\leq \Aut(\mathcal{R}_G)$.

On the other hand, suppose that $h\in \Aut(\mathcal{R}_G)$, and let $r=(\omega_1,\dots, \omega_{t-1})$ be a tuple of distinct elements in $\Omega$ lying in relation $R_j$, for some $j$. The image of this tuple under $h$ also lies in $R_j$; since $R_j$ is an orbit of $G$, this implies that there exists $g\in G$ such that for all $i\in\{1,\dots, t-1\},$ $\omega_i^h=\omega_i^g$. It follows that $\omega_t^h=\omega_t^g$, where $\omega_t$ is the only element of $\Omega$ not represented in the tuple $r$. We conclude that $h=g$ and so, in particular, $G=\Aut(\mathcal{R}_G)$.
 
Finally, suppose that $\Gamma$ and $\Delta$ are proper subsets of $\Omega$ of size $s$ such that the associated induced relational structures are isomorphic, i.e. there exists an isomorphism $h: (\mathcal{R}_G)_\Gamma\to(\mathcal{R}_G)_\Delta$. Let $r_\gamma=(\gamma_1,\dots,\gamma_s)$ be a tuple containing all of the distinct elements of $\Gamma$, and observe that $r_\gamma$ lies in a relation $R_j$ of $\mathcal{R}_G$, for some $j$. Indeed, by construction, $r_\gamma$ lies in the corresponding relation $R_j$ of $(\mathcal{R}_G)_\Gamma$, and so $(r_\gamma)^h$ lies in the corresponding relation $R_j$ of $(\mathcal{R}_G)_\Delta$, and hence also lies in the relation $R_j$ of $\mathcal{R}_G$. In particular, since $R_j$ is an orbit of $G$, we conclude that there exists $g\in G$ such that for all $i\in\{1,\dots, s\}$, $\gamma_i^h=\gamma_i^g$. Since $G=\Aut(\mathcal{R}_G)$, we conclude that $\mathcal{R}_G$ is homogeneous.}
\end{example}

We are ready to give our first definition of relational complexity. Before stating it, we remind the reader that we are assuming that $G$ is a permutation group on a set $\Omega$, and we recall that if $\mathcal{R}$ is any relational structure with vertex set $\Omega$, then $\Aut(\mathcal{R})$ is also a permutation group on $\Omega$.

\begin{defn}\label{d: src}{\rm 
The \emph{structural relational complexity} of a permutation group $G$ is equal to the smallest integer $s\geq 2$ for which there exists a homogeneous relational structure $\mathcal{R}=(\Omega, R_1,\dots, R_k)$ of arity $s$ such that $\Aut(\mathcal{R})$ is permutation isomorphic to $G$.}
\end{defn}

Note that Example~\ref{e: src} implies, in particular, that if $|\Omega|\geq 3$, then the structural relational complexity of $G$ is well-defined, and is bounded above by $|\Omega|-1$ (and is at least 2). In what follows, we will write $\srel(G,\Omega)$ for the structural relational complexity of the permutation group $G$.

One might wonder why we have required that $\srel(G,\Omega)\geq 2$. The reason is that, in the next section we will define a different statistic $\trel(G,\Omega)$ using a completely different approach, and we will also require that $\trel(G,\Omega)\geq 2$. We will then show that $\srel(G,\Omega)=\trel(G,\Omega)$ for all permutation groups $G$ on a set $\Omega$. Were we to omit the requirement that $\srel(G,\Omega)\geq 2$ and $\trel(G,\Omega)\geq 2$, there would be a number of actions for which $\srel(G,\Omega)\neq \trel(G,\Omega)$, for instance the natural action of $\Sym(\Omega)$.

\subsection{Tuples}

In this section we give an alternative approach to the notion of relational complexity based on \cite{cherlin_martin}. We then show that it coincides with the approach of the previous section. As before $G$ is a permutation group on a finite set $\Omega$.

\begin{defn}{\rm
Let $2\leq r \leq n$ be positive integers, and let $I=(I_1,\dots, I_n)$ and $J=(J_1,\dots, J_n)$ be elements of $\Omega^n$. We say that $I$ and $J$ are \emph{$r$-subtuple complete with respect to $G$} if, for all $k_1, k_2, \dots, k_r$ integers with $1\leq k_1, k_2, \dots, k_r\leq n$, there exists $g\in G$ with $I_{k_i}^g=J_{k_i}$ for $i \in\{ 1, \dots, r\}$. 
In this case we write $I\stb{r} J$.}
\end{defn}

Note that if $I\stb{r} J$ and $u\leq r$, then $I\stb{u} J$.

\begin{defn}{\rm
The permutation group $G$ has \emph{tuple relational complexity} equal to $s$ if the following two conditions hold:
\begin{enumerate}
 \item\label{eq:def1} if $n\geq s$ is any integer and $I,J$ are elements of $\Omega^n$ such that $I\stb{s} J$, then there exists $g\in G$ such that $I^g=J$.
 \item\label{eq:def2} $s\geq 2$ is the smallest integer for which~\eqref{eq:def1} holds.
\end{enumerate}
We write $\trel(G,\Omega)$ for the tuple relational complexity of the permutation group $G$.}
\end{defn}

Put another way, the tuple relational complexity of $G$ is the smallest integer $s\geq 2$ such that
\[
 I\stb{s} J \, \Longrightarrow \, I\stb{n} J,
\]
for any integer $n\geq s$, and any pair of $n$-tuples $I$ and $J$.

It is not immediately clear, {\it a priori}, that $\trel(G,\Omega)$ exists for every permutation group $G$ on the set $\Omega$. The next lemma deals with this concern.

\begin{lem}
If $\srel(G,\Omega)=s$, then $\trel(G,\Omega)$ exists and is bounded above by $s$.
\end{lem}
\begin{proof}
Let $n\geq 2$ be some integer, and let $I$ and $J$ be subsets of $\Omega^n$ such that $I\stb{s} J$. We must prove that there exists $g\in G$ such that $I^g=J$.

Let $\mathcal{R}$ be a homogeneous relational structure of arity $s$ for which $G=\Aut(\relR)$. Write $\{I\}$ (resp. $\{J\}$) for the underlying set associated with the $n$-tuple $I$ (resp. $J$); as $s\ge 2$, these sets must be of equal cardinality bounded above by $n$. Now consider the induced substructures $\relR_{\{I\}}$ and $\relR_{\{J\}}$ and consider the map $h:\relR_{\{I\}}\to\relR_{\{J\}}$ for which $h(I_i)=J_i$ for all $i\in\{1,\dots, n\}$. 

We claim that $h$ is an isomorphism of relational structures. Let $(I_{i_1},\dots, I_{i_u})$ be an element of some relation $R_j$ in $\relR_{\{I\}}$. Note that $u\leq s$ and 
%that $(I_{i_1},\dots, I_{i_u})$ and 
recall that $I\stb{u} J$ with respect to the action of $G$. Thus there exists $g\in G$ such that
\[
(J_{i_1},\dots, J_{i_u}) = (I_{i_1}, \dots, I_{i_u})^g.
\]
Then, since $g\in \Aut(\relR)$, we conclude that $(J_{i_1},\dots, J_{i_u})$ is an element of relation $R_j$ in $\relR_{\{J\}}$. We conclude that $h$ is an isomorphism as required.

Now, since $\mathcal{R}$ is homogeneous, there exists $g\in G=\Aut(\relR)$ such that $g_{|\{I\}} = h$; in particular $I^g=J$, as required. 
 \end{proof}

\begin{lem}
 $\srel(G,\Omega)\leq \trel(G,\Omega)$.
\end{lem}
\begin{proof}
Let $r=\trel(G,\Omega)$. Define $\relR=(\Omega, R_1, \dots, R_k)$, where $R_1,\dots, R_k$ are the orbits of $G$ on $\Omega^i$ for all $i\in\{2,\dots, r\}$.

Clearly $G\leq\Aut(\relR)$. Suppose that $\sigma\in \Aut(\relR)$, and let $I=(\omega_1,\dots, \omega_t)$ be a $t$-tuple of distinct elements of $\Omega$, where $t = |\Omega|$  (so every entry of $\Omega$ occurs as an entry in $I$). Then $I\stb{r} I^\sigma$, and so there exists $g\in G$ such that $I^g=I^\sigma$. This implies that $\sigma=g$, and so $\Aut(\relR)\leq G$. We conclude that $G=\Aut(\relR)$.

We must show that $\relR$ is homogeneous. Let $\Gamma$ and $\Delta$ be subsets of $\Omega$ of size $s$ such that there exists an isomorphism $\varphi:\relR_\Gamma \to \relR_\Delta$. Furthermore, let $I=(\gamma_1,\dots, \gamma_s)$ be an $s$-tuple of distinct elements of $\Gamma$. Suppose first $s\le r$. Since $\mathcal{R}$ contains all the orbits of $G$ on $\Omega^s$ and since $\relR_\Gamma\cong\relR_\Delta$, we deduce that $I$ and $\varphi(I)$ are in the same $G$-orbit, that is,  there exists $g\in G$ such that $I^g=\varphi(I)$. Thus $\varphi=g|_\Gamma$, as required. Suppose next $s>r$. Since all $r$-subtuples of $I$ occur as relations in $\relR$ and since $\relR_\Gamma\cong\relR_\Delta$, we conclude that $I\stb{r}\varphi(I)$. Since $r=\trel(G,\Omega)$, we deduce $I\stb{s}\varphi(I)$. As before, this implies that there exists $g\in G=\Aut(\relR)$ such that $I^g=\varphi(I)$; in other words $\varphi=g|_\Gamma$, as required.
\end{proof}

\begin{cor}
 $\srel(G,\Omega)=\trel(G,\Omega)$.
\end{cor}

In light of this corollary, we now drop the distinction between the two types of relational complexity:

\begin{defn}\label{d: rc}{\rm 
 The \emph{relational complexity} of $G$ is equal to the tuple relational complexity of $G$ (and hence also equal to the structural relational complexity of $G$), and is denoted $\rrel(G,\Omega)$.

 In particular, a permutation group $G\leq Sym(\Omega)$ is called \emph{binary} if $\RC(G,\Omega)=2$.}
 \end{defn}

Our definition of relational complexity has, to this point, pertained only to permutation groups, i.e. to \emph{faithful} group actions. It is convenient to extend this definition now to any group action:

\begin{defn}{\rm 
 Suppose that a group $G$ acts on a set $\Omega$. The \emph{relational complexity} of the action, denoted $\rrel(G,\Omega)$, is  the relational complexity of the permutation group induced by the action of $G$ on $\Omega$.}
\end{defn}

Note, finally, that in \cite{cherlin_martin} the word \emph{arity} is used as a synonym for relational complexity.

\section{Basics: Some key examples}\label{s: examples}

Our focus in this monograph is on actions with small relational complexity, thus the examples we present below are skewed in this direction. In particular, all of the actions listed in Conjecture~\ref{conj: cherlin original} are discussed.

As we shall see, there are times when the structural definition of relational complexity is easiest to work with, and times when we prefer the tuple definition.

Before we outline the primary examples, we need to say a few words about the third family in Conjecture~\ref{conj: cherlin original}. This family consists of all groups isomorphic to an affine orthogonal group $V\rtimes {\rm O}(V)$ with $V$ a vector space over a finite field equipped with a non-degenerate anisotropic quadratic form, acting on itself by translation, with complement the full orthogonal group ${\rm O}(V)$. It is a straightforward consequence of the classification of non-degenerate quadratic forms that if $V$ admits an anisotropic quadratic form $Q$ (i.e. one for which $Q(\mathbf{v})\neq 0$ for all $\mathbf{v}\in V\setminus\{\mathbf{0}\}$)\footnote{It may be perhaps better to call such a $Q$ a \emph{non-singular form} rather than an anisotropic form -- a vector $\mathbf{v}$ is generally called \emph{singular} if $Q(\mathbf{v})=0$, and \emph{isotropic} if $\beta(\mathbf{v},\mathbf{v})=0$ where $\beta$ is the polar form of $Q$. If the characteristic of the field is odd, these two definitions coincide, however in characteristic $2$ this is not the case. Our definition of an anisotropic form requires that the only singular vector for $Q$ is the zero vector, but note that all vectors are isotropic in the characteristic $2$ case. In any case, we will stick to calling such a $Q$ anisotropic as it is consistent with what has come before in the literature.}, then $\dim(V)\leq 2$. We will split this family into two smaller families according to whether $\dim(V)$ is $1$ or $2$; 
\begin{enumerate}
 \item[3a.] $\dim(V)=1$: the associated group $G$ is isomorphic to $\Fq\rtimes C_2$, where $C_2$ acts as $-1$ on the finite field $\Fq$ with $q$ elements, and the action is on $\Omega=\Fq$. For $G$ to be primitive we require that $q$ is prime, and we obtain that $G$ is isomorphic to the dihedral group of order $2q$, with the action being on the $q$-gon, as usual.
 \item[3b.] $\dim(V)=2$ and the associated quadratic form is of minus type: the associated group $G$ is isomorphic to $\Fq^2\rtimes \Or_2^-(q)\cong \Fq^2\rtimes D_{2(q+1)}$, where $D_{2(q+1)}$ is a dihedral group of order $2(q+1)$.
\end{enumerate}

First, let us observe that the relational complexity of the natural action of the symmetric group is as small as it can possibly be. %This, as we shall see, is rather unusual: for the most part ``large'' permutation groups tend to have ``large'' relational complexity. ({\color{red} Nott too clear what this remark means...})

\begin{example}\label{e: sn}{\rm
 Consider the natural action of $G=\Sym(t)$ on the set $\Omega=\{1,\dots, t\}$. Define
 \[
  R=\{(i,j) \mid 1\leq i, j\leq t \textrm{ and } i\neq j\}.
 \]
Then $\relR=(\Omega,R)$ is the complete directed graph,
$\relR$ is homogeneous and $G=\Aut(\relR)$. We conclude immediately that $\rrel(G,\Omega)=2$. }
\end{example}

Note that the first family of permutation groups listed in Conjecture~\ref{conj: cherlin original} is precisely the family of finite symmetric groups in their natural action.

In many group-theoretic respects, the alternating group is very like the symmetric group. The next example shows that relational complexity does not conform to this rule-of-thumb: while, as we have just seen, the natural action of the symmetric group has relational complexity as small as it can possibly be, the natural action of the alternating group has relational complexity as large as it can possibly be.

\begin{example}\label{e: an}{\rm
 Consider the natural action of $G=\Alt(t)$ on the set $\Omega=\{1,\dots, t\}$. Consider the tuples
 \[
  I=(1,2,3,\dots, t) \textrm{ and } J=(2,1,3,\dots, t).
 \]
It is straightforward to check that $I\stb{t{-}2} J$; it is equally clear that the only permutation $h$ for which $I^h=J$ is $h=(1,2)\not\in G$. We conclude that $\rrel(G,\Omega) \geq t-1$. Now Example~\ref{e: src} implies that $\rrel(G,\Omega)=t-1$.}
\end{example}

The previous two examples are a salutary warning that, in general, relational complexity behaves badly with respect to subgroups. All is not lost however: Lemma~\ref{l: subgroup} shows that the relational complexity of a group is related to that of some of its subgroups.

Our first aim is to understand the actions listed in Conjecture~\ref{conj: cherlin original}. Note that the Families 2 and 3a (using the notation at the start of this section) consist of primitive actions with very small point-stabilizers (size 1 and 2, respectively). In the next couple of examples we consider this situation.

\begin{example}\label{e: regular}{\rm 
If $G$ acts regularly on  $\Omega$, then $\rrel(G,\Omega)$ is binary. 

\textbf{Proof:} Suppose that $I=(I_1,\dots, I_n)$ and $J=(J_1,\dots, J_n)$ satisfy $I\stb{2} J$. For $i\in\{1,\dots, n-1\}$, let $g_i$ be an element of $G$ that satisfies $I_i^{g_i}=J_i$ and $I_{i+1}^{g_i} = J_{i+1}$. The regularity of $G$ implies that, for $j\in\{1,\dots, n\}$, there is a unique element of $G$ satisfying $I_j^g = J_j$. This fact, applied with $j=2$, implies that $g_1=g_2$; then applied with $j=3$, implies that $g_2=g_3$, and so on. Thus $g_1=\cdots =g_{n-1}$; calling this element $g$, we see that $I^g=J$ and we conclude that $I\stb{n} J$, as required. $\qed$}
\end{example}

Recall that the only regular primitive actions are associated with cyclic groups of prime order; we see, then, that the second family of groups in Conjecture~\ref{conj: cherlin original} are precisely the regular primitive groups.

\begin{example}\label{e: stab size 2}{\rm 
Suppose that $G$ is transitive and a point-stabilizer $H$ has size $2$, and suppose that $x$ is the non-trivial element in $H$. Let $C=x^G$ be the conjugacy class of $x$ in $G$. Then 
\[
 \RC(G)=\begin{cases}
         2, & \textrm{if } C\not\subseteq C^2; \\
         3, & \textrm{otherwise}.
        \end{cases}
\]

\textbf{Proof:}  It is an easy exercise to verify that, under these assumptions, $\RC(G)\leq 3$. One can use, for instance, Lemma~$\ref{l: rc and h}$ below.
% or \ref{l: point stabilizer}.

Since $\RC(G)\leq 3$, it is clear that a pair of $n$-tuples will be $n$-subtuple complete if and only if they are $3$-subtuple complete. Thus, if there exists an $n$-tuple that is $2$-subtuple complete but not $n$-subtuple complete, then there must exist a $3$-tuple that is $2$-subtuple complete but not $3$-subtuple complete. 

 Suppose that $G$ is not binary, and let $(P,Q)=((P_1,P_2,P_3), (Q_1,Q_2,Q_3))$ be a pair of $3$-tuples that is $2$-subtuple complete but not $3$-subtuple complete. Then there is, by assumption, an element $g$ of $G$ that maps $(P_1,P_2)$ to $(Q_1, Q_2)$. Replacing $Q$ by $Q^{g^{-1}}$ and relabelling, we conclude that there exists a pair
 \[
  ((P_1, P_2, P_3), (P_1, P_2, P_4))
 \]
that is $2$-subtuple complete but not $3$-subtuple complete, in particular $P_3\ne P_4$. Write $H_i$ for the stabilizer of $P_i$, and let $x_i$ be the non-trivial element of $H_i$. Then we must have 
\[
 P_3^{x_1} = P_3^{x_2} = P_4.
\]
Since $(P,Q)$ is not $3$-subtuple complete, $x_1\ne x_2$, otherwise $P^{x_1}=Q$. Moreover, since $P_3^{x_1x_2}=P_3$, we conclude that $x_1x_2$ is the non-trivial element in $H_3$. Thus $C\subseteq C^2$, as required.

Suppose now that $C\subseteq C^2$. Let $x_1,x_2,x_3\in C$ with $x_3=x_1x_2$. In particular, there exist three points $P_1,P_2$ and $P_3$ with $G_{P_1}=\langle x_1\rangle$, $G_{P_2}=\langle x_2\rangle$ and $G_{P_3}=\langle x_3\rangle$. Set $P_4:=P_3^{x_1}$. We claim that  $((P_1, P_2, P_3), (P_1, P_2, P_4))$ is a pair of $3$-tuples that is $2$-subtuple complete. In fact,
\begin{align*}
(P_1,P_2)^{1_G}&=(P_1,P_2),\\
(P_1,P_3)^{x_1}&=(P_1^{x_1},P_3^{x_1})=(P_1,P_4),\\
(P_2,P_3)^{x_2}&=(P_2^{x_2},P_3^{x_2})=(P_2,P_3^{x_3x_2})=(P_2,P_3^{x_1})=(P_2,P_4).
\end{align*}
If this pair is $3$-subtuple complete, then there exists $g\in G$ with $P_1^g=P_1$, $P_2^g=P_2$ and $P_3^g=P_4$. In particular, $g\in \langle x_1\rangle\cap \langle x_2\rangle$. If $g=1$, then $P_3=P_4=P_3^{x_1}$ and hence $x_1\in \langle x_3\rangle$. This gives $x_1=x_3$ and hence $x_2=1$ because $x_1x_2=x_3$. However, this is a contradiction. Thus $g=x_1=x_2$ and hence $x_3=x_1x_2=1$, again a contradiction. Therefore, $((P_1, P_2, P_3), (P_1, P_2, P_4))$ is a pair of $3$-tuples that are $2$-subtuple complete but that are not $3$-subtuple complete; hence $G$ is not binary.
$\qed$}
\end{example}

There is an important special case which occurs when point-stabilizers are of size $2$, and $G$ has a regular normal subgroup $N$. In this case it follows immediately that $C\not\subseteq C^2$ (where $C$ is as in Example~\ref{e: stab size 2}), and thus $\RC(G,\Omega)=2$. Such an action is primitive if and only if $N$ is of prime order, and we now see that Family 3a pertaining to Conjecture~\ref{conj: cherlin original} is precisely this.\footnote{The problem of specifying relational complexity when point-stabilizers have size 2  is now reduced to the problem of studying when $C$, a certain conjugacy class of involutions satisfies $C\subseteq C^2$. This problem is, in general, difficult, however one potential avenue of investigation is via the \emph{class constants} of the finite group $G$, denoted $a_{ijv}$. For any conjugacy class $C_i$ in a group $G$, we define $\hat{C_i}=\sum_{c\in C_i}c_i$ to be the class sum of $C_i$ in the group algebra $\mathbb{C}G$. Now write
\[
 \hat{C_i}\hat{C_j}=\sum\limits_{v=1}^k a_{ijv}\hat{C_v},
\]
where $k$ is the number of conjugacy classes in $G$. The non-negative integers $a_{ijv}$ for $1\leq i,j,v\leq k$ are the \emph{class constants} of $G$. Now a well-known formula asserts that
\[
 a_{ijv}=\frac{|C_i| |C_j|}{|G|}\sum\limits_{\chi\in {\rm Irr}_{\mathbb{C}}(G)} \frac{\chi(g_i) \chi(g_j) \chi(g_v^{-1})}{\chi(1)}.
\]
We conclude, therefore, that if a point-stabilizer $H = \langle x \rangle$ has size $2$, then $\RC(G)=2$ if and only if
\[
 \sum\limits_{\chi\in {\rm Irr}_{\mathbb{C}}(G)} \frac{\chi(x)^3}{\chi(1)}=0.
\]}

 Our next example addresses Family~3b in Conjecture~\ref{conj: cherlin original}.

\begin{example}\label{e: weird orthogonal}
{\rm 
This example is Lemma~1.1 of \cite{cherlin2}. We identify $\Omega$ with a vector space $V$ over a field $F$, such that $V$ is endowed with a quadratic form $Q$ such that $Q$ is anisotropic, i.e. $Q(v)\neq 0$ for all $v\in V\setminus\{0\}$. We set $G=V\rtimes \Or(V)$, where $\Or(V)$ is the isometry group of the form $Q$, and the semidirect product is the natural one, as is the action of $G$ on $\Omega=V$.

Let us see that this action is binary. Let $n$ be a positive integer, and assume that $\bbu=(u_0,\dots, u_n)$ and $\bbu'=(u_0',\dots, u_n')$ satisfy $\bbu\stb{2}\bbu'$. Let us show that $\bbu\,\stb{n+1}\,\bbu'$. We may suppose, without loss of generality that $u_0=u_0'=0$. 

Note that $\bbu\stb{2}\bbu'$ implies that $Q(u_i) = Q(u_i')$ for all $i\in\{1,\dots, n\}$. What is more, since the isometry group also preserves the polar form $\beta$ of $Q$, $\bbu\stb{2}\bbu'$ also implies that 
\[\beta(u_i, u_j)=\beta(u_i', u_j'),\]
for any $1\leq i,j\leq n$. This, in turn, implies that
\begin{equation}\label{e:lc}
 Q\left(\sum\limits_{j=1}^n c_j u_j\right) = Q\left(\sum\limits_{j=1}^n c_j u_j'\right),
\end{equation}
for any choice of scalars $c_1,\dots, c_n\in F$.

Let $W={\rm span}(\bbu)$, and let $W'={\rm span}(\bbu')$ and suppose, without loss of generality, that $u_1,\dots, u_m$ is a basis for $W$ (for $m=\dim(W)$). We claim that then $u_1',\dots, u_m'$ is a basis for $W'$. To see this, it is enough to show that if $u_1,\dots, u_k$ are linearly independent, then so too are $u_1',\dots, u_k'$. Suppose that $c_1,\dots, c_k\in F$ such that $c_1u_1'+\cdots+ c_ku_k'=0$. Then, clearly,
\[Q(c_1u_1'+\cdots +c_k u_k')=Q(0)=0.\]
But, by the observation above, this implies that $Q(c_1u_1+\cdots +c_ku_k)=0$, which implies that $c_1u_1+\cdots +c_ku_k=0$, which in turn implies that $c_1=\cdots=c_k=0$. The claim follows.

Now we can define an isometry $f:W\to W'$ by setting $f(u_i)=u_i'$ for $i\in\{1,\dots, m\}$, and extending linearly. Then Witt's Lemma implies that there exists $g\in \Or(V)$ such that $u_i^g=u_i'$ for all $i\in\{1,\dots, m\}$. Let us now consider $m<i\leq n$. Write $u_i=\sum_{j=1}^m c_j u_j$ and now, observe that \eqref{e:lc} yields that
\[
 Q\left(u_i'-\sum\limits_{j=1}^m c_j u_j'\right) =Q\left(u_i-\sum\limits_{j=1}^n c_j u_j\right) = Q(0)=0. 
\]
Now the fact that $Q$ is anisotropic implies that $u_i'-\sum\limits_{j=1}^m c_j u_j'=0$, and we conclude that $u_i^g=u_i'$, as required.}
\end{example}

All of the examples considered so far have been transitive. Let us briefly consider what can happen with intransitive actions.

\begin{example}\label{e: intransitive}{\rm 
 Suppose that the action of $G$ on $\Omega$ is intransitive with orbits $\Delta_1,\dots, \Delta_v$. It is immediate from the definition that
 \[
  \RC(G,\Omega)\geq \max\{\RC(G,\Delta_1),\RC(G,\Delta_2),\dots, \RC(G, \Delta_v)\}.
 \]
On the other hand, let $n\geq 3$ and consider the intransitive action of $G=\Sym(n)$ with two orbits, where the action on the first orbit is the natural one of degree $n$, and the second orbit is of size $2$. Clearly the action of $G$ on each orbit is binary; on the other hand, one can check directly that $\RC(G,\Omega)=n=t-2$.}
\end{example}

This example suggests that the problem of calculating the relational complexity of intransitive actions may be rather difficult. 

\subsection{Existing results on relational complexity}

Results on relational complexity above and beyond the basic examples discussed above  are hard to obtain. Nearly all of the important results are due to Cherlin, and his co-authors, and we briefly mention some of these here. The first result is stated in \cite{cherlin1}, with a small correction in \cite{cherlin2}.

\begin{thm}
Let $\Omega$ be the set of all $k$-subsets of a the set $\{1,\dots, n\}$ with $2k\leq n$. If $G=\Sym(n)$, then $\RC(G,\Omega)=2+\lfloor\log_2k\rfloor$. If $G=\Alt(n)$, then
\[
 \RC(G,\Omega)=\begin{cases}
                n-1, & \textrm{if $k=1$}; \\
                \max(n-2,3), & \textrm{if $k=2$}; \\
                n-2, & \textrm{if $k\geq 3$ and $n=2k+2$}; \\
                n-3, & \textrm{otherwise}.
               \end{cases}
\]
\end{thm}

The actions of the symmetric and alternating groups on partitions, rather than $k$-sets, are currently being studied by Cherlin and Wiscons \cite{cherwisc}. The only general result to date is for $\Sym(2n)$ and $\Alt(2n)$ acting on $\Omega$, the set of partitions of $2n$ into $n$ blocks of size $2$ (so, for $G=\Sym(2n)$, this is the action on cosets of a maximal imprimitive subgroup of form $\Sym(2)\wr\Sym(n)$). The result they have obtained for $n\geq 2$ is as follows:
\begin{align*}
 \RC(\Sym(2n), \Omega) &= n; \\
 \RC(\Alt(2n), \Omega) &= \begin{cases}
                          2, & n=2; \\
                          4, & n\in\{3,4\}; \\
                          n, & n>3 \textrm{ and } n\equiv 0,1,3,5\pmod 6; \\
                          n-1, & n>4 \textrm{ and } n\equiv 2,4\pmod 6.
                         \end{cases}
\end{align*}
%that in this case the complexity for $\Sym(2n)$ is always $n$, except when $n = 1$ or 4, in which case it is $n+1$. For $\Alt(2n)$, the complexity is expected to be either $n$ or $n-1$ (depending on the values of $n$ modulo 2 and 3), except when $n = 1, 3$ or 4, in which case it is $n+1$.

As we shall see below (Theorem~\ref{t: glodas}), when considering large relational complexity, an important family of actions involves groups $G$ which are subgroups of $\Sym(m)\wr\Sym(r)$ containing $(\Alt(m))^r$, where the action of $\Sym(m)$ is on $k$-subsets of $\{1,\dots, m\}$ and the wreath product has the product action of degree $t=\binom{m}{k}^r$. The particular situation where $G=\Sym(m)\wr \Sym(r)$ is studied in \cite{cherlin_martin}. We summarise some of the results there, using the notation just established.

\begin{thm}\label{t: prod}
 Let $G=\Sym(m)\wr\Sym(r)$ acting on a set $\Omega$ of size $t=\binom{m}{k}^r$, as described.
 \begin{enumerate}
  \item If $m=2$, then $k=1$ and $\RC(G, \Omega)=2+\lfloor\log_2 r \rfloor$.
  \item If $k=1$, then $\RC(G, \Omega)\leq m+\lfloor\log_2 r \rfloor$. 
  \item $\RC(G,\Omega)\leq \floor{2+\log_2 k}\floor{1+\log_2 r}$ with equality if $m\geq 2k\floor{1+\log_2 r}$.
 \end{enumerate}
\end{thm}

The particular situation where $k=1$ and $G=\Sym(m)\wr\Sym(r)$ (so we are considering the natural product action of degree $m^r$) has been taken much further in a series of papers by Saracino \cite{sar1, sar2, sar3}. Saracino's results effectively yield an exact value for the relational complexity of this family of actions. We do not write this value here as the precise formulation of the results is slightly involved; instead we refer to \cite[\S6]{cherlin_martin} and to the papers of Saracino, particularly the first.

\section{Motivation: On homogeneity}\label{s: m h}

In his paper \cite{cherlin1}, Cherlin chooses a quote from Aschbacher as an epigraph. This quote, plus some more, goes as follows:

\begin{quotation}
 Define an object $X$ in a category $\mathfrak{C}$ to possess the \emph{Witt property} if, whenever $Y$ and $Z$ are subobjects of $X$ and $\alpha:Y\to Z$ is an isomorphism, then $\alpha$ extends to an automorphism of $X$. Witt's Lemma says that orthogonal spaces, symplectic spaces, and unitary spaces have the Witt property in the category of spaces with forms and isometries. All objects in the category of sets and functions have the Witt property. But in most categories few objects have the Witt property; those that do are very well behaved indeed. If $X$ is an object with the Witt property and $G$ is its group of automorphisms, then the representation of $G$ on $X$ is usually an excellent tool for studying $G$. \cite[pp. 81, 82]{aschbacher}
\end{quotation}

One should think of ``the Witt property'' as a generalization of the notion of homogeneity which we have introduced in the specific setting of relational structures. The study of homogeneous objects in different categories has a long and interesting history.\footnote{There is some inconsistency in terminology across the literature -- homogeneity as we have defined it here is sometimes called ``ultra-homogeneity'' while homogeneity refers to a strictly weaker property.}

Before discussing this history, let us delve a little deeper into why such objects have received attention: Aschbacher's answer is given above. This approach has its roots in the {\it Erlangen Programme} of Klein, in which the key features of a particular ``geometry'' define, and are defined by, the group of automorphisms of said geometry. The idea here is that one studies the geometry in question, one deduces information about the geometry, which one then reinterprets as information about the associated group; one can use this information about the group to deduce further information about the geometry and so on. Thus the process of mathematical inquiry moves back-and-forth between geometrical study and algebraic (group theoretic).

The efficacy of this approach varies considerably -- if an object has a very small automorphism group for instance, then group theory may provide very little insight. On the other hand, as Aschbacher suggests, this approach is most spectacularly successful when the object in question is homogeneous. Indeed the two examples which Aschbacher mentions clearly illustrate the success of this approach.

First, we note that the category of sets and functions have the Witt property. If we restrict ourselves to finite objects in this category, then the associated automorphism groups are the finite symmetric groups, $\Sym(n)$. Of course, all of the basic group-theoretical information about these groups is most naturally expressed in the language of their natural (homogeneous) action on a set of size $n$. This includes their conjugacy class structure (via cycle type), and their subgroup structure (via the O'Nan--Scott-Aschbacher Theorem \cite{as, scott}; see also \cite{lps1}).

Second, in the category of spaces with forms, basic linear algebra asserts that objects associated with a zero form (i.e.\ naked vector spaces) have the Witt property; Witt's Lemma extends this to cover objects associated with either a non-degenerate quadratic or non-degenerate sesquilinear form. Again, restricting ourselves to finite such objects, we obtain the finite classical groups as the associated automorphism groups. As before, the basic group-theoretical properties of these groups are most naturally expressed in the language of their natural homogeneous action on the associated vector space. This includes their conjugacy class structure (via rational canonical form for $\GL_n(q)$, and the variants due to Wall for the other classical groups \cite{Wall}), and their subgroup structure (via Aschbacher's Theorem \cite{aschbacher2}).

In light of all this, a natural question when studying some (permutation) group $G$ is whether we can find an object in some category on which $G$ acts homogeneously. Example~\ref{e: src} gives an easy answer to this: it turns out that there is always such an object in the category of relational structures. The bad news is that the object provided by Example~\ref{e: src} is little more than an encoding of the complete structure of the permutation group in terms of a relational structure -- studying the structure $\relR_G$ will hardly be easier than studying the original group and its associated action. 

The investigation of relational complexity seeks to remedy this disappointing state of affairs: given a group $G$ and an associated action, $\RC(G,\Omega)$ gives us an indication of the efficiency with which we can build a relational structure on which $G$ can act homogeneously. From this point of view, an ``efficient'' representation of $G$ acting homogeneously on a relational structure is one for which the arity of the structure is as small as possible.

There is an alternative way of viewing efficiency in this context where one is, instead, interested in using relational structures with as few relations as possible (but not necessarily worrying about the arity of the relations used). We will not pursue this point of view here, but we refer to \cite{klm} (for the primitive case) and to \cite{cherlin_hrushovski} (for the general case), for results that pertain to this approach.

\subsection{Existing results on homogeneity}

We briefly review some important results on homogeneity for particular finite relational structures.

The classification of homogeneous graphs was partially completed by Sheehan \cite{sheehan}, and then completely by Gardiner \cite{gardiner}. Indeed, Gardiner's result applies to a wider class of graphs than those we would call homogeneous. This classification was then extended by Lachlan to homogeneous digraphs \cite{lachlan_digraphs}. 

In order to state these results we need some terminology: a \emph{digraph}, $\Gamma$, is an ordered pair $(V(\Gamma), E(\Gamma))$, where $V(\Gamma)$ is a non-empty set, and $E(\Gamma)$ is an irreflexive binary relation on that set. The digraph is \emph{symmetric} (resp. \emph{anti-symmetric}) if, whenever $(x,y)\in E(\Gamma)$, we have $(y,x)$ in (resp. not in) $E(\Gamma)$. So a symmetric digraph is the object commonly called a \emph{graph} in the literature.

 If $\Gamma$ and $\Delta$ are two digraphs, then we can construct two new digraphs with vertex set $V(\Gamma)\times V(\Delta)$:
\begin{enumerate}
 \item in the \emph{composition of $\Gamma$ and $\Delta$}, $\Gamma[\Delta]$, vertices $(u_1, v_1)$ and $(u_2, v_2)$ are connected if and only if $(u_1,u_2)\in E(\Gamma)$, or $u_1=u_2$ and $(v_1,v_2)\in E(\Delta)$;
 \item in the \emph{direct product of $\Gamma$ and $\Delta$}, $\Gamma\times \Delta$ vertices $(u_1, v_1)$ and $(u_2, v_2)$ are connected if and only if $(u_1,u_2)\in E(\Gamma)$ and $(v_1,v_2)\in E(\Delta)$.
\end{enumerate}

We write $K_n$ for the complete (symmetric di)graph on $n$ vertices. We also define two infinite families of graphs, both indexed by a parameter $n\in\Z$ with $n\geq 3$:
\begin{enumerate}
 \item $\Lambda_n$ is the digraph with vertex set $\{0,1,\dots, n-1\}$ and $(x,y)\in E(\Gamma_n)$ if and only if $x-y\equiv 1\pmod n$;
 \item $\Delta_n$ is the symmetric digraph with vertex set $\{0,1,\dots, n-1\}$ and $(x,y)\in E(\Delta_n)$ if and only if $x-y\equiv\pm1\pmod n$.
\end{enumerate}

Thus $\Lambda_n$ is the directed cycle on $n$ vertices, and $\Delta_n$ is the undirected cycle on $n$ vertices. Let $\mathcal{S}$ (resp. $\mathcal{A}$) denote the set of homogeneous symmetric (resp. antisymmetric) digraphs. We write $\overline{\Gamma}$ for the complement of $\Gamma$. Then Gardiner's result is the following:

\begin{thm}\label{t: gardiner}
 A digraph $\Gamma$ is in $\mathcal{S}$ if and only if $\Gamma$ or $\overline{\Gamma}$ is isomorphic to one of
 \[
  \Delta_5,\, K_3\times K_3,\, K_m[\overline{K_n}],
 \]
where $m,n\in\Z^+$.
\end{thm}

Now we will state Lachlan's result in three stages. First we need to define three ``sporadic homogeneous digraphs''; this is done in Figure~\ref{f: digraph}. 

\begin{figure}
\usetikzlibrary{decorations.markings}
\begin{subfigure}[b]{0.32\textwidth}
        \centering
        \resizebox{\linewidth}{!}{
\newcount\mycount
 \begin{tikzpicture}[scale=0.6]%[transform shape]
  %the multiplication with floats is not possible. Thus I split the loop in two.
  \foreach \number in {1,...,8}{
      % Computer angle:
        \mycount=\number
        \advance\mycount by -1
  \multiply\mycount by 45
        \advance\mycount by 0
      \node[draw,circle,inner sep=0.05cm] (N-\number) at (\the\mycount:5.4cm) {};
    }
\begin{scope}[very thick,decoration={
    markings,
    mark=at position 0.5 with {\arrow{>}}}
    ] 
    \draw[postaction={decorate}] (N-1) -- (N-4);
    \draw[postaction={decorate}] (N-1) -- (N-6);
    \draw[postaction={decorate}] (N-1) -- (N-7);
    \draw[postaction={decorate}] (N-2) -- (N-3);
    \draw[postaction={decorate}] (N-2) -- (N-4);
    \draw[postaction={decorate}] (N-2) -- (N-1);
    \draw[postaction={decorate}] (N-3) -- (N-6);
    \draw[postaction={decorate}] (N-3) -- (N-8);
    \draw[postaction={decorate}] (N-3) -- (N-1);
    \draw[postaction={decorate}] (N-4) -- (N-3);
    \draw[postaction={decorate}] (N-4) -- (N-5);
    \draw[postaction={decorate}] (N-4) -- (N-6);
    \draw[postaction={decorate}] (N-5) -- (N-2);
    \draw[postaction={decorate}] (N-5) -- (N-3);
    \draw[postaction={decorate}] (N-5) -- (N-8);
    \draw[postaction={decorate}] (N-6) -- (N-5);
    \draw[postaction={decorate}] (N-6) -- (N-7);
    \draw[postaction={decorate}] (N-6) -- (N-8);
    \draw[postaction={decorate}] (N-7) -- (N-5);
    \draw[postaction={decorate}] (N-7) -- (N-2);
    \draw[postaction={decorate}] (N-7) -- (N-4);
    \draw[postaction={decorate}] (N-8) -- (N-2);
    \draw[postaction={decorate}] (N-8) -- (N-7);
    \draw[postaction={decorate}] (N-8) -- (N-1);
\end{scope}   
    
\end{tikzpicture}
        }
        \caption{$H_0$}
        \label{fig:H0}
    \end{subfigure}
    \begin{subfigure}[b]{0.32\textwidth}
    \centering
        \resizebox{\linewidth}{!}{
 \begin{tikzpicture}%[scale=0.6]%[transform shape]
  %the multiplication with floats is not possible. Thus I split the loop in two.
      \node[draw,circle,inner sep=0.05cm] (N-1) at (3,2) {};
      \node[draw,circle,inner sep=0.05cm] (N-2) at (2,3) {};
      \node[draw,circle,inner sep=0.05cm] (N-3) at (-2,3) {};
      \node[draw,circle,inner sep=0.05cm] (N-4) at (-3,2) {};
      \node[draw,circle,inner sep=0.05cm] (N-5) at (-3,-2) {};
      \node[draw,circle,inner sep=0.05cm] (N-6) at (-2,-3) {};
      \node[draw,circle,inner sep=0.05cm] (N-7) at (2,-3) {};
      \node[draw,circle,inner sep=0.05cm] (N-8) at (3,-2) {};
\begin{scope}[very thick] 
 \draw (N-1) -- (N-2);
 \draw (N-3) -- (N-4);
 \draw (N-5) -- (N-6);
 \draw (N-7) -- (N-8);
 \end{scope}
      \begin{scope}[very thick,decoration={
    markings,
    mark=at position 0.5 with {\arrow{>}}}
    ] 
    \draw[postaction={decorate}] (N-1) -- (N-8);
    \draw[postaction={decorate}] (N-1) -- (N-3);
    \draw[postaction={decorate}] (N-2) -- (N-7);
    \draw[postaction={decorate}] (N-2) -- (N-4);
    \draw[postaction={decorate}] (N-3) -- (N-6);
    \draw[postaction={decorate}] (N-3) -- (N-2);
    \draw[postaction={decorate}] (N-4) -- (N-1);
    \draw[postaction={decorate}] (N-4) -- (N-5);
    \draw[postaction={decorate}] (N-5) -- (N-7);
    \draw[postaction={decorate}] (N-5) -- (N-3);
    \draw[postaction={decorate}] (N-6) -- (N-4);
    \draw[postaction={decorate}] (N-6) -- (N-8);
    \draw[postaction={decorate}] (N-7) -- (N-6);
    \draw[postaction={decorate}] (N-7) -- (N-1);
    \draw[postaction={decorate}] (N-8) -- (N-2);
    \draw[postaction={decorate}] (N-8) -- (N-5);
\end{scope}   
    
\end{tikzpicture}
         }
        \caption{$H_1$}
        \label{fig:H1}
    \end{subfigure}
    \begin{subfigure}[b]{0.32\textwidth}
    \centering
        \resizebox{\linewidth}{!}{
\newcount\mycount
 \begin{tikzpicture}[scale=0.6]%[transform shape]
  %the multiplication with floats is not possible. Thus I split the loop in two.
  \foreach \number in {1,...,12}{
      % Computer angle:
        \mycount=\number
        \advance\mycount by -1
  \multiply\mycount by 30
        \advance\mycount by 0
      \node[draw,circle,inner sep=0.05cm] (N-\number) at (\the\mycount:5.4cm) {};
    }
\begin{scope}[very thick] 
 \draw (N-1) -- (N-2);
 \draw (N-3) -- (N-4);
 \draw (N-5) -- (N-6);
 \draw (N-7) -- (N-8);
 \draw (N-9) -- (N-10);
 \draw (N-11) -- (N-12);
 \end{scope}
      \begin{scope}[very thick,decoration={
    markings,
    mark=at position 0.5 with {\arrow{>}}}
    ] 
    \draw[postaction={decorate}] (N-1) -- (N-12);
    \draw[postaction={decorate}] (N-1) -- (N-10);
    \draw[postaction={decorate}] (N-2) -- (N-5);
    \draw[postaction={decorate}] (N-3) -- (N-2);
    \draw[postaction={decorate}] (N-4) -- (N-5);
    \draw[postaction={decorate}] (N-4) -- (N-7);
    \draw[postaction={decorate}] (N-6) -- (N-7);
    \draw[postaction={decorate}] (N-8) -- (N-9);
    \draw[postaction={decorate}] (N-9) -- (N-6);
    \draw[postaction={decorate}] (N-11) -- (N-10);
    \draw[postaction={decorate}] (N-11) -- (N-8);
    \draw[postaction={decorate}] (N-12) -- (N-3);
\end{scope}   
    
\end{tikzpicture}
        }
        \caption{$H_2$}
        \label{fig:H2}
    \end{subfigure}
 \caption{Three homogeneous digraphs. The presence of an undirected edge $\{v,w\}$ in the diagrams for $H_0$ and $H_1$ indicates that both directed edges between $v$ and $w$ are present. In the diagram for $H_2$ we have omitted most of the directed edges. To obtain the remaining edges, note first that each vertex in $H_2$ has a unique \emph{mate}, to which it is connected by an undirected edge (indicated in the diagram). Next, let $v$ and $w$ be vertices, and let $w'$ be the mate of $w$. Finally, if $(v,w)$ is a directed edge, then $(w',v)$ is a directed edge, and if $(w,v)$ is a directed edge, then $(v,w')$ is a directed edge. This leads to the insertion of another 36 directed edges.}\label{f: digraph}
\end{figure}
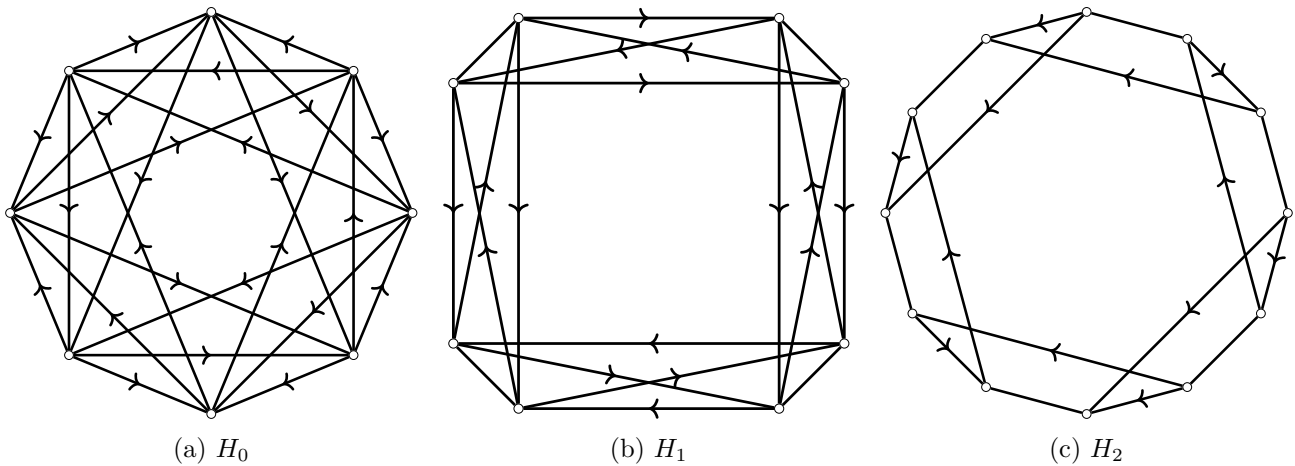

Second we classify the homogeneous antisymmetric digraphs.

\begin{thm}\label{t: lachlan antisymmetric}
 A digraph $\Gamma$ is in $\mathcal{A}$ if and only if $\Gamma$ is isomorphic to one of 
 \[
  \Lambda_4, \,\overline{K_n},\, \overline{K_n}[\Lambda_3],\, \Lambda_3[\overline{K_n}],\, H_0,
 \]
where $n\in\mathbb{Z}^+$.
\end{thm}

Finally we can state Lachlan's classification of homogeneous digraphs.
\begin{thm}\label{t: lachlan}
 A digraph $\Gamma$ is homogeneous if and only if $\Gamma$ or $\overline{\Gamma}$ is isomorphic to a digraph with one of the following forms:
 \[
  K_n[A],\, A[K_n], \,S,\, \Lambda_3[S],\, S[\Lambda_3],\, H_1,\, H_2,
 \]
where $n\in \Z^+$, $A\in \mathcal{A}$ and $S\in \mathcal{S}$.
\end{thm}

Lachlan's result, expressed in our terms, is \emph{almost} a classification of those homogeneous relational structures $\mathcal{R}=(\Omega, R_1)$ such that $R_1$ is binary. We write ``almost'' because Lachlan imposes the condition that $R_1$ is irreflexive whereas we make no such restriction. Nonetheless, given that in this monograph we are focusing on transitive actions, Lachlan's result is sufficient: any relational structure $\mathcal{R}=(\Omega, R_1)$ for which $R_1$ is binary and $\Aut(\mathcal{R})$ is transitive on $\Omega$, will either be precisely of the form listed in Theorem~\ref{t: lachlan}, or else will be of the form listed in Theorem~\ref{t: lachlan} with the addition of a loop at every vertex. We have made no attempt to extend this classification to the situation where $\Aut(\mathcal{R})$ is not transitive on $\Omega$ although we note that in this situation, $\Aut(\mathcal{R})$ would have exactly two orbits on $\Omega$ -- one corresponding to vertices with loops, one corresponding to vertices without.

The groups $\Aut(\Gamma)$ for $\Gamma$ appearing in Theorem~\ref{t: lachlan} have not been explicitly listed to our knowledge. We will not calculate this list, but we can at least start the task: It is easy to check that $\Aut(\Lambda_n)$ is the cyclic group of order $n$, $\Aut(\Delta_n)$ is the dihedral group of order $2n$ and $\Aut(K_n)$ is the symmetric group of degree $n$. It is slightly more involved to check the larger sporadic examples; the automorphism group and the action on points (which is necessarily binary) are as follows:
\begin{enumerate}
 \item $\Aut(K_3\times K_3)=\Sym(3)\wr\Sym(2)$ in the product action on 9 points;
 \item $\Aut(H_0)\cong \SL_2(3)$ acting on the 8 cosets of a Sylow $3$-subgroup;
 \item $\Aut(H_1)$ is the semidihedral group of order 16 -- it has presentation $\langle x,y|x^8=y^2, x^y=x^3\rangle$ -- in an action of degree $8$;
 \item $\Aut(H_2)\cong \Alt(4)\rtimes C_4$ where $C_4=\langle x\rangle$ acts by conjugation on $\Alt(4)$ via $g^x=g^{(1,2,3,4)}$ for all $g\in \Alt(4)$; as an abstract group $\Aut(H_2)\cong (\Alt(4)\times 2).2$, and the action is of degree $12$.
\end{enumerate}
To complete the enumeration of the automorphism groups of homogeneous digraphs, we would need to  study the automorphisms of the various graphs arising from the composition of two others: for instance, we would need to calculate $\Aut(A[K_n])$ and $\Aut(K_n[A])$ for each $A\in \mathcal{A}$. We will not do this.
%For a complete enumeration of the groups involved, we would need to study the automorphisms of the various graphs arising from the composition of two others -- typically this will involve wreath products.

There are a multitude of results that extend Gardiner, Sheehan and/or Lachlan's results to finite (di)graphs with automorphism groups that satisfy weaker properties than homogeneity. We particularly mention \cite{gmpr} which considers so-called \emph{set-homogeneous} digraphs. In a different direction Cherlin has classified the homogeneous countable digraphs \cite{cherlin_digraph} extending work of Lachlan and Woodrow classifying the homogeneous countable graphs \cite{lw}, and of Lachlan classifying the homogeneous countable tournaments \cite{lachlan_tour}.

Analogues for some of the given results exist for relational structures containing a single relation which may not be binary. Lachlan and Tripp have classified the homogeneous 3-graphs \cite{lachlan_tripp} and Cameron has done the same for homogeneous $k$-graphs with $k\geq 6$ \cite{cameron6}; these results are analogues of Gardiner's result for ternary relational structures with a single relation. Devillers has studied a rather similar problem in her work on homogeneous Steiner systems, however the notion of homogeneity considered there is different to ours \cite{adevil}.

\section{Motivation: On model theory}

Cherlin's conjecture arises from model theory considerations rooted in Lachlan's theory of finite homogeneous relational structures (see, for instance, \cite{kl2, lachlan}). We give a brief summary of some of the main ideas; the origin of nearly everything we consider here is \cite{cherlin1}.

Let us consider a family of theorems indexed by parameters $k$ and $\ell$, with $k,\ell\in\Z^+$ and $\ell\geq 2$. Theorem $(k,\ell)$ is a full classification of the homogeneous relational structures with at most $\ell$ relations, and with arity at most $k$. So, for instance, the first theorem we are likely to consider is Theorem $(2,1)$ which (modulo the transitivity assumption we discussed above) is just Theorem~\ref{t: lachlan}, a result of Lachlan himself that classifies finite binary relational structures with one relation; in other words finite simple homogeneous directed graphs.

Lachlan's theory of finite homogeneous relational structures asserted a number of facts about the form of these theorems, and about the relationships between them. With regard to the form of the theorem, Lachlan's theory asserts that each theorem can be written as follows:
\begin{center} ``\emph{A finite homogeneous relational structure of arity at most $k$ with at most $\ell$ relations lies in one of a number of infinite families, or else is one of a finite number of sporadic individuals.}'' 
\end{center}
 The power of this assertion is in the restrictions which Lachlan placed upon the definition of the word ``family'': a family of finite homogeneous relational structures in Lachlan's sense is an infinite collection of structures that can be constructed from a single infinite relational structure via a set of explicitly described operations.

With regard to the relationships between these theorems, Lachlan's theory gives us information about what the word ``sporadic'' means in these theorems. Specifically he asserts that any sporadic individual cropping up in Theorem $(k,\ell)$, say, will appear later as part of an infinite family in Theorem $(k', \ell')$ for some $k'\geq k$ and $\ell'\geq \ell$. Thus the ``sporadic-ness'' of a particular homogeneous relational structure is, in some sense, not genuine -- rather, it is an artefact of restricting our investigations to particular values of $k$ and $\ell$.

The significance of all of this from a group-theoretic point of view lies in Cherlin's observation that every finite permutation group can be viewed as the automorphism group of a homogeneous relational structure -- we demonstrated one way of seeing this in Example~\ref{e: src}. This observation allows us to shift our point of view on the family of theorems studied by Lachlan: we can think of them as being about finite permutation groups. 

In this setting the parameters $k$ and $\ell$ can be seen as providing some kind of stratification on the universe of finite permutations groups, and Lachlan's results concerning ``families'' and ``sporadic-ness'' can be seen as statements about groups as well as structures. Finally, we can rewrite the theorems themselves from a group-theoretic point of view; they take the following form: 
\begin{center}``\emph{Let $G$ be the automorphism group of a homogeneous relational structure $\mathcal{R}$ on a set $\Omega$ of arity at most $k$ with at most $\ell$ relations. Then, viewed as a permutation group on $\Omega$, $G$ lies in one of a number of infinite families, or else is one of a finite number of sporadic individuals}.''
\end{center}

With this set-up, any given permutation group $G$ will occur in an infinite number of Theorems $(k,\ell)$. Typically, though, we are interested in the \emph{first} such occurrence: we are interested in the pair $(k,\ell)$ for which $k$ is minimal, and having fixed $k$ as this minimal value, we then seek the minimum possible value of $\ell$.  The resulting pair $(k,\ell)$ is a measure of the \emph{complexity} of $G$ from the model-theoretic point of view or, using the point of view espoused in \S\ref{s: m h}, gives a measure of the efficiency with which $G$ can be represented as the automorphism group of a homogeneous relational structure.

Of course, plenty remains: we know that these theorems about finite permutation groups exist; we know their form, and we know something about the relationships that exist between them. We would like to know the statements of these theorems, and we would like to prove them!

As described in the previous section, this last task has only been completed for Theorem $(2,1)$ (and, even then, with a small caveat). The main theorem of this monograph completes the task of ascertaining which groups appear as \emph{primitive} permutation groups in any Theorem $(2,\ell)$.

\section{Motivation: Other important statistics}

It turns out that relational complexity is closely connected to a number of other permutation group statistics, some of which have received a great deal of attention in the literature. Our reference for the following definitions is \cite{bailey_cameron}.

For $\Lambda = \{\omega_1,\dots,\omega_k\} \subseteq \Omega$ and for $G\le \Sym(\Omega)$, we write
$G_{(\Lambda)}$ or $G_{\omega_1, \omega_2, \dots, \omega_k}$  for the
point-wise stabilizer. If $G_{(\Lambda)} = \{1\}$, then we say that $\Lambda$ is a \emph{base}. The size of the smallest possible base is known as the \emph{base size} of $G$ and is denoted $\base(G)$. 

We say that a base is a \emph{minimal base} if no proper subset of it is a base. We denote the maximum size of a minimal base by $\Base(G)$.

Given an ordered sequence of elements of $\Omega$, $[\omega_1,\omega_2,\dots, \omega_k]$, we can study the associated \emph{stabilizer chain}:
\[
  G \geq G_{\omega_1} \geq G_{\omega_1, \omega_2}\geq 
  G_{\omega_1, \omega_2, \omega_3} \geq \dots \geq
  G_{\omega_1, \omega_2, \dots, \omega_k}.
\]
If all the inclusions given above are strict, then the stabilizer chain is called \emph{irredundant}. If, furthermore, the group $G_{\omega_1,\omega_2,\dots, \omega_k}$ is trivial, then the sequence $[\omega_1,\omega_2,\dots, \omega_k]$ is called an \emph{irredundant base}. The size of the longest possible irredundant base is denoted $\Irred(G)$.

Finally, let $\Lambda$ be any subset of $\Omega$. We say that $\Lambda$ is an \emph{independent set} if its point-wise stabilizer is not equal to the point-wise stabilizer of any proper subset of $\Lambda$. We define the \emph{height} of $G$ to be the maximum size of an independent set, and we denote this quantity by $\Height(G)$. 

Note that if $G$ is a transitive permutation group on a set $\Omega$, then $\Height(G)=1$ if and only if $G$ is regular; similarly, $\Height(G)=2$ if and only if the stabilizer of a point is a non-trivial TI-subgroup of $G$. (Recall that $X$ is said to be a non-trivial TI-subgroup of a group $G$ if $X$ is a proper subgroup of $G$ and $X\cap X^g=1$, for every $g\in G\setminus N_G(X)$.)

There is a basic connection between the four statistics we have defined so far:
\begin{equation}\label{eq: inequality}
 \base(G) \leq \Base(G) \leq \Height(G) \le \Irred(G)\leq \base(G)\log t.
\end{equation}
Recall that in this document, if the base is not specified, then ``$\log$'' always means ``$\log$ to the base $2$''; recall, also, that $t=|\Omega|$. Let us see why \eqref{eq: inequality} is true:

The first inequality is obvious. For the second, suppose that $\Lambda$ is a minimal base; then $\Lambda$ is an independent set. For the third, suppose that $\Lambda:=\{\omega_1,\omega_2,\ldots,\omega_k\}$ is an independent set and observe that \[
  G> G_{\omega_1} > G_{\omega_1, \omega_2}> 
  G_{\omega_1, \omega_2, \omega_3} > \dots >
  G_{\omega_1, \omega_2, \dots, \omega_k}
\]
is a strictly decreasing sequence of stabilizers.
In particular, $[\omega_1,\omega_2,\ldots,\omega_k]$ is irredundant and we may extend this irredundant sequence to an irredundant base. Hence $\Height(G)\le \Irred(G)$. 

The fourth inequality has been attributed to Blaha \cite{blaha} who, in turn, describes it as an ``observation of Babai'' \cite{babai}. Suppose that $G$ has a base of size $b=\base(G)$. Then, in particular $|G|\leq t^b$. On the other hand, any irredundant base and any independent set have size at most $\log|G|$. We conclude that $\Irred(G) \leq \log(t^b)$, and the result follows.
%
%Note that the preceding proof connects $\Irred(G)$ to $\base(G)$ by way of $|G|$. It is worth noting how this goes for : we have
%\[
% \base(G)\geq \log_t|G| \textrm{ and } \Irred(G) \leq \log|G|.
%\]

We are ready to connect relational complexity to the four statistics we have just defined. The key result is the following.

\begin{lem}\label{l: rc and h}
 $\RC(G) \leq \Height(G) + 1$.
\end{lem}
\begin{proof} %{\color{red}(Incomplete - needs completing or replacing)}
Let $h = H(G)$ and consider a pair $(I,J)\in \Omega^n$ such that $I\stb{h{+}1} J$. We must show that $I\stb{n} J$.

Observe that we can reorder the tuples without affecting their subtuple completeness. Hence, without loss of generality, we can assume that
\[ G_{I_1} > G_{I_1,I_2}> \dots > G_{I_1, I_2, \dots, I_{\ell}},
\] for some $\ell \leq h$ and then this chain stabilizers, i.e. 
\[ G_{I_1, \dots, I_{\ell}} = G_{I_1, \dots, I_{\ell+j}},
\] for all $1 \leq j \leq n-\ell$.
From the assumption of $h$-subtuple completeness it follows that there exists an element $g \in G$ such that $I_i ^g = J_i$ for all $1 \leq i \leq \ell$ and observe that the set of all such elements $g$ forms a coset of $G_{I_1, \dots, I_{\ell}}$. 

The assumption of $(h+1)$-subtuple completeness implies, moreover, that for all $1 \leq j \leq n-\ell$ there exists $g_j \in G$ such that 
\[ \begin{cases}
I_i ^{g_j} = J_i, \quad \text{for} \quad 1\leq i \leq \ell, \\
I_{\ell+j}^{g_j} = J_{\ell+j}.
\end{cases}
\]
The set of all such elements $g_j$ forms a coset of $G_{I_1, \dots, I_\ell, I_{\ell+j}}$, which is, again, a coset of $G_{I_1, \dots, I_\ell}$. Since any coset of $G_{I_1,\dots, I_\ell}$ is defined by the image of the points $I_1,\dots, I_\ell$ under an element of the coset, we conclude that elements of the same coset of $G_{I_1, \dots, I_\ell}$ map $I_{\ell+j}$ to $J_{\ell+j}$ for all $1 \leq j \leq n-\ell$. In particular, $I\stb{n} J$, as required.
\end{proof}

Lemma~\ref{l: rc and h} has been exploited in \cite{glodas}, where an upper bound on the height of a primitive permutation group is proved, from which the obvious upper bound on relational complexity is deduced. The main result on height is the following:

\begin{thm}\label{t: glodas} Let $G$ be a finite primitive group of degree $t$. Then one of the following holds:
\begin{enumerate}
\item $G$ is a subgroup of $\Sym(m) \mathrm{wr} \Sym(r)$ containing $(\Alt(m))^r$, where the action of $\Sym{(m)}$ is on $k$-subsets of $\{ 1, \dots, m\}$ and the wreath product has the product action of degree $t= \binom{m}{k} ^r$;
\item $\Height(G) < 9 \log{t} $.
\end{enumerate}
\end{thm}

Note that various members of the family listed at item (1) of Theorem~\ref{t: glodas} genuinely violate the bound at item (2): for example, when $r=k=1$, we obtain the groups $\Sym(t)$ and $\Alt(t)$ in their natural action, for which the height is $t-1$ and $t-2$, respectively. In fact, though, we do not know the exact height of the groups listed at item (1) for all possible values of $k$, $m$ and $r$.

The proof of Theorem~\ref{t: glodas} exploits the rich array of results in the literature giving bounds on $\base(G)$ for various families of permutation groups. In particular, use is made of the proof of the Cameron-Kantor conjecture \cite{cameron-kantor} by Liebeck and Shalev \cite{lie_shal}, and of Cameron's follow-up conjecture giving a value for the associated constant \cite{cameron} by many authors \cite{burn, burgs, burls, burow}. These results mean that, in the almost simple case, work is only required for the so-called ``standard actions''.

Theorem~\ref{t: glodas} is an analogue of an existing result for $\base(G)$ \cite{liebeck_base}; now \eqref{eq: inequality} and Lemma~\ref{l: rc and h} yield analogues for $\Base(G)$ and $\RC(G)$. With this result for $\RC(G)$, and with the proof of Conjecture~\ref{conj: cherlin original}, we now have a good handle on those permutation groups $G$ for which $\RC(G)$ is either very large, or as small as possible. In the case where $\RC(G)$ is large, work remains to be done to ascertain the relational complexity of the groups listed at item (1) of the theorem; the most important results in this direction can be found in \cite{cherlin_martin}, and we summarised some of these above in Theorem~\ref{t: prod}.

The relationship between the various statistics occurring in \eqref{eq: inequality}, and between these statistics and $\RC(G)$ is an intriguing area of investigation, although not one that has hitherto received much attention. Cherlin and Wiscons have started to study some of these questions, and we mention two of their remarks \cite{cw}:
\begin{enumerate}
 \item From computational evidence,  it appears that $\RC(G)$ and $H(G)$ are ``close'' (say, $\RC(G)\geq\Height(G)-3$). The obvious exceptions to this rule of thumb are the symmetric groups in their natural action; more generally, among primitive groups of degree at most 100, the only groups for which $\RC(G)<\Height(G)-3$ are various members of the family listed at item (1) of Theorem~\ref{t: glodas}.
 \item Again, from computational evidence, more often than not, it appears that $\Base(G)$ and $\Height(G)$ coincide for primitive groups. Moreover, for all primitive groups of degree at most 100, $\Height(G)-\Base(G)\le 3$.
\end{enumerate}
We shy away from making conjectures about the general pattern for larger $n$ but, still, these lines of inquiry seem promising.

\section{Methods: basic lemmas}

Most of the results in this section were first written down in \cite{ dgs_binary, ghs_binary, gs_binary}. All of these papers were focused on showing that certain group actions are not binary, hence the lemmas we present here tend to yield lower bounds for relational complexity. 

As always $G$ is a group acting on a set $\Omega$. In what follows, we will write $I,J\in \Omega^n$ to mean that $n\geq 2$ is a positive integer and $I,J$ are elements of $\Omega^n$; we will always assume that $I=(I_1,\dots, I_n)$ and $J=(J_1,\dots, J_n)$. We will write $I\stb{k} J$ to mean that the pair $(I,J)$ is $k$-subtuple complete; we will write $I\stc{k}{n} J$ to mean that the pair $(I,J)$ is $k$-subtuple complete but not $n$-subtuple complete with respect to the action of $G$.

\subsection{Relational complexity and subgroups}

Examples~\ref{e: sn} and \ref{e: an} serve as a warning that relational complexity can behave badly with respect to arbitrary subgroups of the group $G$. Nonetheless, something can still be said.

\begin{lem}\label{l: point stabilizer}
 Let $G$ be a transitive permutation group on $\Omega$ and let $M$ be a point-stabilizer in this action. Let $\Lambda$ be a non-trivial orbit of $M$. Then
 \[
  \RC(G,\Omega) \geq \RC(M,\Lambda).
 \]
\end{lem}

Note, in particular, that if $G$ is binary, then the action of $M$ on all non-trivial suborbits must be binary. This will be useful later, particularly when we consider actions in which $G$ is very large and $M$ relatively small (for instance, $G=E_8(2)$, and $M=\Aut(\PSU_3(8))$), in which case it is sometimes possible to use {\tt magma} to list all of the transitive binary actions of $M$.

\begin{proof}
 Write $\alpha$ for an element of $\Omega$ stabilized by $M$. Let $r=\RC(M,\Lambda)$; then there exist $I,J\in\Lambda^n$ such that $I\stc{r{-}1}{n} J$ with respect to the action of $M$ on $\Lambda$. But now observe that if we define
 \[
  I^*=(\alpha, I_1,\dots, I_n) \textrm{ and } J^*=(\alpha, J_1, \dots, J_n),
 \]
then $I^*\stc{r{-}1}{n{+}1} J^*$, and the result follows.
\end{proof}

We write $(G:M)$ here, and below, to mean the set of right cosets of $M$ in $G$.

\begin{lem}\label{l: subgroup}
Let $M<H<G$. Then $\RC(G: (G:M))\geq \RC(H, (H:M))$.
\end{lem}

\begin{proof}
 Write $r=\RC(H: (H:M))$, and observe that $\Lambda=(H:M)$ is a subset of $\Omega=(G:M)$. Then there exist $I,J\in \Lambda^n$ such that $I\stc{r-1}{n}J$ with respect to the action of $H$.
 
 We must show that $I\stc{r-1}{n}J$ with respect to the action of $G$. That $I\stb{r{-}1} J$ with respect to the action of $G$ is immediate. Suppose that $I\stb{n} J$ with respect to the action of $G$. Then there exists $g\in G$ such that $I_i^g = J_i$ for all $i\in\{1,\dots, n\}$. Since $I_i, J_i \in (H:M)$ for all $i\in\{1,\dots, n\}$, we must have $g\in H$. But then $I\stb{n} J$ with respect to the action of $H$, which is a contradiction.
\end{proof}

\subsection{Relational complexity and subsets}

For $\Lambda$ a subset of $\Omega$ we write $G_\Lambda$ for the \emph{set-wise} stabilizer of $\Lambda$, and $G_{(\Lambda)}$ for the \emph{point-wise} stabilizer of $\Lambda$. We write $G^\Lambda$ for the permutation group induced on $\Lambda$ by $G_\Lambda$; note that $G^\Lambda \cong G_\Lambda/G_{(\Lambda)}$.

In this section we present some results connecting $\RC(G,\Omega)$ with $\RC(G^\Lambda, \Lambda)$.

\begin{defn}{\rm 
Let $t:=|\Omega|$. For $k\in \Z^+$ with $k \geq 2$, we say that the action of $G$ on $\Omega$ is \emph{strongly non-$k$-ary} if there exist $I,J\in \Omega^t$ such that $I\stc{k}{t} J$, and all elements of $I$ (resp. $J$) are distinct.}
\end{defn}

Note that this definition requires the existence of $I,J\in\Omega^t$ with $I\stc{k}{t} J$ and with every element of $\Omega$ occurring as an entry of $I$ (and, therefore, also of $J$). If $k=2$, then we tend to write \emph{strongly non-binary} as a synonym for strongly non-$k$-ary.

The notion of a strongly non-$k$-ary set is connected to a classical notion in permutation group theory which was introduced by Wielandt \cite{Wielandt}.

\begin{defn}{\rm
Let $G\leq \Sym(\Omega)$ and let $k\in \Z^+$. The $k$-closure of $G$ is the set
\[
 G^{(k)} = \{\sigma\in\Sym(\Omega) \mid \forall I\in \Omega^k, \textrm{ there exists } g\in G, I^g=I^\sigma\}.
\]
We say that $G$ is \emph{$k$-closed} if $G=G^{(k)}$.}
\end{defn}

Observe that $G^{(k)}$ is the largest subgroup of $\Sym(\Omega)$ that has the same orbits on the set of $k$-tuples of $\Omega$ as $G$. Now the connection with strongly non-$k$-ary sets is as follows.

\begin{lem}\label{l: fedup}
The group $G$ is strongly non-$k$-ary if and only if $G$ is not $k$-closed.
\end{lem}
\begin{proof}
 Write $\Omega:=\{\omega_1,\ldots,\omega_t\}$. If $G$ is not $k$-closed, then there exists $\sigma\in G^{(k)}\setminus G$. Now, it is easy to verify that $I:=(\omega_1,\ldots,\omega_t)$ and $J:=I^\sigma=(\omega_1^\sigma,\ldots,\omega_t^\sigma)$ are $k$-subtuple complete (because $\sigma\in G^{(k)}$) and are not $t$-subtuple complete (because $\sigma\notin G$). Thus $I\stc{k}{t} J$, and we conclude that the action of $G$ on $\Omega$ is strongly non-$k$-ary. The converse is similar.
\end{proof}

The most important example, for us, of a permutation group that is not $k$-closed is as follows.

\begin{example}\label{e: ktrans}{\rm
 Let $G$ be a $k$-transitive permutation group on $\Omega$, for some integer $k\geq 2$. The definition implies that $G^{(k)}=\Sym(\Omega)$.
 
We immediately conclude that $\Alt(\Omega)$ is not $(t-2)$-closed, and we obtain (again) that $\RC(\Alt(\Omega), \Omega)\geq t-1$.

Recall that the Classification of Finite Simple Groups implies that examples of $k$-transitive permutation groups that do not contain $\Alt(\Omega)$ only exist for $k\leq 5$. What is more, all such groups are classified for $k\geq 2$ (see, for instance \cite[\S7.7]{dixon_mortimer}).}
\end{example}

The next lemma shows how we will use the notion of a strongly non-$k$-ary permutation group in what follows.

\begin{lem}\label{l: again12}
Let $\Lambda \subseteq \Omega$. If $G^\Lambda$ is strongly non-$k$-ary, then $\RC(G,\Omega)>k$.
\end{lem}
\begin{proof}
Suppose that $|\Lambda|=\ell$, and let $I,J$ be $\ell$-tuples of distinct elements of $\Lambda$ such that $I\stc{k}{\ell} J$ with respect to the action of $G^\Lambda$. It is enough to show that $I\stc{k}{\ell} J$ with respect to the action of $G$. It is immediate that $I\stb{k} J$ with respect to the action of $G$. On the other hand, if $I\stb{\ell} J$, then there exists $g\in G$ such that $I^g=J$. Since $I$ contains all elements of $\Lambda$, we conclude that $g\in G_\Lambda$ which contradicts the fact that $I\nstb{\ell} J$ with respect to the action of $G^\Lambda$.
\end{proof}

\subsection{Strongly non-binary subsets}

Our final few results apply specifically to the study of binary actions. As usual $G$ acts on a set $\Omega$, and we refer to a subset $\Lambda\subseteq\Omega$ as \emph{strongly non-binary} if $G^\Lambda$ is strongly non-binary.

The next lemma details our first example of such a subset. This example was first described in \cite{gs_binary}; its key properties are a consequence of Example~\ref{e: ktrans} and Lemma~\ref{l: again12}.

\begin{lem}\label{l: 2trans gen}{\rm 
 Suppose that there exists a subset $\Lambda \subseteq \Omega$ such that $|\Lambda|\geq 2$ and $G^\Lambda$ is a 2-transitive proper subgroup of $\symme(\Lambda)$. Then $G^\Lambda$ is strongly non-binary and the action of $G$ on $\Omega$ is not binary.}
\end{lem}

In subsequent chapters, our focus is on proving that certain actions are not binary. Lemma~\ref{l: 2trans gen} means that we will be interested in finding subsets which have 2-transitive set-wise stabilizers. The next lemma requires no proof, but we include it as it clarifies when such subsets exist.

\begin{lem}\label{l: 2t}
Let $K$ be some $2$-transitive group, and let $K_0$ be a point-stabilizer in $K$.
Let $H$ be a subgroup of $G$ and suppose that $\varphi:H\to K$ is a surjective homomorphism. Let $M$ be the stabilizer in $G$ of a point $\omega\in\Omega$ and let $C$ be the core of $H\cap M$ in $H$. If $\mathrm{Ker}(\varphi)=C$ and  $\varphi(H\cap M)=K_0$, then $H$ acts $2$-transitively on the orbit $\omega^H$.
\end{lem}

The next lemma is a useful tool in finding subsets on which a set-stabilizer acts 2-transitively (recall that, when $r\ge 2$, the affine special linear group $\ASL_r(q)$ is 2-transitive in its natural action on $q^r$ points).

\begin{lem}\label{aff}
Let $G$ be a finite group acting transitively on a set $\O$ with point-stabilizer $M$, and suppose that the following two conditions hold:
\begin{itemize}
\item[{\rm (i)}] $M$ has a subgroup $A \cong \SL_r(q)$, where $r \ge 2$, and
\item[{\rm (ii)}] $G$ has a subgroup $S$ that is a central quotient of $\SL_{r+1}(q)$, such that $A\le S$ (the natural completely reducible embedding) and $S \not \le M$.
\end{itemize}
Then there is a subset $\D$ of $\O$ such that $|\D| = q^r$ and $G^\D \ge \ASL_r(q)$.
\end{lem}

\begin{proof}
We have $A \le S\cap M < S$. Since $A$ is embedded in $S$ via the natural completely reducible embedding, we have $S\cap M \le P_i(S)$ with $i\in  \{1,r\}$, where $P_i(S)$ is a maximal parabolic subgroup of $S$ stabilizing a $1$-dimensional or an $r$-dimensional subspace. Say $i=1$ (the case $i=r$ is entirely similar). Then writing matrices with respect to a suitable basis, 
\[
S\cap M \le P_1(S) = \left\{ \left(\begin{matrix}Y & v \\ 0 & \l\end{matrix}\right) : Y \in \mathrm{GL}_r(q), v\in \mathbb{F}_q^r, \det(Y)\l = 1\right\},
\]
where $A$ is the subgroup obtained by setting $\l=1$, $\det(Y)=1$ and $v=0$.
Define 
\[
U = \left\{ \left(\begin{matrix}I & 0 \\ a & 1\end{matrix}\right) : a \in \F_q^r \right\},
\]
and set $\D = \{Mu : u \in U\} \subseteq \O$ (where we identify $\O$ with the set $(G:M)$ of right cosets of $M$ in $G$).

Since $M\cap U=1$, the cosets $Mu\,(u\in U)$ are all distinct, and so $|\D|=q^r$. Since $A$ normalizes $U$ and $A \le M$, the subgroup $UA \cong q^r.\SL_r(q)$ stabilizes $\D$, and since $UA\cap M = A$, we have $(UA)^\D = \ASL_r(q) \le G^\D$. 
\end{proof}

It turns out that in the context of almost simple groups, it is convenient to use a variant of Lemma~\ref{l: 2trans gen} where we don't just seek proper 2-transitive subgroups of $\Sym(\Omega)$, but also exclude $\Alt(\Omega)$ from our consideration. To that end we include the following definition which first appeared in \cite{gs_binary}.

\begin{defn}\label{d: beautiful}{\rm 
A subset $\Lambda\subseteq \Omega$ is a \emph{$G$-beautiful subset} if $G^\Lambda$ is a 2-transitive subgroup of $\symme(\Lambda)$ that is isomorphic to neither $\alter(\Lambda)$ nor $\symme(\Lambda)$.}
\end{defn}

In what follows, if the group $G$ is clear from the context, we will speak of a beautiful subset rather than a $G$-beautiful subset of $\Omega$. Observe that a beautiful subset of $\Omega$ is a strongly non-binary subset. The reason for the stronger definition is explained by the following result.

\begin{lem}\label{l: beautiful}
 Suppose that $G$ is almost simple with socle $S$. If $\Omega$ contains an $S$-beautiful subset, then $G$ is not binary.
\end{lem}
\begin{proof}
 Let $\Lambda$ be an $S$-beautiful subset and observe that $\Lambda$ has cardinality at least $5$. Then, since $S$ is normal in $G$, the group $(S_\Lambda G_{(\Lambda)})/G_{(\Lambda)}$ is a normal subgroup of $G_\Lambda/G_{(\Lambda)}$. This implies that $G_\Lambda/G_{(\Lambda)}$ is (isomorphic to) a $2$-transitive proper subgroup of $\symme(\Lambda)$. Then Lemma~\ref{l: 2trans gen} implies that $G$ is not binary.
\end{proof}

Although in this paper we do not need to deal with $\mathcal{C}_1$-actions for classical groups since they were dealt with in \cite{gs_binary}, we include the next lemma because it clearly illustrates the beautiful subsets method. The lemma has the added advantage of giving the reader an idea of how to deal with $\mathcal{C}_1$-actions in general. (These actions all yield to the method of beautiful subsets provided $n$ and $q$ are not too small.)

\begin{lem}\label{l: illustrate beautiful}
Let $S=\PSL_n(q)$ and for $n=2$ assume $q>5$. Let $M$ be a maximal parabolic subgroup of $S$, and let $\Omega$ be the set of right cosets of $M$. Then $\Omega$ contains an $S$-beautiful subset.
  \end{lem}
\begin{proof}%[Proof of Table~$\ref{t: c1 sln}$]
Here $M$ is the stabilizer of a subspace $W$ of $V$, where $V$ is the natural $n$-dimensional module for $\SL_n(q)$. Since the action of $S$ on the $k$-dimensional subspaces of $V$ is permutation isomorphic to the action on the $(n-k)$-subspaces of $V$, we may assume that $\dim (W)\le n/2$. 

If $\dim(W)=1$, then the action of $S$ on $\Omega$ is $2$-transitive. Now $\Omega$ itself is an  $S$-beautiful subset, because we are assuming $q>5$ when $n=2$.
  
 Suppose next that $\dim(W)>1$. Observe that this implies that $n\geq 4$. Let $W'$ be a subspace of  $W$ with $\dim(W')=\dim(W)-1$ and consider $\Lambda=\{W''\le V\mid W'\subset W'',\dim(W'')=\dim(W)\}$. Clearly, $S_\Lambda=\stab_S(W')$ and the action of $S^\Lambda$ on $\Lambda$ is permutation isomorphic to the natural $2$-transitive action of $\GL(V/W')$ on the $1$-dimensional subspaces of $V/W'$. Since $\dim(V/W')\geq 3$, the action of $\GL(V/W')$ induces neither the alternating nor the symmetric group on the set $P_1(V/W')$ of $1$-dimensional subspaces of $V/W'$; therefore $\Lambda$ is a beautiful subset. %The exception $\SL_2(4)$ does not arise here because $\dim(V/W')>2$.
 \end{proof}

Our second example of a strongly non-binary subset is taken from \cite[Example~2.2]{ghs_binary}
 
\begin{example}\label{ex: snba2}{\rm
Let $G$ be a subgroup of  $\Sym(\Omega)$, let $g_1, g_2,\ldots,g_r$ be elements of $G$, and let $\tau,\eta_1,\ldots,\eta_r$ be elements of $\Sym(\Omega)$ with
\[
 g_1=\tau\eta_1,\,\,g_2=\tau\eta_2,\,\,\ldots,\,\,g_r=\tau\eta_r.
\]
Suppose that, for every $i\in \{1,\ldots,r\}$, the support of $\tau$ is disjoint from the support of $\eta_i$; moreover, suppose  that, for each $\omega\in\Omega$, there exists $i\in\{1,\ldots,r\}$ (which may depend upon $\omega$) with $\omega^{\eta_i}=\omega$. Suppose, in addition, $\tau\notin G$.
Now, writing $\Omega=\{\omega_1,\dots, \omega_t\}$, observe that
 \[
  ((\omega_1,\omega_2,\dots, \omega_t), (\omega_1^{\tau},\omega_2^{\tau}, \ldots,\omega_t^{\tau}))
 \]
is a non-binary witness. Thus the action of $G$ on $\Omega$ is strongly non-binary.}
\end{example}

The next two lemmas which are taken from \cite{dgs_binary} are based on Example~\ref{ex: snba2}. In both cases, the given assumptions on the permutation group $G$ are enough to conclude that a strongly non-binary subset of the type described in Example~\ref{ex: snba2} must exist. In both lemmas, given a permutation or a permutation group $X$ on $\Omega$, we let $\mathrm{Fix}_\Omega(X)$ define the subset of $\Omega$ fixed point-wise by $X$; if $\Omega$ is clear from the context, we drop the label $\Omega$.

\begin{lem}[{{\cite[Lemma 2.5]{dgs_binary}}}]\label{l: M2}
Let $G$ be a transitive permutation group on $\Omega$, let $\alpha\in \Omega$ and let $p$ be a prime with $p$ dividing both $|\Omega|$ and $|G_\alpha|$ and with $p^2$ not dividing $|G_\alpha|$. Suppose that $G$ contains an elementary abelian $p$-subgroup $V=\langle g,h\rangle$ with $g\in G_\alpha$, with $\langle h\rangle$ and $\langle gh\rangle$ conjugate to $\langle g\rangle$ via $G$. Then $G$ is not binary. 
\end{lem}
In~\cite[Lemma 2.5]{dgs_binary}, the hypothesis actually requires that $h$ and $gh$ are conjugate to $g$  via $G$; however the same proof yields the conclusion that $G$ is not binary under the weaker assuption that $\langle h\rangle$ and $\langle gh\rangle$ are conjugate to $\langle g\rangle$ in $G$, as stated in the lemma. We will need this strengthening in what follows.

\begin{lem}[{{\cite[Lemma 2.6]{dgs_binary}}}]\label{l: added}
Let $G$ be a permutation group on $\Omega$ and suppose that $g$ and $h$ are $G$-conjugate elements of prime order $p$, and $gh^{-1}$ is conjugate to $g$  (and so to $h$). Suppose that $V=\langle g, h\rangle$ is elementary abelian of order $p^2$. Suppose, 
finally, that $G$ does not contain any elements of order $p$ that fix more points of $\Omega$ than $g$. If $|\Fix(V)|<|\Fix(g)|$, then $G$ is not binary.
\end{lem}

%We remark that there are well-known formulae that we can use to calculate $\Fix(V)$ and $|\Fix(g)|$ when $G$ is transitive (see for instance~\cite[Lemma~$2.5$]{LiebeckSaxl}). Suppose that $M$ is the stabilizer of a point in $\Omega$; then we have
%\begin{equation}\label{e: fora}
% |\Fix_\Omega(g)| = \frac{|\Omega|\cdot |M\cap g^G|}{|g^G|},\qquad
% |\Fix_\Omega(V)| = \frac{|\Omega|\cdot |\{V^g\mid g\in G,V^g\le M\}|}{|V^G|}.
%\end{equation}

\section{Methods: Frobenius groups}

It turns out that the presence of Frobenius groups can be a powerful tool in proving that certain actions are not binary. We give three lemmas in this direction; the first was proved independently by Wiscons, although a proof has not appeared in the literature.\footnote{Here is the shorter and more elegant argument due to Wiscon for Lemma~\ref{l: frobenius}

For distinct $a,b,c\in \Omega$, binarity implies that the intersection of the suborbits $cG_a$ and $cG_b$ is equal to $cG_{a,b}$, so as the action is Frobenius, $(cG_a ) \cap (cG_b) = \{c\}$. Also, using again that the action is Frobenius,  $|cG_a| = |G_a| = |G_b| = |cG_b|$. This shows that $\bigcup_{a\neq c} (cG_a \setminus \{c\})$ is a disjoint union of sets of constant size  $|G_a| - 1$.  So, letting $N = |\Omega|$, we find that  $N - 1 = |\Omega \setminus \{c\}| \ge (N-1)( |G_a| - 1)$, implying that $|G_a| = 2$.
}

\begin{lem}\label{l: frobenius}
Let $G$ be a  Frobenius permutation group on $\Omega$ (that is, $G$ acts transitively on $\Omega$, $G_\omega\ne 1$ for every $\omega\in \Omega$ and $G_{\omega\omega'}=1$ for every $\omega,\omega'\in \Omega$ with $\omega\ne \omega'$). If $G$ is binary, then a Frobenius complement has order equal to $2$.
\end{lem}

\begin{proof}
Throughout this proof we write $G=N\rtimes H$ where $N$ is the Frobenius kernel, and $H$ is a Frobenius complement
(and point-stabilizer). Suppose that $G$ is binary.  Let $a$ and $b$ be distinct non-trivial elements of $N$. We claim that the binary condition on triples implies that
 \[
  HH^a\cap HH^b=H.
 \]
To see this, assume $H$ stabilizes $\alpha \in \Omega$. Let $\beta \in \alpha H^a \cap \alpha H^b$. Then the tuples 
$(\alpha, \alpha^a,\alpha^b) \stb{2} (\beta, \alpha^a,\alpha^b)$. As $G$ is binary, there exists $g \in G$ mapping the
first tuple to the second. 
Then $g \in H^a \cap H^b$ and
since the action is Frobenius, $g = 1$ and hence $\alpha = \beta$. So $\alpha H^a\cap \alpha H^b=\alpha$, and considering the isomorphic action 
on cosets of $H$ yields $HH^a\cap HH^b=H$.  

We will show that if $|H|>2$, then this equality cannot hold. Suppose, then, that $h_1,h_2,h_3,h_4\in H$ are such that
\begin{equation}\label{e: p}
 h_1a^{-1}h_2a = h_3b^{-1}h_4b.
\end{equation}
Observe first that if the element represented by the two sides of this equation is equal to an element of $H$, then  $a^{-1}h_2 a$ must also be an element of $H$ and so $h_2$ is equal to $1$, as is $h_4$, and in addition $h_1=h_3$.

Thus it suffices to find a solution to \eqref{e: p} for which $h_1\neq h_3$. To do this we start by rearranging to obtain that
\[
 h_3^{-1}h_1a^{-1}h_2ah_4^{-1} = b^{-1}h_4bh_4^{-1}
\]
and observe that the right-hand side lies in $N$. Thus the left-hand side lies in $N$. Doing some rearranging  we find that the left-hand side can be rewritten as
\[
 (h_3^{-1}h_1a^{-1}h_1^{-1}h_3)(h_3^{-1}h_1h_2ah_2^{-1}h_1^{-1}h_3)(h_3^{-1}h_1h_2h_4^{-1}).
\]
Since the first two bracketed terms lie in $N$, we conclude that $h_3^{-1}h_1h_2h_4^{-1}=1$. This allows us to replace $h_4$ in \eqref{e: p} to get
\[
 h_1a^{-1}h_2a = h_3b^{-1}h_3^{-1}h_1h_2b,
\]
which we rearrange one last time to obtain
\begin{equation}\label{e: q}
 (a^{-1})(h_2ah_2^{-1})(h_2b^{-1}h_2^{-1}) = h_1^{-1}h_3b^{-1}h_3^{-1}h_1
\end{equation}
where, again, we have put terms that lie in $N$ in brackets.

Now, for $h\in H\setminus\{1\}$, we define
\begin{align*}
 \phi_{h}:& N\to N\\
 &n\mapsto n^{-1}hnh^{-1}.
\end{align*}
We claim that this map is a bijection. We need only show injectivity: suppose that $n_1, n_2\in N$ with
\[
 n_1^{-1}hn_1h^{-1} = n_2^{-1}hn_2h^{-1}.
\]
Then $n_2n_1^{-1} h n_1n_2^{-1} = h$ and hence $n_1n_2^{-1}$ centralizes $h$. Since we have a Frobenius action, we obtain that $n_1=n_2$, as required.

Now fix $b\in N$ and $h_2\in H\setminus\{1\}$ and consider \eqref{e: q}. The first two bracketed terms correspond to $\phi_{h_2}(a)$ and the surjectivity of the function $\phi_{h_2}$ implies that the left-hand side of \eqref{e: q} ranges over all values of $N$ as $a$ varies across $N$. Recall, though, that we require that $a\neq b$: this restriction tells us that the left-hand side equals all but one of the elements of $N$, as $a$ varies.

On the other hand if $H$ has orbits of size at least $3$, we obtain that \eqref{e: q} has a solution in which $h_1\neq h_3$. We are done. 
\end{proof}

\begin{lem}\label{l: frobenius cyclic kernel}
 Let $F\lhd G \leq \Sym(\Omega)$ with $F$ having an orbit $\Lambda\subseteq\Omega$ on which it acts as a Frobenius group. (As usual, $F^\Lambda$ is the permutation group induced by the action of $F$ on $\Lambda$.) Write $F^\Lambda=T\rtimes C$, where $T$ is the Frobenius kernel, and $C$ is a Frobenius complement. If $T$ is cyclic, and $C$ contains an element $x$ of order strictly greater than $2$, then $G$ is not binary.
\end{lem}
\begin{proof}
Let $\alpha \in \Lambda$. Since $\Lambda$ is a block of imprimitivity for $G$, the group $G_\alpha$ must preserve $\Lambda$ set-wise. Observe that $G_\Lambda$ normalizes $F$, because $F\unlhd G$. In particular, $F^\Lambda\unlhd G^\Lambda$.  Since the non-identity elements of $T$ are precisely those elements of $F^\Lambda$ that are fixed-point-free, $G^\Lambda$ also normalizes $T$.  Thus $T$ is a regular normal subgroup of $G^\Lambda$. As $T$  acts regularly on $\Lambda$, from the Frattini argument we obtain $G^\Lambda=T\rtimes G^\Lambda_\alpha$.

We can, therefore, identify $T$ with $\Lambda$ in such a way that the action of $G^\Lambda_\alpha$ on $\Lambda$ is permutation isomorphic to the conjugation-action of $G^\Lambda_\alpha$ on $T$. To see this, define
\begin{align*}
 \theta:&T \to \Lambda\\
&y\mapsto \alpha^y,
\end{align*}
and observe that, for $y\in T$ and $g\in G_\alpha$,
\[
 \theta(y^g) = \alpha^{(y^g)} = \alpha^{g^{-1}yg} = \alpha^{yg}= (\alpha^y)^g = (\theta(y))^g.
\]

With this set-up, we write $n=|T|$ and we let $y$ be a generator of $T$. We will construct (for the action of $G$) a $2$-subtuple complete pair of the form
 \begin{equation}\label{e: pair}
  \bigg((1,y,y^a), \,\, (1,y,y^b)\bigg).
 \end{equation}
We must choose $a$ and $b$ appropriately. Let $x\in C$ having order strictly greater than $2$. First, let $k\in \mathbb{Z}^+$ be such that $y^x=y^k$; note that $\gcd(k,n)=1$, and so $k$ is invertible modulo $n$. Then we set $a=\frac{1+k}{k}\in {\mathbb Z}_n$ and set $b={1+k}$. Now observe that
\begin{align*}
 (1,y)^{\rm id} &= (1,y); \\
 (1,y^a)^x &= (1,y^{(k+1)/k})^x = (1,y^{k+1}) = (1,y^b); \\
 (y,y^a)^{y^{-1}x^2 y} &= (y,y^{(k+1)/k})^{y^{-1}x^2 y} = (1,y^{1/k})^{x^2t} = (1, y^k)^y = (y,y^{k+1}) = (y,y^b).
\end{align*}
We see immediately that the pair \eqref{e: pair} is $2$-subtuple complete.
 
Note on the other hand that, provided $a\neq b$, this pair cannot be $3$-subtuple complete: suppose that an element $g\in G$ sends the first triple in \eqref{e: pair} to the second. Then $g$ fixes $1$ and, as we saw above, this means that the action of $g$ on $\Lambda$ is isomorphic to the action of $g$ by conjugation on $T$. Since $y^g=y$, we conclude that, if $(y^a)^g=y^b$, then we must have $a=b$ modulo $n$. But now observe that
\[
 a= b \Longleftrightarrow \frac{1+k}{k}= 1+k \Longleftrightarrow k^2= 1.
\]
Since we chose $x$ to have order strictly greater than $2$, we see that $k^2\neq 1$, and we conclude that \eqref{e: pair} is a pair which is 2-subtuple complete but not 3-subtuple complete. The result follows.
\end{proof}

\begin{lem}\label{l: frobenius subgroup}
 Let $F=T\rtimes C\leq G\leq \Sym(\Omega)$ with $C$ acting by conjugation fixed-point-freely on $T$. Suppose there exists $\alpha\in \Omega$  such that $F_\alpha=C$, and let $\Lambda$ be the orbit of $\alpha$ under $F$. Define
 \[
  m:=\min\{ |G_{\gamma_1,\gamma_2}| \mid \gamma_1,\gamma_2 \textrm{ distinct elements of } \Lambda \}.
 \]
If $\left\lceil\frac{(|C|-1)(|C|-2)}{|\Lambda|-2}\right\rceil\geq m$, then $G$ is not binary. In particular, if $|G:F|\leq \left\lceil\frac{(|C|-1)(|C|-2)}{|\Lambda|-2}\right\rceil$, then $G$ is not binary.
\end{lem}
\begin{proof}
Observe that $F$ acts as a Frobenius group on $\Lambda$, where $T$ is the Frobenius kernel, and $C$ is a Frobenius complement. It is useful to observe that the regularity of $T$ on $\Lambda$ implies that, for every $c\in C$ and for every $\beta\in\Lambda$, there exists a unique $x\in T$ such that $\beta^{xc}=\beta$. 
%An easy counting argument implies that the choice of $t$ is unique.
 
We study triples of the form 
\begin{equation}\label{eq: triple}
\bigg((\alpha,\beta,\gamma), \,\, (\alpha, \beta, \delta)\bigg),
\end{equation} for $\alpha,\beta,\gamma,\delta\in\Lambda$. We make the following claim:

\textbf{Claim}: for any distinct pair of elements $(\alpha, \beta)$, there are at least $(|C|-1)(|C|-2)$ choices for $(\gamma, \delta)$ such that the set $\{\alpha,\beta,\gamma,\delta\}$ has size $4$, and the pair \eqref{eq: triple} is $2$-subtuple complete.

\textbf{Proof of claim}: First we consider the set of pairs of distinct non-trivial elements in $C$, i.e.
\[
 C^{(2)} := \{(c_1, c_2) \mid c_1, c_2\in C\backslash\{1\}, c_1\neq c_2\}.
\]
Now we construct a function $\phi: C^{(2)} \to \Omega^2$ as follows. For $(c_1,c_2)\in C^{(2)}$, we let $t_1$ be the unique non-trivial element of $T$ such that $t_1c_1\in G_\beta$. Now, since $c_1\neq c_2$, we can define $\gamma$ to be the unique point in $\Lambda$ fixed by $t_1c_1c_2^{-1}$. Observe that $\gamma$ is distinct from both $\alpha$ and $\beta$. 

Next, we see that
\[
 \gamma^{t_1c_1c_2^{-1}}=\gamma \Longleftrightarrow \gamma^{t_1c_1}=\gamma^{c_2}.
\]
We define $\delta:=\gamma^{c_2}$, and we set $\phi(c_1,c_2)=(\gamma, \delta)$. An easy argument shows that $\delta$ is distinct from all of $\alpha,\beta$ and $\gamma$. Furthermore we claim that, with these definitions the pair \eqref{eq: triple} is $2$-subtuple complete. Indeed, observe that
\[
 (\alpha, \beta)^1=(\alpha, \beta), \,\, (\alpha, \gamma)^{c_2}=(\alpha, \delta)\textrm{ and } (\beta, \gamma)^{t_1c_1}=(\beta, \delta).
\]
Thus every element $(\gamma, \delta)$ in the image of $\phi$ gives rise to a $2$-subtuple complete pair as
in~\eqref{eq: triple}. Since the domain of $\phi$, $C^{(2)}$ has order $(|C|-1)(|C|-2)$, the claim will follow if we prove that $\phi$ is one-to-one.

Suppose, then, that $\phi(c_1,c_2)=(\gamma, \delta)=\phi(c_1', c_2')$. Let $t_1$ (resp. $t_1'$) be the unique element of $T$ such that $t_1c_1$ (resp. $t_1'c_1'$) is in $G_\beta$. Then $t_1c_1c_2^{-1}$ and $t_1'c_1'(c_2')^{-1}$ fix $\gamma$. What is more $\gamma^{c_2} = \gamma^{c_2'}=\delta$ and so $c_2'c_2^{-1}$ fixes $\gamma$. However $c_2, c_2'\in C=F_\alpha$ and so $c_2'c_2^{-1}$ fixes two points of $\Lambda$. We conclude that $c_2=c_2'$. But now, write $h_1:=t_1c_1$ and $h_1':=t_1'c_1'$; observe that $h_1,h_1'\in G_\beta$ and $\gamma^{h_1}=\gamma^{h_1'}$. As before we conclude that $h_1'h_1^{-1}$ fixes $\beta$ and $\gamma$, and so $h_1=h_1'$. Then $t_1c_1=t_1'c_1'$ and so $t_1^{-1}t_1' = c_1'c_1^{-1}$; since $T\cap C=\{1\}$, this gives $c_1=c_1'$, as required.

The claim and the pigeon-hole principle imply that there exists some $\gamma\in\Lambda\setminus\{\alpha,\beta\}$ for which there are
$k:=\left\lceil\frac{(|C|-1)(|C|-2)}{|\Lambda|-2}\right\rceil$ choices for $\delta$ such that all pairs of the form \eqref{eq: triple} are 2-subtuple complete; call these elements $\delta_1,\dots, \delta_k$. If $G$ is binary, then all of these pairs are 3-subtuple complete and we conclude that the set $\{\gamma, \delta_1,\dots, \delta_k\}$ is a subset of an orbit of $G_{\alpha, \beta}$. But this is only possible if $k+1\leq m$, and the result follows.
\end{proof}

\section{Methods: On computation}\label{s: computation}

We will use \magma very frequently in what follows to verify that certain actions are not binary. The methods we use to do this are largely drawn from \cite{dgs_binary}. We give a brief summary of some of the key methods here. In what follows $G$ acts transitively on the set $\Omega$, and $M$ is the stabilizer of a point. 

\textbf{Test 1: Using the permutation character.} Given $\ell\in\mathbb{N}\setminus\{0\}$, we denote by $\Omega^{(\ell)}$ the subset of the Cartesian product $\Omega^\ell$ consisting of the $\ell$-tuples $(\omega_1,\ldots,\omega_\ell)$ with $\omega_i\ne \omega_j$, for every two distinct elements $i,j\in \{1,\ldots,\ell\}$. We denote by $r_\ell(G)$  the number of orbits of $G$ on $\Omega^{(\ell)}$. The next result is Lemma 2.7 of \cite{dgs_binary}.

\begin{lem}\label{l: characters}
If $G$ is transitive and binary, then $r_\ell(G)\le r_2(G)^{\ell(\ell-1)/2}$ for each $\ell\in\mathbb{N}$.
\end{lem}

Let $\pi:G\to\mathbb{N}$ be the permutation character of $G$, that is, $\pi(g)=\fix_\Omega(g)$ where $\fix_{\Omega}(g)$ is the cardinality of the fixed point set $\Fix_{\Omega}(g):=\{\omega\in \Omega\mid \omega^g=\omega\}$ of $g$. From the Orbit Counting Lemma, we have
\begin{align*}
r_\ell(G)&=\frac{1}{|G|}\sum_{g\in G}\fix_\Omega(g)(\fix_\Omega(g)-1)\cdots (\fix_\Omega-(\ell-1))\\
&=\langle \pi(\pi-1)\cdots (\pi-(\ell-1)),1\rangle_G,
\end{align*}
where $1$ is the principal character of $G$ and $\langle \cdot,\cdot\rangle_G$ is the natural Hermitian product on the space of $\mathbb{C}$-class functions of $G$.

Clearly whenever the permutation character of $G$ is available in \texttt{magma}, we can directly check the inequality in Lemma~\ref{l: characters}, and this is often enough to confirm that a particular action is not binary.

\textbf{Test 2: using Lemma~\ref{l: fedup}.}
By connecting the notion of strong-non-binariness to 2-closure, Lemma~\ref{l: fedup} yields an immediate computational dividend: there are built-in routines in \texttt{magma} to compute the $2$-closure of a permutation group.

Thus if $\Omega$ is small enough, say $|\Omega|\le 10^7$, then we can easily check whether or not the group $G$ is $2$-closed. Thus we can ascertain whether or not $G$ is strongly non-binary.

\textbf{Test 3: a direct analysis.}
The next test we discuss is feasible once again provided $|\Omega|\le 10^7$. It simply tests whether or not $2$-subtuple-completeness implies $3$-subtuple completeness, and the procedure is as follows:

We fix $\alpha\in \Omega$, we compute the orbits of $G_\alpha$ on $\Omega\setminus\{\alpha\}$ and we select a set of representatives $\mathcal{O}$ for these orbits. Then, for each $\beta\in \mathcal{O}$, we compute the orbits of $G_{\alpha}\cap G_{\beta}$ on $\Omega\setminus\{\alpha,\beta\}$ and we select a set of representatives $\mathcal{O}_\beta$. Then, for each $\gamma\in \mathcal{O}_\beta$, we compute $\gamma^{G_\alpha}\cap \gamma^{G_\beta}$. Finally, for each $\gamma'\in \gamma^{G_\alpha}\cap \gamma^{G_\beta}$, we test whether the two triples $(\alpha,\beta,\gamma)$ and $(\alpha,\beta,\gamma')$ are $G$-conjugate. If the answer is ``no'', then $G$ is not binary because by construction $(\alpha,\beta,\gamma)$ and $(\alpha,\beta,\gamma')$ are $2$-subtuple complete. In particular, in this circumstance, we can break all the ``for loops'' and deduce that $G$ is not binary. 

If the answer is ``yes'', for every $\beta,\gamma,\gamma'$, then we cannot deduce that $G$ is binary, but we can keep track of these cases for a deeper analysis. We observe that, if the answer is ``yes'', for every $\beta,\gamma,\gamma'$, then $2$-subtuple completeness implies $3$-subtuple completeness. At this point, we may either use a different method for checking whether the permutation group is genuinely binary or, with a similar method, we can check whether $3$-subtuple completeness implies $4$-subtuple completeness. This test is very expensive in terms of time, therefore before starting this whole procedure, we do a preliminary check: for $10^6$ times, we select  $\beta,\gamma,\gamma'$ as above at random and we check this random triple.

\textbf{Test 4: studying suborbits.} 
Lemma~\ref{l: point stabilizer} implies that if $G$ is binary, then the action of $M$ on any suborbit is also binary. This fact is particularly useful for computation in situations where the group $G$ is very large compared to the group $M$.

In general, our approach is to  demonstrate that there must be some suborbit on which the action of $M$ is not binary. For instance, this would follow in the case where $|\Omega|=|G:M|$ is divisible by some integer $d$, and all non-trivial transitive binary actions of $M$ are also of degree divisible by $d$.

This last approach sometimes fails for just a few possible actions of $M$; in this situation, provided the action of $G$ on $\Omega$ is primitive, the following lemma is often useful.

\begin{lem}[{{\label{l: higman}\cite[Theorem~18.2]{Wielandt}}}]
Suppose that $G$ is a finite primitive subgroup of $\mathrm{Sym}(\Omega)$. Let $\Gamma$ be a non-trivial orbit of $M$. Then, every simple section of $M$ is isomorphic to a section of the group $M^\Gamma$ which $M$ induces on $\Gamma$. In particular, each composition factor of $M$ is isomorphic to a section of $M^\Gamma$.
\end{lem}

This lemma means that when studying possible suborbits of our action we may disregard the actions of $M$ (on a set $\Gamma$ say) where $M$ has a simple section not isomorphic to a section of the group $M^\Gamma$. If the resulting set of actions are all not binary, then we can conclude that the action of $G$ on $\Omega$ is also not binary. The method is summarised in Lemma~3.1 of \cite{dgs_binary}:

\begin{lem}\label{l: alot}
Let $G$ be a primitive group on a set $\Omega$, let $\alpha$ be a point of $\Omega$, let $M$ be the stabilizer of $\alpha$ in $G$ and let $d$ be an integer with $d\ge 2$. Suppose $M\ne 1$ and, for each transitive action of $M$ on a set $\Lambda$ satisfying:
\begin{enumerate}
\item $|\Lambda|>1$, and 
\item every composition factor of $M$ is isomorphic to some section of $M^\Lambda$, and
\item either $M_{(\Lambda)}=1$ or, given $\lambda\in \Lambda$, the stabilizer $M_\lambda$ has a normal subgroup $N$ with $N\ne M_{(\Lambda)}$ and $N\cong M_{(\Lambda)}$, and
\item $M$ is binary in its action on $\Lambda$,
\end{enumerate}
we have that $d$ divides $|\Lambda|$. Then either $d$ divides $|\Omega|-1$ or $G$ is not binary.
\end{lem}

\textbf{Test 5: special primes.} We have turned Lemmas~\ref{l: M2} and~\ref{l: added} into a routine in \texttt{magma}. Both of these lemmas are rather convenient from a computational point of view because they do not require us to construct the permutation representation of $G$ on $(G:M)$. For example, the only critical step in the routine for Lemmas~\ref{l: M2} and Lemma~\ref{l: added} is the construction of the centraliser in $G$ of an element $g$ in $M$ of prime order $p$. There is a stardard built-in command in \texttt{magma} for constructing centralizers. Most often than not, this command is sufficient for our computations. However, for very large groups, where it is computationally out of reach to use a general command for computing centralizers, we have constructed  $C_G(g)$  with {\it ad hoc} methods exploiting the subgroup structure of the group $G$ under consideration.

\smallskip
\textbf{Test 6: $M$ very small.} This method draws on the following lemma.

\begin{lem}[{{\cite[Lemma~$2.5$]{gs_binary}}}]\label{l: auxiliary}
Let  $\omega_0,\omega_1,\omega_2\in \Omega$ with $G_{\omega_0}\cap G_{\omega_1}=1$. Suppose there exists $g\in G_{\omega_0}\cap G_{\omega_2}G_{\omega_1}$ with $g\notin G_{\omega_2}$. Then the two triples $(\omega_0,\omega_1,\omega_2)$ and $(\omega_0,\omega_1,\omega_2^g)$ are $2$-subtuple complete but are not $3$-subtuple complete. In particular, $G$ is not binary.
\end{lem}
This method is particularly useful when $M$ ($G_{\omega_0}$ in Lemma~\ref{l: auxiliary}) is small compared to $G$ because in this case it is more likely that $G_{\omega_0}\cap G_{\omega_1}=1$, for some $\omega_1$. This method also has the benefit that it does not require us to construct the permutation representation of $G$ on $(G:M)$, and that all the computations are performed locally. Since this method is designed to deal with the case that $(G:M)$ is large compared to $M$, we do not exhaustively check all triples $\omega_0,\omega_1,\omega_2\in (G:M)$. In practice, we let $\omega_0:=M$, we generate at random $g_1,g_2\in G$, we let $G_{\omega_1}:=M^{g_1}$ and $G_{\omega_2}:=M^{g_2}$ and we check whether Lemma~\ref{l: auxiliary} applies to $\omega_0,\omega_1,\omega_2$. We repeat this routine $10^5$ times and if at some point we find a triple satisfying Lemma~\ref{l: auxiliary}, then $G$ acting on $(G:M)$ is not binary and we stop the routine. If, after the $10^5$ trials, we have not found any triple satisfying Lemma~\ref{l: auxiliary}, then we turn to a different method.

\chapter{Preliminary results for groups of Lie type}\label{ch: prelim}

In this chapter we collect a number of results that will be needed when we come to prove Theorem~\ref{t: main}. All of these results involve the finite groups of Lie type, so let us first establish the notation that we will use in this chapter and those that follow.

Our notation for the classical groups is standard and is consistent with, for instance, \cite[Table 2.1.B]{kl}. We write, for example, $\SOr_n^+(q)$ to mean a group of special isometries associated with a $+$-type quadratic form on an $n$-dimensional vector space over the finite field $\Fq$ having $q$ elements, and we write $\PSO_n^+(q)$ for the projective version of the same. We write $\SOr_n^\pm(q)$ or $\SOr^\varepsilon_n(q)$ if we wish to allow the quadratic form to have either $+$ or $-$ type.

We shall also use the general notation $\Cl_n(q)$ to denote a quasisimple classical group with natural module of dimension $n$ over the field $\Fq$ (over $\F_{q^2}$ for unitary groups).

Our Lie notation is also standard: we write $A_n(q)$, $B_n(q)$, $C_n(q)$, and so on, for quasisimple groups of Lie type associated with Dynkin diagrams of type $A_n, B_n, C_n,\dots$; similarly we write ${^2\!A_n}(q)$, ${^2\!B_2}(q)$, and so on, for twisted versions of the same. Note that the Lie notation does not specify the group up to isomorphism in all cases. For instance, $A_n(q)$ can stand for both $\SL_{n+1}(q)$ and $\PSL_{n+1}(q)$.) We write $A_n^-(q), D_n^-(q)$ and $E_6^-(q)$ as alternative notation for ${^2\!A_n}(q), {^2\!D_n}(q)$ and ${^2\!E_6}(q)$ respectively, and we write $A_n^\pm(q), D_n^\pm(q), E_6^\pm(q)$ or $A_n^\varepsilon(q)$, $D_n^\varepsilon(q)$, $E_6^\varepsilon(q)$ if we wish to consider both the twisted and untwisted version at the same time.

The results collected here are of six kinds:

\begin{enumerate}
 \item \textbf{Results concerning alternating sections}: We consider a simple group of Lie type, $G,$ and we specify for which values of $r$ the alternating group, $\Alt(r)$, is a section of $G$. These results will be used later, in conjunction with Definition~\ref{d: beautiful}, when we study the primitive actions of $G$ -- one frequently-used method for showing that these actions are not binary will be to show that they exhibit a beautiful subset.
 \item \textbf{Stabilizer results}: We consider a group $G$, and we consider all faithful transitive actions of $G$ in which the stabilizer of a point, $H$, contains a particular element $g$. We will prove that, for an appropriately chosen $G$ and $g$, such an action is always not binary. We call these ``stabilizer results'' because these lemmas will  typically be applied in later chapters in contexts where  $G$ is a point-stabilizer and we are seeking to use Lemma~\ref{l: point stabilizer}. These applications motivate the choices of $G$ which we consider in this section. 
 \item \textbf{Odd degree results}: We consider a group $M$, normally a small group of Lie type, and we use \magma to show that all of the transitive actions of odd degree of $M$ are not binary. Although it is not about groups of Lie type, we also include one result -- Lemma~\ref{l: sporadic small-odd} -- which does the same thing for the sporadic groups.
 \item \textbf{Centralizer results}: We will present a number of results giving lower bounds for the size of a centralizer of a non-trivial element in a simple group of Lie type.
 \item \textbf{Automorphism results}: We present a well-known result classifying the outer automorphisms of prime ofrder of finite groups of Lie type.
 \item \textbf{Fusion and factorization results}: All these results will be used in conjunction with Lemma~\ref{aff} to prove the existence of beautiful subsets (Definition~\ref{d: beautiful}).
\end{enumerate}

We will use the stabilizer results in two ways when it comes to the proof of Theorem~\ref{t: main}. For the proof we study an almost simple group $G$ acting on the cosets of a maximal subgroup $M$. Now, the first use of our stabilizer results is direct: if $M$ contains the element $g$, then we immediately know that the action is not binary and we are done.

The second use is slightly less direct. In this case, we wish to apply our stabilizer results to the group $M$, rather than the group $G$: so we pick a distinguished element $g\in M$ and appeal to our stabilizer results to assert that if $H$ is any core-free subgroup of $M$ that contains $g$, then the action of $M$ on $(M:H)$ is not binary. Next we use our centralizer results, to show that, in general $|C_M(g)|$ is smaller than the smallest centralizer in $G$. We conclude that there exists $x\in C_G(g)\setminus C_M(g)$. Now $M\cap M^x$ is a core-free subgroup of $M$ that contains $g$. We conclude that the action of $M$ on $(M:M\cap M^x)$ is not binary. Then Lemma~\ref{l: point stabilizer} implies that the action of $G$ on $(G:M)$ is not binary.

This second method explains the selection of groups under consideration for our stabilizer results: for instance the group $G$ appearing in Lemma~\ref{l: pgl} is studied because such a group is maximal in $E_8(q)$.

The second method also applies to the odd degree results: if we are studying the action of a group $G$ on the cosets of a subgroup $M$ and we know (a) that $|G:M|$ is even, (b) that all odd-degree actions of $M$ are not binary, then Lemma~\ref{l: point stabilizer} implies that the action of $G$ on $(G:M)$ is not binary.

\section{Results on alternating sections}

Let $G$ be a simple group of Lie type. We wish to know for which values of $r$ the alternating group, $\Alt(r)$, is a section of $G$. 

We first consider classical groups. 

\begin{lem}\label{l: alt sections classical}
Let $\mathrm{Cl}_n(q)$ be a simple classical group with natural module of dimension $n$  and $p$ is a prime number. If $\mathrm{Cl}_n(q)$ has a section isomorphic to the alternating group $\Alt(r)$, then
\begin{equation}\label{rpalt}
n \ge R_p(\Alt(r)),
\end{equation}
where $R_p(\Alt(r))$ denotes the smallest dimension of a non-trivial projective representation of $\Alt(r)$ over a field of characteristic $p$. In particular, for $r\ge 9$, we have 
\[
R_p(\Alt(r)) =r-1-\d,\]
where 
\[\d=
\begin{cases}
1,&\textrm{if }p\mid r,\\
0,&\textrm{otherwise.}
\end{cases}
\] 
For $5\le r\le 8$, the values for  $R_p(\Alt(r))$ are as in Table~$\ref{t: smallvalues}$.

\begin{table}[!ht]
\centering
\begin{tabular}{|c|c|c|c|c|}\hline
$r$&$R_2(\Alt(r))$ & $R_3(\Alt(r))$  & $R_5(\Alt(r))$ &$R_p(\Alt(r))$, $p\ge 7$ \\
\hline
5&2&2&2&2\\
6&3&2&3&3\\
7&4&4&3&4 \\
8&4&7&7&7\\\hline
\end{tabular}
\caption{Values for $R_p(\Alt(r))$, with $5\le r\le 8$}\label{t: smallvalues}
\end{table}
\end{lem}
\begin{proof}
The inequality $R_p(\Alt(r)) =r-1-\d$ follows from~\cite[Proposition 4.1]{FT}. 
The values of $R_p(\Alt(r))$ are well-known (see \cite[Proposition~5.3.7]{kl}). 
\end{proof}

If $G$ is exceptional, then the following lemma gives the result that we need. (Here,  $\delta_{x,y}$ is the usual Kronecker delta.)
 
\begin{lem}\label{altsec}
Let $G=G(q)$ be a finite simple group of exceptional Lie type as in the table below, where $q=p^a$, $a\ge 1$ and $p$ is a prime number. If $\Alt(r)$ is a section of $G$, then $r\le N_G$, where $N_G$ is as in the table below.
\[
\begin{array}{c|cccccc}
G & E_8(q) & E_7(q) & E_6^\e(q) & F_4(q) & G_2(q),\,^3\!D_4(q)  & {^2\!F_4}(q)\\
\hline
N_G & 17+\d_{p,3} & 13+\d_{p,7} & 11+\d_{p,2}+\d_{p,5} & 10+\d_{p,11} & 6+\d_{p,5} & 6
\end{array}
\]
\end{lem}

\begin{proof}
Fix $r\ge 5$, and let $K \triangleleft H \le G$ with $H/K \cong \Alt(r)$, and $|H|$ minimal. Choose a minimal subfield $\F_{q_0} \subseteq \F_q$ such that $H \le G(q_0)$, and a maximal subgroup $M$ of $G(q_0)$ such that $H \le M$.

%Consider the cases where $G(q_0) =  E_8(q_0)$, $E_7(q_0)$, $E_6^\e(q_0)$, or $F_4(q_0)$ or $G_2(q_0)$. We apply Theorem \ref{MAXSUB} to the maximal subgroup $M$. 

Consider first  $G(q_0) = {^2\!F_4(q_0)}$. The maximal subgroups are given by \cite{malle}, 
from which it follows that $\Alt(r)$ is a section of one of the groups $\Sp_4(q_0)$ or $\PSU_3(q_0)$. 
Hence by Lemma~\ref{l: alt sections classical} we have $4 \ge R_2(\Alt(r))$, forcing $r\le 8$. As $\Alt(7)$ is not a section of 
$\Sp_4(q_0)$ or $\PSU_3(q_0)$ (see, for instance,~\cite{bhr}), we in fact have $r\le 6$, as in the conclusion. 

The cases where $G(q_0) = G_2(q_0)$ or ${^3\!D_4}(q_0)$ are dealt with similarly, using \cite{cooperstein, kleidman, K} for the lists of maximal subgroups.

Now consider the remaining cases, where $G$ is of type $E_8$, $E_7$, $E_6^\e$ or $F_4$. 
By the minimality of $H$ we have $K\le \Phi(H)$, where $\Phi(H)$ is the Frattini subgroup of $H$. So, $K$ is nilpotent. 

Suppose $Z(H) \ne 1$, and let $1 \ne x \in Z(H)$. Then $H \le C_G(x)$, which is contained in a parabolic or a subsystem subgroup, and it follows that $\Alt(r)$ is a section of one of the following subgroups of $G$:
\[
\begin{array}{c|c}
G & \Alt(r) \hbox{ section of one of} \\
\hline
F_4(q) & B_4(q),\,C_3(q) \\
E_6^\e(q) & A_5^\e(q), \,D_5^\e(q) \\
E_7(q) & A_7^\pm (q),\,D_6(q), \,E_6^\pm (q) \\
E_8(q) & D_8(q),\,A_8^\pm (q),\,E_7(q) 
\end{array}
\]
Working down from $F_4(q)$, the bounds in the conclusion now follow using Lemma~\ref{l: alt sections classical}. (Note that the possibilities $r=18$ in $E_8\,(p=2)$ and $r=14$ in $E_7\,(p=2)$ are excluded by the fact that $D_8(2^a)$ (resp. $D_6(2^a)$) does not have a section isomorphic to $\Alt(18)$ (resp. $\Alt(14)$) (see \cite[(5.3.8)]{kl}).

Suppose finally that $Z(H)=1$. If $Z(K)=1$, then $K=1$ (as $K$ is nilpotent), so $H \cong \Alt(r)$, and the conclusion follows from \cite[Table 10.1]{LSei}. So assume $Z(K)\ne 1$. If $p$ divides $|Z(K)|$, then $H$ is contained in a parabolic subgroup of $G$ by \cite{BT}, a case already considered above. Hence we may assume that $Z(K)$ has order divisible by a prime $s$ with $s\ne p$. As $Z(H)=1$, it must be the case that $H/K \cong \Alt(r)$ acts non-trivially on the elementary abelian group $E = \O_1(O_s(Z(K)))$. Say $E \cong (C_s)^\kappa$, of rank $\kappa$. Then $\kappa \ge R_s(\Alt(r))$. On the other hand, \cite{CS} shows that $\kappa \le R+1$, where $R$ is the untwisted Lie rank of $G$. Hence 
\[
R_s(\Alt(r)) \le R+1,
\]
and the bounds for $r$ in the conclusion follow from this. This completes the proof. 
\end{proof}

\section{Stabilizers containing certain elements}\label{a1lem}

In this section we prove results that are (more or less) of the following kind: we suppose that $x$ is an element of a group $G$, and we prove that, if $H$ is any core-free subgroup of $G$ containing $x$, then the action of $G$ on $(G:H)$ is not binary. In the first subsection we consider groups $G$ of a variety of isomorphism types; in subsequent subsections, $G$ will always be almost simple.

\subsection{Some groups that are not almost simple}

\begin{lem}\label{l: pgl}
 Let $S=\PGL_2(q)\times \Sym(5)$ with $q>5$, and suppose that $S\unlhd G$ with $G/S$ solvable. Let $L$ be the normal subgroup in $S$ that is isomorphic to $\PGL_2(q)$, and suppose that $g\in L$ has order $q-1$; let $M$ be a subgroup of $G$ that contains $g$. If the action of $G$ on $(G:M)$ is binary, then $M$ contains $L$.
\end{lem}
\begin{proof}
Assume that the action of $G$ on $(G:M)$ is binary. Notice that the element $g$ normalizes, and acts fixed-point-freely by conjugation upon two unipotent subgroups of $L$ of order $q$; we call these $U_1$ and $U_2$. 

Suppose that $M$ does not contain $U_1$. Since $g\in M$, we have $M\cap U_1=\{1\}$. Now, we define $\Lambda=\{Mu \mid u\in U_1\}$. It is easy to see that $U_1\rtimes \langle g\rangle$ acts 2-transitively on $\Lambda$, which is a subset of $(G:M)$ of size $q$.  Since $G$ is binary on $(G:M)$, the group $G^\Lambda$ is isomorphic to the symmetric group of degree $q$. As $q>5$ and $G/S$ is solvable, by Lemma~\ref{l: alt sections classical}, $G$ has no section isomorphic to $\Alt(q)$, which is a contradiction.

Thus $M$ contains $U_1$ and, by the same reasoning, $U_2$. But now $\langle U_1,U_2,g\rangle=L$ and we are done.
\end{proof}

\begin{lem}\label{l: a1a1}
 Let $S=\PSL_2(q)\times\PSL_2(q)$ with $q\geq 4$ and $q\ne 5$, and suppose that $S=F^*(G)$, where $F^*(G)$ is the generalized Fitting subgroup of $G$. Let $L$ be a subgroup of $S$ isomorphic to $D_{t(q-1)}\times D_{t(q-1)}$ (where $t=(2,q)$ and where $D_{t(q-1)}$ denotes the dihedral group of order $t(q-1)$), and let $M$ be a subgroup of $G$ that contains $L$. If the action of $G$ on $(G:M)$ is binary, then $M\geq S$.
\end{lem}

\begin{proof}
We write $S=S_1\times S_2$ and $L=L_1\times L_2$, where $D_{t(q-1)}\cong L_i < S_i \cong \PSL_2(q)$ for $i\in\{1,2\}$. Assume first that $q\notin \{4,7,9,11\}$. 

Suppose, first, that $M\cap S=L$; we must show that the action of $G$ on $(G:M)$ is not binary. Let $H=\langle M, S\rangle = MS$. Lemma~\ref{l: subgroup} implies that it is sufficient to show that the action of $H$ on $(H:M)$ is not binary. Now observe that
\[
 H/S = MS/S \cong M/(M\cap S) = M/L.
\]
Thus $|H:M|=|S:L|$ and we can identify $(H:M)$ with the set of conjugates of $L$ in $S$, by using the map
\[
 (H:M)\to \{L^s \mid s\in S\}, \quad Mg\mapsto L^g.
\]
Now define
\[
 \Gamma=\{ L_1\times L_2^g \mid g\in S_2\}.
\]
The intersection of the elements of $\Gamma$ is $L_1$ and so $H_\Gamma\le N_H(L_1)$. Since the reverse inclusion is also true, we deduce 
\begin{equation}\label{scared}
H_\Gamma=N_H(L_1).
\end{equation}
Observe that the action of $H_\Gamma$ on $\Gamma$ is isomorphic to the action of an almost simple group with socle $S_2=\PSL_2(q)$ on the cosets of a subgroup $M_2$ for which $M_2\cap S_2\cong D_{t(q-1)}$. When $q\notin \{ 4, 7, 9, 11\}$, the action of $H_\Gamma$ on $\Gamma$ is primitive by~\cite[Table~$8.1$]{bhr} and hence the main theorem of \cite{dgs_binary} implies that this action is not binary.  Thus there is an integer $k\geq 3$ and two $k$-tuples $I,J\in \Gamma^k$ that are $2$-subtuple complete but not $k$-subtuple complete with respect to the action of $H^\Gamma$. Using~\eqref{scared}, one can see that any $h\in H$ for which $I^h=J$ must satisfy $h\in H_\Gamma$, and so $I,J$ are not $k$-subtuple complete with respect to the action of $H$. Thus the action of $H$ on $(H:M)$ is not binary, and so the action of $G$ on $(G:M)$ is not binary, as required.

We conclude that $L$ is a proper subgroup of $M\cap S$. We may assume, without loss of generality, that $M\cap S$ contains $S_1$ but not $S_2$. Then the action of $G$ on $(G:M)$, modulo the kernel, is isomorphic to the action of an almost simple group with socle $S_2=\PSL_2(q)$ on the cosets of a maximal subgroup $M_2$ for which $M_2\cap S_2\cong D_{t(q-1)}$. Once again the main theorem of \cite{dgs_binary} implies that this action is not binary.

The only remaining possibility is that $M\geq S$, as required.

Assume now that $q\in \{4,7,9,11\}$. With the help of {\tt magma}, we have constructed all the groups $G$ with $F^*(G)=S\cong \PSL_2(q)\times \PSL_2(q)$ and all the subgroups $H$ of $G$ containing $D_{t(q-1)}\times D_{t(q-1)}$. Then we have verified, by witnessing non-binary triples, that the action of $G$ on $(G:H)$ is binary only when $S\le H$.
\end{proof}

\subsection{Classical groups}

\begin{lem}\label{l: psl2 element}
 Let $G$ be almost simple with socle $S=\PSL_2(q)$, and let $x$ be the projective image of an element $\tilde{x}$ as given in Table~$\ref{t: as stab}$ line~$1$. Let $M<G$ be core-free with $x\in M\cap S$. Then, provided $q\notin\{4,5\}$, the action of $G$ on $(G:M)$ is not binary.
 \end{lem}
\begin{proof}
For $q\in\{7,9,11,13,27\}$, we confirm the result using {\tt magma}. In particular, for the rest of the proof we suppose $q=8$ or $q>13$ with $q\ne 27$.

Let $d=(2,q-1)$.  The stated conditions imply that $M\cap S \in\{ C_{(q-1)/d}, D_{2(q-1)/d}, B\}$, where $B$ is a Borel subgroup of $S$. We set $q=p^f$, for a prime $p$ and positive integer $f$, and we consider the three cases separately.

\smallskip

\noindent\textsc{Case 1:} Suppose that $M\cap S=C_{(q-1)/d}$. In particular, $T:=M\cap S$ is a split torus in $\PSL_2(q)$. Since distinct split tori in $\PGL_2(q)$ intersect trivially, we conclude $|G_{\alpha, \beta}|\leq f$. On the other hand, let $B=U\rtimes T$, a Borel subgroup of $S$, and observe that $B$ acts as a Frobenius group on the set $\Lambda=\{Mu\mid u\in U\}\subset (G:M)$. Clearly $|\Lambda|=q$; if $p=2$, then $d=1$, $\Lambda$ is a beautiful subset and we are done. Suppose, then, that $q$ is odd; Lemma~\ref{l: frobenius subgroup} implies that, if $G$ is binary, then
 \[
  \left\lceil\frac{(\frac{q-1}{2}-1)(\frac{q-1}{2}-2)}{q}\right\rceil < f.
 \]
This implies $\lceil q/4-2-15/q\rceil <f$. It is easy to verify that, when $q>13$, $$\lceil q/4-2-15/(4q)\rceil=\lceil q/4-2\rceil\ge (q-7)/4$$ and hence, in particular, $q-7<4f$.  However, this inequality is never satisfied when $q> 13$.  

\smallskip

\noindent\textsc{Case 2:} Suppose that $M\cap S=D_{2(q-1)/d}$. The analysis of the previous case still applies: for $q$ even, we obtain a beautiful subset again and are done; for $q$ odd, we proceed as before except that this time $|G_{\alpha, \beta}|\leq 2f$, which implies that $q-7<8f$. However, for $q>13$, this inequality is satisfied only when $q=27$, but we are excluding this case here.  

\smallskip

\noindent\textsc{Case 3:} Suppose that $M\cap S=B$. Let $K=\langle M, S\rangle$. Then $(K:M)$ is a set of size $q+1$ that is stabilized by $K$ and on which $K$ acts $2$-transitively. By Lemma~\ref{l: alt sections classical}, any alternating section, $\Alt(r)$, of $\PGammaL_2(q)$ has $r\leq 6$, hence $(K:M)$ is a beautiful subset of $(G:M)$, and we are done.~\qedhere
\end{proof}

We shall also need the following variant of Lemma \ref{l: psl2 element}.

\begin{lem}\label{l: psl2var}
 Let $G$ be almost simple with socle $S=\PSL_2(q^2)$, and let $x$ be the projective image of the diagonal matrix ${\rm diag}(a,a^{-1})$, where $a \in \Fqt$ has order $(q-1,2)(q-1)$. Let $M<G$ be core-free with $x\in M\cap S$. Then, provided $q \ge 7$, the action of $G$ on $(G:M)$ is not binary.
 \end{lem}

\begin{proof}
For $\lambda \in \Fqt$, define subgroups $U_\l^\pm$ of $S$ by 
\[
U_\l^+ = \{I+\l tE_{12} : t \in \Fq\},\;\; U_\l^- = \{I+\l tE_{21} : t \in \Fq\},
\]
where as usual $E_{ij}$ denotes the matrix with $ij$-entry 1 and 0 elsewhere. Then $T = \la x \ra$ normalizes $U_\l^\pm$ and acts transitively on $U_\l^\pm \setminus \{1\}$. Since $\la U_\l^\pm : \l \in \Fqt\rangle = S$, there exists $\l$ and $\e = \pm$ such that $U = U_\l^\e \not \le M$. Then $UT$ acts 2-transitively on the set 
$\Lambda=\{Mu\mid u\in U\}\subset (G:M)$ of size $q$, and since $\Alt(q)$ is not a section of $G$ for $q\ge 7$, it follows that the action of $G$ on $(G:M))$ is not binary for $q\ge 7$. 
\end{proof}

\begin{lem}\label{l: psl element}
 Let $G$ contain a subgroup $S\cong\SL_n(q)/Z$, where $Z$ is a central subgroup of $\SL_n(q)$, and such that $n\geq 3$. Let $x\in S$ be the  projection in $S$ of an element $\tilde{x}\in \SL_n(q)$ as given in Table~$\ref{t: as stab}$ lines~$2$ and~$3$. Let $M<G$ with $x\in M$. Then one of the following holds:
 \begin{enumerate}
  \item $G$ contains a section isomorphic to $\Sym(q^{n-2})$ (if $q>2$) or $\Sym(2^{n-1})$ (if $q=2$);
  \item $M$ contains $S$;
  \item the action of $G$ on $(G:M)$ is not binary.
 \end{enumerate}
\end{lem}

\begin{table}[!ht]
\centering
\begin{tabular}{|c|c|c|c|}
\hline
Line&$S/\Zent S$ & $\tilde{x}$ & Conditions  \\
\hline
1&$\PSL_2(q)$ & $\begin{pmatrix}
               a & \\ & a^{-1}
              \end{pmatrix}$ & $a$ of order $q-1$ \\
\hline
2&\begin{tabular}{l} $\PSL_n(q)$\\ $n\geq 3$ \\ $q\geq 3$ \end{tabular} & $\begin{pmatrix}
   1 & & \\ & A & \\ & & a^{-1}
  \end{pmatrix}$ & \begin{tabular}{l}$A\in \GL_{n-2}(q)$\\ $A$ of order $q^{n-2}-1$ \\ $\det(A)=a\in\Fq$\end{tabular}\\
\hline
3&\begin{tabular}{l} $\SL_n(2)$ \\ $n\geq 3$ \end{tabular} & $\begin{pmatrix}
   1 &  \\ & A 
  \end{pmatrix}$ &  \begin{tabular}{l} $A\in \GL_{n-1}(2)$ \\ $A$ of order $2^{n-1}-1$ \end{tabular} \\
  \hline
\end{tabular}
\caption{Auxiliary table for Lemma~\ref{l: psl element}}\label{t: as stab}
\end{table}

\begin{proof}
We assume that none of the three possibilities hold, and we reach a contradiction. In particular, the action of $G$ on $(G:M)$ is binary. Since $S\cong \SL_n(q)/Z$, there exists a surjective group homomorphism $\pi:\SL_n(q)\to S$.

\smallskip

\noindent\textsc{Case 1: $q>2$}.
We observe first that $\langle \tilde{x}\rangle$ normalizes two distinct elementary abelian subgroups of $\SL_n(q)$ of order $q^{n-2}$, namely those having shape
\[
 U_1=\left\{\begin{pmatrix}
               1 & u_1 & \cdots & u_{n-2} & 0 \\
                & 1 & & & \\
                &  & \ddots &  &\\
                & & & 1 &   \\
                & & &  & 1
              \end{pmatrix}\mid u_1,\ldots,u_{n-2}\in \mathbb{F}_q\right\}, 
 U_2=\left\{\begin{pmatrix}
               1 &  &  & &  \\
                u_1 & 1 & & & \\
                \vdots &  & \ddots &  &\\
                u_{n-2} & & & 1 &   \\
                 0 & & &  & 1
              \end{pmatrix}\mid u_1,\ldots,u_{n-2}\in \mathbb{F}_q\right\}.
\]
%(in each case, the group is the set of all matrices obtained by allowing the parameters to range over $\Fq$). 
Observe that $\langle \tilde{x}\rangle$ acts by conjugation fixed-point-freely on each of these two groups. Let us suppose that $\pi(U_1)\not\leq M$. As $\pi(\tilde{x})=x\in M$, the fixed-point-freeness of the action yields $\pi(U_1)\cap M=\{1\}$.

Now let $\Lambda$ be the set of cosets of $M$ corresponding to $M\pi(U_1)$, that is, $\Lambda=\{Mh\mid h \in \pi(U_1)\}$. Then $\Lambda$ is a set of size $q^{n-2}$ on which the group $M_1=\pi(U_1\rtimes \langle \tilde{x}\rangle)$ acts 2-transitively. Since we are assuming that $G$ on $(G:M)$ is binary, $\Lambda$ is not a beautiful subset. Therefore, $G^\Lambda\ge \Sym(q^{n-2})$; however this contradicts the fact that we are assuming that $G$ has no section isomorphic to $\Sym(q^{n-2})$. Thus $\pi(U_1)\le M$.

A similar argument applies to $U_2$. Thus $\pi(U_2)\le M$ and hence $\langle \pi(U_1), \pi(U_2),x\rangle\le M$. Observe that
$$\left\langle \begin{pmatrix}1&u\\0&1\end{pmatrix},\begin{pmatrix}1&0\\u&1\end{pmatrix}\mid u\in \mathbb{F}_q\right\rangle=\SL_2(q).$$ Now, an easy inductive argument on $n$ shows that
\[
\langle U_1,U_2\rangle=\left\{\begin{pmatrix}Z&0\\0&1\end{pmatrix}\mid Z\in\SL_{n-1}(q)\right\}
\]
and hence, from the definition of $\tilde{x}$, we obtain
that  $\langle U_1, U_2, \tilde{x}\rangle$ contains all matrices of the form
\[
 \begin{pmatrix}
  Z & \\ & z^{-1}
 \end{pmatrix},
\]
where $Z\in \GL_{n-1}(q)$ has determinant $z\in\Fq$.  But now we define two elementary abelian subgroups of $\SL_n(q)$ of  order $q^{n-1}$, namely those having shape
\[
 U_3=\left\{\begin{pmatrix}
      1 & & & u_1 \\ & \ddots & & \vdots \\ & & \ddots  & u_{n-1} \\ & & & 1
     \end{pmatrix}\mid u_1,\ldots,u_{n-1}\in \mathbb{F}_q\right\},
U_4=\left\{\begin{pmatrix}
     1 & & & \\
      & \ddots & & \\ & & \ddots & \\ u_1 & \cdots & u_{n-1} & 1
    \end{pmatrix}\mid u_1,\ldots,u_{n-1}\in \mathbb{F}_q\right\}.
\]
Repeating the same argument as before, with $U_1$ and $U_2$ replaced by $U_3$ and $U_4$, we obtain that $M$ contains $\pi(U_3)$ and $\pi(U_4)$. But then $M\geq \langle \pi(U_1),\pi(U_2),\pi(U_3),\pi(U_4),x\rangle =S$, a contradiction, and we are done.

\smallskip

\noindent\textsc{Case 2: $q=2$}. Clearly, in this case, $Z=1$ and we may think of $\pi$ as the identity mapping. We define two elementary abelian subgroups of $\SL_n(2)$ of order $2^{n-1}$, namely those having shape
\[
 U_1=\left\{\begin{pmatrix}
      1 & & & u_1 \\ & \ddots & & \vdots \\ & & \ddots  & u_{n-1} \\ & & & 1
     \end{pmatrix}\mid u_1,\ldots,u_{n-1}\in \mathbb{F}_q\right\},
U_2=\left\{\begin{pmatrix}
     1 & & & \\
      & \ddots & & \\ & & \ddots & \\ u_1 & \cdots & u_{n-1} & 1
    \end{pmatrix}\mid u_1,\ldots,u_{n-1}\in \mathbb{F}_q\right\}.
\]
We suppose that $U_1\not\leq M$. As in the previous proof we use the fact that $\langle x\rangle$ normalizes, and acts fixed-point-freely on, $U_1$. As before, either $G$ contains a section isomorphic to $\Sym(2^{n-1})$ (but this contradicts our hypothesis) or else we obtain a beautiful subset (but this contradicts again our hypothesis). Hence $M$ contains $U_1$ and, by the same argument, $U_2$. Since $\langle U_1, U_2\rangle = \SL_n(2)$, we obtain a contradiction and are done.
\end{proof}

Applying Lemma~\ref{l: psl element} to the case where $G$ is almost simple, and applying \cite[Proposition~5.3.7]{kl} to establish when $G$ may contain the relevant alternating section, we obtain the following result.

\begin{lem}\label{l: psl element 2}
 Let $G$ be almost simple with socle $S=\PSL_n(q)$, and let $x$ be the projective image of an element $\tilde{x}$ as given in Table~$\ref{t: as stab}$. Let $M<G$ be core-free with $x\in M\cap S$. Then, provided $(n,q)\notin\{(2,4),(2,5)\}$ and $(G,M)\not\in \{(\Sym(8),\Alt(7)),(\Sym(8),\Sym(7))\}$, the action of $G$ on $(G:M)$ is not binary.

Moreover, if $(n,q)\notin\{(2,3),(3,3)\}$, then $|C_S(x)|<q^n$. 
\end{lem}
\begin{proof}
Since $M$ is core-free in $G$, we have $S\nleq M$. When $n=2$, the proof follows from Lemma~\ref{l: psl2 element}. When $n\ge 3$, from Lemma~\ref{l: psl element}, if the action of $G$ on $(G:M)$ is binary, then $G$ contains a section isomorphic to $\Sym(q^{n-2})$ (if $q>2$) or $\Sym(2^{n-1})$ (if $q=2$). From this it follows from Lemma~\ref{l: alt sections classical} that $(n,q)\in\{(3,3), (3,4), (3,5),(3,2), (4,2)\}$. For these values of $(n,q)$, we have constructed all the permutation representations under consideration and we have checked that none is binary unless $(n,q)=(4,2)$ and $(G,M)$ is one  for the cases listed in the statement.

When $\tilde{x}$ is as in Line~1 of Table~\ref{t: as stab}, $\langle x\rangle$ is a torus in $\PSL_2(q)$ of cardinality $(q-1)/2$ when $q$ is odd, and $q-1$ when $q$ is even. Thus $|C_S(x)|\le q-1<q$, except when $q=3$. Similarly, using the fact that, if $A\in\GL_k(q)$ has order $q^{k}-1$ (that is, $\langle A\rangle$ is a Singer cycle), then $C_{\GL_k(q)}(A)=\langle A\rangle$, we deduce that $|C_S(x)|=(q^{n-1}-1)(q-1)/(n,q-1)<q^n-1$ when $\tilde{x}$ is as in Line~2 of Table~\ref{t: as stab} and $q\ne 3$, and  $|C_S(x)|=2^{n-1}-1<2^n$ when $\tilde{x}$ is as in Line~3 of Table~\ref{t: as stab}.
\end{proof}

The fact that $|C_S(x)|<q^n$ will be important later on -- in Lemma~\ref{l: psl element}, and in the results that follow, we have tried to pick distinguished elements $x\in S$ for which $C_S(x)$ is relatively small.

For groups with socle $\PSL_4(q)$ we shall also need the following special result.

\begin{lem}\label{l: psl4spec}
Let $G$ be almost simple with socle $S=\PSL_4(q)$, and let $x$ be the projective image of a diagonal matrix $\tilde{x} = {\rm diag}(1,1,a,a^{-1})$, where $a \in \Fq$ has order $q-1$.  Let $M<G$ be core-free with $x\in M\cap S$. Then, provided $q\ge 8$, the action of $G$ on $(G:M)$ is not binary.
\end{lem}

\begin{proof} The proof is very similar to that of Lemma~\ref{l: psl element}. Let $T = \langle x \rangle$, and for $i\ne j$ define $U_{ij} = \{I+\alpha E_{ij} : \alpha \in \Fq\}$, where $E_{ij}$ denotes the matrix with $ij$-entry 1 and 0 elsewhere. Then $T$ acts fixed-point-freely on the groups $U_{ij}$ for $i \in \{1,2\}, j\in \{3,4\}$ or vice versa. Since these subgroups $U_{ij}$ generate $S$, at least one of them is not contained in $M$. Hence we obtain a subset $\Delta$ of size $q$ on which $G_\Delta$ acts 2-transitively. If $\Delta$ is a beautiful subset then $(G,(G:M))$ is not binary. So suppose $\Delta$ is not beautiful.  Then $\Alt(q)$ is a section of $S$, and moreover $\Alt(q-1)$ is a section of $M$. By Lemma \ref{l: alt sections classical}, $\Alt(q)$ is a section of $S$ only if $q\le 8$; moreover, if $q=8$, then $\Alt(7)$ can only be a section of a maximal core-free subgroup $M$ of $G$ if $M$ is a subfield subgroup of type $\PSL_4(2)$ (see \cite[Tables 8.8, 8.9]{bhr}) -- but such a subgroup does not contain the element $x$. Hence if $q\ge 8$ we have a contradiction, and the proof is complete. 
\end{proof}

We now need to prove an analogue of Lemma~\ref{l: psl element} for the other classical groups, albeit subject to some conditions (including lower bounds on $n$). Some of the situations excluded by these conditions are studied in subsequent lemmas. In the statement and proof of the lemma, if $S$ is orthogonal or symplectic, we set $\K=\Fq$; if $S$ is unitary, then we set $\K=\mathbb{F}_{q^2}$. In either case, for a scalar $a\in\K$ we define $\overline{a}:=a^q$; for a matrix $A=(a_{ij})_{i,j}\in \GL_d(\K)$ we write $\overline{A}$ for the matrix $(\overline{a}_{ij})_{i,j}$.

\begin{lem}\label{l: classical element}
Suppose that one of the following holds:
\begin{enumerate}
 \item $G$ contains a subgroup $S\cong\SU_n(q)/Z$ where $Z$ is a central subgroup of $\SU_n(q)$ and $n\geq 5$;
 \item $G$ contains a subgroup $S\cong\Sp_n(q)/Z$ where $Z$ is a central subgroup of $\Sp_n(q)$ and $n\geq 4$;
 \item $G$ contains a subgroup $S\cong\Omega_n^\varepsilon(q)$, $q$ is even and $n\geq 8$;
 \item $G$ contains a subgroup $S\cong\SOr_n^\varepsilon(q)/Z$ where $Z$ is a central subgroup of $\SOr_n^\varepsilon(q)$, $q$ is odd and $n\geq 7$.
 \end{enumerate}
Let $k$ be the Witt index of the associated formed space. If $S\neq \SU_n(q)/Z$ with $n$ even, then we define $j=k,$ otherwise $j=k-1$. We let $\mathcal{B}=\{e_1,\dots, e_j, f_1, \dots, f_j\}\cup Y$ be a  hyperbolic basis; thus $Y$ is a set of linearly independent anisotropic vectors if $S\neq \SU_{2j+2}(q)/Z$, otherwise $Y=\{e_{j+1}, f_{j+1}\}$. We set $y=|Y|\in\{0,1,2\}$.

Let $M<G$ with $x\in M$, where $x$ is the projective image in $S$ of
 \[
  \tilde{x}=\begin{pmatrix}
   1 & & &  & \\ & A & & & \\ & & 1 & & \\ & & & \overline{A^{-T}} & \\ &&&& J_y
  \end{pmatrix},
 \]
written with respect to $\mathcal{B}$, $A\in \GL_{j-1}(\K)$ is of order $|\K|^{j-1}-1$, and $J_y$ is some $y$-by-$y$ matrix. If $S$ is not unitary, then we can take $J_y$ to be the identity matrix; if $S$ is unitary, then $J_y$ is a matrix such that $\det(\tilde{x})=1$ (thereby ensuring that $x\in S$).

 Then one of the following holds:
 \begin{enumerate}
  \item $G$ contains a section isomorphic to $\Sym(|\K|^{j-1})$;
  \item $M$ contains $S$;
  \item the action of $G$ on $(G:M)$ is not binary.
 \end{enumerate}

In particular if $G$ is almost simple, $M$ is core-free and the action of $G$ on $(G:M)$ is binary, then $S$ is symplectic and one of the following holds:
\begin{enumerate}
\item $(k,q)=(2,2)$, or 
\item $ (k,q)=(2,3)$ and $M=\langle x\rangle$.
\end{enumerate}
\end{lem}

Note that if $S$ is orthogonal and $q$ is even, then \cite[Lemmas 2.5.7 and 2.5.9]{bg} imply that $\tilde{x}$ lies in $\Omega_n^\varepsilon(q)$. %If $S=\SU_{2j+2}(q)/Z$, then we must choose $J_y$ appropriately in order to ensure that the element $\tilde{x}\in \SU_{2j+2}(q)$, as required.

\begin{proof}
We suppose throughout that the action of $G$ on $(G:M)$ is binary. Our argument is the same for all families, more or less, but the details are different; we will, therefore, need to do some case work -- especially in the third stage of the proof.

\textsc{ Step~1}. We observe first that $\langle x\rangle$ normalizes $U_1$ and $U_2$, two distinct elementary-abelian subgroups of $S$ of order $|\K|^{j-1}$, namely those having shape
\[
\begin{pmatrix}
               1 & u_1 & \cdots & u_{j-1} & & & & \\
                & 1 & & & & & & \\
                &  & \ddots & && &  &\\
                & & & 1 &  & &  &\\
                & & & & 1   & &  &\\
                & & & & -\overline{u_1} & 1 & &  \\
                & & &  & \vdots & & \ddots & \\
                & & &  & -\overline{u_{j-1}} & & & 1\\
              \end{pmatrix},
 \textrm{ and } \begin{pmatrix}
               1 & & &  &  & & & \\
               u_1 & 1 & & & & & & \\
               \vdots  &  & \ddots & & & &  &\\
               u_{j-1} & & & 1 &  & & & \\
               & & & & 1 & -\overline{u_1} & \cdots & -\overline{u_{j-1}}  \\
               & & & & & 1 & & \\
               & & & & &  & \ddots & \\
               & & & & &  & & 1\\
              \end{pmatrix},
\]
respectively. In each case we write only the first $2j$ rows and columns of each matrix -- the remaining rows and columns are completed by setting off-diagonal entries to be $0$, and diagonal entries to be $1$. The resulting group is the set of all matrices obtained by allowing the parameters $u_i$ to range over $\K$. %Define $U_1^T$ and $U_2^T$ to be the two elementary-abelian subgroups whose elements are the transpose of $U_1$ and $U_2$, respectively.

Observe that $\langle x\rangle$ acts fixed-point-freely on each of these two groups. Let us suppose that $U_1\not\leq M$; then the fixed-point-freeness of the action means that $U_1\cap M=\{1\}$. Now let $\Lambda$ be the set of cosets $\{Mu\mid u\in U_1\}$ of $M$ corresponding to $MU_1$. Then this is a set of size $|\K|^{j-1}$ on which the group $M_1=U_1\rtimes \langle x\rangle$ acts 2-transitively. Now Lemma~\ref{l: 2trans gen} implies that $G$ contains a section isomorphic to $\Sym(|\K|^{j-1})$ and the result follows. The same argument works with $U_2$ so we may assume hereafter that $M$ contains $\langle U_1, U_2\rangle$.

Observe that
$$\left\langle 
\begin{pmatrix}
1&u&0&0\\
0&1&0&0\\
0&0&1&0\\
0&0&-\bar{u}&1
\end{pmatrix},
\begin{pmatrix}
1&0&0&0\\
u&1&0&0\\
0&0&1&-\bar{u}\\
0&0&0&1
\end{pmatrix}\mid u\in \mathbb{F}_q\right\rangle=
\left\{
\begin{pmatrix}
Z&0\\
0&\bar{Z}^{-T}
\end{pmatrix}\mid
Z\in \SL_2(q)
\right\}.$$ Now, an easy inductive argument on $j$ shows that
\[
\langle U_1,U_2\rangle=
\left\{\begin{pmatrix}Z&0\\0&\bar{Z}^{-T}\end{pmatrix}\mid Z\in\SL_{j}(q)\right\}
\]
and hence, from the definition of $\tilde{x}$,  $K=\langle U_1, U_2, x\rangle$ contains all matrices of the form
\[
 \begin{pmatrix}
  Z & \\ & \overline{Z}^{-T}
 \end{pmatrix},
\]
where $Z\in \GL_{j}(\mathbb{K})$.

\textsc{ Step~2a.} Next we define $U_3,\dots, U_{j+2}$, $j$ elementary-abelian subgroups of $S$ of order $q^{j-1}$, namely those having shape
\[
\left(\begin{array}{cccc;{2pt/2pt}cccc}
      1 & & & & 0 & u_2 & \cdots & u_j \\
       & \ddots & & & -\overline{u_2} & & & \\
        & & \ddots & & \vdots & & & \\
        & & & 1 & -\overline{u_j} & & & \\ \hdashline[2pt/2pt]
         & & & & 1 & & & \\
         & & & & & 1 & & \\
         & & & & & & \ddots & \\
         & & & & & & & 1
     \end{array}\right), \, \, 
 \textrm{ } \left(\begin{array}{ccccc;{2pt/2pt}cccccc}
      1 && & & & & -\overline{u_1} & & & \\
       & \ddots && & & u_1 & 0 & u_3 & \cdots & u_j\\
        & & \ddots& & & & -\overline{u_3} & & & \\
        & & &\ddots & & & \vdots & & & \\
        & & && 1 & & -\overline{u_j} & & & \\ \hdashline[2pt/2pt]
         & & && & 1 & & & &\\
         & & & && & 1 & & & \\
         & & & && & & \ddots & &\\
         & & & && & & & \ddots &\\
         & & & & && & & & 1 
     \end{array}\right), \textrm{ and so on.}
\]
Note, first, that we have placed dotted lines to mark the point where the ``$e$-vectors'' change to ``$f$-vectors''; note, second, that we have omitted columns and rows corresponding to basis elements from $Y$; note, third, that in the case where $S$ is symplectic the given matrices do not lie in $S$ -- but this is fixed by removing all minus signs, and proceeding in the same way. 

It is easy enough to see that $\langle x\rangle$ normalizes, and acts fixed-point freely on $U_3$. Similarly, $K$ contains a conjugate of $\langle x \rangle$ that acts fixed-point-freely on $U_4$, and so on. By the same argument as before, we have two possibilities: 
\begin{itemize}
\item[(a)] $M$ contains $U_3$; 
\item[(b)] there is a set $\Lambda\subset \Omega$ such that $|\Lambda|=|\K|^{j-1}$ and on which $S^\Lambda$ acts $2$-transitively; then Lemma~\ref{l: 2trans gen} implies that $G$ contains a section isomorphic to $\Sym(|\K|^{j-1})$ and the result follows.
\end{itemize}
Thus, again, we may assume that $M$ contains $U_3$.

Since the same argument works for $U_4,\dots, U_{j+2}$, we conclude that the group $M$ must contain the group $W_1$, consisting of all matrices of the form
\begin{equation}\label{e: unip2a}
 \begin{pmatrix}
  I & Z & \\ & I & \\ & & I_y 
 \end{pmatrix}, 
\end{equation}
where $Z$  is a $j$-by-$j$ matrix satisfying $Z=-\overline{Z}^T$ and having zero diagonal entries (or, in the case where $S$ is symplectic, $Z$ satisfies $Z=Z^T$ and has zero diagonal entries). 

\textsc{ Step~2b}. Now we repeat the argument of Step~2a but this time, all the matrices we use are the transposes of those in Step~2a. We conclude that $M$ must contain the group $W_2$, consisting of all matrices of the form
\begin{equation}\label{e: unip2b}
 \begin{pmatrix}
  I & & \\ Z & I & \\ & & I_y 
 \end{pmatrix}, 
\end{equation}
where $Z$  is a $j$-by-$j$ matrix satisfying $Z=-\overline{Z}^T$ and having zero diagonal entries (or, in the case where $S$ is symplectic, $Z$ satisfies $Z=Z^T$ and has zero diagonal entries).

\textsc{Step~3}. We use the fact that $M$ contains the group $\langle K, W_1, W_2\rangle$ and we split into cases, depending on the particular family of classical groups which we are dealing with.

\textsc{ Case~3A: $S$ is unitary}. In this case an easy argument says that, since $M$ contains the group $K$, the group $M$ contains all matrices of the form \eqref{e: unip2a}, where $Z$ is a $j$-by-$j$ matrix satisfying $Z=-\overline{Z}^T$, i.e. we can drop the requirement that $Z$ has zero diagonal entries. The resulting set of matrices forms an elementary-abelian group $U$ of size $q^{j^2}$ which is the unipotent radical of a parabolic subgroup $P_j$ in $\SU_{2j}(q)$.

The same argument works ``with transposes'' and we obtain that $M$ contains all matrices of the form \eqref{e: unip2b}, where $Z$ is a $j$-by-$j$ matrix satisfying $Z=-\overline{Z}^T$. We split into two cases, depending on the parity of $n$. 

Assume, first, that $n=2k+1$ with $k\geq 2$. Then $M$ contains the projective image of $M_0\cong\SU_{2k}(q)$, where $M_0$ stabilizes the unique non-isotropic basis vector, $v$, in $Y$. Without loss of generality, we may suppose that
$v$ has norm $1$.

Let $\alpha_2,\ldots,\alpha_k\in \mathbb{K}$ and for simplicity set $\alpha_1:=0$. For each $i\in \{1,\ldots,k\}$, let $\beta_{i,i}\in \mathbb{K}$ with $\beta_{i,i}+\overline{\beta_{i,i}}+\alpha_i\overline{\alpha_i}=0$. (Observe that the existence of $\beta_{i,i}$ is guaranteed by Hilbert's Theorem 90.) For each $i,j\in \{1,\ldots,k\}$ with $i\ne j$, let $\beta_{i,j}=0$ when $i>j$ and $\beta_{i,j}=-\overline{\alpha_i}\alpha_j$ when $i<j$. Now, let $g\in S$ be the element fixing $e_1,\dots, e_k$ pointwise and which satisfies 
\begin{align*}
v&\mapsto v+\alpha_2e_2+\cdots +\alpha_k e_k,\\
f_i&\mapsto f_i-\overline{\alpha_i}v+\sum_{j=1}^k\beta_{i,j}e_j.
\end{align*}
In particular, the matrix form of $g$ with respect to the basis $(e_1,\ldots,e_k,f_1,\ldots,f_k,v)$ is
\[
\begin{pmatrix}
I&0&0\\
B&I&-\overline{d}\\
d^T&0^T&1
\end{pmatrix},\,\, \hbox{where }d=\begin{pmatrix}\alpha_1\\\alpha_2\\\vdots\\\alpha_k\end{pmatrix},\,\,
B=\begin{pmatrix}
\beta_{1,1}&\beta_{1,2}&\cdots &\beta_{1,k}\\
\beta_{2,1}&\beta_{2,2}&\cdots &\beta_{2,k}\\
\vdots&\vdots&\ddots&\vdots\\
\beta_{k,1}&\beta_{k,2}&\cdots &\beta_{k,k}
\end{pmatrix}.
\]
Let $U_0$ be the subgroup of $S$ consisting of all of these elements, as $\alpha_2,\ldots,\alpha_k$ run through $\mathbb{K}$. Let $T_0=\nor{M_0}{U_0}$ and observe that $T_0$ acts transitively on the non-identity elements of $U_0$. We conclude that either $M$ contains $U_0$ or else $U_0\rtimes T_0$ acts 2-transitively by right multiplication on the set of right cosets $\{Mu\mid u\in U_0\}$, a set of size $\K^{k-1}$. In the latter case Lemma~\ref{l: 2trans gen} implies that $G$ contains a section isomorphic to $\Sym(|\K|^{k-1})$ and the result follows. In the former case $M$ contains $U_0$ which in turn implies that $M\geq S$ and the result follows.

Assume, next, that $n=2j+2$. Then $M$ contains the projective image of $M_0\cong \SU_{2k-2}(q)$, where $M_0$ stabilizes the basis vectors $e_k$ and $f_k$. In this case we define two subgroups:
\begin{enumerate}
 \item $U_1$ is the subgroup of $S$ whose elements $g$ fix $e_1,\dots, e_{k-1}$ and which satisfy $e_k\mapsto e_k+\alpha_1 e_1 + \cdots + \alpha_{k-1}e_{k-1}$ for some $\alpha_1,\dots, \alpha_{k-1}\in \K$.
 \item $U_2$ is the subgroup of $S$ whose elements $g$ fix $f_1,\dots, f_{k-1}$ and which satisfy $f_k\mapsto f_k+\beta_1 f_1 + \cdots + \beta_{k-1}f_{k-1}$ for some $\beta_1,\dots, \beta_{k-1}\in \K$.
\end{enumerate}
The same argument as for $n$ odd allows us to conclude that either $G$ contains a section isomorphic to $\Sym(|\K|^{k-1})$ or else $M$ contains $U_1$ and $U_2$ and so contains $S$ and the result follows.

From here, we have, by definition, that $j=k$.

\textsc{ Case~3B: $S$ is symplectic and $q$ is odd}. Again an easy argument asserts that, since $M$ also contains the group $K$, then $M$ must contain all matrices of the form \eqref{e: unip2a}, where $Z$ is \emph{any} symmetric matrix. These matrices together form an elementary abelian group $U$ of size $q^{\frac12k(k+1)}$, which is the unipotent radical of a parabolic subgroup $P_k$. Applying the same argument ``with transposes'' allows us to conclude that $M\geq S$, and the result follows.

\textsc{ Case~3C: $S$ is symplectic and $q$ is even}. In this case, the set of matrices of the form \eqref{e: unip2a}, where $Z$ is symmetric with zero diagonal entries, forms an elementary abelian group $U$ of size $q^{\frac12k(k-1)}$, which is the unipotent radical of a parabolic subgroup $P_k$ of an orthogonal group $L=\Omega_{2k}^+(q)$ (this is the particular orthogonal group corresponding to the quadratic form for which our basis is hyperbolic). 

Now, as in the odd case, we can apply the same argument to the transpose of these matrices to conclude that $M\cap S$ contains the group $L\cong\Omega_{2k}^+(q)$. In particular $M\cap S$ is either $L$, $L.2$ or $S$. The result follows if $M\cap S=S$, so assume $M\cap S$ is either $L$ or $L.2$. We define an element $\tilde{g}$ whose action on $\langle e_1, f_2\rangle$ is given by the matrix $\begin{pmatrix}
  1 & 1 \\ 0 & 1
 \end{pmatrix}$ 
and which fixes all elements of $\B\setminus\{e_1, f_1\}$. Clearly $\tilde{g}$ is an element of $\Sp_{2k}(q)$; we take $g$ to be the projective image of $\tilde{g}$ in $S$. Observe that $g$ centralizes $x$ but does not normalize $L$. We can, therefore, repeat all of the preceding argument using subgroups of $L^g$ instead of $L$. The same case is left: when $M\cap S$ contains both $L$ and $L^g$. Since $\langle L, L^g\rangle=S$ the result follows.

\textsc{Case 3D: $S$ is orthogonal and $n=2k$}. In this case, $S=\Omega^+_{2k}(q)/Z$ and the groups $W_1$ and $W_2$ are both unipotent radicals of parabolic subgroups $P_k$ in $S$. From this, we conclude that $M\geq S$ and the result follows.

\textsc{ Case 3E: $S$ is orthogonal, and $n\in\{2k+1, 2k+2\}$}. In this case, $S=\Omega_{2k+1}(q)$ or $\Omega_{2k+2}^-(q)/Z$ and, arguing {\it \`a la} Case~3D, we see that $M$ contains the projective image of $L\cong\Omega_{2k}^+(q)$, where $L$ fixes all vectors in the non-degenerate subspace $\langle Y\rangle$. Recall that, by construction, the element $\tilde{x}$ fixes all vectors in $\langle Y \rangle$.

Suppose first that $q$ is odd, let $z\in Y$ and suppose that $\varphi(z,z)=\eta$ where $\varphi$ is the symmetric form associated with the covering group of $S$. We define an element $\tilde{g}$ whose action on $\langle e_1, z, f_1\rangle$ is given by the matrix
$$\begin{pmatrix}
  1 & a & -\frac12a\eta \\ & 1 & -a\eta \\ & & 1
 \end{pmatrix},$$ where $a$ is some non-zero element of $\Fq$, and $\tilde{G}$ fixes all elements of $\B\setminus\{e_1, f_1, z\}$. Clearly $\tilde{g}$ is an element of $\SOr^\varepsilon_n(q)$; we take $g$ to be the projective image of $\tilde{g}$ in $S$. Observe that $g$ centralizes $x$ but does not normalize $L$. We can, therefore, repeat all of the preceding argument using subgroups of $L^g$ instead of $L$. The same case is left: when $M\cap S$ contains both $L$ and $L^g$. Notice that we can repeat this argument for any choice of $z\in Y$ and any choice of $a\in\Fq$. It is straightforward to conclude that the resulting collection of conjugates of $L$ generates $S$ and, hence $M\geq S$ and the result follows.
 
 Suppose next that $q$ is even, in which case $n=2k+2$, $S=\Omega_{2k+2}(q)$ and $Y=\langle x, y\rangle$. Let $Q$ be the quadratic form associated with the covering group of $S$ and consider the restriction of $Q$ to the subspace $W=\langle e_1, f_1, x, y\rangle$. Let $\tilde{g}$ be a linear transformation which fixes all elements of $\B\setminus\{e_1, f_1, x,y\}$ and, on $W$, restricts to an element of the group $J=\Omega^-_4(q)$ associated with $Q|_W$. By \cite[Lemma 2.5.9]{bg}, $\tilde{g}$ is an element of the covering group of $S$ and we take $g$ to be its projective image in $G$. Again $g$ centralizes $x$ and, again, we must deal with the case where $M\cap S$ contains $L$ and $L^g$ for all such $g$. Thus we may assume that $M$ contains $L_1=\langle L^g \mid g \in J\rangle$. There are two possibilities: either $J$ normalizes $L_1$ or else we can repeat the same argument with $L_1$ in place of $L$ and we are able to assume that $M$ contains $L_2=\langle L^g \mid g\in J\rangle$. Repeating as many times as necessary we are left with the situation where $M$ contains a group $L_\infty$ that contains $L\cong\Omega_{2k}^+(q)$ and is normalized by $J\cong \Omega^-_4(q)$.
 
Let $X=\langle e_1, f_1, e_2, f_2, x, y\rangle$ and consider the group $H\leq S$ that fixes every vector in $X^\perp$ and induces $\Omega_6^-(q)$ on $X$. Observe that $H$ contains $J$ and so, in particular, $J$ normalizes $H\cap L_\infty$. Then $H\cap L_\infty$ is a subgroup of $H=\Omega_6^-(q)\cong \SU_4(q)$ that contains a group isomorphic to $\Omega_4^+(q)\cong \SL_2(q)\times \SL_2(q)$ and is normalized in $H$ by a group isomorphic to $\Omega_4^-(q)\cong \SL_2(q^2)$. Checking \cite[Tables 8.10 and 8.14]{bhr} we conclude that $H\cap L_\infty$ is either $H$ or a subgroup of $H$ isomorphic to $\Sp_4(q)$. Suppose $H\cap L_\infty\cong \Sp_4(q)$. Checking \cite[Table 8.10]{bhr} we see that there is precisely one conjugacy class of subgroups of $H\cong\SU_4(q)$ isomorphic to $\Sp_4(q)$ hence, regarding $H$ as $\Omega_6^-(q)$ these are the stabilizers of non-singular vectors. Note that $H\cap L_\infty$ contains all $J$-conjugates of $H\cap L$. What is more the non-singular vectors fixed by $H\cap L$ are precisely those in $\langle x, y\rangle$. This, in turn, means that, for all $j\in J$, the non-singular vectors fixed by $(H\cap L)^j$ are precisely those in $\langle x, y\rangle ^j$. Thus if $v$ is a non-singular vector fixed by $H\cap L_\infty$, then $v^j\in\langle x, y\rangle$ for all $j\in J$. Direct calculation (or using the fact that $J$ is irreducible on $W$) confirms that no such vector exists. We conclude that $H\cap L_\infty=H$.

Now observe that, working with respect to the basis $\B$, $M$ contains all of the fundamental root groups for $S$, and hence $M$ contains $S$ as required.

Finally, suppose that $G$ is almost simple and $M$ is core-free. Either the action of $G$ on $(G:M)$ is not binary (and we are done) or else $G$ contains a section isomorphic to $\Sym(|\K|^{j-1})$. Lemma~\ref{l: alt sections classical} (and \cite{bhr}) yield the result barring only a few values of $k$ and $q$. In particular we use \magma to verify the result when $S=\PSp_n(q)$ with $(k,q)\in \{(2,2),(2,3),(2,4),(2,5),(2,7),(3,2),(3,3),(4,2)\}$, when $S=\PSU_n(q)$ with  $(k,q)=(2,2)$ and when $S$ is an orthogonal group with $(k,q)\in\{(3,2), (3,3),(4,2)\}$.
\end{proof}

The following proposition deals with one of the lacunae in the previous: when $S$ is orthogonal, $q$ is odd, and $G$ does not contain $\PSO_n^\varepsilon(q)$. The statement of the proposition uses the notation established in the statement of the previous; to make matters more straightforward we assume that $G$ is almost simple.

\begin{lem}\label{l: classical element 2}
Suppose that $q$ is odd, and that $S=\POmega_n^\varepsilon(q)\unlhd G \leq \Aut(\POmega_n^\varepsilon(q))$ with $n\geq 7$. Let $k$ be the Witt index of the associated formed space and 
let $M<G$ be core-free with $x\in M\cap S$, where $x$ is the projective image of
 \[
  \tilde{x}=\begin{pmatrix}
   1 & & &  & & & \\ & A & & & & & \\ &  & \zeta & & & &  \\ & &  & 1  & & & \\ & & & &  A^{-T}  & & \\ & & & & & \zeta^{-1} & \\ &&&&&& I_y
  \end{pmatrix}
 \]
written with respect to $\mathcal{B}$, $A\in \GL_{k-2}(q)$ is of order $q^{k-2}-1$, $\zeta$ is a non-square in $\Fq$ and $I_y$ is the $y$-by-$y$ identity matrix. If the action of $G$ on $(G:M)$ is binary, then 
\[
(k,q)\in\{(3, 3), (3,5), (3,7), (3,9), (4, 3)\}.
\]
\end{lem}
\begin{proof}
 We refer, first, to \cite[Lemma~2.5.7]{bg} to confirm that $\tilde{x}$ is indeed an element of $\Omega_n^\varepsilon(q)$. Now the action of $\tilde{x}$ on the subspace $W:=\langle e_1,\dots, e_{k-1}, f_1, \dots, f_{k-1}, Y\rangle$ is identical to that studied in the previous proposition; the arguments given there allow us to assume that $M$ contains (the projective image of) the group
 \[
  K:= \{g\in \Omega_n^\varepsilon(q) \mid e_k^g=e_k, f_k^g=f_k, g|_W\in \Omega(W)\}.
 \]
We should be careful about exceptions however: studying the proof we see that our conclusion is valid only when $\Alt(q^{k-2})$ is not a section in $S$. Now, Lemma~\ref{l: alt sections classical} implies that exceptions occur only when $q^{k-2}\leq 2k+4$; this yields the given list.% {\color{red} Question by Pablo: (with one exception: in theory, we could have $(k,q)=(6,2)$ if $\Alt(16)$ were a section of $\Omega_{14}^-(2)$; the discussion after \cite[Proposition 5.3.7]{kl} rules this out) isn't $q$ odd here?}.

Now we study the normalizer in $K$ of four different elementary-abelian subgroups $U_1,\dots, U_4$ of $S$ of order $q^{k-2}$. We choose these groups so that they stabilize the subspaces $E=\langle e_1,\dots, e_k\rangle$ and $F=\langle f_1,\dots, f_k\rangle$. We require furthermore that $\langle Y\rangle$ is in the 1-eigenspace of each of the groups, thus to specify the elements of these groups it is enough to specify their action on the subspace $E$:
\begin{align*}
 U_1 &:= \{g \mid e_1^g = e_1;\,\textrm{for all $i=2,\dots, k-1$, there exist $\alpha_i$ such that $e_i^g = e_i+\alpha_i e_k$}\}; \\
 U_2 &:= \{g \mid e_2^g = e_2;\,\textrm{for all $i=1,3,\dots, k-1$, there exist $\alpha_i$ such that $e_i^g = e_i+\alpha_i e_k$}\}; \\
 U_3 &:= \left\{g\, \Big| \begin{array}{l} e_1^g = e_1, \dots, e_{k-1}^g=e_{k-1}; \\ \,\textrm{for all $i=2,\dots, k-1$, there exist $\alpha_i$ s.t. $e_k^g = e_k+\alpha_2e_2+\cdots+\alpha_{k-1}e_{k-1}$}\end{array}\right\}; \\
 U_4 &:= \left\{g\, \Big| \begin{array}{l} e_1^g = e_1, \dots, e_{k-1}^g=e_{k-1}; \\ \,\textrm{for all $i=1,3,\dots, k-1$, there exist $\alpha_i$ s.t. $e_k^g = e_k+\alpha_1e_1+\alpha_3e_3+\cdots+\alpha_{k-1}e_{k-1}$}\end{array}\right\}.
\end{align*}

It is a simple matter to check that, for each $i=1,\dots, 4$, $N_K(U_i)$ acts transitively on the non-trivial elements of $U_i$. Thus, by the same argument as before, we have three possibilities:
\begin{itemize}
\item[(a)] $M$ contains $U_i$; 
\item[(b)] $G$ admits a beautiful subset of size $q^{k-2}$;
\item[(c)] $S$ admits a section isomorphic to $\Alt(q^{k-2})$.
\end{itemize}
The second possibility is ruled out because $G$ is binary on $(G:M)$ and the third is ruled out as before, except for the listed exceptions. Therefore, $M$ contains $U_i$ for each $i$ and hence $\langle K, U_1,\dots, U_4\rangle=S$, and the result follows.
\end{proof}

Lemmas~\ref{l: psu3 element} and~\ref{l: psu4} deal with some small rank cases that were not covered by Lemma~\ref{l: classical element}.

\begin{lem}\label{l: psu3 element}
Let $G$ contain a subgroup $S\cong\SU_3(q)/Z$, where $Z$ is a central subgroup of $\SU_3(q)$ and $q>2$.
We let $\mathcal{B}:=(e_1, f_1, x)$ be a hyperbolic basis for the underlying unitary space. Let 
\[
  \tilde{g}=\begin{pmatrix}
             t & & \\ & t^{-q} & \\ & & 1
            \end{pmatrix}\in \mathrm{SU}_3(q),
 \]
where $t\in \Fq$ is of order $q-1$; let
\[
  \tilde{g}'=\begin{pmatrix}
             u & & \\ & u^{-q} & \\ & & u^{q-1}
            \end{pmatrix}\in \mathrm{SU}_3(q),
 \]
where $u\in \Fq$ is of order $q^2-1$. Let $g$ and $g'$ be the projective images of $\tilde{g}$ and $\tilde{g}'$ in $S$. Let $M<G$ and, if $q$ is odd, then suppose that $g\in M$; if $q$ is even, then suppose that $g'\in M$. Then one of the following holds:
 \begin{enumerate}
  \item $G$ contains a section isomorphic to $\Sym(q)$;
  \item $M$ contains $S$;
  \item the action of $G$ on $(G:M)$ is not binary.
 \end{enumerate}
In particular, if $G$ is almost simple with socle $S$ and $M$ is core-free, then the action of $G$ on $(G:M)$ is not binary.

\end{lem}
\begin{proof}
Assume first that $q$ is odd and suppose that the action of $G$ on $(G:M)$ is binary. Write $T_0=\langle g\rangle$ and observe that $T_0$ normalizes the following groups of order $q$:
\[
 U_1:=\left\{\begin{pmatrix}
             1 &  -\frac12 a^2 & a\\ & 1 & \\ & -a & 1
            \end{pmatrix} \mid a\in\Fq\right\} \textrm { and } 
            U_2:= \left\{\begin{pmatrix}
             1 & & \\ -\frac12a^2 & 1 & -a \\ a & & 1
            \end{pmatrix} \mid a\in\Fq\right\}.
\]
What is more,  $U_1\rtimes T_0$ and $U_2\rtimes T_0$ are both Frobenius subgroups.

Suppose that $M$ does not contain the projective image of $U_1$. Then, since $M$ contains $T_0$, we conclude that $M\cap U_1=\{1\}$. Define $M_1$ to be the projective image in $S$ of $U_1\rtimes T_0$, and observe that $M_1$ acts 2-transitively on $(M_1:M\cap M_1)$. Let $\Lambda:=\{Mu\mid u\in U_1\}$; we conclude that the set-wise stabilizer of $\Lambda$ acts $2$-transitive. Lemma~\ref{l: 2trans gen} implies that $G$ contains a section isomorphic to $\Sym(q)$ and the result follows.

Clearly the same argument applies if $M$ does not contain the projective image of $U_2$. Thus we may assume that $M$ contains the projective image of $\langle U_1, U_2\rangle$. This projective image is a subfield subgroup isomorphic to $\SOr_3(q)$. Thus $\SOr_3(q)\le M$. Now observe that $g$ acts fixed-point-freely by conjugation on the conjugates $U_i^{g'}$ for $i \in\{ 1,2\}$. Therefore, we may apply the argument above also to the groups $U_1^{g'}$ and $U_2^{g'}$. We conclude, again, that $G$ contains a section isomorphic to $\Sym(q)$, or else that $M$ contains $\langle \SOr_3(q),U_1^{g'},U_2^{g'} \rangle = S$ and the result follows.

\smallskip

Assume now that $q>2$ is even and $g'\in M$. Let $X$ be the subgroup of $S$ that is isomorphic to $\SU_2(q)$ and acts trivially on $\langle x\rangle$, where $x$ is the third basis vector of the basis $\mathcal{B}$ for $V$. Then $(g')^{q+1}$ acts fixed-point-freely on 
the unipotent subgroups 
\[U_1:=\left\{\begin{pmatrix}1&\alpha&0\\0&1&0\\0&0&1\end{pmatrix}\mid \alpha\in\mathbb{F}_q\right\}, U_2:=\left\{\begin{pmatrix}1&0&0\\\alpha&1&0\\0&0&1\end{pmatrix}\mid \alpha\in\mathbb{F}_q\right\}\] of $X$. Therefore, arguing as usual, we can assume $\langle U_1,
U_2 \rangle = X \le M$. Hence $Y = \langle X,g' \rangle\le M$.

Now $Y$ is a maximal subgroup of $S$ in the $\mathcal{C}_1$ class (see for instance~\cite{bhr}). Then \cite[Prop.~4.2]{gs_binary} implies that $(S:Y)$ contains a beautiful subset with respect to the action of $S$ and, checking the proof of \cite[Prop.~4.4]{gs_binary} we see that there is always a beautiful subset of size at least $q$. We conclude that either $M$ contains $S$ (and the result follows) or else $M\cap S=Y$ and $(G:M)$ contains a subset $\Lambda$ of size at least $q$ on which $S_\Lambda$ acts 2-transitively. Then Lemma~\ref{l: 2trans gen} implies that $G$ contains a section isomorphic to $\Sym(q)$ and the result follows.

Finally, suppose that $G$ is almost simple and $M$ is core-free. Either the action of $G$ on $(G:M)$ is not binary (and we are done) or else $G$ contains a section isomorphic to $\Sym(q)$. Lemma~\ref{l: alt sections classical} implies that $q\leq 5$; we confirm the result for $q\in\{3,4,5\}$ with a {\tt magma} computation.
\end{proof}

\begin{lem}\label{l: psu4}
 Let $G$ contain a subgroup $S\cong\SU_4(q)/Z$, where $Z$ is a central subgroup of $\SU_4(q)$ and $q>2$. We let $\mathcal{B}:=(e_1, e_2, f_1, f_2)$ be a hyperbolic basis for the underlying unitary space. Let $x\in S$ be the  projection in $S$ of
 \[
  \tilde{x}=\begin{pmatrix}
   a & & & \\ & 1 & & \\ & & a^{-1} & \\ & & & 1
  \end{pmatrix}
 \]
written with respect to $\mathcal{B}$, where $a$ is an element of $\Fq^\ast$ of order $q-1$. Let $M<G$ with $x\in M$. Then one of the following holds:
 \begin{enumerate}
  \item $G$ contains a section isomorphic to $\Sym(q)$;
  \item $M$ contains $S$;
  \item the action of $G$ on $(G:M)$ is not binary.
 \end{enumerate}
In particular, if $G$ is almost simple with socle $S$ and $M$ is core-free, then the action of $G$ on $(G:M)$ is not binary.
\end{lem}
\begin{proof}
Suppose that the action of $G$ on $(G:M)$ is binary, and write $X=\langle x \rangle$. Let $y$ be any element of one of the following forms:
\[
 \begin{pmatrix}
  1 & \alpha & & \\ & 1 & & \\ &  & 1 & \\ & & -\overline{\alpha} & 1
 \end{pmatrix} \textrm { or }
 \begin{pmatrix}
  1 & & & \alpha\\ & 1 & - \overline{\alpha} & \\ &  & 1 & \\ & &  & 1
 \end{pmatrix},
\]
or the transpose of these forms (in each case $\alpha\in\Fq^*$). Now let $U=\langle y^h \mid h\in X \rangle$. In all four cases, $U$ is a group of order $q$ that is normalized by $X$. In the usual way, we conclude that either $M$ contains $U$, or else there is a subset, $\Lambda:=\{Mu\mid u\in U\}$, of $(G:M)$, on which $G_\Lambda$ acts $2$-transitively. In the latter case Lemma~\ref{l: 2trans gen} implies that $G$ contains a section isomorphic to $\Sym(q)$ and the result follows. On the other hand, if the former case holds for all four unipotent subgroups in question, then, since these four subgroups generate $S$, we conclude that $M$ contains $S$ and the result follows.

Finally, suppose that $G$ is almost simple and $M$ is core-free. Either the action of $G$ on $(G:M)$ is not binary (and we are done) or else $G$ contains a section isomorphic to $\Sym(q)$. Lemma~\ref{l: alt sections classical} implies that $q\leq 8$. One can check directly that $\SU_4(8)$ does not contain a section isomorphic to $\Alt(8)$; we confirm the result for $q\in\{3,4,5,7\}$ with a {\tt magma} computation.
\end{proof}

The groups we deal with in Lemmas~\ref{l: sp4 element},~\ref{l: psu3 element2} and~\ref{l: b2 small} have already been considered in previous lemmas; however, here, we choose a different distinguished element and we prove that every faithful transitive action containing this element gives rise to a non-binary action.

\begin{lem}\label{l: sp4 element}
 Let $S=\Sp_4(q)$ where $q=2^{a}$ with $a\geq 2$, and suppose that $S\leq G \leq \Aut(S)$. Let $g$ be the element 
 \[
  \begin{pmatrix}
   a & & & \\ & b & & \\ & & b^{-1} & \\ & & & a^{-1}
  \end{pmatrix}
 \]
written with respect to a hyperbolic basis $(e_1,e_2,f_2, f_1)$, where $a,b\in \Fq$ are of order $q-1$. Let $M$ be any core-free subgroup of $G$ that contains $g$. Then the action of $G$ on $(G:M)$ is not binary.
\end{lem}
\begin{proof}
We assume that the action of $(G:M)$ is binary, and show a contradiction. Suppose, first, that $q \ge 8$. Let $T=\langle g\rangle$ and consider the groups $U_1,\dots, U_4$, all of order $q$, which contain elements of shape
\[
 \begin{pmatrix}
  1 & & & \\ & 1 & u & \\ & & 1 & \\ & & & 1
 \end{pmatrix}, \quad
 \begin{pmatrix}
  1 & & & \\ & 1 & & \\ & u & 1 & \\ & & & 1
 \end{pmatrix}, \quad
 \begin{pmatrix}
  1 & & & u \\ & 1 & & \\ & & 1 & \\ & & & 1
 \end{pmatrix}, \quad
 \begin{pmatrix}
  1 & & & \\ & 1 & & \\ & & 1 & \\ u & & & 1
 \end{pmatrix},
\]
respectively. Observe that, for all $i=1,\dots, 4$, the group $T$ normalizes $U_i$ and acts fixed-point-freely upon it. Thus, using our usual argument, either $G$ contains $U_i$ or $G$ has a section isomorphic to $\Alt(q)$.
From Lemma~\ref{l: alt sections classical}, $G$ does not contain a section isomorphic to $\Alt(7)$. Thus $G$ contains $\langle U_1,U_2,U_3, U_4\rangle\cong \Sp_2(q)\times \Sp_2(q)$, where $\langle U_1,U_2,U_3,U_3\rangle$ is the subgroup of $S$ that stabilizes the subspaces $\langle e_1,f_1\rangle$ and $\langle e_2,f_2\rangle$.

Now we repeat the argument with the groups $U_5,\dots, U_8$, all of order $q$, which contain elements of shape
\[
 \begin{pmatrix}
  1 & u & & \\ & 1 &  & \\ & & 1 & u\\ & & & 1
 \end{pmatrix}, \quad
 \begin{pmatrix}
  1 & & u & \\ & 1 & & u \\ & & 1 & \\ & & & 1
 \end{pmatrix}, \quad
 \begin{pmatrix}
  1 & & &  \\ u & 1 & & \\ & & 1 & \\ & & u & 1
 \end{pmatrix}, \quad
 \begin{pmatrix}
  1 & & & \\ & 1 & & \\ u & & 1 & \\ & u & & 1
 \end{pmatrix},
\]
respectively. As before we find that either there is a beautiful subset, or else $G$ contains each of the groups $U_5,\dots, U_8$. The first possibility is ruled out as before because $S$ does not admit a section isomorphic to $\Alt(7)$. Therefore $M\ge \langle \Sp_2(q)\times \Sp_2(q), U_5,\ldots,U_8\rangle=S$ and hence $M$ contains $S$, a contradiction.

If $q=4$, then the result is confirmed using a {\tt magma} computation.% Similarly, if $q=8$, then this result was confirmed using {\tt magma}. This computation was rather lengthy, so let us explain our strategy. There are four groups $G$ such that $\Sp(4,8)\unlhd G\leq \Aut(\Sp_4(8))$. For each such subgroup $G$, we have determined a set of representatives for the conjugacy classes of core-free subgroups $M$ of $G$ of order divisible by $7$. (Note that we did not insist that $M$ contains an element of the given form; we simply require that $|M|$ has order divisible by 7. 
%For each such subgroup $M$, we have analysed with a computer whether the action of $G$ on $(G:M)$ is binary. In  this computation we have used three steps:
%\begin{enumerate}
%\item When $|G:M|\le 10^6$, we have constructed the permutation representation under consideration and we have checked (exhaustively) if we have non-binary triples.
%\item For the groups passing the previous test we have applied the character-theoretic criterion from \cite[Lemma 2.7]{dgs_binary}.
%\item For the groups passing the previous test we have randomly generated, $10^5$ times, two cosets of $M$ in $G$, say $M^{g_1}$ and $M^{g_2}$. Then we have tested \cite[Lemma 2.5]{gs_binary}. using (with the notation therein) $\omega_0:=M$, $\omega_1:=Mg_1$ and $\omega_2:=Mg_2$. Observe that this computation does not require that we construct the permutation representation of $G$ on the cosets of $M$ and hence it is less demanding on memory, however we pay the price of having to run this routine a large number of times in order to guarantee the existence of $\omega_0$, $\omega_1$ and $\omega_2$ satisfying the lemma. After this step no group was left.\qedhere
%\end{enumerate}
\end{proof}

\begin{lem}\label{l: psu3 element2}
 Let $S=\PSU_3(q)$ with $q>2$, and suppose that $S\leq G \leq \Aut(S)$. Let $T$ be the projective image of a maximal torus of $\SU_3(q)$ of order $(q+1)^2$, and let $T.2$ be a subgroup of $N_{S}(T)$. Let $M$ be any core-free subgroup of $G$ that contains $T.2$. Then the action of $G$ on $(G:M)$ is not binary.
\end{lem}
\begin{proof}
When $q\in \{3,4,5\}$, the veracity of this lemma is verified with the auxiliary help of {\tt magma}.%{\color{red} Pablo: I should check what happens when $q\in \{2,3,4,5\}$.}

Suppose that $q>5$. Consulting~\cite[Table~$8.5$]{bhr}, we see that there are two possibilities for $M\cap S$: either $M\cap S=T.2$ or $M\cap S=T.\Sym(3)$. In the latter case, $M$ is a maximal subgroup of $SM$. Therefore, the action of $SM$ on $(SM:M)$ is primitive, and we know that the action is not binary, thanks to the argument in~\cite[Section~6]{ghs_binary}. Thus, by Lemma~\ref{l: subgroup}, the action of $G$ on $(G:M)$ is also not binary.

It turns out that the argument in~\cite[Section~6]{ghs_binary} can be used for the case $M\cap S=T.2$ as well.
First, recall that if $M\cap S=T.\Sym(3)$, then the action of $SM$ on $(SM:M)$ is permutation equivalent to the natural action of $SM$ on 
\begin{align*}
\{\{V_1,V_2,V_3\}\mid &\dim_{\mathbb{F}_{q^2}}(V_1)=
\dim_{\mathbb{F}_{q^2}}(V_2)=\dim_{\mathbb{F}_{q^2}}(V_3)=1,V=V_1\perp V_2\perp V_3,\\
& V_1,V_2,V_3 \textrm{ non-degenerate}\}. 
\end{align*}
If $M\cap S=T.2$, then the action of $SM$ on $(SM:M)$ is permutation equivalent to the natural action of $SM$ on 
\begin{align*}
\Lambda:=\{(V_1,\{V_2,V_3\})\mid &\dim_{\mathbb{F}_{q^2}}(V_1)=
\dim_{\mathbb{F}_{q^2}}(V_2)=\dim_{\mathbb{F}_{q^2}}(V_3)=1,V=V_1\perp V_2\perp V_3,\\
& V_1,V_2,V_3 \textrm{ non-degenerate}\}. 
\end{align*}

Now fix $M\cap S=T.2$ and identify $(SM:M)$ with the given set $\Lambda$. Let $e_1,e_2,e_3$ be a basis of $V$ such that the matrix of the Hermitian form with respect to this basis is the identity. Thus $\lambda_{0}:=(\langle e_1\rangle,\{\langle e_2\rangle,\langle e_3\rangle\})\in\Lambda$. 

Consider $\Lambda_0:=\{(V_1,\{V_2,V_3\})\in \Lambda\mid V_1=\langle e_1\rangle\}$. Clearly, $(SM)_{\Lambda_0}=(SM)_{\langle e_1\rangle}$,  $(SM)_{\Lambda_0}/\Zent {(SM)_{\Lambda_0}}$ is almost simple with socle isomorphic to $\PSL_2(q)$ (here we are using $q>3$), and the action of $(SM)_{\Lambda_0}$ on $\Lambda_0$ is permutation equivalent to the action of $(SM)_{\langle e_1\rangle}$ on $\Lambda_0':=\{\{W_1,W_2\}\mid \dim(W_1)=\dim(W_2), \langle e_1\rangle^\perp=W_1\perp W_2, W_1,W_2 \textrm{ non degenerate}\}$. Therefore $(SM)^{\Lambda_0}$ is an almost simple primitive group with socle isomorphic to $\PSL_2(q)$ and having degree $|\Lambda_0|=q(q-1)/2$. Applying \cite[Theorem~1.3]{ghs_binary} to $(SM)^{\Lambda_0}$, we obtain that $(SM)^{\Lambda_0}$ is not binary and hence there exist two $\ell$-tuples $(\{W_{1,1},W_{1,2}\},\ldots,\{W_{\ell,1},W_{\ell,2}\})$ and $(\{W'_{1,1},W'_{1,2}\},\ldots,\{W'_{\ell,1},W'_{\ell,2}\})$ in $\Lambda_0^\ell$ which are $2$-subtuple complete for the action of $(SM)_{\Lambda_0}$ but not in the same $(SM)_{\Lambda_0}$-orbit. By construction the two $\ell$-tuples 
 \begin{align*}
I& :=((\langle e_1\rangle,\{W_{1,1},W_{1,2}\}),(\langle e_1\rangle,\{W_{2,1},W_{2,2}\}),\ldots,(\langle e_1\rangle,\{W_{\ell,1},W_{\ell,2}\})), \\
J&:=((\langle e_1\rangle,\{W'_{1,1},W'_{1,2}\}),(\langle e_1\rangle,\{W'_{2,1},W'_{2,2}\}),\ldots,(\langle e_1\rangle,\{W'_{\ell,1},W'_{\ell,2})\})  
 \end{align*}
 are in $\Lambda^\ell$ and are $2$-subtuple complete. Moreover,  $I$ and $J$  are not in the same $SM$-orbit. Thus $SM$ is not binary on $\Lambda=(SM:M)$. Now, $G$ is not binary on $(G:M)$ by Lemma~\ref{l: subgroup}.
\end{proof}

\begin{lem}\label{l: b2 small}
 Let $S=\PSp_4(q)$ where $q\in\{3,5\}$, and suppose that $S\leq G \leq \Aut(S)$. Let $T$ be a torus of $S$ of size $\frac12(q-1)^2$, and let $M$ be any core-free subgroup of $G$ that contains $T$. Then the action of $G$ on $(G:M)$ is not binary.
\end{lem}
\begin{proof}
In each case we use {\tt magma}: we consider all almost simple groups $G$ with socle one of these two groups; we then compute all the core-free subgroups $M$ having order divisible by $(q-1)^2/2$; finally we prove, in all cases, that the action of $G$ on the right coset of $M$ is not binary. 

To test this, we have divided our algorithm in two cases: when $|M|^3\le |G|$, since we could not afford to determine the permutation representation explicitly having too many points available, we have generated, for $10^6$ times, two cosets $Mg_1$ and $Mg_2$ of $M$ in $G$, and we tested whether Lemma~\ref{l: auxiliary} applies with $\omega_0:=M$, $\omega_1:=Mg_1$ and $\omega_2:=Mg_2$ (observe that for this test we do not need to construct the permutation representation of $G$ on the right cosets of $M$); when $|M|^3>|G|$, we have constructed the permutation representation of $G$ on the cosets of $M$ and we looked (extensively) for pairs for the form $((\omega_1,\omega_2,\ldots,\omega_\ell),(\omega_1',\omega_2',\ldots,\omega_\ell'))$, with $\ell\le 4$, which are $2$-subtuple complete but not in the same orbit. 
 \end{proof}

\subsection{Exceptional groups}

\begin{lem}\label{l: chev exce element} Suppose that $G$ is almost simple with socle $G_0 = G(q)$, an exceptional group of Lie type as in Table~$\ref{tab: m lower}$, and let $m$ be the value given in the table.
\begin{itemize}
\item[(i)] Then $G_0$ has a subgroup $L \cong \SL_m(q)/Z$, where $Z$ is central in $\SL_m(q)$.
\item[(ii)] Adopt the assumptions on $q$ in the last line of Table~$\ref{tab: m lower}$, and let $x \in L$ be the element as in the statement of Lemma~$\ref{l: psl element}$, of order $q^{m-2}-1$ (if $q>2$) or $2^{m-1}-1$ (if $q=2$). If $M$ is any core-free subgroup of $G$ that contains $x$, then the action of $G$ on $(G:M)$ is not binary.
\item[(iii)] If $x$ is the element in part (ii), then $|C_G(x)| < N$, where $N$ is as in Table~$\ref{tab: m lower}$.
\end{itemize}
\end{lem}

\begin{table}
\centering
\begin{tabular}{|r||c|c|c|c|c|c|c|}
\hline
$G(q)$ &  $E_8(q)$ & $E_7(q)$ & $E_6(q)$ & ${{^2\!E_6}}(q)$ & $F_4(q)$ & $G_2(q)$ &  ${^3\!D_4}(q)$  \\
\hline
$m$ & 9 & 8 & 6 & 4 & 4 & 3 & 3  \\
\hline
$r$ & 7 & 5 & 4 & 2 & 2 & 2 & 2  \\
\hline
$N$ & $q^8$ & $q^7$ & $q^8$ & $q^{18}$ & $q^{10}$ & $q^4$ & $q^{10}$ \\
\hline
$q$ & &&& $q>3$ & $q>3$ & $q>5$ & $q>5$ \\
\hline
\end{tabular}
\caption{Values of $m$ such that $\SL_m(q)/Z\leq G(q)$}\label{tab: m lower}
\end{table}

\begin{proof}
(i) The existence of these subgroups $L$  follows easily from inspection of extended Dynkin diagrams (this fact will also be used in the proofs of Propositions~\ref{parab} and~\ref{maxrk}, where we also provide additional comments).

(ii) Let $L \cong \SL_m(q)/Z$ be the subgroup of (i), and let $x \in L$ be the element as in the statement of Lemma~\ref{l: psl element}. Suppose that $M$ is a core-free  subgroup of $G$ containing $x$, and assume that the action of $G$ on $(G:M)$ is binary. We apply Lemma~\ref{l: psl element}: by our assumptions on $q$, the only possibilities are
\begin{itemize}
\item[(a)] $G$ has a section isomorphic to $\Alt(q^{m-2})$ (if $q>2$) or $\Alt(2^{m-1})$ (if $q=2$), or
\item[(b)] $M$ contains $L$.
\end{itemize}
The possibility (a) is excluded by Lemma \ref{altsec} together with our assumptions on $q$. 

Hence $L \le M$. Using Theorem \ref{MAXSUB}, it is straightforward to see that any core-free maximal subgroup $H$ of $G(q)$ containing $M$ is either parabolic, or of maximal rank, or a subgroup $F_4(q)$ or $C_4(q)$ of ${^2\!E_6}(q)$.
(Actually this follows from \cite{LSroot} except for the case where $q=2$.) Hence we can argue exactly as in the proofs of Propositions \ref{parab}, \ref{maxrk} and \ref{type5} (case (4) of the proof) that there are subgroups $A<S$ of $G$ with the following properties:
\begin{itemize}
\item[(1)] $A\le L$ and $A\not \le H$ for any maximal subgroup $H$ of $G(q)$ containing $M$, and
\item[(2)] $A\cong \SL_r(q)$, $S \cong \SL_{r+1}(q)/Z$, where $r$ is as in Table \ref{tab: m lower}.
\end{itemize}
By Lemma \ref{aff}, there is a subset $\D$ of $(G:M)$ such that $|\D| = q^r$ and $G^\D \ge \ASL_r(q)$. 
Since $\Alt(q^r)$ is not a section of $G$, by Lemma \ref{altsec} and our assumptions on $q$,  this contradicts the assumption that  the action of $G$ on $(G:M)$ is binary, completing the proof of (ii).

(iii) There is a simple adjoint algebraic group $\bar G$ over the algebraic closure $\bar\F_q$, and a Frobenius endomorphism $F$ of $\bar G$, such that $G(q)$ is the socle of the fixed point group $(\bar G^F)'$.  The element $x$ lies in $L =
( \bar L^F)'$, where $\bar L$ is a subsystem subgroup $A_{m-1}$ of $\bar G$. The composition factors of the restriction of the Lie algebra $L(\bar G)$ to $\bar L$ can be found in Tables 8.1 - 8.5 of \cite{LS96}, from which it is easy to work out the action of $x$ on $L(\bar G)$, and hence obtain the upper bound 
 $\dim C_{L(\bar G)}(x)\le R$, where $R$ is as follows:
\[
\begin{array}{c|ccccccc}
G(q) & E_8(q) & E_7(q) & E_6(q) & {{^2\!E_6}}(q) & F_4(q) & G_2(q) &  {^3\!D_4}(q) \\
\hline
R  & 8 & 7 & 8 & 18 & 10 & 4 & 10 
\end{array}
\]
Here is an example of such a computation for the case where $G(q) = {{^2\!E_6}}(q)$: here $\bar L = A_3$ and 
\[
L(\bar G) \downarrow A_3 = L(A_3)/\l_1^4/\l_3^4/\l_2^4/0^7,
\]
where in this notation, $\l_1$ denotes the irreducible 4-dimensional module $V_{A_3}(\l_1)$, and so on. The action of $x$ on the module $\l_1$ is $(1,a,\l,\l^q)$, where $\l \in \bar \F_q$ has order $q^2-1$ and $a = \l^{-1-q}$. Hence $x$ has fixed point space of dimension 1 on $\l_1$ and $\l_3$, and of dimension 0 on $\l_2$ (which is $\wedge^2(\l_1)$). It follows that 
\[
\dim C_{L(\bar G)}(x) \le 3 + 4 + 4 + 7 = 18.
\]
In each case there is in fact a subgroup of $\bar G$ of dimension $R$ centralizing $x$: for $E_8$ and $E_7$ this is just a maximal torus; for $E_6$ it is a subgroup $T_5A_1$ (where $T_5$ denotes a 1-dimensional torus); for ${{^2\!E_6}}$ it is a subgroup $T_2A_2A_2$ (since $x \in A_2 < A_3 = \bar L$, and this $A_2$ centralizes $A_2A_2$ in $\bar G$); similarly in $F_4$ it is $T_2A_2$; in $G_2$ it is $T_1A_1$ and in ${^3\!D_4}$ it is $T_1A_1^3$. Hence these subgroups are the full centralizers of $x$ in $\bar G$ (noting that $C_{\bar G}(x)$ is connected), and hence taking fixed points under the Frobenius endomorphism $F$, we see that $|C_{\bar G^F}(x)|$ is as follows:
\[
\begin{array}{|cc||cc|}
\hline
G(q) & |C_{\bar G^F}(x)| & G(q) & |C_{\bar G^F}(x)| \\
\hline
E_8(q) & (q^7-1)(q-1),\,q>2 & {{^2\!E_6}}(q)  & (q^2-1)|A_2(q^2)| \\
           & 2^8-1,\,q=2            & F_4(q) & (q^2-1)|A_2(q)| \\
E_7(q) &   (q^6-1)(q-1),\,q>2 &  G_2(q) & (q-1)|A_1(q)| \\
           & 2^7-1,\,q=2               & {^3\!D_4}(q) & (q-1)|A_1(q^3)| \\
E_6(q) &   (q^4-1)(q-1)|A_1(q)|,\,q>2 & & \\
           & (2^5-1)|A_1(2)|,\,q=2 && \\
\hline
\end{array}
\]
Since $x$ is centralized by no graph or field automorphisms, it follows that $|C_G(x)| < N$, where $N$ is as in 
Table \ref{tab: m lower}. This completes the proof.
\end{proof}

\begin{lem}\label{l: exc rank 2 element}
Let $G$ contain a subgroup $S$ isomorphic to ${^2\!F_4}(q)$ ($q>2$), 
${^2\!G_2}(q)$ ($q>3$) or ${^2\!B_2}(q)$ ($q>2$). Then there exists an element $x \in S$ of order $q-1$ such that, if $x\in M < G$, then one of the following holds:  \begin{enumerate}
  \item $G$ contains a section isomorphic to $\Sym(q)$;
  \item $M$ contains $S$;
  \item the action of $G$ on $(G:M)$ is not binary.
 \end{enumerate}
In particular if $G$ is almost simple, then we can choose $x$ to have the following properties:
\begin{itemize}
\item[(i)] If $M$ is any core-free subgroup of $G$ that contains $x$, then the action of $G$ on $(G:M)$ is not binary.
\item[(ii)] $|C_G(x)| = (q-1)^2,\,q-1$ or $q-1$, according as $S = {^2\!F_4}(q)$, ${^2\!G_2}(q)$ or 
${^2\!B_2}(q)$, respectively.
\end{itemize}
\end{lem}

\begin{proof}
%We assume that the action of $G$ on $(G:M)$ is binary, and aim for a contradiction.
Suppose first that $S={^2\!F_4}(q)$. Let $T$ be a maximal torus of $S$ of order $(q-1)^2$, and choose $x \in T$ of order $q-1$ such that $C_{\Aut(S)}(x) = T$ (such an element exists by \cite{shinoda}).  Let $M$ be a subgroup of $G$ containing $x$ and assume that the action of $G$ on $(G:M)$ is binary.

The structure of the root subgroups of $S$ with respect to $T$ can be found in \cite[Theorem 2.4.5(d)]{gls3}. 
If $U_1$ is a  root subgroup of type $A_1^2$ with respect to $T$, then either $U_1$ is contained in $M$ or else $U_1\rtimes\langle x\rangle$ acts 2-transitively on $\Lambda= \{Mu\mid u\in U_1\}$, a set of size $q\geq 8$. In the latter case Lemma~\ref{l: 2trans gen} implies that $G$ contains a section isomorphic to $\Sym(q)$ and the result follows. Suppose, then that $U_1\leq M$. The same argument applies to $U_1^-$, the ``opposite'' root group of type $A_1^2$. We can apply the same argument to a  root group $U_2$, of type $B_2$, although in this case we consider $Z(U_2)\rtimes \langle x \rangle$, and we conclude that $Z(U_2)\leq M$. The same argument applies to $U_2^-$, the ``opposite'' root group of type $B_2$. Since $\langle U_1^\pm, Z(U_2^\pm)\rangle=S$, it follows that $M\geq S$ and the result follows. 

In the special case where $G$ is almost simple and $M$ is core-free, Lemma~\ref{altsec} implies that $G$ does not contain a section isomorphic to $\Alt(q)$ and the result follows.

Now suppose that $S={^2\!G_2(q)}$. We refer to the main theorem of  \cite{ward} for basic information about this group. 
We may choose an element $x \in S$ of order $q-1$ such that $C_{\Aut(S)}(x) = \langle x \rangle$ and $x$ normalizes a  Sylow $3$-subgroup $P$ of $S$. If $Z=Z(P)$, then $|Z|=q$ and~\cite[item~(3)]{ward} implies that $\langle x\rangle$ acts fixed-point-freely on $Z$. Let $M$ be a subgroup of $G$ containing $x$ and assume that the action of $G$ on $(G:M)$ is binary.

Suppose that $Z\not\leq M$. Then $Z\cap M=\{1\}$ and, identifying $(G:M)$ with the cosets of $M$ we can set $\Lambda=\{Mz\mid z\in Z\}$, a subset of $(G:M)$ of size $q$. Then $Z\rtimes \langle g\rangle$ acts 2-transitively on $\Lambda$. Lemma~\ref{l: 2trans gen} implies that $G$ contains a section isomorphic to $\Sym(q)$ and the result follows.

We may suppose, then, that $Z\leq M$. The same argument applies to the ``opposite'' Sylow $3$-subgroup $P^-$ of $S$ and so we obtain a second subgroup, $Z^-$, on which $\langle x\rangle$ acts fixed-point-freely, and which is contained in $M$. Thus $\langle Z,Z^-,x\rangle\le M$. From the list of maximal subgroups of ${}^2G_2(q)$ in~\cite{Ree}, we see that either $\langle Z,Z^-,x\rangle=S$ or $\langle Z,Z^-,x\rangle \le 2\times\mathrm{PSL}_2(q)$. The latter is not possible since $\langle x\rangle$ is fixed point free on $Z$.
%Since $C_{S}(Z)\le N_{S}(Z)=P$ (see~\cite{ward}), and there is no involution in $P$ centralizing $Z$, 
Hence $\langle Z,Z^-,x\rangle=S$. Therefore, $M\geq S$ and the result follows.

In the special case where $G$ is almost simple and $M$ is core-free, we note that $G$ has no section isomorphic to $\Alt(5)$ (because 5 does not divide $|S|$) and the result follows.

Suppose finally that $S={^2\!B_2}(q)$. We refer to~\cite{suzuki}. The first part of the argument for ${^2\!G_2(q)}$ applies here  word-for-word, except that this time $P$ is a Sylow $2$-subgroup of $S$. As in that previous case, we obtain that $M$ contains $\langle Z_1, Z_2\rangle$, where $Z_1$ and $Z_2$ are the centres of two distinct Sylow $2$-subgroups of $S$. From the list of the maximal subgroups of $S$ in~\cite[Theorem~9]{suzuki} we obtain $\langle Z_1,Z_2\rangle = S$, completing the proof as before. In the special case where $G$ is almost simple and $M$ is core-free, we note that $G$ has no section isomorphic to $\Alt(3)$ (because 3 does not divide $|S|$) and the result follows. 
\end{proof}

\section{Results on odd-degree actions}\label{s: odd order}

In this section we present two results, both proved using {\tt magma}. Our methods are described in full in \S\ref{s: computation}.

\begin{lem}\label{l: odd degree Lie}
Let $M_0$ be one of the following groups 
\begin{equation*}
\begin{array}{l}
\PSL_2(r)\,(r\le 31), \PSL_3(r) \, (r\in \{2,3,4,5\}), \PSL_4(r)  \, (r\in \{2,3,5\}),\\
 \PSU_3(r) \, (r\in \{3,4,5,8\}), \PSU_4(r)\,(r\in \{2,3,4,5,7\}), \PSU_5(r)\,(r\in \{2,3,4,5,7\}), \PSU_6(2),\\
\PSp_4(r)\, (r\in \{2,3,4,5,7\}), \PSp_6(r)\, (r\in \{2,3\}), \PSp_8(2),\\
\Omega_7(r)\, (r\in \{3,5,7,9\}), \POmega_8^-(r)\,(r\in \{2,3,4\}),\Omega_8^+(2), \Omega_{10}^\pm(2),\,\Omega_{10}^-(3),\,\\
{^2\!B_2(8)}, {^2\!B_2(32)}, G_2(r)\,(r\in \{3,4,5\}),\,{^3\!D_4}(r)\,(r\in \{2,3\}),\, F_4(r)\, (r\in \{2,3\}), \,{^2\!F_4}(2)',\,{^2\!E_6(2)}.
\end{array}
\end{equation*}
Let $M$ be an almost simple group with socle $M_0$ and let $H$ be a core-free subgroup of $M$ with $|M:H|$ odd. Then either the action of $M$ on $(M:H)$ is not binary or $M,M_0$ and $H$ are as in Table $\ref{oddtab}$.
\end{lem}

\begin{table}[!ht]
\caption{Some odd-degree binary actions} \label{oddtab}
\[
\begin{array}{|c|c|c|}
\hline
M_0 & M & |M:H| \\
\hline
\Alt(5) &  \Sym(5)& 5 \\
\Alt(5) & \Alt(5) & 15 \\
 \PSL_2(8) & \PSL_2(8) & 63 \\
 \PSL_2(8) & \PSL_2(8).3 & 189 \\
\PSL_2(16) & \PSL_2(16) & 255 \\
\PSL_2(16) & \PSL_2(16).2 & 51 \\
\hline
\end{array}
\]
\end{table}

\begin{proof}
Suppose first that $M_0\notin\{ F_4(2),F_4(3),^{3}D_4(3),\,{^2\!E_6(2)}\}$.
We have constructed all the groups $M$ under consideration and all odd index subgroups $H$ of $M$. The construction of $H$ can be done quite efficiently working recursively: for each group $M$ under consideration, the list of the maximal core-free subgroups $X$ of $M$ is either already available in {\tt magma}, or it can be constructed. Then, we can simply select the subgroups $X$ with $|M:X|$ odd. In all cases, $X$ is considerably smaller than $M$ and we can directly compute the odd index subgroups of $X$. Thus, we obtain all odd index subgroups of $M$.

We then check that the action of $M$ on $(M:H)$ is not binary with a combination of techniques. First, we have checked the permutation character bound, see  Lemma~\ref{l: characters}, then we have tried to apply Lemma~\ref{l: auxiliary} and finally Lemma~\ref{l: M2}. For permutation groups failing this method, the degree of the action was less than $10^7$ and hence we simply searched for non-binary $t$-tuples (with $t$ relatively small: except when $M_0=\mathrm{PSU}_4(2)$, it was sufficient to consider $t\in \{3,4\}$). %For permutation groups failing the permutation character bound and of degree larger than $10^7$, we used Lemma~$2.5$ in~\cite{gs_binary}, a.k.a Lemma~\ref{l: auxiliary}. 

\smallskip

Suppose now that $M_0=F_4(2)$. In particular, either $M=F_4(2)$ or $M=F_4(2).2$. Let $H$ be a subgroup of $M$ with $|M:H|$ odd and let $K$ be a maximal core-free subgroup of $M$ with $H\le K$. We have reported in Table~\ref{t: F4(2)} the maximal subgroups of $M$. We have proved in Proposition~\ref{smallex} that the action of $M$ on $(M:K)$ is not binary and hence we may suppose that $H<K$. Using the information on $K$ in Table~\ref{t: F4(2)}, we have computed the odd index subgroups of $K$ and we have checked that, except when $M=F_4(2).2$, $K=[2^{22}](\mathrm{Sym}(2)\mathrm{wr}\Sym(2))$ and $H$ is a Sylow $2$-subgroup of $K$, the action of $K$ on $(K:H)$ is not binary. In particular, we may suppose that  $M=F_4(2).2$, $K=[2^{22}](\mathrm{Sym}(2)\mathrm{wr}\Sym(2))$. Now, let $T$ be a maximal subgroup of $M$ with $H\le T$ and with $T\cong [2^{20}].\mathrm{Alt}(6)\cdot 2^2$ (clearly, this is possible from Table~\ref{t: F4(2)} and from Sylow's theorem). Now, the action of $T$ on $(T:H)$ is not binary and hence so is the action of $M$ on $(M:H)$.

\smallskip

Suppose now that $M_0=F_4(3)$. In particular, $M=F_4(3)=M_0$. Let $H$ be a subgroup of $M$ with $|M:H|$ odd and let $K$ be a maximal subgroup of $M$ with $H\le M$. From~\cite{liesax}, we see that $K$ is isomorphic to either $2.\Omega_9(3)$ or to $2^2.\POmega_8^+(3).\Sym(3)$. For each of these two groups, we have computed all the subgroups $H$ with $|K:H|$ odd and we have checked that either $K$ is not binary on $(K:H)$, or $K=H$, or $K=2^2.\POmega_8^+(3).\Sym(3)$ and $|K:H|=3$. In the first case, we deduce that the action of $M$ on $(M:H)$ is not binary and hence we may consider one of the remaining cases. In all cases, $7$ is a divisor of both $|M:H|$ and $|H|$ and also $7^2$ is the largest power of $7$ dividing $|M|$. Let $V$ be a Sylow $7$-subgroup of $M$. Now, $N_M(V)$ lies in the maximal rank subgroup ${}^3D_4(3).3$, since the normalizer of $V$ in ${}^3D_4(3)$ is $V.\SL_2(3)$, see~\cite{K} for example. Therefore, we have $N_M(V)\cong V.(3\times \mathrm{SL}_2(3))$, where the action of $3\times \SL_2(3)$ by conjugation on $V$ has two orbits of cardinality $4$ on the subgroups of $V$ having order $7$. This shows that we are in the position to apply Lemma~\ref{l: M2} with the prime $7$ and we deduce that also the action of $M$ on the remaining cases is not binary.

\smallskip

Suppose now that $M_0={^3}D_4(3)$. In this case, the maximal subgroups of $M$ are not available in magma. However, using~\cite{liesax}, we see that if $H$ is a core-free maximal subgroup of $M$ and $|M:H|$ is odd, then $H\cap M_0$ is either $G_2(3)$, or 
$(7\times \mathrm{SU}_3(3)).2$, or $(\mathrm{SL}_2(27)\circ \mathrm{SL}_2(3)).2$. Now, using the structure  of  these groups, we may construct them as subgroups of $M$ (for instance, when $H\cap M_0\cong 7\times\mathrm{SU}_3(3).2$, we may construct $H$ by computing the normalizer of a suitable subgroup of $M$ of order $7$). Then we have checked that the action was not binary using the permutation character method. Then, we worked recursively on the subgroups of $H$, as explained in the first part of this proof.

\smallskip

Suppose now that $M_0={^2E}_6(2)$. We use the information in~\cite{Wilson2E62}. In this case, the maximal subgroups of $M$ are not available in magma. However, using the 
information in the work of Wilson~\cite{Wilson2E62}, we see that if $|M:H|$ is odd, then $X$ is contained in a parabolic subgroup $P$ of $M$. The information in~\cite{Wilson2E62} is also enough to construct the (abstract) group $P$ explicitly using {\tt magma}. At this point, we have constructed for each parabolic subgroup $P$ of $M$, all the subgroups $H$ of $P$ with $|P:H|$ odd. We have checked that the action of $P$ on $(P:H)$ is not binary (by witnessing non-binary triples or quadruples), unless $|P:H|\in \{1,3,9,15,45\}$. At this point, the only actions that we need to discuss are the actions of $M$ on $(M:H)$, where $H\le P$ for some maximal parabolic subgroup $P$ of $M$ and for some subgroup $H$ of $P$ with $|P:H|\in \{1,3,9,15,45\}$. 

Using the structure of $P$, we deduce that $7$ divides both $|M:H|$ and $|H|$. Now, a Sylow $7$-subgroup $V$ of $M$ has order $49=7^2$ and $M$ has two conjugacy classes of elements of order $7$, which are referred to as type 7A and 7B. The group $V$ contains $8$ subgroups of order $7$, where $4$ of these subgroups consist only of 7A elements and $4$ of these subgroups consist only of 7B elements. Using this information, it is readily seen that we may use Lemma~\ref{l: M2} to show that $M$ is not binary on $(M:H)$.
\end{proof}

\begin{lem}\label{l: sporadic small-odd}
 Let $M$ be an almost simple group  with socle $M_0$ a sporadic simple group. 
Then every faithful odd degree action of $M$ is not binary.
\end{lem}

\begin{proof}
 We use \magma to verify the statement of the lemma. We divide the proof into three cases.

\smallskip

\noindent \textbf{(1) } Suppose that $M_0$ is one of the following groups:
\[
  M_{11}, \, M_{12}, \, M_{22}, \, M_{23}, \, M_{24}, \, J_1, \, J_2, \, J_3, \, HS, \, McL,\, He, \, Ru, \, Suz, \, Co_1, \, Co_2, \, Co_3,\]
\[ Fi_{22}, Fi_{23}, Fi_{24}', HN \textrm{ or } O'N.
 \]
Let $M$ be an almost simple group with socle $M_0$. We use \magma to construct all odd index subgroups $H$ of $M$, using a recursive routine as described at the start of the previous proof. For the groups $Co_1$ and $Fi_{24}'$, we used some extra information at the start of the recursion: in these cases, rather than computing all maximal subgroups of $M$ and selecting those of odd index, we explicitly constructed only the maximal subgroups of $M$ having odd index using the information in~\cite{atlas} or in the online atlas of finite group representations. Then, to determine all the other odd index subgroups of $M$ we can simply run the procedure above, for each of the maximal subgroups that we have constructed.

Given $(M,H)$ as above, and using the terminology in \S\ref{s: computation}, either Test~3 (direct analysis) or Test~1 (using the permutation character) is enough to complete the \magma calculations. When $M_0$ is one of $Fi_{23}$, $Fi_{24}'$, $HN$, $O'N$ and $Co_3$, we used Test 1 as well as Lemma~\ref{l: M2}. 

\smallskip

\noindent \textbf{(2) } Suppose now that $M_0$ is one of the following groups:
\[ J_4, Ly, Th, B.\] 
In this case we proceeded similarly at first, by constructing all core-free odd index subgroups $H$ of $M$. However in some cases the index $|H:M|$ was too large to prove directly that the action of $M$ on $(M:H)$ is not binary. Thus we took a different approach as follows.

Let $M$ be an almost simple group with socle $M_0$ in the given list, let $H$ be a core-free subgroup of $M$ with $|M:H|$ odd and let $K$ be a maximal core-free subgroup of $M$ with $H\le K$. Recall that Conjecture~\ref{conj: cherlin} has been verified for primitive actions of almost simple groups having socle a sporadic simple group \cite{dgs_binary}. Therefore, we may suppose that $H<K$. Now, rather than studying the action of $M$ on $(M:H)$, we study the action of $K$ on $(K:H)$. We use \magma to confirm that, except when $(M,M_0,K,H)$ is in Table~\ref{tab: small exceptions}, the action of $K$ on the right cosets of $H$ is not binary (by witnessing a non-binary triple or a non-binary $4$-tuple). Now Lemma~\ref{l: subgroup} implies that the action of $M$ on $(M:H)$ is not binary.

For the remaining cases in  Table~\ref{tab: small exceptions}, we have computed the permutation character for the action of $M$ on the right cosets of $H$ and we have checked that this action is not binary using the permutation character bound (Test~1 in \S\ref{s: computation}).

\begin{table}
\centering
 \begin{tabular}{|c|c|c|c|}
  \hline
 $M_0$ & $K$ & $H$ &$|H:K|$\\
\hline
$J_4$&$2^{3+12}\cdot(\Sym(5)\times \PSL_3(2))$&$2^{3+12}\cdot(\Sym(4)\times \PSL_3(2))$&$5$\\
$B$&$2^{2+10+20}. (M_{22}:2\times \Sym(3))$&$2^{2+10+20}. (M_{22}:2\times 2)$&$3$\\
$B$&$[2^{35}]. (\Sym(5)\times \PSL_3(2))$&$[2^{32}]. (\Sym(4)\times \PSL_3(2))$&$5$\\
$M$&$2^{2+11_+22}.(M_{24}\times \Sym(3))$&$2^{2+11_+22}.(M_{24}\times 2)$&$3$\\
$M$&$2^{5+10+20}.(\Sym(3)\times \PSL_5(2))$&$2^{5+10+20}.(2\times \PSL_5(2))$&$3$\\\hline
  \end{tabular}
\caption{}\label{tab: small exceptions}
\end{table}

\smallskip

\noindent \textbf{(3) } Finally suppose that $M$ is the Monster group. The maximal subgroups $K$ of $M$ having odd index can be deduced from \cite{wilsonmax} and are: 
$$2^{1+24}.Co_1,\;\,\,\,\, 2^{10+16}.\Or_{10}^+(2), \;\,\,\,\,2^{2+11+22}.(M_{24}\times \Sym(3)),\; $$
$$2^{5+10+20}.(\Sym(3)\times \PSL_5(2)),\;\,\,\,\,
2^{3+6+12+18}.(\PSL_3(2)\times 3.\Sym(6)).$$
Let $K$ be one of these groups and let $H$ be a subgroup of $K$ with $|M:H|$ odd. We show that the action of $M$ on the right cosets of $H$ is not binary. If $K=H$, then this follows from~\cite{dgs_binary}. Suppose then $H<K$. Observe that, when $K\cong 2^{1+24}.Co_1$, we have $O_2(K)\le H\le K$ and $K/O_2(K)\cong Co_1$. Therefore, in this case, the proof follows from the fact that the faithful transitive actions of $Co_1$ of odd degree are not binary. 

For the remaining three groups $K$, we have constructed all odd index subgroups $H$ of $K$. Except when $(M_0,K,H)$ is in the last two lines of Table~\ref{tab: small exceptions}, we have verified that the action of $K$ on the right cosets of $H$ is not binary (by using three techniques: via the permutation character method, or  via Lemma~\ref{l: auxiliary}, or when the degree of the action is not very large via witnessing a non-binary triple or a non-binary $4$-tuple). In particular, the action of $M$ on the right cosets of $H$ is not binary in these cases. 

It remains to deal with the cases in Table~\ref{tab: small exceptions}: here we cannot argue as in the paragraph above, because the information in the character table stored in \magma is not enough to construct the permutation character under consideration. When $K=2^{5+10+20}.(\Sym(3)\times L_5(2))$ and $H=2^{5+10+20}.(2\times L_5(2))$, the action of $M$ on the right cosets of $H$ is  not binary by using Lemma~\ref{l: M2} applied with the prime $p=7$ (for details see~\cite[Lemma~$5.1$ and~$5.2$]{dgs_binary}). When $K=2^{2+11+22}.(M_{24}\times \Sym(3))$ and $H=2^{2+11+22}.(M_{24}\times 2)$, the action of $M$ on the right cosets of $H$ is  not binary by using Lemma~\ref{l: M2} applied with the prime $p=7$ (for details see~\cite[Lemma~$5.1$ and~$5.2$]{dgs_binary}).
\end{proof}

\section{Results on centralizers}\label{s: centralizer}

The first result in this subsection is taken from \cite[\S 6]{fulman_guralnick}. 

\begin{prop} \label{centbd}
Let $G=\Cl_n(q)$ be a simple classical group, and let $1\ne g\in H$.
\begin{itemize}
\item[(i)] Then $|C_G(g)| > f(n,q)$, where $f(n,q)$ is as in Table~$\ref{fnq}$.
\item[(ii)] In particular, for any $G=\Cl_n(q)$ we have 
\[
|C_G(g)| > \frac{q^{\lceil (n-1)/2 \rceil}}{4}\left(\frac{q-1}{2qe(\log_q(2n) +4)}\right)^{1/2}.
\]
\end{itemize}
\end{prop}

\begin{table}[!ht]
\caption{Lower bounds for centralizers in classical groups}
\label{fnq}
\[
\begin{array}{|c|c|}
\hline
H & f(n,q) \\
\hline
\PSL_n(q) & \frac{q^{n-1}}{e(1+\log_q(n+1))\,(n,q-1)} \\
\hline
\PSU_n(q) & \frac{q^{n-1}}{(n,q+1)}. \left(\frac{q-1}{e(q+1)(2+\log_q(n+1))}\right)^{1/2} \\
\hline
\PSp_n(q),\,\mathrm{P}\O_n^\e(q) & \frac{q^{\lceil (n-1)/2 \rceil}}{4}\left(\frac{q-1}{2qe(\log_q(2n) +4)}\right)^{1/2} \\
\hline
\end{array}
\]
\end{table}

The next result is Lemma~5.7 of \cite{mps}. %Note that, in the statement of the lemma, $q=t^{1/r}$ where $r$ is the order of the graph automorphism of the Dynkin diagram for $S$, and $t$ is the size of the field where the group is defined; so, for example, for the unitary group of order $q=q^3(q^2+1)(q^3+1)/\gcd(3,q+1)$, we have $r=2$ and $t=q^2$.
 
\begin{lem}\label{l: cent rank}
Let $S=G(q)$ be a be a simple group of Lie type, let $d$ be the untwisted rank of $S$, and let $g$ be an element of $S$. Then
\[
 |C_S(g)| \geq \frac{(q-1)^d}{|{\rm Inndiag}(S)|}.
\]
\end{lem}

\section{Outer automorphisms of groups of Lie type}

Here we record a well-known result which classifies all outer
automorphisms of prime order of finite groups of Lie type. In the terminology
of \cite[Defn. 2.5.13]{gls3}, all such are diagonal, field, graph-field or graph
automorphisms. A proof can be found in \cite[Prop. 1.1]{LawtherLiebeckSeitz}. 

\begin{prop}\label{outeraut}
 Let $L = L(q)$ be a simple  group of Lie type over
$\bF_q$, and let $\a$ be an automorphism of $L$ of prime order. If $L$ is
classical with natural module $V$, suppose that $\a$ does not lie in $\mathrm{PGL}(V)$;
and if $L$ is exceptional, suppose that $\a \not \in \hbox{Inndiag}(L)$. Then
one of the following holds:
\begin{itemize}
\item[{\rm (i)}] $\a$ is a field or graph-field automorphism, and $C_L(\a)$ is of
type $L(q^{1/|\a|})$ or $^2\!L(q^{1/2})$ (or $^3\!D_4(q^{1/3})$ when $L =
D_4(q)$);

\item[{\rm (ii)}] $\a$ is a graph automorphism and the possibilities are as in Table $\ref{autposs}$.
(In the last column of the table, $t$ denotes a long root element.)
\end{itemize}
\end{prop}

\begin{table}[!ht]
\caption{} \label{autposs}
\[
\begin{array}{|l|l|l|}
\hline
L & |\a| & \hbox{possible types for }C_L(\a) \\
\hline
\PSL_n^{\e}(q) & 2 & \PSO_n(q) \,(n \hbox{ odd}) \\
            &   & \PSO^{\pm}_n(q), \PSp_n(q) \,(n \hbox{ even}, \,q \hbox{
odd}) \\
             &   & \Sp_n(q),\,C_{\Sp_n(q)}(t)\,(n \hbox{ even},\,q \hbox{
even}) \\
D_4(q),\,^3\!D_4(q) & 3 & G_2(q),A_2^{\e}(q) \hbox{ if }(3,q)=1 \\
      & & G_2(q),\,C_{G_2(q)}(t) \hbox{ if } 3 \hbox{ divides }q \\ 
E_6^{\e}(q) & 2 & F_4(q), \,C_4(q) \,(q \hbox{ odd}) \\
    & & F_4(q), \,C_{F_4(q)}(t)\, (q \hbox{ even}) \\
\hline
\end{array}
\]
\end{table}

\section{On fusion and factorization}

Before working our way through the families of maximal subgroups given in Theorem~\ref{MAXSUB} we record a few useful lemmas.

In the next lemma, given a group $G$ and two subgroups $X$ and $Y$ with  $X<Y<G$, we say that $Y$ {\it controls fusion} of $X$ in $G$ if, whenever $X^g < Y$ for some $g \in G$, there exists $y \in Y$ such that $X^g=X^y$.

\begin{lem}\label{factn}
Let $G$ be a finite group, and let $A<S<H$ be subgroups of $G$ with the following properties:
\begin{itemize}
\item[{\rm (i)}] $S$ controls fusion of $A$ in $G$;
\item[{\rm (ii)}] $H$ controls fusion of $S$ in $G$;
\item[{\rm (iii)}] $S^x \le H$ for all $x \in N_G(A)$.
\end{itemize}
Then $N_G(A) = N_H(A)\,(N_G(S) \cap N_G(A))$.
\end{lem}

\begin{proof}
Let $x \in N_G(A)$. Then $S^x \le H$ by (iii), so by (ii) there exists $h \in H$ such that $S^x=S^h$. Hence $x^{-1} \in HN_G(S)$, so 
\begin{equation}\label{con1}
N_G(A) \subseteq HN_G(S).
\end{equation}
Now let $y \in N_G(S)$. Then $A^y \le S$, so by (i) there exists $s \in S$ such that $A^y=A^s$. Hence $N_G(S) \subseteq SN_G(A)$, and, by intersecting both sides of this inclusion by $N_G(S)$, it follows that 
\begin{equation}\label{con2}
N_G(S)=S\,(N_G(S)\cap N_G(A)).
\end{equation}
From (\ref{con1}) and (\ref{con2}), we deduce
$$N_G(A)\subseteq H\,(N_G(S)\cap N_G(A))$$
and the proof follows by intersecting both sides of this inclusion by $N_G(A)$.
\end{proof}

In our application of the above lemma we will also need the following result on factorizations of simple groups, which is a consequence of Theorem A of \cite{LPS}.

\begin{lem}\label{factnlem}
Let $G$ be an almost simple group with socle $G_0 = \PSL_n(q^a)$, where $a\ge 2$. Suppose $G$ has a factorization $G = AB$, where $A,B$ are core-free subgroups and $A$ is contained in a subfield subgroup $N_G(\PSL_n(q^b))$, where $\F_{q^b} \subset \F_{q^a}$. Then $(n,q^a) = (2,4)$, $(2,9)$, $(2,16)$ or $(3,4)$, and the possibilities for $A,B$ are as follows:
\[
\begin{array}{|c|c|c|}
\hline
G_0 & A \cap G_0 & B\cap G_0 \\
\hline
\PSL_2(4) & \Sym(3), C_3 & D_{10} \\
\PSL_2(9) & \Sym(4), \Alt(4), C_3 & \Alt(5) \\
\PSL_2(16) & \PSL_2(4) & D_{34} \\
\PSL_3(4) & \PSL_3(2) & \Alt(6) \\
\hline
\end{array}
\]
\end{lem}
\begin{proof}
Theorem~A of \cite{LPS}, together with \cite{LPSb} imply the listed restrictions on the pairs $(n,q^a)$. What is more we know the maximal subgroups of $G_0$ which contain $A\cap G_0$ and $B\cap G_0$; these are listed as the first entry in each column in the tables in~\cite{LPSb}. We then check directly whether it is possible for $A\cap G_0$ or $B\cap G_0$ to be non-maximal. The proof follows with a case-by-case analysis or with a \texttt{magma} computation.
\end{proof}

\chapter{Exceptional Groups}\label{ch: exceptional}

In this chapter we prove the following theorem.

\begin{theo} \label{binex}
Let $G$ be an almost simple primitive permutation group with socle an exceptional group of Lie type. Then $G$ is not binary. 
\end{theo}

Note that the Suzuki and Ree groups $^2\!B_2(q)$ and  $^2\!G_2(q)$ have been dealt with in \cite{ghs_binary}, so we do not consider them here. Note, too, that the groups with socle $^2\!F_4(2)'$ were dealt with in \cite{dgs_binary}, hence these too are excluded from what follows.

Our notation for finite groups of Lie type is in line with standard references such as \cite{gls3}. Dynkin diagrams are labelled as in \cite{bourbaki}.

\section{Maximal subgroups of exceptional groups of Lie type}\label{maxexc}

We shall need a substantial amount of information about maximal subgroups of finite exceptional groups of Lie type, taken from many sources. A summary follows; note that we write ${\rm Lie}(p)$ to mean the set of simple groups of Lie type that are defined over a field of characteristic $p$. By the {\it rank} of a finite group of Lie type $G(q)$, we mean the Lie rank of the corresponding simple algebraic group.

\begin{thm}\label{MAXSUB} {\rm (\cite[Theorem 8]{LSsurv})}  
Let $G$ be an almost simple group with socle $G(q)$, an exceptional group of Lie type over
$\F_q$, $q=p^a$, and let $H$ be a maximal subgroup of $G$. Then one of the following holds:
\begin{itemize}
\item[{\rm (I)}]  $H$ is a parabolic subgroup;
\item[{\rm (II)}] $H$ is reductive of maximal rank: the possibilities
for $H$ are determined in {\rm \cite[Tables 5.1,5.2]{LSS}};
\item[{\rm (III)}]$G(q) = E_7(q)$, $p>2$ and $H\cap G(q) = 
(2^2 \times \POmega_8^+(q).2^2).\Sym(3)$ or $H\cap G(q) = {^3}\!D_4(q).3$;
\item[{\rm (IV)}] $G(q) = E_8(q)$, $p>5$ and $H\cap G(q) = \PGL_2(q) \times
\Sym(5)$;
\item[{\rm (V)}] $H\cap G(q)$ is as in Table $\ref{tbl}$ below;
\item[{\rm (VI)}] $H$ is of the same type as $G$ -- that is, $H'\cap G(q) = G(q_0)$ or a twisted version, where $\F_{q_0}$ is a subfield of $\F_q$;
\item[{\rm (VII)}] $H$ is an exotic local subgroup, as in Table $\ref{exot}$;
\item[{\rm (VIII)}] $G(q) = E_8(q)$, $p>5$ and $H = (\Alt(5) \times \Alt(6)).2^2$;
\item[{\rm (IX)}] $F^*(H) = H_0$ is simple, and not in ${\rm Lie}(p)$: the possibilities
for $H_0$ are given up to isomorphism by {\rm \cite{LSei}};
\item[{\rm (X)}] $F^*(H) = H(q_0)$ is simple and in ${\rm Lie}(p)$; moreover 
$\hbox{rank}(H(q_0)) \leq \frac12\hbox{rank}(G)$, and one of the following
holds:
\begin{itemize}
 \item[{\rm (a)}] $q_0 \leq 9$;
  \item[{\rm (b)}] $H(q_0) = A_2^\e(16)$;
  \item[{\rm (c)}] $H(q_0) = A_1(q_0),
 \,^2\!B_2(q_0)$ or $^2\!G_2(q_0)$, and  $q_0 \leq t(G)$ where $t(G)$ is as in Table $\ref{tg}$ (given by~{\rm \cite{law}}).
\end{itemize}
\end{itemize}
\noindent In cases {\rm (I)-(VIII)}, $H$ is determined up to $G(q)$-conjugacy.
\end{thm}

%\textbf{Note: inserted new subgroups $F_4(q) < E_8(q)$ for $p=3$. Need to deal with these in the proof...}
Note that Table \ref{tbl} includes the subgroups $F_4(q) < E_8(q)$ for $q=3^a$; these were omitted from the list in \cite{LSsurv}, but discovered later in \cite{CST}.

Recent work of Craven has eliminated many of the possibilities left in parts (IX) and (X) of the above theorem:

\begin{thm}\label{crav} {\rm (\cite{crav3, crav2, crav1})}
Let $G$ be as in Theorem $\ref{MAXSUB}$, and let $H$ be a maximal subgroup of $G$.
\begin{itemize}
\item[{\rm (i)}] Suppose $F^*(H)$ is an alternating group $\Alt(n)$. Then $n\in \{6,7\}$. Moreover, if $n=7$, then $G$ is of type $E_7$ or $E_8$.
\item[{\rm (ii)}] Suppose $H$ is as in part $\mathrm{(X)}$ of Theorem $\ref{MAXSUB}$ (and not in any of the other parts), and $H(q_0) \ne A_1(q_0)$. Then one of the following holds:
\begin{itemize}
\item[{\rm (a)}] $G(q) = E_8(q)$, $q=3^a$, and $H(q_0) = \PSL_3(3)$ or $\PSU_3(3)$;
\item[{\rm (b)}] $G(q) = E_8(q)$, $q=2^a$ and $H(q_0) = \PSL_3(4)$, $\PSU_3(4)$, $\PSU_3(8)$, $\PSU_4(2)$  or $^2\!B_2(8)$.
\end{itemize}
\item[{\rm (iii)}] Suppose $H$ is as in part $\mathrm{(X)}$ of Theorem $\ref{MAXSUB}$ (and not in any of the other parts), and $H(q_0) = A_1(q_0)$. Then one of the following holds:
\begin{itemize}
\item[{\rm (a)}] $q_0=q$;
\item[{\rm (b)}] $G(q) = E_7(q)$ and $q_0 = 7,8$ or $25$;
\item[{\rm (c)}] $G(q)=E_8(q)$.
\end{itemize}
\end{itemize}
\end{thm}

\begin{table}[!ht]
\caption{Possibilities for $H$ in (V) of Theorem \ref{MAXSUB}} \label{tbl}
\[
\begin{array}{|l|l|l}
\hline
G(q) & \hbox{possibilities for }F^*(H\cap G(q)) \\
\hline
\hline
G_2(q) & A_1(q)\,(p\geq7) \\
\hline
^3\!D_4(q) & G_2(q)', \;A_2^\pm(q) \\
\hline
F_4(q) & A_1(q)\,(p\geq 13), \;G_2(q)\,(p=7),\; A_1(q) G_2(q)\,(p\geq
3,q\geq 5) \\
\hline
E_6^\e(q) & A_2^\pm (q)\,(\textrm{only for }\e = + \textrm{ and }p\geq 5),\; G_2(q)'\,(p\neq 7, (q,\e)\neq(2,-)), \\
& C_4(q)\,(p\geq 3), \;F_4(q), \, A_2^\e(q) G_2(q)' \\
\hline
E_7(q) & A_1(q)\,(2 \hbox{ classes}, p\geq 17,19), \;A_2^\e(q)\,(p\geq 5),
\;A_1(q) A_1(q)\,(p \geq 5), \\
& A_1(q) G_2(q)\,(p\geq 3,q\geq 5),
\;A_1(q) F_4(q)\,(q\geq 4), \;G_2(q)' C_3(q) \\
\hline
E_8(q) &  A_1(q)\,(3 \hbox{ classes}, p\geq 23,29,31),\; B_2(q)\,(p\geq 5),\;F_4(q)\,(p=3),
\;A_1(q) A_2^\e(q)\,(p \geq 5), \\
& G_2(q)' F_4(q), 
\;A_1(q) G_2(q)  G_2(q)\,(p \geq 3,q\geq 5), \\
& A_1(q) G_2(q^2)\,(p \geq 3,q\geq 5) \\
\hline
\end{array}
\] 
\end{table}

\begin{table}[!ht]
\caption{Exotic local subgroups in (VII) of Theorem \ref{MAXSUB}} \label{exot}
\[
\begin{array}{|llll|}
\hline
2^3.\SL_3(2)& < & G_2(p)&p > 2 \\
3^3.\SL_3(3)& < & F_4(p)&p \geq 5 \\
3^{3+3}.\SL_3(3)& < & E_6^\e(p)&p \equiv \e \hbox{ mod }3, p \geq 5 \\
5^3.\SL_3(5) & < & E_8(p^a)&p \neq 2,5; a=1 \hbox{ or 2, } \\
                &&& \hbox{as } p^2 \equiv 1\hbox{ or }-1 \hbox{ mod }5 \\
2^{5+10}.\SL_5(2) & < & E_8(p)&p > 2 \\
\hline
\end{array}
\]
\end{table}

\begin{table}[!ht]
\caption{Values of $t(G)$ in (X)(c) of Theorem \ref{MAXSUB}: notation $d=(2,p-1)$} \label{tg}
\[
\begin{array}{c|ccccc}
G & G_2(q) & F_4(q) & E_6^\e (q) & E_7(q) & E_8(q) \\
\hline
t(G) & 12d & 68d & 124d & 388d & 1312d 
\end{array}
\]
\end{table}

Note that Table \ref{tbl} contains a small refinement of the corresponding table in \cite[Theorem 8]{LSsurv} for $G(q) = E_6^\e(q)$ and the $A_2^\pm(q)$ and $A_2^\e(q)G_2(q)'$ entries. This refinement is justified in Remark 5.2 of \cite{burthom}. 
%(\textbf{``Involution fixity" paper}). 
Note also that in Table~\ref{tbl}, we write $G_2(q)'$ rather than $G_2(q)$ whenever $q=2$ is allowed; this is because $G_2(2)$ is not itself simple, but its derived subgroup is. There is another fact, concerning $A_2^-(2)$, we need to clarify in Table~\ref{tbl}: there are two embeddings involving $A_2^-(q)$ with $q=2$, namely $A_2^-(q)$ in ${^3\!D_4}(q)$, and $A_2^-(q) G_2(q)'$ in $E_6^-(q)$. Since $A_2^-(2)$ is not simple and not nilpotent for $q=2$, the listed groups are not equal to $F^*(H\cap G(q))$ in these cases; instead one should replace $A_2^-(2)$ by $3^2$.

We shall divide the proof of Theorem \ref{binex} according to the various parts of Theorem \ref{MAXSUB}.
Note for future reference that by Proposition \ref{outeraut}, the maximal subgroups in the theorem that centralize field, graph-field or graph automorphisms of $G(q)$ are as follows: 
\begin{itemize}
\item[(i)] subfield or twisted subgroups as in part (VI); 
\item[(ii)] the following subgroups in part (V): 
\[
\begin{array}{l}
C_4(q), F_4(q) < G(q) = E_6^\e(q), \\
G_2(q), A_2^\e(q) < G(q) = \,^3\!D_4(q).
\end{array}
\]
\end{itemize}

\section{Small exceptional groups of Lie type}\label{s: small}

In this section, we deal with some small exceptional groups of Lie type; this will allow us to avoid some degeneracies in later arguments.

\begin{prop} \label{smallex}
Theorem~$\ref{binex}$ holds when the socle of $G$ is one of the following exceptional groups of Lie type:
\[
\begin{array}{l}
^2\!B_2(q), \,^2\!G_2(q), \\
^2\!F_4(2)', \,^3\!D_4(2), \, F_4(2), \\
 G_2(3), G_2(4), G_2(5).
\end{array}
 \]
\end{prop}
\begin{proof}
The groups with socle $^2\!B_2(q)$, $^2\!G_2(q)$ were dealt with in \cite{ghs_binary}; groups with socle $^2\!F_4(2)'$ were dealt with in \cite{dgs_binary}. The other possibilities have been handled using computational methods, and we describe these in turn.

\smallskip

\noindent\textbf{Socle $^{3}D_4(2)$. } Let $G$ be an almost simple group with socle $^{3}D_4(2)$. We have computed all the core-free maximal subgroups $M$ of $G$ and we have checked that the action of $G$ on $(G:M)$ is not binary.  Except when $M=3^2:2\Alt(4)$ or $M=13:4$ and $G={^3\!D_4(2)}$, or $M=3^2:2\Alt(4)\times 3$ or $M=13:12$ and $G={^3\!D_4(2):3}$, we have used the permutation character method, a.k.a. Lemma~\ref{l: characters}. In the remaining cases, where the permutation character method does not work, we have used Lemma~\ref{l: auxiliary}.
%since we could not afford to determine the permutation representation explicitly having too many points available, we have generated, for $10^6$ times, two cosets $Mg_1$ and $Mg_2$ of $M$ in $G$, and we tested whether \cite[Lemma 2.5]{gs_binary} applies with $\omega_0:=M$, $\omega_1:=Mg_1$ and $\omega_2:=Mg_2$ (observe that for this test we do not need to construct the permutation representation of $G$ on the right cosets of $M$). After a few iterations we have found a suitable $g_1$ and $g_2$ and hence the action is non-binary by Lemma \cite[Lemma 2.5]{gs_binary}.

\smallskip

%\noindent\textbf{The group $^2\!F_4(2)'$ and $^2\!F_4(2)$ }Let $G$ be an almost simple group with socle $^2\!F_4(2)'$. We have constructed all core-free maximal subgroups $M$ of $G$ and we have constructed the permutation representation of $G$ on the right cosets of $M$. Then, we have checked that this action is non-binary witnessing non-binary triples.

\smallskip

\noindent\textbf{Socle  $F_4(2)$. } Note that $|F_4(2)|=2^{24}\cdot 3^6\cdot 5^2\cdot 7^2\cdot 13\cdot 17$ and $G = F_4(2)$ or $F_4(2).2$. In Table \ref{t: F4(2)} we list the maximal subgroups of $G$  and their indices in $G$, as given in \cite{NW}. Let $M$ be a core-free maximal subgroup of $G$.

\begin{table}\centering
\scalebox{0.9}{
\begin{tabular}{c|cc|cc}
\toprule[1.5pt]
Line&Max. subgroups $F_4(2)$&Index & Max. subgroups of $F_4(2).2$&Index \\
\midrule[1.5pt]
1& $(2_+^{1+8}\times 2^6):\Sp_6(2)$ &$3^2\cdot 5\cdot 7\cdot 13\cdot 17$&$[2^{20}]:\Alt(6)\cdot 2^2$&$3^4\cdot 5\cdot 7^2\cdot 13\cdot 17$ \\
2&$\Sp_8(2)$&$2^8\cdot 3\cdot 7\cdot 13$&$\Sp_4(4):4$&$2^{15}\cdot 3^4\cdot 7^2\cdot 13$\\
3&$[2^{20}]:(\Sym(3)\times \PSL_3(2))$&$3^4\cdot 5^2\cdot 7\cdot 13\cdot 17$&$(\Sym(6)\wr 2).2$&$2^{15}\cdot 3^2\cdot 7^2\cdot 13\cdot 17$\\
4&${\rm O}_8^+(2):\Sym(3)$&$2^{11}\cdot 7\cdot 13\cdot 17$& $[2^{22}](\Sym(3)\wr 2)$&$3^4\cdot 5^2\cdot 7^2\cdot 13\cdot 17$\\
5&${^3\!D_4}(2):3$&$2^{12}\cdot 3\cdot 5^2\cdot 17$&$7^2:(3\times 2\Sym(4))$&$2^{21}\cdot 3^{4}\cdot 5^2\cdot 13\cdot 17$\\
6&$^2\!F_4(2)$&$2^{12}\cdot 3^3\cdot 7^2\cdot 17$&$^2\!F_4(2)\times 2$&$2^{12}\cdot 3^3\cdot 7^2\cdot 17$\\
7&$\PSL_4(3).2$&$2^{16}\cdot 5\cdot 7^2\cdot 17$&$[\PSL_4(3).2].2$&$2^{16}\cdot 5\cdot 7^2\cdot 17$\\
8&$(\PSL_3(2)\times L_3(2)):2$&$2^{17}\cdot 3^4\cdot 5^2\cdot 13\cdot 17$&$[(\PSL_3(2)\times \PSL_3(2)):2].2$&$2^{17}\cdot 3^4\cdot 5^2\cdot 13\cdot 17$\\
9&$3.(3^2:Q_8\times 3^2:Q_8).\Sym(3)$&$2^{17}\cdot 5^2\cdot 7^2\cdot 13\cdot 17$&$[3.(3^2:Q_8\times 3^2:Q_8).\Sym(3)].2$&$2^{17}\cdot 5^2\cdot 7^2\cdot 13\cdot 17$\\
\bottomrule[1.5pt]
\end{tabular}
}
\caption{Maximal subgroups of $F_4(2)$ and $F_4(2).2$}\label{t: F4(2)}
\end{table}

Observe that $G$ has a unique conjugacy class of elements of order $5$. Moreover, $F_4(2).2$ has a unique conjugacy class of elements of order $7$. In $F_4(2)$ this conjugacy class of $7$-elements splits into two distinct $F_4(2)$-conjugacy classes; furthermore,  $F_4(2)$ has two conjugacy classes of cyclic subgroups of order $7$. This information can be deduced from \cite{NW}.   

Let $p\in \{5,7\}$. From the information in the previous paragraph and from Lemma~\ref{l: M2}, we deduce that, if $p\mid |M|$ and $p\mid |G:M|$, then the action of $G$ on the cosets of $M$ is not binary. In particular, it remains to consider the case that, for each $p\in \{5,7\}$, $p^2$ divides $|M|$ or $p^2$ divides $|G:M|$.

When $M\in \{{^3}D_4(2):3, {^2}F_4(2), (\PSL_3(2)\times \PSL_3(2)):2\}$ and $G=F_4(2)$, or when $M\in \{\Sp_4(4):4, (\Sym(6)\wr 2).2, {^2}F_4(2)\times 2, [(\PSL_3(2)\times \PSL_3(2)):2].2\}$ and $G=F_4(2).2$, we have verified that the hypothesis of Lemma~\ref{l: alot} with $d=2$ holds true (by computing all proper subgroups $X$ of $M$ with $|M:X|$ odd). %Then, for each $X$, we have computed the permutation group $M_X$ induced by $M$ in its action on the right cosets of $X$ in $M$ and we only considered for the final computation the groups $M_X$ satisfying Lemma~$3.1(2)$ in~\cite{dgs_binary}, that is, every section of $M$ is isomorphic to some section of $M_X$. For every subgroup $X$ of $M$ as before, we have checked that $M_X$ is not binary by witnessing a non-binary triple. 
Thus, we deduce that either $G$ is not binary in its action on $(G:M)$ or $2$ divides $|G:M|-1$. However, the second possibility yields a contradiction (in each case under consideration $|G:M|-1$ is odd). Therefore, $G$ is not binary on $(G:M)$. (Observe that for this computation we only need $M$ as an abstract group and we do not require the embedding of $M$ in $G$.)

When $M=3.(3^2:Q_8\times 3^2:Q_8).\Sym(3)$ and $G=F_4(2)$, or when $M=[3.(3^2:Q_8\times 3^2:Q_8).\Sym(3)].2$ and $G=F_4(2).2$, since there is not enough information in Table~\ref{t: F4(2)} to determine the isomorphism class of $M$, we have used {\tt magma} to construct $M$ inside $G$. For this we used the fact that $M$ is the normalizer of a cyclic group of order $3$ generated by an element in the conjugacy class $3C$. This was possible because generators of $G$ and an element in the class $3C$ are available  in the online atlas webpage. Then, we have argued as in the previous paragraph applying Lemma~\ref{l: alot}. The group $M$ contains a unique subgroup $X$ (up to conjugacy), such that 
\begin{itemize}
\item $|M:X|$ is odd, 
\item the permutation group $M_X$ induced by $M$ on $(M:X)$ is binary and
\item every section of $M$ is isomorphic to some section of $M_X$.
\end{itemize} This subgroup $X$ has index $3$ in $M$ and $M_X\cong \Sym(3)$. As $M$ is maximal in $G$, we obtain that $G$ in its action on $(G:M)$ is primitive. If $G$ acting on $(G:M)$ has a suborbit of cardinality $3$, then it follows from~\cite{Sims} that $|M|$ divides $48$, which is clearly a contradiction. Therefore,  $G$ in its action on $(G:M)$ has no suborbits of cardinality $3$. Thus, if $G$ is binary in its action on $(G:M)$, then, from the {\tt magma} computation above, $G$ has no non-trivial suborbits of odd size in its action on $(G:M)$. However, this implies that $|G:M|-1$ is even, which is clearly a contradiction. 

Using Table~\ref{t: F4(2)}, we see that it remains to deal with the action of $G=F_4(2).2$ on the right cosets of $M=[2^{22}]:(\Sym(3)\wr 2)$. First, we work with the restriction of this action to  $G':=F_4(2)$. Using the generators of $G'$, we may construct $M\cap G'$ using the fact that it is a local subgroup (first by finding a Sylow $2$-subgroup $P$ of $G'$ and then by computing the normalizer of a suitable subgroup of $P$ having index $4$). Let $K$ be a Sylow $3$-subgroup of $M\cap G'$. We see that $K$ contains four $3$-elements in the class $3C$, two $3$-elements in the class $3A$ and two more $3$-elements in the class $3B$. Thus we may write $K=\langle g,h\rangle$, where $g$ and $h$ are $3A$ and $3B$ elements (respectively) and $gh$ is a $3C$ element. Using the formula $|x^G\cap M|/|x^G|$ we may compute the number of fixed points of $x\in M$ without constructing the permutation representation explicitly. We see that $g$ and $h$ both fix $945$ points, $gh$ fixes $81$ points and $K$ fixes $9$ points. Using this information, we see that there exists a $K$-invariant subset $\Lambda\subseteq (G:M)$ having cardinality $10$, say $\Lambda=\{\lambda_0,\lambda_1,\ldots,\lambda_9\}$, such that
\begin{align*}
g^\Lambda:=(\lambda_0)(\lambda_1,\lambda_2,\lambda_3)(\lambda_4,\lambda_5,\lambda_6)(\lambda_7)(\lambda_8)(\lambda_9),\\
h^\Lambda:=(\lambda_0)(\lambda_1,\lambda_2,\lambda_3)(\lambda_4)(\lambda_5)(\lambda_6)(\lambda_7,\lambda_8,\lambda_9),
\end{align*}
where $g^\Lambda$ and $h^\Lambda$ are the restrictions of $g$ and $h$ to $\Lambda$.
It is now easy to verify that the two $10$-tuples
$$(\lambda_0,\lambda_1,\lambda_2,\lambda_3,\lambda_4,\lambda_5,\lambda_6,\lambda_7,\lambda_8,\lambda_9)\textrm{ and }(\lambda_0,\lambda_2,\lambda_3,\lambda_1,\lambda_4,\lambda_5,\lambda_6,\lambda_7,\lambda_8,\lambda_9)$$
are $2$-subtuple complete for the action of $G$ on $(G:M)$. If the action of $G$ on $(G:M)$ is binary, then there exists $a\in G$ mapping the first $10$-tuple into the second $10$-tuple. As $a$ fixes $\lambda_0$, we get $a\in G_{\lambda_0}=M$. Moreover, $a$ fixes set-wise $\Lambda$ and $a^\Lambda=(\lambda_1,\lambda_2,\lambda_3)$. Therefore, $G_{\Lambda}\cap G_{\lambda_0}$ has a Sylow $3$-subgroup of cardinality divisible by $3^3$, but this contradicts the fact that a Sylow $3$-subgroup of $M$ has cardinality $|K|=9$.
(This construction is inspired from Example~$2.2$ in~\cite{ghs_binary}.)
\smallskip

\noindent\textbf{Socle  $G_2(q)$ ($q\le 5$). }These groups (and their automorphism groups) are available in \magma. For each possible group $G$ we have computed its maximal subgroups. When $q=3$, we have then constructed the permutation representations and checked that the group is not binary by witnessing non-binary triples. When $q\in \{4,5\}$, we have computed the permutation characters and used Lemma~\ref{l: characters}: this test was always successful for proving that the action was not binary except when $q=5$ and $M\cong 2^3.\PSL_3(2)$. In this last case we generated,  for $10^6$ times, two cosets $Mg_1$ and $Mg_2$ of $M$ in $G$, and we tested whether Lemma~\ref{l: auxiliary} applies with $\omega_0:=M$, $\omega_1:=Mg_1$ and $\omega_2:=Mg_2$. After a few iterations we have found a suitable $g_1$ and $g_2$ and hence the action of $G$ on  $(G:M)$ is not binary.
\end{proof}

In light of Proposition \ref{smallex}, we assume for the remainder of this section that the socle of $G$ is not one of the groups listed in the proposition. 

%\begin{hypothesis}\label{h: ex}
% Suppose that $G$ is an almost simple group with socle $G(q)$ an exceptional group of Lie type over $\Fq$, $q=p^a$. We suppose, moreover, that $G(q)$ is not included in the following list:
% \[
%   {^2\!B_2}(q), {^2\!G_2}(q), {^2\!F_4}(2)', {^3\!D_4}(2), G_2(2), G_2(3), G_2(4), G_2(5), F_4(2).
% \]
%
%\end{hypothesis}

\section{Parabolic subgroups}\label{parsec}

In this section we prove Theorem \ref{binex} for parabolic actions of exceptional groups of Lie type. We use the notation $P_i$ (resp. $P_{ij}$) for a parabolic subgroup which corresponds to deleting node $i$ (resp. nodes $i,j$ etc.) from the Dynkin diagram. For twisted groups we shall adopt a similar convention using the untwisted Dynkin diagram: for example for $^3\!D_4(q)$ the maximal parabolic subgroups are denoted by $P_2$ and $P_{134}$, and so on.

Here is the main result of the section. The cases excluded in the proposition (those in Table~\ref{ex}) will be dealt with in Lemma~\ref{l: parab}.

\begin{prop}\label{parab}
Assume $G$ is almost simple with socle $G(q)$, an exceptional group of Lie type over $\F_q$, and suppose $G(q)$ is not as in Proposition~$\ref{smallex}$. Let $H$ be a maximal parabolic subgroup of $G$, and $\O = (G:H)$. Suppose further that $(G(q), H)$ is not as in Table~$\ref{ex}$. Then $(G,\O)$ is not binary.
\end{prop}

%Note that the numbering convention for parabolic subgroups in Table~\ref{ex} is consistent with \cite{bourbaki}.

\begin{table}[!ht]
\caption{Exceptions in Prop. \ref{parab}} \label{ex}
\[
\begin{array}{|c|c|}
\hline
G(q) & H \\
\hline
%G_2(q) & P_1\hbox{ or Borel} & q\le 5 \\
^3\!D_4(q), \, q\in\{3,4,5\} & P_2 \\
%F_4(q) & P_1,P_2,P_3,P_4 & 2 \\
%F_4(2^a).2 & P_{14},P_{23} & 2 \\
E_6(2) & P_2\\
%E_6(2).2 & P_{35} \\
{^2\!E_6}(2) & P_2,P_4,P_{16} \\
%^2\!F_4(q) & \hbox{any}  & 2 \\
\hline
\end{array}
\]
\end{table}

\begin{proof}
Let $H=P_i$ (or $P_{ij}$ in cases where $G$ contains a graph automorphism of $G(q)$). Inspection of extended Dynkin diagrams shows that there exist subgroups $A \cong \SL_r(q)$ in $H$, and $S \cong \SL_{r+1}(q)/Z$ in $G$ (where $Z$ is central), such that $A\le S$ and $S \not \le H$, where $r$ is as in Table \ref{rvals}. Hence Lemma \ref{aff} produces a subset $\D$ of $\O$ for which $G^\D$ contains the 2-transitive group $\ASL_r(q)$ of degree $q^r$. If $G^\D$ does not contain  $\Alt(\D)$, this implies that $G$ is not binary by Lemma \ref{l: beautiful}, as required. So assume that $G^\D \ge \Alt(\D)$. Then $q^r \le N_G$, where $N_G$ is as defined in Lemma \ref{altsec}. Also $\Alt(q^r-1)$ must be a section of $H$. 
This implies that $(G,H,q)$ is either  as in Table \ref{ex}, or is one of the following:
\[
\begin{array}{r|ccc}
G(q) &  G_2(q) & {^3\!D_4}(q) & {^2\!F_4}(q) \\
\hline
H &  P_1,\hbox{ Borel} & P_2 & \hbox{any parabolic} \\
\end{array}
\]
%Assume one of these cases holds. For the first three columns, $q^r=9$, so $\Alt(8)$ must be a section of $H$. But this is not the case for these parabolics with $q=3$. 
Consider $G_2(q)$. Here $H$ contains $T_1 = \{h_{\a_1}(c) : c \in \F_q\} \cong C_{q-1}$, and this acts fixed-point-freely on the root group $U = U_{-\a_0}$, where $\a_0$ is the longest root. Observe that $T_1U \cap H = T_1$. Hence, if we set $\D = \{Hu : u \in U\} \subseteq \O$, then $|\D|=q$ and $G^\D \ge (T_1U)^\D = \AGL_1(q)$. Hence, if $q > N_G = 6+\d_{p,5}$ (which is the case, as $q\ne 3,4,5$ by hypothesis), then as above, $G$ is not binary. A similar proof applies to the case $G(q) = \,^3\!D_4(q)$: here $q=3,4,5$ are not excluded in the hypothesis, so these cases are included in Table~\ref{ex}.

Finally, consider $G(q) = \,^2\!F_4(q)$, and note that $q>2$ here, by hypothesis. In this case, the maximal parabolics are $P_i=Q_iL_i$ for $i=1,2$, where $Q_i$ is the unipotent radical and 
\[
L_1 = \GL_2(q),\; L_2 = \,^2\!B_2(q)\times (q-1).
\]
Let $H = P_i$ and $\O = (G:H)$, and suppose $(G,\O)$ is binary.  Let $S \cong \F_q$ be the root subgroup corresponding to the highest root, and $S^-$ its negative. For $i=1,2$ there is a torus $T_1 < L_i$ of order $q-1$ acting fixed-point-freely on both $S$ and $S^-$. Since $S^- \not \le P_i$, the Frobenius group $F = S^-T_1$ satisfies $F\cap P_i = T_1$, and so we obtain in the usual way a subset $\Delta$ of $\Omega$ with $G^\D \ge \AGL_1(q)$, forcing $q\le 8$ by Lemma \ref{altsec}. If $q=8$ and $G^\D \ge \Alt(8)$, then $H=P_i$ must contain a section isomorphic to $\Alt(7)$, which is not the case. This final contradiction completes the proof. 
\end{proof}

\begin{table}[!ht]
\caption{Values of $r$ in proof of Prop. \ref{parab}} \label{rvals}
\[
\begin{array}{|c|r|cccccccc|}
\hline
G(q)=E_8(q) & H=P_i, i= & 1&2&3&4&5&6&7&8 \\
                     & r= & 7&8&7&5&5&5&6&7 \\
\hline
G(q)=E_7(q) & H=P_i, i= & 1&2&3&4&5&6&7& \\
                     & r= & 6&7&5&4&4&5&6 & \\
\hline
G(q)=E_6(q) & H=P_i, i= & 1&2&3&4&16&35&& \\
                     & r= & 5&2&4&3&4&3 && \\
\hline
G(q)=\,{^2\!E_6}(q) & H=P_i, i= & 16&2&35&4&&&& \\
                     & r= & 3&2&3&2 &&&& \\
\hline
G(q)=F_4(q) & H=P_i, i= & 1&2&3&4&14&23&& \\
                     & r= & 2&2&3&3&2&2 && \\
\hline
G(q)=G_2(q) & H=P_i, i= & 2&&&&&&& \\
                     & r= & 2 &&&&&&& \\
\hline
G(q)=\,^3\!D_4(q) & H=P_i, i= & 134&&&&&&& \\
                     & r= & 2 &&&&&&& \\

\hline
\end{array}
\]
\end{table}

The remaining cases are resolved by {\tt magma} computations:

\begin{lem}\label{l: parab}
Let $G$ be as in Proposition $\ref{parab}$, and let $H$ be a maximal parabolic subgroup of $G$ as listed in Table~$\ref{ex}$.
Let $\O = (G:H)$. Then $(G,\O)$ is not binary.
\end{lem}
\begin{proof}
%We prove something stronger when the socle of $G$ is $G_2(q)$, for $q\le 5$. We have checked that every primitive faithful action of $G$ is not binary. The character tables of these groups are readily available in GAP and we used \cite[Lemma~2.7]{dgs_binary} to test whether these actions are binary. This test works except when $G=G_2(2)\cong U_3(3)$ and $H=L_2(7)$, or $G=G_2(5)$ and $H=2^3:L_3(2)$. In the first case we have constructed the primitive permutation action and we have checked that it is not binary witnessing two suitable triples. In the second case a direct application of \cite[Lemma~2.5]{dgs_binary} with the prime $p:=3$ yields that also this action is not binary. The other groups require detailed analysis.

Suppose first that $G(q)={\!^3\!D_4(q)}$; we refer to \cite{geck} for a description of the parabolic subgroup $H$ here (note that, although \cite{geck} assumes that $q$ is odd, \cite{himstedt2} confirms that the same description applies for $q$ even). 
%Let $d=(2,q-1)$, let $M\cap G(q)$ be a parabolic subgroup of $G(q)$, and 
Let $T$ be a maximal torus contained in $H\cap G(q)$ that is isomorphic to $C_{q-1}\times C_{q^3-1}$. 
%and which normalizes all root groups. 
We assume that $H\cap G(q)$ contains the Borel subgroup generated by all positive root subgroups. Let $\alpha$ (resp. $\beta$) be the short (resp. long) fundamental root, and let $U$ be the short root group $X_{-2\alpha-\beta}$; then $|U|=q^3$ and $U$ is not contained in $H$. What is more, \cite[Table~2.3]{geck} confirms that $T$ acts transitively on the non-identity elements of $U$. Define $\Gamma=\{Hu : u\in U\}$, a subset of $\Omega$ of order $q^3$. Then $U\rtimes T$ stabilizes $\Gamma$ and acts 2-transitively on $\Gamma$. Since $q\ge 3$, $G(q)$ does not contain a section isomorphic to $\Alt(q^3)$ by Lemma \ref{l: alt sections classical}. Therefore, $\Gamma$ is a beautiful subset and  Lemma~\ref{l: beautiful} yields the conclusion.

\smallskip

In the case where $G = E_6(2)$ or $E_6(2).2$ and $H=P_2$, we compute the index $|G:H|$ and we select the complex irreducible characters of $G$ having degree at most $|G:H|$. Then we find all non-negative integer linear combinations of these irreducible characters having degree $|G:H|$. These combinations are our putative permutation characters. Then, for each of these characters, we use Lemma~\ref{l: characters} to prove that the action under consideration is not binary. 

\smallskip

Finally, for $G(q):={^2}E_6(2)$, let $H$ be one of the parabolic subgroups in Table \ref{ex}. Then we see that $5$ divides both $|G:H|$ and $|H|$, but $5^2$ does not divide $|H|$. Moreover, $G$ contains a unique conjugacy class of elements of order $5$ (see \cite{atlas}). Therefore Lemma~\ref{l: M2} implies that the action of $G$ on $(G:H)$ is not binary.
\end{proof}

\section{Maximal rank subgroups}\label{maxrksec}

In this section we prove Theorem \ref{binex} in the case where the point stabilizer $H$ is a subgroup of maximal rank that is not the normalizer of a maximal torus in $G$. Such maximal subgroups are listed in Table 5.1 of \cite{LSS}. They will be listed in Tables \ref{e8maxrk} - \ref{3d4maxrk} below, where for notational convenience we list each possibility for $H$ as a ``type", which is a subgroup (usually equal to $H^{(\infty)}$) of small index in $H$.

Here is the main result of this section. The cases excluded in the proposition (those in Table \ref{maxrkex} and also the case of socle $^2\!F_4(q)$) will be handled later in Lemmas \ref{lines12}, \ref{lines34} and \ref{2f4soc}.

\begin{prop}\label{maxrk}
Assume $G$ is almost simple with socle $G(q)$, an exceptional group of Lie type over $\F_q$. Suppose $G(q)$ is not as in Proposition~$\ref{smallex}$, and suppose also that $G(q) \ne \,^2\!F_4(q)$. Let $H$ be a maximal  subgroup of maximal rank in $G$, as in ${{\cite[\textrm{Table}~5.1]{LSS}}}$, and let $\O = (G:H)$. Then either $(G,\O)$ is not binary, or $(G(q),H)$ is as in Table $\ref{maxrkex}$.
\end{prop}

%We remind the reader that Hypothesis~\ref{h: ex} was stated at the end of \S\ref{s: small} and that, in particular, Proposition~\ref{maxrk} does not cover certain ''small`` possibilities for $G$ (these having been dealt with already in \S\ref{s: small}).

\begin{table}[!ht]
\caption{Exceptions in Prop. \ref{maxrk}} \label{maxrkex}
\[
\begin{array}{|c|c|}
\hline
G(q) & \hbox{type of }H  \\
\hline
E_7(q) & A_1(q^7) \\
E_6^\e(q) & A_2^\e(q^3) \\
E_8(2) & A_2^-(2)^4, A_2^-(2^4) \\
%E_7(q) & A_7^-(q) & q=2,3 \\
%E_6(q) & A_2(q^2)A_2^-(q) & 2 \\
{^2\!E_6}(2) & A_2^-(2)^3,D_4(2)T_2 \\
%F_4(2) & A_1(2)C_3(2),A_2^-(2)^2,B_2(2)^2,B_2(2^2)  \\
%          & B_4(2),D_4(2)  \\
%^3\!D_4(2) & A_1(2)A_1(2^3),A_2^\e(2)T_2  \\
%^2\!F_4(q) & Sp_4(q)  & \hbox{any} \\
\hline
\end{array}
\]
\end{table}

\begin{proof}
We adopt the same method as in the previous section, using Lemma \ref{aff}. 
%Suppose $G(q) \ne \,^2\!F_4(q)$ (we shall handle this case at the end of the proof). 
In Tables \ref{e8maxrk} - \ref{3d4maxrk} we have listed the possibilities for $H$, together with a subgroup $A \cong \SL_r(q^u)$ of $H$ (where $u=1$ or 2), such that $A$ is contained in a subgroup $S \cong \SL_{r+1}(q^u)/Z$ of $G$ that does not lie in $H$. We shall justify these assertions below.
%we also point out that in three cases in the tables (one for each of $E_7(q)$, $E_6(q)$ and ${^2\!E_6}(q)$), we have not defined a subgroup $A$, and a different argument will be given below.

\begin{table}[!ht]
\caption{Subgroups $H$ and $A$ for $E_8(q)$} \label{e8maxrk}
\[
\begin{array}{|r|ccccccc|}
\hline
\hbox{type of }H & D_8(q) & A_1(q)E_7(q) & A_8(q) & A_2(q)E_6(q) & A_4(q)^2 & D_4(q)^2 & A_2(q)^4  \\
A & \SL_8(q) &  \SL_7(q) & \SL_7(q) & \SL_6(q) & \SL_5(q) & \SL_4(q) & \SL_3(q)  \\ 
\hline
\hline
\hbox{type of }H &  A_1(q)^8(q>2) & A_8^-(q) & A_2^-(q)E_6^-(q) & A_4^-(q)^2 & A_4^-(q^2) & D_4(q^2) & {^3\!D_4}(q)^2  \\
A & \SL_2(q) & \SL_3(q^2) &  \SL_3(q^2) &\SL_2(q^2) &\SL_2(q^2) &\SL_2(q^2) &\SL_3(q)  \\
\hline
\hline 
\hbox{type of }H &  {^3\!D_4}(q^2)  & A_2^-(q)^4 & A_2^-(q^2)^2 & A_2^-(q^4) &&& \\
A &  \SL_3(q^2)  & \SL_2(q)  & \SL_2(q^2)  & \SL_2(q) &&& \\
\hline
\end{array}
\]
\end{table}

\begin{table}[ht!]
\caption{Subgroups $H$ and $A$ for $E_7(q)$} \label{e7maxrk}
\[
\begin{array}{|r|cccccc|}
\hline
\hbox{type of }H & A_1(q)D_6(q) & A_7(q) & A_2(q)A_5(q) & A_1(q)^3D_4(q) & A_1(q)^7(q>2) & E_6(q)T_1  \\
A & \SL_6(q) &  \SL_5(q) & \SL_5(q) & \SL_4(q) & \SL_2(q) & \SL_6(q)  \\ 
\hline
\hline
\hbox{type of }H &  A_7^-(q) & A_2^-(q)A_5^-(q) & A_1(q^3)\,^3\!D_4(q) & E_6^-(q)T_1 &  & \\
A & \SL_2(q^2) & \SL_3(q^2) &  \SL_3(q) &\SL_3(q^2) &- & \\
\hline
\end{array}
\]
\end{table}

\begin{table}[ht!]
\caption{Subgroups $H$ and $A$ for $E_6(q)$} \label{e6maxrk}
\[
\begin{array}{|r|cccccc|}
\hline
\hbox{type of }H & A_1(q)A_5(q) & A_2(q)^3 & A_2(q^2)A_2^-(q) & D_4(q)T_2 & {^3\!D_4}(q)T_2 & D_5(q)T_1  \\
A & \SL_4(q) &  \SL_3(q) & \SL_2(q^2) & \SL_4(q) & \SL_3(q) & \SL_5(q)  \\ 
\hline
\end{array}
\]
\end{table}

\begin{table}[ht!]
\caption{Subgroups $H$ and $A$ for ${{^2\!E_6}}(q)$} \label{2e6maxrk}
\[
\begin{array}{|r|ccccc|}
\hline
\hbox{type of }H & A_1(q)A_5^-(q) & A_2^-(q)^3 & A_2(q^2)A_2(q) & D_4(q)T_2 &  D_5^-(q)T_1 \\
A & \SL_2(q^2) &  \SL_2(q) & \SL_3(q) & \SL_3(q) & \SL_2(q^2)  \\ 
\hline
\end{array}
\]
\end{table}

\begin{table}[ht!]
\caption{Subgroups $H$ and $A$ for $F_4(q)$} \label{f4maxrk}
\[
\begin{array}{|r|cccccc|}
\hline
\hbox{type of }H & A_1(q)C_3(q) & B_4(q) & D_4(q) & {^3\!D_4}(q) &  A_2(q)^2 & A_2^-(q)^2  \\
A & \SL_2(q) &  \SL_2(q) & \SL_2(q) & \SL_3(q) & \SL_3(q) & \SL_2(q)  \\ 
\hline
\hline
\hbox{type of }H & B_2(q)^2 & B_2(q^2) &&&& \\
A & \SL_2(q) & \SL_2(q) &&&& \\
\hline
\end{array}
\]
\end{table}

\begin{table}[ht!]
\caption{Subgroups $H$ and $A$ for $G_2(q)$} \label{g2maxrk}
\[
\begin{array}{|r|ccc|}
\hline
\hbox{type of }H & A_1(q)^2 & A_2(q) & A_2^-(q)   \\
A & \SL_2(q) &  \SL_2(q) & \SL_2(q)   \\ 
\hline
\end{array}
\]
\end{table}

\begin{table}[ht!]
\caption{Subgroups $H$ and $A$ for ${^3\!D_4}(q)$} \label{3d4maxrk}
\[
\begin{array}{|r|ccc|}
\hline
\hbox{type of }H & A_1(q)A_1(q^3) & A_2(q) & A_2^-(q)   \\
A & \SL_2(q) &  \SL_2(q) & \SL_2(q)   \\ 
\hline
\end{array}
\]
\end{table}

Given the assertions on the tables, the argument proceeds as in the proof of Proposition \ref{parab}: Lemma~\ref{aff} produces a subset $\D$ of $\O$ of size $q^{ru}$, for which $G^\D$ contains $\ASL_r(q^u)$. If $G^\D \ge \Alt(\D)$, then $q^{ru} \le N_G$ (as defined in Lemma \ref{altsec}), and also $\Alt(q^{ru}-1)$ must be a section of $H$. By Lemmas~\ref{l: alt sections classical} and~\ref{altsec}, this eliminates all possibilities except for the list in Table \ref{maxrkex}, together with the following cases:
\begin{equation}\label{smallcas}
\begin{array}{r|c}
G &  F_4(q),\,q=3 \\
\hline
H  & B_4(q),D_4(q) \\
\end{array}
\end{equation}

We shall handle the cases in (\ref{smallcas}) after first justifying the assertions in Tables \ref{e8maxrk} - \ref{3d4maxrk}. For the cases in the tables where the maximal rank subgroup $H$ is just an untwisted subsystem subgroup over $\F_q$, the existence of the subgroups $A<S$ is clear from inspection of the extended Dynkin diagram of $G$.  (The cases $(A,H) = (\SL_2(q),D_4(q))$ for $F_4(q)$, and also $(A,H) = (\SL_2(q),A_2(q))$ for $G_2(q)$ and ${^3\!D_4}(q)$ require some small additional observations: 
in the first case, $D_4(q)$ contains a subgroup $A=\SL_2(q)$ corresponding to a short root in the $F_4$-system, and this lies in a short root $S=\SL_3(q)$ which is not contained in $D_4(q)$; and in the second case, there exists $x \in C_G(A)\setminus H$, and we can take $S = H^x$.)

Now consider cases in Tables \ref{e8maxrk} - \ref{3d4maxrk} where $H$ involves a twisted group, or a group over a proper extension field of $\F_q$.

Consider first Table \ref{e8maxrk}, where $G(q) = E_8(q)$. In the cases where $H$ is of type $^3\!D_4(q)^2$ or $A_2^-(q)^4$, we choose $A$ to be a subsystem subgroup $\SL_3(q)$ or $\SL_2(q)$ of one of the factors. Now suppose $H$ is of type $A_8^-(q)$. Then $H$ has a Levi subgroup $S = \SL_4(q^2)$, and we let $A$ be a natural subgroup $\SL_3(q^2)$ of this. We use Lemma \ref{factn} to show that there is a conjugate $S^x$ such that $A < S^x \not \le H$. 
First observe that the fusion control hypotheses of the lemma for $A<S<H$ clearly hold. 
Now $N_G(A)$ contains a subgroup $A_2(q)A_2^-(q)$ (a subgroup $A_2(q^2)A_2(q)A_2^-(q)$ can be seen inside a subsystem subgroup of type $E_6A_2$), whereas $N_H(A)$ normalizes a subgroup $A_2^-(q)T_2A$ of $H$, where $T_2$ is a torus of order $q^2-1$. The factor $A_2(q)$ of $N_G(A)$ does not have a factorization with one of the factors being $N(T_2)$ (see \cite{LPS}); hence 
$N_G(A) \ne N_H(A)\,(N_G(S) \cap N_G(A))$ and the required conjugate of $S$ exists by Lemma \ref{factn}. 

Next suppose $H$ is of type $A_2^-(q)E_6^-(q)$. Here we choose $A$ to be a subgroup $\SL_3(q^2)$ of a Levi subgroup $\SU_6(q)$ of the $E_6^-(q)$ factor; this is contained in a subgroup $S = \SL_4(q^2)$ as defined in the previous paragraph, and $S \not \le H$. A similar argument applies to produce a suitable subgroup $A=\SL_2(q^2)$ when $H$ has type $A_4^-(q)^2$, and also a subgroup $A=\SL_3(q^2)$ when $H$ has type $^3\!D_4(q^2)$. In the case where $H$ is of type $A_4^-(q^2)$, 
we choose $A$ to be a subgroup $\SL_2(q^2)$ corresponding to a natural subgroup $\SU_2(q^2)$ of the unitary group; this arises from a subsystem $A_1A_1$ of the ambient algebraic group, and is conjugate to the subgroup $\SL_2(q^2)$ of the previous case. The same subgroup $A=\SL_2(q^2)$ pertains when $H$ is of type $A_2^-(q^2)^2$ or $D_4(q^2)$. In the latter case, we also need to apply Lemma \ref{factn} to produce a subgroup $S = \SL_3(q^2)$ such that $A<S \not \le H$: here $N_G(A)$ contains $D_6^-(q)$, which does not factorize as $N_H(A)\,(N_G(S) \cap N_G(A))$.

Finally, for $H$ of type $A_2^-(q^4)$, let $A$ be a natural subgroup $\SL_2(q)$ of $A_2^-(q^4)$ (acting as $2\oplus 1$ on the associated 3-dimensional unitary module). Then $A$ is a diagonal subgroup of a subsystem subgroup of type $A_1(q)^4$ which lies in a subsystem $A_2(q)^4$, and hence $A$ lies in a diagonal $A_2(q)$ in the latter. This completes the justification for Table \ref{e8maxrk}. 

For $G(q)=E_7(q)$, $E_6^\e(q)$, $F_4(q)$ or $G_2(q)$ the justification for the existence of the subgroups $A<S$ uses the same arguments as above. Extra argument using Lemma \ref{factn} is needed just for the cases 
\begin{center}$(G(q),H) = (E_7(q), A_7^-(q))$, $(E_6(q),\,A_2(q^2)A_2^-(q))$ and $(^2\!E_6(q), A_1(q)A_5^-(q))$;
\end{center} observe that $C_G(A)$ contains $^2\!D_4(q)A_1(q)$, $^2\!A_3(q)$ or $A_3(q)$ in the respective cases, from which it can be seen that $N_G(A)$ does not factorize as $N_H(A)\,(N_G(S) \cap N_G(A))$, so that Lemma \ref{factn} applies.

We have now justified all the assertions in Tables \ref{e8maxrk} - \ref{3d4maxrk}. 

%It remains to handle the cases in (\ref{smallcas}). First let $(G(q),H,q) = (E_7(q),A_7^-(q),2 \hbox{ or }3)$. Choose subgroups $A<S<H$ with $A\cong \SL_2(q^2), S\cong \SL_3(q^2)$. Regarding $A$ as $\O_4^-(q)$ in a subsystem subgroup $D_7(q)$, we see that $C_G(A)$ contains $D_4^-(q)$, while $C_H(A)= \GU_4(q)$, and hence using \cite{LPS} it follows that $N_G(A) \ne N_H(A)\,(N_G(S) \cap N_G(A))$. Therefore Lemma \ref{factn} implies that there exists $x \in N_G(A)$ such that $S^x \not \le H$. Hence $A < S^x \not \le H$, and so by Lemma \ref{aff} there is a subset $\D$ of $\O$ such that $G^\D\ge \ASL_2(q^2)$. Now Lemma \ref{altsec} gives a contradiction. We obtain a similar contradiction using subgroups $\SL_2(q^2)<A_2(q^2)$ in $H$ for the $E_6(q)$ and ${{^2\!E_6}}(q)$ cases in (\ref{smallcas}). 

It remains to handle the cases in (\ref{smallcas}). Let $G=F_4(3)$, and let $H$ be a maximal rank subgroup $B_4(3)$ or $D_4(3).\Sym(3)$. First consider the case where $H = D_4(3).\Sym(3)$. Let $S = \SL_4(3) < H$ be generated by root subgroups, and 
$A = \SL_3(3)<S$. Then $A<S<H$, and $H$ controls fusion of $S$ in $G$ (as all subgroups $\SL_4(3)$ generated by root groups in $H$ are $H$-conjugate). We claim that $N_G(A)$ does not factorize as $N_H(A)\,(N_G(S) \cap N_G(A))$. To see this, observe that $N_G(A)/A$ contains $\tilde A_2(3)$ (generated by short root groups); while $|N_H(A)/A|_3 = 3$ and 
$|(N_G(S) \cap N_G(A))/A|_3 = |N_{A_1(3)\,S}(A)/A|_3 = 3$, proving the claim. It then follows from Lemma \ref{factn} that there is a conjugate $S^g$ such that $A<S^g \not \le H$, so as usual there is a subset $\Delta$ with $G^\Delta \ge \ASL_3(3)$, showing that $(G,(G:H))$ is not binary.

Now consider the case $H=B_4(3)$. Again take $A<S<H$ with $A=\SL_3(3)$, $S = \SL_4(3)$ generated by root subgroups. This time $H$ does not control fusion of $S$ in $G$, as there are two classes of subgroups isomorphic to $\SL_4(3)$ in $B_4(3)$ with representatives $S_1$, $S_2$ of types $\SL_4(3)$ and $\O_6^+(3)$, respectively. We again aim to find a conjugate $S^g$ such that $A<S^g \not \le H$, but we need to do this a little differently. Define
\[
\begin{array}{l}
\Lambda = \{R<G : A\le R, R \hbox{ conjugate to }S \hbox{ in }G\}, \\
\Phi=\{R \in \Lambda : R < H\}.
\end{array}
\]
We shall show that $|\Lambda| > |\Phi|$, which will achieve our aim, completing the proof that the action of $G$ on $(G:H)$ is not binary.

First observe that $N_G(A)$ acts transitively on $\Lambda$, since 
\[
\begin{array}{ll}
R \in \Lambda & \Rightarrow R=S^g\,(g\in G) \\
                     & \Rightarrow A,A^{g^{-1}} < S \\
                     & \Rightarrow A^{g^{-1}} = A^s \,(s \in S) \\
                     & \Rightarrow R = S^{sg},\,sg \in N_G(A).
\end{array}
\]
Hence $|\Lambda| = |N_G(A):N_G(A)\cap N_G(S)| = |\tilde A_2(3).2:T_1A_1(3).2|$, which is divisible by $3^2.13$.

In similar fashion, we see that $N_H(A)$ has two orbits $\Phi_1$, $\Phi_2$ on $\Phi$, with orbit representatives $S_1$ and $S_2$. The orbit sizes are $|\Phi_i| = |N_H(A):N_H(A)\cap N_H(S_i)|$ for $i=1,2$. Hence $|\Phi_1| = |T_1A_1(3).2:T_2.2|$ divides 24, while $|\Phi_2| = 1$. Therefore $|\Lambda| > |\Phi|$, as required. 
\end{proof}

The next three results deal with the cases not covered by Proposition \ref{maxrk} (those in Table \ref{maxrkex} and also the $^2\!F_4(q)$ case).

%Finally in the $F_4(q)$ cases, choose $A<S<H$ with $A\cong \SL_3(q)$, $S\cong A_3(q)$. Then $N_G(A)$ contains a subgroup $A_2(q) \tilde A_2(q)$ and again does not factorize as $N_H(A)\,(N_G(S) \cap N_G(A))$, so as above there is a subset $\D$ with $G^\D \ge \ASL_3(q)$ (where $q=3$), again contradicting Lemma \ref{altsec}. 
\begin{lem} \label{lines12}
Let $G$ be as in Proposition $\ref{maxrk}$, and suppose that $(G(q),H)$ is as in line $1$ or $2$ of Table~$\ref{maxrkex}$. If $\O=(G:H)$, then $(G,\O)$ is not binary.
\end{lem}

\begin{proof}
In these cases  $G(q) = E_7(q)$ or $E_6^\e(q)$ ($\e = \pm$), and $H$ is of type $A_1(q^7)$ or $A_2^\e(q^3)$, respectively. 
In the first case, $H\cap G(q) = A_1(q^7).7$, and we choose a subfield subgroup $X = A_1(q)\times 7 = X_0 \times 7$ of $H$. The factor $X_0 = A_1(q)$ is contained diagonally in a subsystem subgroup $A_1(q)^7$ of $E_7(q)$, which has normalizer acting as $L_3(2)$ on the 7 factors (see \cite[Table 5.1]{LSS}). Hence $N_G(X)$ is not contained in $H$, and so there is a suborbit here on which $H$ acts as $(H:Y)$, where $X \leq Y \leq N_H(X_0)$. If $q>2$, then $Y$ is maximal in $\langle H\cap G(q),Y\rangle$, and we know, by \cite{ghs_binary} that the associated action on cosets is not binary. Appealing to Lemma~\ref{l: subgroup} if necessary, we conclude that the action of $H$ on $(H:Y)$ is not binary, and now Lemma~\ref{l: point stabilizer} implies that the action of $G$ on $(G:H)$ is not binary. Suppose now $q=2$. Let $x$ be an element of order $7$ in $H$. From~\cite[Table~$2$]{bbr}, we see that there exists $g\in C_G(x)\setminus H$ and hence $H\cap H^g$ is a proper subgroup of $H$ containing $x$.
We have calculated with {\tt magma} the faithful transitive actions of $H$ having point stabiliser of order divisible by $7$. We find that  all such actions are  not binary. Therefore, the action of $H$ on $(H:H\cap H^g)$ is not binary, and hence by Lemma~\ref{l: point stabilizer} so is the action of $G$ on $(G:\Omega)$.
%Now we use Lemma~\ref{l: higman} which asserts that, since the action of $G$ on $H$ is primitive, each section of $H$
%appears as a section of $H^\Delta$ for any non-trivial suborbit $\Delta$. If $\Delta$ has size at most 7, then we conclude that
%all sections of $H$ are isomorphic to a section of $\Sym(7)$ which is a contradiction. We conclude, therefore,
%that if the action of $G$ on $(G : H)$ is binary, then all non-trivial subdegrees must be even or equal to 16383. This contradicts
%the fact that $|G : H|$ is even.

For the $E_6^\e(q)$ cases, we argue similarly. Choose a subfield subgroup $X=A_2^\e(q)\times 3$ of $H$. The $A_2^\e(q)$ factor is contained diagonally in a subsystem $A_2^\e(q)^3$, which has normalizer acting as $\Sym(3)$ on the factors. Hence again $N_G(X) \not \le H$ and so there is a suborbit here on which $H$ acts as $(H:Y)$ where $X\leq Y \leq N_H(A_2^\e(q))$. It will be sufficient to show that this action is not binary.

If $\e=+$, then we take a subgroup $S=\SL_2(q)$ of $Y$, and it is easy to verify that $S$ normalizes and acts transitively on an elementary abelian subgroup $E=E_{q^2}$ of $H$ that is not contained in $Y$. Defining $\Delta=\{Ye:e \in E\}$, we obtain, in the usual way, that either $\Delta$ is a beautiful subset in the action of $H$ on $(H:Y)$ (and we are done), or else $H^\Delta \ge \Alt(\Delta)$. But by Lemma \ref{altsec}, this is not possible unless $q=2$. When $q=2$,  for $G$ with $E_6(2)\le G\le \Aut(E_6(2))$ and for $H:=N_G(A_2(8))$, we have computed the subgroups $K$ of $H$ with $|H:K|$ odd. For each such pair $(H,K)$, we have checked that, if the action of $H$ on $(H:K)$ is binary, then $K$ contains $A_2(8)$. Hence there is no binary action of $H$ of odd degree with $A_2(8)$ acting non-trivially. From this we deduce that, if the action of $G$ on $(G:H)$ is binary, then each non-trivial subdegree of $G$ must be even, which implies that $|G:H|$ is odd, a contradiction. 
%(See also Lemma~\ref{l: alot}.)

 If $\epsilon=-$, then we argue similarly with a cyclic subgroup $K$ of order $q-1$ in $Y$. There is a subgroup $E=E_q$ that $K$ normalizes and upon which it acts fixed-point-freely. The same line of argument rules out all cases with $q>8$. Also, using Lemma \ref{l: alt sections classical}, we can rule out $q=8$ (since $\Alt(7)$ is not a section of $\SU_3(8))$ and $q=7$ (since $\Alt(6)$ is not a section of $\SU_3(7)$). To deal with $q<7$, we use \magma as follows.
 
Let $M$ be a maximal subgroup of $G$ with socle ${^2\!A_2}(q^3)$. We consider the permutation action of $G$ on the right cosets $\Omega$ of $M$ in $G$. Observe that $q$ divides $|\Omega|=|G:M|$ and hence $G$ has a suborbit of cardinality relatively prime to $q$. Using {\tt magma}, we have verified that all faithful transitive actions of $M$ on a set of cardinality relatively prime to $q$ are non binary. Therefore, the action of $M$ on each non-trivial suborbit is non binary.    From this, it follows that $G$ is not binary on $\Omega$. %Let $S_\alpha$ be the socle of $G_\alpha$ so that $S_\alpha\cong {^2\!A_2}(q^3)$. Since $|\beta^{G_\alpha}|$ is relatively prime to $q$, we see that $S_{\alpha\beta}$ contains a Sylow $2$-subgroup of $S_\alpha$ when $q=4$ and a Sylow $5$-subgroup of $S_\alpha$ when $q=5$. We consider in $S_\alpha$ a subgroup $X_\alpha\cong {{^2\!A_2}(q)}$. We are not sure about the exact partition of $\beta^{G_\alpha}$ into $X_\alpha$-orbits, nevertheless there exists a non-trivial orbit of $X_\alpha$ on $\beta^{G_\alpha}$ of cardinality relatively prime to $q$. Let $\Delta$ be such orbit.

\end{proof}

\begin{lem} \label{lines34}
Let $G$ be as in Proposition $\ref{maxrk}$, and suppose that $(G(q),H)$ is as in line $3$ or $4$ of Table~$\ref{maxrkex}$. If $\O=(G:H)$, then $(G,\O)$ is not binary.
\end{lem}

\begin{proof}
Adopt the hypothesis of the lemma, and assume that $(G,\O)$ is binary.

Suppose, first, that $G(q)={{^2\!E_6}(2)}$ and that $H$ is of type $A_2^-(2)^3$. An inspection of extended Dynkin diagrams confirms that $H$ contains a subgroup $K$ of type $A_2^-(2)^2$ that, in turn, is embedded in the natural way inside a subgroup $L=\PSU_6(2)$ in $G(q)$. Now $K$ has a ``diagonal" subgroup $Q\cong  Q_8$ which normalizes, and acts fixed-point-freely upon an elementary abelian subgroup, $E_9$, of $L$. Now we study the conjugates of $E_9$ under $N_L(Q)$ and observe that $Q$ normalizes each of these conjugates. There are two possibilities: first, one of these conjugate subgroups, say $E$, does not lie in $H$. In this case, $E\cap H=\{1\}$ and setting $\Delta=H^E$ we see that the set-wise-stabilizer of $\Delta$ acts $2$-transitively on $\Delta$. Hence, as $(G,\O)$ is binary, we have $G^\Delta \ge \Alt(\Delta)$. However $|\Delta|=9$, and $H$ does not have a section isomorphic to $\Alt(8)$, so this is impossible. This possibility is, therefore, excluded, and we conclude that all of the conjugate subgroups lie in $H$. But now direct calculation, using for instance {\tt GAP}, confirms that $\langle E_9 ^g \mid g \in N_L(Q)\rangle=L$, which is a contradiction. Thus this possibility is also excluded.

Consider, next, the situation where $G(q)=E_8(2)$ and $H$ is of type $A_2^-(2)^4$. In this case a version of the previous argument yields a contradiction, this time using a subgroup $K$ of type $A_2^-(2)^3$ embedded in a subgroup isomorphic to $L=\PSU_9(2)$ in $G(q)$. 
%Once again, we use the fact that $H$ is solvable, and we obtain a contradiction.

Suppose, next, that $G(q)=E_8(2)$ and $H$ has type $A_2^-(2^4)$. In this case, $G=E_8(2)$ and $H\cong \Aut(\PSU_3(16)) = \PSU_3(16).8$ (see \cite{LSS}). We have computed all the core-free subgroups $K$ of $H$ with $|H:K|$ odd and shown that, for each of these subgroups, the action of $H$ on $(H:K)$ is not binary, by witnessing a non-binary triple. Hence, as $(G,\O)$ is binary (by assumption), $|G:H|$ must be odd, which is clearly a contradiction.

Suppose, finally, that $G(q)={{^2\!E_6}(2)}$ and that $H$ is of type $D_4(2)T_2$. Consulting \cite{Wilson2E62}, we see that $H\cap G(q)$ has shape $(3\times \O_8^+(2):3):2$, extending to $(3^2:2\times \O_8^+(2)):\Sym(3)$ in $G(q).\Sym(3)=\Aut(G(q))$. We have calculated the transitive actions of all groups of the relevant shapes on sets of odd cardinality using {\tt magma}. We find that the only binary actions for such groups occur when the set is of size $1,3$ or $9$, in which case the kernel of the action contains $\O_8^+(2)$.  As $(G,\O)$ is binary, we conclude, therefore, that all non-trivial subdegrees must be even. This contradicts  the fact that $|G:H|$ is even.
\end{proof}

\begin{lem} \label{2f4soc}
Let $G$ be as in Proposition $\ref{maxrk}$, and suppose that $G(q) = \,^2\!F_4(q)$ $(q>2)$. If $\O=(G:H)$, then $(G,\O)$ is not binary.
\end{lem}

\begin{proof}
Here the possibilities for the maximal rank subgroup $H(q):=H\cap G(q)$ are:
\[
\SU_3(q).2,\, \PGU_3(q).2,\, (Sz(q)\times Sz(q)).2,\, \Sp_4(q).2.
\]
We write $G = G(q)\langle \phi \rangle$, where $\phi$ is a field automorphism of odd order $f$ (possibly $f=1$).

In the first two cases, let $T<H(q)$ be a maximal torus of order $(q+1)^2$. Then $N_{H(q)}(T) = T.(\Sym(3)\times 2)$, while $N_{G(q)}(T) = T.\GL_2(3)$. Hence $N_{H(q)}(T.2^2) = T.2^2$ and $N_{G(q)}(T.2^2) = T.D_8$. It follows that there exists $x \in G\setminus H$ such that $H\cap H^x$ contains $T.2^2$ but does not contain $H$. We can chose $\phi$ to normalize all of the above subgroups, so that we also have $H\cap H^x \ge T.(2^2\times f)$. Now Lemma~\ref{l: psu3 element2} implies that the action of $H$ on $(H:H\cap H^x)$ is not binary, and the result follows from Lemma~\ref{l: point stabilizer}.

Next consider the possibility that $H(q)=(Sz(q)\times Sz(q)).2$. Here, as we shall see, when we study suborbits we find a primitive group via a product action hence we could appeal to \cite{wiscons}; nonetheless we give a direct argument. First, let $T<H(q)$ be the direct product of two maximal tori of $Sz(q)$ of order $q-1$. Then $N_{H(q)}(T)=T.[8]$, while $N_{G(T)}(T)\geq T.[16]$ (one can see this, for instance, by using the fact that $T$ is also a subgroup of $\Sp_4(q).2$). Again we obtain an $x\in G\setminus H$ such that $H\cap H^x$ contains $T.2^2$ but does not contain $H$. Choosing $\phi$ appropriately we have $H\cap H^x\ge T.([8]\times f)$. Then it must be the case that $H\cap H^x = T.([8]\times f)$. In particular, we can write $H\cap H^x = (M\times M).(2\times f)$, where $M$ is a maximal subgroup of $Sz(q)$ of order $2(q-1)$. 

Identify $(Sz(q):M)$ with the set of conjugates of $M$ in $Sz(q)$ and identify $(H:H\cap H^x)$ with the set 
\[
 \Gamma:=\{(M_1,M_2) \mid M_1, M_2 \in (Sz(q):M)\}.
\]
Now we fix $M$ and define
\[
 \Gamma_0:=\{(M, M_1) \mid M_1\in (Sz(q):M)\}.
\]
Clearly $H_{\Gamma_0}\cong (M\times S).f$ and $H^{\Gamma_0}$ is almost simple and isomorphic to $Sz(q).f$, with the action on $\Gamma$ being isomorphic to the action of $Sz(q).f$ on $(Sz(q):M)$. We know that this action is not binary by \cite[Theorem 1.3]{ghs_binary}; thus there exist $k$-tuples $I=(M_1,\dots, M_k)$ and $J=(M_1',\dots, M_k')$ such that $(I,J)$ is $2$-subtuple complete but not $k$-subtuple complete with respect to the action of $Sz(q).f$. Now the same is true for the pair of elements of $\Gamma^k$,
\[
 \Big(  \big( (M,M_1), (M,M_2),\dots, (M,M_k) \big), \big( (M,M_1'), (M,M_2'),\dots, (M,M_k')  \big)\Big),
\]
with respect to the action of $H$. Now, the result follows from Lemma~\ref{l: point stabilizer}.

Consider, finally, the possibility that $H(q)=\Sp_4(q).2$. In this case we use the fact that $H(q)$ contains an element $g$ of order $q-1$ that is centralized in $G(q)$ by a subgroup isomorphic to ${^2\!B_2}(q)$. 
In \cite[Table IV]{shinoda} a parametrization of such elements $g$ is given: they are conjugate to the element $t_1$ in the table. Working in the $F_4$ root system, it can be seen that there is a conjugate $g$  of $t_1$ that can be written in $H'=\Sp_4(q)$ as a diagonal matrix with all its  eigenvalues of order $q-1$.
In particular, there is an element $x\in G\setminus H$ that centralizes $g$ and so we conclude that there is a suborbit of $G$ on which the action of $H$ is isomorphic to the action of $H$ on $(H:M)$, where $M$ is a subgroup of $H$ containing $g$. Since, by assumption, $q\geq 8$, Lemma~\ref{l: sp4 element} implies that this action is not binary, and the result follows by Lemma~\ref{l: point stabilizer}.
\end{proof}

\section{Maximal torus normalizers}\label{tori}

In this section we prove Theorem \ref{binex} in the case where the point stabilizer $H$ is the normalizer of a maximal torus. Such maximal subgroups are listed in Table 5.2 of \cite{LSS}. The main result of the section follows. The cases excluded in the proposition (those in Table~\ref{nortorex}) will be dealt with in Lemma~\ref{lastcase}.

\begin{prop}\label{nortor}
Assume $G$ is almost simple with socle $G(q)$, an exceptional group of Lie type over $\F_q$, and suppose $G(q)$ is not as in Proposition~$\ref{smallex}$. Let $H$ be a maximal subgroup of $G$ that is the normalizer of a maximal torus $T$, as in $\cite[\textrm{Table}~5.2]{LSS}$, and let $\O = (G:H)$. Then either $(G,\O)$ is not binary, or $(G(q), |T \cap G(q)|)$ is as in Table $\ref{nortorex}$.
\end{prop}

%We remind the reader that Hypothesis~\ref{h: ex} was stated at the end of \S\ref{s: small} and that, in particular, Proposition~\ref{nortor} does not cover certain ''small`` possibilities for $G$ (these having been dealt with already in \S\ref{s: small}).

\begin{table}[ht]
\caption{Exceptions in Prop. \ref{nortor}} \label{nortorex}
\[
\begin{array}{|c|c|}
\hline
G(q) &|T\cap G(q)| \\
\hline
%E_8(q) & (q+1)^8 & q=2,3 \\
E_7(2) & 3^7 \\
{^2\!E_6}(2) & 3^5 \\
\hline
\end{array}
\]
\end{table}

For the proof we need the following lemma. In the statement, by a semisimple group we mean a perfect group that is a central product of quasisimple groups. 

\begin{lem}\label{torlem}
Let $G$ be an almost simple group with socle of Lie type, and let $H = N_G(T)$ be a maximal subgroup of $G$ that is the normalizer of a maximal torus $T$. Write $\O = (G:H)$. 
\begin{itemize}
\item[(A)] Suppose there exist subgroups $A,D$ of $G$ with the following properties:
\begin{itemize}
\item[{\rm (i)}] $A$ is quasisimple, $D$ is either semisimple or a torus containing $Z(A)$, $[A,D]=1$ and $C_G(A)=DZ(A)$;
\item[{\rm (ii)}] $T \le N_G(A)$ and $T\cap AD = T_1T_0$, where $T_1=T\cap A$, $T_0 = T\cap D$;
\item[{\rm (iii)}] $C_G(T_0)'=A$.
\end{itemize}
Define
\[
\D = \{T^g : g \in N_D(T_0)A\} \subseteq \O.
\]
Then $G^\D$ has socle $A/Z(A)$ acting on $(A:N_A(T_1))$. 

\item[(B)] Suppose that in addition to (i)-(iii) above, the following hold:
\begin{itemize}
\item[{\rm (iv)}] $C_G(T_1T_0)=T$;
\item[{\rm (v)}] for any distinct $T',T'' \in \D$ we have $T'\cap T'' = \bigcap_{a\in A}T^a$;
\item[{\rm (vi)}] for any $g \in N_G(A)$, there exists $a\in A$ such that $T_1^g=T_1^a$;
\item[{\rm (vii)}] the action of $G^\D$ on $\D$ is not binary.
\end{itemize}
Then the action of $G$ on $\O = (G:H)$ is not binary.
\end{itemize}
\end{lem}

\begin{proof}
(A) Write $K = G_{(\D)}$, the point-wise stabilizer of $\D$. We claim first that $K$ normalizes $A$. To see this, observe first that $K$ normalizes $X:=\cap_{a\in A}T^a$. Since $X$ is $A$-invariant and $A$ is quasisimple, $X\cap A = Z(A)$. Also, 
$X \le T$ and so $X$ normalizes $A$ by (ii). Hence $[X,A] \le X\cap A \le Z(A)$. As $A$ is perfect, this
 implies that $[X,A]=1$, and hence $X \le DZ(A)$ by (i). It follows that $X=T_0Z(A)$. By (iii) we have $C_G(X)'=A$, and hence $K$ normalizes $A$, as claimed. 

Next, we claim that 
\begin{equation}\label{cgk}
C_G(K)'=A.
\end{equation}
Clearly $T_0 \le K$, so $C_G(K)'\le C_G(T_0)'=A$, by (iv). For the reverse containment, let $x \in K$. Then $T^{ax}=T^a$ for all $a \in A$. Now $x$ normalizes $A$, hence normalizes $T^a\cap A = T_1^a$ for all $a\in A$. In other words, $x$ induces an automorphism of $A$ that lies in the kernel, $L$ say, of the action on the set of $A$-conjugates of $T_1$.
As $A$ is quasisimple, either $L\le Z(A)$ or $A\le L$. If $L\le Z(A)$, then $x$ commutes with $A$ and hence~(\ref{cgk}) holds. So assume the latter. Since the action in question is on $A$-conjugates of $T_1$ and $A\le L$, we get $T_1\unlhd A$. As $A$ is quasisimple, this means that $T_1\le Z(A)$. Therefore, $A$ centralizes $T$, a maximal torus of $G$. But then $A$ must be in the centre of $G$ which is a contradiction to the fact that $A$ is quasisimple. Summing up, in all cases $x$ commutes with $A$ and  hence (\ref{cgk}) holds. 

Now $G_\D$ normalizes $K = G_{(\D)}$, hence normalizes $A$, by (\ref{cgk}). Therefore by (i), $G_\D$ also normalizes  
 $DA$. Let $g \in G_\D$. Then $T^g \in \D$, so by definition of $\D$, intersecting with $DA$ gives $(T_0T_1)^g = T_0T_1^a$ for some $a\in A$, and since $g \in N_G(DA)$ this implies that $T_0^g=T_0$. Hence $G_\D \le N_G(T_0)$ and 
$G_\D\cap DA = N_D(T_0)A$. As $K \ge N_D(T_0)Z(A)$, it follows that $G^\D = G_\D/K$ has socle $A/Z(A)$ acting on the conjugates of $T_1$, as required (note that $N_A(T_1) \le N_A(T)$ since $N_A(T_1)$ normalizes $T_1T_0$, hence normalizes $C_G(T_1T_0) =T$).  

(B) By condition (vii), there is a non-binary witness  $(\d, \l)$ for $G^\D$, where $\d=(\d_1,\d_2,\ldots)$, 
$\l=(\l_1,\l_2,\ldots)$. Suppose there exists $g \in G$ sending $\d \to \l$. Then $g$ sends $\d_1\cap \d_2 \to \l_1\cap \l_2$, and so by condition (v), $g$ normalizes the group $X = T_0Z(A)$. Hence as above, $g$ normalizes $A$, hence also $D$ and $T_0$. Now for $x \in N_D(T_0)A$ we have $T^x\cap DA = T_0T_1^a$ for some $a\in A$, and hence using (vi),
\[
T^{xg}\cap DA = (T^x\cap DA)^g = T_0T_1^{ag} = T_0T_1^{a'},
\]
for some $a'\in A$. Hence by (iv) we see that $T^{xg} \in \D$. This shows that $g \in G_\D$, contradicting the fact that 
$(\d, \l)$ is a non-binary witness  for $G^\D$. Hence $\d$ and $\l$ are in different $G$-orbits, showing that $(G,\O)$ is not binary. 
\end{proof}

\noindent \textbf{Remark } The proof shows that condition (v) could be replaced by 
\begin{itemize}
\item[(v')] there exists a non-binary witness  $(\d, \l)$ for $G^\D$ such that 
\[
\bigcap\limits_{i=1}^k \delta_i = \bigcap\limits_{i=1}^k \lambda_i = \bigcap\limits_{a\in A}T^a
\]
\end{itemize}

%\textbf{Note: don't think hypothesis (iv) was used...}

\begin{table}[ht]
\caption{Possibilities for $T,A,D$} \label{ads}
\[
\begin{array}{|c|c|c|c|c|c|}
\hline
G(q) & |T| & A & D & \hbox{Maximality} & \hbox{Comment}\\
&&&& \hbox{condition} &\\
\hline
E_8(q) & (q-1)^8 & A_1(q) & E_7(q) & q\ge 5 & \\
          & (q+1)^8 & A_1(q) & E_7(q) &  & q>3 \\
           & (q+1)^8 & A_4^-(q) & A_4^-(q) & &q\le 3   \\
            & (q^2+\e q+1)^4 & A_2^\e(q) & E_6^\e(q) & (q,\e)\ne (2,-) &\\
            & (q^2+1)^4 & A_1(q^2) & D_6^-(q) & & AD<D_8(q) \\
            & (q^4+\e q^3+q^2+\e q+1)^2 & A_4^\e(q) & A_4^\e(q) && \\
\hline
E_7(q) & (q-1)^7 & A_1(q) & D_6(q) & q\ge 5 & \\
          & (q+1)^7 & A_1(q) & D_6(q) &  & q>3  \\
           & (q+1)^7  & A_2^-(q) & A_5^-(q) & &q\le 3   \\
\hline
E_6^\e(q) & (q-1)^6\,(\e=+) & A_1(q) & A_5(q) & q\ge 5& \\
                & (q+1)^6\,(\e=-) & A_1(q) & A_5^-(q) & & q>3 \\
               & (q+1)^6\,(\e=-) & A_2^-(q) & A_2^-(q)^2 & &q\le 3   \\
                  & (q^2+\e q+1)^3 & A_2^\e(q) & A_2^\e(q)^2 & (q,\e)\ne (2,-) &\\
                 
\hline
F_4(q), & (q-\e)^4 & A_1(q) & C_3(q) & q\ge 4 &\\
q \hbox{ even} & (q^2+\e q+1)^2 & A_2^\e(q) & A_2^\e(q) & (q,\e)\ne (2,-)& \\
              & (q^2+1)^2 & A_1(q^2) & B_2(q) & & AD<B_4(q) \\
\hline
G_2(q),  & (q-\e)^2 & A_1(q) & A_1(q) & q\ge 9 &\\
q=3^a  & &&&& \\
\hline
^2\!F_4(q)', & (q+1)^2 & A_1(q) & A_1(q) & q\ge 8& \\
  & (q+\e\sqrt{2q}+1)^2 & {^2\!B_2}(q) & {^2\!B_2}(q) & (q,\e)\ne (2,-) & \\
\hline
^3\!D_4(q) & (q^2+\e q+1)^2 & A_2^\e(q) & q^2+\e q+1 && \\
\hline
\end{array}
\]
\end{table}

\begin{table}[ht!]
\caption{Remaining possibilities for $T$} \label{others}
\[
\begin{array}{|c|c|c|}
\hline
G(q) & |T| & N_{G(q)}(T)/T \\
\hline
E_8(q) & q^8+\e q^7 -\e q^5 -q^4 -\e q^3 +\e q+1 & Z_{30} \\
\hline
F_4(q),q=2^a>2 & q^4-q^2+1 & Z_{12} \\
\hline
G_2(q),q=3^a>3 & q^2+\e q+1 & Z_{6} \\
\hline
^2\!F_4(q)' & q^2+\e \sqrt{2q^3}+q+\e \sqrt{2q}+1 & Z_{12} \\
\hline
^3\!D_4(q) & q^4-q^2+1 & Z_4 \\
\hline
\end{array}
\]
\end{table}

\begin{proof}[Proof of Proposition~$\ref{nortor}$]
Let $G$ be almost simple with socle $G(q)$ an exceptional group of Lie type, and let $H=N_G(T)$ be a maximal subgroup of $G$ normalizing a maximal torus $T$, as in \cite[Table 5.2]{LSS}.  We aim to apply Lemma \ref{torlem}. Tables \ref{ads} and \ref{others} together list all possibilities for $T$, and the first table also lists a pair of subgroups $A,D$ that, as we shall see, satisfy the hypotheses of Lemma \ref{torlem} (there are no such subgroups for the cases in Table \ref{others}). Note that in the tables, the values for $|T|$ are those for the relevant maximal torus in the inner-diagonal group ${\rm InnDiag}(G(q))$ rather than in the simple group $G(q)$ itself.

%\color{orange} 
Suppose $T,A,D$ are as in Table \ref{ads}. The cases where $A$ is not quasisimple are those listed in Table~\ref{nortorex}, and so are excluded from further consideration here. Thus $A$ is quasisimple, and we must check that $A,D$ satisfy the first three hypotheses of Lemma \ref{torlem}. In all cases except the two with entries in the ``Comment" column of the table, $AD$ is a subsystem subgroup with maximal normalizer in $G(q)$ as in \cite[Table 5.1]{LSS}, so condition (i) holds. Moreover, $N_G(AD)$ contains a maximal torus $T = T_1T_0$ of order as in column 2 of the table, giving (ii). Finally, we can check that condition (iii) holds by computing the action of $T_0$ on the Lie algebra $L(\bar G)$ (where $\bar G$ is the ambient algebraic group) and seeing that the zero-weight space has dimension equal to that of $A$.
 Hence, by Lemma \ref{torlem}, there is a subset $\D$ of $\O = (G:H)$ such that $G^\D$ has socle $A/Z(A)$ acting on 
$(A:N_A(T_1))$, where $T_1 = T\cap A$.

Suppose $A$ is of type $A_1$. We check that the further conditions (iv) - (vii) of Lemma~\ref{torlem} hold. Condition (vii) holds, since  the group $G^\D$ is not binary, by \cite{ghs_binary}; and to verify (iv), we compute the action of $T_1T_0$ on $L(\bar G)$ again to see that $C_G(T_1T_0)$ is a maximal torus, which must be $T$. For (v), let $T',T'' \in \D$. Then $T' \cap DA = T_0T_1^{a'}$ and $T'' \cap DA = T_0T_1^{a''}$, for some $a',a'' \in A$. Since $A = \SL_2(q)$, we have $T_1^{a'} \cap T_1^{a''} = Z(A)$, and so $T'\cap T'' \cap DA = T_0Z(A)$. As in the proof of Lemma \ref{torlem}(A), it follows that 
$T'\cap T'' = T_0Z(A)$, and this is equal to $\cap_{a\in A}T^a$, giving (v). Finally, (vi) is a standard property of tori in $\SL_2(q)$. Hence conditions (iv) - (vii) in Lemma \ref{torlem} hold, and so the lemma shows that $G$ is not binary.

If $A$ is not of type $A_1$, then we have three families of examples and three sporadic examples. Let us consider the families first: we find that $$(A,|T_1|) = (A_2^\e (q), q^2+\e q+1),\, (A_4^\e (q), q^4+\e q^3+q^2+\e q+1)\, \hbox{ or }(^2\!B_2(q), q+\e \sqrt{2q}+1).$$ Lemma \ref{l: frobenius cyclic kernel} implies that $G^\D$ is not binary and we again check that Lemma \ref{torlem} applies to show that $G$ is not binary.

Finally we must deal with the remaining sporadic examples: here $$(A,|T_1|) = (A_4^-(q),(q+1)^4)\,(q=2,3) \hbox{ or }(A_2^-(q),(q+1)^2)\,(q=3).$$ A \magma calculation verifies that, in each case, the action of an almost simple group $X$ with socle $A$ on $\D=(X:N_X(T_1))$ is not binary and, what is more, there exists a non-binary witness  $(\d, \l)=((\d_1,\d_2,\d_3,\d_4),(\l_1,\l_2,\l_3,\l_4))$ of length $4$ for $X^\D$ such that 
\[
\bigcap\limits_{i=1}^4 \delta_i = \bigcap\limits_{i=1}^4 \lambda_i = \bigcap\limits_{a\in A}T^a.
\]
Note that the computation here is straightforward: we have constructed the permutation representations under consideration and then we have checked $4$-tuples until we found one satisfying the required property. Thus condition (vii), and also condition (v') of the Remark following Lemma~\ref{torlem} hold. Conditions (iv) is verified as before, and (vi) is  straightforward, as $T_1$ is the unique maximal torus of its order up to conjugacy in $A$. Hence Lemma \ref{torlem} gives the conclusion in these cases also. 

%{\color{orange} 
Suppose finally that $T$ is as in Table \ref{others}. In these cases $T$ is cyclic. One can check that for each prime $t$ dividing $|T|$, $T$ contains a Sylow $t$-subgroup of $G$. Now \cite{Deriz2, DerizLie, Enom1, shinoda1} imply that, for every $g\in T\setminus\{1\}$, $C_G(g)=T$, and we conclude that $N=N_{G(q)}(T)$ is a Frobenius group, with $T$ the Frobenius kernel. Let $C$ be a Frobenius complement; observe that $C$ is cyclic, and let $c$ be a generator of $C$. Now Lemma~\ref{l: cent rank} implies that $C_G(c)>C$ and so we can choose an element $x\in C_G(c)\setminus N_G(T)$. Then the action of $N$ on $(N:N\cap N^x)$ is a Frobenius action and, since $|C|>2$ in every case, and, since $N\cap N^x = N\cap H\cap H^x$, Lemma~\ref{l: frobenius cyclic kernel} implies that the action of $H$ on $(H:H\cap H^x)$ is not binary; hence $G$ is not binary by Lemma~\ref{l: point stabilizer}.
\end{proof}

%{\color{orange} Martin will adjust the comments column of Table\ref{ads}. Done!}

The remaining cases are resolved by calculations with {\tt magma}:

\begin{lem}\label{lastcase}
  Let  $G$ be as in Proposition $\ref{nortor}$, and suppose that $H=N_G(T)$ where $(G(q), |T|)$ is as listed in Table $\ref{nortorex}$. If $\O = (G:H)$, then $(G,\O)$ is not binary.
\end{lem}
\begin{proof}
For the group $G(q):= {^2\!E_6(2)}$, observe that $N_{G(q)}(T)\cong 3^5.\SOr_5(3)$ and also  that $G(q)$ has a unique conjugacy class of elements of order $5$ (see \cite{atlas}). Hence Lemma~\ref{l: M2}, applied with the prime $ p:=5$, gives the conclusion.

When $G(q):=E_7(2)$, we gain use Lemma~\ref{l: M2} with the prime $p:=7$. Using information from \cite{luebecke72}, we see that there exists a unique conjugacy class of elements of order $7$. Furthermore, from \cite{LSS}, we have $H\cong 3^7.2.\Sp_6(2)$. Since the Sylow $7$-subgroup of $G(q)$ is elementary-abelian of order $7^3$ and a Sylow $7$-subgroup of $H$ is of order $7$,  Lemma \ref{l: M2} implies that $G(q)$ is not binary.\end{proof}

\section{Maximal subgroups in (V) of Theorem \ref{MAXSUB}}

The main result of this section is the following proposition. The cases excluded in the proposition (those in Table~\ref{type5ex}) will be dealt with in Lemma~\ref{l: type5 extra}.

\begin{prop}\label{type5}
Assume $G$ is almost simple with socle $G(q)$, an exceptional group of Lie type over $\F_q$, and suppose $G(q)$ is not as in Proposition~$\ref{smallex}$. Let $H$ be a maximal subgroup of $G$  as in part (V) of Theorem $\ref{MAXSUB}$. Let $\O = (G:H)$. Then either $(G,\O)$ is not binary, or $(G,H)$ is as in Table $\ref{type5ex}$.
\end{prop}

\begin{table}[ht!]
\caption{Exceptions in Prop. \ref{type5}} \label{type5ex}
\[
\begin{array}{|c|c|}
\hline
G(q) & H\cap G(q) \\
\hline
%G_2(q), \, q\in \{7,11\} & A_1(q) \\
E_6(2) & G_2(2) \\
%E_6(2) & A_2^-(2)G_2(2) \\
^2\!E_6^-(2) & F_4(2) \\
E_7(2) & G_2(q)C_3(2) \\
%^3\!D_4(q) & G_2(q) \\
%                 & PGU_3(q) & \hbox{any} \\
\hline
\end{array}
\]
\end{table}

\begin{proof}
%\textbf{We need to adjust refs to Lemmas \ref{l: classical element}, \ref{l: psu3 element} and \ref{l: psu4} in this proof.}
Here $H$ is one of the subgroups given in Table \ref{tbl}. 

If $H$ has socle $A_1(q)$, then $q>5$ and we consider an element $x \in H$ of order $\frac{q-1}{(2,q-1)}$, as given in Table~\ref{t: as stab}. As $C_G(x)$ contains a maximal torus of $G(q)$, there exists $g\in C_G(x)\setminus H$. Now Lemmas~\ref{l: psl element} and \ref{l: psl element 2} imply that the action of $H$ on $(H:H\cap H^g)$ is not binary. The result then follows by Lemma~\ref{l: point stabilizer}. 

If $H$ has socle $B_2(q)$ then $q\ge 5$ and we proceed in the same way, using an element of order $q-1$ together with Lemma~\ref{l: classical element} in place of Lemma~\ref{l: psl element}. %This argument accounts for all cases where $q>3$. Table~\ref{tbl} implies that there are no remaining cases.% are $q\in\{5,7\}$ with $G(q)=E_8(q)$. In this case we let $T$ be a torus of $B_2(q)$ of order $\frac12(q-1)^2$. It is easy to check that $T$ is centralized by an element of $G\setminus H$. Now Lemma~\ref{l: b2 small} implies the action of $(H:H\cap H^x)$ is not binary, and the result follows as before.
A similar argument applies in the case where $F^*(H \cap G(q)) = C_4(q) \cong \PSp_8(q)$.
% but there is an extra complication, so we shall give the argument in detail. Here $G(q) = E_6^\e(q)$ and $q$ is odd.  Observe that $H$ centralizes an involutory graph automorphism $\tau$ of $G(q)$ (see Proposition \ref{outeraut}). If $\tau \not \in G$, then $C_G(H \cap G(q)) = 1$ and $H$ is almost simple; hence we can use an element $x \in H\cap G(q)$ of order $q^3-1$ together with Lemma~\ref{l: classical element} to obtain the conclusion as above. However, 
%i%f $\tau \in G$ then $H$ is not almost simple (as it centralizes $\tau$), so we cannot directly apply Lemma~\ref{l: classical element}. Instead we observe that the proof of that lemma actually shows that if $x \in M <H$ where $M$ does not contain $C_4(q)$, then the action of $H$ on $(H:M)$ has a beautiful subset. Hence this is the case for the action on $(H:H \cap H^g)$ where $g \in C_G(x) \setminus H$, and the result follows. 
%(\textbf{Note that we could apply Lemma 4.1.2 here and avoid this extra argt, if 4.1.2 proves to be OK.})

Now consider the cases listed in Table \ref{list1}. In each of these cases $F^*(H)$ has a factor that is generated by long root subgroups of $G$, from which it follows easily that there are subgroups $A\cong \SL_r(q)$ of $H$, and $S \cong \SL_{r+1}(q)/Z$ of $G$, satisfying the hypotheses of Lemma \ref{aff}, where $r$ is as indicated in Table~\ref{list1}. Thus Lemmas \ref{aff} provides a subset $\D$ of $\O$ of size $q^r$, and this is a beautiful subset unless $\Alt(q^r)$ and $\Alt(q^r-1)$ are sections of $G(q)$ and $H$, respectively. Hence Lemmas~\ref{l: alt sections classical} and~\ref{altsec} show that if $(G,\O)$ is binary, the only possibility for $(G(q),H)$ is $(E_7(2),\,G_2(2)C_3(2))$, as in Table \ref{type5ex}.

\begin{table}[ht]
\caption{Subgroups in Table \ref{tbl} with a long root factor} \label{list1}
\[
\begin{array}{|c|c|c|}
\hline
G(q) & H & A \\
\hline
E_8(q) & G_2(q)F_4(q)& \SL_4(q) \\
           & A_1(q)G_2(q)G_2(q) & \SL_3(q) \\
E_7(q) & G_2(q)C_3(q)& \SL_3(q) \\
           & A_1(q)F_4(q) & \SL_4(q) \\
E_6(q) & F_4(q) & \SL_4(q) \\
           % & C_4(q) & \SL_4(q) \\
           & A_2(q)G_2(q) & \SL_3(q) \\
E_6^-(q)   & A_2^-(q)G_2(q) & \SL_3(q) \\
F_4(q)   & A_1(q)G_2(q) & \SL_3(q) \\
\hline
\end{array}
\]
\end{table}

Similarly, the subgroup $A_1(q)G_2(q^2)$ of $E_8(q)$ in Table \ref{tbl} is a twisted version of the subgroup $A_1G_2G_2$ in the algebraic group $E_8$; hence this contains a subgroup $A\cong \SL_3(q^2)$ which lies in a subgroup $S\cong \SL_4(q^2)$ of $G$, and again Lemmas \ref{aff}  and \ref{altsec} give the conclusion.

It remains to deal with the following subgroups from Table \ref{tbl}:
\begin{itemize}
\item[(1)] $E_8(q)$: $F^*(H)= F_4(q)\,(p=3)$ or $A_1(q)A_2^\e(q)\,(p\ge 5)$ 
\item[(2)] $E_7(q)$: $F^*(H) = A_2^\e(q)\,(p\ge 5),\,A_1(q)A_1(q)\,(p\ge 5) ,\hbox{ or }A_1(q)G_2(q)\,(p\ge 3,q\ge 5)$
\item[(3)] $E_6^\e (q)$: $F^*(H)=A_2^\pm (q)\,(\e = +,\,p\ge 5)  \hbox{ or }G_2(q)\,(p\ne 7)$
\item[(4)] $E_6^-(q)$: $F^*(H)= F_4(q)$
\item[(5)] $F_4(q)$: $F^*(H)= G_2(q)\,(p=7)$
\item[(6)] ${^3\!D_4}(q)$: $F^*(H)= G_2(q)\hbox{ or }A_2^\e(q)$.
\end{itemize}

\no \textbf{Case (1) } Here $G(q) = E_8(q)$. First consider $F^*(H) = F_4(q)$ with $q = 3^a$. If $q>3$, let $x \in F^*(H)$ be the semisimple element defined in Lemma \ref{l: chev exce element}. Then $C_G(x)$ contains a maximal torus of $G(q)$, hence there exists $g \in C_G(x) \setminus H$, and so $x \in H\cap H^g$, a core-free subgroup of $H$. By Lemma 
\ref{l: chev exce element}, the action of $H$ on $(H:H\cap H^g)$ is not binary, giving the conclusion. If $q=3$, we use the result  of Lemma \ref{l: odd degree Lie} for the group $H=F_4(3)$: since $|G:H|$ is even, there exists a non-trivial orbit of $H$ on $\O$ of odd size, and the action of $H$ on this orbit is not binary by Lemma~\ref{l: odd degree Lie}, giving the conclusion.

Now consider the other possibility $F^*(H) = A_1(q)A_2^\e(q)\,(p\ge 5)$.   Let $R$ be the $A_2^\e(q)$ factor of $F^*(H)$. From the construction of the corresponding maximal subgroup $A_1A_2$ in the algebraic group $E_8$ given in  \cite[p.46]{Sei}, we see that $R$ lies in a Levi subgroup $L=A_7(q)$ of $G$, with embedding given by the adjoint representation. Let $A$ be a natural subgroup $\SL_2(q)$ of $R$ (i.e. acting as $1\oplus 0$ on the natural 3-dimensional module, where we denote by a non-negative integer $r$ the irreducible $\F_qA$-module of highest weight $r$). 
The restriction of the natural 8-dimensional $L$-module to $A$ is $2\oplus 1^2\oplus 0$, so in particular $C_L(A)$ contains a subgroup $\SL_2(q)$ and also $A$ lies in a subgroup $S=A_2(q)$ of $L$. Hence there exists $x \in C_L(A) \setminus H$ such that  $A < S^x \not \le H$. Now an application of Lemma \ref{aff} yields a subset $\D$ of $\O$ such that $G^\D \ge \ASL_2(q)$, giving the conclusion in the usual way using Lemma \ref{altsec}.

\vspace{2mm}
\no \textbf{Case (2) } Here $G(q) = E_7(q)$. First consider $F^*(H)=A_2^\e(q)$. Again let $A$ be a natural $\SL_2(q)$ in $F^*(H)$. Since maximal subgroups $A_2^\e(q)$ exist for both $\e=+$ and $\e=-$, and each of these arises from a fixed maximal $A_2$ in the algebraic group $E_7$, it follows that there is a subgroup $S \cong A_2(q)$ of $G$ containing $A$ (it could be that $S=F^*(H)$). From \cite[p.83]{Sei}, it follows that $A$ is contained in a Levi subgroup $L$ of $G$ of type $A_1A_4T_2$. The torus $T_2$ centralizes $A$, and so there exists $x \in C_G(A)\setminus H$. Then $A<S^x \not\le H$, and the conclusion follows as in Case (1) above.

Next consider $F^*(H)= A_1(q)G_2(q)$. Let $A$ be an $\SL_3(q)$ subgroup of the $G_2(q)$ factor. From \cite[3.12]{Sei} we see that the $G_2(q)$ factor lies in a Levi subgroup $L=A_6(q)$ of $G$, acting irreducibly on the natural 7-dimensional $L$-module $V_7$. Then $V_7\downarrow A = 10\oplus 01\oplus 00$. Now $L$ lies in a subsystem subgroup $M=A_7(q)$ of $G$, and so we see that $A$ is contained in a subgroup $S\cong A_3(q)$ of $M$ acting on the natural $M$-module as $100\oplus 001$. Hence $A<S\not \le H$, and now Lemmas \ref{aff} and \ref{altsec} give the conclusion.

Finally consider $F^*(H)= A_1(q)A_1(q)=A_1^{(1)}A_1^{(2)} \cong \PSL_2(q)^2\, (p\geq 5)$. We assume that the action is binary, and 
for $i=1,2$ let $T^{(i)}$ be a torus in $A_1^{(i)}$ of order $\frac{q-1}{2}$. Write $T =T_1^{(1)}T_1^{(2)}$ and observe that $N_{F^*(H)}(T)$ contains $D_{q-1}\times D_{q-1}$.  From \cite[p.37]{Sei}, we see that (re-labelling the $A_1^{(i)}$ if necessary), we have $C_G(T^{(1)}) \ge T^{(1)}A_2(q)A_4(q)$, and $A_1^{(2)}$ is embedded in this via the irreducibles of highest weight 2 and 4. Hence $T_1^{(2)}$ is diagonal in $A_2(q)A_4(q)$, and so $C_{G(q)}(T)$ contains a maximal torus of order $(q-1)^7/2$. 
Hence a Sylow $2$-subgroup of $\cent{G(q)}{T}$ is strictly larger than the Sylow $2$-subgroup of $T$. This means, in particular, that the group $(D_{q-1}\times D_{q-1})/T$, which is a Klein $4$-group, is a proper subgroup of a Sylow $2$-subgroup of $\nor{G(q)}{T}/T$. The group $H \cap G(q)$ is $(\PSL_2(q) \times \PSL_2(q)).2$, of index $2$ in $\PGL_2(q) \times \PGL_2(q)$. Now, observe that $|C_{G(q)}(T)|_2$ is strictly larger than $|C_{H\cap G(q)}(T)|_2$. We conclude that there exists $x$ in $G(q)\setminus H$ such that $H^x \cap H$ contains $D_{q-1} \times D_{q-1}$.
 We therefore obtain a suborbit of $G$ on which the action of the stabiliser $H$ is isomorphic to the action of $H$ on $(H:M)$ where $M$ is a subgroup of $H$ containing $D_{q-1}\times D_{q-1}$. But now, when $q>5$, Lemma~\ref{l: point stabilizer} and Lemma~\ref{l: a1a1} imply that if $(G,\O)$ is binary then $M$ must contain $F^*(H)$; this would mean that $x\in H$, a contradiction, as required. Suppose now $q=5$. We have $F^*(H) \cong \Alt(5) \times \Alt(5)$. Write $F^*(H)=A\times B$, with $A\cong \Alt(5)\cong B$. Pick $F < H$ with $F$ elementary abelian of order $5^2$. Clearly, there exists $x$ in $N_G(F) \setminus H$, so $F \le H^x \cap H$. Suppose one of the factors, say $A$, of $F^*(H)$ is contained in $H^x\cap H$. Then $A$ and $A^{x^{-1}}$ are contained in $H$. Hence $A^{x^{-1}}$ is equal to $A$, to $B$ or to a diagonal subgroup of $A\times B$. Unipotent elements of order $5$ in $B$ or a diagonal subgroup are in different classes to those in $A$ (see~\cite[Table~34]{Lawther}).
Hence $A^{x^{-1}} = A$ and so $x \in N_G(A) = H$, a contradiction. We conclude that neither factor $A$ or $B$ is contained in $H^x \cap H$.  Now, the proof follows with a \magma computation: we have verified that, for every group $H$ with $F^*(H)=\Alt(5)\times \Alt(5)$ and for every subgroup $X$ of $H$, the action of $H$ on $(H:X)$ is binary only when $X$ contains the whole of $F^*(H)$ or one of the two simple factors $\Alt(5)$ of $F^*(H)$.

 \vspace{2mm}
\no \textbf{Case (3) } Let $G(q)=E_6^\e (q)$. First consider  $F^*(H)=A_2^\pm (q)$.  Here $\e=+$. Let $A$ be a natural $\SL_2(q)$ in $F^*(H)$. Since maximal subgroups $A_2^+(q)$ and $A_2^-(q)$ both exist (actually just for $q \equiv \e \hbox{ mod }4$), and each of these arises from a fixed maximal $A_2$ in the algebraic group $E_6$, it follows that there is a subgroup $S \cong A_2(q)$ of $G$ containing $A$ (it could be that $S=F^*(H)$). From \cite[5.5]{Sei} we know that $L(E_6) \downarrow A_2 = 11 \oplus 41 \oplus 14$. 
Hence we can work out $L(E_6)\downarrow A$, and in particular compute that, if $t$ denotes the central involution of $A$, then $\dim C_{L(E_6)}(t) = 38$, whence $C_{E_6}(t) = A_1A_5$. Also using $L(E_6)\downarrow A$, the only possible embedding of $A$ in $A_1A_5$ is via the representations $1, 2^2$. Hence $C_{A_5}(A) = A_1$, and so $C_G(A) \not \le H$. If we pick $x \in C_G(A)\setminus H$, then $A<S^x \not \le H$, and the conclusion follows in the usual way.

Now consider $F^*(H) = G_2(q)\,(p\ne 7)$. First suppose that $p\ne 2$. Then $G_2(q)$ has an involution $t$ with centralizer $A\tilde A$, where $A$ and $\tilde A$ are long and short $\SL_2(q)$ subgroups (respectively) in $G_2(q)$. 
Arguing as above using $L(E_6)\downarrow G_2$ (given in \cite[Table 10.1]{LSmax}), we see that $A\tilde A < 
C_{E_6}(t) = A_1A_5$, with embedding given by $0\otimes 1,1\otimes 2$. Hence $C_{A_5}(A) = A_2$, and so $N_G(A)$ contains a subgroup $A_2^\e(q)A$.

Now $G_2(q)$ has a subgroup $S\cong \SL_3(q)$ containing $A$. The composition factors of $S$ on $L(E_6)$ are given in \cite[Table 5 and Lemma 5.5]{th}, from which we can deduce that the only subsystem subgroup containing $S$ is $A_2^3$ (this is $C_{E_6}(Z(S))$ unless $p=3$). It follows that $C_G(S) = Z(S)$, and so $N_{G(q)}(S) = S.2 < H$. In particular, it follows that $N_G(A) \ne N_H(A)\,(N_G(S) \cap N_G(A))$. Therefore Lemma \ref{factn} implies that there exists $x \in N_G(A)$ such that $S^x \not \le H$. Hence $A < S^x \not \le H$, giving the conclusion in the usual way.

It remains to consider the case where $p=2$. Again let $\SL_3(q) \cong S < G_2(q)$, let $A$ be a natural subgroup 
$\SL_2(q)$ of $S$, and let $\tilde A = C_{G_2(q)}(A) \cong \SL_2(q)$. Now \cite[5.5]{th} shows that $S$ lies in a subsystem subgroup $A_2^3$ of the algebraic group $E_6$, and so $A < A_1^3 < A_5$. Therefore $C_{A_5}(A)$ contains an $A_1$ subgroup, and so $C_G(A)$ contains $A_1(q)^2$. Hence we see as in the previous paragraph that 
$N_G(A) \ne N_H(A)\,(N_G(S) \cap N_G(A))$, and now the argument goes through as before, the only difference being that this time Lemma \ref{altsec} does not give a contradiction when $q=2$, leaving that possibility in Table \ref{type5ex}.
 
\vspace{2mm}
\no \textbf{Case (4) } Let $G(q) = E_6^-(q)$ and $F^*(H) = F_4(q)$. There are long root subsystem subgroups $A<S<H$ with $A \cong \SL_3(q)$, $S\cong \SL_4(q)/Z$. Their centralizers can be read off using \cite[Sec. 4]{LSroot}, and we have $C_G(A) = A_2(q^2)$, $C_G(S) = A_1(q^2)T_1$ and $C_H(A) = A_1(q)T_1$. Hence $N_G(A) \ne N_H(A)\,(N_G(S) \cap N_G(A))$, and so Lemma \ref{factn} yields an element $x \in G$ such that $A < S^x \not \le H$. Now Lemmas \ref{aff} and \ref{altsec} give a contradiction, except when $q=2$, leaving that possibility in Table \ref{type5ex}.

\vspace{2mm}
\no \textbf{Case (5)  } Let $G(q) = F_4(q)$ and $F^*(H) = G_2(q)$ with $p=7$. We argue as for $G_2(q)$ in case (3) above.
For an involution $t \in G_2(q)$ we have $C_{G_2(q)}(t) = A\tilde A$ where $A$ and $\tilde A$ are long and short 
$\SL_2(q)$ subgroups, and also $A<S< G_2(q)$ with $S \cong \SL_3(q)$. Then $A\tilde A < C_{F_4}(t) = A_1C_3$ with embedding $0\otimes 1, 1\otimes 2$, and so $C_{C_3}(A) = A_1$. Again $N_G(S) < H$, and hence 
$N_G(A) \ne N_H(A)\,(N_G(S) \cap N_G(A))$. Now Lemma \ref{factn} yields an element $x \in G$ such that $A < S^x \not \le H$ and we proceed as before. 

\vspace{2mm}
\no \textbf{Case (6)  } Here $G(q) = \,^3\!D_4(q)$ and $F^*(H)= G_2(q)$ or $A_2^\e(q)$. Consider the first case. Let $A<S <G_2(q)$ with $A \cong \SL_2(q)$, $S\cong \SL_3(q)$ generated by long root subgroups of $H$ (and of $G$). Then $C_G(A) = A_1(q^3)$, $C_H(A) = A_1(q)$ and $N_{G(q)}(S) = S.(q^2+q+1).2$. There is no factorization of a group with socle $A_1(q^3)$ with one of the factors being $N(A_1(q))$ (see Lemma \ref{factnlem}), and so Lemma \ref{factn} applies to give an element $x \in G$ such that  $A<S^x \not \le H$. Now Lemmas \ref{aff} and \ref{altsec} give a contradiction (except when $q=2$, in which case $G(q) = \,^3\!D_4(2)$, excluded by hypothesis).

Now let $F^*(H) = A_2^\e(q)$. From \cite{K}, we see that $H\cap G(q) = \PGL_3^\e(q)$ with $q \equiv \e \hbox{ mod }3$ and $q>2$. First assume that $\e=+$, and let $A$ be a natural $\SL_2(q)$ subgroup of $H$. Then $A$ centralizes an element $g$ of order $q-1$ in $H$, and from the list of centralizers in $G$ (see for example \cite[p.184]{K}), we see that $C_G(g)'$ must be $\SL_2(q^3)$, or possibly $\SL_3(q)$ when $q=4$. Excluding the latter possibility, it follows that $C_G(A)$ contains the centralizer of $\SL_2(q^3)$, which is a root subgroup $\SL_2(q)$. Hence in any case (including the extra $q=4$ possibility), there is a group $S \cong A_2(q)$ such that $A<S\not \le H$, giving the result in the usual way.

This leaves the case where $\e=-$, so that $H\cap G(q) = \PGU_3(q)$ with $q \equiv -1 \hbox{ mod }3$. We refer to Lemma~\ref{l: psu3 element}, and let $g$ be the element of $H\cap G(q)$ defined in that lemma. 
%There is a complication in that $H\cap G(q)$ is centralized by an outer automorphism $\tau$ of order 3 (see Proposition \ref{outeraut}); so $H$ may not be almost simple, and we cannot directly apply Lemma~\ref{l: psu3 element} -- instead, we apply its proof. 
Observe that $g$ is semisimple in $G(q)$, hence, using the list of maximal tori of $G(q)$ given in \cite{kantorseress}, we can conclude that there exists $x\in G(q)\setminus H$ such that $x\in C_G(g)$. Now consider the action of $H$ on the cosets of $H\cap H^x$, a subgroup containing the element $g$ and not containing $\PSU_3(q)$. 
If $q\le 5$, we use \magma to show that this action is not binary, giving the conclusion.  
And if $q\ge 7$, Lemma \ref{l: alt sections classical} shows that $H$ has no section $\Sym(q)$, and so 
Lemma~\ref{l: psu3 element} shows that the action $(H,(H:H\cap H^x))$ is not binary, again giving the conclusion.

%So assume $q\ge 7$. If $q$ is odd, the proof of Lemma~\ref{l: psu3 element} shows that the is a beautiful subset of size $q$. 
%Assume finally that $q$ is even. Also, we can assume that $\tau \in H$, since other wise $H$ is almost simple and Lemma~\ref{l: psu3 element} gives the conclusion. Now the proof of that lemma shows that either there is a beautiful set, or $H\cap H^x = N_1$, the stabilizer of a nonsingular 1-space in the associated unitary 3-space $V_3(q^2)$.  If $H\cap H^x$ contains no element in the coset  $G(q)\tau$, then we can use Lemma~\ref{l: psu3 element} together with \ref{l: subgroup} to see that $(H:H\cap H^x)$ is not binary. 
%Otherwise, $H$ must contain $n\tau$ for some $n$ normalizing $H\cap G(q)$. Hence $n \in H$ as $N_1$ is self-normalizing, and so $\tau \in H$. 

%But now the action of $H$ on $(H:H\cap H^x)$ is isomorphic to the action of an almost simple group with socle $\PSU_3(q)$ on a subgroup containing the element $g$. Proposition~\ref{l: psu3 element} implies that this action is not binary, and the result follows by Lemma~\ref{l: point stabilizer}.

\end{proof}

The remaining cases are resolved with the aid of {\tt magma}:

\begin{lem}\label{l: type5 extra}
Let $G$ be as in Proposition $\ref{type5}$, and suppose that
 $(G(q),H)$ is listed in Table~$\ref{type5ex}$. Then $(G,\O)$ is not binary.
\end{lem}
\begin{proof}
%Suppose $G(q)=G_2(q)$. Then $q\in \{7,11\}$ is prime and hence $G=G(q)$, moreover $H=A_1(q)$ because $H=G(q)\cap H=A_1(q)$. Using \magma we have computed all the transitive actions of $\PSL_2(7)=A_1(7)$ and $\PSL_2(11)=A_1(11)$ and the only binary actions are of the regular actions and the trivial actions. Then, from Lemma~$2.8$ in \cite{dgs_binary}, it follows that $H\cap H^g=1$ for every $g\in G\setminus H$. That is, $G$ is a Frobenius group with Frobenius complement $H$, but this is clearly a contradiction.

Suppose that $G(q)=E_6(2)$ and $H\cap G(q)=G_2(2)$. Referring to \cite{KW}, we see that $H$ is maximal in $G$ only when $G=G(q)=E_6(2)$, thus we assume this from here on. Now, using {\tt magma}, we have computed all the binary transitive actions of $G_2(2)$, and we have found that these have degree $1$, $2$, $4032$, $6048$ and $12096$. Now Lemma~\ref{l: point stabilizer} implies that, if the action of $G$ on $(G:H)$ is  binary, then the action of $H$ on each of its suborbits must be binary -- thus all suborbits must have size one of the five listed numbers. There is precisely one suborbit of size $1$ (by maximality), the other suborbits are of even size, hence $|E_6(2): G_2(2)|$ is odd, a contradiction.

%Assume, next, that $G(q)={{^2\!E_6}}(2)$ and $H\cap G(q)=A_2^-(2)G_2(2)$. In this case we use \magma as follows: there are four possible groups $G$ with socle $G(q)$ and accordingly four subgroups $H$. For each possible group $H$, we  have computed all the transitive actions of  odd degree and we have shown that the only actions that are binary have degree $1$ or $3$. Suppose that there exists a suborbit, for the action of $G$ on $(G:H)$, of size $3$. Then $H$ induces a subgroup of $\Sym(3)$ on this suborbit. However, Lemma~\ref{l: higman} implies that $H$ is a $\{2,3\}$-group, which is clearly a contradiction. We conclude, therefore, that if the action of $G$ on $(G:H)$ is binary, then all non-trivial subdegrees must be even. However, this contradicts the fact that $|G:H|$ is even.

Next assume that $G(q)={{^2\!E_6}}(2)$ and $H\cap G(q)=F_4(2)$. Here $H$ is either $F_4(2)$ or $F_4(2)\times 2$. Now $F_4(2)$ has a maximal subgroup isomorphic to $D_4(2).\Sym(3)$; let $X$ be the subgroup $D_4(2)$ of this. Then $X$ is centralized by an element $g$ of order $3$ in $G(q)\setminus H$ (see \cite{atlas}). Hence $X \triangleleft H \cap H^g$. At this point we can argue as in the proof of Proposition~\ref{smallex} (the $F_4(2)$ case); indeed, Lemma~\ref{l: M2} applied with $p=7$ shows that $(H,(H:H\cap H^g))$ is not binary. The conclusion follows. 

%We conclude that there is a suborbit of $H$, for which the stabilizer intersects $F_4(2)$ in $X$ or $X.2$. Now in \S\ref{s: small}, we proved that the action of $F_4(2)$ on any subgroup of even order and even index is not binary. It is easy enough (\textbf{no, we must give an argt or a reference -- it will be Lemma 4.1.2 if this proves correct} )to conclude that the action of $F_4(2)\times 2$ on any subgroup of even order and even index that does not contain $F_4(2)$ is not binary. We conclude that the action of $H$ on one of its suborbits is not binary, and so Lemma~\ref{l: point stabilizer} yields the result.

Finally, assume that $G = E_7(2)$ and $H = G_2(2)C_3(2)$. Choose $x \in G \setminus H$ normalizing a Sylow 2-subgroup of $H$, so that $|H:H\cap H^x|$ is odd and greater than 1. A \magma computation show that all transitive actions of $H$ of odd degree greater than 1 are not binary, so the conclusion follows.
\end{proof}

\section{Maximal subgroups in (VI) of Theorem \ref{MAXSUB}}

In this section we prove

\begin{prop}\label{type6}
Assume $G$ is almost simple with socle $G(q)$, an exceptional group of Lie type over $\F_q$, and suppose $G(q)$ is not as in Proposition~$\ref{smallex}$. Let $H$ be a maximal subgroup of $G$  as in part (VI) of Theorem $\ref{MAXSUB}$. Let $\O = (G:H)$. Then either $(G,\O)$ is not binary, or $(G,H)$ is as in Table $\ref{type6ex}$.
\end{prop}

\begin{table}[ht!]
\caption{Exceptions in Prop. \ref{type6}} \label{type6ex}
\[
\begin{array}{|c|c|}
\hline
G(q) & H \cap G(q) \\
\hline
G_2(2^e),\,e \hbox{ prime} & G_2(2)  \\
F_4(2^e),\,e \hbox{ prime} & F_4(2)  \\
E_6^\e (2^e),\,e \hbox{ prime} & E_6^\e (2) \\
%F_4(q) & ^2\!F_4(q) & q=2 \\
%G_2(q),\,q=3^{2a+1} & ^2\!G_2(q) & q=3 \\
\hline
\end{array}
\]
\end{table}

\begin{proof}
Here $H$ is of the same type as $G$ -- that is, one of the following holds:
\begin{itemize}
\item[(i)] $H \cap G(q) = G(q_0)$, where $\F_{q_0} \subset \F_q$;
\item[(ii)] $H\cap G(q)< G(q) $ is a twisted subgroup, namely one of 
\[
\begin{array}{l}
{^2\!E_6}(q^{1/2}) < E_6(q), \\
^2\!F_4(q) < F_4(q), \\
^2\!G_2(q) < G_2(q).
\end{array}
\]
\end{itemize}

Consider first case (i). Here for each possible $G(q)$, we define subgroups $A<S<G(q_0)$ with $A\cong \SL_r(q_0)$ and $S \cong \SL_{r+1}(q_0)$, both subsystem subgroups of $G(q_0)$, as in Table \ref{q0table}. In each case $C_G(A)'$ and $C_G(S)'$ are as indicated in Table~\ref{q0table} and $C_H(A)$ is of the same type as $C_G(A)$ over the subfield $\F_{q_0}$. It then follows from Lemma \ref{factnlem} that $N_G(A) \ne N_H(A)\,(N_G(S) \cap N_G(A))$, and so Lemma \ref{factn} yields an element $x \in G$ such that $A < S^x \not \le H$. Now Lemmas \ref{aff} and \ref{altsec} give a contradiction, except in the cases with  $q_0=2$ in Table \ref{type6ex}.

\begin{table}[ht!]
\caption{} \label{q0table}
\[
\begin{array}{|ccccc|}
\hline
G(q) & r & C_G(A) & C_H(A) & C_G(S) \\
\hline
G_2(q) & 2 & A_1(q) & A_1(q_0) & (3,q_0-1) \\
F_4(q) & 3 & A_2(q) & A_2(q_0) & A_1(q) \\
E_6(q) & 3 & A_2(q)^2 & A_2(q_0)^2 & A_1(q)^2 \\
{^2\!E_6}(q) & 3 & A_2(q^2) & A_2(q_0^2) & A_1(q^2) \\
E_7(q) & 4 & A_3(q)A_1(q) & A_3(q_0)A_1(q_0) & A_2(q)T_1 \\
E_8(q) & 5 & A_4(q) & A_4(q_0) & A_2(q)A_1(q) \\
\hline 
\end{array}
\]
\end{table}

Now consider case (ii). In the first case, $F^*(H) =\, {{^2\!E_6}}(q^{1/2}) < E_6(q)$, and as above we pick $A<S$ with $A\cong \SL_4(q^{1/2})$ and $S\cong \SL_5(q^{1/2})$. Then $S\not \le H$ as $H$ has no subgroup of type $A_4(q^{1/2})$, and the conclusion follows as usual. 

%{\color{orange}
Next let $F^*(H) =\, {^2\!F_4}(q) < F_4(q)$ with $q=2^{2a+1}$, and note that $q>2$ by hypothesis. Regard $F^*(H)$ as the centralizer in $F_4(q)$ of a graph automorphism $\tau$. Then $H$ has a subgroup $A\cong \SL_2(q)$ arising as the fixed point group of $\tau$ on a  subsystem subgroup $A_1(q)\tilde A_1(q)$ in $F_4(q)$, and this lies in a subgroup $S=A_2(q)$ of $F_4(q)$ that is a diagonal subgroup of a  subsystem $A_2(q)\tilde A_2(q)$. As $H$ has no subgroup $A_2(q)$, we have $A<S \not \le H$, giving the conclusion. 
%This leaves just the $q=2$ case in Table \ref{type6ex}.

Now consider the case where $H\cap G(q) =\, {^2\!G_2}(q) < G_2(q)$, and note that  $q>3$ by hypothesis. 
%Also $H\cap G(q) = C_{G_2(q)}(\sigma)$, where $\sigma$ is an involutory automorphism of $G_2(q)$. 
Choose $x \in H \cap G(q)$ of order $q-1$, as in Lemma \ref{l: exc rank 2 element}. Since $C_{G(q)}(x)$ is a torus of order $(q-1)^2$, there exists 
 $g\in C_G(x)\setminus H$. Then $x \in H\cap H^g$, and the action of $H$ on $(H:H\cap H^g)$ is not binary by 
Lemma~\ref{l: exc rank 2 element}, giving the conclusion.
%(\textbf{I made some changes in the proof of this lemma}) implies that this action is not binary, and the conclusion follows by Lemma~\ref{l: point stabilizer}. Finally, suppose that $\sigma \in G$, so that $C_G(H\cap G(q)) = \langle \sigma \rangle$. Here Lemma~\ref{l: exc rank 2 element} does not apply directly, as $H$ is not almost simple since it contains $\sigma$ as a central involution. However, inspection of the proof of Lemma~\ref{l: exc rank 2 element} reveals that it shows that the action of $H$ on $(H:H\cap H^g)$ (where $x \in H\cap H^g$ and $H\cap H^g$ does not contain $H \cap G(q)$) necessarily contains a beautiful subset of size $q$, and the conclusion again follows by Lemma~\ref{l: point stabilizer}. 
%(\textbf{Note: the last part would follow immediately from a corrected 4.1.2.})
\end{proof}

The treatment of groups of type (VI) is completed with the following result.

\begin{lem}\label{type6b}
Let $G$ be as in Proposition $\ref{type6}$, and suppose that
 $(G(q),H)$ is listed in Table~$\ref{type6ex}$. Then $(G,\O)$ is not binary.
\end{lem}
\begin{proof}
Consider the action in Line~1 of the table. Here $G = G_2(2^e)\langle \phi\rangle$ where $\phi$ is a field automorphism of order 1 or $e$, and $H = G_2(2) \times \langle \phi \rangle$. Choose $x \in G\setminus H$ normalizing a Sylow 2-subgroup $P$ of $H\cap G_2(2)$, and let $X = H\cap H^x \cap G_2(2)$, so that $P \le X$. The subgroups of $H \cap G(q) = G_2(2)$ containing $P$ are the Borel subgroup $P$ itself, and two maximal parabolics of shape $[2^5].\Sym(3)$. These are all self-normalizing in $G_2(2)$, so it follows that $H\cap H^x = X \times \langle \sigma\rangle$, where $\sigma = \phi$ or 1. Note that $\sigma$ is in the kernel of the action of $H$ on $(H:H\cap H^x)$. Thus the latter action is either $(G_2(2),(G_2(2):X))$ or $(G_2(2) \times e,\,(G_2(2)\times e : X))$. Using \magma we check that the action $(G_2(2),(G_2(2):X))$ is not binary for each of the three possibilities for $X$. Hence also $(G_2(2) \times e,\,(G_2(2)\times e : X))$ is not binary, by Lemma \ref{l: subgroup}. It follows that the action of $H$ on $(H:H\cap H^x)$ is not binary, giving the conclusion.

%For the action in Line~1 of the table we confirm, first, using {\tt magma}, that if $H$ is any almost simple group with socle $G_2(2)'=\PSU_3(3)$, and $M$ is any core-free subgroup of $H$ of odd index, then the action of $H$ on cosets of $M$ is not binary. Now let $x$ be any member of $G\setminus H$ that normalizes a Sylow $2$-subgroup of $H$ -- such an element must exist, since the Sylow $2$-subgroup of $G$ is strictly larger than the Sylow $2$-subgroup of $H$. Now, the action of $(H:H\cap H^x)$ is isomorphic to the action of an almost simple group with socle $G_2(2)$ on the cosets of some core-free subgroup $M$ of odd index. (\textbf{This is not correct -- could have $G = G_0.r$ and $H=G_2(2) \times r$. Version of 4.1.2 needed????}) We know that this action is not binary, and Lemma~\ref{l: point stabilizer} yields the result.
 
%{\color{orange} 
Now consider Line 3 of Table~\ref{type6ex}. First suppose $H\cap G(q) = \,{^2\!E_6}(2)$, with $G(q) = \,^2\!E_6(2^e)$. Let $D_0 = \,^2\!D_5(2)$ be a subsystem subgroup of $H\cap G(q)$. Then $D < \,^2\!D_5(q) < G(q)$, a subgroup centralized by a torus of order $\frac{q+1}{3}$. Choosing $g \in C_{G(q)}(D_0) \setminus H$, we have $H\cap H^g \triangleright D$. As in the proof of Proposition~\ref{maxrk}, there is a subgroup $A = \mathrm{SL}_2(4)$ of $D$ and a subgroup $S=\PSL_3(4)$ of $H$ such that $A<S \not \le H\cap H^g$. Hence it follows in the usual way using Lemmas \ref{aff} and \ref{altsec} that the action of $H$ on $(H:H\cap H^g)$ is not binary, giving the conclusion in this case.
 A similar argument handles the case where $H = E_6(2)$: here we take $D = A_5(2)$ and again choose $g \in C_{G(q)}(D_0) \setminus H$, so that $H\cap H^g \triangleright D$. There is a subgroup $A=\SL_4(2)$ of $D$ and a subgroup $S=\SL_5(2)$ of $H$ such that $A<S \not \le H$, and the conclusion again follows.

%observe that $N_G(D)\not\leq H$, thus there is a suborbit corresponding to an action of $H$ on the cosets of some group $X$ such that $D\leq X \leq H$. Then we can choose a subgroup $A$ in $D$ of type $A_3$ so that $A\cong \SL_4(2)$ and $A$ lies inside a group $L$ of type $A_4$ with $L\not\leq D$. Now Lemma~\ref{aff} yields a subset $\Delta$ of $(H:X)$ of size $2^4$ on which the set-stabilizer acts $2$-transitively. Now Lemma~\ref{altsec} implies that the set-wise stabilizer of $\Delta$ cannot contain $\Alt(\Delta)$, and so the action of $H$ on $(H:X)$ is not binary.

Finally, consider $H\cap G(q) = F_4(2)$ with $G(q) = F_4(2^e)$. Choose a subgroup $D = A_2(2) \times 7$ lying in a subsystem subgroup $A_2(2) \times \tilde A_2(2)$ of $H$, where the factor $\tilde A_2(2)$ is generated by short root elements. There is an element 
$x \in N_G(D)\setminus H$, and so $D \le H\cap H^x < H$. From \cite{NW}, it follows that $H\cap H^x$ is contained in a subsystem subgroup $(A_2(2) \times \tilde A_2(2)).2$ of $H$. The factor $A_2(2)$ of $D$ lies in a subgroup $A_3(2)$ of $H$ that is not contained in $H\cap H^x$, and so it follows, using Lemma \ref{aff} in the usual way, that the action of $H$ on the suborbit $(H:H\cap H^x)$ is not binary. This completes the proof.
\end{proof}

\section{The remaining families in Theorem \ref{MAXSUB}}

We proceed family by family.

\subsection{Type (III)}

\begin{lem}\label{type3}
Assume $G$ is almost simple with socle $G(q)$, an exceptional group of Lie type over $\F_q$, and let $H$ be a maximal subgroup of $G$  as in part (III) of Theorem $\ref{MAXSUB}$. If $\O = (G:H)$, then $(G,\O)$ is not binary.
\end{lem}
\begin{proof}
 Here $G(q)=E_7(q)$, $p>2$ and $H\cap G(q) = (2^2\times D_4(q).2^2).\Sym(3)$ or ${^3\!D_4}(q).3$.  Let $D:= D_4(q)$ or ${^3\!D_4}(q)$ in $H$, and let $A$ be a subsystem subgroup $\SL_3(q)$ of $D$. Here $D$ arises from a subgroup $D_4$ of the algebraic group $E_7(\bar \F_q)$ that lies in a subsystem $A_7$ (see the discussion after \cite[Theorems 1,7]{LSsurv}), and we see that $A$ lies in a subgroup $A_7(q)$ of $G(q)$, acting on the natural 8-dimensional module as $10\oplus 01\oplus 00^2$. Then $A$ lies in a subgroup $A_3(q)$ of this $A_7(q)$ that does not lie in $D$. At this point we can apply Lemma~\ref{aff} to see that there is a subset $\D$ of $\O$ such that $G^\D \ge \ASL_3(q)$. This shows that $G$ is not binary in the usual way using Lemma \ref{altsec}. 
\end{proof}

\subsection{Type (IV)}

\begin{lem}\label{l: type 4}
Assume $G$ is almost simple with socle $G(q)$, an exceptional group of Lie type over $\F_q$, and let $H$ be a maximal subgroup of $G$  as in part (IV) of Theorem $\ref{MAXSUB}$. If $\O = (G:H)$, then $(G,\O)$ is not binary.
\end{lem}
\begin{proof}
In this case $G(q)=E_8(q)$ with $p>5$ and $H\cap G(q) = \PGL_2(q) \times \Sym(5)$. Let $L$ be the factor $\PGL_2(q)$ and let $g$ be an element of order $q-1$ in $L$. A consideration of the centralizers of semisimple elements in $E_8(q)$ implies that there exists $x\in C_{G(q)}(g)\setminus H$. Note that $x\not\in N_G(L)$ because the maximality of $H\cap G(q)$ requires that $N_{G(q)}(L)=H\cap G(q)$. Then $H\cap H^x$ contains the element $g$ but does not contain the subgroup $L$, and now Lemma~\ref{l: pgl} implies that the action of $H$ on $(H:H\cap H^x)$ is not binary. The result follows by Lemma~\ref{l: point stabilizer}.
\end{proof}

\subsection{Type (VII)}

%Here $H$ is an exotic local subgroup as listed in Table~\ref{exot}.

\begin{lem}\label{l: type 7}
Assume $G$ is almost simple with socle $G(q)$, an exceptional group of Lie type over $\F_q$, and let $H$ be a maximal subgroup of $G$  as in part (VII) of Theorem $\ref{MAXSUB}$. If $\O = (G:H)$, then $(G,\O)$ is not binary.
\end{lem}
\begin{proof}
Here $H$ is an exotic local subgroup as listed in Table~\ref{exot}. 
When $H\cap G(q) =2^3.\SL_3(2)$, let $r=3$; and when $H\cap G(q) \in \{3^3.\SL_3(3),\,3^{3+3}.\SL_3(3),\,5^3.\SL_3(5),\,2^{5+10}.\SL_5(2)\}$, let $r=2$.  We have verified with \texttt{magma} that every non-trivial transitive action of $H$ of degree coprime to $r$ is not binary. 
%(\textbf{Note that $H$ may contain some extra auts for $G(q) = E_6^\e$ or $E_8(p^2)$, need to check this for groups $H.2$ etc.}) 
In particular, if the action of $G$ on $(G:H)$  is binary, then every non-trivial suborbit of $G$ has cardinality divisible by $r$ and hence $r$ divides $|G:H|-1$. However in all cases $r$ divides $|G:H|$, and hence we reach a contradiction.
\end{proof}

\subsection{Type (VIII)}

\begin{lem}\label{l: type 8}
Assume $G$ is almost simple with socle $G(q)$, an exceptional group of Lie type over $\F_q$, and let $H$ be a maximal subgroup of $G$ as in part (VIII)  of Theorem $\ref{MAXSUB}$. If $\O = (G:H)$, then $(G,\O)$ is not binary.
\end{lem}
\begin{proof}
 Here $H = (\Alt(5) \times \Alt(6)).2^2 < E_8(q)$, where the Klein $4$-group acts faithfully on $F^*(H)=\Alt(5)\times\Alt(6)$. There are several non-isomorphic groups having this shape, but a \magma calculation confirms that if $H$ is any such group, and $M$ is a subgroup of $H$ of odd index, then either the action of $H$ on cosets of $M$ is not binary or $M$ contains the simple factor $\Alt(6)$ of $H$.  Now let $x$ be any member of $G\setminus H$ that normalizes a Sylow $2$-subgroup of $H$. If the action of $H$ on $(H:H\cap H^x)$ is binary, then by the previous sentence, $H\cap H^x$  contains $\Alt(6)$, and hence $x \in N_G(\Alt(6)) = H$, a contradiction.
 \end{proof}

\subsection{Type (IX)}

\begin{lem}\label{l: type 9}
Assume $G$ is almost simple with socle $G(q)$, an exceptional group of Lie type over $\F_q$, and let $H$ be a maximal subgroup of $G$  as in part (IX)  of Theorem $\ref{MAXSUB}$. If $\O = (G:H)$, then $(G,\O)$ is not binary.
\end{lem}
\begin{proof}
Here $F^*(H)$ is a simple group not in ${\rm Lie}(p)$, as listed in Tables 10.1--10.4 of \cite{LSei}. Using also  Theorem~\ref{crav}(i), we see that the possibilities for $F^*(H)$ are:
\begin{itemize}
\item[(1)] $\Alt(6)$, $\Alt(7)$;
\item[(2)] $M_{11}$, $M_{12}$, $M_{22}$, $J_1$, $J_2$, $J_3$, $Ru$, $Fi_{22}$, $HS$, $Th$;
\item[(3)] $\PSL_2(r)$ for $r\le 61$;
\item[(4)] $\PSL_3(3)$, $\PSL_3(4)$, $\PSL_3(5)$, $\PSL_4(3)$, $\PSL_4(5)$, $\PSU_3(3)$, $\PSU_3(8)$, $\PSU_4(2)$, $\PSU_4(3)$, $\PSp_4(5)$, 
$\Sp_6(2)$, $\O_7(3)$, $\O_8^+(2)$, $G_2(3)$, ${^3\!D_4}(2)$, $^2\!F_4(2)'$, $^2\!B_2(8)$, $^2\!B_2(32)$.
\end{itemize} 

Suppose first that $F^*(H)$ is not $\Alt(6)$, $\Alt(7)$ or $\PSL_2(r)$, so that $H$ is as in (2) or (4). 
Observe that $|G:H|$ is even (see \cite{liesax}), so there must be a non-trivial odd subdegree. However Lemmas~\ref{l: odd degree Lie} and \ref{l: sporadic small-odd} imply that if $M$ is any core-free subgroup of $H$ of odd index, then the action of $H$ on cosets of $M$ is not binary. Now Lemma~\ref{l: point stabilizer} implies that $(G,\O)$ is not binary.

 Suppose next that $F^*(H)\cong \Alt(7)$. Then $G(q) = E_7(q)$ or $E_8(q)$ by Theorem \ref{crav}(i), and hence $|G|$ is divisible by $7^2$. 
Therefore there is an element $g \in G\setminus H$ such that $H\cap H^g$ has order divisible by 7. However a \magma computation shows that all faithful transitive actions of $H$ of degree coprime to 7 are not binary, completing the proof in this case.
  
Suppose next that $F^*(H)\cong \PSL_2(r)$ for some $r\le 61$. If $r = 4$ or 5 then $F^*(H) \cong \Alt(5)$, contrary to Theorem \ref{crav}(i). Hence $r\ge 7$. Let $g$ be an element of $H$ of order $\frac{r-1}{(2,r-1)}$. Note that $g$ has order at most $31$. 
We claim that $C_G(g) \not \le H$: for if $C_G(g) \le H$, then $C_G(g) = C_H(g)$  is a cyclic maximal torus of $G(q)$ of order either $\frac{r-1}{(2,r-1)}$ or $r-1$ (the latter only if $H$ contains $\PGL_2(r)$). The orders of cyclic maximal tori are given in \cite[Sec. 2]{kantorseress}. Recalling that $G(q)$ is not as in Proposition \ref{smallex}, we see that the only possibility is that $G(q) = \,^2\!E_6(2)$ and $g$ has order 13, 19 or 21. However, $\PSL_2(r)$ is not a subgroup of $^2\!E_6(2)$ for $r=27$, 39 or 43, as shown in \cite[Sec. 12]{Wilson2E62}. Thus 
$C_G(g)\not \le H$,  as claimed. Hence there exists $x \in C_G(g)\setminus H$, and so $H\cap H^x$ is a core-free subgroup of $H$ containing $g$. Now Lemma \ref{l: psl2 element} implies that $(H, (H:H\cap H^x))$ is not binary and the conclusion follows from Lemma~\ref{l: point stabilizer}. \end{proof}

%If $C_G(g)\neq \langle g\rangle$, then one immediately obtains that there exists $x\in G\setminus H$ centralizing $\langle g\rangle$. Now the proof utilises Proposition~\ref{p: psl element} and proceeds along the lines of the proof of Lemma~\ref{l: type 7}. We must confirm that $C_G(g)\neq \langle g\rangle$. If $q>2$ with $G(q)$ not ${^2\!F_4(q)}$, then this follows from the standard bound $|C_G(g)|>(q-1)^\kappa$, where $\kappa$ is the rank of $G(q)$ (for $F_4(q)$ and $G_2(q)$ we must use the finer information about possible values of $\kappa$ given in \cite{LSei}, along with the fact that Proposition \ref{smallex} excludes small values of $q$ for both families). The case where $G(q)={^2\!F_4(q)}$ can be eliminated using \cite{malle}. We are left with the possibility that $G(q)$ is one of $E_6(2)$, ${{^2\!E_6}(2)}$, $E_7(2)$ or $E_8(2)$; the first can be eliminated using \cite{KW}, the second using \cite{Wilson2E62}; the third using \cite{bbr}; the third can be excluded using \cite{luebecke82} which lists the orders of all conjugacy classes in $E_8(2)$, and from which we can deduce that all centralizers of non-trivial elements have order (greatly) in excess of $31$.

\subsection{Type (X)}

\begin{lem}\label{l: type 10}
Assume $G$ is almost simple with socle $G(q)$, an exceptional group of Lie type over $\F_q$, and let $H$ be a maximal subgroup of $G$  as in part (X)  of Theorem $\ref{MAXSUB}$. If $\O = (G:H)$, then $(G,\O)$ is not binary.
\end{lem}
\begin{proof}
Here $F^*(H)$ is a simple group in ${\rm Lie}(p)$. By Theorem \ref{crav}(ii),(iii), the possibilities for $F^*(H)$ are
\begin{itemize}
\item[(1)] $\PSL_2(q_0)$, $q_0 \le t(G)$ and  as in Theorem \ref{crav}(iii);
\item[(2)] $\PSL_3(3)$, $\PSU_3(3)$ (with $G(q) = E_8(q)$, $q=3^a$);
\item[(3)] $\PSL_3(4)$, $\PSU_3(4)$, $\PSU_3(8)$, $\PSU_4(2)$, $^2\!B_2(8)$ (with $G(q) = E_8(q)$, $q=2^a$).
\end{itemize}

Suppose first that $F^*(H)$ is as in (2) or (3).  Using similar {\tt magma} computations to those described in the proof of Lemma~\ref{l: type 9}, we verify that if $M$ is any core-free subgroup of $H$ of index coprime to $p$, then the action of $H$ on $(H:M)$ is not binary. Since $|G:H|$ is divisible by $p$, there exists $x \in G\setminus H$ normalizing a Sylow $p$-subgroup of $H$. Hence $H\cap H^x$ is a core-free subgroup of $H$ of index coprime to $p$, and so $(H,(H:H\cap H^x))$ is not binary, completing the proof in cases (2) and (3).

%Now let $x$ be any member of $G\setminus H$ that normalizes a Sylow $p$-subgroup of $H$ -- such an element must exist, since the Sylow $p$-subgroup of $G$ is strictly larger than the Sylow $p$-subgroup of $H$. Now the proof utilises the \magma calculation above, and proceeds along the lines of the proof of Lemma~\ref{l: type 7}.

Suppose finally  that $F^*(H)$ is isomorphic to $\PSL_2(q_0)$ as in (1), and note that $q_0 \ne 4,5$ by Theorem~\ref{crav}(i). Let $g$ be an element of $H$ of order $\frac{q_0-1}{(2,q_0-1)}$. As in the last paragraph of the proof of Lemma~\ref{l: type 9}, it is enough to show that $C_G(g) \not \le H$. So assume that $C_G(g) \le H$, in which case $C_G(g) = C_H(g)$ is a cyclic maximal torus of $G$ of order 
$\frac{q_0-1}{(2,q_0-1)}$ or $q_0-1$. Also, by Theorem \ref{crav}(iii), if $G(q) \ne E_8(q)$, then either $q_0=q$ or $G(q) = E_7(q)$ and $q_0 = 7,8$ or 25. The orders of cyclic maximal tori of $G(q)$ are given in \cite[Sec. 2]{kantorseress}, and there are none of order $q_0-1$ or $(q_0-1)/2$ with $q_0$ as in the previous sentence. Hence we may assume that $G(q) = E_8(q)$. Here the only possible cyclic maximal tori of order 
 $q_0-1$ or $(q_0-1)/2$ (and also with $q_0 \le t(G)$) have $q=2$ and $q_0 = 2^7$ or $2^8$. 
However, $\PSL_2(2^8) \not \le E_8(2)$, as $E_8(2)$ has no torus of order $2^8+1$. And if $F^*(H) = \PSL_2(2^7)$, then the element $g \in H$ of order $2^7-1$ lies in a subgroup $E_7(2)$ of $G$, and is centralized by an element of order 3 in $G\setminus H$, so $C_G(g) \not \le H$ and the conclusion follows. 
\end{proof}

\vspace{6mm}
This completes the proof of Theorem \ref{binex}.

\chapter{Classical Groups}\label{ch: classical}

In this chapter we prove Theorem \ref{t: main} for classical groups:

\begin{theo}\label{t: classical}
Let $G$ be an almost simple group with socle a classical group, and assume that $G$ has a primitive and binary action on a set $\Omega$. Then $|\Omega| \in \{5,6,8\}$ and $G \cong \symme(\Omega)$.
\end{theo}

The examples with $|\Omega| \in \{5,6,8\}$ arise via the isomorphisms listed after the statement of Theorem \ref{t: main}.

The case where $G$ has socle isomorphic to $\PSL_2(q)$ or $\PSU_3(q)$ has been dealt with in \cite{ghs_binary}, so Theorem~\ref{t: classical} is already proved in this case.

\section{Background on classical groups}\label{s: background classical}

Let us set up the group-theoretic notation that we need to prove Theorem~\ref{t: classical}. We assume throughout that our group $G$ is almost simple with socle a finite simple classical group. We write $M$ for the stabilizer in $G$ of a point in the action on $\Omega$. Since the action is primitive, $M$ is a maximal subgroup of $G$, and so we can use the classification of the maximal subgroups of the almost simple finite classical groups due to Aschbacher \cite{aschbacher2}. 
%In particular, we will use the version of Aschbacher's theorem due to Kleidman and Liebeck \cite{kl}.
This classification divides the maximal subgroups into nine families, labelled $\C_1$-$\C_8$ and $\mathcal{S}$. We shall give rough descriptions of these families at the beginning of each section of this chapter; full details can be found in \cite[Chapter 4]{kl}, to which we will often refer. The case where $M$ is in family $\C_1$ has been handled in \cite{gs_binary}. In this chapter we deal with the families $\C_2$-$\C_8$ and $\mathcal{S}$, in Sections \ref{s: c2}- \ref{s: S}. Some almost simple groups with socle $\mathrm{P}\Omega _8^+(q)$ or $\Sp_4(2^a)$ have extra families of maximal subgroups, and these are handled in the last Section \ref{s: missing}.

In what follows we shall take $S$ to be a certain quasisimple classical group for which $S/Z(S)$ is isomorphic to the socle of $G$: namely, $S$ will be one of $\SL_n(q)$, $\Sp_n(q)$, $\SU_n(q)$, $\Omega_n(q)$ (with $nq$ odd) or 
$\Omega_n^\varepsilon(q)$ (with $n$ even and $\varepsilon\in\{+,-\}$). 
As in \cite{kl}, we denote these cases by L, S, U and O.  
Sometimes, for uniformity of notation, we shall allow ourselves to write $\Omega_n^\varepsilon(q)$ also in the case where $n$ is odd -- in which case it  just denotes $\O_n(q)$. 
Note that we can think of $S$ as acting on the set $\Omega$ -- although we emphasise that this action is not necessarily primitive, and not necessarily faithful. We shall always write $\bar S$ for the simple group $S/Z(S)$.

The group $S$ is a subgroup of the group of isometries of some fixed bilinear, quadratic or sesquilinear form $\varphi$. We will write $V$ for the associated vector space of dimension $n$ over the field $\K$ where $\K=\bF_{q^u}$ 
with $u=2$ in case U, and $u=1$ otherwise.
The form $\varphi$ is either non-degenerate or the zero form (in the case $S=\SL_n(q)$). 

When $\varphi$ is non-degenerate, we will make use of a \emph{hyperbolic basis} $\B$ of $V$ of form \[\{e_1,\dots, e_k, f_1,\dots, f_k\}\cup \mathcal{A},\] where $k$ is the Witt index of $\varphi$, $\langle e_i, f_i\rangle$ are hyperbolic lines for $i=1,\dots, k$ and either $\mathcal{A}$ is empty, or $S$ is orthogonal and $\mathcal{A}$ has size at most $2$ and spans an anisotropic subspace of $V$.

\subsection{Basic assumptions}\label{s: assumption}

We make use of isomorphisms between classical groups of small dimension, as well as known results on Cherlin's conjecture to make the following assumptions.
\begin{enumerate}
 \item If $S=\SL_n(q)$, then $n\geq 3$ (using the main result of \cite{ghs_binary}).
 \item If $S=\SU_n(q)$, then $n\geq 4$ (using the main result of \cite{ghs_binary}).
 \item If $S=\Sp_n(q)$, then $n\geq 4$.
 \item If $S=\Omega_n(q)$ with $n$ odd, then $q$ is odd and $n\geq 7$.
 \item If $S=\Omega^\varepsilon_n(q)$ with $n$ even and with $\varepsilon\in\{+,-\}$, then $n\geq 8$.
\end{enumerate}

Notice that, under these assumptions, $S$ is quasisimple, unless $S=\Sp_4(2)$, in which case $S\cong \Sym(6)$.
%; the case of alternating socle was dealt with in \cite{gs_binary}.
% \item The only binary actions that we expect to encounter are:
 %\begin{itemize}
 % \item $S=\SL_4(2).2$, and the stabilizer of a point is isomorphic to $\symme(7)$, so is a member of Aschbacher's family $\C_9$.
 % \item $S=\Sp_4(2)$, and the stabilizer of a point is isomorphic to $\symme(5)$; there are two conjugacy classes of such subgroups, corresponding to Aschbacher's families $\C_3$ and $\C_8$.
 %\end{itemize}
%\end{enumerate}

In addition, by \cite{gs_binary}, we can assume that $M$ does not lie in Aschbacher's family $\C_1$. We also use \magma to exclude some small cases: 

\begin{lem}\label{l: beautifulsetssmall}
Let $G$ be an almost simple primitive group with socle one of the following groups 
\begin{enumerate} 
\item $\PSL_3(q)$ with $q\le 25$, $\PSL_4(q)$ with $2<q\le 9$ or $q \in \{16,25\}$, $\PSL_5(q)$ with $q\le 7$, $\PSL_6(q)$ with $q\le 4$, $\PSL_7(3)$, $\PSL_8(q)$ with $q\le 3$;
 \item $\PSU_4(q)$ with $q\le 7$, $\PSU_5(q)$ with $q\le 5$, $\PSU_6(q)$ with $q\le 3$, $\PSU_7(q)$ with $q\le 3$,  $\PSU_8(2)$;
 \item $\PSp_4(q)$ with $q\in\{4,5,8,16\}$, $\PSp_6(q)$ with $q\le 5$, $\PSp_8(q)$ with $q\le 3$;
 \item $\POmega_7(3),\,\POmega_8^-(2),\,\POmega_8^+(2),\,\POmega_8^+(3),\,\POmega_8^+(4)$, $\POmega_9(5)$, $\POmega_{10}^-(2)$, $\POmega_{12}^+(2)$.
\end{enumerate}
Then the action of $G$ is not binary.
\end{lem}
\begin{proof}
The \magma computations here are all rather similar. We give an indication of what we have done in the unitary case only.

We have computed all the possible almost simple groups $G$ and all of their (faithful) primitive actions on a set $\Omega$. We have tested that each of these actions is not binary. Indeed, except when $S=\SU_4(2)$ and $|\Omega|=27$, or  $S=\SU_4(3)$ and $|\Omega|=112$, or $S=\SU_4(4)$ and $|\Omega|=325$, we can witness that $G$ is non binary by applying Lemmas~\ref{l: M2},~\ref{l: added},~\ref{l: characters},~\ref{l: auxiliary}, or by finding a suitable non-binary triple. When $S=\SU_4(3)$ and $|\Omega|=112$, or $S=\SU_4(4)$ and $|\Omega|=325$, we can witness that $G$ is not binary by finding a suitable non-binary $4$-tuple. The case $S=\SU_4(2)$ and $|\Omega|=27$ requires a little more care because triples and $4$-tuples are not enough to witness that $G$ is not binary. We have proved that this group is non-binary using longer tuples (of length 7). 
\end{proof}

Finally, from here on, except for the final two sections (\S\ref{s: S} and \S\ref{s: missing}), we will assume that
\begin{itemize}
\item  if $S=\Sp_4(2^a)$, then $G\leq \GammaSp_4(2^a)$ (and so does not contain a graph automorphism); and 
\item if $S=\Omega_8^+(q)$, then $G\leq \mathrm{P\Gamma O}_8^+(q)$ (and so does not contain a triality automorphism).
\end{itemize}
These assumptions ensure that if $V$ denotes the natural $n$-dimensional module for $S$, then $G \le \PGammaL(V)$, except for the case where $S = \SL(V)$, in which case $G\le \PGammaL(V).2$ (where the $.2$ denotes a graph automorphism).

Note also, for future reference,  that by Proposition \ref{outeraut}, the maximal subgroups of $G$ that centralize field, graph-field or graph automorphisms are in families $\C_5$ (subfield subgroups) and $\C_8$ (classical subgroups).

\section{Family \texorpdfstring{$\C_2$}{C2}}\label{s: c2}

In this case $M$ is the projective image of the stabilizer of a direct sum decomposition $D$ of $V$ into $t$ subspaces $W_1,\dots, W_t$, each of dimension $m$, as described in \cite[\S 4.2]{kl}. In particular, $n=mt$. The possibilities are summarized in Table \ref{c2poss}.

\begin{table}[ht!]
\[
\begin{array}{|c|c|c|}
\hline
\hbox{case} &  \hbox{type} & \hbox{conditions} \\
\hline
{\rm L} & \GL_m(q) \wr \Sym(t) & \\
{\rm U} & \GU_m(q) \wr \Sym(t) & W_i \hbox{ non-degenerate} \\
{\rm S} & \Sp_m(q) \wr \Sym(t) & W_i \hbox{ non-degenerate} \\
{\rm O} & \Or_m^\d (q) \wr \Sym(t) & W_i \hbox{ non-degenerate} \\
{\rm U, S, O}^+ & \GL_{n/2}(q^u).2 & W_i \hbox{ totally singular,} \\
&& q \hbox{ odd in case S} \\
\hline
\end{array}
\]
\caption{Maximal subgroups in family $\C_2$} \label{c2poss}
\end{table}

The main result of this section is the following. The result will be proved in a series of lemmas.

\begin{prop}\label{p: c2}
 Suppose that $G$ is an almost simple group with socle $\bar S = \Cl_n(q)$, and assume that 
\begin{itemize}
\item[{\rm (i)}] $n\ge 3,4,4,7$ in cases $L,U,S,O$ respectively, and 
\item[{\rm (ii)}] $\Cl_n(q)$ is not one of the groups listed in Lemma $\ref{l: beautifulsetssmall}$.
\end{itemize}
Let $M$ be a maximal subgroup of $G$ in the family $\mathcal{C}_2$. Then the action of $G$ on $(G:M)$ is not binary.
\end{prop}

\subsection{Case \texorpdfstring{$S=\SL_n(q)$}{S=SL(n,q)}}

Assume that $S  = \SL_n(q)$ with $n\ge 3$, and the socle of $G$ is not as in Lemma \ref{l: beautifulsetssmall}(1). Assume also that $\Omega = (G:M)$, where $M$ is in the family $\C_2$ (as in the first row of Table \ref{c2poss}).

\begin{lem}\label{l: c2 sl}
 In this case, $\Omega$ contains a beautiful subset.
\end{lem}

%Note that, in light of Lemma~\ref{l: beautiful}, Lemma~\ref{l: c2 sl} implies Theorem~\ref{t: classical} for the case where $S=\SL_n(q)$ and $M$ is in the $\mathcal{C}_2$-family.

\begin{proof}
There is a basis $\B = \{v_1,\dots, v_{mt}\}$ of $V$  such that $M$ stabilizes the deomposition $V = W_1\oplus \cdots \oplus W_t$, where 

$$W_i = \langle v_{m(i-1)+1}, v_{m(i-1)+2}, \dots, v_{mi}\rangle.$$

First, assume that $q\geq 5$. We let $U$ be the subgroup whose elements fix all elements of $\B$ except $v_1$ and satisfy
 \[ 
 v_1\mapsto v_{1}+ k_1 v_{m+1},
 \]
 for some $k_1\in\Fq$, and we define $\Lambda=D^U$. For $k\in\Fq$, we define
 \[
  W_1(k)=\langle v_1+kv_{m+1}, v_2,\dots, v_m\rangle,
 \]
and observe that $\Lambda=\{D(k) \mid k\in\Fq\}$, where
\[
 D(k) = W_1(k) \oplus W_2 \oplus \cdots \oplus W_t.
\]
Note, in particular, that $D(0)=D$.

Let $T$ be the maximal split torus whose elements are diagonal when written with respect to $\B$. Then $U\rtimes T$ is $2$-transitive on $\Lambda = D^U$.

 Now suppose that $g\in S_\Lambda$ and suppose that $g$ maps $W_i$ to $W_1(k)$ for some $i>1$ and some $k\in\Fq$. This implies that there exists $v\in W_i$ such that $v^g = v_1+kv_{m+1}$. But $v_1+kv_{m+1}$ lies in $W_1(k)$ and not in $W_1(\ell)$ for all $\ell\neq k$ and, similarly, $v_1+kv_{m+1}$ does not lie in $W_j$ for all $j>1$. Thus $g$ does not preserve $\Lambda$, a contradiction. We conclude that $g$ preserves $\{W_1(k) \mid k\in\Fq\}$ set-wise, and preserves $\{W_2,\dots, W_t\}$ set-wise.

Our aim now is to show that $\Lambda$ is an $S$-beautiful subset; to do this, we will show that $S^\Lambda\cong U\rtimes T$. For this, we suppose  that $g\in S_\Lambda$ fixes both $D(0)$ and $D(1)$ and we will show that $g\in S_{(\Lambda)}$. Observe first that, since $g$ fixes $D(0)$, it follows that $g$ fixes $W_1$ and so
\[
 v_1^g\in\langle v_1,v_2,\dots, v_m\rangle.
\]
Similarly, since $g$ fixes $D(1)$, we conclude that $g$ fixes $W_1(1)$ and so
\[
 (v_1+v_{m+1})^g\in\langle v_1+v_{m+1},v_2,\dots, v_m\rangle.
\]
Hence $v_{m+1}^g\in \langle v_1,v_2,\dots, v_m, v_{m+1}\rangle$.

On the other hand $g$ preserves $\{W_2,\dots, W_t\}$ set-wise, and so 
\[
 v_{m+1}^g\in\langle v_{m+1}, \dots, v_{mt}\rangle
\]
which implies that $v_{m+1}^g\in \langle v_{m+1}\rangle$. In other words, for some $\ell_1\in\Fq$, we have
\begin{equation}\label{e: l1}
 v_{m+1}^g=\ell_1v_{m+1}.
\end{equation}
Now, since $g$ fixes $W_1$, we conclude that there exist $\ell_2, c_2, \dots, c_m$ such that
\begin{equation}\label{e: l2}
 v_1^g = \ell_2 v_1+\sum_{i=2}^m c_i v_i.
\end{equation}
Finally, since $g$ fixes $W_1(1)$,  there exist $\ell_3, d_2, \dots, d_m$ such that
\begin{equation}\label{e: l3}
 (v_1+v_{m+1})^g = \ell_3(v_1+v_{m+1})+\sum_{i=2}^m d_i v_i.
\end{equation}
From,~\eqref{e: l1},~\eqref{e: l2} and~\eqref{e: l3}, we conclude that $c_i=d_i$ for all $i=2,\dots, m$ and that $\ell_1=\ell_2=\ell_3=\ell$.

We finally obtain that, for each $k\in\Fq$,
\begin{align*}
 (v_1+kv_{m+1})^g &= (v_1+v_{m+1})^g + (k-1)v_{m+1}^g \\
 &= \ell(v_1+v_{m+1}) + \sum_{i=2}^m c_i v_i + (k-1)\ell v_{m+1} \\
 &= \ell(v_1+kv_{m+1}) + \sum_{i=2}^m c_i v_i \in W_1(k).
\end{align*}
Thus $g$ fixes $W_1(k)$ for each $k\in\Fq$, and so $g$ fixes $D(k)$ for each $k\in\Fq$. We conclude that $g\in S_{(\Lambda)}$, as required.

Next, assume that $q\in\{3,4\}$; then \cite[Tables 3.5.A and 3.5.H]{kl} allows us to assume that $m\geq 2$. We let $U$ be the subgroup whose elements fix all elements of $\B$ except $v_1$ and satisfy
 \[ 
 v_1\mapsto v_{1}+ k_1 v_{m+1} + k_2 v_{m+2},
 \]
 for some $k_1,k_2\in\Fq$, and we define $\Lambda=D^U$. For $k_1, k_2\in\Fq$, we define
 \[
  W_1(k_1, k_2)=\langle v_1+k_1v_{m+1}+k_2v_{m+2}, v_2,\dots, v_m\rangle,
 \]
and observe that $\Lambda=\{D(k_1, k_2) \mid k_1, k_2 \in\Fq\}$, where
\[
 D(k_1,k_2) = W_1(k_1,k_2) \oplus W_2 \oplus \cdots \oplus W_t.
\]
Note, in particular, that $D(0,0)=D$. Let $X$ be the stabilizer of the subspaces
\[
 \langle v_1\rangle , \, \ldots , \, \langle v_m\rangle , \, \langle v_{m+1} , v_{m+2} \rangle , \, \langle v_{m+3}\rangle , \dots , \, \langle v_{mt}\rangle .
\]
Then $U\rtimes X$ is $2$-transitive on $\Lambda = D^U$. Our aim now is to show that $\Lambda$ is a beautiful subset.

Take $g\in S_\Lambda$ and suppose that $\Lambda$ is not beautiful. An analogous argument to the previous case allows us to conclude that $g$ preserves $\{W_1(k_1, k_2) \mid k_1, k_2 \in\Fq\}$ set-wise, and preserves $\{W_2,\dots, W_t\}$ set-wise. This implies, moreover, that $g$ preserves the subspaces
\[
 Y_1:= {\rm span}_{\Fq}\{W_1(k_1, k_2) \mid k_1, k_2 \in\Fq\} \textrm{ and }  Y_0 :=\bigcap\limits_{k_1,k_2\in \Fq} W_1(k_1,k_2).
\]
Thus there is a homomorphism $\theta:S_\Lambda\to \GL(Y_1/ Y_0)\cong \GL_3(q)$. Since $\GL_3(q)$ does not contain a subgroup with a composition factor isomorphic to $\Alt(s)$ for $s\geq 8$, we conclude that the action of $\ker(\theta)$ on $\Lambda$ must induce $\Alt(\Lambda)$ or $\Sym(\Lambda)$. However $\ker(\theta)$ is not transitive on $\Lambda$ so we have a contradiction. %(\textbf{Isn't $\ker(\theta)$ in fact trivial on $\Lambda$????})

Next, assume that $q=2$; then \cite[Tables 3.5.A and 3.5.H]{kl} allows us to assume that $m\geq 3$. We let $U$ be the subgroup whose elements fix all elements of $\B$ except $v_1$ and satisfy
 \[ 
 v_1\mapsto v_{1}+ k_1 v_{m+1} + k_2 v_{m+2} + k_3 v_{m+3},
 \]
 for some $k_1,k_2,k_3\in\Fq$, and we define $\Lambda=D^U$. For $k_1, k_2, k_3\in\Fq$, we define
 \[
  W_1(k_1, k_2, k_3)=\langle v_1+k_1v_{m+1}+k_2v_{m+2} + k_3v_{m+3}, v_2,\dots, v_m\rangle,
 \]
and observe that $\Lambda=\{D(k_1, k_2,k_3) \mid k\in\Fq\}$, where
\[
 D(k_1,k_2,k_3) = W_1(k_1,k_2,k_3) \oplus W_2 \oplus \cdots \oplus W_t.
\]
Note, in particular, that $D(0,0,0)=D$. Let $X$ be the stabilizer of the subspaces
\[
 \langle v_1\rangle , \, \ldots , \, \langle v_m\rangle , \, \langle v_{m+1} , v_{m+2}, v_{m+3} \rangle , \, \langle v_{m+4}\rangle , \dots , \, \langle v_{mt}\rangle. 
\]
Then $U\rtimes X$ is $2$-transitive on $\Lambda = D^U$. Suppose that $g\in S_\Lambda$ fixes both $D(0,0,0)$ and $D(1,0,0)$. Observe first that, since $g$ fixes $D(0,0,0)$, it follows that $g$ fixes $W_1$ and so
\[
 v_1^g\in\langle v_1,v_2,\dots, v_m\rangle.
\]
Similarly, since $g$ fixes $D(1,0,0)$, it also fixes $W_1(1,0,0)$ and so
\[
 (v_1+v_{m+1})^g\in\langle v_1+v_{m+1},v_2,\dots, v_m\rangle.
\]
We conclude that $v_{m+1}^g\in \langle v_1,v_2,\dots, v_m, v_{m+1}\rangle$.

On the other hand $g$ preserves $\{W_2,\dots, W_t\}$ set-wise, and so 
\[
 v_{m+1}^g\in\langle v_{m+1}, \dots, v_{mt}\rangle,
\]
which implies that $v_{m+1}^g\in \langle v_{m+1}\rangle$. Since we are working over $\mathbb{F}_2$, we conclude that $v_{m+1}^g=v_{m+1}$. Now, we can repeat this same argument assuming that $g$ also fixes $D(0,1,0)$ and $D(0,0,1)$ and we see that $v_{m+2}^g=v_{m+2}$ and $v_{m+3}^g=v_{m+3}$. But in this case $g$ clearly fixes $W_1(k_1,k_2,k_3)$ for all $k_1,k_2,k_3\in \mathbb{F}_2$, and so $g$ fixes $D(k_1,k_2,k_3)$ for all $k_1, k_2,k_3\in \mathbb{F}_2$. Thus if $g\in S_\Lambda$  and fixes the four points $D(0,0,0)$, $D(1,0,0)$, $D(0,1,0)$ and $D(0,0,1)$ of $\Lambda$, then $g$ fixes all of $\Lambda$. Since $|\Lambda|=8$, we conclude that $S^\Lambda$ does not contain $\Alt(\Lambda)$, and so $\Lambda$ is a beautiful subset.
\end{proof}

\subsection{The totally singular case}

In this section we deal with the case when $S$ preserves a non-degenerate form on $V$, $t=2$ and $W_1$ and $W_2$ are both totally  singular. This occurs when $S$ is unitary, symplectic with $q$ odd, or of type $\Or^+$ (as in the last row of Table \ref{c2poss}). So assume that $S$ is one of these types, and also that the socle of $G$ is not as in Lemma \ref{l: beautifulsetssmall}.

\begin{lem}\label{l: c2 ti}
 In this case $\Omega$ contains a beautiful subset.
\end{lem}

\begin{proof}
 We can assume that $W_1=\langle e_1,\dots, e_m\rangle$ and $W_2=\langle f_1,\dots, f_m\rangle$, where $\B = \{e_1,\dots, e_m, f_1, \dots, f_m\}$ is a hyperbolic basis for $V$ (and $m\ge 2$). Note that, with respect to the basis $\B$,  $M$ contains the group of matrices
\begin{equation}\label{e: mat}
 \left\{ \begin{pmatrix}
          A & \\ & A^{-T}
         \end{pmatrix} \, \mid \, A\in \GL_m(\Fq)\right\},
\end{equation}
except in the $\Or^+$ case with $q$ odd, when we need to add the requirement that ${\rm det}(A)$ is a square in $\Fq$ for such matrices.
 Notice that in the unitary case we have written $A^{-T}$ rather than $A^{-T\sigma}$ (where $\sigma$ is the involutory automorphism of the field $\Fqt$) since we are only considering matrices with entries from the field $\Fq$.

First, assume that $q\geq 5$. We let $U$ be the subgroup whose elements fix all elements of $\B$ except $e_1$ and $e_2$ and satisfy
 \[ 
 e_1\mapsto e_{1}+ kf_2 \textrm{ and } e_2\mapsto e_2\pm k f_1,
 \]
 for some $k\in\Fq$. (The choice of sign for the image of $e_2$ will depend on the type of form preserved by $S$.) We define $\Lambda=D^U$ and, for $k\in\Fq$, we define
 \[
  W_1(k)=\langle e_1+kf_2, e_2\pm k f_1, e_3,\dots, e_m\rangle.
 \]
 Observe that $\Lambda=\{D(k) \mid k\in\Fq\}$, where
\[
 D(k) = W_1(k) \oplus W_2.
\]
Note, in particular, that $D(0)=D$.

Let $T$ be the maximal split torus whose elements are diagonal when written with respect to $\B$. Then $U\rtimes T$ is $2$-transitive on $\Lambda = D^U$. Our aim now is to show that $\Lambda$ is an $S$-beautiful subset; to do this, we will show that $S^\Lambda\cong U\rtimes T$. For this, we suppose  that $g\in S_\Lambda$ fixes both $D(0)$ and $D(1)$ and we will show that $g\in S_{(\Lambda)}$. 

Observe that, since $g$ fixes $D(0)$ and $D(1)$, $g$ must fix $W_2$ and hence must fix $W_1$ and $W_1(1)$. The fact that $g$ fixes $W_1$ implies that
\[
 e_1^g\in\langle e_1,e_2,\dots, e_m\rangle.
\]
Similarly, the fact that $g$ fixes $W_1(1)$ implies that
\[
 (e_1+f_2)^g\in\langle e_1+f_2,e_2\pm f_1, e_3, \dots, e_m\rangle.
\]
We conclude that $f_2^g\in \langle e_1,e_2,\dots, e_m, f_1,f_2\rangle$.

On the other hand $g$ also fixes $W_2$, and so $f_2^g\in W_2$ and we conclude that $f_2^g\in \langle f_1, f_2\rangle$; in other words, there exist $\ell_1, \ell_1'\in\Fq$ such that
\begin{equation}\label{e: ll1}
 f_2^g=\ell_1'f_1\pm\ell_1f_2.
\end{equation}
Now, since $g$ fixes $W_1$, we conclude that there exist $\ell_2, \ell_2', c_3, \dots, c_m$ such that
\begin{equation}\label{e: ll2}
 e_1^g = \ell_2 e_1+ \ell_2' e_2 + \sum_{i=3}^m c_i e_i.
\end{equation}
Finally, since $g$ fixes $W_1(1)$, we conclude that there exist $\ell_3, \ell_3', d_3, \dots, d_m$ such that
\begin{equation}\label{e: ll3}
 (e_1+f_2)^g = \ell_3(e_1+f_2)+\ell_3'(e_2\pm f_1) + \sum_{i=3}^m d_i e_i.
\end{equation}
From~\eqref{e: ll1},~\eqref{e: ll2} and~\eqref{e: ll3}, we conclude that $c_i=d_i$ for all $i=3,\dots, m$, $\ell_1=\ell_2=\ell_3$ and $\ell_1'=\ell_2'=\ell_3'$. Set $\ell:=\ell_1$ and $\ell':=\ell_1'$.

We finally obtain that, for each $k\in\Fq$,
\begin{align*}
 (e_1+kf_2)^g &= (e_1+f_2)^g + (k-1)f_2^g \\
 &= \ell(e_1+f_2)+\ell'(e_2\pm f_1) + \sum_{i=3}^m c_i e_i + (k-1)(\pm\ell'f_1+\ell f_2)\\
 &= \ell(e_1+kf_2)+\ell'(e_2\pm kf_1) + \sum_{i=3}^m c_i e_i \in W_1(k).
\end{align*}
Thus $g$ fixes $W_1(k)$ for each $k\in\Fq$, and so $g$ fixes $D(k)$ for each $k\in\Fq$. We conclude that $g\in S_{(\Lambda)}$, as required.

Next assume that $q\in\{3,4\}$. Since $\SU_4(3), \SU_4(4)$ and $\Sp_4(3)$ were dealt with at the start (recall that in the symplectic case, $q$ is odd), we require $n\geq 6$. Our argument is similar to before. We let $U$ be the subgroup whose elements fix all elements of $\B$ except $e_1$, $e_2$ and $e_3$ and satisfy
 \[ 
 e_1\mapsto e_{1}+ k_2 f_2 + k_3 f_3,
 \]
 for some $k_2,k_3\in\Fq$, and we define $\Lambda=D^U$. For $k_2, k_3\in\Fq$, we define
 \[
  W_1(k_2, k_3)=\langle e_1+k_2f_2+k_3f_3, e_2,\dots, e_m\rangle,
 \]
and observe that $\Lambda=\{D(k_2, k_3) \mid k_2, k_3 \in\Fq\}$ where
\[
 D(k_2,k_3) = W_1(k_2,k_3) \oplus W_2 .
\]
Note, in particular, that $D(0,0)=D$. Let $X$ be the stabilizer of the subspaces
\[
 \langle e_1\rangle , \, \langle e_2,e_3\rangle, \, \langle e_4\rangle,\dots , \, \langle e_m\rangle , \, \langle f_1\rangle , \, \langle f_2 , f_3 \rangle , \, \langle f_4\rangle ,\dots , \, \langle f_m \rangle. 
\]
Then $U\rtimes X$ is $2$-transitive on $\Lambda = D^U$ (making use of Lemma~\ref{l: 2t}  and the fact that $M$ contains the matrices in \eqref{e: mat}). Our aim now is to show that $\Lambda$ is a beautiful subset of size $q^2$.

Take $g\in S_\Lambda$ and suppose that $\Lambda$ is not beautiful. An analogous argument to the previous case allows us to conclude that $g$ preserves $\{W_1(k_2, k_3) \mid k_2, k_3 \in\Fq\}$ set-wise. This implies that $g$ preserves the subspaces
\[
 Y_1:= {\rm span}_{\K}\{W_1(k_2, k_3) \mid k_2, k_3 \in\Fq\} \textrm{ and }Y_0:=\bigcap\limits_{k_2,k_3\in \Fq} W_1(k_2,k_3).
\]
Thus there is a homomorphism $\theta:S_\Lambda\to \GL(Y_1/Y_0)\cong \GL_3(\K)$. Since $\GL_3(\K)$ does not contain a subgroup with a composition factor isomorphic to $\Alt(s)$ for $s>7$, we conclude that the action of $\ker(\theta)$ on $\Lambda$ must induce $\Alt(\Lambda)$ or $\Sym(\Lambda)$. However $\ker(\theta)$ is not transitive on $\Lambda$ so we have a contradiction.

Finally, assume that $q=2$. In this case, we require that $n\geq 8$. This requirement excludes the groups $\SU_4(2)$ and $\SU_6(2)$ which were dealt with at the start. In addition Lemma~\ref{l: beautifulsetssmall} allows us to exclude the case when $S=\Omega_8^+(2)$ (notice that, in the proof, we prove that there is a beautiful subset).  If $n\geq 10$, then our method here is similar to the previous case, but we start with a subgroup $U$ whose elements fix all elements of $\B$ except $e_1$, $e_2$, $e_3$, $e_4$ and $e_5$ and satisfy
 \[ 
 e_1\mapsto e_{1}+ k_2 f_2 + k_3 f_3 + k_4f_4 + k_5 f_5,
 \]
 for some $k_2,k_3, k_4, k_5 \in\Fq$. We obtain a beautiful subset of cardinality $16$; we leave the details to the reader. For the case when $S=\SU_8(2)$, we use a subgroup $U$ whose elements fix all elements of $\B$ except $e_1$, $e_2$ and $e_3$ and satisfy 
 \[ 
 e_1\mapsto e_{1}+ k_2 f_2 + k_3 f_3,
 \]
 for some $k_2,k_3 \in\mathbb{F}_4$. Again we obtain a beautiful subset of cardinality $16$.
\end{proof}

\subsection{A general reduction}

In light of the previous subsections, we can now assume that we are in the case when $S$ preserves a non-degenerate form on $V$ and $W_1,\dots, W_t$ are all non-degenerate subspaces of $V$ of dimension $m$. In the end we will need to split into separate cases, depending on the type of the form, but before we do that we give three general lemmas that significantly reduce the subsequent case work.

\begin{lem}\label{l: c2 q5}
 If $q\geq 5$, then either $\Omega$ contains a beautiful subset or else one of the following holds:
 \begin{enumerate}
  \item $m=1$;
  \item $m=2$, $S=\Omega_n^\varepsilon(q)$ for some $\varepsilon\in\{+,-\}$, and $W_1,\dots, W_t$ are all of type $\Or_2^-$.
 \end{enumerate}
\end{lem}
\begin{proof}
Suppose that neither of the listed outcomes occurs -- we must show that $\Omega$ contains a beautiful subset. Choose a hyperbolic basis for each of $W_1,\dots, W_t$ and let $\B$ be the  union of these bases.

Since we have excluded the two listed outcomes, we can let $(e_1,f_1)$ (resp. $(e_2,f_2)$) be hyperbolic pairs whose elements are in $\B\cap W_1$ (resp. $\B\cap W_2$). Let $U$ be the subgroup whose elements fix all elements of $\B$ except $e_1$ and $f_2$ and satisfy
 \[ 
 e_1\mapsto e_{1}+ k e_2, \,\,\, f_2 \mapsto f_2\pm k f_1,
 \]
for some $k\in\Fq$, and we let $\Lambda=M^U$. (The choice of sign for the image of $f_2$ will depend on the type of form preserved by $S$.) For $k\in\Fq$, we define
 \begin{align*}
  W_1(k) &=\langle e_1+ke_2, f_1, x_1\dots, x_{m-2}\rangle, \\
  W_2(k) &=\langle e_2, f_2\pm k f_1, y_1 ,\dots, y_{m-2}\rangle,
  \end{align*}
where $\B\cap W_1 = \{e_1, f_1, x_1,\dots, x_{m-2}\}$ and $\B\cap W_2 = \{e_2, f_2, y_1,\dots, y_{m-2}\}$. Observe that $\Lambda=\{D(k) \mid k\in\Fq\}$, where
\[
 D(k) = W_1(k) \oplus W_2(k) \oplus W_3 \oplus \cdots \oplus W_t.
\]
Note, in particular, that $D(0)=D$. Now we follow the argument of Lemma~\ref{l: c2 ti} with some slight adjustments. As before, we are able to conclude that, if $g\in S_\Lambda$, then $g$ preserves the set $\{W_3,\dots, W_t\}$ as well as the set
\[
 \{W_1(k) \mid k \in \Fq\} \cup \{W_2(k) \mid k \in \Fq\}.
\]
Now suppose that $g\in S_\Lambda$ and $g$ fixes both $D(0)$ and $D(1)$. We study the image of the spaces $W_1, W_2, W_1(1)$ and $W_2(1)$. 

Suppose that $W_1^g=W_1$ and $W_1(1)^g=W_2(1)$. This implies that $f_1^g\in W_1\cap W_2(1)=\{0\}$, a contradiction. Similarly, we cannot have $W_1^g=W_2$ and $W_1(1)^g=W_1(1)$. 

Suppose next that $W_1^g=W_1$ and $W_1(1)^g=W_1(1)$. 
 Then there exist $a,b, c_1,\dots, c_{m-2}$ such that
 \[
  e_1^g = ae_1+bf_1+\sum\limits_{i=1}^{m-2} c_i x_i.
 \]
Similarly there exist $a',b', c_1',\dots, c_{m-2}'$ such that
 \[
  (e_1+e_2)^g = a'(e_1+e_2)+b'f_1+\sum\limits_{i=1}^{m-2} c_i' x_i.
 \]
We obtain that
\[
  e_2^g = (a'-a)e_1+ a'e_2+(b'-b)f_1+\sum\limits_{i=1}^{m-2} (c_i'-c_i) x_i.
\]
But, since $e_2\in W_2\cap W_2(1)$ and since $g$ fixes $D(0)$ and $D(1)$, we deduce $e_2^g\in W_2\cap W_2(1)={\rm span}_\K\{e_2,y_1,\dots, y_{m-2}\}$. Now,  we conclude $e_2^g\in{\rm span}_\K\{e_2\}$. This implies, in particular, that $a'=a$, $b'=b$, $c_1'=c_1, \dots, c_{m-2}'=c_{m-2}$ and, in particular $e_2^g=ae_2$.

But now observe that
\begin{align*}
 (e_1+ke_2)^g&=(e_1+e_2)^g+(k-1)e_2^g \\
 &= a(e_1+e_2)+bf_1+\sum\limits_{i=1}^{m-2} c_i x_i + (k-1)ae_2 \\
 &= a(e_1+ke_2)+bf_1+\sum\limits_{i=1}^{m-2} c_i x_i \in W_1(k).
\end{align*}
This shows that $W_1(k)^g=W_1(k)$, for every $k\in \Fq$. We conclude that $D(k)^g=D(k)$ for all $k\in \Fq$.

So let us consider the remaining case, when $W_1^g=W_2$ and $W_1(1)^g=W_2(1)$.
 Then there exist $a,b, c_1,\dots, c_{m-2}$ such that
 \[
  e_1^g = ae_2+bf_2+\sum\limits_{i=1}^{m-2} c_i y_i.
 \]
Similarly there exist $a',b', c_1',\dots, c_{m-2}'$ such that
 \[
  (e_1+e_2)^g = a'e_2+b'(f_2\pm f_1)+\sum\limits_{i=1}^{m-2} c_i' y_i.
 \]
We obtain that
\[
  e_2^g = (a'-a)e_2+b'(f_2\pm f_1)-bf_2+\sum\limits_{i=1}^{m-2} (c_i'-c_i) y_i.
\]
But, since $e_2\in W_2\cap W_2(1)$ and since $g$ fixes $D(0)$ and $D(1)$, we deduce $e_2^g\in W_1\cap W_1(1)={\rm span}_\K\{f_1,x_1,\dots, x_{m-2}\}$, and we conclude that $e_2^g\in{\rm span}_\K\{f_1\}$. This implies, in particular, that $a'=a$, $b'=b$, $c_1'=c_1, \dots, c_{m-2}'=c_{m-2}$ and, in particular $e_2^g=\pm bf_1$.

But now observe that
\begin{align*}
 (e_1+ke_2)^g&=(e_1+e_2)^g+(k-1)e_2^g \\
 &=ae_2+b(f_2\pm f_1)+\sum\limits_{i=1}^{m-2} c_i y_i + (k-1)(\pm b)f_1 \\
 &= ae_2+b(f_2\pm kf_1)+\sum\limits_{i=1}^{m-2} c_i y_i \in W_2(k).
\end{align*}
This shows that $W_1(k)^g=W_2(k)$, for every $k\in \Fq$. We conclude that $D(k)^g=D(k)$ for all $k\in \Fq$.

In all cases, then, we conclude that, if $g\in S_\Lambda$ and $g$ fixes the two points $D(0)$ and $D(1)$ of $\Lambda$, then $g$ fixes all elements of $\Lambda$. But this implies that $S^\Lambda$ does not contain $\Alt(\Lambda)$ and we are done.
\end{proof}

%Note that, in the preceding proof, we refer to the field $\Fq$ \emph{even in the unitary case} where $\Fq$ is a proper subfield of $\K$. 
The next case deals with the first outcome of the preceding lemma, but also applies when $q=4$. 
%It, and the lemma that follows it, are good examples of ``the quotient method'' (\S\ref{s: quotient}).

\begin{lem}\label{l: c2 m1}
 If $m=1$ and $q\geq 4$, then either the action is not binary or else $S$ is orthogonal and $q=5$.
\end{lem}

\begin{proof}
Our method is based on the treatment of this case for $\PSU_3(q)$ in \cite{ghs_binary}. Note that $m=1$ cannot occur if $S$ is symplectic -- thus we may assume that $S$ is either unitary or orthogonal. In the orthogonal case, $m=1$ implies that $q$ is an odd prime number by~\cite[\S4.2]{kl}.

The action of $G$ on $\Omega$ is permutation equivalent to the natural action of $G$ on 
\[
\left\{\{X_1,X_2,\dots, X_n\} \,\middle|\, 
\begin{array}{c}
\dim_{\K}(X_1)=\dim_{\K}(X_2)=\cdots=\dim_{\K}(X_n)=1;\\
V=X_1\perp X_2\perp \cdots \perp X_n; \, X_1,X_2, \dots, X_n \textrm{ non-degenerate}
\end{array}
\right\}. 
\]
Let $\{v_1, \dots, v_n\}$ be an orthonormal basis for $V$; thus $\omega_0:=\{\langle v_1\rangle,\langle v_2\rangle,\dots,\langle v_n\rangle\}\in\Omega$. If $S$ is unitary, then define $n_0=3$; if $S$ is orthogonal, then define $n_0=4$.

Now consider 
\[
\Lambda:=\{\{X_1, X_2, \dots, X_n\}\in \Omega\mid X_i=\langle v_i\rangle \textrm{ for } i=n_0,\dots, n\}. 
\]
Then $G_{\Lambda}$ is equal to the stabilizer of $\{\langle v_{n_0}\rangle,\dots, \langle v_n\rangle\}$. In the unitary case $G^\Lambda$ is almost simple with socle $\PSU_2(q)$; in the orthogonal case, $G^\Lambda$ is almost simple with socle $\Omega_3(q)$ (here we are using $q\geq 4$). In both cases the socle is isomorphic to $\PSL_2(q)$, and the action of $G_\Lambda$ on $\Lambda$ is permutation equivalent to the action of $G_{\{\langle v_{n_0}\rangle,\dots, \langle v_n\rangle\}}$ on 
\begin{align*}
\Lambda':=\{\{X_1,\cdots,X_{n_0-1}\}\mid& \dim(X_1)=\cdots =\dim(X_{n_0-1}),\\
& \langle v_{n_0},\dots, v_n\rangle^\perp=X_1\perp\cdots\perp X_{n_0-1}, X_1,\ldots,X_{n_0-1} \textrm{ non-degenerate}\}. 
\end{align*}
Suppose first that $S$ is unitary. Then, this action of $G^\Lambda$ has degree $|\Lambda|=q(q-1)/2$. By consulting the table of the maximal subgroups of almost simple groups with socle $\PSL_2(q)$ in~\cite{bhr}, we see that provided $q\notin \{7,9\}$ this action is primitive and hence by applying~\cite[Theorem~1.1]{ghs_binary} to $G^{\Lambda}$, we obtain that $G^{\Lambda}$ is not binary. Moreover, when $q\in \{7,9\}$, it can be easily checked with \texttt{magma} that the action of $G^\Lambda$ is not binary.

Suppose now that $S$ is orthogonal. (Recall that $q$ is a prime number.) In particular, $G^\Lambda\cong \PSL_2(q)$ or $G^\Lambda\cong \PGL_2(q)$. Let us denote by $X$ the socle of $G^\Lambda$ and by $Y$ the stabilizer in $G^\Lambda$ of an element of $\Lambda$. By consulting the table of the maximal subgroups of almost simple groups with socle $\PSL_2(q)$ in~\cite[Table~$8.7$]{bhr}, we have
\[
X\cap Y\cong 
\begin{cases}
\Sym(4),&\textrm{when }q\equiv \pm 1\pmod 8\\
\Alt(4),&\textrm{when }q\equiv \pm 3,5,\pm 11,\pm 13,\pm 19\pmod{40}.
\end{cases}
\]
From the same table we infer that $X \cap Y$ is a maximal subgroup of $X$ unless 
$$G^\Lambda=X\hbox{ and }q\equiv \pm 11,\pm 19\pmod{40}.$$
Therefore, except for the cases where $q=p\equiv \pm 11,\pm 19 \pmod {40}$ and  $G^\Lambda=X\cong\PSL_2(q)$, by applying~\cite[Theorem~1.1]{ghs_binary} to $G^{\Lambda}$ we obtain that $G^{\Lambda}$ is not binary for $q\ne 5$. (Observe that $q=5$ is the exception listed in the statement of the lemma.)   We claim that this is the case also when $q=p\equiv \pm 11,\pm 19\pmod {40}$ and $G^\Lambda=X\cong\PSL_2(q)$. To see this, observe that from~\cite[Table~$8.7$]{bhr}, there exists a subgroup $Z$ of $X=G^\Lambda$ with $Y<Z$, $Z\cong \Alt(5)$ and with $Z$ maximal in $X=G^\Lambda$. Now, the action of $Z$ on $(Z:Y)$ is not binary because it is permutation isomorphic to the non-binary degree $5$ action of $\Alt(5)$.  Hence $G^\Lambda$ is not binary by Lemma~\ref{l: subgroup}, as claimed. 

Summing up, for the rest of the proof we may suppose that $G^\Lambda$ is not binary.
 In particular there exist two $\ell$-tuples $$(\{W_{1,1},\dots, W_{1,n_0-1}\},\ldots,\{W_{\ell,1},\dots, W_{\ell,n_0-1}\})$$ and $$(\{W'_{1,1},\dots, W'_{1,n_0-1}\},\ldots,\{W'_{\ell,1},\dots, W'_{\ell,n_0-1}\})$$ in $\Lambda'^\ell$ which are $2$-subtuple complete for the action of $G_{\Lambda}$ but not in the same $G_{\Lambda}$-orbit. By construction the two $\ell$-tuples 
 \begin{align*}
I :=&(
\{W_{1,1},\dots, W_{1,n_0-1},\langle v_{n_0}\rangle,  \dots, \langle v_n\rangle\},\{W_{2,1},\dots,W_{2,n_0-1},\langle v_{n_0}\rangle,  \dots, \langle v_n\rangle\},\ldots, \\
&\{W_{\ell,1},\dots,W_{\ell,n_0-1},\langle v_{n_0}\rangle,  \dots, \langle v_n\rangle\}), \\
J:=&(\{W'_{1,1},\dots,W'_{1,n_0-1},\langle v_{n_0}\rangle,  \dots, \langle v_n\rangle\},\{W'_{2,1},\dots,W'_{2,n_0-1},\langle v_{n_0}\rangle,  \dots, \langle v_n\rangle\},\ldots, \\
&\{W'_{\ell,1},\dots,W'_{\ell,n_0-1},\langle v_{n_0}\rangle,  \dots, \langle v_n\rangle\})  
 \end{align*}
 are in $\Omega^\ell$ and are $2$-subtuple complete. Furthermore it is easy to see that $I$ and $J$  are not in the same $G$-orbit. Thus $G$ is not binary.
\end{proof}

In the next lemma we write $X_n(q)$ to represent one of the three families of classical groups associated with non-degenerate forms. So, for instance, to get the result in the unitary case, the reader should read ``$\PSU$'' wherever $X$ occurs.

\begin{lem}\label{l: c2 small m and q}
Let $m$ be a fixed positive integer, and let $n_0$ be a multiple of $m$ such that $X_{n_0}(q)$ is almost simple. If the primitive $\mathcal{C}_2$-action of $X_{n_0}(q)$ on $m$-decompositions of fixed type is not binary, then the same is true of $X_n(q)$, for all $n$ that are multiples of $m$ and that exceed $n_0$.
 \end{lem}

Note that the caveat ``of fixed type'' is included to account for the orthogonal case with $m$ even, where we have decompositions of type $\Or^+$ and $\Or^-$.
 
\begin{proof}
We assume that $n>n_0$ and we proceed similarly to the previous lemma. First we identify $\Omega$ with the set
 \[
\left\{\{X_1,X_2,\dots, X_t\} \,\middle|\, 
\begin{array}{c} \dim_{\K}(X_1)=
\dim_{\K}(X_2)=\cdots=\dim_{\K}(X_t)=m; \\
V=X_1\perp X_2\perp \cdots \perp X_t; \, X_1,X_2, \dots, X_t \textrm{ non-degenerate}
\end{array}\right\}. 
\]
Now, fix an element $\{W_1, \dots, W_t\}$  of $\Omega$ and consider 
\[
\Lambda:=\left\{\{X_1, X_2, \dots, X_t\}\in \Omega \,\middle|\,  X_i=W_i \textrm{ for } i=\frac{n_0}{m}+1,\dots, t\right\}. 
\]
Clearly, $G_{\Lambda}$ is equal to the stabilizer of $\{W_{n_0/m+1},\dots, W_t\}$ and $G^\Lambda$ is almost simple with socle isomorphic to $X_{n_0}(q)$, and the action of $G_\Lambda$ on $\Lambda$ is permutation equivalent to the action of $G_{\{ W_{n_0/m+1},\dots,  W_t \}}$ on 
\[
\Lambda':=\left \{\{W_1,\dots, W_{n_0/m}\}\,\middle|\,
\begin{array}{c}\dim(W_1)=\cdots = \dim(W_{n_0/m}); \, W_1,\dots, W_{n_0/m} \textrm{ non-degenerate}; \\
\langle W_{n_0/m+1},\dots, W_t\rangle^\perp=W_1\perp \cdots\perp W_{n_0/m} 
  \end{array}
\right\}. 
\]
Therefore $G^\Lambda$ is an almost simple primitive group with socle isomorphic to $X_{n_0}(q)$ in a $\C_2$-action on  $m$-decompositions of given type. By assumption, $G^{\Lambda}$ is not binary and hence there exist two $\ell$-tuples $(\{W_{1,1},\dots, W_{1,n_0/m}\},\ldots,\{W_{\ell,1},\dots W_{\ell,n_0/m}\})$ and $(\{W'_{1,1},\dots W'_{1,n_0/m}\},\ldots,\{W'_{\ell,1},\dots W'_{\ell,n_0/m}\})$ in $\Lambda^\ell$ which are $2$-subtuple complete for the action of $G_{\Lambda}$ but not in the same $G_{\Lambda}$-orbit. By construction the two $\ell$-tuples 
 \begin{align*}
I& :=(\{W_{1,1},\dots,W_{1,n_0/m},W_{n_0/m+1},\dots, W_t\},\{W_{2,1},\dots,W_{2,n_0/m},W_{n_0/m+1},\dots, W_t\},\ldots, \\
&\{W_{\ell,1},\dots,W_{\ell,n_0/m},W_{n_0/m+1},\dots, W_t\}), \\
J&:=(\{W'_{1,1},\dots,W'_{1,n_0/m},W_{n_0/m+1},\dots, W_t\},\{W'_{2,1},\dots,W'_{2,n_0/m},W_{n_0/m+1},\dots, W_t\},\ldots, \\
&\{W'_{\ell,1},\dots,W'_{\ell,n_0/m},W_{n_0/m+1},\dots, W_t\})  
 \end{align*}
 are in $\Omega^\ell$ and are $2$-subtuple complete. As before, $I$ and $J$  are not in the same $G$-orbit. Thus $G$ is not binary.
\end{proof}

\subsection{Case where \texorpdfstring{$S=\SU_n(q)$}{S=SU(n,q)} and the \texorpdfstring{$W_i$}{Wi} are non-degenerate}\label{s: c2 su}

Assume that $S=\SU_n(q)$, the $W_i$ are non-degenerate, and the socle of $G$ is not as in Lemma \ref{l: beautifulsetssmall}. Here $V = V_n(\K)$ where $\K = \Fqt$, and we denote by $\sigma$ the involutory field automorphism of $\K$.

\begin{lem}\label{l: c2 su}
 In this case either $\Omega$ contains a beautiful subset or else $S$ is listed in Table~$\ref{t: c2 sun}$.
\end{lem}

\begin{table}\centering
\begin{tabular}{cc}
\toprule[1.5pt]
Group & Details of action \\
\midrule[1.5pt]
$\SU_n(q)$ & $m=1$ \\
$\SU_4(3)$, $\SU_4(4)$ & $m=2$ \\
$\SU_n(2)$ & $m=3$ \\
%$\SU_8(2)$ & $m=4$ \\
\bottomrule[1.5pt]
\end{tabular}
\caption{$\C_2$ -- $S=\SU_n(q)$ and the $W_i$ are non-degenerate.}\label{t: c2 sun}
\end{table}

\begin{proof}
Lemma~\ref{l: c2 q5} implies that when $q\geq 5$, either $\Omega$ contains a beautiful subset or else we obtain the first line of Table~\ref{t: c2 sun}. Now assume that $q\leq 4$ and $m\ge 2$.

If $q\in\{3,4\}$, then we repeat the same set-up as Lemma~\ref{l: c2 q5}, except that this time $U$ is the subgroup whose elements fix all elements of $\B$ except $e_1$ and $f_2$ and satisfy
 \begin{align*}
 e_1&\mapsto e_{1}+ k e_2,\\
 f_2& \mapsto f_2-k^\sigma f_1,
 \end{align*}
 for some $k\in\K=\mathbb{F}_{q^2}$, and we let $\Lambda=M^U$. Notice that $\Lambda$ is of size $q^2$ rather than $q$ as in Lemma~\ref{l: c2 q5}. Now the same argument as before allows us to conclude that $\Lambda$ is a beautiful subset of order $q^2$, provided that $n> 4$. (When $n=4$, we cannot conclude that $G_\Lambda$ acts 2-transitively on $\Lambda$; notice that the groups $\SU_4(3)$ and $\SU_4(4)$ are listed as exceptions in Table~\ref{t: c2 sun}.) Now the argument of Lemma~\ref{l: c2 q5} implies that $\Lambda$ is a beautiful subset, as required.
 
If $q=2$ and $m\geq 2$, then \cite[Table 3.5.H]{kl} implies that $m\geq 3$; if $m=3$, the action is listed in Table~\ref{t: c2 sun}, hence we assume that $m\geq 4$. We consider hyperbolic pairs from $\B$ as before; this time assume that $e_1,f_1,e_2,f_2\in W_1$ and $e_3,f_3,e_4,f_4\in W_2$. Let $x_1,\dots, x_{m-4}, y_1,\dots, y_{m-4} \in \B$ be such that 
\[
 W_1={\rm span}_\K \{e_1, e_2, f_1, f_2, x_1,\dots, x_{m-4}\} \textrm{ and }
 W_2={\rm span}_\K \{e_3, e_4, f_3, f_4, y_1,\dots, y_{m-4}\}.\]
We let $U$ be the subgroup whose elements fix all elements of $\B$ except $e_1, f_3$ and $f_4$ and satisfy
 \begin{align*}
 e_1&\mapsto e_{1}+ k_3e_3+ k_4e_4,\\
 f_3& \mapsto f_3-k_3^\sigma f_1,\\
  f_4& \mapsto f_4-k_4^\sigma f_1,
 \end{align*}
 for some $k_3, k_4\in\K$, and we define $\Lambda=D^U$. For $k_1, k_2\in\K$, we define
 \begin{align*}
  W_1(k_1, k_2)& =\langle e_{1}+ k_1 e_3 + k_2e_4, e_2, f_1, f_2, x_1, \dots, x_{m-4}\rangle \textrm{ and }\\
  W_2(k_1, k_2)& =\langle e_3, e_4, f_3-k_1^\sigma f_1, f_4-k_2^\sigma f_1, y_1, \dots, y_{m-4}\rangle.
   \end{align*}
Observe that $\Lambda=\{D(k_1, k_2) \mid k_1, k_2 \in\K\}$, where
\[
 D(k_1,k_2) = W_1(k_1,k_2) \oplus W_2(k_1,k_2) \oplus W_3 \oplus \cdots \oplus W_t.
\]
Note, in particular, that $D(0,0)=D$. Let $X$ be the stabilizer of the subspaces
\[
 \langle e_1\rangle, \, \langle e_2\rangle, \, \langle e_3, e_4\rangle, \, \langle f_1\rangle, \, \langle f_2\rangle, \, \langle f_3, f_4\rangle, \, \langle x_1,\dots, x_{m-4}\rangle, \, \langle y_1,\dots, y_{m-4}\rangle, \, W_3,\dots, \, W_t. 
\]
Then $U\rtimes X$ is $2$-transitive on $\Lambda = D^U$, a set of size $16$. Our aim now is to show that $\Lambda$ is a beautiful subset. Let $g\in S_\Lambda$. As in Lemma~\ref{l: c2 q5} we can see that $g$ preserves the set 
\[\{W_1(k_1, k_2)\mid k_1, k_2\in\K\}\cup \{W_2(k_1, k_2)\mid k_1, k_2\in\K\}.\] 
Now suppose that there exist $k_1, k_2, k_1', k_2'\in \K$ such that $W_1(k_1,k_2)^g=W_2(k_1',k_2')$; then, by considering the vectors $e_2^g, f_1^g, f_2^g$, it is clear that for all $k_1,k_2\in\K$, there exist $k_1', k_2'\in \K$ such that $W_1(k_1,k_2)^g=W_2(k_1',k_2')$. We conclude that $S_\Lambda$ has a subgroup $H$ of index at most $2$ such that, if $g\in H$, then for all $k_1, k_2\in \K$ there exist $k_1', k_2' \in \K$ (which may depend upon $g,k_1,k_2$) such that $W_1(k_1,k_2)^g=W_1(k_1',k_2')$.

This implies that $H$ preserves the subspaces
\[
 Y_1:= {\rm span}_{\K}\{W_1(k_1, k_2) \mid k_1, k_2 \in\K\} \textrm{ and }Y_0:=\bigcap\limits_{k_1,k_2\in \K} W_1(k_1,k_2).
\]
Thus there is a homomorphism $\theta:H\to \GL(Y_1/Y_0)\cong \GL_3(\K)$. Since $\GL_3(\K)$ does not contain a subgroup with a composition factor isomorphic to $\Alt(s)$ for $s>7$, we conclude that either $\Lambda$ is beautiful or the action of $\ker(\theta)$ on $\Lambda$ must induce $\Alt(\Lambda)$ or $\Sym(\Lambda)$. However $\ker(\theta)$ is not transitive on $\Lambda$ and the result follows.
\end{proof}

%DETAILS OF MAGMA CALCULATIONS THAT ARE NO LONGER NEEDED. When $S=\PSU_8(2)$, the group is too large to use the previous method. In this case, $G=\PSU_8(2)$ or $G=\mathrm{P}\Gamma\mathrm{U}_8(2)$. We have considered only the maximal subgroups $M$ in the Aschbacher class $\mathcal{C}_2$ and we have constructed the permutation character $\pi$ for the permutation action of $G$ on the right cosets of $M$ in $G$. We have checked that $\langle \pi(\pi-1)(\pi-2),1\rangle>\langle\pi(\pi-1),1\rangle^3$ and hence this action is not binary.

We need to deal with the cases listed in Table~\ref{t: c2 sun}. Lemma~\ref{l: beautifulsetssmall} deals with the second line of the table.
Now Lemma~\ref{l: c2 m1} means that, to deal with the first line of Table~\ref{t: c2 sun}, we may assume that $q\in\{2,3\}$. Thus the next lemma deals with what remains.

 \begin{lem}\label{l: c2 su m=1 or 3}
 Suppose that $(q,m)$ is one of $(2,3)$, $(2,1)$ or $(3,1)$. Then the action is not binary.
 \end{lem}
\begin{proof}
By Lemma~\ref{l: beautifulsetssmall} we have  $n\ge 7$.
Now Lemma~\ref{l: c2 small m and q} implies that the result holds for $n\geq 7$.
\end{proof}

\subsection{Case where \texorpdfstring{$S=\Sp_n(q)$}{S=Sp(n,q)} and the \texorpdfstring{$W_i$}{Wi} are non-degenerate}

Assume that $S=\Sp_n(q)$ with $n\ge 4$, the $W_i$ are non-degenerate, and the socle of $G$ is not as in Lemma~\ref{l: beautifulsetssmall}.

\begin{table}\centering
\begin{tabular}{cc}
\toprule[1.5pt]
Group & Details of action \\
\midrule[1.5pt]
$\Sp_n(2)$, $\Sp_n(3)$, $\Sp_n(4)$ & $m=2$ \\
%$\Sp_n(2)$ & $m=2$ or $4$ \\
\bottomrule[1.5pt]
\end{tabular}
\caption{$\C_2$ -- $S=\Sp_n(q)$ and the $W_i$ are non-degenerate.}\label{t: c2 spn}
\end{table}

\begin{lem}\label{l: c2 sp}
 In this case either $\Omega$ contains a beautiful subset or else $S$ is listed in Table~$\ref{t: c2 spn}$.
\end{lem}
\begin{proof}
Lemma~\ref{l: c2 q5} implies that, when $q\geq 5$, $\Omega$ contains a beautiful subset. Now assume that $q\leq 4$. Choose a hyperbolic basis for each of $W_1,\dots, W_t$ and let $\B$ be the union of these bases. Write $m=2\ell$ and order the hyperbolic basis so that $e_1, f_1,\dots, e_\ell, f_\ell\in W_1$; $e_{\ell+1}, f_{\ell+1},\dots, e_{2\ell}, f_{2\ell}\in W_2$ and so on.

We exclude the case $m=2$, since this is listed in Table~\ref{t: c2 spn} and we assume that $m\geq 4$. Now let $U$ be the subgroup whose elements fix all elements of $\B$ except $e_1, e_{\ell+1}, e_{\ell+2}, f_{\ell+1}$ and $f_{\ell+2}$ and satisfy
 \[ 
\begin{array}{l}
 e_1\mapsto e_{1}+ k_1 e_{\ell+1} + k_2e_{\ell+2} + k_3 f_{\ell+1} + k_4 f_{\ell+2},\\
e_{\ell+1} \mapsto e_{\ell+1}+k_3f_1,\\
e_{\ell+2} \mapsto e_{\ell+2}+k_4f_1, \\
f_{\ell+1} \mapsto f_{\ell+1}-k_1f_1,\\ f_{\ell+2} \mapsto f_{\ell+2}-k_2f_1, 
\end{array}
 \]
 for some $k_1,k_2, k_3, k_4\in\Fq$, and we define $\Lambda=D^U$. For $k_1, k_2, k_3, k_4\in\Fq$, we define
 \begin{align*}
  W_1(k_1, k_2, k_3, k_4)& =\langle e_{1}+ k_1 e_{\ell+1} + k_2e_{\ell+2} + k_3 f_{\ell+1} + k_4 f_{\ell+2},e_2,\dots, e_\ell, f_1, \dots, f_\ell\rangle \textrm{ and }\\
  W_2(k_1, k_2, k_3, k_4)& =\langle e_{\ell+1} + k_3 f_1, e_{\ell+2}+k_4 f_1, e_{\ell+3},\dots, e_{2\ell}, f_{\ell+1}-k_1 f_1, f_{\ell+2}-k_2f_1, f_{\ell+3}, \dots, f_{2\ell}\rangle,
   \end{align*}
and observe that $\Lambda=\{D(k_1, k_2, k_3, k_4) \mid k_1, k_2, k_3, k_4 \in\Fq\}$, where
\[
 D(k_1,k_2,k_3,k_4) = W_1(k_1,k_2,k_3,k_4) \oplus W_2(k_1,k_2,k_3,k_4) \oplus W_3 \oplus \cdots \oplus W_t.
\]
Note, in particular, that $D(0,0,0,0)=D$. Let $X$ be the stabilizer of the subspaces
\[
 \langle e_{\ell+1} , e_{\ell+2}, f_{\ell+1} , f_{\ell+2} \rangle , \, \langle e_1\rangle , \dots , \, \langle e_\ell\rangle , \, \langle e_{\ell+3}\rangle , \dots , \, \langle e_{\ell t}\rangle , \,
 \langle f_1\rangle , \dots , \, \langle f_\ell\rangle , \, \langle f_{\ell+3}\rangle , \dots , \, \langle f_{\ell t}\rangle.
\]
Then $U\rtimes X$ is $2$-transitive on $\Lambda = D^U$ (making use of Lemma~\ref{l: 2t} and the fact that $\Sp_4(q)$ acts transitively on the set of non-zero vectors in the natural $\Fq$-module), a set of size $q^4$. Our aim now is to show that $\Lambda$ is a beautiful subset. Let $g\in S_\Lambda$. As before we can see that $g$ preserves the set 
\[\{W_1(k_1,k_2,k_3,k_4)\mid k_1,k_2,k_3,k_4\in\Fq\}\cup \{W_1(k_1,k_2,k_3,k_4)\mid k_1,k_2,k_3,k_4\in\Fq\}.\] 
Now suppose that there exist $k_1,k_2,k_3,k_4, k_1',k_2',k_3',k_4'\in \Fq$ such that $W_1(k_1,k_2,k_3,k_4)^g=W_2(k_1',k_2',k_3',k_4')$; then, by considering the vectors $e_2^g, f_1^g, f_2^g$, it is clear that, for all $k_1,k_2,k_3,k_4\in\Fq$, there exist $k_1',k_2',k_3',k_4'\in \Fq$ such that $W_1(k_1,k_2,k_3,k_4)^g=W_2(k_1',k_2',k_3',k_4')$. We conclude that $S_\Lambda$ has a subgroup $H$ of index at most $2$ such that, if $h\in H$, then for all $k_1,k_2,k_3,k_4\in \Fq$ there exist $k_1',k_2',k_3',k_4' \in \Fq$ with  $W_1(k_1,k_2,k_3,k_4)^g=W_1(k_1',k_2',k_3',k_4')$.

We conclude that $H$ preserves the subspaces
\[
 Y_1:= {\rm span}_{\K}\{W_1(k_1,k_2,k_3,k_4) \mid k_1,k_2,k_3,k_4 \in\Fq\} \textrm{ and }Y_0:=\bigcap\limits_{k_1,k_2,k_3,k_4\in \Fq} W_1(k_1,k_2,k_3,k_4).
\]
Thus there is a homomorphism $\theta:S_\Lambda\to \GL(Y_1/Y_0)\cong \GL_5(\Fq)$. By Lemma \ref{l: alt sections classical}, 
$\GL_5(\Fq)$ does not have a section isomorphic to $\Alt(s)$ for $s>8$, so we conclude that either $\Lambda$ is a beautiful set, or the action of $\ker(\theta)$ on $\Lambda$ must induce $\Alt(\Lambda)$ or $\Sym(\Lambda)$. However $\ker(\theta)$ is not transitive on $\Lambda$, and we conclude that $\Lambda$ is beautiful as required.
\end{proof}

We must show that the actions listed in Table~\ref{t: c2 spn} are not binary; the next lemma does the job.

  \begin{lem}\label{l: c2 sp m=2 or 4}
 Suppose that $(q,m)$ is one of $(2,2)$, $(3,2)$ or $(4,2)$. Then the action is not binary.
 \end{lem}
\begin{proof}
Lemma \ref{l: beautifulsetssmall} gives the result for $n=4$. 
\begin{comment}
MAGMA CALCS FOR N<=8 ARE HERE BUT NOT ALL NEEDED.
 We use \magma to verify the truth of this statement for $n\leq 8$, using various methods. First, the groups with socle $\mathrm{PSp}_6(2)$ are of no concern here because $\mathrm{PSp}_6(q)$ has no maximal $\mathcal{C}_2$-subgroups when $q=2$. Second, the group $\mathrm{PSp}_4(2)'\cong\mathrm{Alt}(6)$ and hence the statement is true by \cite{gs_binary}. 
 
 Next, for the groups having socle $\mathrm{PSp}_4(3)$, $\mathrm{PSp}_4(4)$, $\mathrm{PSp}_6(3)$ and $\mathrm{PSp}_8(2)$, we have verified that the primitive actions under considerations admit a beautiful subset and hence are not binary.
 
 The remaining groups are groups having socle $S$ one of $\mathrm{PSp}_6(4)$, $\mathrm{PSp}_8(3)$ and $\mathrm{PSp}_8(4)$; due to their size, we have dealt with these cases differently. Let $M$ be a maximal subgroup of $S$ in the Aschbacher class $\mathcal{C}_2$, let $1_S$ be the principal character of $S$ and let $\pi_M$ be the permutation character for the action of $S$ on the right cosets of $M$. We have verified that in all cases $$\langle \pi(\pi-1_S)(\pi-2\cdot 1_S),1_S\rangle>(|\mathrm{Out}(S)|\langle \pi(\pi-1_S),1_S\rangle)^3.$$
In particular, all actions under consideration are not binary in view of Lemma~\ref{l: characters}.
\end{comment} 
  Now Lemma~\ref{l: c2 small m and q} implies that the result holds for $n>4$.
\end{proof}

\subsection{Case where \texorpdfstring{$S=\Omega_n(q)$}{S=Omega(n,q)} for \texorpdfstring{$nq$}{nq} odd, and the \texorpdfstring{$W_i$}{Wi} are non-degenerate}

\begin{table}\centering
\begin{tabular}{cc}
\toprule[1.5pt]
Group & Details of action \\
\midrule[1.5pt]
$\Omega_n(p)$ & $m=1$ \\
$\Omega_n(3)$ & $m=3$ \\
%$\Sp_n(2)$ & $m=2$ or $4$, $U_1$ is non-degenerate. \\
\bottomrule[1.5pt]
\end{tabular}
\caption{$\C_2$ -- $S=\Omega_n(q)$ with $nq$ odd and the $W_i$ are non-degenerate.}\label{t: c2 omegaoddn}
\end{table}

\begin{lem}\label{l: c2 omegaodd}
 In this case either $\Omega$ contains a beautiful subset or else $S$ is listed in Table~$\ref{t: c2 omegaoddn}$.
\end{lem}

\begin{proof}
Lemma~\ref{l: c2 q5} implies that when $q\geq 5$, either $\Omega$ contains a beautiful subset or else $m=1$; in the latter case, \cite[Table 3.5.D]{kl} implies that $q=p$, a prime, and we obtain the first line of Table~\ref{t: c2 omegaoddn}. If $q=3$ and $m=3$, then we obtain the second line of Table~\ref{t: c2 omegaoddn}.

Assume, then, that $q=3$ and $m\geq 5$. The assumption on $m$ means that each $W_i$ contains at least two hyperbolic pairs. Now the argument of Lemma~\ref{l: c2 su} for $q=2$ carries over here and we obtain a beautiful subset of size $9$.
\end{proof}

We must show that the actions listed in Table~\ref{t: c2 omegaoddn} are not binary. Lemma~\ref{l: c2 m1} deals with the first line, provided $q>5$; the next lemma deals with what remains.

\begin{lem}\label{l: c2 omegaodd remaining}
 Suppose that $(q,m)$ is one of $(3,1)$, $(5,1)$, $(3,3)$. Then the action is not binary.
 \end{lem}
\begin{proof}
We begin by checking the truth of this statement for $S\in \{\Omega_5(3), \Omega_5(5), \Omega_9(3)\}$. For the first two cases it follows from Lemma \ref{l: beautifulsetssmall}. And when $S=\Omega_9(3)$, we use the permutation character method. Let $M$ be a maximal subgroup of $S$ in the Aschbacher class $\mathcal{C}_2$, let $1_S$ be the principal character of $S$ and let $\pi_M$ be the permutation character for the action of $S$ on the right cosets of $M$. We have verified that in all cases $$\langle \pi(\pi-1_S)(\pi-2\cdot 1_S),1_S\rangle>(|\mathrm{Out}(S)|\langle \pi(\pi-1_S),1_S\rangle)^3.$$
In particular, all actions under consideration are not binary in view of Lemma~\ref{l: characters}.

Now Lemma~\ref{l: c2 small m and q} implies that the result holds for all $n\geq 7$, as required.
\end{proof}

\subsection{Case where  \texorpdfstring{$S=\Omega_n^\pm(q)$}{S=Omega+-(n,q)} and the \texorpdfstring{$W_i$}{Wi} are non-degenerate}

\begin{table}\centering
\begin{tabular}{cl}
\toprule[1.5pt]
Group & Details of action \\
\midrule[1.5pt]
$\Omega_n^+(p)$ & $p\geq 3$, $m=1$ \\
$\Omega_n^+(q)$ & $W_i$ of type ${\rm O}_2^-$ \\
$\Omega_n^+(4)$ & $W_i$ of type $\Or_2^\pm$ or $\Or_4^-$ \\
$\Omega_n^+(3)$ & $W_i$ of type $\Or_2^\pm$, $\Or_3$ or $\Or_4^-$ \\
$\Omega_n^+(2)$ & $W_i$ of type $\Or_2^\pm$, $\Or_4^\pm$ or $\Or_6^-$ \\
%$\Omega_{12}^+(2)$ & $t=2$ \\
\bottomrule[1.5pt]
\end{tabular}
\caption{$\C_2$ -- $S=\Omega_n^+(q)$ -- and the $W_i$ are non-degenerate.}\label{t: c2 omegaplusn}
\end{table}

\begin{table}\centering
\begin{tabular}{cl}
\toprule[1.5pt]
Group & Details of action \\
\midrule[1.5pt]
$\Omega_n^-(p)$ & $p\geq 3$, $m=1$ \\
$\Omega_n^-(q)$ & $W_i$ of type ${\rm O}_2^-$ \\
$\Omega_n^-(4)$ & $W_i$ of type $\Or_2^-$ or $\Or_4^-$ \\
$\Omega_n^-(3)$ & $W_i$ of type $\Or_2^-$, $\Or_3$ or $\Or_4^-$ \\
$\Omega_n^-(2)$ & $W_i$ of type $\Or_2^-$, $\Or_4^-$ or $\Or_6^-$ \\
%$\Omega_{10}^-(3)$ & $t=2$ \\
\bottomrule[1.5pt]
\end{tabular}
\caption{$\C_2$ -- $S=\Omega_n^-(q)$ -- and the $W_i$ are non-degenerate.}\label{t: c2 omegaminusn}
\end{table}

\begin{lem}\label{l: c2 omegaeven}
 In this case either $\Omega$ contains a beautiful subset or else $S$ is listed in Table~$\ref{t: c2 omegaplusn}$ or Table~$\ref{t: c2 omegaminusn}$.
\end{lem}
\begin{proof}
If $q\geq 5$, then Lemma~\ref{l: c2 q5} yields the first two lines of each table. We also use the fact, from \cite[Tables~3.5.E and 3.5.F]{kl}, that if $m=1$, then $q=p\geq 3$, where $p$ is prime. %Note that the restriction on $n$ is required to obtain an $\Or_2^-$-type decomposition in an $\Or_n^+(q)$-module.
Assume, then, that $q\leq 4$. Recall that if $q$ is even, then $m$ is even. 
We consider the case where $q\in\{3,4\}$ first. We require that $W_1$ contains at least two orthogonal hyperbolic lines. All cases that do not satisfy this requirement are listed in the tables.

Now we let $e_1,f_1$ be a hyperbolic pair in $W_1$, and $e_{\ell+1},f_{\ell+1},e_{\ell+2},f_{\ell+2}$ two hyperbolic pairs in $W_2$. We let $U$ be the subgroup whose elements fix all elements of $\B$ except $e_1, f_{\ell+1}$ and $f_{\ell+2}$ and satisfy
 \[
 \begin{array}{l} 
 e_1\mapsto e_{1}+ k_1 e_{\ell+1} + k_2e_{\ell+2}, \\
  f_{\ell+1}\mapsto f_{\ell+1}-k_1 f_1, \\
   f_{\ell+2}\mapsto f_{\ell+2}-k_2f_1,
   \end{array}
 \]
 for some $k_1,k_2\in\Fq$, and we define $\Lambda=D^U$. We now proceed using the argument for $q=2$ in Lemma~\ref{l: c2 su} to conclude that we have a beautiful subset of size $q^2$. (Note that $S^\Lambda$ contains $\ASL_2(q)$, hence the  $2$-transitivity of $S^\Lambda$ is immediate.) 
 
If $q=2$, then the argument is similar, but we require that $W_1$ contains at least three orthogonal hyperbolic lines. All cases that do not satisfy this requirement are listed in the tables. We let $U$ be the subgroup whose elements fix all elements of $\B$ except $e_1, f_{\ell+1}, f_{\ell+2}$ and $f_{\ell+3}$ and satisfy
\[
 \begin{array}{l}
 e_1\mapsto e_{1}+ k_1 e_{\ell+1} + k_2e_{\ell+2}+k_3 e_{\ell+3}, \\
  f_{\ell+1}\mapsto f_{\ell+1}-k_1 f_1, \\ 
  f_{\ell+2}\mapsto f_{\ell+2}-k_2f_1, \\ f_{\ell+3}\mapsto f_{\ell+3}-k_3f_1,
 \end{array}
 \]
 for some $k_1,k_2, k_3\in\Fq$, and we define $\Lambda=D^U$. As before we get a homomorphism $\theta: S_\Lambda \mapsto \GL(Y_1/Y_0) \cong \mathrm{GL}_4(2)$. Moreover, $\Lambda$ corresponds to a set of 8 vectors in $Y_1/Y_0$, namely the set of vectors $e_{1}+ k_1 e_{\ell+1} + k_2e_{\ell+2}+k_3 e_{\ell+3}+Y_0$. Since the stabilizer in $\mathrm{GL}_4(2)$ of this set does not induce $\Alt(8)$, it follows as before that $\Lambda$ is a beautiful subset of size 8,  completing the proof. %(\textbf{Proof changed -- OK???}) 
\end{proof}

We must show that the actions listed in Tables~\ref{t: c2 omegaplusn} and \ref{t: c2 omegaminusn} are not binary. Lemma~\ref{l: c2 m1} deals with the first line of each table, provided $q>5$. The next lemma deals with the second line of each table.

\begin{lem}\label{l: c2 or m2}
 If the $W_i$ are of type $\Or_2^-$, then the action is not binary.
\end{lem}

\begin{proof}
The proof is similar to that of Lemma~\ref{l: c2 m1}. 

Note that the action of $G$ on $\Omega$ is permutation equivalent to the natural action of $G$ on 
\[
\left\{\{X_1,X_2,\dots, X_{n/2}\} \,\middle|\, 
\begin{array}{c}
X_1,\dots, X_{n/2} \textrm{ of type } \Or_2^-;\\
V=X_1\perp X_2\perp \cdots \perp X_{n/2}; \, X_1,X_2, \dots, X_{n/2} \textrm{ non-degenerate}
\end{array}
\right\}. 
\]
Now consider
\[
\Lambda:=\{\{X_1, X_2, \dots, X_{n/2}\}\in \Omega\mid X_i=W_i \textrm{ for } i\in \{4,\dots, n/2\}\}. 
\]
Then $G_{\Lambda}$ is equal to the stabilizer of $\{W_4, \dots, W_{n/2}\}$ and  $G^\Lambda$ is almost simple with socle $\mathrm{P}\Omega_6^-(q)$; therefore, the socle of $G^\Lambda$ is isomorphic to $\PSU_4(q)$ and the action is isomorphic to a $\mathcal{C}_2$-action of an almost simple group with socle $\PSU_4(q)$ on non-degenerate $1$-spaces of $\mathbb{F}_{q^2}^4$. We saw in \S\ref{s: c2 su} that this action is not binary, thus there exist two $\ell$-tuples $(\{W_{1,1},W_{1,2}, W_{1,3}\},\ldots, \{W_{\ell,1},W_{\ell,2}, W_{\ell,3}\})$ and $(\{W'_{1,1},W'_{1,2}, W'_{1,3}\},\ldots$ ,$\{W'_{\ell,1},W'_{\ell,2}, W'_{\ell,3}\})$ in $\Lambda^\ell$ which are $2$-subtuple complete for the action of $G_{\Lambda}$ but not in the same $G_{\Lambda}$-orbit. By construction the two $\ell$-tuples 
 \begin{align*}
I :=&(
\{W_{1,1}, W_{1,2}, W_{1,3}, W_4, \dots, W_{n/2}\},\{W_{2,1}, W_{2,3}, W_{2,3} ,W_4, \dots, W_{n/2}\},\ldots, \\
&\{W_{\ell,1}, W_{\ell,2}, W_{\ell,3},W_4, \dots, W_{n/2}\}), \\
J:=&(\{W'_{1,1},W'_{1,2}, W'_{1,3},W_4, \dots, W_{n/2}\},\{W'_{2,1},W'_{2,2},W'_{2,3},W_4, \dots, W_{n/2}\},\ldots, \\
&\{W'_{\ell,1},W'_{\ell,2},W'_{\ell,3},W_4, \dots, W_{n/2}\})  
 \end{align*}
 are in $\Omega^\ell$ and are $2$-subtuple complete. Moreover, $I$ and $J$  are not in the same $G$-orbit. Thus $G$ is not binary.
\end{proof}

The next lemma deals with the remaining lines of Tables \ref{t: c2 omegaplusn} and \ref{t: c2 omegaminusn}.

\begin{lem}\label{l: c2 omegaeven remaining}
Suppose that $(q,m)$ is  in
 \[
  \{(3,1), (5,1), (2,2), (3,2), (4,2), (3,3), (2,4), (3,4), (4,4), (2,6)\}.
 \]
Then the action is not binary.
 \end{lem}
\begin{proof}
If $m=1$, then we use the fact that we have already studied all $\mathcal{C}_2$ actions for $n$ odd. In particular Lemma~\ref{l: c2 omegaodd remaining} attends to the case where $(m,n,q)\in\{(1,7,3), (1,7,5)\}$. Now Lemma~\ref{l: c2 small m and q} yields the result for $(q,m)\in\{(3,1), (5,1)\}$ and $n\geq 8$. 

If $m=3$, a similar argument works, using the fact that Lemma~\ref{l: c2 omegaodd remaining} attends to the case where $(m,n,q)=(3,9,3)$.

If $m=2$ or $4$, then we must deal with the cases where $q\in\{2,3,4\}$ (note that when $m=2$, Lemma~\ref{l: c2 or m2} allows us to assume that the $W_i$ are of type $\Or_2^+$). When $n=8$, Lemma~\ref{l: beautifulsetssmall} gives the result for $q=2$. We use \magma to verify that, when $n=8$ and $q\in\{3,4\}$, then the corresponding $\mathcal{C}_2$ actions are not binary. Lemma~\ref{l: c2 small m and q} then implies the result for $n>8$.

If $m=6$, then we must deal with the case $q=2$ and the situation where the $W_i$ are of type $\Or_6^-$. We consider first what happens when $n=12$: note that, in this case, $S=\Omega_{12}^+(2)$, since $n/m$ is even. Now we use {\tt magma}, with the permutation character method (using Lemma~\ref{l: characters}), to confirm that, in the case $S=\Omega_{12}^+(2)$, the action is not binary. Now Lemma~\ref{l: c2 small m and q} implies the result for $n>12$.
\end{proof}

\section{Family \texorpdfstring{$\C_3$}{C3}}\label{s: c3}

In this section, the subgroup $M$ is a ``field extension subgroup''. Such subgroups are described in \cite[Section 4.3]{kl}, and are listed in Table \ref{c3poss}. In every case we start with a field extension $\K_\#$ of $\K$ of prime degree. We will usually denote this degree by the letter ``$r$'', although in a few subfamilies, the degree is always equal to $2$.
We set $m=n/r$. 

\begin{table}[ht!]
\[
\begin{array}{|c|c|c|}
\hline
\hbox{case} &  \hbox{type} & \hbox{conditions} \\
\hline
{\rm L} & \GL_m(q^r) & \\
{\rm U} & \GU_m(q^r) & r \hbox{ odd} \\
{\rm S} & \Sp_m(q^r) &  \\
{\rm O^\e } & \Or_m^\e (q^r) & m \ge 3 \\
{\rm S, O}^\e & \GU_{n/2}(q) & r=2, q\hbox{ odd in case S} \\
\hline
\end{array}
\]
\caption{Maximal subgroups in family $\C_3$} \label{c3poss}
\end{table}

In the case $S=\SL_n(q)$, the group $M$ is embedded in $G$ by considering the group $\GammaL_m(\K_\#)$ acting on an $m$-dimensional vector space $V_\#$ over $\K_\#$ and then considering those $\K_\#$-semilinear transformations of $V_\#$ that induce $\K$-linear transformations on $V$, where $V$ is simply $V_\#$ viewed as a $\K$-vector space. For the other cases, one must also consider $\K_\#$-forms defined on $V_\#$; full details are given in \cite{kl}. 

It is convenient to give a geometrical interpretation for the set of right cosets of $M$  (which is the set on which we are acting). To do this, we take $\K_\#$ to be a field extension of our field $\K$ and we wish to define a \emph{$\K_\#$-structure on $V$}. 

We start by considering a $\K$-linear isomorphism $\phi:V_\#\to V$. Let $\Sigma$ be the set of all such isomorphisms, and we observe that two groups act naturally on $\Sigma$: 
\begin{enumerate}
 \item $\GL_m(\K_\#)$ acts on $\Sigma$ via $\phi^g(\mathbf{v}) = \phi(\mathbf{v}^{g^{-1}})$;
 \item $\GL_n(\K)$ acts on $\Sigma$ via $\phi^h(\mathbf{v}) = (\phi(\mathbf{v}))^h$.
\end{enumerate}
Clearly these two actions commute. Thus we define a $\K_\#$-structure on $V$ to be an orbit of the group $\GL_m(\K_\#)$ on $\Sigma$, and (using commutativity of the actions) we observe that $\GL_n(\K)$ acts on the set of all $\K_\#$-structures on $V$. What is more, the stabilizer of such an action is a field extension subgroup $M$, hence we have the geometrical interpretation that we require.

Note that we can replace the word ``linear'' with the word ``semilinear'' in the previous paragraph to extend this geometrical interpretation to subgroups of $\GammaL_n(\K)$.

The main result of this section is the following. The result will be proved in a series of lemmas.

\begin{prop}\label{p: c3}
 Suppose that $G$ is an almost simple group with socle $\bar S = \Cl_n(q)$, and assume that 
\begin{itemize}
\item[{\rm (i)}] $n\ge 3,4,4,7$ in cases $L,U,S,O$ respectively, and 
\item[{\rm (ii)}] $\Cl_n(q)$ is not one of the groups listed in Lemma $\ref{l: beautifulsetssmall}$.
\end{itemize}
Let $M$ be a maximal subgroup of $G$ in the family $\mathcal{C}_3$. Then the action of $G$ on $(G:M)$ is not binary.
\end{prop}

\subsection{Case \texorpdfstring{$S=\SL_n(q)$}{S=SL(n,q)}}

\begin{lem}\label{l: c3 sln}
  Suppose that $G$ is almost simple with socle equal to $\PSL_n(q)$. Let $M$ be a $\mathcal{C}_3$-maximal subgroup such that $F^*(M)$ contains $M_1$, a quasisimple cover of $\PSL_m(q^r)$, where $n=mr$, $r$ is prime and $m\geq 3$. Then the action of $G$ on $(G:M)$ is not binary.
\end{lem}
\begin{proof}
%Our proof mimics the proof of Proposition~\ref{p: psl element}, although we are not able to apply that proposition directly. 
We define 
\[
 \tilde{x}=\begin{pmatrix}
            1 & & \\ & A & \\ & & a^{-1}
           \end{pmatrix} \in \mathrm{SL}_m(q^r),
\]
where $A$ is an element of $\GL_{m-2}(q^r)$ of order $(q^r)^{m-2}-1$ and $a=\det(A)$. We let $x$ be the element in $F^*(M)$ which is the projective image of $\tilde{x}$ in $M_1$. Observe that $\tilde{x}$ has a 1-eigenspace over $\mathbb{F}_{q^r}$ of dimension 1, and so has a 1-eigenspace over $\Fq$ of dimension $r$; we conclude that $C_M(x)$ is a proper subgroup of $C_G(x)$. Thus there is a suborbit, $\Delta$, on which the action of $M$ is isomorphic to the action of $M$ on $(M:H)$, where $H$ is a subgroup of $M$ that does not contain $M_1$ and does contain $x$. 

We now refer to Lemma~\ref{l: psl element}. This shows that either the action of $M$ on $\Delta$ is not binary, or $M$ has a section $\Alt(q^{r(m-2)})$. In the former case, the conclusion of the lemma holds. In the latter case, Lemma \ref{l: alt sections classical} implies that the only possibility is $m=3$, $q=2$, $r=2$. But then $S=\mathrm{SL}_6(2)$, a case covered by Lemma \ref{l: beautifulsetssmall}.
\end{proof}

%If $(r,q)=(2,3)$, then $|\Delta|=9$ and $F^*(M)=\SL_3(9)$. Since $\SL_3(9)$ does not contain a section isomorphic to $\Alt(8)$, the argument in Proposition~\ref{p: psl element} yields a beautiful subset, and Lemma~\ref{l: beautiful} implies the result. The same argument works when $(r,q)=(3,2)$: here $|\Delta|=8$ and $F^*(M)=\SL_3(8)$ which does not contain a section isomorphic to $\Alt(7)$.

%Suppose now $G=S=\mathrm{SL}_6(2)$. In the case that the maximal $\mathcal{C}_3$-subgroup is of the for $\mathrm{PSL}_2(8).7.3$, we have computed the permutation representation of this action. The stabilizer $G_\omega $ of a point $\omega\in \Omega$ in this action has an orbit $\Delta$ of cardinality $63$. Let us denote by $H$ the permutation group induced by the point stabilizer on this suborbit $\Delta$. Then $G_\omega$ acts faithfully on $\Delta$ and hence $G_\omega\cong H\cong \mathrm{PSL}_2(8).7.3=(\mathrm{PSL}_2(8)\times C_7)\rtimes C_3$. The group $H$ is not $2$-closed and the $2$-closure of $H$ is isomorphic to 
%$$(C_7\rtimes C_3)\mathrm{wr}\Sym(9)$$
%and hence has order $21^9 9!$. In particular, $G$ cannot  be binary because $G_\omega$ is not binary in its action on $\Delta$. Clear, from this it also follows that $G=\mathrm{SL}_6(2).2$ is not binary in its primitive action on a maximal subgroup in the Aschbacher class $\mathcal{C}_3$ of type $\mathrm{SL}_2(8)$.
%\end{proof}

The remaining lemmas deal with the case when $m\leq 2$.

\begin{lem}\label{l: c3 sln m=2}
 Suppose that $G$ is an almost simple group with socle $\PSL_n(q)$, where $n=2r$ for $r$ a prime. Suppose that $M$ is a $\mathcal{C}_3$-maximal subgroup such that $M\triangleright \mathrm{PSL}_2(q^r)$. Then the action of $G$ on $(G:M)$ is not binary.
\end{lem}

\begin{proof}
 Let $\K_\#$ be the field $\mathbb{F}_{q^r}$, and let $\{v_1, v_2\}$ be a $\K_\#$-basis for $V$. Let $\lambda$ be an element of $\K_\#$ such that $\{\lambda,\lambda^q, \dots, \lambda^{q^{r-1}}\}$ is a basis for $\K_\#$ over $\K=\Fq$ and observe that \[\B=\{v_1\lambda, v_1\lambda^q, \dots, v_1\lambda^{q^{r-1}}, v_2\lambda, v_2\lambda^q, \dots, v_2\lambda^{q^{r-1}}\}\] is an $\Fq$-basis for $V$.
 We take $M\cap \mathrm{PGL}(V)$ to be the subgroup that preserves the (semilinear) $\K_\#$-vector space structure of $V$. 

Suppose first that $r>2$ and $q>2$. Then $M\cap \PSL(V)$ has the structure $(C\times \PSL_2(q^r)).r$ with $C$ cyclic (see \cite[4.3.6]{kl}). 
Let $M_0$ be a subgroup of $M\cap \PSL_n(q)$ isomorphic to $\PSL_2(q^r).r$. Then $M_0$ has a subgroup $H=H_0 \times \langle \sigma \rangle$, where $H_0 \cong \PSL_2(q)$ and $\sigma$ is a field automorphism of order $r$. Moreover $H$ is maximal in $M_0$ (see \cite[Table 8.1]{bhr}).

Consider the direct sum $\Fq$-decomposition
\begin{equation}\label{decom}
 V=\langle v_1\lambda, v_2\lambda\rangle \oplus \langle v_1\lambda^q, v_2\lambda^q\rangle \oplus \cdots \oplus \langle v_1\lambda^{q^{r-1}}, v_2\lambda^{q^{r-1}}\rangle.
\end{equation}
Observe that $H_0$ stabilizes each subspace in the decomposition, while the field automorphism $\sigma: x\mapsto x^q$ induces an $r$-cycle on the $r$ subspaces in the decomposition. Thus $H$ stabilizes the decomposition. Since $r>2$, there is an element $g$ of order $r-1$ stabilizing  the decomposition that is not in $M$, centralizes $H_0$ and normalizes $\langle\sigma\rangle$. Then $H \le M \cap M^g$, and so $M \cap M^g$ is a maximal subgroup of $M_1 = (M\cap M^g)M_0'$. Hence the action of $M_1$ on $(M_1:M\cap M^g)$ is  isomorphic to the action of an almost simple group with socle $\PSL_2(q^r)$ on a maximal $\mathcal{C}_5$-subgroup containing $\PSL_2(q)$. By \cite{ghs_binary}, this action is not binary, and now Lemma~\ref{l: subgroup} implies that the action of $M$ on $(M:M\cap M^g)$ is not binary. Then Lemma~\ref{l: point stabilizer} implies that the action of $G$ on $(G:M)$ is not binary.

Suppose, next, that $r=2$. Again $M$ has a subgroup $M_0 \cong \PSL_2(q^2)$ preserving the decomposition (\ref{decom}), and $M_0$ has a subgroup $H_0 \cong \PGL_2(q)$. Define $U$ to be the set of elements in $S$ that fix all elements of $\B$ except $v_1\lambda$ and which satisfy
\[
 v_1\lambda \mapsto v_1+\alpha v_2 \lambda \;\;\;\;(\alpha \in \Fq).
\]
Then $U$ is a group of order $q$ with $U\cap M = 1$, and $U$  is normalized by a torus $T<H_0$ of order $q-1$, acting fixed-point-freely. In the usual way we obtain a subset $\Delta$ of $\Omega$ of size $q$ for which $G^\Delta$ is $2$-transitive. Now, by Lemma~\ref{l: alt sections classical}, $M$ does not have a section isomorphic to $\Alt(t)$ for $t\geq 7$ and, by Lemma~\ref{l: beautiful} the conclusion of the lemma follows for $q>5$. If $q=2$, then $S=\SL_4(2)\cong \Alt(8)$ and the result follows from \cite{gs_binary}. And if $q\in\{3,4,5\}$, then the result follows from Lemma~\ref{l: beautifulsetssmall}.

Finally assume that $q=2$ and $r>2$. In this case, writing matrices with respect to $\B$, $M$ contains an element
\[
 g=\begin{pmatrix} I_r & \\ & A \end{pmatrix},
\]
where $A\in \GL_r(2)$ is an element of order $2^r-1$, and we let $T=\langle g\rangle$. Now let $U$ be the subgroup of $S$ consisting of elements $u$ that fix all elements of $\B$ except $v_1\lambda$ and for which $$v_1\lambda^u-v_1\lambda \in {\rm span}_{\Fq}\{v_2\lambda, v_2\lambda^q, \dots, v_2\lambda^{q^{r-1}}\}.$$ Then $U$ is a group of order $q^r$ and $U\rtimes T$ is a Frobenius group. Then the set $\Lambda=M^U$ is a set of size $q^r$ on which $G^\Lambda$ acts $2$-transitively. By Lemma~\ref{l: alt sections classical}, $M$ does not contain a section isomorphic to $\Alt(t)$ for $t>6$. Thus, we conclude that $\Lambda$ is a beautiful set, and Lemma~\ref{l: beautiful} yields the result.
\end{proof}

\begin{lem}\label{l: c3 sln singer}
 Suppose that $G$ is almost simple with socle equal to $\PSL_n(q)$ and $n$ is an odd prime. Let $M$ be the normalizer of a Singer cycle in $G$. Then the action of $G$ on cosets of $M$ is not binary.
\end{lem}
\begin{proof}
We can write the group $F=M\cap \PSL_n(q)$ as a semidirect product $T\rtimes C$ where $T$ is cyclic of order $\frac{q^n-1}{(q-1)(q-1,n)}$ and $C$ is cyclic of order $n$, and acts fixed-point-freely on $T$. Choosing an appropriate basis we may take $C$ to be generated by a permutation matrix $c$ corresponding to an $n$-cycle, and one sees immediately that $C_{\PSL_n(q)}(c)>\langle c \rangle$. Let $x \in C_{\PSL_n(q)}(c)\setminus \langle c\rangle$ and observe that the group $F$ acts as a Frobenius group on the set $(F:F\cap F^x)$. 

Since $n>2$, Lemma~\ref{l: frobenius cyclic kernel} implies that the action of $M$ on $(M:M\cap M^x)$ is not binary. Now Lemma~\ref{l: point stabilizer} yields the result.
\end{proof}

\subsection{Case \texorpdfstring{$S=\SU_n(q)$}{S=SU(n,q)}}

\begin{lem}\label{l: c3 sun}
 Suppose that $G$ is almost simple with socle equal to $\PSU_n(q)$. Let $M$ be a $\mathcal{C}_3$-subgroup. Then the action of $G$ on $(G:M)$ is not binary.
\end{lem}
\begin{proof}
Note that $F^*(M)$ contains a normal subgroup $M_1$ which is a quasisimple cover of $\PSU_m(q^r)$ and, by \cite[Table 3.5.B]{kl}, $r\geq 3$.
 
First suppose that $m\geq 5$. Here we refer to Lemma~\ref{l: classical element} and we write elements of $M_1$ with respect to the basis $\B$ of $V_m(q^{2r})$ in that lemma. Define $j=\lfloor\frac{m-1}{2}\rfloor$ and $y=(m,2)$, so that $m=2j+y$. Then let
\[
 \tilde{x}=\begin{pmatrix}
            1 & &  & & \\ & A & & & \\ & & 1 & & \\ & & & \overline{A^{-T}} & \\ & & & & J_y &
           \end{pmatrix},
\]
where $A$ is an element of $\GL_{j-1}(q^{2r})$ of order $(q^{2r})^{j-1}-1$ and $J_y$ is a $y$-by-$y$ matrix chosen so that $\tilde{x}$ is an element of $\SU_m(q^r)$. Observe that $1$ is not an eigenvalue for $J_y$. Now we let $x$ be the element of $M_1$ that is an image of $\tilde{x}$.
 
% We refer to Lemma~\ref{l: classical element} which we apply to the group $M$, and take $x$ to be (a lift of) the given %element in $M$. 
Observe that $\tilde{x}$ has a 1-eigenspace over $\mathbb{F}_{q^r}$ of dimension 2, and so has a 1-eigenspace over $\Fq$ of dimensions $2r$. From this  we conclude that $C_M(x)$ is a proper subgroup of $C_G(x)$. Let $g\in C_G(x)\setminus C_M(x)$ and set $H:=M\cap M^g$. Then $H$ contains $x$ but does not contain $M_1$, and there is a suborbit, $\Delta$, on which the action of $M$ is isomorphic to the action of $M$ on $(M:H)$.

We now refer to Lemma~\ref{l: classical element}. This shows that $(M:H)$ has a subset $\Delta$ of size $q^{2r(j-1)}$ such that $M^\Delta$ is 2-transitive. Since $M$ does not have a section isomorphic to $\Alt(q^{2r(j-1)})$ by Lemma \ref{l: alt sections classical}, it follows that $\Delta$ is a beautiful subset, and the conclusion holds by Lemma \ref{l: point stabilizer}.

\smallskip 

 If $m=4$, then we use the same argument with Lemma~\ref{l: psu4} in place of Lemma~\ref{l: classical element}. 
Now suppose that $m=3$. 
Choose a hyperbolic basis, $\B_0=\{e_1,x, f_1\}$, for a 3-dimensional Hermitian space $V$ over $\K_\#=\mathbb{F}_{q^{2r}}$ associated with a form $\varphi$. We will use the fact that the isometry group of $\varphi$ contains an element 
 \[
  g:=\begin{pmatrix} a & & \\ & 1 & \\ & & a^{-1}\end{pmatrix},
 \]
where $a$ is an element of $\mathbb{F}_{q^r}$ of order $q^r-1$. Now we can take $M$ to contain a projective image of the special isometry group of $\varphi$, and let $S$ be the special isometry group of the form ${\rm Tr}_{\K_\#/\K}(\varphi)$ on $V$, considered as an $\Fq$-space.

Set $E={\rm span}_{\K_\#}\{e_1\}$, $F={\rm span}_{\K_\#}\{f_1\}$ and $X={\rm span}_{\K_\#}\{x\}$ and observe that $\langle E,F\rangle$ is a non-degenerate $2r$-dimensional $\Fq$-subspace of $V$, while $X$ is a non-degenerate $r$-dimensional $\Fq$-subspace of $V$. Choose a hyperbolic $\Fq$-basis for $\langle E, F\rangle$, $\B_1=\{e_1,\dots, e_r, f_1, \dots, f_r\}$, where $e_1,\dots, e_r\in E$ and $f_1,\dots, f_r\in F$. Let $\B_2$ be a hyperbolic basis $\Fq$-basis for $X$ and assume that $\B_2$ contains elements $e_{r+1},f_{r+1}$ such that $(e_{r+1}, f_{r+1})$ is a hyperbolic pair. Observe that $\B=\B_1\cup \B_2$ is a hyperbolic $\Fq$-basis for $V$.

Define a group $U$ in $S$ whose elements fix all elements in $\B$ except $e_{1},\dots, e_r$ and $f_{r+1}$, and  which satisfy
\[
\begin{array}{l}
e_{i}  \mapsto e_{i}+c_i e_{r+1} \hbox{ for }1\le i\le r, \\
f_{r+1}  \mapsto f_{r+1}-c_1f_1-c_2f_2-\cdots-c_rf_r \;\;(c_1,\dots, c_r \in \mathbb{F}_{q}).
\end{array}
\]
Now $U$ is of order $q^r$ and $\langle g\rangle$ normalizes and acts fixed-point-freely on $U$. What is more, $U$ is not in $M$ (since it contains non-trivial elements with a $1$-eigenspace of dimension at least $2r+1$ over $\Fq$). Thus, in the usual way, we obtain $\Delta$, a set of size $q^r$ whose stabilizer is $2$-transitive. By Lemma \ref{l: beautifulsetssmall}, $M$ does not have a section isomorphic to $\Alt(q^r)$ (recall that $r\ge 3$ here), so $\Delta$ is a beautiful set and Lemma~\ref{l: beautiful} yields the result.

\smallskip 

Suppose now that $m=2$ and $q>2$. Here we proceed in a similar fashion to Lemma~\ref{l: c3 sln m=2}. In this case we start with $\B_1=\{e, f\}$, a hyperbolic $\K_\#$-basis for $V$ with respect to a unitary form $\varphi_\#$. Let $\lambda$ be an element of $\K_\#$ such that $\B_2=\{\lambda,\lambda^q, \dots, \lambda^{q^{r-1}}\}$ is a basis for $\K_\#$ over $\K$. Then taking tensor products of elements of $\B_1$ and $\B_2$ we obtain a $\K$-basis for $V$. We write $V_1$ for the $\K$-span of $\B_1$ and $V_2$ for the $\K$-span of $\B_2$.

Now we take $S$ to preserve the form $\varphi={\rm Tr}_{\K_\#/\K}(\varphi_\#)$, and $M$ to preserve $\varphi_\#$, so that $M$ is a subgroup of $G$. What is more we observe that
\begin{align*}
 \varphi(e\otimes \lambda^{q^i}, e\otimes \lambda^{q^j})&=0, \\
 \varphi(f\otimes \lambda^{q^i}, f\otimes \lambda^{q^j})&=0,\\
 \varphi(e\otimes \lambda^{q^i}, f\otimes \lambda^{q^j})&={\rm Tr}_{\K_\#/\K}(\lambda^{q^i+q^{r+j}}).
\end{align*}
In particular this means that $\varphi$ can be written as a product $\varphi=\varphi_1\varphi_2$, where $\varphi(u_1\otimes u_2, v_1\otimes v_2)=\varphi_1(u_1,v_1)\varphi_2(u_2,v_2)$; here $u_1,v_1\in V_1, u_2,v_2\in V_2$, $\varphi_1 = \varphi|_{V_1}$ and $\varphi_2$ is defined by setting $\varphi_2(\lambda^{q^i}, \lambda^{q^j})={\rm Tr}_{\K_\#/\K}(\lambda^{q^i+q^{r+j}})$ and extending linearly on $V_2$.

With this set-up we see that the group $M$ contains a subfield subgroup isomorphic to the projective image of a group $H=\SU_2(q)\times r$ which preserves this tensor product structure (observe that the Frobenius automorphism on $V_2$ preserves the form $\varphi_2)$. Then $H$ lies in a group of the form $K=\SU_2(q)\times \SU_r(q)$, and it is clear that $H$ is not self-normalizing in $K$.

Now we proceed as before: we obtain a suborbit of $M$ whose action is isomorphic to the action of an almost simple group with socle $\PSL_2(q^r)$ on a maximal $\mathcal{C}_5$-subgroup containing $\PSL_2(q)$; by \cite{ghs_binary}, this action is not binary, and now Lemma~\ref{l: subgroup} implies that the action of $M$ on $(M:M\cap M^g)$ is not binary. Then Lemma~\ref{l: point stabilizer} implies that the action of $G$ on $(G:M)$ is not binary, as required.

\smallskip

Suppose next that $m=2$ and $q=2$. Then $S=\SU_{2r}(2)$. If $r=3$, then the result follows from Lemma~\ref{l: beautifulsetssmall}. Assume, then, that $r>3$ and notice that $S$ is simple. For convenience we shall work with the group $X = \GU_{2r}(2) = Z \times S$, where the centre $Z$ has order 3, and replace $M$ by $ZM$; the centre $Z$ will act trivially on all the sets we consider in the rest of the proof. 

We have $N_X(M \cap X) = \GU_2(q_0).r$, where $q_0 = 2^r$, and this contains a maximal torus $T$ of order $(q_0+1)^2$. Then $$N_X(T) \ge T.(r\times r) \hbox{ and } T.(r\times r)\not\le M,$$ while $N_M(T) \ge T.r$. Hence there exists $g \in N_X(T)\setminus M$ such that $T.r \le M\cap M^g$. Note also that $N:=N_M(T)\cap X = T.2r$. Since $N_X(N) \le M$, it follows that $N \not \le M\cap M^g$, and hence $M\cap M^g \cap X = T.r$. 

Let $H$ be the subgroup $\GU_2(q_0)$ of $M$, so that $H \cap M \cap M^g = T$. We consider the action of $H$ on $(H:T)$; the kernel of this action is the centre $Z_0$ of $H$, of order $q_0+1$. 
Let $V$ be a 2-space over $\F_{q_0^2}$ with unitary form $(\,,\,)$ preserved by $H$, and let $v_1,v_2$ be an orthonormal basis of $V$. Replacing $T$ by a conjugate if necessary, we may take $T$ to be the stabilizer in $H$ of $\langle v_1\rangle$ (hence also of $\langle v_2\rangle$). So the action of $H$ on $(H:T)$ is equivalent to the action on the set 
$\Lambda = \{\langle v \rangle : (v,v)=1\}$. 

Write $\a \to \bar \a$ for the involutory automorphism of $\F_{q_0^2}$. The orbits of the point-stabilizer $H_{\langle v_1\rangle}$ on $\Lambda$ are the singletons $\langle v_1\rangle$, $\langle v_2\rangle$ and the sets $\Lambda_\l$, for $\l \bar \l \notin\{0,1\}$, where $\lambda\in \F_{q_0^2}$ and
\[
\Lambda_\l = \{ \langle w \rangle\in \Lambda : (w,w)=1, \,(v_1,w)=\l\}.
\]
Note that $\D_\l = \D_{\a\l}$ if and only if $\a\bar \a=1$, so there are precisely $q_0-2$ suborbits $\Lambda_\l$, all of size $q_0+1$. In particular, $H_{ab}= Z_0$, for any distinct $a,b \in \Lambda$ such that $b \ne a^\perp$.

We claim that there exist scalars $\l_1,\l_2,\b \in \F_{q_0^2}$ with the following properties:
\begin{itemize}
\item[(i)] $\l_1\bar \l_1 \ne \l_2\bar \l_2$ and $\l_i\bar \l_i \ne 0,1$ for $i=1,2$,
\item[(ii)] $\b\bar \b = 1+\l_1\bar \l_1$,
\item[(iii)] $\l_1\bar \l_2\bar \b =1$.
\end{itemize}
To see this, first choose $\l_1$ and $\b$ such that $\l_1\bar \l_1 \ne 0,1$ and $\b\bar \b = 1+\l_1\bar \l_1$. Define $\l_2 = \b^{-1}\bar \l_1^{-1}$. Then setting $y = \l_1\bar \l_1$, we have 
\[
\l_2\bar \l_2 = (\b\bar \b)^{-1}(\l_1\bar \l_1)^{-1} = \frac{y^{-1}}{1+y} = \frac{1}{y+y^2}.
\]
If this is equal to $y$, then $y^3+y^2+1=0$; but there is no such $y \in \F_{q_0} = \F_{2^r}$, as $r\ge 5$ by assumption. Similarly, $\l_2\bar \l_2$ is not equal to 1 or 0. Thus (i)-(iii) hold. 

Now choose $\g \in \F_{q_0^2}\setminus \F_{q_0}$ such that $\g \bar \g = 1+\l_2\bar \l_2$, and define the following four points $a,b,c,d \in \Lambda$:
\[
\begin{array}{l}
a = \langle v_1\rangle,\\
b = \langle \l_1v_1+\b v_2\rangle, \\
c = \langle \l_2v_1+\g v_2\rangle, \\
d = \langle \l_2v_1+\bar \g v_2\rangle. 
\end{array}
\]
We shall show that the triples $(a,b,c)$ and $(a,b,d)$ are 2-subtuble complete, but not 3-subtuple complete under the action of $H$. 

Since $c,d \in \Lambda_{\l_2}$, we have $(a,c) \sim (a,d)$. Also $(b,c) \sim (b,d)$ if and only if 
\[
(\l_1v_1+\b v_2,\,\l_2v_1+\g v_2) = \nu (\l_1v_1+\b v_2,\,\l_2v_1+\bar \g v_2),
\]
for some $\nu \in\F_{q_0^2}$ satisfying $\nu \bar \nu =1$. This is equivalent to the equation 
\[
(\l_1\bar \l_2+\b\bar \g)(\bar \l_1\l_2+\bar \b\g) = (\l_1\bar \l_2+\b \g)(\bar \l_1\l_2+\bar \b\bar \g),
\]
 which boils down to $\l_1\bar \l_2\bar b(\g+\bar \g) = \bar \l_1\l_2\b(\g+\bar \g)$. This holds, since 
$\l_1\bar \l_2\bar \b =1$ by (iii). 

Hence $(a,b,c)$ and $(a,b,d)$ are 2-subtuple complete. They are clearly not 3-subtuple complete under the action of $H$, since $H_{ab} = Z_0$ which is the kernel of the action on $\Lambda$. 

Thus the action of $H$ on $\Lambda$ is non-binary. The same is true when we add field automorphisms to get the group $M=H.r$ or $H.(2r)$ acting on $\Lambda$: for any non-trivial field automorphism does not fix any of the suborbits $\Lambda_\l$, and hence $M_{ab} = H_{ab}=Z_0$, so $(a,b,c)$ and $(a,b,d)$ are not in the same orbit under the action of $M$.

We have now established that the action of $M$ on $(M:M\cap M^g)$ is not binary. Hence the result follows by Lemma~\ref{l: point stabilizer}.

\smallskip

Suppose finally that $m=1$. Here we use the method of Lemma~\ref{l: c3 sln singer}: first we write the group $F=M\cap \PSU_n(q)$ as a semidirect product $T\rtimes C$, where $T$ is cyclic of order $\frac{q^n+1}{(q+1)(q+1,n)}$ and $C$ is cyclic of order $n$, and acts fixed-point-freely on $T$. Proposition~\ref{centbd} implies that $C_{\PSU_n(q)}(c)>\langle c \rangle$ unless $(n,q)=(5,2)$, but this case can be excluded since $M$ is not maximal, see \cite[Table~8.20]{bhr}. Let $x \in C_{\PSU_n(q)}(c)\setminus \langle c\rangle$ and observe that the group $F$ acts as a Frobenius group on the set $(F:F\cap F^x)$. Since $n>2$, Lemma~\ref{l: frobenius cyclic kernel} implies that  the action of $M$ on $(M:M\cap M^x)$ is not binary. Now Lemma~\ref{l: point stabilizer} yields the result.
\end{proof}

\subsection{Case \texorpdfstring{$S=\Sp_n(q)$}{S=Sp(n,q)}}

\begin{lem}\label{l: c3 spn}
 Suppose that $G$ is almost simple with socle equal to $\PSp_n(q)$ with $n\ge 4$ and $(n,q) \ne (4,2)$, and let $M$ be a $\mathcal{C}_3$-subgroup. Then the action of $G$ on $(G:M)$ is not binary.
\end{lem}

\begin{proof}
There are two cases to consider here, namely $M\triangleright \PSp_m(q^r)$ with $mr = n$, and $M \triangleright \mathrm{SU}_{n/2}(q)/Z$ ($q$ odd). 

Consider the first case, where $M$ is almost simple with socle $\PSp_m(q^r)$ and $m$ is even, $r$ is prime and $n=mr$. If $m\geq 4$, then we let $x$ be the element in Lemma~\ref{l: classical element} (applied to the group $M$ rather than $G$). Since $x$ has a $1$-eigenspace of dimension $2$ over $\mathbb{F}_{q^r}$, it is easy to see that $C_G(x)\setminus C_M(x)$ is non-empty, and so we take an element $g$ in $C_G(x)\setminus C_M(x)$ and appeal to Lemma~\ref{l: classical element} to see that the action of $M$ on $(M:M\cap M^g)$ is not binary. Then Lemma~\ref{l: point stabilizer} yields the result.

Now suppose that $m=2$ and $r>2$ (we will deal with the case where $m=r=2$ in the last paragraph of the proof). Here $M\cap S/Z(S)$ is the projective image of $M_0\cong \Sp_2(q^r).r$, and we let $H_0\cong \SL_2(q)\times \langle \sigma\rangle$ be a subfield subgroup of $M_0$, where $\sigma$ is a field automorphism of order $r$; this is maximal in $M_0$ for $q>2$, see for instance~\cite[Table 8.1]{bhr}. We claim that there exists $g \in S\setminus M_0$ normalizing $H_0$. Once we have shown this, then in a similar manner to the previous paragraph, we obtain a suborbit for which the action is isomorphic to the action of an almost simple group with socle $\PSL_2(q^r)$ on a maximal subfield subgroup; then \cite{ghs_binary} yields the result for $q>2$.

To see the existence of the element $g$ we note that the subfield subgroup $H_0$ preserves a tensor product structure on $V$ and so lies in a maximal subgroup $\Sp_2(q)\circ I_r(q)$, where $I_r(q)$ is the isometry group of a symmetric bilinear form having matrix $I_r$, the identity. We can choose $g$ in $ I_r(q)$ of order $r-1$ normalizing the subgroup $\langle \sigma\rangle$ of order $r$ in $I_r(q)$. The claim follows.

We also need to deal with the case where $m=2$, $r>2$ and $q=2$. In this case an element $g$ as above exists, but this time the group $M\cap M^g$ is not necessarily maximal in $M$. However, in this case $M\cap M^g$ contains $\Sp_2(2)$ and has order not divisible by 4, so the conditions for applying Lemma~\ref{l: M2} to the action $(M, (M:M\cap M^g))$  are met with the prime $2$, and the result follows.

\smallskip
Now consider the second case, where $M \triangleright \mathrm{SU}_{n/2}(q)/Z$ with $q$ odd.
 For $n\geq 6$, we proceed as in Lemma~\ref{l: c3 sun}, taking $x$ to be the element given in Lemma~\ref{l: classical element} (for $n/2 \ge 5$), in Lemma \ref{l: psu4} (for $n/2=4$) and in Lemma \ref{l: psu3 element} (for $n/2 = 3$). Lemma~\ref{l: point stabilizer} implies that, choosing (as we may) $g \in C_G(x) \setminus M$, the action of $M$ on $(M:M\cap M^g)$ is not binary and thus the listed lemmas imply that $M$ must contain a section isomorphic to $\Sym(s)$ where
\[
s = \left\{ 
\begin{array}{l} 
q, \hbox{ if } n/2 = 3 \hbox{ or } 4; \\ 
q^{2\lfloor (n-6)/4\rfloor}, \hbox{ if }n/2 \ge 5.
\end{array}
\right.
\]
By Lemma \ref{l: alt sections classical}, this implies that one of the following holds:
\begin{itemize}
\item[(i)] $q=7$, $n/2 = 3$ or 4,
\item[(ii)] $q\le 5$, $n/2 = 3$ or 4.
\end{itemize}
Using \cite[Table 8.5]{bhr}, we see that  $\Alt(7)$ is not a section of $\PSU_3(7)$. Hence the socle of $G$ is $\PSp_6(q)$ ($q\le 5)$ or $\PSp_8(q)$ ($q\le 7$). 
All of these groups apart from $\PSp_8(q)\,(5\le q\le 7)$ are covered by Lemma~\ref{l: beautifulsetssmall}; and the $C_3$-actions of the remaining groups were shown to be not binary by a computation using the permutation character method of Lemma~\ref{l: characters}.

\smallskip

We are left with the situation when $S=\Sp_4(q)$ and $r=2$. Here we need to consider both the case where $M$ is almost simple with socle $\PSp_2(q^2)$ (the situation we have postponed above) and the case where $q$ is odd and $M$ contains a subgroup isomorphic to $\GU_{2}(q).2$. Lemma~\ref{l: beautifulsetssmall} and \cite{gs_binary} imply that we can assume that $q\geq 7$. 

For the case where $M$ has socle $\PSp_2(q^2)$, we choose a hyperbolic basis, $\B_0=\{e_1, f_1\}$, for a 2-dimensional space $V$ over $\K_\#=\mathbb{F}_{q^2}$ associated with an alternating form $\varphi$. We will use the fact that the isometry group of $\varphi$ contains an element 
 \[
  g:=\begin{pmatrix} a & \\ & a^{-1}\end{pmatrix},
 \]
where $a$ is an element of $\mathbb{F}_{q^2}$ of order $(q,2)(q-1)$. Now we can take $M$ to contain a projective image of the special isometry group of $\varphi$, and let $S$ be the special isometry group of the form ${\rm Tr}_{\K_\#/\K}(\varphi)$ on $V$, considered as an $\Fq$-space.

Set $E={\rm span}_{\K_\#}\{e_1\}$, $F={\rm span}_{\K_\#}\{f_1\}$. Choose a hyperbolic $\Fq$-basis for $V$, $\B=\{e_1, e_2, f_2, f_1\}$, where $e_1,e_2\in E$ and $f_1,f_2\in F$.

Define a group $U$ in $S$ whose elements can be written with respect to $\B$ as
\[
 \begin{pmatrix}
  1 & &  & \\ & 1 & b & \\ & & 1 & \\ & & & 1 
 \end{pmatrix},
\]
for some $b\in \Fq$. Now $U$ is of size $q$ and $\langle g\rangle$ normalizes, and acts transitively on the set of non-trivial elements of, $U$. What is more, $U$ is not in $M$ (since it contains non-trivial elements with a $1$-eigenspace of dimension $3$ over $\Fq$). We therefore obtain $\Delta$, a set of size $q$ whose stabilizer is $2$-transitive. Note that $M$ does not contain a section isomorphic to $\Alt(q-1)$ for $q\geq 7$, thus $\Delta$ is a beautiful set and Lemma~\ref{l: beautiful} yields the result.

The other case is handled very similarly: here $q$ is odd and $M$ contains $\GU_2(q)$. As before we start with a hyperbolic basis $\B_0=\{e_1, f_1\}$, but this time for a 2-dimensional space $V$ over $\K_\#=\mathbb{F}_{q^2}$ associated with a unitary $\sigma$-form $\varphi$. We use the same element $g$, and we let $S$ be the special isometry group of the form ${\rm Tr}_{\K_\#/\K}(\zeta\varphi)$ on $V$, considered as an $\Fq$-space; here $\zeta$ is an element of $\K_\#$ that satisfies $\zeta^\sigma=-\zeta$. Now the proof proceeds as before.  \end{proof}

\begin{comment}
REST OF PROOF SHOULD NOT BE NEEDED.

Corollary~\ref{c: psl element} (when $n=4$), Proposition~\ref{l: psu3 element} (when $n=6$) and  (when $n\geq 8$). In all cases a consideration of the $1$-eigenspace is enough to conclude that $C_M(x)$ is a proper subgroup of $C_G(x)$, and the proof proceeds as before. We must deal with the exceptions which occur when $(n,q)$ is in the following set:
\[
\Big\{ (4,3), (4,5), (6,3), (8,3), (8,5), (8,7), (10,3), (10,5), (10,7), (12,3), (14,3)\Big\}.
\]
The first of these can be eliminated by referring to \cite[Table 8.12]{bhr}. {\color{blue} 
The remainder, along with the first exception above, are dealt with using \magma and the permutation character method. %In these computations we are also including the action of $\mathrm{Sp}_8(2)$ acting on the right cosets of $\mathrm{Sp}_4(2).2$. 
Let $M$ be a maximal subgroup of $S$ in the Aschbacher class $\mathcal{C}_3$, let $1_S$ be the principal character of $S$, let $\pi_M$ be the permutation character for the action of $S$ on the right cosets of $M$. We have verified that in all cases
$$\frac{\langle \pi(\pi-1_S)(\pi-2\cdot 1_S),1_S\rangle}{|\mathrm{Out}(S)|}>(\langle \pi(\pi-1_S),1_S\rangle)^3.$$
In particular, all actions under consideration are not binary in view of~\cite[Lemma~$3.1$]{dgs_binary}.}
\end{comment}

\subsection{Case \texorpdfstring{$S=\Omega_n^\varepsilon(q)$}{S=Omega(n,q)}}

\begin{lem}\label{l: c3 son}
 Suppose that $G$ is almost simple with socle equal to $\POmega_n^\varepsilon(q)$ ($n\ge 7$), and let $M$ be a maximal $\mathcal{C}_3$-subgroup. Then the action of $G$ on $(G:M)$ is not binary.
\end{lem}

\begin{proof}
First assume that $n$ is odd, so $q$ is odd. In this case $M$ is almost simple with socle $\Omega_m(q^r)$ where $n=mr$. 
For $m\ge 7$, let $x \in M$ be as in Lemma~\ref{l: classical element 2}; then there exists $g \in  C_G(x)\setminus 
C_M(x)$, and the lemma shows that the action of $M$ on $(M:M\cap M^g)$ is not binary, giving the conclusion. 
If $m=5$ we use the isomorphism $\Omega_5(q^r) \cong \PSp_4(q^r)$: the element $x \in \PSp_4(q^r)$ defined in 
Lemma~\ref{l: classical element} acts as ${\rm diag}(1,a,a,a^{-1},a^{-1})$ in $\Omega_5(q^r)$, so again 
there exists $g \in  C_G(x)\setminus C_M(x)$, and the action $(M,(M:M\cap M^g))$ is not binary by the lemma. Finally, if $m=3$, the element $x \in \PSL_2(q^r)$ defined in Lemma~\ref{l: psl2 element} acts as ${\rm diag}(1,a^2,a^{-2}) \in \Omega_3(q^r)$, and so again there exists $g$ for which $(M,(M:M\cap M^g))$ is not binary. 

 \smallskip

Next assume that $n$ is even and $S=\Omega_n^\varepsilon(q)$, where $\varepsilon\in \{+,-\}$. We refer to \cite[Tables 3.5.E and 3.5.F]{kl} and split into three cases:
\begin{itemize}
\item[(1)] $m=n/r\ge 4$ is even and $M$ is of type  $\Or_m^\varepsilon(q^r)$;
\item[(2)] $qm=qn/2$ is odd, $r=2$ and $M$ is of type  $\Or_{n/2}(q^2)$;
\item[(3)] $m=n/2\geq 4$, $r=2$ and $M$ is of type $\SU_m(q)$. 
\end{itemize}  

{\it Case~$(1)$ } Assume that $m=n/r\ge 4$ is even and $M$ is of type $\Or_m^\varepsilon(q^r)$. Observe that provided $(m,\varepsilon) \ne (4,+)$,  $M$ is almost simple with socle $\POmega_m^\varepsilon(q^r)$. 

Proceeding as before, using Lemma~\ref{l: classical element 2}, the conclusion follows directly for $m\geq 8$, except for $M$ of type $\Omega_8^-(9)$ in $S = \Omega_{16}^-(3)$. Therefore, we look closer at this embedding. Choose Sylow $13$-subgroups $Q$ of $M$ and $P$ of $S$ such that $Q<P$. Then $|Q|=13$, $P\cong 13^2$. Observe that a Sylow $13$-subgroup of the $\mathcal{C}_1$-subgroup $\Omega_8^-(3)\times \Omega_8^+(3)$ has order $13^2$; therefore, we may assume that $P\le \Omega_8^-(3)\times \Omega_8^+(3)$ and $P=P_{-}\times P_+$, where $P_{-}$ is a Sylow $13$-subgroup of $\Omega_{8}^-(3)$ and $P_+$ is a Sylow $13$-subgroup of $\Omega_8^+(3)$. Thus $13=|P_-|=|P_+|.$ Write $P_-=\langle g_-\rangle$ and $P_+=\langle g_+\rangle$. Now, as $Q\le P$, we may write $Q=\langle g_-^ig_+^j\rangle$, for some positive integers $i$ and $j$. Replacing the original generators of $P_-$ and $P_+$ if necessary, we may suppose that $i=j=1$: observe that neither $i$ nor $j$ cannot be $0$, because a Sylow $13$-subgroup of $\Omega_8^-(9)$ cannot fix an $8$-dimensional subspace of the underlying vector space for $\Omega_{16}^-(3)$. It can be easily verified (for instance with~\cite{atlas} or with \texttt{magma}) that  $g_+^{4}$ and $g_+$ are in the same $\Omega_8^+(3)$-conjgacy class. Therefore, $h:=g_-g_+^4$ and $g:=g_-g_+$ are conjugate in $\Omega_8^-(3)\times \Omega_8^+(3)$ and so also in $S$. Moreover, $P=\langle g,h\rangle$. Now, $gh=g_-^2g_+^5$ and $\langle gh\rangle=\langle(gh)^2\rangle=\langle g_-^4g_+^{10}\rangle$. Again, it is easy to verify that $g_-$ and $g_-^4$ are in the same $\Omega_8^-(3)$-conjugacy class and $g_+$ and $g_+^{10}$ are in the same $\Omega_8^+(3)$-conjugacy class. Therefore, $\langle h\rangle$ and $\langle gh\rangle$ are conjugate in $\Omega_8^-(3)\times \Omega_8^+(3)$ and so also in $S$. Summing up, $P=\langle g,h\rangle$, $Q=\langle g\rangle$ and $\langle g\rangle$, $\langle h\rangle$ and $\langle gh\rangle$ are $S$-conjugate. Hence $(G,\Omega)$ is not binary in this case by Lemma~\ref{l: M2}. 

Now consider $m=6$. Here $M$ has socle $P\Omega_6^\varepsilon(q^r) \cong \PSL_4^\varepsilon(q^r)$, and we use the usual argument using the elements given in Lemma \ref{l: psl4spec} (for $\varepsilon = +$) and  Lemma~\ref{l: psu4} 
for $\varepsilon = -$. These give the conclusion unless $\varepsilon = +$ and $q^r=4$. In the latter case, $S = \Omega_{12}^+(2)$, which is covered by Lemma~\ref{l: beautifulsetssmall}.
 
If $m=4$ and $M$ has socle $\POmega_4^-(q^r)\cong \PSL_2(q^{2r})$, then we use Lemma \ref{l: psl2var}.  The element	 $x$ given in that lemma acts in $\Omega_4^-(q^r)$ as ${\rm diag}(b,b^{-1},-1,-1)$ for some $b$, so has a larger centralizer in $G$ than in $M$, and so the result follows as usual, unless $q^r=4$, in which case $S = \O_8^-(2)$, which is covered by Lemma \ref{l: beautifulsetssmall}.

Suppose now that $m=4$ and $M$ has socle $\POmega_4^+(q^r)\cong \PSL_2(q^r)\times \PSL_2(q^r)$. Thus $S\cong\Omega_{4r}^+(q)$. Note that, if $r=2$, then \cite{bhr} implies that $M$ is conjugate by a triality automorphism to a maximal subgroup of the $\mathcal{C}_2$-class, hence we know already that the action in this case is not binary. Assume, from here on, that $r\geq3$.
First assume that $q$ is odd. Here $M\cap S = \POmega_4^+(q^r).r \cong (\PSL_2(q^r) \times \PSL_2(q^r)).r$ (see \cite[4.3.14]{kl}). Write $M_0 = M\cap S$, and  let $H$ be a maximal subgroup $\O_3(q^r).r$ of $M_0$. Then $H < \O_{3r}(q) <S$, and $C_S(H)$ contains a subgroup $\O_r(q)$. Picking $g \in C_S(H) \setminus M$, we have $M_0^g\cap M_0 = H$, a maximal subgroup of $M_0$. The action of $M_0$ on $(M_0:H)$ is a primitive permutation action of diagonal type, and is not binary by \cite[Proposition~$4.1$]{wiscons}. Hence the action of $S$ on $(S:M_0)$ is also not binary. To prove the same assertion for $(G:M)$, let $G_1 = \langle M^g \cap M, S\rangle$, and $M_1 = M\cap G_1$. Then $M^g\cap M$ is maximal in $M_1$, and the action of $M_1$ on $(M_1:M^g\cap M)$ is not binary, again by \cite[Proposition~4.1]{wiscons}. Hence $(M,\,(M:M^g\cap M))$ is also not binary, by Lemma~\ref{l: subgroup}, and so $(G,\,(G:M))$ is not binary by Lemma~\ref{l: point stabilizer}.

Next consider the case where $q$ is even. Again let $H = \O_3(q^r) \cong \PSL_2(q^r)$ be a diagonal subgroup of $M_0 = {\rm soc}(M) \cong \PSL_2(q^r) \times \PSL_2(q^r)$, and let $T$ be a cyclic torus of order $q^r-1$ in $H$. Then $T$ lies in a subgroup $\O_2^+(q^r)$ of $H$, so is centralized by a subgroup $\O_{2r}^+(q)$ of $S$. Pick $g \in C_S(T) \setminus M$, so that $T < M^g \ne M$. As $T$ centralizes a $2r$-subspace of $V=V_{4r}(q)$, it must act on the natural module $V_4(q^r)$ for $M^g$ with eigenvalues $(\l,\l^{-1},1,1)$ for some $\l \in GF(q^r)$, and so $T$ lies in $q^r-1$ nonsingular point-stabilizers $\O_3(q^r)$ in $M^g$. These generate $M^g$, so not all of them can lie in $M$. Hence there exists a subgroup $H_1 =\O_3(q^r)$ of $M^g$ such that $T < H_1 \not \leq M$. Hence there is a Frobenius group $UT < H_1$ of order $q^r(q^r-1)$ with $U \not \leq M$, and so in the usual way we obtain a subset $\Delta$ of $\O = (G:M)$ such that $G^\Delta$ contains the 2-transitive group $\AGL_1(q^r)$. Then $\Delta$ is a beautiful subset, unless possibly $\Alt(q^r-1)$ is a section of $M$. Lemma~\ref{l: alt sections classical} implies that $q^r-1\leq 6$ and, since $r>2$, we obtain a contradiction as required.% we obtain that $q=r=2$. {\color{red} We use \magma to rule out the embedding $\POmega_4^+(4)< \POmega_8^+(2)$.}

%It remains to deal with the case where $q$ is odd and $r=2$. {\color{red} We use \magma to deal with the case where $q\leq 9$.} Assume, then, that $q\geq 9$ and, as before, let $H = \O_3(q^2) \cong \PSL_2(q^2)$ be a diagonal subgroup of $M_0 = {\rm soc}(M) \cong \PSL_2(q^2) \times \PSL_2(q^2)$. Let $X \cong PGL_2(q)$ be a subgroup of $H$, and $T$ a cyclic torus of order $q-1$ in $X$. As in the previous paragraph, there exists $g \in C_S(T) \setminus M$ such that $T < M^g \ne M$, and a conjugate $X'$ of $X$ such that $T<X' \not \le M$. 
%Hence there is a Frobenius group $UT < X'$ of order $q(q-1)$ with $U \not \le M$, and we obtain a subset $\Delta$ of size $q$ such that $G^\Delta$ is 2-transitive. Since $q>9$, $\Alt(q)$ is not a section of $S = P\O_8^+(q)$, and so the action of $G$ on $(G:M)$ is not binary.

\medskip
{\it Case $(2)$ } 
Next assume that $r=2$ and $qm=qn/2$ is odd, and that $M$ has socle $\Omega_{n/2}(q^2)$. If $m\geq 7$, then we proceed as before using Lemma~\ref{l: classical element 2}: the result follows except for the embedding $\Omega_7(9)$ in $\POmega_{14}^\varepsilon(3)$. For this embedding we observe that $|G:M|$ is even and $M$ is almost simple. Let $P$ be a Sylow $2$-subgroup of $M$, let $Q$ be a Sylow $2$-subgroup of $G$ that contains $P$, and let $x$ be an element in $G\setminus M$ that normalizes $P$. Then $|M:M\cap M^x|$ is odd and $M\cap M^x$ is core-free. Now Lemma~\ref{l: odd degree Lie} implies that the action of $M$ on $(M:M\cap M^x)$ is not binary, and Lemma~\ref{l: point stabilizer} yields the result. If $m=5$, then we use the fact that $\Omega_{5}(q^2)\cong \Sp_4(q^2)$ and the result follows using the same method replacing Lemma~\ref{l: classical element 2} with Lemma~\ref{l: classical element}.

\medskip
{\it Case $(3)$ } 
Finally assume that $r=2$ and $m=n/2\geq 4$, and that $M$ contains a normal subgroup that is a quotient of $\SU_m(q)$. Note that when $n=8$, \cite{kl} implies that $\varepsilon=+$ while \cite{bhr} implies that these $\mathcal{C}_3$-subgroups of $\POmega_8^+(q)$ are conjugate, via a triality automorphism, to $\mathcal{C}_1$-subgroups, hence are already dealt with in \cite{gs_binary}. Assume, then, that $n\geq 10$.

We proceed as before: let $x \in M$ be the element given in Lemma~\ref{l: classical element}. As this has a non-trivial $1$-eigenspace,  there exists $g\in C_G(x)\setminus C_M(x)$, and so in the action of $M$ on  $(M:M\cap M^g)$, the stabilizer is a subgroup of $M$ containing $x$ but not containing a homomorphic image of $\SU_m(q)$. If the action of $G$ on $(G:M)$ is binary, then Lemma~\ref{l: point stabilizer} implies that the action of $M$ on $(M:M\cap M^g)$ is binary and Lemma~\ref{l: classical element} implies that $M$ contains a section isomorphic to $\Sym(s)$ where $s=q^{2(\lfloor(m-3)/2\rfloor)}$. Then Lemma \ref{l: alt sections classical} implies that $(m,q)=(5,2)$ or $(6,2)$. In the former case $S=\Omega_{10}^-(2)$, while in the latter case $S=\Omega_{12}^+(2)$; both are covered by Lemma \ref{l: beautifulsetssmall}. Thus in all cases the action of $M$ on $(M:M\cap M^g)$ is non-binary, and hence the same is true of the action of $G$ on $(G:M)$, completing the proof.
\end{proof}

\begin{comment}
Taking into account the restrictions given in \cite[Tables 3.5.E and 3.5.F]{kl} we obtain the following exceptions when $\varepsilon=+$: $(n,q)$ in
\[
 \Big\{ (8,2), (8,3), (8,4), (8,5), (12,2), (12,3), (16,2)\Big\}
\]
{\color{blue} We use \magma to do these using a combination of the usual methods: the permutation character method, odd degree-actions, looking for not binary triples, Lemma~\ref{l: M2} and Lemma~\ref{l: added}.}
We obtain the following exceptions when $\varepsilon=-$: $(n,q)$ in 
\[
 \Big\{ (10,2), (10,3), (10,4), (10,5), (14,2), (14,3), (18,2)\Big\}
\]
{\color{blue} We use \magma as before to exclude all of these. The computations when $(n,q)\in \{(10,4),(10,5)\}$ are quite involved and hence we give some details. We have constructed all the almost simple groups $G$ having the required socle and we have constructed the maximal subgroup $M$ of $G$ in the required Aschbacher class $\mathcal{C}_3$ and of the correct type. Then we have generated at random $g\in G$ and $m\in M$ and we have computed first $T:=M\cap M^g$ and then $K:=T\cap T^m$. Then we have checked whether the action of $T$ on $(T:K)$ was binary. We have continued this process until we found $g$ and $m$ such that $T$ on $(T:K)$ is not binary, which proves that the action of $G$ on $(G:M)$ is also not binary. The search for suitable $g$ and $m$ was quite time demanding.}
\end{comment}

\section{Family \texorpdfstring{$\C_4$}{C4}}\label{s: c4}

In this section, the subgroup $M$ preserves a tensor product. We start with two $\K$-vector spaces, $W_1$ and $W_2$, of dimension $n_1$ and $n_2$, respectively, and satisfying $n=n_1n_2$. Roughly speaking, we identify $V$ with the tensor product $W_1\otimes W_2$, and $M$ is the subgroup of $G$ that preserves this identification. The list of subgroups is given in Table \ref{c4poss}; the details of their precise structure and embeddings can be found in \cite[\S 4.4]{kl}.

\begin{table}[ht!]
\[
\begin{array}{|c|c|c|}
\hline
\hbox{case} &  \hbox{type} & \hbox{conditions} \\
\hline
{\rm L^\e} & \GL_{n_1}^\e (q)\otimes \GL_{n_2}^\e (q) & n_1<n_2 \\
{\rm S} & \Sp_{n_1}(q) \otimes \Or_{n_2}^\e(q) & n_2\ge 3,\,q \hbox{ odd} \\
{\rm O}^+ & \Sp_{n_1}(q) \otimes \Sp_{n_2}(q) & n_1<n_2 \\
{\rm O}& \Or_{n_1}^{\e_1}(q) \otimes  \Or_{n_2}^{\e_2}(q)  & n_i\ge 3,\,q \hbox{ odd}  \\
\hline
\end{array}
\]
\caption{Maximal subgroups in family $\C_4$} \label{c4poss}
\end{table}

As in the $\C_3$-case we give a geometrical interpretation to the set of cosets of $M$  (which is the set on which we are acting); this will involve defining a \emph{tensor product structure on $V$}. Let us start with the case where $S=\SL_n(q)$.

We begin with a $\K$-linear isomorphism $\phi:W_1\otimes W_2 \to V$. Let $\Sigma$ be the set of all such isomorphisms, and we observe that two groups act naturally on $\Sigma$: 
\begin{enumerate}
 \item $\GL(W_1)\circ \GL(W_2) $ acts on $\Sigma$ via $\phi^g(\mathbf{w_1}\otimes \mathbf{w_2}) = \phi(\mathbf{w_1}^{g^{-1}}\otimes \mathbf{w_2}^{g^{-1}})$ (and extended linearly);
 \item $\GL_n(\K)$ acts on $\Sigma$ via $\phi^h(\mathbf{w_1}\otimes \mathbf{w_2}) = (\phi(\mathbf{w_1}\otimes \mathbf{w_2}))^h$ (and extended linearly).
\end{enumerate}
As in the $\C_3$-case, these two actions commute. Thus we define a \emph{tensor product structure on $V$} to be an orbit of the group $\GL(W_1)\circ \GL(W_2)$ on $\Sigma$, and (using commutativity of the actions) we observe that $\GL_n(\K)$ acts on the set of all tensor product structures on $V$. What is more the stabilizer of this action is the subgroup $M$, hence we have the geometrical interpretation that we require.

Again, just as before, we can replace the word ``linear'' with the word ``semilinear'' in the previous paragraph to extend this geometrical interpretation to subgroups of $\GammaL_n(\K)$.

For the remaining classical groups, we need to clarify what is meant by a tensor product structure on a vector space equipped with a form. 
%It turns out, due to the restrictions listed in \cite{kl} which specify that $q$ must be odd for most of the $\C_4$-groups to occur, that it is enough to consider sesquilinear forms (i.e. we do not consider quadratic forms). 
So let us assume that our two vector spaces, $W_1$ and $W_2$ are equipped with forms $\langle,\rangle_1$ and $\langle, \rangle_2$, respectively. Now we define
\begin{align*}
 \langle\cdot, \cdot \rangle &:(W_1\otimes W_2)\times(W_1\otimes W_2)\to \K,\\
&\left(\sum_i v_1^i\otimes v_2^i, \sum_j w_1^j\otimes w_2^j\right) \mapsto \sum_{i,j} \langle v_1^i, w_1^j\rangle_1\langle v_2^i, w_2^j\rangle_2,
\end{align*}
where $v_1^i, w_1^j\in W_1$ and $v_2^i, w_2^j\in W_2$ for all $i$ and $j$. One can check that $\langle, \rangle$ is a well-defined form on $W_1\otimes W_2$. Now our map $\phi: W_1\otimes W_2\to V$ carries this form onto the vector space $V$, and we obtain a map to a formed space. Following the same approach as above, we see that there are actions of ${\rm Isom}(\langle, \rangle_1)\circ {\rm Isom} (\langle, \rangle_2)$ and of ${\rm Isom} (\langle, \rangle)$ acting  on the set of all such maps; this yields a definition of a tensor product structure on a formed space, and provides an embedding of ${\rm Isom}(\langle, \rangle_1)\circ {\rm Isom} (\langle, \rangle_2)$ in the group ${\rm Isom} (\langle, \rangle)$, as the stabilizer of such a tensor product structure. Moreover, in the case where the characteristic $p=2$ and both $\langle,\rangle_1$ and $\langle, \rangle_2$ are symplectic, the group ${\rm Isom}(\langle, \rangle_1)\circ {\rm Isom} (\langle, \rangle_2)$ also preserves a quadratic form on $W_1\otimes W_2$ with associated bilinear form $\la\,,\,\ra$, yielding an embedding into $O^+(V)$. Thus we obtain all the embeddings listed in Table \ref{c4poss}.

%There is some work to do to establish the type of the form $\langle, \rangle$ (unitary, symplectic, orthogonal), depending on the type of the two forms $\langle, \rangle_1$ and $\langle, \rangle_2$. This is worked out in \cite[\S4.4]{kl}, and results in the embeddings listed there.

%Note that in all cases (i.e. when we have naked vector spaces, and when we have vector spaces equipped with forms), the map $\phi$ can be specified by giving the image of the $n$ elements $w_1\otimes w_2$ obtained by allowing $w_1$ (resp. $w_2)$ to range over a basis of the vector space $W_1$ (resp. $W_2$). 

In the formed space case, it is useful to observe that if we start with  hyperbolic bases $e_1,\dots, f_1,\dots$ for $W_1$ and   $u_1,\dots, v_1,\dots$ for $W_2$, then, by taking pure tensors, we obtain a hyperbolic basis for $W_1\otimes W_2$; the hyperbolic pairs are
\[
 (e_i\otimes u_j, f_i\otimes v_j) \textrm{ and } (e_i\otimes v_j, f_i\otimes u_j).
\]
Similarly, if $W_1$  contains a vector $x$ such that 
 $\langle x, x \rangle_1=1$, then $(x\otimes u_i, x\otimes v_i)$ is a hyperbolic pair in $W_1\otimes W_2$; and also 
if $W_{2,0}$ is a non-degenerate subspace of $W_2$, then $x\otimes W_{2,0}$ 
is  a non-degenerate subspace of the tensor product.

The main result of this section is the following. The result will be proved in a series of lemmas.

\begin{prop}\label{p: c4}
 Suppose that $G$ is an almost simple group with socle $\bar S = \Cl_n(q)$, and assume that 
\begin{itemize}
\item[{\rm (i)}] $n\ge 3,4,4,7$ in cases $L,U,S,O$ respectively, and 
\item[{\rm (ii)}] $\Cl_n(q)$ is not one of the groups listed in Lemma $\ref{l: beautifulsetssmall}$.
\end{itemize}
Let $M$ be a maximal subgroup of $G$ in the family $\mathcal{C}_4$. Then the action of $G$ on $(G:M)$ is not binary.
\end{prop}

\subsection{Case \texorpdfstring{$S=\SL_n(q)$}{S=SL(n,q)}}

\begin{table}\centering
\begin{tabular}{cl}
\toprule[1.5pt]
Group & Details of action \\
\midrule[1.5pt]
$\SL_6(q)$ & $q\in \{3,4,5\}$, $n_1=2$, $n_2=3$: $M\triangleright \PSL_2(q)\times \PSL_3(q)$. \\
$\SL_{12}(2)$ & $n_1=3$, $n_2=4$: $M\triangleright \PSL_3(2)\times \PSL_4(2)$. \\
\bottomrule[1.5pt]
\end{tabular}
\caption{$\C_4$ -- $\SL_n(q)$ -- Cases where a beautiful subset was not found.}\label{t: c4 sln}
\end{table}

\begin{lem}\label{l: c4 sln}
 In this case either $\Omega$ contains a beautiful subset or else $S$ is listed in Table~$\ref{t: c4 sln}$.
\end{lem}

\begin{proof}
 In this case $M$ contains a normal subgroup isomorphic to $\SL_{n_1}(q)\circ \SL_{n_2}(q)$ where $n=n_1 n_2$ and we may assume that $2\leq n_1<n_2$.

Let $\B_1 = \{u_1,\dots, u_{n_1}\}$ be a basis for $W_1$, $\B_2= \{w_1,\dots, w_{n_2}\}$ a basis for $W_2$. Now
$\B= \{u_i\otimes w_j\,:\,\hbox{ all }i,j\}$ 
is a basis for $W_1\otimes W_2$, which is mapped to $V$ via a map $\phi$ contained in a tensor product structure $\mathcal{P}$, which is stabilized by $M$. 

Assume that $q\ge 7$. Let $T_1$ (resp. $T_2$) be the maximal split torus in $\GL(W_1)$ (resp. $\GL(W_2)$) that is diagonal with respect to $\B_1$ (resp. $\B_2$). Let $T$ be the intersection of the tensor product of $T_1$ and $T_2$ with the group $S$.

Now we let $U$ be the subgroup whose elements fix all elements of $\B$ except $u_1\otimes w_1$, where we require
 \[ 
 u_1\otimes w_1\mapsto u_1\otimes w_{1}+ \alpha u_1\otimes w_2,
 \]
 for some $\alpha\in\Fq$. Note that $U\not\leq M$ (consider, for instance, the $1$-eigenspace of non-trivial elements of $U$), and that $n_2$ is necessarily greater than or equal to $3$. Hence the group $H=U\rtimes T$ acts 2-transitively on the set $\Lambda = \mathcal{P}^U :=\{\mathcal{P}u\,:\,u \in U\}$, and $|\Lambda|=q$. 
 
Let $G_1 \cong \mathrm{GL}_{n_1-1}(q)$ be the subgroup of $\mathrm{GL}(W_1)$ fixing $u_1$ and $\la u_2,\ldots ,u_{n_1}\ra$; 
and let $G_2 \cong \mathrm{GL}_{n_2-2}(q)$ be the subgroup of $\mathrm{GL}(W_2)$ fixing $w_1,w_2$ and $\la w_3,\ldots ,w_{n_2}\ra$.
Then $M_{(\Lambda)}$, the pointwise stabilizer of $\Lambda$ in $M$, contains $(G_1\otimes G_2)\cap S$ (since this subgroup commutes with $U$). It follows that any (non-abelian) simple section of $M^\Lambda = M_\Lambda / M_{(\Lambda)}$ is isomorphic to a section of $\mathrm{GL}_2(q)$. By Lemma \ref{l: alt sections classical}, for $q\geq 7$ this precludes the possibility that $M^\Lambda\geq \Alt(q-1)$, and we obtain that $\Lambda$ is a beautiful subset; now Lemma~\ref{l: beautiful} allows us to conclude that there are no such binary actions.

%Now, making use of our observation about the map $\phi$ being specified via the list of images of $\mathcal{B}$, we see that any element of $M$ that fixes $u_1\otimes w_1$ and $u_1\otimes w_2$ fixes every element of $\Lambda$. Conversely, any element of $M$ that moves $u_1\otimes w_1$ or $u_1\otimes w_2$ moves an element of $\Lambda$; we conclude that $S_\Lambda\cong S\cap (\GL_{n_1-1}(q)\circ \GL_{n_2-2}(q))$, and we obtain that any non-solvable section in $M^\Lambda$ is necessarily isomorphic to a section of $\GL_2(q)$. For $q\geq 7$ this precludes the possibility that $M^\Lambda\geq \Alt(q-1)$, and we obtain that $\Lambda$ is a beautiful subset; now Lemma~\ref{l: beautiful} allows us to conclude that there are no such binary actions.

For $q\in \{3,4,5\}$ we exclude the case $(n_1,n_2)=(2,3)$ (since this appears on the table), and so we assume that $n_2\geq 4$. Now we proceed as before, this time taking $U$ to be the subgroup whose elements fix all elements of $\B$ except $u_1\otimes w_1$ and
 \[ 
 u_1\otimes w_1\mapsto u_1\otimes w_{1}+ \alpha u_1\otimes w_2 + \beta u_1\otimes w_3,
 \] 
 for some $\alpha,\beta\in\Fq$. Now we take $T_2$ to be a maximal torus of $\GL(W_2)$ that preserves the decomposition 
 \[
  \langle w_1\rangle \oplus \langle w_2,w_3\rangle \oplus\langle w_4\rangle \oplus \cdots
 \]
 and induces a Singer cycle on the subspace $\langle w_2,w_3\rangle$. Defining $T$ as before, we obtain a beautiful set of size $q^2$ unless $\Alt(q^2-1)$ is isomorphic to a section of $\GL_3(q)$. By Lemma~\ref{l: alt sections classical},  $\Alt(q^2-1)$ is not isomorphic to a section of $\GL_3(q)$ and  hence the result follows.

For $q=2$, \cite[Table~3.5.A]{kl} allows us to assume that $n_1>2$. Now we exclude the case $(n_1,n_2)=(3,4)$ (this is in Table \ref{t: c4 sln}), and we conclude that $n_2\geq 5$. The argument proceeds as before, taking $U$ to be the subgroup whose elements fix all elements of $\B$ except $u_1\otimes w_1$ and
 \[ 
 u_1\otimes w_1\mapsto u_1\otimes w_{1}+ \alpha u_1\otimes w_2 + \beta u_1\otimes w_3 + \gamma u_1\otimes w_4 + \delta u_2\otimes w_5
 \]
 for some $\alpha,\beta, \gamma, \delta\in\Fq$. We obtain a beautiful set of order $16$ unless $\Alt(15)$ is a isomorphic to a section of $\GL_4(2)$. Again, by Lemma~\ref{l: alt sections classical},  $\Alt(15)$ is not isomorphic to a section of $\GL_4(2)$ and  hence the result follows.
 \end{proof}

 \subsection{Case \texorpdfstring{$S=\SU_n(q)$}{S=SU(n,q)}}

\begin{table}\centering
\begin{tabular}{cl}
\toprule[1.5pt]
Group & Details of action \\
\midrule[1.5pt]
$\SU_6(q)$, $\SU_8(q)$ & $q\in\{3,4,5\}$, $n_1=2$: $M\triangleright \PSU_{2}(q)\times \PSU_{n_2}(q)$. \\
$\SU_{12}(q)$, & $q\in\{3,4,5\}$, $n_1=3$: $M\triangleright \PSU_{3}(q)\times \PSU_{4}(q)$. \\
$\SU_n(2)$ & $3\leq n_1<n_2 \leq 5$: $M\triangleright \PSU_{n_1}(2)\times \PSU_{n_2}(2)$. \\
\bottomrule[1.5pt]
\end{tabular}
\caption{$\C_4$ -- $\SU_n(q)$ -- Cases where a beautiful subset was not found.}\label{t: c4 sun}
\end{table}

\begin{lem}\label{l: c4 sun}
 In this case either $\Omega$ contains a beautiful subset or else $S$ is listed in Table~$\ref{t: c4 sun}$.
\end{lem}
\begin{proof}
 In this case $M$ contains a normal subgroup isomorphic to $\SU_{n_1}(q)\circ \SU_{n_2}(q)$ where $n=n_1 n_2$ and we may assume that $2\leq n_1<n_2$.

%Assume we have a hyperbolic basis $\mathcal{B}_1$ for $W_1$ (with hyperbolic pairs labelled $e_i, f_i$), and likewise, a hyperbolic base $\mathcal{B}_2$ for $W_2$ (with hyperbolic pairs labelled $u_j, v_j$). 
 
Assume first that $q\geq 7$. Our method here exploits the existence of a Frobenius group inside $\SU_3(q)$, as follows: let $W_{2,0}=\langle u_1, x, v_1\rangle$ be a non-degenerate 3-dimensional subspace of $W_2$ and observe that we have two subgroups of $\SU_3(q)$ consisting of elements of the form
\begin{align}
U &=\left\{\begin{pmatrix}
 1 & b & c \\ & 1 & -b^q \\ & & 1                                                                                                                                                                                                                                                                               \end{pmatrix} \mid b,c \in \K \textrm{ with } b^{q+1}+c+c^q=0 \right\}; \label{e: u} \\
T &=\left\{\begin{pmatrix}
 r &  &  \\ & r^{q-1} &  \\ & & r^{-q}                                                                                                                                                                                                                                                                               \end{pmatrix} \mid r \in \K^\times \right\}. \label{e: t}
\end{align}
Then $U\rtimes T$ is a Borel subgroup of $\SU_3(q)$.

Now, first, assume that $q$ is odd. Take $U_0$ to be the subgroup of $U$ obtained by requiring that $b\in \Fq$ and that $c=-\frac 12 b^2$; take $y\in W_1$ such that $\langle y, y\rangle_1=1$. We now define an isomorphic group in $S$: let $U_1$ consist of those elements for which there exists $b\in \Fq$ such that
\begin{align*}
 y\otimes u_1 &\mapsto y\otimes u_1 + b y\otimes x -\frac12 b^2 y\otimes v_1, \\
  y\otimes x &\mapsto y\otimes x - b y\otimes v_1, \\
   y\otimes v_1&\mapsto y\otimes v_1,
   \end{align*}
and all elements of $\langle y\otimes u_1, y\otimes x, y\otimes v_1\rangle^\perp$ are fixed. Then $U_1$ is a subgroup of order $q$ that is not contained in $M$.

Similarly, think of $T$ as a subgroup of $\SU(W_2)$ by requiring that it fixes all elements in $W_{2,0}^\perp$, and take $T_0$ to be the subgroup of $T$ obtained by requiring that $r\in \Fq$; let $T_1$ be the subgroup in $S$ obtained by tensoring elements of $T_0$ with $1\in \SU(W_1)$. Then $T_1$ is a group of order $q-1$ that normalizes $U_1$ and acts transitively on the set of non-trivial elements in $U_1$. 

In the case when $q$ is even, we do similarly; this time $U_0$ is the subgroup of $U$ obtained by setting $b=0$ and letting $c$ range through $\Fq$, while $T_0=T$. Again, $T_0$ acts transitively upon the non-trivial elements of $U_0$; the same is therefore true of $U_1$.

In both cases in the usual way, we set $\Lambda$ to be $\mathcal{P}^{U_1}$, where $\mathcal{P}$ is the tensor product structure stabilized by $M$, and we see that $S^\Lambda$ acts 2-transitively upon $\Lambda$, with $\Lambda$ a set of size $q$.

We wish to show that this set is beautiful. As before, we see that $M_{(\Lambda)}$ contains $S\cap (\GU_{n_1-1}(q)\circ \GU_{n_2-3}(q))$, where the first factor fixes $y$ and the second fixes $u_1,x,v_1$. Hence we see that any non-abelian simple section of $M^\Lambda$ is isomorphic to a section of $\GU_3(q)$. Since $q\geq 7$, by Lemma \ref{l: alt sections classical} this precludes the possibility that $M^\Lambda\geq \Alt(q-1)$, and hence $\Lambda$ is a beautiful subset; now Lemma~\ref{l: beautiful} allows us to conclude that there are no such binary actions.

For $q\in \{3,4,5\}$ we assume that $n_2\geq 5$ (the first two lines of Table \ref{t: c4 sun} cover $n_2\leq 4$). We proceed as for $q\geq 7$ but we use the existence of a Frobenius group in $\SU_5(q)$ this time. We let $W_{2,0}:= \langle u_1, u_2, x, v_2, v_1\rangle$ be a non-degenerate 5-subspace of $W_2$, and consider the group:
\[
 U\rtimes T = \left\langle \begin{pmatrix}
                      1 & a & & & \\ & 1 & & & \\ & & 1 & & \\ & & & 1 & -a^q \\ & & & & 1
                     \end{pmatrix}, \begin{pmatrix}
                      r &  & & & \\ & 1 & & & \\ & & r^{q-1} & & \\ & & & 1 & \\ & & & & r^{-q}
                     \end{pmatrix} \mid a, r\in \K, r\neq 0
    \right\rangle.
\]
Now we define $U_1$, the subgroup for which there exists $a\in \K$ such that
\begin{align*}
 y\otimes u_1 &\mapsto y\otimes u_1 + a y\otimes u_2,\\
 y\otimes v_2 &\mapsto y\otimes v_2 - a^q y\otimes v_1,
\end{align*}
and which fixes $y\otimes u_1$, $y\otimes v_1$, and the orthogonal complement of $\langle y\otimes u_1, y\otimes u_2, y\otimes v_1, y\otimes v_2\rangle$. Note that this group is not contained in $M$, so we define $\Lambda=\mathcal{P}^{U_1}$ as before, this time a set of size $q^2$.

We take $T_1$ to be the group obtained by tensoring elements of $T_0$ with $1\in \SU(W_1)$. Then $T_1$ is a group of order $q^2-1$ that normalizes $U_1$ and acts transitively on $U_1\setminus \{1\}$, and as usual we conclude that $S^\Lambda$ acts 2-transitively on $\Lambda$.

Arguing as above we see that that any simple section of $M^\Lambda$ must appear as a section of $\GU_5(q)$ and so $\Lambda$ is beautiful provided $\Alt(q^2-1)$ is not a section of $\GU_5(q)$ -- this is true for $q\geq 3$ by Lemma~\ref{l: alt sections classical}.

Finally for $q=2$, we do as in the previous case, but we use the existence of a 2-transitive group in $\SU_6(q)$ this time. We require that $n_2\geq 6$ and we let $W_{2,0}= \langle u_1, u_2, u_3, v_1, v_2, v_3\rangle$ be a non-degenerate 6-subspace of $W_2$. Now consider the group
\[
 U\rtimes L = \left\langle \begin{pmatrix}
                      1 & a & b & &  & \\ & 1 & & & & \\ & & 1  & & & \\ & & & 1 & & \\ & & & -a^q & 1 &   \\ & & & - b^q & &  1
                     \end{pmatrix}, \begin{pmatrix}
                      1 & & & \\ & A & & \\ & & 1 & \\ & & & \bar A^{-T} 
                     \end{pmatrix} \mid a, b \in \K, A\in \SL_2(\K)
    \right\rangle,
\]
where we write $\bar A$ to denote the matrix obtained from $A$ by raising each entry to the $q$-th power. Proceeding as before we obtain a beautiful set provided $\Alt(2^4-1)$ does not appear as a section of $\GU_6(2)$ -- it does not, so we are done.
\end{proof}

\subsection{Case where \texorpdfstring{$S$}{S} is symplectic or orthogonal}

In all of the remaining cases the formed spaces $W_1$ and $W_2$ are either symplectic or orthogonal. These are the embeddings in the last three lines of Table~\ref{c4poss}. Our strategy is similar to the one already used, namely:
\begin{enumerate}
 \item We identify a subspace $W_{2,0}$ in $W_2$, and we identify a group $U\rtimes T$ in ${\rm Isom}(W_{2,0})$ for which $T$ acts transitively on the non-identity elements of $U$.
 \item If $W_1$ is orthogonal, we choose a non-degenerate 1-space $X = \la x \ra \subseteq W_1$, and 
 identify a subgroup $U_1$ of ${\rm Isom}(V)$ whose action on $X \otimes W_{2,0}$ is isomorphic to the action of $U$ on $W_{2,0}$, and which fixes the vectors in $(X\otimes W_{2,0})^\perp$.  In particular, since $\dim(W_{1})>1$, this means that $U_1$ is not a subgroup of $M$. If $W_1$ is symplectic, we do soemthing similar, working with a non-degenerate 2-space $X \subseteq W_1$.
 \item We define $T_1$ to be $1\otimes T$, and observe that $T_1$ normalizes $U_1$, and lies in $M$. This then allows us to define $\Lambda=\mathcal{P}^{U_1}$, where $\mathcal{P}$ is the tensor product structure stabilized by $M$, and we observe that $S^\Lambda$ is 2-transitive.
 \item We then identify $M_\Lambda$ and use this to define a monomorphism from $M^\Lambda$ into a small rank classical group, $H$. The proof is complete, by Lemma~\ref{l: beautiful}, provided $M$ does not contain a section isomorphic to $\Alt(|\Lambda|-1)$ .
\end{enumerate}

%There are a number of complications which we must address for our approach to work:
%\begin{enumerate}
 %\item The final step (asserting that $\Alt(|\Lambda|-1)$ is not a section of $H$) may fail for small $n$ and $q$ -- this is dealt with as in previous sections.
 %\item In our first step, we identify a group $U\rtimes T$ in ${\rm Isom}(W_{2,0})$, but of course we must be careful because, in general, it is not true that a tensor product of \emph{all} of ${\rm Isom}(W_{2,0})$ with ${\rm Isom}(W_{1})$ will lie in $S$. In previous steps it has been enough to check that the elements of $U\rtimes T$ have determinant $1$, however in the orthogonal case this is not always sufficient. The results of \cite[\S4.4]{kl} tell us when we need to do better, i.e. when we need to ensure that $U\rtimes T$ lies in $\Omega(W_{1})$.
 %\item If $W_{1}$ is symplectic, then we cannot choose a non-isotropic element $x$, as such elements do not exist. Instead we will have to choose a non-degenerate subspace of dimension $2$ in $W_{1}$, and take tensor products with this space.
%\end{enumerate}

We start by considering the possibilities for $W_{2,0}$ for $q$ not too small. First, suppose that $W_2$ is symplectic and contains a subspace $W_{2,0}:= \langle u_1, u_2, v_2, v_1\rangle$, where $\langle u_i, v_i\rangle$ are mutually perpendicular hyperbolic pairs. Then with respect to this basis we define
\[
U\rtimes T :=\left\langle\begin{pmatrix} 1 & a & & \\ & 1 & & \\ & & 1 & -a \\ & & & 1 \end{pmatrix}, \begin{pmatrix} r &  & & \\ & 1 & & \\ & & 1 & \\ & & & r^{-1} \end{pmatrix} \mid a, r \in \Fq, r\neq 0
     \right\rangle.
\]
%here $\eta=1$ if $W_i$ is symplectic, and $\eta=-1$ if $W_i$ is orthogonal. In either case it is clear that $U\rtimes T$ is a subgroup of the special isometry group of $W_i$; when this is sufficient we will take $W_{i,0}=\langle u_1,u_2, v_2, v_1\rangle$.
Observe that $U\rtimes T< \Sp(W_2)$, where we take this subgroup to fix $W_{2,0}^\perp$ pointwise.

Suppose next that $W_2$ is orthogonal, in which case $q$ is odd.  In some cases we need $U\rtimes T$ to lie in $\Omega(W_2)$, and in such cases we will assume that $W_2$ has a non-degenerate subspace $W_{2,0}$ with standard basis 
$u_1,u_2, x, v_2, v_1$. With respect to this basis we define
\[
U\rtimes T :=\left\langle\begin{pmatrix} 1 & & a & & -\frac12 a^2 \\ & 1 & & & \\ & & 1 & & -a \\ & & & 1 & \\ & & & & 1 \end{pmatrix}, \begin{pmatrix} r & & & & \\ & s & & & \\ & & 1 & & \\ & & & s^{-1} & \\  & & & & r^{-1} \end{pmatrix} \mid a, r, s \in \Fq, r,s\neq 0, \,rs \textrm{ a square in } \Fq
     \right\rangle.
\]
Note that $U\rtimes T <\Omega(W_2)$ by \cite[Lemmas 2.5.7 and 2.5.9]{bg}.

If we only need $U \rtimes T$ to lie in the special isometry group $\SOr(W_2)$, then we take $W_{2,0}=\langle u_1, x, v_1\rangle$, and with respect to this basis we define
\[
U\rtimes T :=\left\langle\begin{pmatrix} 1 & a & -\frac12 a^2 \\ & 1 & -a \\ & & 1  \end{pmatrix}, \begin{pmatrix} r &  & \\ & 1 & \\ & & r^{-1} \end{pmatrix} \mid a, r \in \Fq, r\neq 0
     \right\rangle.
\]
%here $\eta=1$ if $W_i$ is symplectic, and $\eta=-1$ if $W_i$ is orthogonal. In either case it is clear that $U\rtimes T$ is a subgroup of the special isometry group of $W_i$; when this is sufficient we will take $W_{i,0}=\langle u_1,u_2, v_2, v_1\rangle$.
Observe that $U\rtimes T< \SOr(W_2)$.

\begin{table}\centering
\begin{tabular}{cl}
\toprule[1.5pt]
Group & Details of action \\
\midrule[1.5pt]
%$\Sp_6(q)$ & $n_1=2$, $n_2=3$: $M\geq \Sp_2(q)\circ \Omega_3(q)$. \\
%$\Sp_8(q)$ & $n_1=2$, $n_2=4$: $M\geq \Sp_2(q)\circ \Omega_4^-(q)$. \\
$\Sp_n(q)$ & $q\in\{3,5,7\}$, $n_1\in\{2,4\}$, $n_2\in \{3,4\}$: $M\triangleright \PSp_{n_1}(q) \times \POmega_{n_2}^{\varepsilon}(q)$. \\
\bottomrule[1.5pt]
\end{tabular}
\caption{$\C_4$ -- $\Sp_n(q)$ -- Cases where a beautiful subset was not found.}\label{t: c4 spn}
\end{table}

\begin{table}\centering
\begin{tabular}{rl}
\toprule[1.5pt]
Group & Details of action \\
\midrule[1.5pt]
 $\Omega_{15}(q)$ & $q\in\{5,7\}$, $M\triangleright \Omega_3(q)\times \Omega_5(q)$ \\
\bottomrule[1.5pt]
\end{tabular}
\caption{$\C_4$ -- $\Omega_n(q)$ -- Cases where a beautiful subset was not found.}\label{t: c4 omegan}
\end{table}

\begin{table}\centering
\begin{tabular}{cl}
\toprule[1.5pt]
Group & Details of action \\
\midrule[1.5pt]
% $\Omega^+_{12}(q)$ & $q$ odd, $q>3$, $n_1=3$, $n_2=4$: $M\geq \Omega_3(q) \circ \Omega_4^+(q)$. \\
% $\Omega^+_{16}(q)$ & $q$ odd, $n_1=n_2=4$: $M\geq \Omega_4^+(q) \circ \Omega_4^-(q)$. \\
% $\Omega^+_8(q)$  & $q\in\{3,5, 7\}$, $n_2=4$: $M\geq \Sp_2(q)\circ \Sp_4(q)$. \\
 $\Omega^+_{24}(2)$, $\Omega^+_{32}(2)$, $\Omega^+_{48}(2)$  & $4\leq n_1 < n_2\leq 8$: $M\triangleright \PSp_{n_1}(2)\times \PSp_{n_2}(2)$. \\
 $\Omega^+_{16}(q)$ & $q\in\{3,5,7\}$, $n_1=n_2=4$: $M\triangleright \POmega_{4}^-(q) \times \POmega_{4}^{+}(q)$.\\
 $\Omega^+_{12}(q)$ & $q\in\{5,7\}$, $M\triangleright \POmega_3(q)\times \POmega_4^+(q)$ \\
 $\Omega^+_{18}(q)$ & $q\in\{5,7\}$, $M\triangleright \Omega_3(q)\times \POmega_6^+(q)$ \\
\bottomrule[1.5pt]
\end{tabular}
\caption[Caption for C4O+]{$\C_4$ -- $\Omega_n^+(q)$ -- Cases where a beautiful subset was not found.}\label{t: c4 omegaplusn}
\end{table}

\begin{table}\centering
\begin{tabular}{cl}
\toprule[1.5pt]
Group & Details of action \\
\midrule[1.5pt]
 $\Omega^-_{12}(q)$ & $q\in\{5,7\}$, $M\triangleright \Omega_3(q)\times \POmega_4^-(q)$ \\
 $\Omega^-_{18}(q)$ & $q\in\{5,7\}$, $M\triangleright \Omega_3(q)\circ \POmega_6^-(q)$ \\
\bottomrule[1.5pt]
\end{tabular}
\caption{$\C_4$ -- $\Omega_n^-(q)$ -- Cases where a beautiful subset was not found.}\label{t: c4 omegaminusn}
\end{table}

\begin{lem}\label{l: c4 sp so q large}
 In this case, if $q\geq 8$, then $\Omega$ contains a beautiful subset. % or else $S$ is listed in Tables~\ref{t: c4 spn}, \ref{t: c4 omegan}, \ref{t: c4 omegaplusn} or \ref{t: c4 omegaminusn}.
\end{lem}
\begin{proof}
 Suppose that (relabelling $W_1$ and $W_2$ if necessary) $W_2$ satisfies one of the following possibilities:
 \begin{enumerate}
  \item $W_2$ is symplectic of dimension at least 4;
  \item $W_2$ is orthogonal of dimension at least 5;
  \item $W_2$ is orthogonal of dimension $3$ or $4$, and there exists $X\leq {\rm Isom}(W_1)$ such that $X\otimes \SOr(W_2)$ embeds in $S$.
 \end{enumerate}
In each of these cases we take $W_{2,0}$ to be the space described above, and $U\rtimes T$ as defined above.
 
Suppose, first, that $W_1$ is not symplectic. Then \cite[\S4.4]{kl} confirms that $q$ is odd, and we take $x$ to be a non-isotropic element of $W_1$. Then we proceed as detailed above: so $U_1$ is a subgroup of $S$ whose action on $x\otimes W_{2,0}$ is isomorphic to the action of $U$ on $W_{2,0}$, and $T_1$ is the group $1\otimes T$. Setting $\Lambda=\mathcal{P}^{U_1}$, we observe that $\Lambda$ is a set of size $q$ such that $S^\Lambda$ is $2$-transitive.

Now let $Y_2$ be the subgroup of ${\rm Isom}(W_2)$ that fixes point-wise the elements of $W_{2,0}$, and let $Y_1$ be the subgroup of ${\rm Isom}(W_1)$ that fixes the element $x$. Then  $M_{(\Lambda)}$ contains $S\cap (Y_1\otimes Y_2)$; hence any non-abelian simple section of $M^\Lambda$ is isomorphic to a section of ${\rm Isom}(W_{2,0})$. Now $W_{2,0}$ is either symplectic of dimension $4$, or orthogonal of dimension $3$ or $5$. By Lemma \ref{l: alt sections classical}, for $q \ge 9$, 
${\rm Isom}(W_{2,0})$ does not have a section $\Alt(q-1)$. Thus $\Lambda$ is a beautiful subset, and the action is not binary by Lemma~\ref{l: beautiful}.

This argument yields the result except when one of the following holds:
\begin{itemize}
 \item[(a)] both $W_1$ and $W_2$ are symplectic;
 \item[(b)] both $W_1$ and $W_2$ are orthogonal, and cannot be labeled so that $W_2$ satisfies the restrictions stated at the start;
 \item[(c)] labelling appropriately, $W_1$ is symplectic, and $W_2$ is orthogonal and does not satisfy the restrictions stated at the start.
\end{itemize}
We see that situation (b) occurs only if $n_i=\dim(W_i)\leq 4$ for $i=1,2$; however \cite[4.4.13]{kl} implies that $1 \otimes SO(W_2) < S$ for $n_2=3$ or 4, so in fact case 3 above pertains and we are done. 
 Situation (c) is similarly ruled out, except when $W_1$ is symplectic of dimension $2$. 
%In particular, using \cite[\S4.4]{kl}, from here on we can assume that $S\cong \Sp_n(q)$ or $\Omega_n^+(q)$.
Suppose, then, that we are in this case:  $W_1$ is symplectic of dimension $2$, and $W_2$ is orthogonal. Again, $q$ is odd here, $S$ is symplectic, and \cite[Lemma 4.4.11]{kl} implies that  $1\otimes \Or(W_2)$ embeds in $S$. We write $W_1=\langle e_1, f_1\rangle$, and we take $W_{2,0}$ to be the 3-dimensional subspace of $W_2$ described before the statement of the lemma. Then $W_{1}\otimes W_{2,0}$ is a 6-dimensional non-degenerate symplectic space with a hyperbolic basis given as follows (we omit the tensor sign for clarity, and we list hyperbolic pairs together, starting with the first two):
\[
 \{e_1u_1, \, f_1v_1, \, e_1v_1, \, f_1u_1, \, e_1x, \, f_1x\}.  
\]
Now if we consider the group $T_1=1\otimes T$ with respect to this basis, we see that it is diagonal with entries
\[
 [r, r^{-1}, r^{-1}, r, 1, 1].
\]
 On the other hand, we can take $U_0$ to be the set of elements given with respect to this basis by
 \[
  \left\{ \begin{pmatrix}
           1 & & & & a & \\ & 1 & & & & \\ & & 1 & & & \\ & & & 1 & & \\ & & & & 1 & \\ & -a & & & & 1
          \end{pmatrix} \mid a \in \Fq\right\},
 \]
and we take $U_1$ to be the subgroup of $S$ which acts like $U_0$ on $W_{1}\otimes W_{0,2}$, and fixes the elements in its orthogonal complement. Then $U_1$ is not a subgroup of $M$, and is normalized by $T_1$, and so we obtain a set $\Lambda=\mathcal{P}^{U_1}$ of size $q$ for which $S^\Lambda$ is 2-transitive. Arguing as before, we find that a simple section of $M^\Lambda$ is isomorphic to a section of either ${\rm Isom}(W_{2,0})$ or ${\rm Isom}(W_1)$. We obtain, therefore, a beautiful subset, provided $\Alt(q-1)$ is not a section of $\Or_3(q)$ or $\Sp_2(q)$. This is true for $q\geq 7$, and we are done.

Finally, we suppose that situation (a) holds. Here both $W_1$ and $W_2$ are symplectic, $S=\Omega_n^+(q)$ and, since $n\geq 8$, we can assume without loss of generality that $\dim(W_2)\geq 4$. In this case, we set $W_{2,0}=\langle u_1, u_2, v_2,v_1\rangle$ as detailed above, and we consider the basis of $\langle e_1, f_1\rangle\otimes W_{2,0}$ given by
\[
 \{-f_1u_1, \, e_1v_1, \, -f_1u_2, \, e_1v_2, \, f_1v_1, \, e_1u_1, \, f_1v_2, \, e_1u_2\}.
\]
Again $T_1=1\otimes T$ is given by the diagonal matrix with entries
\[
 [r, r^{-1}, 1, 1, r^{-1}, r, 1, 1].
\]
 On the other hand, we can take $U_0$ to be the set of elements given with respect to the subspace, $Y$, spanned by the first four of these elements
 \[
  \left\{ \begin{pmatrix}
           1 & & a & \\ & 1 & & \\ & & 1 & \\ & -a & & 1
          \end{pmatrix} \mid a \in \Fq\right\},
 \]
and we take $U_1$ to be the subgroup of $S$ which acts like $U_0$ on $Y$, and fixes the elements in its orthogonal complement. Then $U_1< \SOr^+_n(q)$.

Now $U_1$ is not a subgroup of $M$, is normalized by $T_1$, and so we obtain a set $\Lambda=\mathcal{P}^{U_1}$ of size $q$ for which $S^\Lambda$ is 2-transitive. Arguing as before, we find that a simple section of $M^\Lambda$ is isomorphic to a  section of either ${\rm Isom}(W_{2,0})$ or ${\rm Isom}(\langle e_1, f_1\rangle)$. We obtain, therefore, a beautiful subset, provided $\Alt(q-1)$ is not a section of $\Sp_4(q)$. This is true for $q\geq 8$, and we are done.
\end{proof}

\begin{lem}\label{l: c4 sp so q small}
 In this case, if $q\leq 7$, then $\Omega$ contains a beautiful subset or else $S$ is listed in Tables~$\ref{t: c4 spn}$,~$\ref{t: c4 omegan}$,~$\ref{t: c4 omegaplusn}$ or~$\ref{t: c4 omegaminusn}$.
\end{lem}
\begin{proof}
Let us suppose first that $W_1$ and $W_2$ are symplectic, and so $S=\Omega_n^+(q)$; then \cite[Table~4.4.A]{kl} implies that we can assume that $n_2> n_1$.

Suppose, first, that $n_2=4$; then $n_1=2$ and $n=8$. Now \cite[Table~8.50]{bhr} confirms that no $\mathcal{C}_4$-maximal subgroups exist when $q$ is even. What is more, when $q$ is odd, all $\mathcal{C}_4$-subgroups are conjugate, via a triality automorphism, to certain maximal $\mathcal{C}_1$-subgroups; then \cite[Proposition 4.6]{gs_binary} asserts that $\Omega$ contains a beautiful subset.

Suppose from here on that $n_2\geq 6$. If $q>2$, then the procedure is virtually identical to that in the previous lemma, but this time we build a beautiful set of size $q^2$. To do this we start with a 6-dimensional subspace of $W_2$: define $W_{2,0} = \langle u_1, u_2, u_3, v_1, v_2, v_3\rangle$. Then with respect to this basis we define
\begin{equation}\label{e:lkj}
U\rtimes T :=\left\langle\begin{pmatrix} 1 & a_1 &  a_2  & & & \\  & 1 &    & & & \\ &  & 1 & &  & \\ & &  & 1 & &  \\ & & & 
-a_1 & 1 &  \\ & & & -a_2 & & 1 \end{pmatrix}, \begin{pmatrix} 1 &  & & \\ & A  & & \\ & & 1 & \\ & & & A^{-T} \end{pmatrix} \mid a_1, a_2 \in \Fq, A\in \GL_2(q)
     \right\rangle.
\end{equation}
Observe that $U\rtimes T< \Sp(W_2)$. We set $T_1=1\otimes T$, and we set $U_1$ to be the set of elements given by the same matrices as $U$ above, but with respect to the basis
\[
 \{e_1u_1, \, e_1u_2, \, e_1u_3, \, f_1v_1, \, f_1v_2, \, f_1v_3\},
\]
and fixing the elements in the orthogonal complement. As before we obtain a beautiful subset of size $q^2$, provided $\Alt(q^2-1)$ is not a section of $\Sp_6(q)$. This is true for $q\geq 3$. 

Now assume that $q=2$, in which case \cite[Table~3.5.E]{kl} implies that $n_1>2$. We proceed as in the previous paragraph, but this time, we assume that $n_2\geq 10$, and we take $T$ to be a group isomorphic to $\GL_4(q)$. We obtain a beautiful subset of size $q^4=16$, provided $\Alt(q^4-1)=\Alt(15)$ is not a section of $\Sp_{10}(2)$; it is not (by Lemma \ref{l: alt sections classical}), so the result follows. The exceptions occur when $4\leq n_1<n_2\leq 8$, and are listed in Table~\ref{t: c4 omegaplusn}. This completes the case where both $W_1$ and $W_2$ are symplectic. 
 
%From here on, we may assume that $q$ is odd, i.e. $q\in\{3,5,7\}$. Suppose, first, that either $W_1$ or $W_2$ is symplectic, but not both. Then $S=\Sp_n(q)$.

Suppose now that $W_1$ is orthogonal. Then $q$ is odd, so $q \in \{3,5,7\}$. 

First, assume that  $W_2$ is symplectic of dimension at least $6$. As $W_1$ is orthogonal, it contains a non-isotropic vector $x$. We define $U\rtimes T$ exactly as for \eqref{e:lkj}. We let $T_1:= 1\otimes T$, and we set $U_1$ to be the set of elements given by the same matrices as $U$ above, but with respect to the basis
\[
 \{xu_1, \, xu_2, \, xu_3, \, xv_1, \, xv_2, \, xv_3\},
\]
and fixing the elements in the orthogonal complement. As before we obtain a beautiful subset of size $q^2$, provided $\Alt(q^2-1)$ is not a section of $\Sp_6(q)$ (which is true, by Lemma~\ref{l: alt sections classical},  since $q\geq 3$).

Now assume that $W_2$ is symplectic of dimension 2 or 4, and also that $\dim W_1 \ge 5$.
To keep notation consistent, relabel $W_2$ as $W_1$ and vice versa. 
 Then $W_2$ is orthogonal of dimension at least $5$, and we define $W_{2,0} = \langle u_1, u_2, x, v_1, v_2\rangle$. Then with respect to this basis we define
\begin{equation}\label{e:lk}
U\rtimes T :=\left\langle\begin{pmatrix} 1 & &  a_1  & -\frac 12 a_1^2 & \\  & 1 &  a_2  & & -\frac12 a_2^2 \\ &  & 1 & -a_1 & -a_2 \\ & &  & 1 &  \\ & & &  & 1  \end{pmatrix}, \begin{pmatrix} A &  & \\ & 1 & \\ & & A^{-T}  \end{pmatrix} \mid a_1, a_2 \in \Fq, A\in \GL_2(q)
     \right\rangle.
\end{equation}
As usual we set $T_1=1\otimes T$. To define $U_1$ we let $e_1, f_1$ be a hyperbolic pair in $W_1$, and we consider the space 
\[
W_{2,0}' = \la e_1u_1, \, e_1u_2, \, e_1x, \, f_1x, \, f_1v_1, \, f_1v_2 \ra,
\]
which we observe is a non-degenerate symplectic 6-space. We define $U_1$ to act as $1\otimes U$ on $W_{2,0}'$, and to fix $W_{2,0}'^\perp$. Now, proceeding as above we obtain a beautiful subset of size $q^2$, provided $\Alt(q^2-1)$ is not a section of $\Or_5(q)$ (which is true, by Lemma~\ref{l: alt sections classical},
 since $q\geq 3$).

The previous two paragraphs cover all cases where $S=\Sp_n(q)$, since Table~\ref{t: c4 spn} contains the remaining cases with $n_1,n_2\leq 4$.

Finally, suppose that $W_1$ and $W_2$ are both orthogonal. Recall that $q \in \{3,5,7\}$.

Assume $n_1$ and $n_2$ are both even. Then \cite[4.4.2, 4.4.13]{kl} implies that $S=\Omega_n^+(q)$, and that $1\otimes \SOr(W_2)$ and $\SOr(W_1)\otimes 1$ both embed into $S$. We suppose now that $n_2\geq n_1$ with $n_2\geq 6$. Then we define $W_{2,0}$, $U$ and $T$ via \eqref{e:lk}. We set $T_1=1\otimes T$, and we set $U_1$ to be the set of elements given by the same matrices as $U$ above, but with respect to the basis
\[
 \{yu_1, \, yu_2, \, yx, \, yv_1, \, yv_2\}
\]
(where $y$ is an anisotropic element of $W_1$), and fixing the elements in the orthogonal complement. As before we obtain a beautiful subset of size $q^2$, provided $\Alt(q^2-1)$ is not a section of $\SOr_5(q)$. This is true, by Lemma~\ref{l: alt sections classical},  for $q\geq 3$. This leaves the case where $n_1=n_2=4$, which is in Table \ref{t: c4 omegaplusn}.
%In the case $n_1=2$, \cite[Table 8.50]{bhr} implies that no $\mathcal{C}_4$-subgroups of this kind occur as maximal subgroups, hence the only exceptions occur when $(n_1, n_2, n)=(4,4,16)$.

Notice that the same argument works if $n_1$ is even and $n_2$ is odd with $n_2\geq 5$ (again using the fact, given in \cite[4.4.13]{kl}, that $1\otimes \SOr(W_2)$ embeds into $S$). Again we obtain a beautiful subset since $q\geq 3$.

We are left with the possibility that both $n_1$ and $n_2$ are odd (in which case we may assume that $3\leq n_1<n_2$), or (relabelling if necessary), that $n_1=3$ and $n_2$ is even. In this case we assume now that $n_2\geq 7$. Now the argument of the previous paragraph works except that we cannot assume that $1\otimes \SOr(W_2)\leq S$; this means that we must adjust the definition of $W_{2,0}$ to ensure that $T\leq \Omega(W_2)$. To do this we define $W_{2,0} = \langle u_1, u_2, u_3, x, v_1, v_2, v_3\rangle$; we take $U$ to be identical to that given in \eqref{e:lk}, except that we prescribe that the $7\times 7$-matrices of $U$ fix $u_3$ and $v_3$; then we define 
\[
 T=\left\{\begin{pmatrix}
           A & & & & \\ & s & & & \\ & & 1 & & \\ & & & A^{-T} & \\ & & & & s^{-1} 
          \end{pmatrix} \,\middle\vert\, s\in \Fq^*, A \in \GL_2(q) \textrm{ and }s.\det(A) \textrm{ is a square}\right\}.
\]
Now the argument proceeds as before.

The cases not yet covered have $n_1=3$, $n_2=4,5,6$, and are listed in Tables \ref{t: c4 omegaplusn} and 
\ref{t: c4 omegaminusn}. Note that \cite[Tables 3.5.D, 3.5.E and 3.5.F]{kl} imply that if $n_1=3$, then we can exclude $q=3$.
\end{proof}

\subsection{The remaining cases}

The remaining cases are dealt with by the following result.

\begin{lem}\label{l: c4 q small}
If the action is listed in Tables~$\ref{t: c4 sln}$,~$\ref{t: c4 sun}$,~$\ref{t: c4 spn}$,~$\ref{t: c4 omegan}$,~$\ref{t: c4 omegaplusn}$ or~$\ref{t: c4 omegaminusn}$, then it is not binary.
\end{lem}
\begin{proof}
The socle of $M$ is a direct product, as given by the tables. Our method for most cases is as follows: suppose that the action of $G$ on $(G:M)$ is binary. In every case we can see that $|G:M|$ is even. Thus, given a Sylow $2$-subgroup $P$ of $M$, there exists an element $x$ of order a power of $2$ in $G\setminus M$ that normalizes $P$. Then $|M:M\cap M^x|$ is odd and $M\cap M^x$ is core-free in $M$. Now a {\tt magma} computation shows that every faithful transitive action of odd degree of a group $M$, with socle as given in one of the tables, is not binary. Hence $(M,\,(M:M\cap M^x))$ is not binary, and the conclusion follows by Lemma~\ref{l: point stabilizer}. 

In a couple of cases where the {\tt magma} computation required too much time we have, instead, found a suitable group $H<M$ with the property that $N_G(H)$ is not contained in $M$. This guarantees that there is a suborbit of $M$ for which the stabilizer contains $H$. Now we use {\tt magma} to show that the action of $M$ on such a suborbit is not binary, and the result follows, again, by Lemma~\ref{l: point stabilizer}.

\end{proof}

\section{Family \texorpdfstring{$\C_5$}{C5}}\label{s: c5}

In this case $M$ is a ``subfield subgroup'': let $\mathbb{F}_{q_0}$ be a subfield of $\K$ (with $|\K|=q_0^r$ for some prime $r$), and let $\B$ be a basis of $V$. Then $V_0=\spann_{\mathbb{F}_0}(\B)$ is an $n$-dimensional $\mathbb{F}_{q_0}$-vector space. The group $\GL_n(q)$ acts naturally on the set of all such vector spaces and $M$ can be taken to be a subgroup of the stabilizer of $V_0$ in this action.

When $G$ is not $\SL_n(q)$, the group $S$ is a set of isometries for some non-degenerate form $\varphi$ on $V$. Now we require that $M$ is also a subset of the set of isometries of the form $\varphi_0$ on $V_0$ which is the restriction of $\varphi$ (or a scalar multiple of $\varphi$); full details are given in \cite[\S4.5]{kl}. We list the embeddings in Table \ref{c5poss}. Note that the subfield subgroups are centralized by outer automorphisms of $S$ (see Proposition \ref{outeraut}), so $M$ may not be almost simple. We will need to take account of this possibility in the proofs below.
%In particular, in the case where $\varphi$ is unitary, then we must also consider the possibility that $\varphi_0$ is alternating; this situation does not arise by just {\it restricting} $\varphi$ -- one must ``twist'' $\varphi$ using a scalar: details are in~\cite[\S4.5]{kl}.

\begin{table}[ht!]
\[
\begin{array}{|c|c|c|}
\hline
\hbox{case} &  \hbox{type} & \hbox{conditions} \\
\hline
{\rm L} & \GL_{n} (q^{1/r}) &  \\
{\rm S} & \Sp_{n}(q^{1/r}) & \\
{\rm O}^\e & \Or_n^\d(q^{1/r}) & \e = \d^r \\
{\rm U}& \GU_n(q^{1/r}) & r \hbox{ odd}  \\
{\rm U}& \Or_n^\e(q) & r=2, \,q \hbox{ odd} \\
{\rm U} & \Sp_n(q) & r=2,\, n \hbox{ even} \\
\hline
\end{array}
\]
\caption{Maximal subgroups in family $\C_5$} \label{c5poss}
\end{table}

The main result of this section is the following. The result will be proved in a series of lemmas.

\begin{prop}\label{p: c5}
 Suppose that $G$ is an almost simple group with socle $\bar S = \Cl_n(q)$, and assume that 
\begin{itemize}
\item[{\rm (i)}] $n\ge 3,4,4,7$ in cases $L,U,S,O$ respectively, and 
\item[{\rm (ii)}] $\Cl_n(q)$ is not one of the groups listed in Lemma $\ref{l: beautifulsetssmall}$.
\end{itemize}
Let $M$ be a maximal subgroup of $G$ in the family $\mathcal{C}_5$. Then the action of $G$ on $(G:M)$ is not binary.
\end{prop}

\subsection{Case \texorpdfstring{$S=\SL_n(q)$}{S=SL(n,q)}}

\begin{table}\centering
\begin{tabular}{cl}
\toprule[1.5pt]
Group $S$ & Details of action \\
\midrule[1.5pt]
%$\SL_2(q)$ & $M\in \mathcal{C}_5$. \\
$\SL_3(2^r)$ &  $r$ prime, $M\triangleright \SL_3(2)$. \\
$\SL_4(2^r)$ & $r$ prime, $M\triangleright \SL_4(2)$. \\
\bottomrule[1.5pt]
\end{tabular}
\caption{$\C_5$ -- $\SL_n(q)$ -- Cases where a beautiful subset was not found.}\label{t: c5 sln}
\end{table}

\begin{lem}\label{l: c5 sln}
 In this case either $\Omega$ contains a beautiful subset or else $S$ is listed in Table~$\ref{t: c5 sln}$.
\end{lem}

\begin{proof}
Let $\B=\{v_1,\dots, v_n\}$ be a basis for $V$, and we assume that $M$ stabilizes the $\Fqz$-span of $\B$. We use Lemma~\ref{aff} for which we need to exhibit two subgroups, as follows.

We set $A\cong\SL_{n-2}(q_0)$ to be the subgroup of $M$ that fixes $v_{n-1}$ and $v_n$; we let $B_0\cong\SL_{n-1}(q_0)$ be the subgroup of $M$ which fixes $v_n$. Then let $g\in C_G(A)$ such that $v_{n-1}^g$ is not in the $\Fqz$-span of $\B$; we set $B=B_0^g$ and note that $B\not\leq M$. Now Lemma~\ref{aff} implies that there is a subset $\Delta$ of $\Omega$ such that $|\Delta|=q_0^{n-2}$ and $G_\Delta$ acts 2-transitively on $\Delta$.

If $\Delta$ is a beautiful subset, then Lemma~\ref{l: beautiful} yields the result; if $\Delta$ is not a beautiful subset, then $\Alt(q_0^{n-2})$ must be a section of $\SL_n(q)$. By Lemma \ref{l: alt sections classical}, this is impossible unless $(n,q_0)\in\{(4,2),(5,2)\}$ or $n=3, q_0\le 7$.

Consider the remaining situations, and set $A\cong \SL_{n-1}(q_0)$ to be the subgroup of $M$ that fixes $v_n$. Let $g$ be the diagonal matrix with entries
\[
(\lambda, \dots, \lambda, \lambda^{-n+1}),
\]
where $\lambda$ is an element of $\Fq\setminus \Fqz$ such that $\lambda^n\not\in\Fqz$; this is possible unless $(n,q)=(3,4)$. Setting $B=M^g$, and applying Lemma~\ref{aff} yields a set $\Delta$ as above, except that this time $|\Delta|=q_0^{n-1}$. Again we obtain a beautiful subset unless $\Alt(q_0^{n-1})$ is a section of $\SL_n(q)$; we conclude that $n\leq 4$ and $q_0=2$.
\end{proof}

\begin{lem}\label{l: c5 sln 2}
 If $S$ is listed in Table~$\ref{t: c5 sln}$, then the action is not binary.
\end{lem}
\begin{proof}
 Here $n\in\{3,4\}$, $S=\SL_n(2^r)$ with $r$ a prime, and $M$ contains a normal subgroup $\SL_n(2)$. If $r\in\{2,3\}$, then  Lemma~\ref{l: beautifulsetssmall} yields the result.
  
 Assume from here on that $r\geq 5$.  We have $M=M_0 \times \la \phi \ra$, where $M_0 \cong \SL_n(2).a$ with $a\in \{1,2\}$, and $\phi$ is either 1 or a field automorphism of $S$ of order $r$. 

Let $Q$ be a Sylow 2-subgroup of $M_0$. As $|G:M|$ is even , there exists $g \in N_G(Q)\setminus M$. Then $M_0\cap M_0^g$ contains $Q$, hence is a parabolic subgroup $P$ of $M_0$, and $N_{M_0}(P)=P$. It follows that $M\cap M^g = P\times \la \s \ra$, where $\s = 1$ or $\phi$. In particular, $\s$ is in the kernel of the action of $M$ on $(M:M\cap M^g)$. Hence this action is isomorphic to either $(M_0,(M_0:P))$ or $(M_0\times \la \phi\ra,\,(M_0\times \la \phi\ra:P))$. Lemma \ref{l: odd degree Lie}
shows that the first action is not binary, and it follows using Lemma \ref{l: subgroup} that the second action is also not binary.
\end{proof}

% Then $M$ is isomorphic to one of the following four groups:
% \[
 % \SL_n(2), \, \SL_n(2):2, \, \SL_n(2) \times r, \, \SL_n(2):2 \times r.
 %\]
% Lemma~\ref{l: odd degree Lie} implies that if $M$ is $\SL_n(2)$ or $\SL_n(2):2$, then $M$ has no non-trivial binary actions of odd degree. Now, since $|G:M|$ is even, there exists $g\in M$ such that $M\cap M^g$ is of odd index in $M$ and Lemma~\ref{l: point stabilizer} implies that, in the first two cases, the action of $G$ on $\Omega$ is not binary as required.

%Suppose, then, that $M$ contains $M_0$ where $M_0=\SL_n(2)\times r$, and let $K$ be a proper odd index subgroup of $M$. Then, since $|M:M_0|$ is a power of $2$, $M=M_0K$ and we can apply Lemma~\ref{l: central product} along with Lemma~\ref{l: odd degree Lie} (as in the previous paragraph) to conclude that the action of $M$ on $(M:K)$ is not binary. Now, as before, the fact that $|G:M|$ is even, together with Lemma~\ref{l: point stabilizer} yields the result.
%\end{proof}

\subsection{Case \texorpdfstring{$S=\SU_n(q)$}{S=SU(n,q)}}\label{s: c5 sun}

Note that we are assuming that $n\geq 4$, since the case where $S=\SU_3(q)$ is covered in \cite{ghs_binary}. Note, though, that an inspection of the proof \cite{ghs_binary} shows up a missing case when $q_0=2$. Let us deal with that case now. 

\begin{lem}\label{l: c5 sun extra}
 Suppose that $S=\SU_3(2^r)$ with $r$ an odd prime, and that $M$ is a subfield subgroup of $G$ containing $\PSU_3(2)$. Then the action of $G$ on $(G:M)$ is not binary.
\end{lem}

\begin{proof}
We use {\tt magma}, first, to confirm the result when $r=3$. For the rest of the proof we suppose $r\ge 5$. 
We have $M = M_0 \times \la \phi\ra$, where 
$M_0\in\{\PSU_3(2), \PSU_3(2).2, \PGU_3(2), \PGU_3(2).2\}$ and $\phi$ is either 1 or a field automorphism of $S$ of order $r$. 
Another {\tt magma} computation confirms that all non-trivial odd-degree core-free actions of $M_0$ are not binary. 

The proof is now similar to that of the previous lemma. 
Let $Q \in Syl_2(M_0)$ and $g \in N_G(Q)\setminus M$. Then $Q \le M_0\cap M_0^g \le Q\la h\ra$, where $h$ has order 1 or 3. 
If $M\cap M^g \not \le M_0$, then $M\cap M^g$ contains an element $h^i\phi$ for some $i$, and hence also contains $\phi$ (as $\phi$ has order $r>3$). Thus $M\cap M^g = (M_0\cap M_0^g)\times \la \s \ra$, where $\s = 1$ or $\phi$. Now we complete the proof as in Lemma \ref{l: c5 sln 2}. 
\end{proof}

\begin{table}\centering
\begin{tabular}{cl}
\toprule[1.5pt]
Group $S$ & Details of action \\
\midrule[1.5pt]
%$\SU_n(2^r)$ & $q_0=2$, $n\in\{4,5\}$, $M\geq \SU_n(q_0)$ \\
$\SU_n(q_0^r)$ & $n\in\{4,5\}$, $q_0\in\{2,3,4,5,7\}$, $r$ odd prime, $M\triangleright \PSU_n(q_0)$\\

%$\SU_{10}(2)$ & $M\geq \POmega^-_{10}(2)$ \\
$\SU_8(3)$ & $M\triangleright \POmega_8^{\pm}(3),\,\PSp_8(3)$ \\
%$\SU_8(2)$ & $M\geq \POmega^+_8(2)$ \\
%$\SU_6(q)$ & $q\in\{2,3\}$, $M\geq \POmega^+_6(q)$ \\
$\SU_8(2)$ & $M\triangleright \PSp_8(2)$ \\
$\SU_7(3)$ & $M\triangleright \POmega_7(3)$ \\
$\SU_6(2)$ & $M\triangleright \PSp_6(2)$ \\
%$\SU_8(3)$ & $M\geq \POmega_8^+(3)$ \\
%$\SU_8(3)$ & $M\geq \PSp_8(3)$ \\
$\SU_6(q)$ & $q\in\{3,5,7\}$, $M\triangleright \POmega_6^-(q)$ \\
$\SU_5(q)$ & $q\in\{3,5,7\}$, $M\triangleright \POmega_5(q)$ \\
$\SU_4(q)$ & $q\in\{2,3,4,5,7\}$, $M\triangleright \PSp_4(q)$ \\
$\SU_4(q)$ & $q\in\{3,5,7\}$, $M\triangleright \POmega^\pm_4(q)$ \\
\bottomrule[1.5pt]
\end{tabular}
\caption{$\C_5$ -- $\SU_n(q)$ -- Cases where a beautiful subset was not found.}\label{t: c5 sun}
\end{table}

\begin{lem}\label{l: c5 sun}
 In this case either $\Omega$ contains a beautiful subset or else $S$ is listed in Table~$\ref{t: c5 sun}$.
\end{lem}

\begin{proof}
Our proof splits into two cases, depending on whether $r$ is odd or even. Suppose, first, that $r$ is odd. In this case we let $\B=\{e_1,\dots, e_m, x, f_1\dots, f_m\}$ be a hyperbolic basis for $V$ (for $n$ even we do not need the element $x$); then $\varphi_0$, the restriction of $\varphi$ to the $\mathbb{F}_{q_0^2}$-span of $\mathcal{B}$, is unitary. 

First assume that $m\geq 3$, in which case we will use Lemma~\ref{aff}. We start by defining $A\cong \SL_{m-1}(q_0^2)$ to be the set of elements stabilizing the $\mathbb{F}_{q_0^2}$-subspaces
\[
 \langle e_1,\dots, e_{m-1}\rangle, \, \langle e_m\rangle, \,\, \langle f_1,\dots, f_{m-1}\rangle, \,\, \langle f_m\rangle\,\, \textrm{ (and $\langle x \rangle$ if $n$ is odd)},
\]
and acting on each as an element of determinant $1$.

Now let $g$ be the diagonal element with respect to $\B$ whose diagonal entries are $1$ except in the entries corresponding to $e_m$ and $f_m$, in which case the entries are $\mu$ and $\mu^{-1}$, respectively, where $\mu\in\Fq^*\setminus\Fqz^*$. 
Let $B_0\cong \SL_m(q_0^2)$ be the subgroup of $M$ stabilizing the subspaces \[
 \langle e_1,\dots, e_{m-1}, e_m\rangle\, \textrm{ and } \,\langle f_1,\dots, f_{m-1}, f_m\rangle\, \textrm{ (and $\langle x \rangle$ if $n$ is odd)},
\]
and acting on each as an element of determinant $1$. Let $B=B_0^g$, and observe that $A< B$ and $B\not\leq M$; thus Lemma~\ref{aff} implies that we have a subset $\Lambda\subseteq \Omega$ of cardinality $q_0^{2(m-1)}$ such that $S^\Lambda$ is a $2$-transitive group. This yields a beautiful subset unless $\Alt(q_0^{2(m-1)})$ is a section of $\SU_n(q)$; since we are assuming that $m\geq 3$, Lemma \ref{l: alt sections classical} eliminates the latter possibility, and we are done.

\smallskip

We are left with the possibility that $m=2$, in which case $n\in\{4,5\}$. We define two subgroups, $T$ and $U$, as follows. First, if $n=5$, then both subgroups fix the vector $x$. Then, in both cases, we fix an element $\zeta\in \Fq^*\setminus\mathbb{F}_{q_0}^*$ and describe the action of the two groups on the space $\langle e_1, e_2, f_1, f_2\rangle$ (writing elements with respect to the ordered basis $\{e_1, e_2, f_1, f_2\}$):
\begin{equation}\label{utdef}
%\begin{align*}
 T  =\left\{\begin{pmatrix} a & & & \\ & 1 & &  \\ & & a^{-1} &  \\ & & & 1\end{pmatrix} \mid a\in {\mathbb{F}_{q_0}^*}\right\},\;\;\;\;
 U  =\left\{\begin{pmatrix} 1 & \zeta x & & \\ & 1 & &  \\ & & 1 &  \\ & & -\zeta x & 1\end{pmatrix} \mid x\in {\mathbb{F}_{q_0}}\right\}.
%\end{align*}
\end{equation}
As usual, we can check that $T\le M$, $U\not \le M$, and $T$ normalizes $U$ and acts transitively on the set of non-identity elements of $U$. Then $\Delta=M^U$ is a subset of $\Omega$ on which $G_\Delta$ acts 2-transitively, and we have a beautiful subset unless $\Alt(q_0)$ is a section of $\SU_n(q)$, which by Lemma \ref{l: alt sections classical} can only occur if $q_0\le 7$, as listed in the first line of Table~\ref{t: c5 sun}.

\smallskip

Suppose, next, that $r=2$. In this case $\varphi_0$ is either symmetric (and $q$ is odd) or alternating (and $q$ can be either even or odd). These are the embeddings in the last two lines of Table \ref{c5poss}. In the case where $\varphi_0$ is symmetric and not of type $\Or^-$, we take $\mathcal{B}$, as before, to be a hyperbolic basis. 

For the other two cases, we adjust $\mathcal{B}$ slightly in order to  see more clearly the embeddings 
(namely,  $\SOr_n^-(q) < \SU_n(q)$ and $\Sp_n(q) < \SU_n(q)$).  
In the symplectic case we take $\mathcal{B}=\{e_1,\dots, e_m, f_1,\dots, f_m\}$ such that
\[
 \varphi(e_i,e_j)=\varphi(f_i,f_j)=0 \textrm{ and } \varphi(e_i, f_j)= \delta_{ij}\zeta,
\]
where $\zeta\in\mathbb{F}_{q^2}$ satisfies $\zeta^q=-\zeta$. It is easy to see that the restriction $\varphi_0$ of $\varphi$ to the $\bF_q$-span of $\mathcal{B}$ is symplectic; what is more the matrix for $\varphi_0$ written in block form with respect to 
$\mathcal{B}$ is
\[
 \zeta \begin{pmatrix}
       0 & I \\ -I & 0
      \end{pmatrix},
\]
a scalar multiple of the ``usual'' alternating matrix; hence ${\rm Isom}(\varphi_0)$ is a symplectic group $\Sp_n(q)$.

In the $\Or^-$ case we take $\mathcal{B}=\{e_1,\dots, e_m, f_1, \dots, f_m, x, y\}$ to be a hyperbolic basis for $\varphi_0$ over $\Fq$, and we simply define $\varphi$ to be the Hermitian form obtained by extending $\varphi_0$ to include scalars over $\mathbb{F}_{q^2}$.

In all cases, $m$ is the Witt index of $\varphi_0$, and we now proceed as in the first part of the proof. 
First assume that $m\geq 3$ and define $A\cong \SL_{m-1}(q)$ to be the set of elements in $M$ stabilizing the $\Fq$-subspaces
\[
 \langle e_1,\dots, e_{m-1}\rangle, \, \langle e_m\rangle, \, \langle f_1,\dots, f_{m-1}\rangle, \, \langle f_m\rangle \textrm{ (and $\langle x \rangle$ and $\langle y\rangle$ if needed)},
\]
and acting on each as an element of determinant $1$.

Now we define $g$ according to two cases:
\begin{enumerate}
 \item If $\varphi_0$ is orthogonal with $n$ odd, or of type $\Or^-$ with $n$ even, then let $g$ send
\[
e_m\mapsto \mu e_m,\;f_m \mapsto \mu^{-q}f_m,\; x \mapsto \mu^{q-1}x,
\]
and fix all other elements of $\mathcal{B}$, where $\mu$ is a primitive element of $\mathbb{F}_{q^2}$.
 \item If $\varphi_0$ is symplectic or of type $\Or^+$, then we let $\lambda, \mu \in\mathbb{F}_{q^2}$ with $\lambda$ primitive, and let $g$ act as
 \begin{align*}
 & \lambda I \textrm{ on } \langle e_1, \dots, e_{m-1}\rangle,\\
 & \,\,\mu \textrm{ on } \langle e_m \rangle, \\
 & \lambda^{-q}I \textrm{ on } \langle f_1,\dots, f_{m-1}\rangle,\\
 & \,\,\mu^{-q} \textrm{ on } \langle f_m\rangle.
 \end{align*}
We require that $\lambda^{-(q-1)(m-1)}=\mu^{q-1}$ to ensure that $\det(g)=1$, and we require that $\lambda \mu^{-1}\not\in\Fq$ (this condition ensures that $B=B_0^g\not\leq M$, see next paragraph). This can be done provided $q+1$ does not divide $m$ -- we defer this remaining case  for the moment.
\end{enumerate}

%Now let $g$ be the diagonal element with respect to $\B$ whose diagonal entries are $1$ except in the entries corresponding to $e_m$ and $f_m$, in which case the entries are $\mu$ and $\mu^{-q}$, respectively, where $\mu\in\Fq\setminus\Fqz$ satisfies $\mu^{q-1}=1$. 

Let $B_0\cong \SL_m(q)$ be the subgroup of $M$ stabilizing the subspaces 
\[
 \langle e_1,\dots, e_{m-1}, e_m\rangle \textrm{ and } \langle f_1,\dots, f_{m-1}, f_m\rangle \textrm{ (and $\langle x \rangle$ and $\langle y \rangle$ if needed)},
\]
and acting on each as an element of determinant $1$. Let $B=B_0^g$, and observe that $A< B$ and $B\not\leq M$; thus Lemma~\ref{aff} implies that we have a subset $\Lambda\subseteq \Omega$ of order $q^{m-1}$ such that $S_\Lambda$ is $2$-transitive. This yields a beautiful subset unless $\Alt(q^{m-1})$ is a section of $\SU_n(q)$. By Lemma \ref{l: alt sections classical}, the latter is only possible if $(n,q)$ is one of $(6,2)$, $(7,3)$, $(8,2)$, $(8,3)$ (recall that we are assuming $m\ge 3$ here), and $M$ is as in Table~\ref{t: c5 sun}.

%If $n$ is odd and $\varphi_0$ is symmetric, then $m=\frac{n-1}{2}$ and \cite[Proposition~5.3.7]{kl} implies that $(n,q)=(7,3)$. (This case is listed in the second line of Table~\ref{t: c5 sun}.) If $n$ is even and $\varphi_0$ is symmetric of type $\Or^-$, then $m=\frac{n-2}{2}$ and \cite[Proposition~5.3.7]{kl} implies that $(n,q)=(8,3)$. (This case is listed in the third line of Table~\ref{t: c5 sun}.)  If $n$ is even and $\varphi_0$ is symmetric of type $\Or^+$ or symplectic, then $m=\frac{n}{2}$ and \cite[Proposition~5.3.7]{kl} implies that $(n,q)\in\{(8,2), (6,2)\}$. Since $q$ is odd when $\varphi_0$ is symmetric, only the symplectic case is a problem. (These cases are listed in the fourth and in the fifth line of Table~\ref{t: c5 sun}.) 

\smallskip

Now let us deal with the deferred case: we suppose that $\varphi_0$ is symplectic or of type $\Or^+$ and $q+1$ divides $m=\frac{n}{2}$. Then we repeat the argument with $m$ redefined to equal $\frac{n-2}{2}$. Note, though, that for the argument to work we must have $(n-2)/2=m\geq 3$, that is, $n\ge 8$. We set $g$ to act as
 \begin{align*}
&  I \textrm{ on } \langle e_1, \dots, e_{m-1}\rangle,\\
 & \mu \textrm{ on } \langle e_m \rangle,\\
&\mu^{-1} \textrm{ on } \langle e_{m+1}\rangle, \\
&  I \textrm{ on } \langle f_1,\dots, f_{m-1}\rangle,\\ 
& \mu^{-q} \textrm{ on } \langle f_m\rangle,\\ 
&\mu^{q} \textrm{ on } \langle e_{m+1}\rangle.
 \end{align*}
 We obtain the same outcome: a beautiful subset of size $q^{m-1}$ unless $\Alt(q^{m-1}$) is a section of $\SU_n(q)$. The latter is only possible when $(n,q)=(8,3)$, which situation is listed in the second line of Table~\ref{t: c5 sun}.
 
If we are in the deferred case with $n<8$, then $n=6$. As $q+1$ divides $m+1=3$, we have $q=2$ and $S = \SU_6(2)$, and $\overline{S}$ is listed in Lemma~\ref{l: beautifulsetssmall}. This concludes the analysis of the deferred case.

%In this case we use the method that we used above when $r$ was odd and $m=2$. We let $\zeta$ be primitive in $\mathbb{F}_{q^2}$ and, using a hyperbolic basis $\{e_1, e_2, e_3, f_1, f_2, f_3\}$, define
%\begin{align*}
 %T & =\left\{\begin{pmatrix} A & & & \\ & 1 & &  \\ & & A^{-1} &  \\ & & & 1\end{pmatrix} \mid A\in\SL_2(q)\right\}; \\
 %U & =\left\{\begin{pmatrix} 1 & & \zeta x_1 & & &  \\ & 1 & \zeta x_2 & & &   \\ & &  1 & & &  \\ & & & 1 & &  \\ & &  & -\zeta^q x_1 & 1 & \\ & & & -\zeta^q x_2 & & 1 \end{pmatrix} \mid x_1, x_2\in\Fq\right\}.
%\end{align*}
%As before we get a beautiful set of size $q^2$ unless $\Alt(q^2)$ is isomorphic to a section of $\SU_6(q)$. Since $q$ is odd, this never happens, and we are done. This concludes the analysis of the deferred case.

\smallskip

Next, we consider the possibility that $m=2$ (defined, as it was originally, to be the Witt index of $\varphi_0$) in which case $n\in\{4,5,6\}$. In this case we proceed as at the start of this proof -- defining two subgroups $U$ and $T$ as in (\ref{utdef}) -- so that we obtain a beautiful subset unless $\Alt(q)$ is a section of $\SU_n(q)$. 
Using Lemma \ref{l: alt sections classical}, we conclude that $q\leq 7$ in the latter case, giving the cases listed in Table \ref{t: c5 sun}.

\smallskip

Finally, if $m=1$, then $n=4$ and $M$ is of type $\Or^-$. We shall work with the quasisimple group $S = \SU_4(q)$ with centre $Z$ of order $d = (4,q+1)$. In this group, the corresponding maximal subgroup, which we shall also denote as $M$, has structure $\mathrm{SO}_4^-(q).d$ (see \cite[Table 8.10]{bhr}). Let $X = \mathrm{SO}_3(q) < M$, and let $T = \{(\l, \l^{-1},1) : \l \in {\mathbb F}_q^*\}$ be a maximal torus of order $q-1$ in $X$ (matrices relative to a standard basis for the $\mathrm{O}_3$-space). Thus
\begin{equation}\label{seq}
T < X < M < S.
\end{equation}

We claim that  there is an $S$-conjugate $Y$ of $X$ such that $T < Y \not \le M$. 
Given the claim, we can  complete the proof as follows. Since $Y \cong \mathrm{SO}_3(q) \cong \PGL_2(q)$, there are subgroups $U_+, U_-$ of order $q$ in $Y$ such that $T$ acts by conjugation fixed-point-freely on both of them. These cannot both be contained in $M$, as $Y \not \le M$. Hence, say, $U_+ \cap M = 1$. Then in the usual way, $\D = M^{U_+}$ is a set of $q$ points on which $TU_+$ acts 2-transitively. For $q>7$, $\Alt(q)$ is not a section of $G$, and so $\D$ is a beautiful subset of $\O$, giving the conclusion; when $q\leq 7$, these case are listed in the last line of Table~\ref{t: c5 sun}.

So it remains to prove the claim. The claim would follow by applying Lemma \ref{factn} to the sequence (\ref{seq}) if we knew that $M$ controls fusion of $X$ in $S$, but this may not be the case: there are two conjugacy classes of subgroups $\mathrm{SO}_3(q)$ in  $\mathrm{SO}_4^-(q)$, with representatives $X_1,X_2$, say; then $X_1$ and $X_2$ are $S$-conjugate, but may or may not be $M$-conjugate (this depends on certain congruences of $q$ which we do not need to state here). Therefore, our argument is different. Define 
\[
\begin{array}{l}
\Lambda = \{ Y < S \;:\; T<Y,\,Y \hbox{ conjugate to }X \hbox{ in }S\}, \\
\Phi = \{Y \in \Lambda \,:\, Y<M\}.
\end{array}
\]
We shall compute the sizes of $\Lambda$ and $\Phi$, showing that $|\Lambda| > |\Phi|$, hence proving the claim. 

First observe that $N_S(T)$ acts on $\Lambda$. The action of $N_S(T)$ on $\Lambda$ is transitive; indeed,
\[
\begin{array}{ll}
Y \in \Lambda & \Rightarrow Y=X^s\;(s \in S) \\
                     & \Rightarrow T,T^{s^{-1}} < X \\
                     & \Rightarrow T^{s^{-1}} = T^x \hbox{ for some }x \in X \\
                     & \Rightarrow Y = X^{xs} \hbox{ with } xs \in N_S(T).
\end{array}
\]
Hence $|\Lambda| = |N_S(T) : N_S(T) \cap N_S(X)|$. Since $N_S(T)$ has a subgroup of order $|\GU_2(q).(q-1)/|Z|$, while the order of $N_S(T) \cap N_S(X)$ divides $2(q^2-1)/|Z|$, it follows that $|\Lambda|$ is divisible by $\frac{1}{2}q(q^2-1)$. 

In the same way we see that $N_M(T)$ has at most 2 orbits on $\Phi$. The orbit $\Phi_1$ of $X_1$ has size $|N_M(T):N_M(T)\cap N_M(X_1)|$. Since $|N_M(T)|$ divides $4(q^2-1)$ and $|N_M(T)\cap N_M(X_1)|$ is divisible by $2(q-1)$, it follows that $|\Phi_1|$ divides $2(q+1)$. If there is a second orbit $\Phi_2$, its size also divides $2(q+1)$. Hence $|\Phi|$ divides $4(q+1)$. 

As $q>7$, it is clear from the previous two paragraphs that $|\Lambda| > |\Phi|$. This yields the claim and completes the proof.
\end{proof}

\begin{comment}
\begin{table}\centering
\begin{tabular}{cl}
\toprule[1.5pt]
Group & Details of action \\
\midrule[1.5pt]
%$\SU_n(2^r)$ & $q_0=2$, $n\in\{4,5\}$, $M\geq \SU_n(q_0)$ \\
$\SU_n(q_0^r)$ & $n\in\{4,5\}$, $q_0\in\{2,3,4,5,7\}$, $M\geq \SU_n(q_0)$\\
\bottomrule[1.5pt]
\end{tabular}
\caption{$\C_5$ -- $\SU_n(q)$ -- Cases where the non-binarity is not yet proven.}\label{t: c5 sun0}
\end{table}

\begin{lem}\label{l: c5 sun0}
Let $S$ be one of the groups in Table~\ref{t: c5 sun}. In this case either, the action is not binary or $S$ is listed in Table~\ref{t: c5 sun0}.
\end{lem}

\begin{proof}
\end{proof}
\end{comment}

\begin{lem}\label{l: c5 sun 2}
 If $S$ is listed in Table~$\ref{t: c5 sun}$, then the action is not binary.
\end{lem}

\begin{proof}
Suppose, first, that $r=2$ -- this covers all but the first line of Table~\ref{t: c5 sun}. Now Lemma~\ref{l: beautifulsetssmall} deals with all the possible groups $S$ except for $\SU_8(3)$, $\SU_6(5)$, $\SU_6(7)$ and $\SU_5(7)$. We handled these cases with  {\tt magma} computations using  the permutation character method. 

Suppose, next, that $r\geq 3$, so we are in the first line of the table. Here $M = M_0 \times \la \phi \ra$, where $M_0$ has socle $\PSU_4(q_0)$ or $\PSU_5(q_0)$ with $q_0 \le 7$, and $\phi$ is either 1 or a field automorphism of $S$ of order $r$. 
We adopt the strategy of the proof of Lemma \ref{l: c5 sln 2}. Let $p$ be the characteristic of $\Fq$, let $Q \in Syl_p(M_0)$, and choose $g \in N_G(Q)\setminus M$. Then $Q \le M_0\cap M_0^g$, and so (by the well-known ``Tits lemma", or by computation) there is a parabolic subgroup $P$ of $M_0$ such that $UL' \le M_0\cap M_0^g \le P$, where $U$ is the unipotent radical and $L$ a Levi factor. 

Write $M_1 = M_0\cap M_0^g$, a core-free subgroup of $M_0$. A {\tt magma} computation shows that any transitive action of $M_0$ of $p'$-degree is not binary.  Hence, if $M\cap M^g = M_1 \times \la \s \ra$ with $\s = 1$ or $\phi$, then we obtain the conclusion as in the proof of Lemma \ref{l: c5 sln 2}. Otherwise, $M\cap M^g = M_1\la h\phi\ra$, where $h \in N_{M_0}(M_1)$. Analysing this normalizer, we see that we can take $h$ to be diagonal of order dividing $q_0^2-1$. Since $\phi \not \in M\cap M^g$, the order of $h$ must be divisible by $r$, and hence as $q_0\le 7$, we must have $r=3$ or 5. 
Hence  $M = M_0 \times r$, $r=3$ or 5, and $|M:M\cap M^g|$ is coprime to $p$.  Now a further {\tt magma} computation shows that any such action $(M,\,(M:M\cap M^g))$ (with ${\rm soc}(M_0) \not \le M\cap M^g$) is not binary.
\end{proof}

%Then $M$ equals either $M_0$ or $M_0\times r$ where $F^*(M_0)$ is a quasisimple cover of $\PSU_n(q_0)$. Now Lemma~\ref{l: odd degree Lie} implies that $M_0$ has no non-trivial transitive binary actions of odd degree.  Now, since $|G:M|$ is even, Lemma~\ref{l: point stabilizer} implies that, in the case $M=M_0$, the action of $G$ on $\Omega$ is not binary as required.
 
%For the case where $M=M_0\times r$, we conclude the proof using, word-for-word, the final three paragraphs of Lemma~\ref{l: c5 sun extra}.

\subsection{Case \texorpdfstring{$S=\Sp_n(q)$}{S=Sp(n,q)}}\label{s: c5 spn}

\begin{table}[!ht]\centering
\begin{tabular}{cl}
\toprule[1.5pt]
Group $S$ & Details of action \\
\midrule[1.5pt]
$\Sp_4(2^r)$ & $r$ prime, $M\triangleright \Sp_4(2)$. \\
%$\Sp_n(2^r)$ & $q_0=2$, $n\in\{6,8\}$, $r$ prime, $M\geq \Sp_n(2)$. \\
\bottomrule[1.5pt]
\end{tabular}
\caption{$\C_5$ -- $\Sp_n(q)$ -- Cases where a beautiful subset was not found.}\label{t: c5 spn}
\end{table}

\begin{lem}\label{l: c5 spn}
 In this case either $\Omega$ contains a beautiful subset or else $S$ is listed in Table~$\ref{t: c5 spn}$. In all cases the action of $G$ on $\Omega$ is not binary.
\end{lem}
\begin{proof}
Let $\mathcal{B}=\{e_1,\dots, e_k, f_k, \dots, f_1\}$ be a hyperbolic basis for $V$ with $k=\frac{n}{2}$. Let $M$ be the group stabilizing the $\mathbb{F}_{q_0}$-span of $\mathcal{B}$. 

First fix an element $\zeta\in \Fq\setminus\mathbb{F}_{q_0}$. We define two subgroups, writing elements with respect to $\mathcal{B}$:
\[
 T=\left\{\begin{pmatrix} 1 & & \\ & A &  \\ & & 1\end{pmatrix} \mid A\in \Sp_{n-2}(q_0)\right\}, \;\;\;\;
 U=\left\{\begin{pmatrix} 1 & \zeta x & \\  & I_{n-2} & \zeta x'^T \\ & & 1\end{pmatrix} \mid x\in {\mathbb{F}_{q_0}^{n-2}}\right\},
\]
where $x' = xJ$, $J$ being the matrix of the form relative to the basis $\mathcal{B}$, omitting $e_1,f_1$.
As usual, we can check that $T\le M$, $U\not \le M$, and $T$ normalizes $U$ and acts transitively on $U\setminus 1$.
Then $\Delta=M^U$ is a subset of $\Omega$ of size $q_0^{n-2}$ on which $G_\Delta$ acts 2-transitively, and we have a beautiful subset unless $\Alt(q_0^{n-2})$ is a section of $\Sp_n(q)$. By Lemma \ref{l: alt sections classical}, the latter is only possible if  $n=4$ and $q_0=2$, the case listed in Table~\ref{t: c5 spn}. Finally, the case in the table is dealt with exactly as in Lemma \ref{l: c5 sln 2}.
\end{proof}

\subsection{Case \texorpdfstring{$S$}{S} is orthogonal}\label{s: c5 omegan}

In this section we deal with all of the orthogonal families in one go. Recall from Section \ref{s: assumption} that $n\ge 7$, and also if $S=\Omega_8^+(q)$, then we are assuming that $G\leq{\rm P\Gamma}\Omega_8^+(q)$.

\begin{table}[!ht]\centering
\begin{tabular}{rl}
\toprule[1.5pt]
Group $S$ & Details of action \\
\midrule[1.5pt]
$\Omega_7(3^r)$ & $q_0=3$, $M\triangleright \Omega_7(3)$ \\
$\Omega_8^-(q_0^r)$ & $q_0\in\{2,3\}$, $r$ odd, $M\triangleright \POmega_8^-(q_0)$ \\
$\Omega_8^+(q_0^2)$ & $q_0\in\{2,3\}$, $M\triangleright \POmega_8^-(q_0)$ \\
$\Omega_8^+(2^r)$ & $q_0=2$, $M\triangleright\Omega_8^+(2)$ \\
$\Omega_{10}^-(2^r)$ & $q_0=2$, $r$ odd, $M\triangleright \Omega_{10}^-(2)$ \\
$\Omega_{10}^+(4)$ & $q_0=2$, $M\triangleright \Omega_{10}^-(2)$ \\

\bottomrule[1.5pt]
\end{tabular}
\caption{$\C_5$ -- $\Omega^\varepsilon_n(q)$ -- Cases where a beautiful subset was not found.}\label{t: c5 omegan}
\end{table}

\begin{lem}\label{l: c5 omegan}
 In this case either $\Omega$ contains a beautiful subset or else $S$ is listed in Table~$\ref{t: c5 omegan}$.
\end{lem}
\begin{proof}
Let $W$ be an $n$-dimensional orthogonal space over $\mathbb{F}_{q_0}$, with associated quadratic form $Q_W$, and let $\mathcal{B}$ be a hyperbolic basis for $W$. Define $V=W\otimes_{\mathbb{F}_{q_0}}\Fq$, with $Q_W$ extended to a quadratic form, $Q_V$, on $V$. This yields an embedding of ${\rm Isom}(Q_W)\leq {\rm Isom}(Q_V)$.  The embeddings listed in row 3 of Table \ref{c5poss} follow immediately. Note that, in the case where $\Omega_n^-(q_0)$ is embedded in $\Omega_n^+(q_0^2)$, $\mathcal{B}$ is not a hyperbolic basis for $V$.

We write $\mathcal{B}=\{e_1,\dots, e_k, f_1, \dots, f_k, x, y\}$ (omitting $x$ if $n$ is odd, and omitting $x$ and $y$ if $n$ is even and $\varepsilon=+$). We write $\mathcal{A}$ for the $\Fq$-span of the anisotropic vectors in $\mathcal{B}$; so $\dim(\mathcal{A})\in\{0,1,2\}$.

We define two subgroups:
\begin{align*}
 A &= \{g\in M \mid g \textrm{ stabilizes } \langle e_k\rangle, \langle f_k \rangle, \langle e_1,\dots, e_{k-1}\rangle \textrm{ and } \langle f_1,\dots, f_{k-1}\rangle; v^g=v \, \forall v\in \mathcal{A}\}; \\
 B_0 &= \{g\in M \mid g \textrm{ stabilizes } \langle e_1,\dots, e_{k}\rangle \textrm{ and } \langle f_1,\dots, f_{k}\rangle; v^g=v \, \forall v\in \mathcal{A}\}.
\end{align*}
Observe that $A\triangleright \SL_{k-1}(q_0)$ and $B_0\triangleright \SL_{k}(q_0)$. Now define $g\in G$ to 
send $e_k \mapsto \l e_k,\,f_k \mapsto \l^{-1}f_k$ and to fix the other elements of $\mathcal{B}$, where $\lambda\in \Fq\setminus\mathbb{F}_{q_0}$. Set $B=B_0^{g}$ and observe that $B$ contains $A$ but is not contained in $M$. Then Lemma~\ref{aff} implies that there is a subset $\Delta$ of $\Omega$ such that $|\Delta|=q_0^{k-1}$ and $G_\Delta$ acts 2-transitively on $\Delta$. Then $\Delta$ is a beautiful subset (and we are done) or else $\Alt(q_0^{k-1})$ is a section of $\Omega^\varepsilon_n(q)$. In the latter case, Lemma \ref{l: alt sections classical} implies that $S,M$ are as in Table~\ref{t: c5 omegan}. 
\end{proof}

%Recall that $2k\leq n \leq 2k+2$ and, since we assume that $n\geq 7$, we have $k\geq 3$ in general. If $k\geq 5$, then \cite[Proposition~5.3.7]{kl} implies that $\Alt(q_0^{k-1})$ is not a section of $\Omega^\varepsilon_n(q)$, and the result follows.

%If $k=4$, then \cite[Proposition~5.3.7]{kl} implies that $q_0=2$. This yields the lines where $M$ contains $\Omega_8^+(2)$ or $\Omega_{10}^-(2)$. If $k=3$, then \cite[Proposition~5.3.7]{kl} implies that $q_0\leq 3$. This yields the lines where $M$ contains $\Omega_7(3)$, $\Omega_8^-(2)$ or $\Omega_8^-(3)$.

\begin{lem}\label{l: c5 omegan 2}
 If $S$ is listed in Table~$\ref{t: c5 omegan}$, then the action is not binary.
\end{lem}
\begin{proof}
Suppose, first, that $r=2$. In this case, $S \in\{\Omega_8^+(4), \Omega_8^+(9), \Omega_{10}^+(4)\}$ and we confirm the result using {\tt magma}. 

Now suppose that $r\geq 3$. Then  $M = M_0 \times \la \phi \ra$, where $M_0$ has socle $\POmega_n^\e(q_0)$ with $n\le 10$ and $q_0 \le 3$, and $\phi$ is either 1 or a field automorphism $S$ of order $r$. We use the same argument as for Lemma \ref{l: c5 sun 2}. First, a {\tt magma} computation shows that any transitive action of $M_0$ of $p'$-degree is not binary. Then the argument shows that there exists $g\in G$ such that $(M,(M:M\cap M^g))$ is not binary, unless possibly $r$ divides $q_0^2-1$. As $q_0\le 3$, this forces $r=3$, and now a further {\tt magma} computation shows that any transitive $p'$-action of $M_0 \times 3$ is not binary, completing the proof. 
\end{proof}

\section{Family \texorpdfstring{$\C_6$}{C6}}\label{s: c6}

The members in the Aschbacher class $\mathcal{C}_6$ arise as local subgroups; more specifically they are normalizers of certain absolutely irreducible $r$-groups $R$ of \emph{symplectic-type}, where $r$ is a prime number with $r\ne p$ and $p$ is the characteristic of the defining field for the classical group. For $r$ odd, the $r$-group $R$ is extraspecial of exponent $r$, denoted by its order $r^{1+2a}$; and for $r=2$, either $R$ is an extraspecial group $2^{1+2a}_{\pm}$, or is a central product $4 \circ 2^{1+2a}$. These $r$-groups have absolutely irreducible embeddings in various classical groups of dimension $r^a$, and the normalizers of $R$ in these classical groups comprise the $\mathcal{C}_6$ subgroups; more precisely, if $G$ is an almost simple classical group and $\bar{R}$ is the projective image of $R$ in $G$, then $M=N_G(\bar{R})$ is in the $\mathcal{C}_6$ class. Full details are given in \cite[\S4.6]{kl}, and 
we give a list of the embeddings in Table \ref{c6poss}.

\begin{table}[ht!]
\[
\begin{array}{|c|c|c|}
\hline
\hbox{case} &  \hbox{normalizer} & \hbox{conditions} \\
\hline
{\rm L}^\e & r^{1+2a}.\Sp_{2a}(r) < \GL^\e_{r^a}(q) & r\hbox{ odd}, \,q\equiv \e \hbox{ mod }r  \\
{\rm L}^\e &  4\circ 2^{1+2a}.\Sp_{2a}(2) < \GL^\e_{2^a}(q) & q=p\equiv \e \hbox{ mod }4  \\
{\rm S} & 2_-^{1+2a}.\Or^-_{2a}(2) < \GSp_{2^a}(q) & q=p \\
{\rm O}^+ & 2_+^{1+2a}.\Or^+_{2a}(2) < \Or^+_{2^a}(q) & q=p \\
\hline
\end{array}
\]
\caption{Maximal subgroups in family $\C_6$} \label{c6poss}
\end{table}

In Line~1 of Table \ref{c6poss} there is a further condition on $q$: namely, let $e$ be the smallest positive integer such that $p^e \equiv 1 \hbox{ mod }r$. If $e$ is odd, then $\e = +$  and $q=p^e$; and if $e$ is even, then $\e = -$ and $q = p^{e/2}$.

The main result of this section is the following. The result will be proved in a series of lemmas.

\begin{prop}\label{p: c6}
 Suppose that $G$ is an almost simple group with socle $\bar S = \Cl_n(q)$, and assume that 
\begin{itemize}
\item[{\rm (i)}] $n\ge 3,4,4,7$ in cases $L,U,S,O$ respectively, and 
\item[{\rm (ii)}] $\Cl_n(q)$ is not one of the groups listed in Lemma $\ref{l: beautifulsetssmall}$.
\end{itemize}
Let $M$ be a maximal subgroup of $G$ in the family $\mathcal{C}_6$. Then the action of $G$ on $(G:M)$ is not binary.
\end{prop}

%In what follows we analyse in detail one case (when the symplectic-type group is an extraspecial group of ``plus type'') and then we discuss all other cases, omitting some of the computations.

\medskip

Our first lemma deals with the situation when $r$ is odd, in which case $S=\SL_{r^a}^\varepsilon(q)$ and $q$ is as given above, so that $\K=\mathbb{F}_{p^e}$. To prove the lemma we recall the set-up described in \cite[\S4.6]{kl} and establish some notation.

We let  $R:=\langle x_1,\ldots,x_{a},y_1,\ldots,y_a,z\rangle$ be the extraspecial $r$-group with center $\Zent R=\langle z\rangle$ and where, for every $i,j\in \{1,\ldots,a\}$,
$$x_i^r=y_j^r=[x_i,x_j]=[y_i,y_j]=1$$
and
\[[y_i,x_j]=
\begin{cases}
z&\textrm{when }j=i,\\
1&\textrm{otherwise}. 
\end{cases}
\] 
Clearly, $\bar{R}:=R/Z(R)$ is an elementary abelian $r$-group and one can see that $\bar{R}$ embeds naturally in $C_{\Aut(R)}(Z(R))$. We use an additive notation for the elements of $\bar{R}$ and a multiplicative notation for the elements of $R$ and observe that the commutator function 
\begin{align*}
 \mathfrak{B}: \bar{R}\times\bar{R}& \longrightarrow Z(R) \\
 (gZ(R), h(Z(R))&\longmapsto [g,h]
\end{align*}
defines a non-degenerate symplectic form on $\bar{R}$, which
endows $\bar{R}$ with the structure of a symplectic space over the field $\mathbb{F}_r$. Using the basis $(\bar{x}_1,\ldots,\bar{x}_a,\bar{y}_1,\ldots,\bar{y}_a)$ of $\bar{R}$, the symplectic form $ \mathfrak{B}$ on $\bar{R}$ is represented by the skew-symmetric matrix in block form
\[
J:=\begin{pmatrix}
0&-I\\
I&0
\end{pmatrix}.
\]
Observe that, under the natural projection $R\to\bar{R}$, the abelian subgroups of $R$ correspond to the totally isotropic subspaces of $\bar{R}$.
We let $X:=\langle x_1,\ldots,x_a\rangle$ and $Y:=\langle y_1,\ldots,y_a\rangle$. Observe that $X$ and $Y$ are elementary abelian subgroups of $R$ of cardinality $r^a$ and $\bar{X}$ and $\bar{Y}$ are maximal totally isotropic subspaces of $\bar{R}$. 

From the structure of $R$, it is clear that each element of $R$ can be written uniquely in the form
$$x_1^{\varepsilon_1}\cdots x_a^{\varepsilon_a}y_1^{\eta_1}\cdots y_a^{\eta_a}z^\nu,$$
where $\varepsilon_1,\ldots,\varepsilon_a,\eta_1,\ldots,\eta_a,\nu$ can be taken in $\mathbb{F}_r$. Given $v=\sum_{i=1}^a\varepsilon_i\bar{x_i}+\sum_{i=1}^a\eta_i\bar{y_i}$, an element in $\bar{R}$, we write $\underline{v}=x_1^{\varepsilon_1}\cdots x_a^{\varepsilon_a}y_1^{\eta_1}\cdots y_a^{\eta_a}$, a corresponding element in $R$.

Given a matrix $A\in \GL_{2a}(r)$ that preserves the symplectic form $\mathfrak{B}$, we find that the function
\begin{align*}
 \theta_A: R & \longrightarrow R \\
 \underline{x_i} & \longmapsto \underline{Ax_i} \\
 \underline{y_i} & \longmapsto \underline{Ay_i},
\end{align*}
defined on the generators of $R$ extends to an automorphism of $R$ that centralizes $Z(R)$. In this way we obtain an embedding $\overline{R}.\Sp_{2a}(r)$ in $C_{\Aut(R)}(Z(R))$ and now \cite[Table~4.6.A]{kl} asserts that in fact $C_{\Aut(R)}(Z(R))=\overline{R}.\Sp_{2a}(r)\cong r^{2a}.\Sp_{2a}(r)$.

Now \cite[p.151]{kl} describes an absolutely irreducible representation of $R$ over $\K$ of dimension $r^a$ that induces an embedding of $C_{\Aut(R)}(Z(R))$ into $\PGL_{r^a}(\K)$; this embedding yields the $\mathcal{C}_6$ subgroups for $r$ odd.

\begin{lem}\label{c6rodd} Let $r$ be an odd prime, let $G$ be almost simple with socle $\PSL^\e_{r^a}(q)$, and let $M=N_G(\bar R)$, where $R = r^{1+2a}$, as in line 1 of Table $\ref{c6poss}$. Then the action of $G$ on $(G:M)$ is not  binary.
\end{lem}

\begin{proof}
We adopt the above notation  and, since $M=N_G(\bar R)$, we may identify the set $\O = (G:M)$ with the set of conjugates $\{\bar R^g : g \in G\}$. Recall that $X$ is an elementary abelian group of order $r^a$, so $X\cong \mathbb{F}_r^a$. Moreover, since $\mathbb{F}_{r^a}\cong\mathbb{F}_r^a$ as $\F_r$-vector spaces, $\GL(X)$ contains an automorphism acting as scalar multiplication by a field element of order $r^a-1$. Let $B$ be  the matrix of this automorphism of $X$ with respect to the basis $\{x_1,\ldots,x_a\}$. Then the matrix
\[
A=\begin{pmatrix}
B&0\\
0&B^{-T}
\end{pmatrix} 
\]
preserves the bilinear form $\mathfrak{B}$. Thus $\theta_A$ is an automorphism of $R$ and hence $\theta_A$ determines an element of $\nor {\PGL(V)} {\bar{R}}$. In fact, from~\cite[Proposition~$4.6.5$]{kl}, $\theta_A\in M$ except when $a=1$, $r=3$ and $p=q$. We leave the case $a=1$, $r=3$ and $p=q$ aside for the time being; indeed let us assume, for now, that $n>5$.

Let $C=\langle \theta_A\rangle\in M$ and let $T=\bar{X}\rtimes C\le M$. By construction $T$ is a Frobenius group with Frobenius kernel $\bar{X}$ of cardinality $r^a$ and cyclic Frobenius complement $C$ of cardinality $r^a-1$. We claim that 
\begin{equation}\label{mtgc}
\exists g \in N_G(C) \hbox{ with }M\cap T^g=C.
\end{equation}
We argue by contradiction and we suppose that $M\cap T^g\ne C$, for every $g\in \nor G C$. Let $g \in N_G(C)$. Since $T^g=\bar{X}^g\rtimes C$, $M\cap T^g\ge C$ and $C$ acts transitively by conjugation on the non-identity elements of $\bar{X}^g$, we deduce $M\cap T^g=T^g$, that is, $T^g\le M$. Suppose that $\bar{X}^g\nleq \bar{R}$. Then $T^g$ is a Frobenius group and is isomorphic to a subgroup of $\Sp_{2a}(r)$. 
%From the structure of the matrix $A$ defining the generator $\theta_A$ of $C$, we deduce that $\bar{X}^{g}$ is a subgroup of $\Sp_{2a}(r)$. 
%$$\left\{\begin{pmatrix}I&U\\0&I\end{pmatrix}\mid U\in \mathrm{Mat}_{a\times a}(\mathbb{F}_q)\right\}.$$
Since $r$ is odd, $r^a-1$ is even and hence $A^{(r^a-1)/2}$ is the $-I$ matrix. In particular, $A^{(r^a-1)/2}$ centralizes $\bar{X}^g$. However, this is a contradiction because (since $T^g=\bar{X}^g\rtimes C$ is a Frobenius group) the action of $A^{(r^a-1)/2}$ by conjugation on $\bar{X}^g$ is fixed-point-free. This contradiction yields $\bar{X}^g\le\bar{R}$. 
Thus $\bar{X}^g$ is a totally isotropic subspace of $\bar{R}$ normalized by $C$. The only such totally isotropic subspaces are $\bar X$ and $\bar Y$, 
and hence  $\bar{X}^g=\bar{X}$ or $\bar{X}^g=\bar{Y}$. Now, consider $T':=\bar{Y}\rtimes C$. As $T$ and $T'$ are conjugate in $M$, we obtain $T'^g\cap M\ne C$ because $T^g\cap M\ne C$. Therefore, repeating the argument in this paragraph with the group $T$ replaced by $T'$, we deduce that $g$ normalizes $\bar{X}\bar{Y}=\bar{R}$. Thus $\nor G C\le \nor G {\bar{R}}=M$. Now we apply Proposition~\ref{centbd} and Lemma~\ref{l: cent rank} to establish the existence of an element $g\in G$ normalizing $C$ but not lying in $M$. Therefore our claim (\ref{mtgc}) is now proved.

Let $g\in \nor G C$ with $M\cap T^g=C$ and let $\Lambda:=\{\bar{R}^t\mid t\in T^g\}$. Then $\Lambda$ is a set of cardinality $|T^g:T^g\cap M|=r^a$ and $T^g$ induces on $\Lambda$ a permutation group isomorphic to a  Frobenius group of order $r^a(r^a-1)$. If $\Lambda$ is a beautiful subset, the conclusion follows by Lemma \ref{l: beautiful}. Otherwise, $\Alt(r^a-1)$ must be isomorphic to a section of $M$, hence to a section of $\Sp_{2a}(r)$. Since $n>5$, Lemma \ref{l: alt sections classical} rules out the latter possibility and we are done.

Consider next the case $a=1, r=5$. Here the embedding is $5^{1+2}.\Sp_2(5) < S = \SL^\e_5(q)$, where $q$ is minimal such that $q \equiv \e \hbox{ mod }5$. If $p=2$ then $S = \SU_5(4)$, which is covered by Lemma \ref{l: beautifulsetssmall}. So now assume $p>2$. It is well-known that the extension $R.\Sp_2(5)$ splits. Let $S_0 \cong \Sp_2(5)$ be a complement, and let $t \in S_0$ be the central involution. If $V$ is the natural 5-dimensional module for $S$, then $V_5\downarrow S_0 = V_3 \oplus V_2$, where $V_3=C_V(t)$, $V_2 = C_V(-t)$, of dimensions $3,2$ respectively. Hence there exists a diagonal element of $S$ of the form $\hat g = (\l I_3,\mu I_2)$ such that $\hat g \in C_S(S_0)\setminus N_S(R)$. Denoting by $g$ the projective image of $\hat g$, we then have $\bar RS_0 \cap (\bar RS_0)^g = S_0$. Since $\bar RS_0 = 5^2.\Sp_2(5)$ is a Frobenius group, this give a 2-transitive subset of size 25 in the usual way, and the conclusion follows.

Finally, the case where $a=1,r=3$ is dealt with in similar fashion. Here the embedding is $3^{1+2}.Q_8 < \mathrm{SL}_3^\e(p)$ with $p \equiv \e \hbox{ mod }3$ and $p\ne 2$. As $V\downarrow Q_8 = V_2\oplus V_1$, there exists $g \in C_G(Q_8)\setminus M$, and hence as above we obtain a subset $\Lambda$ of size 9 with $G^\Lambda$ $2$-transitive. This completes the proof. 
\end{proof}

\begin{lem}\label{c6r2} Let $r=2$, let $G$ be almost simple with socle $\PSL^\e_{2^a}(p)\,(a\ge 2)$, $\PSp_{2^a}(p)\,(a\ge 2)$ or 
$\POmega_{2^a}^+(p)\,(a\ge 3)$, and let $M=N_G(\bar R)$, where $R = 4\circ 2^{1+2a}$ or $2_{\pm}^{1+2a}$, as in Lines $2,3,4$ of Table $\ref{c6poss}$. Then the action of $G$ on $(G:M)$ is not  binary.
\end{lem}

\begin{proof} 
Since $M=N_G(\bar R)$, we may identify the set $\O = (G:M)$ with the set of conjugates $\{\bar R^g : g \in G\}$. 
Referring to \cite[\S4.6]{kl}, we have 
\[
R=\langle z\rangle\circ\langle x_1,y_1\rangle\circ\cdots \circ\langle x_a,y_a\rangle,
\]
 where 
\begin{itemize}
\item[] $z$ has order $4$ in types $L,U$, and has order $2$ in types $S,O^+$,
\item[] $\langle x_i,y_i\rangle\cong D_8$ for $i\ge 3$,
\item[] $\langle x_1,y_1\rangle \cong \la x_2,y_2\ra \cong Q_8$ in types $L,U,O^+$, 
\item[] $\langle x_1,y_1\rangle\cong Q_8$ and $\la x_2,y_2\ra \cong D_8$ in type $S$. 
\end{itemize}
The natural $2^a$-dimensional module $V$ has a tensor product decomposition $V = W_1\otimes \cdots \otimes W_a$ under the action of $R$, where each $W_i$ is an irreducible 2-dimensional module for $\la z,x_i,y_i\ra$.

%We highlight some facts which will be of some help later in the text:
%\begin{description}
%\item[type $L^\e:$]$|G:\mathrm{PSL}^\e_{2^a}(q)|$ divides $2$ and $G\cap \mathrm{PGL}^\e_{2^a}(q)=
%\mathrm{PSL}^\e_{2^a}(q)$,
%\item[type $S:$]$|G:\mathrm{PSp}_{2^a}(q)|$ divides $2$,
%\item[type $O^+:$]$|G:\mathrm{P}\Omega_{2^a}^+(q)|$ divides $4$ (in particular, $G$ does not contain a triality automorphism when $a=3$).
%\end{description}
%This information can be inferred from~\cite[Section~$4.6$]{kl}. (For instance, $G$ has no field automorphisms because $q$ is a prime number.) 
From \cite[4.6.6, 4.6.8, 4.6.9]{kl}, writing $\bar S = {\rm soc}(G)$, the precise structure of $M \cap \bar S$ is as follows:
\[
\begin{array}{|l|c|}
\hline
\hbox{case} & M\cap \bar S \\
\hline
L^\e & 2^4.\Alt(6), \hbox{ if }n=4,\,p\equiv \e 5\hbox{ mod }8 \\
        &  2^{2a}.\Sp_{2a}(2),\,\hbox{otherwise} \\
\hline
S & 2^{2a}.\Or^-_{2a}(2), \hbox{ if }p\equiv \pm 1\hbox{ mod }8 \\ 
   & 2^{2a}.\O^-_{2a}(2), \hbox{ if }p\equiv \pm 3\hbox{ mod }8 \\ 
\hline
O^+ & 2^{2a}.\Or^+_{2a}(2), \hbox{ if }p\equiv \pm 1\hbox{ mod }8 \\ 
   & 2^{2a}.\O^+_{2a}(2), \hbox{ if }p\equiv \pm 3\hbox{ mod }8 \\ 
\hline
\end{array}
\]

We now divide the proof in two parts (A) and (B), depending on whether $p\ge 7$ or $p<7$.

\vspace{4mm}
\textbf{(A)} Assume first that $p\ge 7$.
Define $W = W_2\otimes \cdots \otimes W_a$, and note that $\SL(W_1) = \Sp(W_1)$.  
The subgroup of $G$ preserving the tensor decomposition $V = W_1\otimes W$ is the normalizer of the image of $\SL(W_1) \otimes \Cl(W)$, where $\Cl(W)$ is $\SL^\e(W)$, $\O^+(W)$ or $\Sp(W)$ for case $L^\e$, $S$ or $O^+$ repectively.

Again we use the bar notation for the natural homomorphism to the projective version of our classical group. As before, $M$ preserves on $\bar{R} $ a non-degenerate symplectic form $\mathfrak{B}$ in types $L$ and $U$ defined as above. In types $O^+$ and $S$, the group $M$ not only preserves $\mathfrak{B}$ but also a particular quadratic form $\mathfrak{q}:\bar{R}\to \mathbb{F}_r$ that polarizes to $\mathfrak{B}$. Rather than defining $\mathfrak{q}$ explicitly we remark only that, for $i=1,\dots, a$, the 2-spaces corresponding to $\langle x_i, y_i\rangle \cong Q_8$ (resp. $\langle x_i, y_i\rangle\cong D_8$) are of type $\Or_2^-(2)$ (resp. $\Or_2^+(2)$). Then $\bar R_1$ is a non-degenerate 2-space in $\bar R$, and is of type $\Or_2^-(2)$ in cases $O^+$ and $S$.

Define $\bar R_0 = \prod_{i=2}^a \bar R_i = \bar R_1^\perp$, and $W = W_2\otimes \cdots \otimes W_a$. 
By \cite[4.4.3]{kl}, $N_G(\bar R_0)$ preserves the tensor decomposition $V = W_1\otimes W$ and contains the image of $\SL(W_1) \otimes 1_W$. Define a subset $\D$ of $\O = \{\bar R^g : g \in G\}$ by
\[
\D = \{\bar R^g : g \in N_G(\bar R_0)\}.
\]
Let $X$ be the image of the group induced on $W_1$ by $N_G(\bar R_0)$. Then $X\cong \PSL_2(p)$ or $\PGL_2(p)$, and 
\[
\D = \{\bar R_0 \times \bar R_1^x : x \in X\}.
\]
Also, from the structure of $M\cap \bar S$, we see that 
\[
N_X(\bar R_1) \cong 2^2.\Sp_2(2) \cong 2^2.\Or_2^-(2) \cong \Sym(4).
\]
Since the intersection of all the subgroups in $\D$ is $\bar R_0$, we have $G_\D = N_G(\bar R_0)$. Hence the action of $G_\D$ on $\D$ is isomorphic to the action of $X$ on the cosets of $\Sym(4)$. 

Recall that we are assuming $p\ge 7$. Hence $\Sym(4)$ is a maximal subgroup of either $X$ or $X'$, and \cite{ghs_binary} (together with Lemma \ref{l: subgroup}) shows that $(X,\,(X:\Sym(4))$ is not binary. 
Thus there is an integer $k\ge 3$, and $k$-tuples $I=(I_1,\ldots ,I_k)$, $J=(J_1,\ldots ,J_k) \in \D^k$ such that 
$I\stb{2}J$ and $I\nstb{k} J$ with respect to the action of $G_\D$. Since $I\stb{2} J$ we can assume that $I_1=J_1$ and $I_2=J_2$; we also assume that there are no repeated entries in $I$ (and hence there are none in $J$ either).

We need to show that $I\nstb{k} J$ with respect to the action of $G$. Suppose $I^g = J$ for some $g \in G$. Observe that for each $j$ we have $I_j = \bar R_0 \times \bar R_1^{x_j}$ for some $x_j \in G_\Delta$. We claim that 
\[
\bigcap_{j=1}^k \bar R_1^{x_j} = 1.
\]
\textbf{Proof of claim:} Suppose otherwise. Since $\bar R_1$ is a Klein 4-group we must have $\bigcap_{j=1}^k \bar R_1^{x_j} = \langle g_I \rangle$ where $g_I$ is an involution. Now for each $j$ we have $J_j = \bar R_0 \times \bar R_1^{y_j}$ for some $y_j \in G_\Delta$. Since $I\stb{k} J$ with respect to the action of $G$, we conclude that $\bigcap_{j=1}^k \bar R_1^{y_j} = \langle g_J \rangle$ for some involution $g_J$. Observe
that, for distinct $i$ and $j$, we have
\[
 \bar R_1^{x_i} \cap \bar R_1^{x_j}=\langle g_I\rangle \textrm{ and }
 \bar R_1^{y_i} \cap \bar R_1^{y_j}=\langle g_J\rangle.
\]
Since $I_1=J_1$ and $I_2=J_2$ we conclude that $g_I=g_J$. 
%This implies that if $h\in G_\Delta$ such that, for distinct $i$ and $j$,
%\[ I_i^h=J_i \textrm{ and } I_j^h=J_j,
%\]
%then $h$ normalizes $\bar R_0 \times \langle g_I\rangle$. 
Consider what this means for the action of $X$ on $(X:\Sym(4))$: we can think of this action as being the conjugation action of $X$ on a class of Klein 4-subgroups. The tuples $I$ and $J$ correspond to $k$-tuples, $I_X$ and $J_X$, whose entries are Klein 4-subgroups of $X$ all of which contain an involution $g_X$. What is more $I_X\stb{2} J_X$ and $I_X\nstb{k} J_X$ with respect to the action of $X$. Now if $i$ and $j$ are distinct in $\{1,\dots, k\}$ and
\[
 I_i^h=J_i \textrm{ and } I_j^h=J_j \textrm{ for some } h\in X,
\]
then $h\in C_X(g_X)$. The group $C_X(g_X)$ is a maximal dihedral subgroup of $X$ and we define $Y=C_X(g_X)/\langle g_X\rangle$ which is also a dihedral group. The tuples $I_X$ and $J_X$ correspond to $k$-tuples, $I_Y$ and $J_Y$, whose entries are involutions in $Y$. Since $I_X\stb{2} J_X$ with respect to $X$, we have $I_Y\stb{2} J_Y$ with respect to $Y$. But this action is binary (see the discussion of Family~3a at the start of \S\ref{s: examples}) and so $I_Y\stb{k} J_Y$ with respect to $Y$. But this implies that $I_X\stb{k} J_X$ with respect to $X$ which is a contradiction. Hence the claim is proved.

It now follows that
\[
\bigcap_{j=1}^k I_j = \bar R_0,
\]
and similarly $\bigcap_{j=1}^k J_j = \bar R_0$. Therefore $g \in N_G(\bar R_0) = G_\D$, which is a contradiction. This completes the proof under the assumption that $p\ge 7$.

\vspace{4mm}
\textbf{(B)} Now assume that $p<7$, so that $p=3$ or 5. First note that the cases where $a=2$ (in which case $L = \PSL_4^\e(p)$ or 
$\PSp_4(p)$) are covered by Lemma \ref{l: beautifulsetssmall}. So we may assume that $a\ge 3$.

Define $R_a = R_1\times R_3$ and $R_b = \prod_{i\ne 1,3}R_i$, and let $W_a = W_1\otimes W_3$, $W_b = \bigotimes_{i\ne 1,3} W_i$. Then in case $S$ or $O^+$, the subgroup of $G$ preserving the tensor decomposition $W_a\otimes W_b$ is the normalizer of the image of $\Sp(W_a) \otimes \Cl(W_b)$, where $\Cl(W_b)$ is orthogonal or symplectic, respectively. The normalizer $N_G(\bar R_b)$ preserves this tensor decomposition. 

Define a subset $\D$ of $\O$ by $\D = \{\bar R^x : x \in N_G(\bar R_b)\}$. 
Let $X$ be the image of the group induced on $W_a$ by $N_G(\bar R_b)$. Then $X$ has socle $\PSL_4^\e (p)$ or 
$\PSp_4(p)$, and $\D = \{\bar R_b \times \bar R_a^x : x \in X\}$. Also, from the structure of $M\cap L$, we see that 
\[
N_X(\bar R_a) \cong \left\{ \begin{array}{l} 2^4.\Sp_4(2), \hbox{ case }L^\e \\
 2^4.\Or_4^-(2), \hbox{ cases }S,O^+.
\end{array}
\right.
\]
As above, the action of $G_\D$ on $\D$ is isomorphic to the action of $X$ on the cosets of $N_X(\bar R_a)$. Using {\tt magma}, we check in all possible cases that this action is not binary, and that there exist $k$-tuples $I=(I_1,\ldots ,I_k)$, $J=(J_1,\ldots ,J_k) \in \D^k$ such that 
$I\stb{2}J$ and $I\nstb{k} J$ with respect to the action of $G_\D$, and also such that $\bigcap_{j=1}^k I_j = 
\bigcap_{j=1}^k J_j = \bar R_b$. Now we see exactly as in the argument at the end of part (A) that $I\nstb{k} J$ with respect to the action of $G$. Hence $G$ is not binary, and the proof is complete. 
\end{proof}

\section{Family \texorpdfstring{$\C_7$}{C7}}\label{s: c7}

In this case $M$ is the stabilizer of a tensor decomposition of $V$, in much the same way as was detailed at the start of \S\ref{s: c4}. In this case, though, $M$ stabilizes a tensor product of two or more subspaces of the same dimension: we write $V=W_1\otimes \cdots \otimes W_t$, and $m:=\dim(W_1)=\cdots =\dim(W_t)$. Observe that $\dim(V)=n=m^t$. If $G = \PGL_n(q)$, the stabilizer $M$ has the structure $\PGL_m(q) \wr \Sym(t)$, where $\Sym(t)$ permutes the tensor factors.

In the case where $S$ is not $\SL_n(q)$, i.e. $S$ preserves a non-degenerate form $\varphi$, the spaces $W_1\dots, W_t$ are mutually similar spaces equipped with non-degenerate forms $\varphi_1,\dots, \varphi_t$, and  
\[
 \varphi = \begin{cases}
        Q(\varphi_1\otimes\cdots \otimes \varphi_{t}), & \textrm{if $q$ is even and $\varphi_1,\dots, \varphi_{t}$ are non-degenerate alternating};\\
        \varphi_1\otimes \cdots \otimes \varphi_{t}, & \textrm{ otherwise}.
               \end{cases}
\]
The definition of the quadratic form $Q(\varphi_1\otimes\cdots \otimes \varphi_{t})$ is given on \cite[p.127]{kl}: it is the unique non-degenerate quadratic form $Q$ such that
\begin{enumerate}
 \item $Q(w_1\otimes\cdots\otimes w_t)=0$ for all $w_i \in W_i$, and
 \item the polarization of $Q$ is equal to $\varphi_1\otimes\cdots\otimes\varphi_t$.
\end{enumerate}

Again the stabilizer $M$ has a wreath product structure. It is convenient to set $\varphi$ to be the zero map when $S=\SL_n(q)$. We have given a list of all the $\C_7$ embeddings in Table \ref{c7poss}, taken from \cite[\S4.7]{kl}, where the precise structures of the $\C_7$ subgroups can be found.

\begin{table}[ht!]
\[
\begin{array}{|c|c|c|}
\hline
\hbox{case} &  \hbox{type} & \hbox{conditions} \\
\hline
{\rm L}^\e & \PGL^\e_m(q) \wr \Sym(t) & m \ge 3  \\
{\rm S} & \PSp_m(q) \wr \Sym(t) & qt \hbox{ odd}  \\
{\rm O}^+ & \PO^{\pm}_m(q) \wr \Sym(t) & q \hbox{ odd}  \\
{\rm O}^+ & \PSp_m(q) \wr \Sym(t) & qt \hbox{ even}  \\
{\rm O} & \PO_m(q) \wr \Sym(t) & qm \hbox{ odd}  \\
\hline
\end{array}
\]
\caption{Maximal subgroups in family $\C_7$} \label{c7poss}
\end{table}

Note that \cite[p. 156]{kl} details a further restriction on the subgroup $M$, namely that the relevant subgroup $\Cl_m(q)$ must be quasisimple. For instance, in the $\Or^+$ case that is listed on Line~3 of the table, we require that $\Omega_m^\pm(q)$ is quasisimple; thus, for this case, $m\geq 6$ or $(m,\varepsilon)=(4,-)$. In general we have that the socle of $M$ is $(\Cl_m(q))^t$.

The main result of this section is the following. The result will be proved in a series of lemmas.

\begin{prop}\label{p: c7}
 Suppose that $G$ is an almost simple group with socle $\bar S = \Cl_n(q)$, and assume that 
\begin{itemize}
\item[{\rm (i)}] $n\ge 3,4,4,7$ in cases $L,U,S,O$ respectively, and 
\item[{\rm (ii)}] $\Cl_n(q)$ is not one of the groups listed in Lemma $\ref{l: beautifulsetssmall}$.
\end{itemize}
Let $M$ be a maximal subgroup of $G$ in the family $\mathcal{C}_7$. Then the action of $G$ on $(G:M)$ is not binary.
\end{prop}

The following lemma will be used in various special cases.

\begin{lem}\label{l: last para}
Let $t_0$ be an integer, at least $2$; in the case where $\varphi_1,\dots, \varphi_t$ are non-degenerate alternating bilinear forms, we require that $q$ is odd and that $m^{t_0}\geq 8$.% and that $t_0$ and $t$ have the same parity.

Let $V_0$ be the $m^{t_0}$-dimensional formed space that is a tensor product of $t_0$ formed spaces all similar to $(W_1,\varphi_1)$. Let $\bar{S_0}=X_{m^{t_0}}(q)$, where $X\in\{\PSL, \PSU, \PSp,\POmega^\varepsilon\}$, be the simple group associated with $V_0$. Consider all pairs $(G_0, M_0)$ where $G_0$ is an almost simple group with socle $\bar{S_0}$ and $M_0$ is the subgroup of $G_0$ from class $\mathcal{C}_7$ associated with the tensor product decomposition. Let $\Omega_0=(G_0:M_0)$.

\begin{enumerate}
 \item Suppose that $t_0=2$ and that, for all such pairs $(G_0, M_0)$, the action of $G_0$ on $\Omega_0$ is not binary. Then the action of $G$ on $(G:M)$ is not binary.
\item Suppose that $t_0>2$ and that, for all such pairs $(G_0, M_0)$ we can find an integer $k\geq 3$ and tuples $(I_1,\dots, I_k), (J_1,\dots, J_k)\in \Omega_0^k$ such that
\begin{enumerate}
 \item $(I_1,\dots, I_k)\stb{2} (J_1,\dots, J_k)$;
 \item $(I_1,\dots, I_k)\nstb{k} (J_1,\dots, J_k)$;
 \item there is no group isomorphic to $X_m(q)$ that is a normal subgroup of each of the socles of $(G_0)_{I_1},\cdots,(G_0)_{I_k}$.
\end{enumerate}
Then the action of $G$ on $(G:M)$ is not binary.
\end{enumerate}
\end{lem}

Note, first, that the family in which $\bar{S_0}$ lies (i.e. the particular choice of $X$ from $\{\PSL, \PSU, \PSp,\POmega^\varepsilon\}$) is determined by the type of $W_1$ and the value of $t_0$ (see Table~\ref{c7poss}). Then \cite[Tables 3.5.H and 3.5.I]{kl} (and, when $n\in\{8,9\}$, \cite{bhr}) imply that, since $M$ is maximal in $G$, we know that $M_0$ will be a maximal $\C_7$-subgroup of $G_0$ unless $\varphi_1,\dots, \varphi_t$ are non-degenerate alternating and $q$ is even -- but this case is explicitly ruled out by our hypotheses in the statement of this lemma. %This implies, third, that we can identify the elements $I_1,\dots, I_k\in \Omega_0^k$ with conjugates of $M_0$ in $G_0$.% -- the intersection in item 2(c) refers to the intersection of conjugates obtained via this identification.

Note, second, that in most cases the groups $\bar{S}$ and $\bar{S_0}$ will lie in the same family, i.e. if $\bar{S_0}=X_{m^{t_0}}(q)$, then $\bar{S}=X_{m^t}(q)$. The exception to this occurs when $\varphi_1,\dots, \varphi_t$ are non-degenerate alternating bilinear forms and $t$ and $t_0$ have different parity (see Table~\ref{c7poss}).

\begin{proof}
We noted above that the socle of $M$ is $L^t$ where $L=\Cl_m(q)$, a non-abelian simple group. Write $\Gamma$ for the set of semilinear similarities of $\varphi$, so $S=F^*(\Gamma)$, and let $\iota$ be the inverse transpose map. Write $\underline{G}$ (resp. $\underline{M}$) for the preimage of $G$ (resp. $M$) in
$\Gamma$ (or in $\Gamma:\langle \iota\rangle$ if $S=\SL_n(q)$). We write $\mathcal{D}$ for the decomposition preserved by $\underline{M}$:
\[
(V,\varphi)= (W_1, \varphi_1) \otimes\cdots\otimes (W_t, \varphi_t).
\]
%Now \cite[\S4.7]{kl} asserts that $\underline{M}=\underline{G}\cap \Gamma_\mathcal{D}$ where
%\[
% \Gamma_\mathcal{D}:=\Gamma_{(\mathcal{D})}J = (\Delta_1\otimes \cdots \otimes \Delta_t)(\langle \phi_\mathcal{D}\rangle \times J);
%\]
%here $\Delta_i$ is the set of similarities for $(W_i, \varphi_i)$, $\phi_\mathcal{D}$ is a particular field automorphism and $J$ is a particular group isomorphic to $\Sym(t)$.

We define $U= W_1\otimes \cdots \otimes W_{t_0}$ and $\varphi_U = \varphi_1\otimes\cdots\otimes\varphi_{t_0}.$ Let $\underline{G_U}$ be the stabilizer in $\underline{G}$ of the decomposition
\[
 \mathcal{D}_U: (V,\varphi)= (U, \varphi_U)\otimes (W_{t_0+1}, \varphi_{t_0+1}) \otimes\cdots\otimes (W_t, \varphi_t).
\]
%Again, this means that $\underline{G_0}=\underline{G}\cap \Gamma_{\mathcal{D}_0}$ where
%%\[
%\Gamma_{\mathcal{D}_0}:=\Gamma_{(\mathcal{D}_0)}J_0 = (\Delta_0\otimes\Delta_{t_0+1} \cdots \otimes \Delta_t)(\langle \phi_{\mathcal{D}_0}\rangle \times J_0)
%\]
%where $\Delta_0$ is the set of similarities for $(V_0, \varphi_0)$; $\phi_{\mathcal{D}_0}$ is a particular field automorphism and $J_0$ is a particular group isomorphic to $\Sym(t-t_0)$. Examining \cite[\S\S4.4 and 4.7]{kl} we see that we may take $\phi_{\mathcal{D}_0}=\phi_{\mathcal{D}}$ and $J_0< J$.
Now we consider the action of ${G_U}$, the projective image of $\underline{G_U}$ in $G$, on $({G}:{M})$. In particular we can consider the action on the set of cosets ${M}.{G_U}$; the action on this set is isomorphic to the action of ${G_U}$ on $({G_U}: {M_U})$ where ${M_U}=M\cap G_U$.

Clearly the kernel of this action contains the image in $G_U$ of
\[
 \underline{G_U} \cap (\{1\}\otimes\Delta_{t_0+1} \cdots \otimes \Delta_t)J_U
\]
where $J_U\cong \Sym(t-t_0)$. The quotient of $G_U$ by the kernel of this action is an almost simple group $G_0$ with socle $X_{m^{t_0}}(q)$ and the stabilizer in $G_0$ of a point is a subgroup $M_0$ of $G_0$ from class $\mathcal{C}_7$ associated with a decomposition of the associated $m^{t_0}$-dimensional formed space into a tensor product of $t_0$ formed spaces all similar to $(W_1,\varphi_1)$. 

By assumption we know that this action is not binary. Let $I,J$ be elements of $({G_U}: M_U)^k$ for some integer $k\geq 3$ such that $I\stb{2}J$ and $I\nstb{k} J$ with respect to the action of ${G_U}$. Identify the entries of $I$ and $J$ with the corresponding elements of $({G}:{M})$. We can think of the entries of $I$ and $J$ as conjugates of ${M}$ in $G$; now, if $t_0>2$, then assume that $I$ has the property listed at 2(c). It is sufficient to prove that $I\nstb{k} J$ with respect to the action of $G$. 

It is at this point that we use the fact that the socle of $M$ is $L^t$ where $L=\Cl_m(q)$. Define
\[
 K=\underbrace{\{1\}\times \cdots \times \{1\}}_{t_0} \times L^{t-t_0}.
\]
Observe that $K$ is a normal subgroup of the socle of $M$. By construction, $K$ is a normal subgroup of each of the socles of $I_1,\dots, I_k$ and $J_1,\dots, J_k$. Suppose that $I^g=J$. Then $\langle K, K^g\rangle$ is a normal subgroup of the socle of $J_i$ for $i=1,\dots, k$ and we see that $\langle K, K^g\rangle$ is isomorphic to $L^{t-s}$ for some $0\leq s \leq t_0$. We can relabel so that
\[
 \langle K, K^g\rangle=\underbrace{\{1\}\times \cdots \times \{1\}}_{s} \times L^{t-s}
\]

If $s=t_0$, then $K=K^g$. But $N_G(K)=G_U$ and so $g\in G_U$ which is a contradiction. Thus $s<t_0$. If $s=0$, then $J_1,\dots, J_k$ have the same socle and so $J_1=J_2=\cdots=J_k$. Since $I\stb{2} J$ we obtain that $I_1=I_2=\cdots=I_k$ and so $I\stb{k} J$ with respect to the action of $G_U$, which is a contradiction. Thus $0<s < t_0$.

Now we refer to \cite[Lemma 4.4.3]{kl} from which we deduce that $C_{S}(\langle K, K^g\rangle)$ must be a subgroup of $\mathrm{GL}(W_1\otimes \cdots \otimes W_{s})\times 1^{t-s}$ and, of course, must preserve the form $\varphi$. If $s=1$, then this means that there is a unique conjugate of $M$ whose socle contains $\langle K,K^g\rangle$ and, again, $J_1=J_2=\cdots=J_k$ which is a contradiction as before. This proves the result when $t_0=2$.

If $t_0>2$, then the property listed at 2(c) implies that the only conjugate of $K$ that is normal in each of the socles of $I_1,\dots, I_k$ is $K$ itself. Hence, since $J=I^g$ and since $K$ is normal in each of the socles of $J_1,\dots, J_k$, we conclude that the only conjugate of $K$ that is normal in each of the socles of $J_1,\dots, J_k$ is $K$ itself. But this means that $K^g=K$ and, again, the fact that $N_G(K)=G_U$ implies that $g\in G_U$, a contradiction.
\end{proof}

\subsection{Case \texorpdfstring{$S=\SL_n(q)$}{S=SL(n,q)}}

In this case~\cite[Table 3.5.A]{kl} allows us to assume that $m\geq 3$.

\begin{table}\centering
\begin{tabular}{cl}
\toprule[1.5pt]
Group & Details of action \\
\midrule[1.5pt]
$\SL_{3^t}(2)$ & $m=3$: $M\triangleright \PSL_3(2)^t$ \\
$\SL_{4^t}(2)$ & $m=4$: $M\triangleright \PSL_4(2)^t$  \\
\bottomrule[1.5pt]
\end{tabular}
\caption{$\C_7$ -- $\SL_n(q)$ -- Cases where a beautiful subset was not found.}\label{t: c7 sln}
\end{table}

\begin{lem}\label{l: c7 sln}
 In this case either $\Omega$ contains a beautiful subset or else the action is listed in Table~$\ref{t: c7 sln}$.
 \end{lem}
\begin{proof}
We write $W_1=\cdots=W_t$, and let $\mathcal{B}_1=\{e_1,\dots, e_{m}\}$ be a basis for $W_1$. Then
\[
 \mathcal{B} = \{e_{i_1}\otimes \cdots \otimes e_{i_t}\} \mid 1\leq i_1,\dots, i_t \leq m\}
\]
is a basis for $V$, and we take $M$ to be the stabilizer the associated tensor decomposition, so that $M\cap \bar S = (\PGL_m(q) \wr \Sym(t))\cap \bar S$. 

First assume that $q\geq 7$, and let $T_1$ be a split maximal torus in $\SL_{m}(q)$ that is diagonal with respect to $\B_1$; then $T=T_1\otimes 1\otimes \cdots \otimes 1$ is a subgroup of (the preimage of) $M$. Define $U$ to be the set of elements in $S$ for which there exists $\alpha\in \Fq$ such that
\[
 e_1\otimes\cdots\otimes e_1 \mapsto e_1\otimes\cdots\otimes e_1 +\alpha e_2\otimes e_1\otimes\cdots\otimes e_1,
\]
and which fixes all elements $e_{i_1}\otimes \cdots \otimes e_{i_t}\in\B$ for which $i_j>1$ for some $j$. Observe that $U$ is not a subgroup of $M$, that $T$ normalizes $U$ and that $T$ acts transitively on the non-identity elements of $U$. We define $\Delta=M^U$, a subset of $\Omega$ of size $q$ and observe that $U\rtimes T$ acts 2-transitively on $\Delta$. On the other hand, 
\begin{equation}\label{eq:pabloretarted}
 M_{(\Delta)} \geq C_M(U) \geq \Big[\GL_{m-2}(q)\circ (\underbrace{\GL_{m-1}(q)\circ\cdots\circ \GL_{m-1}(q)}_{t-1}).\Sym(t-1)\Big] \cap \bar S.
\end{equation}
Assuming that $\Delta$ is not beautiful, $G^\Delta$ induces at least $\Alt(q)$ on $\Delta$, hence the point stabilizer $M^{\Delta}$ induces at least $\Alt(q-1)$ on $\Delta$. However,  $M_{(\Delta)}$ contains $C_M(U)$, which contains the group on the right hand side of~\ref{eq:pabloretarted}. It follows that any simple section of $M^\Delta$ is a section of $\GL_2(q)$. Since $q\geq 7$, it follows from Lemma \ref{l: alt sections classical}  that $\Alt(q-1)$ is not a section of $M^\Delta$. This implies that $\Delta$ is a beautiful subset and Lemma~\ref{l: beautiful} yields the result.

Next assume that $q\in\{3,4,5\}$, and let $T_2$ be a maximal torus in $\GL_m(q)$ that preserves the decomposition
\[
 \langle e_1\rangle \oplus \langle e_2,e_3\rangle \oplus \langle e_4\rangle \oplus \cdots \oplus \langle e_m \rangle,
\]
and that acts on $\langle e_2,e_3\rangle$ as a Singer cycle; let $T_1$ be as above. Then $T_2\otimes T_1 \otimes 1\otimes \cdots \otimes 1$ is a subgroup of $M$. Define $U$ to be the set of elements in $S$, for which there exists $\alpha,\beta\in \Fq$ such that
\[
 e_1\otimes\cdots\otimes e_1 \mapsto e_1\otimes\cdots\otimes e_1 +\alpha e_2\otimes e_1\otimes\cdots\otimes e_1+\beta e_3\otimes e_1\otimes\cdots\otimes e_1,
\]
and which fixes all other elements of $\B$. Observe that $U$ is not a subgroup of $M$, that $T$ normalizes $U$ and that $T$ acts transitively on the non-identity elements of $U$. We define $\Delta=M^U$, a subset of $\Omega$ of size $q^2$ and observe that $U\rtimes T$ acts 2-transitively on $\Delta$. 
%On the other hand it is clear that
%\[
% M_{(\Delta)} \geq \Big[\GL_{m-3}(q)\circ (\underbrace{\GL_{m-1}(q)\circ\cdots\circ \GL_{m-1}(q)}_{t-1}).\Sym(t-1)\Big] \cap S.
%\]
Arguing as above, we see that any non-abelian simple section of $M^\Delta$ is isomorphic to a  section of $\GL_3(q)$; hence,  since $q\geq 3$, $\Alt(q^2-1)$ is not a section of $M^\Delta$. This implies that $\Delta$ is a beautiful subset and Lemma~\ref{l: beautiful} yields the result.

When $q=2$, we assume that $m\geq 5$ and we proceed similarly: we construct a beautiful subset of size $q^4=16$, using the same method but this time we choose a maximal torus $T_4$ in $\GL_m(q)$ preserving the decomposition
\[
 \langle e_1\rangle \oplus \langle e_2,e_3,e_4,e_5\rangle \oplus \langle e_6\rangle \oplus \cdots \oplus \langle e_m \rangle,
\]
and acting on $\langle e_2,e_3, e_4,e_5\rangle$ as a Singer cycle. At the final stage, we use the fact that $\Alt(q^4-1)=\Alt(15)$ is not a section of $\GL_5(2)$ to conclude that the set we have constructed is indeed beautiful.
\end{proof}

\begin{lem}\label{l: c7 sln 2}
 If the action is listed in Table~$\ref{t: c7 sln}$, then the action is not binary.
\end{lem}
\begin{proof}
We begin with the case when $t=2$ for which we use {\tt magma}.  Let $S$ be either $\SL_{16}(2)$ or $\SL_9(2)$ and let $M$ be a maximal subgroup of $G$ in the Aschbacher class $\mathcal{C}_7$. With {\tt magma}, we have first computed a Sylow $2$-subgroup of $M$, say $Q$. Then, we have computed $P=N_G(N_G(Q))$ and we have found an element $g\in P$, with the property that 
\begin{itemize}
\item $|M:M\cap M^g|=294$ when $G=\mathrm{SL}_9(2)$,
\item $|M:M\cap M^g|=588$ when $G=\Aut(\mathrm{SL}_9(2))$,
\item $|M:M\cap M^g|=11025$ when $G\in \{\mathrm{SL}_{16}(2),\mathrm{Aut}(\mathrm{SL}_{16}(2))\}$.
\end{itemize}
 In particular, in the faithful primitive action of $S$ on the right cosets $\Omega$ of $M$, a point stabilizer has a suborbit $\Delta$ with the property that the action of $M$ on $\Delta$ is permutation isomorphic to the action of $M$ on the right cosets of $M\cap M^g$.
We have constructed the permutation representation of $M$ under consideration (that is, on the right cosets of $M\cap M^g$) and we have verified that in this action $(M:M\cap M^g)$ contains a beautiful subset of cardinality $7$ when $S=\mathrm{SL}_9(2)$ and cardinality $5$ when $S=\mathrm{SL}_{16}(2)$. This immediately yields that the action of $M$ on $\Delta$ is not binary and hence the action of $S$ on $\Omega$ is also not binary.

If $t>2$, then we use the result for $t=2$ combined with Lemma~\ref{l: last para}. 
\end{proof}

\subsection{Case \texorpdfstring{$S=\SU_n(q)$}{S=SU(n,q)}}

In this case \cite[Table 3.5.B]{kl} allows us to assume that $m\geq 3$, and that $(q,m)\neq (2,3)$.

\begin{table}\centering
\begin{tabular}{cl}
\toprule[1.5pt]
Group & Details of action \\
\midrule[1.5pt]
$\SU_{3^t}(q)$ & $q\in\{3,4,5\}$, $m=3$: $M\triangleright \PSU_3(q)^t$ \\
$\SU_{m^t}(2)$ & $m\in\{4,5\}$: $M\triangleright \PSU_m(2)^t$ \\
\bottomrule[1.5pt]
\end{tabular}
\caption{$\C_7$ -- $\SU_n(q)$ -- Cases where a beautiful subset was not found.}\label{t: c7 sun}
\end{table}

\begin{lem}\label{l: c7 sun}
 In this case either $\Omega$ contains a beautiful subset or else the action is listed in Table~$\ref{t: c7 sun}$.
\end{lem}

\begin{proof}
Our method here will be very reminiscent of that used in Lemma~\ref{l: c4 sun}. We start by writing $W=W_1=\cdots=W_t$, and letting $\mathcal{B}_1=\{u_1,\dots, v_1,\dots, x\}$ be a hyperbolic basis for $W_1$ (omitting $x$ if $m$ is even). Taking pure tensors we obtain a hyperbolic basis, $\mathcal{B}$, for $V$, and we let $M$ be the stabilizer of the associated tensor decomposition. Then $M\cap \bar S = \PGU_m(q) \wr \Sym(t)$.

If $q\geq 7$, then we consider subgroups $U$ and $T$ of $\GU(\langle u_1, x, v_1\rangle)$ defined as per \eqref{e: u} and \eqref{e: t}. (In what follows, for $i\in\Z^+$, we write $x^i$ to mean $\underbrace{x\otimes \cdots \otimes x}_i$.)

As in Lemma~\ref{l: c4 sun} we now split into two cases. If $q$ is odd, then we take $U_0$ to be the subgroup of $U$ obtained by requiring that $b\in \Fq$ and that $c=\frac12 b^2$; we define an isomorphic group in $S$: $U_1$ consists of those elements for which there exists $b\in \Fq$ such that
\begin{align*}
 u_1\otimes x^{t-1} &\mapsto u_1\otimes x^{t-1} + bx^t - \frac12b^2 v_1\otimes x^{t-1},\\ 
x^t& \mapsto x^t - bv_1\otimes x^{t-1},\\ 
v_1\otimes x^{t-1} &\mapsto v_1\otimes x^{t-1},
\end{align*}
and all elements of $\langle u_1\otimes x^{t-1}, x^t, v_1\otimes x^{t-1}\rangle^\perp$ are fixed. Then $U_1$ is a subgroup of order $q$ that is not contained in $M$. Now we take $T_0$ to be the subgroup of $T$ obtained by requiring that $r\in \Fq$ and let $T_1=T\circ 1 \circ\cdots\circ 1$, a group of order $q-1$ that normalizes $U_1$ and acts transitively on the set of non-trivial elements in $U_1$.

If $q$ is even, the set-up is slightly different but follows the procedure in Lemma~\ref{l: c4 sun} as above. In both cases, identifying $\Omega$ with conjugates of $M$ we set $\Lambda=M^{U_1}\subset\Omega$, and see that $S^\Lambda$ acts 2-transitively upon $\Lambda$, a set of size $q$. The usual argument shows that that any non-abelian simple section in $M^\Lambda$ is isomorphic to a section of $\GU_3(q)$. By Lemma \ref{l: alt sections classical}, for $q\geq 7$, we conclude that $\Alt(q-1)$ is not a secion of $M^\Lambda$ and so $\Lambda$ is a beautiful subset, and Lemma~\ref{l: beautiful} implies that the action is not binary.

For $q\in \{3,4,5\}$ we diverge from the argument given in Lemma~\ref{l: c4 sun}, and we assume that $m\geq 4$ (the first line of Table \ref{t: c7 sun} covers $m=3$). We proceed as for $q\geq 7$ but we use the existence of a Frobenius group in $\GU_4(q)$ this time. We let $W_0:= \langle u_1, u_2, v_2, v_1\rangle$ be a non-degenerate 4-subspace of $W$, and consider the group:
\[
 U\rtimes T = \left\langle \begin{pmatrix}
                      1 & a & & \\ & 1 & & \\ & & 1 & -a^q \\ & & & 1
                     \end{pmatrix}, \begin{pmatrix}
                      r &  & & \\ & 1 & & \\  & & 1 & \\ & & & r^{-q}
                     \end{pmatrix} \mid a, r\in \K, r\neq 0
    \right\rangle.
\]
Now we define $T_0$ to be the subgroup of $\GU(W)$ which stabilizes $W_0$, whose action on $W_0$ is equal to the action of $T$, and which fixes $W_0^\perp$ point-wise. Then we define 
\[T_1=\Big(T_0\otimes T_0\otimes 1 \otimes \cdots \otimes 1\Big)\cap S.\]
On the other hand we let $x$ be an element of $W_1$ for which $(x,x)=1$, and we define
\[
 V_0:= \langle 
 u_1\otimes u_2 \otimes x^{t-2}, 
 u_2\otimes u_2 \otimes x^{t-2},  v_2\otimes v_2 \otimes x^{t-2},  
 v_1\otimes v_2 \otimes x^{t-2}\rangle.
\]
Observe that $V_0$ is a non-degenerate 4-subspace of $V$; indeed there exists an isomorphism between $W_0$ and $V_0$ which maps the listed ordered basis for $W_0$ to that of $V_0$. We can, therefore, define $U_1$ to be the subgroup of $\SU(W)$ which stabilizes $V_0$,  whose action on $V_0$ is equal to the action of $U$, and which fixes $V_0^\perp$ point-wise. 

Observe that $U_1$ is of order $q^2$, is contained in $S$ but not in $M$, and is normalized by $T_1$. Observe, moreover, that $T_1$ acts transitively on the set of non-identity elements of $U_1$. 
Defining $\Lambda=M^{U_1}\subseteq \Omega$, we therefore conclude that $S_\Lambda$ acts 2-transitively on  $\Lambda$. 
% A straightforward calculation allows us to assert that 
%\[
% M_{(\Lambda)} \geq \Big( \GU_{m-4}(q)\circ \GU_{m-4}(q) \circ \Big[\GU_{m-1}(q)\wr \Sym(t-2)\Big]\Big)\cap S,
%\]
The usual argument shows that any simple section in $M^\Lambda$ is necessarily isomorphic to a section of $\GU_4(q)$. However Lemma \ref{l: alt sections classical}  implies that for $q\ge 3$,  $\Alt(q^2-1)$ is not a section of $M^\Lambda$, and so $\Lambda$ is beautiful and we are done as before.

Finally, for $q=2$, we assume that $m\geq 6$ (the cases where $m\leq 5$ are listed in Table \ref{t: c7 sun}). We proceed as in the previous paragraph, using the 2-transitive group constructed in Lemma~\ref{l: c4 sun} for the $q=2$ case. As there, the fact that $\Alt(15)$ is not a section of $\GU_6(2)$ allows us to construct a beautiful subset.
\end{proof}

\begin{lem}\label{l: c7 sun 2}
 If the action is listed in Table~$\ref{t: c7 sun}$, then the action is not binary.
\end{lem}
\begin{proof}
Our method is entirely analogous to that used in Lemma~\ref{l: c7 sln 2}. We begin with the case when $t=2$. 

Suppose, first, that $M\triangleright \PSU_3(q)^2$ and $S=\SU_9(q)$ with $q\in\{3,4,5\}$. Let $\{e, f, x\}$ and $\{v, w, y\}$ be hyperbolic bases for a Hermitian space of dimension $3$ (where $(e,f)$ and $(v,w)$ are hyperbolic pairs and $x$ and $y$ are anisotropic); taking tensor products we obtain a hyperbolic basis $\mathcal{B}$ for a $9$-dimensional Hermitian space, and we obtain our embedding of $M$ in $S$. We choose an order for $\mathcal{B}$ as follows:
\[
 e\otimes v, e\otimes w, e\otimes y, x\otimes v, x\otimes w, f\otimes y, f\otimes v, f\otimes w, x\otimes y.
\]
Define $T$ to be the subgroup of $M$, whose elements when written with respect to $\mathcal{B}$ consist of all diagonal matrices
\[
 {\rm diag}[a^q, \, a^{-1}, \, 1, \, a, \, a^{-q}, \, 1, \, a^q, \, a^{-1}, \, a^{1-q}],
\]
with $a\in \mathbb{F}_{q^2}^*$. Observe that $T$ normalizes and acts fixed-point-freely upon the group $U$, whose elements fix all elements of $\mathcal{B}$ except $e\otimes y$ and $x\otimes w$, and for which there exists $a\in \mathbb{F}_{q^2}$ such that
\begin{align*}
 e\otimes y&\mapsto e\otimes y+ a x\otimes v,\\
 x\otimes w&\mapsto x\otimes w-a^q f\otimes y.
\end{align*}
Since $U$ is not in $M$, we obtain, in the usual way, a set $\Lambda$ of size $q^2$, on which $S_\Lambda$ acts 2-transitively. Now observe that an alternating section, $\Alt(t)$ of $M$ satisfies $t\leq 7$, and so we conclude that $M^\Lambda$ does not have a section $\Alt(q^2-1)$. We conclude that the set $\Lambda$ is a beautiful subset and Lemma~\ref{l: beautiful} yields the result.

Next suppose that $M\triangleright \PSU_m(2)^2$ and $S=\SU_{m^2}(2)$ with $m\in\{4,5\}$. In both cases we take a pair of hyperbolic bases $\{e_1, e_2, f_1, f_2\}$ and $\{v_1, v_2, w_1, w_2\}$ (adding in an anisotropic element when $m=5$), and we take tensor products to obtain a hyperbolic basis, $\mathcal{B}$, for an $m^2$-dimensional Hermitian space. Now $M$ has a subgroup isomorphic to $A=\SL_2(4)$ that preserves the subspaces $\langle e_1\otimes v_1, e_1 \otimes v_2\rangle$ and $\langle f_1\otimes w_1, f_1\otimes w_2\rangle$ and fixes all other elements of $\mathcal{B}$. 

What is more $A$ lies inside a subgroup $X\cong \SL_3(4)\leq S$ that preserves the subspace $\langle e_1\otimes v_1, e_1 \otimes v_2, f_1\otimes v_1\rangle$, and note that $X\not\leq M$. Then Lemma~\ref{aff} implies that there is a subset $\Delta$ of $\Omega$ of size $16$ on which $S^\Delta$ acts 2-transitively. Since $\SU_5(2)$ does not contain a section isomorphic to $\Alt(16)$, we obtain that $\Delta$ is a beautiful subset and, as before, Lemma~\ref{l: beautiful} yields the result.

Now for $t>2$ we use Lemma~\ref{l: last para} and the fact that the result is proved for $t=2$.
\end{proof}

\subsection{Case \texorpdfstring{$S=\Sp_n(q)$}{S=Sp(n,q)}}

In this case \cite[Table~3.5.C]{kl} implies that $qt$ is odd, that $m$ is even, that $t\geq 3$, and that $(m,q)\neq (2,3)$.

\begin{table}\centering
\begin{tabular}{cl}
\toprule[1.5pt]
Group & Details of action \\
\midrule[1.5pt]
$\Sp_{2^t}(5)$ & $m=2$, $t$ odd: $M\triangleright \PSp_2(5)^t$ \\
%$\Sp_{4^t}(q)$ & $m=4$, $q\in\{3,5\}$, $t$ odd: $M\geq \underbrace{\Sp_4(q)\circ \cdots \circ \Sp_4(q)}_t$. \\
\bottomrule[1.5pt]
\end{tabular}
\caption{$\C_7$ -- $\Sp_n(q)$ -- Cases where a beautiful subset was not found.}\label{t: c7 spn}
\end{table}

\begin{lem}\label{l: c7 spn}
 In this case either $\Omega$ contains a beautiful subset or else the action is listed in Table~$\ref{t: c7 spn}$.
\end{lem}
\begin{proof}
We start by writing $W=W_1=\cdots=W_t$, and letting $\mathcal{B}_1=\{u_1,\dots u_{m/2}, v_{m/2},\dots v_1\}$ be a hyperbolic basis for $W_1$. Taking pure tensors we obtain a hyperbolic basis, $\mathcal{B}$, for $V$, and we let $M$ be the subgroup of $G$ that stabilizes the associated tensor decomposition. Then $M\cap \bar S = (\mathrm{PGSp}_m(q) \wr \Sym(t)) \cap \bar S$.

First suppose that $m\geq 4$. We define two subgroups of $\Sp_m(q)$:
\begin{align*}
 U &:= \left\{\begin{pmatrix}
               1 & a_1 & \cdots & a_{m-2} & \\ & 1 & & & a_{m-2} \\ & & \ddots & & \vdots \\ & & & 1 & -a_1 \\ & & & & 1
              \end{pmatrix} \,\middle\vert\, a_1,\dots a_{m-2}\in \Fq \right\}, \\
T & := \left\{ \begin{pmatrix}
                1 & & \\ & A & \\ & & 1 
               \end{pmatrix} \,\middle\vert\, A \in \Sp_{m-2}(q) \right\}.
\end{align*}

Our construction is inspired by the observation that $T$ normalizes $U$ and acts transitively on the set of non-trivial elements of $U$. We define $T_1=T\circ 1\circ\cdots \circ 1 < S$ and we define the group $U_1$ to be the set of elements for which there exist $a_1,\dots, a_{m-2}$ such that
\begin{align*}
 u_1^t &\mapsto u_1^t + a_1u_2\otimes u_1^{t-1} + \cdots +a_{(m-2)/2}u_{m/2}\otimes u_1^{t-1} + a_{m/2} v_{m/2} \otimes u_1^{t-1} + \cdots+ a_{m-2} v_2 \otimes u_1^{t-1}, \\
 v_i \otimes v_1^{t-1} &\mapsto v_i \otimes v_1^{t-1} - a_{i-1} v_1^t, \\
 u_i \otimes v_1^{t-1} &\mapsto u_i \otimes v_1^{t-1} + a_{m-i} v_1^t,
\end{align*}
for $i=2,\dots, \frac{m}{2}$, and all other elements of $\mathcal{B}$ are fixed. Observe that $U_1$ is of order $q^{m-2}$, is contained in $S$ but not in $M$, and is normalized by $T_1$. Furthermore $T_1$ acts transitively on the set of non-identity elements of $U_1$. Defining $\Lambda=M^{U_1}\subseteq \Omega$, we conclude that $S_\Lambda$ acts 2-transitively on the elements of $\Lambda$. 
 %A straightforward calculation allows us to assert that 
%\[
% M_{(\Lambda)} \geq \Sp_2(q)\wr \Sym(t-1).
%\]
The usual argument shows that any simple section of $M^\Lambda$ is necessarily isomorphic to a section of $\Sp_{m-2}(q)$. We conclude that either $\Lambda$ is beautiful or else $\Sp_{m-2}(q)$ contains a section isomorphic to $\Alt(q^{m-2}-1)$, which is impossible by Lemma \ref{l: alt sections classical} (recall that $q$ is odd).

We are left with the situation where $m=2$, in which case we use the fact that a Borel subgroup of  $\GSp_2(q)=\GL_2(q)$ has a 2-transitive action on $q$ points. We use the basis $\mathcal{B}_1=\{u_1, v_1\}$, and consider the group:
\begin{equation}\label{e: b2}
 B=U\rtimes T = \left\langle \begin{pmatrix}
                      1 & a \\ & 1
                     \end{pmatrix}, \begin{pmatrix}
                      r & \\ & s
                     \end{pmatrix} \mid a, r,s\in \Fq, r\neq 0\neq s
    \right\rangle.
\end{equation}
Then we define
\[T_1=(T\circ \cdots\circ T)\cap S.\]
Next we define the group $U_1$ in $S$ to be the set of elements for which there exists $b\in \Fq$ such that
\begin{align*}
u_1^t& \mapsto u_1^t+ b v_1\otimes u_1^{t-1}, \\ 
u_1\otimes v_1^{t-1}& \mapsto u_1\otimes v_1^{t-1} + b v_1^t, 
\end{align*}
and all elements of $\langle u_1^t, u_1\otimes v_1^{t-1}\rangle^\perp$ are fixed. Observe that $U_1$ is of order $q$, is contained in $S$ but not in $M$, and is normalized by $T_1$. We can use \cite[Proposition~4.7.4]{kl} to check that $T_1$ acts transitively on the set of non-identity elements of $U_1$. Defining $\Lambda=M^{U_1}\subseteq \Omega$, we conclude that $S_\Lambda$ acts 2-transitively on the elements of $\Lambda$. 
As usual, either $\Lambda$ is beautiful or else $\Sp_2(q)$ contains a section isomorphic to $\Alt(q-1)$. This yields the result for $q\geq 7$. We are left with the case listed in Table \ref{t: c7 spn} (recall that $(m,q)=(2,3)$ is excluded).
 \end{proof}

 \begin{lem}\label{l: c7 spn 2}
 If the action is listed in Table~$\ref{t: c7 spn}$, then the action is not binary.
\end{lem}
\begin{proof}
If $t=3$, then $S=\Sp_8(5)$ and we use {\tt magma} to verify the result. If $t>3$, then we use the result for $t=3$ combined with Lemma~\ref{l: last para}. 
Our application of Lemma~\ref{l: last para} requires that we check the property listed at 2(c): suppose that $G_0$ has socle $\bar{S_0}\cong \Sp_8(5)$, that $k, I_1,\dots, I_k, J_1,\dots, J_k$ are as given in the lemma and that they satisfy the properties listed at 2(a) and 2(b) -- our \magma calculations confirm that such cosets do exist. Suppose that the property listed at 2(c) is not satisfied, in which case there exists a group $K\cong \PSp_2(5)$ that is a normal subgroup of the socles of $(G_0)_{I_1},\dots, (G_0)_{I_k}$. Then \cite[Lemma~4.4.3]{kl} implies that $C_{\bar{S_0}}(K)$ is isomorphic to a subgroup of $\Or_4^+(5)$, which has socle isomorphic to $\PSp_2(5)\times \PSp_2(5)$. Since the socles of $(G_0)_{I_1},\dots, (G_0)_{I_k}$ are isomorphic to $\PSp_2(5)\times \PSp_2(5)\times\PSp_2(5)$, we conclude that the socles of $(G_0)_{I_1},\dots, (G_0)_{I_k}$ are all equal and hence $I_1=\cdots=I_k$. Then the property listed at 2(a) implies that $J_1=\cdots=J_k$ and now the property listed at 2(b) yields a contradiction. We conclude, therefore, that the property listed at 2(c) is satisfied.
\end{proof}

\subsection{Case \texorpdfstring{$S=\Omega_n(q)$, $n$ odd}{S=Omega(n,q)}}

In this case note that $m$ and $q$ are odd, and \cite[Table~3.5.D]{kl} implies that $(m,q)\neq (3,3)$.

\begin{table}\centering
\begin{tabular}{rl}
\toprule[1.5pt]
Group & Details of action \\
\midrule[1.5pt]
$\Omega_{3^t}(5)$ & $m=3$: $M\triangleright \Omega_3(5)^t$ \\
\bottomrule[1.5pt]
\end{tabular}
\caption{$\C_7$ -- $\Omega_n(q)$ -- Cases where a beautiful subset was not found.}\label{t: c7 omegan}
\end{table}

\begin{lem}\label{l: c7 omegan}
 In this case either $\Omega$ contains a beautiful subset or else the action is listed in Table~$\ref{t: c7 omegan}$.
\end{lem} 

 \begin{proof}
We start by writing $W=W_1=\cdots=W_t$, and letting $\mathcal{B}_1=\{u_1,\dots, v_1,\dots, x\}$ be a hyperbolic basis for $W_1$. Taking pure tensors we obtain a hyperbolic basis, $\mathcal{B}$, for $V$, and we let $M$ be the subgroup of $G$ that stabilizes the associated tensor decomposition. Then $M\cap \bar S = (\O_m(q) \wr \Sym(t)) \cap \bar S$. 

First suppose that $q\geq 7$; here our method is very similar to that used in Lemma~\ref{l: c7 sun}. We define analogues of the groups defined at \eqref{e: u} and \eqref{e: t}: we consider subgroups of $\SOr(\langle u_1, x, v_1\rangle)$ consisting of elements of the form
\begin{align}
U &=\left\{\begin{pmatrix}
 1 & b & -\frac{b^2}{2} \\ & 1 & -b \\ & & 1                                                                                                                                                                                                                                                                               \end{pmatrix} \mid b\in \Fq  \right\}; \label{e: uu} \\
T &=\left\{\begin{pmatrix}
 r &  &  \\ & 1 &  \\ & & r^{-1}                                                                                                                                                                                                                                                                               \end{pmatrix} \mid r \in \Fq \textrm{ with } r\neq 0 \right\}. \label{e: tt}
\end{align}
Then $U\rtimes T$ is a Borel subgroup of $\SOr_3(q)$, and is a Frobenius group. We let $T_1=(T\circ \cdots \circ T)\cap S$, a subgroup of $M$; we let $U_1$ be the group consisting of elements for which there exists $b\in \Fq$ such that
\begin{align*}
 u_1\otimes x^{t-1}& \mapsto u_1\otimes x^{t-1} + bx^t - \frac12b^2 v_1\otimes x^{t-1}, \\
 x^t& \mapsto x^t - bv_1\otimes x^{t-1},\\
 v_1\otimes x^{t-1}& \mapsto v_1\otimes x^{t-1},
\end{align*}
and all elements of $\langle u_1\otimes x^{t-1}, x^t, v_1\otimes x^{t-1}\rangle^\perp$ are fixed. Then $U_1$ is a subgroup of order $q$ that is not contained in $M$. Now $T_1$ normalizes $U_1$ and acts transitively on the set of non-trivial elements in $U_1$.

In the same way as before we obtain a beautiful subset, provided $\Alt(q-1)$ is not a section of $\SOr_3(q)$; this is true for $q\geq 7$ by Lemma~\ref{l: alt sections classical}.

Suppose that $q\in\{3,5\}$ and $m\ge 5$. We define

\[
 T = \left\{ g\circ \underbrace{1\circ\cdots \circ 1}_{t-1} \mid
 \begin{array}{l}
  g \in \Omega_m(q), \, x^g=x, \\
  g \textrm{ stabilizes both } \langle u_1,\dots, u_{(m-1)/2}\rangle \textrm{ and } \langle v_1, \dots, v_{(m-1)/2} \rangle
 \end{array}\right\}.
\]
Now define $U$ to be the set of elements $g$ such that, for $i=1,\dots, \frac{m-1}{2}$, there exist $a_i \in \Fq$ such that
\begin{align*}
 x\otimes u_1^{t-1} &\mapsto x\otimes u_1^{t-1} + a_1 u_1^t + a_2u_2\otimes u_1^{t-1}+\cdots + a_{(m-1)/2}u_{(m-1)/2}\otimes u_1^{t-1} \\
 v_i\otimes v_1^{t-1} &\mapsto v_i\otimes v_1^{t-1} - a_i x\otimes v_1^{t-1},
\end{align*}
and all other members of $\mathcal{B}$ are fixed. In exactly the same way as before, we see that $T$ normalizes $U$, that $T$ acts transitively on the set of non-trivial elements of $U$, that $T$ is in $M$, and that $U$ is not contained in $M$. Then, identifying $\Omega$ with conjugates of $M$, and setting $\Delta=M^U$, we conclude that $\Delta$ is a set of size $q^{(m-1)/2}$ whose set-wise stabilizer acts 2-transitively.

As usual, either $\Delta$ is a beautiful subset and we are done, or $M^\Delta$ has a section $\Alt(q^{(m-1)/2}-1)$, in which case $\SOr_m(q)$ also has such a section. This is not the case by Lemma \ref{l: alt sections classical}.

The remaining case $m=3,q=5$ is in Table \ref{t: c7 omegan} (recall that $(m,q)=(3,3)$ is excluded).
 \end{proof}

 \begin{lem}\label{l: c7 omegan 2}
 If the action is listed in Table~$\ref{t: c7 omegan}$, then the action is not binary.
\end{lem}
\begin{proof}
When $t=2$, we have $S=\Omega_9(5)$ and we use {\tt magma} to verify the result, see Lemma~\ref{l: beautifulsetssmall}. 
The remainder of the proof, for $t>2$, proceeds using the result for $t=2$ along with Lemma~\ref{l: last para}.
\end{proof}

\subsection{Case \texorpdfstring{$S=\Omega^+_n(q)$}{S=Omega+(n,q)}}

In this case there are two subfamilies, as listed in Table \ref{c7poss}.
  
\begin{table}\centering
\begin{tabular}{cl}
\toprule[1.5pt]
Group & Details of action \\
\midrule[1.5pt]
%$\Omega_{m^t}^+(2)$ & $m\in\{6,8\}$, $U_1$ symplectic:  $M\geq \underbrace{\Sp_m(2)\circ \cdots \circ \Sp_m(2)}_t$. \\
$\Omega_{2^t}^+(5)$ & $U_1$ symplectic: $M\triangleright \PSp_2(5)^t$ \\
$\Omega_{8^t}^+(3)$ & $U_1$ orthogonal:  $M\triangleright \POmega_8^-(3)^t$ \\
$\Omega_{4^t}^+(q)$ & $q\in\{3,5\}$, $U_1$ orthogonal:  $M\triangleright \POmega_4^-(q)^t$ \\
\bottomrule[1.5pt]
\end{tabular}
\caption{$\C_7$ -- $\Omega^+_n(q)$ -- Cases where a beautiful subset was not found.}\label{t: c7 omegaplusn}
\end{table}

\begin{lem}\label{l: c7 omegaplusn}
 In this case either $\Omega$ contains a beautiful subset or else the action is listed in Table~$\ref{t: c7 omegaplusn}$.
\end{lem} 
 \begin{proof}
 Note that \cite[Table~8.50]{bhr} allows us to exclude $n=8$ in all cases; in particular this means $n\geq 16$. We split into two cases. 
 
First consider line 4 of table \ref{c7poss}.  In this case $W$ is equipped with an alternating form, $M\cap \bar S = ({\rm PGSp}_m(q) \wr \Sym(t))\cap \bar S$, and both $m$ and $qt$ are even. Furthermore in the case where $q$ is even, \cite[Tables~3.5.E and 3.5.I]{kl} (and the explanation on p.69) imply that $m\geq 6$. %Finally, note that \cite[(4.7.18)]{kl} implies that $M_0\cap S = M_0\cap \SOr_n^+(q)$, provided $t>2$ or $m\equiv 0\pmod4$.

Our method is virtually identical to that of Lemma~\ref{l: c7 spn}. For $m>2$ we proceed as before, except that we make a sign adjustment for the elements of $U_1$.

We obtain the same conclusion as in Lemma~\ref{l: c7 spn} -- the existence of a beautiful subset of size $q^{m-2}$ -- and we are done.

For $m=2$ our proof is, again, the same as that of Lemma~\ref{l: c7 spn}. Note that, by \cite[Table~3.5.E]{kl}, $q$ is odd, $q\geq 5$; in addition we may assume that $t\geq 4$. We obtain a beautiful set except when $q=5$, and this case is listed in Table \ref{t: c7 omegaplusn}.

Now consider line 3 of Table \ref{c7poss}. Here $q$ is odd, $W$ is equipped with a symmetric form of type $\varepsilon\in\{+,-\}$ and $M\cap \bar S = (\PO_m^\e(q) \wr \Sym(t)) \cap \bar S$. 
Furthermore \cite[Table~3.5.E]{kl} implies that $m\geq 5+\varepsilon1$.
This time $\mathcal{B}_1=\{u_1,\dots, u_{m/2-1}, v_1,\dots, v_{m/2-1}, x,y\}$ is a hyperbolic basis for $W$ if $\varepsilon=-$, while $\mathcal{B}_1=\{u_1,\dots, u_{m/2}, v_1,\dots, v_{m/2}\}$ is a hyperbolic basis for $W$ if $\varepsilon=+$. 
Taking pure tensors we obtain a basis, $\mathcal{B}$, for $V$, and we let $M$ be the subgroup of $G$ that stabilizes the associated tensor decomposition. Write $k$ for the Witt index of $W$. 

Assume first that $k\ge 3$. Consider the following group:
\[
 T= \left\{ g\circ \underbrace{1\circ\cdots \circ 1}_{t-1} \mid
 \begin{array}{l}
  g \in \Omega_m^\varepsilon(q), \, u_{k}^g=u_{k}, \, v_{k}^g=v_{k} \\
  g \textrm{ stabilizes both } \langle u_1,\dots, u_{k-1}\rangle \textrm{ and } \langle v_1, \dots, v_{k-1} \rangle
 \end{array}\right\}.
\]
Now define $U$ to be the set of elements $g$ such that, for $i=1,\dots, k-1$, there exist $a_i \in \Fq$ such that
\begin{align*}
 u_{k}\otimes u_1^{t-1} &\mapsto u_{k}\otimes u_1^{t-1} + a_1 u_1^t + a_2u_2\otimes u_1^{t-1}+\cdots + a_{k-1}u_{k-1}\otimes u_1^{t-1} \\
 v_i\otimes v_1^{t-1} &\mapsto v_i\otimes v_1^{t-1} - a_i v_{k}\otimes v_1^{t-1},
\end{align*}
and all other members of $\mathcal{B}$ are fixed. One can check directly that $T$ normalizes $U$, that $T$ acts transitively on the set of non-trivial elements of $U$, that $T$ is in $M$, and that $U$ is not in $M$. Then, identifying $\Omega$ with conjugates of $M$, and setting $\Lambda=M^U$, we conclude that $\Lambda$ is a set of size $q^{k-1}$ whose set-wise stabilizer acts 2-transitively.

 Either $\Delta$ is a beautiful subset and we are done, or $\Alt(q^{k-1}-1)$ is a section of $\SOr_m^\e(q)$. By Lemma~\ref{l: alt sections classical}, since $k\ge 3$ and $q$ is odd,  the latter can only hold if $q=3$ and $(m,\e) = (8,-)$, a case listed in Table~\ref{t: c7 omegaplusn}.

We are left with the possibility that $k\leq 2$, in which case $\varepsilon=-$ and $m\in\{4,6\}$. Suppose, first, that $m=6$. We define
\[
 T = \left\{ g\circ \underbrace{1\circ\cdots \circ 1}_{t-1} \mid
 \begin{array}{l}
  g \in \Omega_m(q), \, x^g=x, y^g=y \\
  g \textrm{ stabilizes both } \langle u_1,u_2\rangle \textrm{ and } \langle v_1, v_2 \rangle
 \end{array}\right\}.
\]
Note that we take $x$ to satisfy $\varphi(x,x)=1$. Now define $U$ to be the set of elements $g$ for which there exist $a_1, a_2 \in \Fq$ such that
\begin{align*}
 x\otimes u_1^{t-1} &\mapsto x\otimes u_1^{t-1} + a_1 u_1^t + a_2u_2\otimes u_1^{t-1}, \\
 v_1^{t} &\mapsto v_1^{t} - a_1 x\otimes v_1^{t-1}, \\
 v_2\otimes v_1^{t-1} &\mapsto v_2\otimes v_1^{t-1} - a_2 x\otimes v_1^{t-1}
\end{align*}
and all other members of $\mathcal{B}$ are fixed. As before we obtain a set $\Delta$ of size $q^2$ on which $M^\Delta$ acts 2-transitively. Either $\Delta$ is a beautiful subset and we are done, or $\Alt(q^2-1)$ is a section of $M^\Delta$, in which case $\SOr_6^-(q)$ also has such a section. This is not the case by Lemma~\ref{l: alt sections classical}.  

Finally, suppose that $m=4$ and take 
\[ U_0\rtimes T_0\cong [q]\rtimes C_{q-1}< \Omega_4^-(q)\leq {\rm Isom}(W).\]
Define $T=T_0\circ \underbrace{1\circ\cdots\circ1}_{t-1}$. To define $U$, we first let $W_0=W\otimes x^{t-1}$, and we define $U$ to be the subgroup of $\Omega_m^-(q)$ which fixes, point-wise, every element of $W_0^\perp$, and whose action on $W_0$ is isomorphic to the action of $U_0$ on $W$. We can check that $T$ and $U$ have the same properties as before. Thus, following the same argument we are done unless $\Or^-_4(q)$ contains a section isomorphic to $\Alt(q-1)$. Now  Lemma \ref{l: alt sections classical} shows this can only happen if $q \in \{3,5\}$, as in
Table \ref{t: c7 omegaplusn}. 
\end{proof}

 \begin{lem}\label{l: c7 omegaplusn 2}
 If the action is listed in Table~$\ref{t: c7 omegaplusn}$, then the action is not binary.
\end{lem}

\begin{proof}
We work through Table~\ref{t: c7 omegaplusn} line-by-line. 

First consider Line~1. We apply Lemma~\ref{l: last para} with $t_0=3$. In this case $\bar{S_0}\cong \Sp_8(5)$ and we confirm, using {\tt magma}, that the actions of almost simple groups with socle $\bar{S_0}$ on maximal $\mathcal{C}_7$-subgroups of type $\Sp_2(5)\wr \Sym(3)$ are not binary. This yields the result for $t\geq 4$, as required. (Note that to confirm the property listed at 2(c) in Lemma~\ref{l: last para} we argue as per Lemma~\ref{l: c7 spn 2}.)

%when $t=4$. Here $S=\Omega_{16}^+(5)$ and $M\triangleright \PSp_2(5)^t$. From~\cite[Tables~3.5E and~3.5G]{kl}, we deduce that  $M$ is isomorphic to either
%$$\Alt(5)^4.2^3.\Sym(4)\hbox{  or  }\Sym(5)\mathrm{wr}\Sym(4).$$
%For both of these possible choices of $M$, we used \magma to calculate all actions of $M$ on $(M:H)$ where $H\leq M$ and 
%$|M:H|$ is odd. We find that the only binary actions occur when $|M:H|\in \{1,3\}$.

%Now, if the action of $G$ on $\Omega:=(G:M)$ is binary, then the action of $M$ on each of its orbits on $\Omega$ must be binary by Lemma~\ref{l: point stabilizer}. If $M$ acting on $\Omega$ has an orbit of cardinality $3$, then it follows from~\cite{Sims} that $|M|$ divides $48$, which is clearly a contradiction. Thus, from the previous paragraph we infer that every non-trivial orbit of $M$ on $\Omega$ has even cardinality. Thus $|\Omega|=|G:M|$ is odd, which is clearly not the case. This implies that the action of $G$ on $(G:M)$ is not binary.
%Now if $t>4$ we use the result for $t=4$ combined with Lemma~\ref{l: last para}. (\textbf{OK????})

Next consider Line~2. If $t=2$, then we let $\{u_1, u_2, u_3, v_1, v_2, v_3, x,y\}$ be a hyperbolic basis for $W$. Define
\[
 T_1:=\left\{\begin{pmatrix}
    A & & & \\  & A^{-T} & & \\ & & 1 & \\ & & & 1
   \end{pmatrix} \mid A\in\SL_3(3) \right\},
\]
a subgroup of $\Omega_8^-(3)$, and let $T=T_1\circ 1$, a subgroup of $M$. Now consider the subspace
\[
 X:=\langle x\otimes u_1, u_1\otimes u_2, u_2\otimes u_1, u_3\otimes u_1, x\otimes v_1, v_1\otimes v_1, v_2\otimes v_1, v_3\otimes v_1\rangle,
\]
and observe that $X$ is a non-degenerate subspace of $V$ of type $\Or_8^+$. We define $U$ to be the set of elements in $S$ for which there exist $a,b,c\in\Fq$ such that
\begin{align*}
 x\otimes u_1 &\mapsto x\otimes u_1+ a u_1\otimes u_1 + bu_2\otimes u_1 + cu_3\otimes u_1, \\
 v_1\otimes v_1 &\mapsto v_1\otimes v_1- a x\otimes v_1, \\
 v_2\otimes v_1 &\mapsto v_2\otimes v_1- b x\otimes v_1,\\
 v_3\otimes v_1 &\mapsto v_3\otimes v_1- c x\otimes v_1,
\end{align*}
and all elements of $X^\perp$ are fixed. We see that $U$ is a subgroup of $S$ that is not contained in $M$, that $T$ normalizes $U$ and that $T$ acts transitively on the set of non-identity elements of $U$. We obtain, in the usual way, a set $\Lambda$ of size $|U|=27$ on which $G^\Lambda$ acts 2-transitively. Since $\Alt(26)$ is not a section of $M$, we obtain a beautiful subset and we conclude that the action is not binary by Lemma~\ref{l: beautiful}. For $t>2$, we use the result for $t=2$ along with Lemma~\ref{l: last para}.

Finally consider Line~3 and suppose, first, that $t=2$, $S=\Omega_{16}^+(q)$ and $M\triangleright 
M_0:=\POmega_4^-(q)^2$ with $q\in\{3,5\}$. We confirm the result with {\tt magma}, in the following way. For all groups $M$, we calculate all actions of $M$ on the cosets of a subgroup $H\leq M$ where $(M:H)$ is odd. We find that the only binary actions occur when $H=M$.

Now, observe that $|M:H|$ is even, thus a Sylow $2$-subgroup, $P$, of $M$ is normalized by a $2$-group $Q$ that strictly contains $P$. Let $x\in Q\setminus P$ and consider $H=M\cap M^x$. Our \magma calculation implies that the action of $M$ on $(M:M\cap M^x)$ is not binary, and so Lemma~\ref{l: point stabilizer} implies that the action of $G$ on $(G:M)$ is not binary.

Again the proof for $t>2$ is completed using the result for $t=2$ and Lemma~\ref{l: last para}.
\end{proof} 

\section{Family \texorpdfstring{$\C_8$}{C8}}\label{s: c8}

In this case $M$ is the normalizer of a classical subgroup of $G$ having the same natural module $V$. The possiblilities are listed in Table \ref{c8poss}, taken from \cite[\S4.8]{kl}. Note that in case $L$, the classical subgroup $M$ is centralized by a graph or graph-field automorphism of $S$ (see Proposition \ref{outeraut}), so $M$ may not be almost simple.

\begin{table}[ht!]
\[
\begin{array}{|c|c|c|}
\hline
\hbox{case} &  \hbox{type} & \hbox{conditions} \\
\hline
{\rm L} & \Sp_n(q) &  n\ge 4,\,n \hbox{ even} \\
{\rm L} & \SU_n(q^{1/2}) & n\ge 3, \,q \hbox{ square} \\
{\rm L} & \O^\e_n(q) & n\ge 3,\,q \hbox{ odd} \\
{\rm S} & \O_n^\e(q) & n \ge 4,\,q \hbox{ even}\\
\hline
\end{array}
\]
\caption{Maximal subgroups in family $\C_8$} \label{c8poss}
\end{table}

The main result of this section is the following. The result will be proved in a series of lemmas.

\begin{prop}\label{p: c8}
 Suppose that $G$ is an almost simple group with socle $\bar S = \Cl_n(q)$, and assume that 
\begin{itemize}
\item[{\rm (i)}] $n\ge 3,4,4,7$ in cases $L,U,S,O$ respectively, and 
\item[{\rm (ii)}] $\Cl_n(q)$ is not one of the groups listed in Lemma $\ref{l: beautifulsetssmall}$.
\end{itemize}
Let $M$ be a maximal subgroup of $G$ in the family $\mathcal{C}_8$. Then the action of $G$ on $(G:M)$ is not binary.
\end{prop}

%normalizes the group of isometries for some non-degenerate form $\varphi$ on the space $V$. 

%There are only two subfamilies to consider: either $S=\SL_n(q)$ and $M$ normalizes the set of special isometries of $\varphi$, or else $S=\Sp_n(q)$ with $q$ even and $M$ normalizes the set of isometries of some non-degenerate quadratic form $Q$ that polarizes to an alternating form $\varphi$, for which $S$ itself is the set of isometries.

\subsection{Case \texorpdfstring{$S=\SL_n(q)$}{S=SL(n,q)}}\label{s: c8 sl}

\begin{lem}\label{l: c8 sln}
Suppose that $G$ is almost simple with socle equal to $\PSL_n(q)$. If $M$ is a maximal $\C_8$-subgroup, then the action of $G$ on $(G:M)$ is not binary.
\end{lem}

\begin{proof}
By Lemma \ref{l: beautifulsetssmall}, we may assume that $q>25$ when $n=3$, that $q>9$ when $n=4$, that $q>7$ when $n=5$, that $q>4$ when $n=6$ and that $q>3$ when $n=8$. In what follows we suppose, for a contradiction, that the action of $G$ on $(G:M)$ is binary.

 Suppose first that $n\ge 5$ when $M$ is unitary,  and $n\ge 7$ when $M$ is orthogonal. We refer to Lemma~\ref{l: classical element}; let $x$ be the element listed there and observe that $C_S(x)$ is strictly greater than $C_M(x)$. We conclude that there is a suborbit, $\Delta$, on which the action of $M$ is isomorphic to the action of $M$ on $(M:H)$, where $H=M\cap M^g$ (for some $g\in C_S(x)\setminus C_M(x)$) is a subgroup of $M$ containing the element $x$ (and not containing $M\cap \bar S$). Lemma~\ref{l: point stabilizer} implies that the action of $M$ on $(M:H)$ is binary, and now Lemma~\ref{l: classical element} implies that $M$ must contain a section isomorphic to $\Sym(t)$ where $t$ is as follows:
\begin{enumerate}
 \item if $M$ is unitary and $n$ is even, then $t=q^{n-4}$; Lemma~\ref{l: alt sections classical} implies a contradiction.
 \item if $M$ is unitary and $n$ is odd, then $t=q^{n-3}$;  Lemma~\ref{l: alt sections classical} implies a contradiction.
 \item if $M$ is symplectic or orthogonal of type $\Or^+$ with $n$ even, then $t=q^{(n-2)/2}$; given the excluded cases for small $n$ and $q$, Lemma~\ref{l: alt sections classical} implies a contradiction.
 \item if $M$ is orthogonal and $n$ is odd, then $t=q^{(n-3)/2}$; Lemma~\ref{l: alt sections classical} implies a contradiction.
 \item if $M$ is orthogonal of type $\Or^-$ with $n$ even, then $t=q^{(n-4)/2}$; given the excluded cases for small $n$ and $q$, Lemma~\ref{l: alt sections classical} implies a contradiction.
\end{enumerate}

Next assume that $M$ is unitary and $n \in\{3,4\}$. Here we adopt the same argument using Lemmas~\ref{l: psu3 element} and \ref{l: psu4} in place of Lemma~\ref{l: classical element}. Again Lemmas~\ref{l: alt sections classical} and \ref{l: beautifulsetssmall} yield a contradiction except when $(n,q)=(4,49)$. This final case was dealt with using \magma and the permutation characther method (a.k.a. Lemma~\ref{l: characters}).% \textbf{This case has not been covered.}%This time the extra cases for which \magma computation is needed are: 
%\[
%\begin{array}{l|l|l}
%n & M\triangleright & q \\
%\hline
%3 & \PSU_3(5) & 25 \\
%4 & \PSU_4(4) & 16 \\
%4 & \PSU_4(5) & 25 \\
%\hline
%\end{array}
%\]

It remains to consider the case where $M$ is orthogonal (so $q$ is odd) and $3 \le n\le 6$. 
First assume $n\in\{5,6\}$.  We think of $V$ as a formed space with form, $\varphi$, preserved by $M$. 
Let $W=\langle e_1, e_2, f_1, f_2\rangle$ be a non-degenerate subspace of $V$ of type $\Or_4^+$, and consider the group
\[
 T:=\left\{\begin{pmatrix}
            A & \\ & A^{-t} 
           \end{pmatrix} \mid A \in \SL_2(q) \right\},
\]
inside $M$; here we specify the action of the elements of $T$ on $W$ (with respect to the given basis) and we require that elements of $T$ fix all elements of $W^\perp$. Then $T$ is isomorphic to $\SL_2(q)$. 

Now let $\{x\}$ or $\{x,y\}$ be a basis for $W^\perp$ and consider the group, $U<S$, consisting of elements for which there exist $a_1, a_2\in \Fq$ such that
\[
 x\mapsto x+ a_1e_1 + a_2 e_2,
\]
and all vectors in $W$ are fixed, as is $y$ if $n=6$. Then $U$ is a group of order $q^2$, $U$ does not lie in $M$, $U$ is normalized by $T$, and $T$ acts transitively on the set of non-identity elements of $U$. Thus, in the usual way, we obtain a set $\Lambda\subset \Omega$ of size $q^2$ such that $G^\Lambda$ is $2$-transitive. This set is a beautiful set unless $\Alt(q^2)$ is a section of $\SL_n(q)$; however, this is not the case by Lemma \ref{l: alt sections classical}. Hence Lemma~\ref{l: beautiful} implies the result.

Finally, for $n\in\{3,4\}$ we argue similarly.  Let $W=\langle e_1, f_1\rangle$ be a non-degenerate subspace of $V$ of type $\Or_2^+$, and consider the group
\[
 T:=\left\{\begin{pmatrix}
            a & \\ & a^{-1} 
           \end{pmatrix} \mid a \in \Fq^* \right\},
\]
inside $M$; as before we specify the action of the elements of $T$ on $W$ (with respect to the given basis) and we require that elements of $T$ fix all elements of $W^\perp$. Then $T$ is cyclic of order $q-1$. 
Again let $\{x\}$ or $\{x,y\}$ be a basis for $W^\perp$. Consider the group, $U<S$, consisting of elements for which there exists $a\in \Fq$ such that
\[
 x\mapsto x+ ae_1,
\]
and all vectors in $W$ are fixed, as is $y$ if $n=4$. As before we obtain a beautiful set of size $q$ unless $\Alt(q)$ is a section of $\SL_n(q)$, which is not the case as $q>9$.
\end{proof}

\subsection{Case \texorpdfstring{$S=\Sp_n(q)$}{S=Sp(n,q)}}\label{s: c8 sp}

This case is line 4 of Table \ref{c8poss}.

\begin{lem}\label{l: c8 spn}
Suppose that $G$ is almost simple with socle $\PSp_n(q)$, where $q$ is even, $n\ge 4$ and $(n,q)\ne (4,2)$. Let $M = N_G(\Or_n^\e(q))$ be a maximal $\C_8$-subgroup. Then the action of $G$ on $(G:M)$ is not binary.
 \end{lem}

  \begin{proof}
First observe that for $q=2$, the action of $G$ on $(G:M)$ is 2-transitive, hence is clearly not binary. So assume from now on that $q>2$. %We can also assume that $q>16$ when $n=4$, by Lemma \ref{l: beautifulsetssmall}.

Suppose now that $\e = +$, and let $H = \Or_n^+(2)$ be a subfield subgroup of $M$. There is a subfield subgroup $K = \Sp_n(2)$ of $G$ containing $H$, and $K\cap M = H$. If we let $\Lambda = \{Mk : k \in K\}$, then the action of $K$ on $\Lambda$  is isomorphic to the action of $\Sp_n(2)$ on the cosets of $\Or_n^+(2)$, which is 2-transitive of degree $d: = 2^{n/2-1}(2^{n/2}+1)$. As $\Alt(d)$ is not a section of $\Sp_n(q)$ by Lemma \ref{l: alt sections classical}, it follows that $\Lambda$ is a beautiful subset, giving the conclusion in this case.

Suppose finally that $\e=-$. The argument of the previous paragraph does not work, as $\Or_n^-(2^a)$  does not possess a subfield subgroup $\Or_n^{\pm}(2)$ if $a$ is even, so we use a different argument. 

For $n\ge 8$, let $x \in M$ be the element defined in Lemma~\ref{l: classical element}. This has larger centralizer in $G$ than in $M$, so we can choose $g \in C_G(x)\setminus M$. Then $x \in M\cap M^g$, so Lemmas~\ref{l: classical element} and \ref{l: alt sections classical} imply that the action of $M$ on $(M:M\cap M^g)$ is not binary.

Next suppose that $n=6$, so $M \triangleright \Omega_6^-(q) \cong \PSU_4(q)$. This time we use the element $x = \hbox{diag}(1,1,a,a^{-1})$ of $\PSU_4(q)$, defined in Lemma \ref{l: psu4}, where $a \in \Fq$ has order $q-1$. This acts as $\hbox{diag}(1,1,a,a,a^{-1},a^{-1})$ in $\Omega_6^-(q)$, so there exists $g \in C_G(x)\setminus M$. Now, provided $q>8$, we finish the proof as above, using Lemmas~\ref{l: psu4} and~\ref{l: alt sections classical}. If $q=8$, then we use the same argument with Lemma~\ref{l: psu4} and the fact that $\Alt(8)$ is not a section of $\PSU_4(8)$; if $q=4$, then the result follows from Lemma~\ref{l: beautifulsetssmall}.

Finally, suppose that $n=4$, so $M\triangleright \Omega_4^-(q) \cong \PSL_2(q^2)$. This time we use the element $x$ defined in Lemma \ref{l: psl2var} in exactly the same way as in the previous paragraph to obtain the conclusion. 
\end{proof}

\section{Family \texorpdfstring{$\mathcal{S}$}{S}}\label{s: S}

Let us first define the family $\mathcal{S}$ of subgroups of classical groups. Let $G$ be an almost simple group with socle $\Cl_n(q)$, a classical simple group with associated natural module $V$ of dimension $n$ over $\F_{q^u}$, where $u=2$ if $\Cl_n(q)$ is unitary and $u=1$ otherwise. We say that a subgroup $M$ of $G$ is in the family $\mathcal{S}$ if the following hold:
\begin{itemize}
\item[(a)] $M$ is almost simple, with socle $M_0$,
\item[(b)] the action of the preimage of $M_0$ on $V$ is absolutely irreducible, and cannot be realised over a proper subfield of $\F_{q^u}$,
\item[(c)] $M_0$ is not contained in a member of the family $\C_8$ of subgroups of $G$.
\end{itemize}

In this section we prove the following result. We shall adopt the assumptions on the dimension $n$ made at the beginning of Section \ref{s: assumption}.

\begin{prop}\label{p: S}
 Suppose that $G$ is an almost simple group with socle $\bar S = \Cl_n(q)$, and assume that 
\begin{itemize}
\item[{\rm (i)}] $n\ge 3,4,4,7$ in cases $L,U,S,O$ respectively, and 
\item[{\rm (ii)}] $\Cl_n(q)$ is not one of the groups listed in Lemma $\ref{l: beautifulsetssmall}$.
\end{itemize}
Let $M$ be a maximal subgroup of $G$ in the family $\mathcal{S}$. Then the action of $G$ on $(G:M)$ is not binary.
\end{prop}

%We use the notation established in \S\ref{s: background classical} -- in particular the statement of Proposition~\ref{p: S} is conditioned on the assumptions for $n$ and $q$ made in \S\ref{s: assumption}. Recall that $G$ is almost simple with socle $\overline{S}=S/Z(S)$, a classical group over a field of size $q$, and characteristic $p$, and $M$ is a maximal subgroup of $G$. In this section we assume that $M$ is from the $\mathcal{S}$ class; in particular the group $M$ is almost simple and is the projective image of a group which acts absolutely irreducibly on the natural module for $G$. We write $M_0$ for the socle of $M$.

Note that, in all sections up to this point we have assumed (as stipulated in Section \ref{s: assumption}) that
% if $S=\Sp_4(2^a)$, then $G$ does not contain a graph automorphism; and
 if $\bar S=\POmega_8^+(q)$, then $G$ does not contain a triality automorphism. In the current section we shall drop this assumption. To clarify what we mean by assuming that ``$M$ is in the family $\mathcal{S}$'' in this special case: we are allowing the possibility that $\bar S=\POmega_8^+(q)$, that $G$ contains a triality automorphism, that $M$ is almost simple with socle $M_0$, and that $M_0$ satisfies the defining conditions (a,b,c) given above for the family  $\mathcal{S}$.

We have a number of strategies, which we outline first.

\subsection{Strategies}

\subsubsection{Strategy 1: Subgroups containing centralizers}

This strategy is based on the following definition, the value of which is demonstrated in the ensuing proposition.
It will be used for the case where the socle $M_0$ of $M$ is an alternating group.

\begin{defn}{\rm 
Let $L$ be a simple group and $r$ a positive integer. We say that $L$ satisfies \emph{Property($r$)} if there exists an element $x_r \in L$ of order $r$ such that the following hold for any almost simple group $M$ with socle $L$:
\begin{itemize}
\item[(1)] $Z(C_M(x_r)) = \langle x_r\rangle$;
\item[(2)] for any core-free subgroup $H$ of $M$ such that $C_L(x_r) \le H$, the action of $M$ on $(M:H)$ is not binary;
\item[(3)] If $\la x_r\ra \le N \triangleleft C_M(x_r)$ with $C_M(x_r)/N$ solvable, then $N$ contains $C_L(x_r)$.
\end{itemize}}
 \end{defn}

\begin{lem}\label{l: sgroup} Let $G$ be an almost simple group with socle a classical group $G_0 = \Cl_n(q)$, and suppose $M$ is a maximal subgroup of $G$ in the family ${\mathcal S}$, with socle $L$. Assume that $L$ satisfies Property($r$) for each $r \in\{ 3,5,7,11,13\}$. Then the action of $G$ on $(G:M)$ is not binary.
\end{lem}

\begin{proof}
Let $G,G_0,L$ and $M$ be as in the statement, and let $q=p^a$ with $p$ prime. 
%For this proof we shall also regard $G_0$ as the fixed point group $(\bar G^F)'$, where $\bar G$ is the corresponding classical simple algebraic group over the algebraic closure $\bar \F_q$, and $F$ is a Frobenius endomorphism of $\bar G$.
Consider Property($r$), satisfied by $L$ for $r \in \{3,5,7,11,13\}$. If $r\ne p$, then $x_r$ is a semisimple element of $G$, and 
the structure of $C:=C_G(x_r)$ is described in \cite[Thm. 4.2.2]{gls3}: there is a normal subgroup $C^0$ of $C$ that has a non-trivial central torus $T_r$ containing $x_r$, and such that $C/C^0$ is solvable ($C^0$ is called the connected centralizer in \cite[4.2.2]{gls3}). 

Suppose that $T_r \ne \langle x_r\rangle$, and let $g \in T_r\setminus  \langle x_r\rangle$ and $H = M\cap M^g$. 
Then $g$ centralizes $C^0 \cap M$, a normal subgroup of $C_M(x_r)$ with solvable quotient, and hence 
$H$ contains $C_L(x_r)$ by condition (3) in the definition of Property($r$). If also $H$ contains $L$, then $g \in N_G(L) = M$, and so $g \in Z(C_M(x_r))$, which is a contradiction as $Z(C_M(x_r)) = \langle x_r\rangle$ by assumption (1) in the definition of Property($r$). As $M$ is almost simple with socle $L$,  $H$ is core-free in $M$, and so  Property($r$) implies that $(M,(M:H))$ is not binary, whence also $(G,(G:M))$ is not binary, as required. Hence we may assume from now on that 
\begin{equation}\label{trxr}
\hbox{ if }p\ne r, \hbox{ then }T_r = \langle x_r\rangle.
\end{equation}

\no \textsc{Case $G_0$ symplectic or orthogonal.} Assume  that $G_0 = \PSp_n(q)$, or $\mathrm{P}\O_n^\pm (q)$ with $n$ even, or $\mathrm{P}\Omega_n(q)$ with $n$ odd. If $p\ne 3$, then the torus $T_3$ has order divisible by $\frac{q-\e}{d}$, where $\e = \pm 1$, $q \equiv \e \hbox{ mod }3$ and $d$ is 1,2 or 4 (it can only be 4 in the orthogonal case). By (\ref{trxr}) we have $|T_3|=3$. Hence we see that 
\[
q = 2,4,5,7,11, 13 \hbox{ or } 3^a.
\]
If $q=7, 13$ or $3^a$ with $a>2$, then the torus $T_5$ has order greater than 5, so these possibilities are excluded by (\ref{trxr}). 

Now consider the torus $T_7$. For $q = 4,5,9$ or 11, this has order divisible by $\frac{q^3-\d}{e}$, where $q^3 \equiv \d = \pm1 \hbox{ mod }7$ and $e \in \{1,2,4\}$, and so $|T_7|>7$, contradicting (\ref{trxr}). 

We are left with the cases $q=2$ or 3. For these we consider $T_{11}$, which has order divisible by $2^5+1$ or $\frac{3^5-1}{2}$, respectively, again contrary to (\ref{trxr}). This completes the proof for the case of symplectic and orthogonal groups. 

\vspace{2mm}
\no \textsc{Case $G_0$ linear.} Now assume that $G_0 = \PSL_n(q)$. Suppose $p\ne 3$ and let 
$q \equiv \e \hbox{ mod }3$ with $\e = \pm 1$. Consider Property(3). A preimage of the element $x_3$ in $\SL_n(q)$ acts on $\bar V = V_n(q)\otimes \bar \F_q$ with at most three eigenspaces. Hence the central torus $T_3$ (of order 3 by (\ref{trxr})) in $C_G(x_3)$ has order either $q-\e$ or $\frac{q-\e}{(n,q-1)}$, and so one of the following holds:
\begin{itemize}
\item[(i)] $q = 2,4$ or 5,
\item[(ii)] $\e=1$ and $\frac{q-1}{(n,q-1)} = 3$,
\item[(iii)] $q = 3^a$.
\end{itemize}
Now consider Property(5), assuming $p\ne 5$. As above, $T_5$ has order $q-1$ or $\frac{q-1}{(n,q-1)}$ (if $q \equiv 1 \hbox{ mod }5$), order $\frac{q+1}{c}$ with $c\in \{1,2\}$  (if $q \equiv -1 \hbox{ mod }5$), and order 
$\frac{q^4-1}{(q-1)c}$  (if $q \equiv \pm 2 \hbox{ mod }5$). Since $|T_5|=5$, it follows that one of the following holds:
\begin{itemize}
\item[(iv)] $q = 4,5$ or 9,
\item[(v)] $q=5^{2k}$ and $\frac{q-1}{(n,q-1)} = 3$,
\item[(vi)] $q=3^{4k}$ and $\frac{q-1}{(n,q-1)} = 5$.
\end{itemize}
Now Property(7) rules out all possibilities except for $q=4$, since for all the other cases we must have $|T_7|>7$.
Finally, Property(11) excludes $q=4$, since in this case $T_{11}$ must have order divisible by $\frac{4^5-1}{3}$.

\vspace{2mm}
\no \textsc{Case $G_0$ unitary.} To complete the proof of the theorem, assume that $G_0 = \PSU_n(q)$. This is very similar to the linear case. If $p\ne 3$ then consideration of Property(3) shows that either $q \in \{2,4,7\}$ or 
$q \equiv -1 \hbox{ mod }3$ and $\frac{q+1}{(n,q+1)} = 3$. Then Property(5) implies that one of the following holds:
\begin{itemize}
\item[(i)] $q = 2,4$ or 11,
\item[(ii)] $q=5^{k}$ and $\frac{q+1}{(n,q+1)} = 3$,
\item[(iii)] $q=3^{2k}$ and $\frac{q+1}{(n,q+1)} = 5$.
\end{itemize}
Now Property(7) excludes all possibilities except for $q=2$, and that is ruled out by Property(13). 
\end{proof}

\subsubsection{Strategy 2: Odd degree actions}

Our second strategy has been used already at various stages; however it is convenient to write down an explicit statement. Note that the proof of the next proposition appeals to results of \cite{liesax} and \cite{kantor} which detail, amongst other things, all primitive actions of odd-degree for all of the almost simple groups. Note that both sources omit one family of actions for the groups with socle ${^2\!G_2}(q)$ (here the stabilizer contains a group isomorphic to $(2^2 \times D_{\frac12(q+1)}):3$), however this omission does not affect the proof given below.

\begin{lem}\label{l: sgroup2} Let $G$ be an almost simple group with socle a classical group $G_0 = \Cl_n(q)$, 
not one of the groups listed in Lemma $\ref{l: beautifulsetssmall}$. 
Suppose $M$ is a maximal subgroup of $G$ in the family ${\mathcal S}$, with socle $M_0$. Then one of the following occurs:
\begin{enumerate}
 \item the action of $G$ on $(G:M)$ is not binary;
 \item there is a suborbit on which $M$ has a transitive faithful action of odd-degree that is binary;
 \item $(G_0, M\cap G_0) = (\POmega_7(p), \Sp_6(2))$ or $(\POmega_8^+(p), \Omega_8^+(2))$, where $p$ is an odd prime.
\end{enumerate}
\end{lem}

%Note that in the final listed case,  the group $M$ only exists and is maximal of odd-degree in $G$, for certain values of $q$. In particular $q$ must be odd.

\begin{proof}
 Suppose that the third listed possibility does not occur. Then \cite{liesax} (or, equivalently, \cite{kantor}) implies that $|G:M|$ is even. Thus there exists a non-trivial odd subdegree for the action of $G$ on $(G:M)$. Hence there exists $g \in G\setminus M$ such that $|M:M\cap M^g|$ is odd; moreover, by the maximality of $M$ in $G$, $M_0 \not \le M\cap M^g$, so the action of $M$ on $(M:M\cap M^g)$ is faithful.

 Now suppose, in addition, that the second listed possibility does not occur, so that the action of $M$ on $(M:M\cap M^g)$ is  not binary. Then Lemma~\ref{l: point stabilizer} implies that the action of $G$ on $(G:M)$ is not binary, and so the first listed possibility occurs, as required.
\end{proof}

\subsubsection{Strategy 3: Using distinguished elements}\label{s: strat 2}

The strategy here is used primarily for the situation where $M_0$ is a group of Lie type. It has already been used multiple times for other families, and was briefly discussed at the start of Chapter~\ref{ch: prelim}. We briefly summarise:

\begin{enumerate}
 \item We pick a distinguished element $g\in M$ and show that, if $H$ is any core-free subgroup of $M$ that contains $g$, then the action of $M$ on $(M:H)$ is not binary. This was done in \S\ref{a1lem}.
 \item We give an upper bound for $|C_M(g)|$ and we use results of \S\ref{s: centralizer} to show that, in general, $|C_M(g)|$ is smaller than the smallest centralizer in $G$. We conclude that there exists $x\in C_G(g)\setminus C_M(g)$.
 \item Now $M\cap M^x$ is a core-free subgroup of $M$ that contains $g$. We conclude that the action of $M$ on $(M:M\cap M^x)$ is not binary. Then Lemma~\ref{l: point stabilizer} implies that the action of $G$ on $M$ is not binary.
\end{enumerate}

 We shall also need the well-known lower bounds for dimensions of cross-characteristic representations of groups of Lie type, taken from \cite{landseitz}, with improvements as given in \cite{Tiepsurvey}:

\begin{prop} \label{centbd2}
Let $S$ be a simple group of Lie type over $\F_r$, not isomorphic to one of the following groups:
\[
\begin{array}{l}
\PSL_2(r)\,(r\le 9), \,\PSL_4^\pm(r)\,(r=2,3),\,\O_8^+(2),\,\O_7(3),\\
G_2(r)\,(r\le 4),\,{{^2\!E_6}}(2),\,F_4(2),\,{^2\!F_4}(2)',\,{^2\!B_2}(8).
\end{array}
\]
If $V$ is a non-trivial irreducible module for a quasisimple cover of $S$ over a field of characteristic coprime to $r$, then $\dim V \ge R(S)$, where $R(S)$ is as given in Table~$\ref{lanseibd}$.
\end{prop}

\begin{table}[!ht]
\caption{Lower bounds for cross-characteristic representations}
\label{lanseibd}
\[
\begin{array}{|c||c||c|c|c|c|}
\hline
S & \PSL_d(r)\,(d\ge 3) & \PSU_d(r) & \PSp_{2k}(r)\,(r \hbox{ odd}) & \PSp_{2k}(r)\,(r \hbox{ even}) & \POmega_{2k+y}^\e(r)\,(y\le 1) \\
\hline
R(S) & \frac{r^d-r}{r-1}-1 & \frac{r^d-1}{r+1} & \frac{1}{2}(r^k-1) & \frac{(r^k-1)(r^k-r)}{2(r+1)} & 
\frac{(r^k-1)(r^{k-1}-1)}{r^2-1} \\
\hline
\hline
S & E_8(r) & E_7(r) & E_6^\e(r) & F_4(r) & {^2\!F_4}(r) \\
\hline
R(S) & r^{29}-r^{27} & r^{17}-r^{15} & r^{11}-r^9 & r^8-r^6 & r^5-r^4 \\
\hline
\hline
S & G_2(r) & {^3\!D_4}(r) & {^2\!G_2}(r) & {^2\!B_2}(r) & \\
\hline 
R(S) & r^3-r & r^5-r^3 & r^2-r & (r-1)\sqrt{r/2} & \\
\hline
\end{array}
\]
\end{table}

\subsection{The case where \texorpdfstring{$M_0$}{M0} is alternating}

In this case, we use a combination of Strategies~1 and 2.

\begin{lem}\label{l: alt d large}
  Let $G$ be an almost simple group with socle a classical group $G_0 = \Cl_n(q)$, and suppose $M$ is a maximal subgroup of $G$ in the family ${\mathcal S}$, with socle $M_0\cong \Alt(d)$ for some $d\geq 27$. Then the action of $G$ on $(G:M)$ is not binary.
\end{lem}

\begin{proof}
We use Strategy~1: Lemma~\ref{l: sgroup} yields the result provided we can verify Property($r$) for $r=3,5,7,11$ and $13$.
 
In every case, we take $x_r$ to be the $r$-cycle $(1,2,\dots, r)$. Then $C_M(x_r) \cong (\langle x_r\rangle\times \Sym(d-r))\cap M$ and $C_{M_0}(x_r) \cong \langle x_r\rangle\times \Alt(d-r)$. Parts (1) and (3) in the definition of Property ($r$) follow immediately. Hence to prove the result we must prove part (2):  if $H$ is a core-free subgroup of $M$ containing $C_{M_0}(x_r)$, then the action of $M$ on $(M:H)$ is not binary.

We claim that the group $H$ satisfies 
\[\langle x_r\rangle \times \Alt(d-r) \leq H \leq (\Sym(r)\times \Sym(d-r))\cap M.\]
To see this observe that the first inclusion is true by definition; the second will follow if we can show that $H$ is intransitive in the natural action on $d$ points. Suppose, instead, that $H$ is transitive. If $H$ is imprimitive, then $H$ is isomorphic to a subgroup of $\Sym(e)\wr\Sym(f)$ where $ef=d$. Then, since $\max\{e,f\}\leq \frac{d}{2}$, any alternating section of $H$ is of form $\Alt(s)$ with $s\leq \frac{d}{2}$. But, since $d\geq 27$, $r\leq 13$ and $H$ contains $\Alt(d-r)$, we have a contradiction and we conclude that $H$ is primitive. But $H$ contains a 3-cycle hence, by a classical theorem of Jordan, $H$ contains $\Alt(d)$, a contradiction. Thus the claim follows.

Suppose, first, that $M=\Alt(d)$ and let $K=(\Sym(r)\times \Sym(d-r))\cap M$. We have just seen that $K$ contains $H$. Now Lemma~\ref{l: subgroup} implies that if the action of $K$ on $(K:H)$ is not binary, then the result follows. The kernel of the action of $K$ on $(K:H)$ contains a subgroup isomorphic to $\Alt(d-r)$ and we see that the action of $K$ on $(K:H)$ is isomorphic to the action of $\Sym(r)$ on some subgroup $H_1$ that is the projection of $H$ to $\Sym(r)$. Using \magma we confirm that, for $r\in\{5,7,11,13\}$, all such actions are not binary, provided $H_1$ is core-free. Thus we are left with the case where $H_1=\Alt(r)$ or $\Sym(r)$ and we have
\[
 H=\Alt(r)\times \Alt(d-r) \textrm{ or } (\Alt(r)\times \Alt(d-r)).2.
\]

We repeat this analysis with $M=\Sym(d)$ and $K=\Sym(r)\times \Sym(d-r)$. In this case the kernel of the action of $K$ on $(K:H)$ is isomorphic to either $\Alt(d-r)$ or $\Sym(d-r)$ and we see that the action of $K$ on $(K:H)$ is isomorphic to the action of either $\Sym(r)$ or $\Sym(r)\times C_2$ on some subgroup $H_1$. This time {\tt magma} confirms that in all but one case these actions are not binary, provided $H_1$ does not contain $\Alt(r)$. 

Let us deal first with the one exceptional case in which $H_1$ does not contain $\Alt(r)$ : here $r=5$ and the action of $K$ on $(K:H)$ is isomorphic to the action of $\Sym(5)\times C_2=\langle (1,2,3,4,5), (4,5), (6,7)\rangle$ on $\langle (1,2,3,4,5), (2,5)(3,4)(6,7)\rangle$. In particular, we can take $H$ to contain
\[
 H_0=\langle (1,2,3,4,5)\rangle \times \Alt(\{6,7,8,\dots, d\})
\]
as an index $2$ subgroup and we have $H=\langle H_0, (2,5)(3,4)(6,7)\rangle$. We will show directly that the action of $M$ on $(M:H)$ is not binary. We define
\begin{align*}
 I_1&=J_1=H; \\
 I_2&=J_2=H(2,3,4,5,6,7); \\
 I_3&=H(1,3,4,5,6,7); \\
 J_3&=H(1,6,7,5,4,3).
\end{align*}
In addition we set
\[
 g_{12}=(1), \,\,\, g_{13}=(1,5,4,3,2) \textrm{ and } \, g_{23}=(1,5,3,6,4).
\]
Direct calculation confirms that for $i,j\in\{1,2,3\}$, $I_i^{g_{ij}}=J_i$ and $I_j^{g_{ij}}=J_j$; in other words, $(I_1, I_2, I_3)\stb{2}(J_1, J_2, J_3)$. Now suppose that there exists $g\in \Sym(d)$ such that $I_i^g=J_i$ for $i\in\{1,2,3\}$. We note that the stabilizer in $\Sym(d)$ of an element in $(M:H)$ contains a unique normal cyclic subgroup generated by a 5-cycle. For $I_1$ we can take this 5-cycle to be $(1,2,3,4,5)$, for $I_2$ we can take this 5-cycle to be $(1,3,4,5,6)$. Since $I_1=J_1$ and $I_2=J_2$ we conclude that $g$ must normalize the two groups generated by these 5-cycles. Direct calculation confirms that $g$ is, therefore, a subgroup of $\Sym(\{7,8,9,\dots, d\})$. But now we require that $I_3^g=J_3$; the stabilizer of $I_3$ (resp. $J_3$) contains a normal cyclic subgroup generated by $(2,4,5,6,3)$ (resp. $(1,3,4,6,2)$) and $g$ must conjugate the first subgroup to the second. But $g$ clearly commutes with these subgroups and we have a contradiction. Thus $(I_1, I_2, I_3)\nstb{3}(J_1, J_2, J_3)$ and we are done.

We are left with the situation where
\[
 \Alt(r)\times \Alt(d-r) \leq H \leq \Sym(r)\times \Sym(d-r).
\]
Observe that, for fixed $d$ and $r$, there are five such groups. We now divide the proof in two parts, depending on whether $H$ contains $C_M(x_r)$ or not.

\smallskip

\textsc{Suppose that $H$ contains $C_M(x_r)$}. We define a function from $(M:H)$ to the power set of the conjugacy class $x_r^M$:
\begin{align*}
 \psi: (M:H) &\longrightarrow \mathcal{P}(x_r^M) \\
 Hk  & \mapsto \omega_k:=\{k_0^{-1}x_r k_0 \mid k_0 \in Hk\}.
\end{align*}
Notice that $\omega_{x_r}= x_r^H$. We claim that the image, $\psi(M:H)$ is a partition of $x_r^M$. It is clear that $$\bigcup\limits_{X\in\psi(M:H)} X = x_r^M,$$ thus suppose that $\omega_{k_1}\cap\omega_{k_2}\neq \emptyset$. This implies that $(k_1')^{-1}x_r (k_1') = (k_2')^{-1}x_r (k_2')$ for some $k_1'\in Hk_1,k_2'\in Hk_2$. But this implies that $(k_2')(k_1')^{-1}\in C_M(x_r)<H$ and so $Hk_1'=Hk_2'$ which means that $Hk_1=Hk_2$ and so $\omega_{k_1}=\omega_{k_2}$, as required.

Now we define an action of $M$ on $\psi(M:H)$ via
\[
 \omega_{k_1}^k = \{k^{-1}xk \mid x \in \omega_{k_1}\}.
\]
This action is well-defined and is isomorphic to the action of $M$ on $(M:H)$.

Notice that $x_r^M$ is the set of all $r$-cycles in $\Sym(d)$. We showed above that 
\[\Alt(r)\times \Alt(d-r) \leq H \leq (\Sym(r)\times \Sym(d-r))\cap M.\]
This implies that the partition $\psi(M:H)$ of $x_r^M$ is a refinement of the partition where two $r$-cycles are in the same part if and only if they have the same underlying $r$-set. 

Our method will vary slightly depending on precise properties of this partition. To divide our method into cases we define $H_1$ to be the projection of $H$ onto $\Sym(\{1,\dots, r\})$ and we recall that $H_1$ is either $\Alt(r)$ or $\Sym(r)$.
\smallskip

\textsc{Case 1: $x_r$ is conjugate to $x_r^{-1}$ in $H_1$}.
In this case we define 
\begin{align*}
 I_1 &= J_1 = \Big[(1,2,3,\dots, r)\Big],\\
 I_2 &= J_2 = \Big[\left(1,2,\dots, \frac{r-1}{2}, r+1, r+2, \dots, \frac{3r+1}{2}\right)\Big],\\
 I_3 &= \Big[\left(1,2,\dots, \frac{r-1}{2}, \frac{3r+3}{2}, \frac{3r+5}{2} \dots, 2r+1\right)\Big],\\
 J_3 &= \Big[\left(r, r-1,\dots, \frac{r+3}{2}, \frac{3r+3}{2}, r+2,r+3,\dots, \frac{3r+1}{2}\right)\Big],
\end{align*}
where we use ``$\Big[-\Big]$'' to denote the part of $\psi(M:H)$ containing the listed cycle.

It is easy to see that $I\nstb{3} J$: the cycles representing $I_1, I_2, I_3$ all move the points $1,2,\dots, \frac12(r-1)$, whereas the cycles representing $J_1, J_2, J_3$ have no moved points in common.

To see that $I\stb{2} J$ we must define $g_{13}$ and $g_{23}$ such that $I_i^{g_{ij}}=J_i$ and $I_j^{g_{ij}}=J_j$. To this end, we set:
\begin{align*}
 g_{13}&=\Big(1,r\Big)\Big(2,r-1\Big)\cdots\Big(\frac{r-1}{2}, \frac{r+3}{2}\Big)\Big(2r+1, \frac{3r+1}{2}\Big)\Big(2r, \frac{3r-1}{2}\Big)\cdots \Big(\frac{3r+5}{2}, r+2\Big); \\
 g_{23}&=\Big(1, \frac{3r+1}{2}\Big)\Big(2,\frac{3r-1}{2}\Big)\cdots\Big(\frac{r-1}{2}, r+2\Big)\Big(2r+1,r\Big)\Big(2r,r-1\Big)\cdots\Big(\frac{3r+5}{2},\frac{r+3}{2}\Big). 
\end{align*}
It is easy to check that these even permutations do the job; more specifically, we can see that the representative $r$-cycle listed above in the definition of $I_i$ is mapped to either the representative $r$-cycle listed for $J_i$, or to its inverse.
\smallskip

\textsc{Case 2: $x_r$ is not conjugate to $x_r^{-1}$ in $H_1$}. Since $H_1=\Alt(r)$ or $\Sym(r)$, we conclude that $H_1=\Alt(r)$ with $r\equiv 3\pmod 4$. In particular $r\in\{3,7,11\}$ and $H_1=\Alt(r)$. 

Suppose, first, that $r=3$. In this case $H$ is the centralizer of a 3-cycle in $M$ and the set $\psi(M:H)$ can be identified with set of 3-cycles in $\Alt(d)$. We define
\begin{align*}
 I_1&=J_1=g_{13}=(1,2,3), \\
 I_2&=J_2=g_{23}=(1,2,4), \\
 I_3&=(2,3,4), \\
 J_3&=(3,1,4).
\end{align*}
Finally we define $g_{12}=1$, and now one can check directly that, for all $i,j$ such that $1\leq i<j \leq 3$, we have $I_i^{g_{ij}}=J_i$ and $I_j^{g_{ij}}=J_j$. In particular $I\stb{2} J$.

We wish to show that $I\nstb{3} J$. Suppose that $g\in M$ such that $I^g=J$. Clearly $g$ must stabilize the set $\Delta=\{1,2,3,4\}$. But now, since $g$ must fix both $I_1$ and $I_2$, we obtain that $g|_\Delta=1$. This contradicts the fact that $I_3^g=J_3$ and the result follows.

Suppose, next, that $r\geq 7$. In this case we exhibit the presence of a beautiful subset and the result follows thanks to Lemma~\ref{l: beautiful}. We consider the set
\[
 \Lambda=\left\{ 
 \begin{array}{l}
 \Big[(1,2,4, 8,9,11, 15 \dots, r+8)\Big], \Big[(2,3,5,9, 10, 12, 15, \dots, r+8)\Big], \Big[(3,4,6,10,11,13,15, \dots, r+8)\Big], \\ \Big[(4,5,7,11,12,14,15, \dots, r+8)\Big], \Big[(5,6,1,12,13,8,15, \dots, r+8)\Big], \Big[(6,7,2,13,14,9,15, \dots, r+8)\Big],  \\ \Big[(7,1,3, 14,8,10,15, \dots, r+8)\Big] 
 \end{array}
 \right\}.
\]
Note that the parts of the partition of $x_r^M$ correspond to the conjugacy classes of $r$-cycles for the alternating group of the underlying $r$-set. In particular, for instance, the $r$-cycle $(1^\tau,2^\tau,4^\tau,8^\tau,9^\tau,11^\tau, 15, \dots, r+8)$ is in $\Big[(1,2,4, 8,9,11, 15 \dots, r+8)\Big]$ (where $\tau$ is some permutation of $\{1,2,4,8,9,11\}$) if and only if $\tau$ is even.

We have chosen seven $r$-tuples $(\mu_1, \dots, \mu_6, 15, \dots, r+8)$ that satisfy two properties:
\begin{description}
\item[ (a)] the seven $3$-tuples given by $(\mu_1, \mu_2, \mu_3)$ form the lines of a Fano plane; 
\item[(b)] $\mu_{i+3}=\mu_i+7$ for $i=1,2,3$.
\end{description} It is clear that a group preserving $\Lambda$ must stabilize the set $\{15, \dots, r+8\}$; in addition we claim that if $g\in M_\Lambda$, then $\mu_{i+3}^g = \mu_i^g\pm7$ for $i=1,2,3$. To see this, let $g\in M_\Lambda$ and observe that, for $i=1,\dots, 3$, the number $\mu_i+7$ is the only one that occurs in every listed tuple where $\mu_i$ occurs. Thus $(\mu_i+7)^g$ must occur in every listed tuple where $\mu_i^g$ occurs. But this means that $(\mu_i+7)^g=\mu_i^g\pm7$ as required.

These two properties allow us to conclude that $M^\Lambda$ is isomorphic to a subgroup of $\GL_3(2)$ and so, in particular, does not contain $\Alt(7)$. We wish to show that, in fact, $M^\Lambda=\GL_3(2)$ and the result will then follow.% by Lemma~\ref{l: 2trans gen}. 

Write $\Lambda_1$ for the set of seven $3$-tuples obtained by projecting the listed tuple in each element of $\Lambda$ onto its first three entries; similarly $\Lambda_2$ is the set of seven $3$-tuples obtained by projecting the listed tuple in each element of $\Lambda$ onto entries 4,5,6. Both $\Lambda_1$ and $\Lambda_2$ correspond to Fano planes. Let $\theta_1$ be a permutation of $\{1,\dots, 7\}$ corresponding to an automorphism of the $\Lambda_1$-Fano plane and let $\theta_2$ be a permutation of $\{8,\dots, 14\}$ corresponding to the automorphism of the $\Lambda_2$-Fano plane obtained by increasing each entry in the cycle notation of $\theta_1$ by $7$. Now the permutation $\theta_1\theta_2$ is an element of $\Alt(14)$.

Consider the image of a listed tuple $\lambda$ under $\theta_1\theta_2$. The projection of this image onto its first three entries yields a $3$-tuple which is a permutation of the $3$-tuple given by the first three entries of the listed permutation in an element of $\Lambda$. Likewise the projection of this image onto entries $4,5,6$ yields a $3$-tuple which is a permutation of the $3$-tuple given by entries 4,5 and 6 of the same listed permutation. The two resulting permutations are of the same type and so, since $\theta_1\theta_2$ fixes the points $15, \dots r+8$, we conclude that $\lambda^{\theta_1\theta_2}$ is of the form $(\mu_1^\tau, \mu_2^\tau, \dots, \mu_6^\tau, 15,\dots, r+8)$ where $(\mu_1,\dots, \mu_6, 15,\dots, r+8)$ is one of the listed permutations and $\tau=\theta_1\theta_2$ is even. In particular $\lambda^{\theta_1\theta_2}$ lies in an element $[\gamma]$ of $\Lambda$, where $\gamma$ is one of the listed tuples. Now both $[\lambda]$ and $[\gamma]$ are conjugacy classes in conjugates, $H_\lambda$ and $H_\gamma,$ of $H$. Then $H_\lambda^{\theta_1\theta_2}=H_\gamma$ and, since $\lambda^{\theta_1\theta_2}\in[\gamma]$ we conclude that $[\lambda]^{\theta_1\theta_2}=[\gamma]$. We conclude that $\theta_1\theta_2$ is in $M_\Lambda$ and the result follows.
\smallskip

\textsc{Suppose that $H$ contains $C_{M_0}(x_r)$ but not $C_M(x_r)$}.  In this case $M=\Sym(d)$ and $H$ is one of the following groups:
\[
 \Alt(r)\times \Alt(d-r), \,\, \Sym(r)\times \Alt(d-r) \textrm{ or } (\Alt(r)\times \Alt(d-r)).2.
\]
Observe, first, that if $H<\Alt(d)$, then the action of $\Alt(d)$ on cosets of $H$ is considered above. Since we know that this action is not binary, the result follows by Lemma~\ref{l: subgroup}. 

Thus we assume that $H\not<\Alt(d)$, in which case $H=\Sym(r)\times \Alt(d-r)$. But now the analysis of Case~2 for $r\in \{7,11\}$ works, with the r\^oles of $r$ and $d-r$ interchanged. (Note that in Case~2 our only use of the fact that $r\in \{7,11\}$ was when we needed $r\geq 6$ in order to make our definition of $\Lambda$ work; in the current situation we just observe that $d-r\geq6$ in all cases.)
\end{proof}

For the alternating groups of degree less than 27, we shall use a \magma computation together with the following result. 

\begin{lem}\label{le18}
 Let $G$ be a group with socle $G_0 = \Cl_n(q)$ ($q=p^a$), a classical group, not one of the groups listed in Lemma $\ref{l: beautifulsetssmall}$. Suppose $M = N_G(M_0)$ is a maximal subgroup of $G$ in the family ${\mathcal S}$, where $M_0 = \mathrm{Alt}(r)$ with $r$ odd, $r\le 25$. 
\begin{itemize}
\item[{\rm (i)}] If $r$ is prime, then for any $x \in M_0$ of order $r$, we have $C_G(x) \ne \la x\ra$.
\item[{\rm (ii)}] If $r = 9, 15$ or $21$, then $|G|_3 > |M|_3$; if $r=25$, then $|G|_5 > |M|_5$.
\end{itemize}
\end{lem}

%\noindent \textbf{Note: } In the lemma we have assumed that $G_0$ is not isomorphic to an alternating group to avoid the trivial further examples $\Alt(5) < \PSL_2(9), \Sp_4(2)'$ and $\Alt(7) < \PSL_4(2)$.

\begin{proof} (i) First assume $r\ge 11$. Then $n\ge r-\d_{p,r}$ (see Lemma \ref{l: alt sections classical}). The orders of centralizers of elements of prime order in classical groups are given in Tables B3 - B12 of \cite{bg}, and it is straightforward to read off from these tables that for $n \ge r - \d_{p,r}$, no such centralizer in $G_0 = \Cl_n(q)$ can have order equal to $r \in \{11,13,17, 19,23\}$. 

Now suppose $r=7$. The modular character tables of $\Alt(7)$ and its covering groups are given in \cite{JLPW}. We have $n\ge 3$. If $n \ge 9$, then \cite{bg}  gives a contradiction as above. And if $n=7$ or 8, then the characteristic $p\ge 5$, and again we can use \cite{bg} to rule this out. Hence $n\le 6$.

 If $n=3$, then the only characteristic in which there is an irreducible modular representation is 5, yielding a maximal subgroup $\Alt(7) < \PSU_3(5)$ - but this possibility is excluded by Lemma~\ref{l: beautifulsetssmall}.

If $n=4$, then $p\ge 5$ yields a contradiction using \cite{bg} as above; and $p\le 3$ is again excluded by Lemma~\ref{l: beautifulsetssmall}.

If $n=5$, then the only possible characteristic is $p=7$ with $\Alt(7) < \Omega_5(7)$; but then clearly $C_G(x)\ne \langle x \rangle$.

Finally consider $n=6$. Here $p=2$ is excluded by Lemma~\ref{l: beautifulsetssmall}. If $p=3$, then  there are two possible embeddings, $\Alt(7) < \Omega_6^-(3)$ or $\PSp_6(9)$. The first is out by Lemma~\ref{l: beautifulsetssmall}, and the second is excluded  using \cite{bg}. Finally, \cite{bg} rules out all possibilities with $p\ge 5$.

It remains to consider $r=5$. Here $n\le 6$. If $n=2$ then $G_0 = \PSL_2(q)$ and assuming $p\ne 5$, $C_{G_0}(x)$ has order $q\pm 1/(2,q-1)$. Hence $q = 9$ or 11, excluded by Lemma \ref{l: beautifulsetssmall}. When $n=3$, the embedding is $\Alt(5) < \Omega_3(q) \cong \PSL_2(q)$, already handled. When $n=4$, the embeddings are $\Alt(5) < \Omega_4^-(p) \cong \PSL_2(p^2)$ and $\Alt(5) < \PSp_4(p)$; in the latter case $N_G(S)$ is non-maximal. If $n=5$ then $\Alt(5) < \Omega_5(q) \cong \PSp_4(q)$. Finally, if $n=6$, then $\Alt(5) < \PSp_6(p)$, and \cite{bg} gives a contradiction.

\vspace{4mm}
(ii) Suppose $r=9$. Then $n\ge 8-\d_{p,3}$. If $p=3$, the conclusion is clear, so assume $p\ne 3$. 
If $n = 8$ then $G_0 = \POmega_8^+(q)$, which has order divisible by $3^5$, greater than $|\Sym(9)|_3=3^4$.
And if $n\ge 9$ then $|G_0|$ is divisible by $\frac{1}{d}\prod_1^4 (q^{2i}-1)$ (where $d = (2,q-1)$), hence is also divisible by $3^5$.

For $r= 15$ we have $n \ge 14-\d_{p,3}-\d_{p,5}$ and so $|G_0|$ is divisible by  $\frac{1}{d}\prod_1^6 (q^{2i}-1)$, hence by $3^8 > |\Sym(15)|_3$. A similar argument works for $r=21$ or 25; the only extra point to note is that if $r=25$ and $n=24$ then $G_0 = \POmega_{24}^+(q)$ (rather than $\POmega_{24}^-(q)$), and this has order divisible by $5^7 > |\Sym(25)|_5$. This completes the proof.
\end{proof}

We can now complete the proof of Proposition \ref{p: S} for the case of alternating groups:

\begin{lem}\label{alt d small}
 Let $G$ be an almost simple group with socle a classical group $G_0 = \Cl_n(q)$, under the hypotheses of Propsition $\ref{p: S}$, and suppose $M$ is a maximal subgroup of $G$ in the family ${\mathcal S}$, with socle $M_0\cong \Alt(d)$ for some $5\leq d\leq 26$. Then the action of $G$ on $(G:M)$ is not binary.
\end{lem}

\begin{proof}
Recall our assumptions on $n$ in the hypothesis: namely,
$n\ge 3,4,4,7$ in cases $L,U,S,O$ respectively.
Next we check using \magma the following facts, where $6\le d\le 26$: 
%(\textbf{extended from $d\le 18$ to $d\le 26$ -- please check the extra cases????})
\begin{itemize}
\item[(a)] every non-trivial binary action of $\Alt(d)$ has even degree;
\item[(b)] for $d$ even, every non-trivial binary action of $\Sym(d)$ has even degree;
\item[(c)] every non-trivial binary action of $M_{10}$, $\PGL_2(9)$ and $\PGammaL_2(9)$ has even degree;
\item[(d)] every non-trivial binary action of $\Alt(5)$ and $\Sym(5)$ has degree divisible by $5$;
\item[(e)] for $d$ odd, every non-trivial binary action of $\Sym(d)$ (with core-free point stabilizer)  has degree divisible by a prime $s$, as in the following table:
\[
\begin{array}{c|cccccccccc}
d & 7&9&11&13&15&17&19&21&23&25 \\
\hline
s & 7&3&11&13&3&17&19&3&23&5
\end{array}
\]
\end{itemize}
Given these facts, we can complete the proof as follows. Assume for a contradiction that the action of $G$ on $(G:M)$ is binary. We know by Lemma \ref{l: sgroup2} that in this action, there is a non-trivial suborbit of odd degree on which the action of $M$ is binary.  Hence by Fact (a), $M$ cannot be $\Alt(d)$ for $d>5$. Thus, either $d=5$, $d=6$ or $M$ is $\Sym(d)$. But now Fact (b) rules out the possibility that $M$ is $\Sym(d)$ of even degree, and Fact (c) rules out all the possibilities when $d=6$. Thus, in any case, $d$ is odd and $M=\Sym(d)$ except, possibly, when $d=5$ and $M=\Alt(5)$.

Suppose now that $d$ is a prime (so is $5,7,11,13,17,19$ or 23). Let $x \in M$ have order $d$. Then by Lemma~\ref{le18}(i),  there exists $g \in C_G(x) \setminus M$. Thus there is a suborbit $(M:M\cap M^g)$ of size coprime to $d$. Now the action of $M$ on this suborbit is not binary, by Facts (d) and (e). Hence $G$ is not binary on $(G:M)$ by Lemma~\ref{l: point stabilizer}, a contradiction. 

The remaining cases $d = 9,15,21$ or 25 succumb to a similar argument. For these cases, we let $P$ be a Sylow 3-subgroup of $M$ (a Sylow 5-subgroup in the last case), and observe that by Lemma \ref{le18}(ii), there exists $g \in N_G(P)\setminus M$. Hence the  suborbit $(M:M\cap M^g)$ has size coprime to $3$ (or 5), and the action of $M$ on this is not binary, by Fact (e), giving a contradiction as before. 
\end{proof}

This completes the proof of Proposition \ref{p: S} for the case where the socle $M_0$ is an alternating group.

\subsection{The case where  \texorpdfstring{$M_0$}{M0} is sporadic}

In this case we use Strategy 2 and some earlier computations with {\tt magma}.

\begin{lem}\label{l: sporadic small}
 Let $G$ be an almost simple group with socle a classical group $G_0 = \Cl_n(q)$, and suppose $M$ is a maximal subgroup of $G$ in the family ${\mathcal S}$, with socle $M_0$ a  sporadic simple group.
Then the action of $G$ on $(G:M)$ is not binary.
\end{lem}
\begin{proof}
 The proof follows immediately from Lemmas~\ref{l: sporadic small-odd} and  \ref{l: sgroup2}.
\end{proof} 

\subsection{The case where  \texorpdfstring{$M_0$}{M0} is of Lie type}

In this section we prove Proposition \ref{p: S} for the case where $M_0$ is of Lie type. We will use the strategy outlined in \S\ref{s: strat 2}; in particular we will make use of Propositions~\ref{centbd} and \ref{centbd2}.

To start we use \magma to rule out a number of small possibilities for $M$.

\begin{lem}
Let $M_0$ be one of the simple groups listed in Lemma~$\ref{l: odd degree Lie}$,  and let $G$ be an almost simple group with socle a classical group $G_0 = \Cl_n(q)$, not one of the groups listed in Lemma~$\ref{l: beautifulsetssmall}$. Suppose $M$ is a maximal subgroup of $G$ in the family ${\mathcal S}$, with socle $M_0$. Then the action of $G$ on $(G:M)$ is not binary.
 \end{lem}

\begin{proof}
Lemmas~\ref{l: odd degree Lie} and \ref{l: sgroup2} imply the result unless 
$(G_0, M_0) = (\POmega_7(p), \Sp_6(2))$ or $(\POmega_8^+(p), \Omega_8^+(2))$, where $p$ is an odd prime.

If $G_0=\POmega_7(p)$ and $M_0=\Sp_6(2)$, then we let $g\in M_0$  be the element of order $3$ defined in Lemma~\ref{l: classical element}. In that lemma it is proved that if $H$ is any subgroup of $M$ that contains $g$, then the action of $M$ on  $(M:H)$ is not binary. Suppose that there exists $x\in C_{G}(g)\setminus M$. Then the action of $M$ on $(M:M\cap M^x)$ is not binary, and Lemma~\ref{l: point stabilizer} yields the result. It remains to show, therefore, that $C_G(g)$ is strictly larger that $C_M(g)$. Direct calculation implies that $|C_M(g)|=108$ and now Lemma~\ref{l: cent rank} implies the result for $q>7$. For $q\leq 7$, the result is confirmed with {\tt magma} or by direct calculation.

If $G_0=\POmega_8^+(p)$ and $M_0=\Omega_8^+(2)$, then we let $g$  be the element of order $7$ defined in Lemma~\ref{l: classical element}. We proceed as before but must confirm that there exists $x\in C_{G}(g)\setminus M$. Using \cite{bhr} we see that $G_0\cap M=M_0$, and using \cite{atlas} we see that $C_{M_0}(g)=\langle g\rangle$. Thus it is sufficient to prove that $C_{G_0}(g)\neq \langle g \rangle$. When $p=7$, this is immediate; when $p\neq 7$, one can confirm this using, for instance, \cite{bg}.
\end{proof}
 
Let us next deal with some troublesome groups that are just a little too big to be easily handled with {\tt magma}.

\begin{lem}\label{l: left over Lie}
Let $M_0$ be one of the following  groups 
\[
 {^3\!D_4}(4), {^3\!D_4}(5), {^2\!E_6(3)}.
 \]
Let $G$ be an almost simple group with socle a classical group $G_0 = \Cl_n(q)\,(q=p^a)$, and suppose $M$ is a maximal subgroup of $G$ in the family ${\mathcal S}$, with socle $M_0$. Then the action of $G$ on $(G:M)$ is not binary.
\end{lem}

\begin{proof}
\textbf{(1) } Suppose, first that $M_0\cong {^3\!D_4}(5)$.

\smallskip

\textbf{Claim:} There exists an element, $g$, of order $24$ in $M_0$ such that if $M$ is almost simple with socle $M_0$ and $H$ is any core-free subgroup of $M$ containing $g$, then the action of $M$ on $(M:H)$ is not binary.

\textbf{Proof of claim:} Assume $(M,(M:H))$ is binary for some such $H$. 
We use the fact that if $\Alt(t)$ is a section of ${^3\!D_4}(5)$, then $t\leq 7$ by Lemma~\ref{altsec}. We also use the existence of the following subgroup chain:
\[
 \SL_3(5) < G_2(5) < {^3\!D_4}(5).
\]
The group $\SL_3(5)$ contains a maximal parabolic subgroup with unipotent radical $E$, an elementary abelian group of order $25$. In addition $\SL_3(5)$ contains an element $g$ of order $24$ that normalizes, and acts fixed-point-freely upon, $E$.  %To see this, take $E$ to be the product of root groups $X_{10}. X_{11}$ and take $g$ to be an element of order $24$ in a Levi factor of the parabolic subgroup which has $E$ as unipotent radical. Assume that $H$ is a subgroup of $M$ containing $g$, for which the action of $M$ on $(M:H)$ is not binary. We will show that $H$ contains $M_0$. 
Since $E\rtimes \langle g\rangle$ is a Frobenius group, we conclude that either $H$ contains $E$ or else the set of cosets $H.E$ forms a beautiful subset of size $25$, and Lemma~\ref{l: beautiful} yields a contradiction. Thus $H$ contains $E$.

We can now repeat the same argument with the ``opposite'' unipotent radical, $E_1$. The same element $g$ acts fixed-point-freely on $E_1$ and we can now also assume that $H$ contains $E_1$. Since $\langle E,E_1\rangle= \SL_3(5)$, we conclude that $H$ contains $\SL_3(5)$.

The group $\SL_3(5)$ is a subgroup of $K:=\SL_3(5):2$, a maximal subgroup of $G_2(5)$ of maximal rank. We revisit our argument for the action of $G_2(q)$ on cosets of $K$ -- see Table~\ref{3d4maxrk} and the proof Propositions~\ref{maxrk}. As in that proof, we conclude either that we have a beautiful subset of order $25$ (a contradiction) or else $H$ contains $G_2(5)$. Thus the latter holds.

The group $L:=G_2(5)$ is a maximal subgroup of ${^3\!D_4}(5)$ in family (V) of Theorem~\ref{MAXSUB}. Now we revisit our argument for the action of ${^3\!D_4}(q)$ on cosets of $L$ -- see Case (6) of the proof of Proposition~\ref{type5}. Once again we obtain a beautiful subset of order $25$. We conclude, as required, that $H$ contains $M_0$. The claim is proved.

\medskip 
We now show that the claim implies the conclusion of the lemma. Suppose that there exists $x\in C_{G_0}(g)\setminus C_{M}(g)$. Then the claim implies that the action of $M$ on $(M:M\cap M^x)$ is not binary and Lemma~\ref{l: point stabilizer} yields the result.

Thus to complete the proof for this case we must check that the element $x$ exists. Note that $g$ is a regular semisimple element of ${^3\!D_4}(5)$ (which we can see by computing the action of $G$ on the $8$-dimensional module for ${^3\!D_4}(5)$). Hence $C_{M_0}(g)$ is a maximal torus of $M_0$, the sizes of which are listed in \cite{K}. We conclude that $|C_{M_0}(g)|\leq 756$, and so $|C_M(g)|\leq 2268$. On the other hand, if $p\ne 5$ we have $n \ge 5^5-5^3$ by Lemma \ref{centbd2}, and if $p=5$, then either $n=8,q=5^3$ or $n\ge 24$ by \cite[5.4.8]{kl}; hence $|C_{G_0}(g)|>2268$ by 
Lemma~\ref{l: cent rank}.

\medskip

\noindent \textbf{(2) } Suppose next that $M_0\cong {^3\!D_4}(4)$. The proof here is very similar to the previous case.

\smallskip

\textbf{Claim: } There exists an element, $g\in M_0$ of order 15, such that if $M$ is almost simple with socle $M_0$ and $H$ is any core-free subgroup of $M$ containing $g$, then the action of $M$ on $(M:H)$ is not binary.

\textbf{Proof of claim:}  We use the fact that ${^3\!D_4}(4)$ contains a maximal group isomorphic to $J \cong \PGL_3(4)$ (see \cite{K}). 
The group $J$ contains an element $g$ of order $15$ that normalizes and acts fixed-point-freely upon an elementary abelian group $E$ of order $16$. Assume that $H$ is a subgroup of $M$ containing $g$, for which the action of $M$ on $(M:H)$ is  binary. We will show that $H$ contains $M_0$. 

Arguing exactly as in the previous case, we see that either there is a beautiful subset of size $16$ or $H$ contains the group $\PGL_3(4)$; hence we conclude the latter. Now we revisit our argument for the action of ${^3\!D_4}(q)$ on cosets of $J$ -- see Case (6) of the proof of Proposition~\ref{type5}. Once again we obtain a beautiful subset of order $16$. We conclude, as required, that $H$ contains $M_0$. The claim is proved.

\smallskip 
We now show that the claim implies the conclusion of the lemma.
 This proceeds as before, relying on the existence of $x\in C_{G_0}(g)\setminus C_M(g)$. As before, $g$ is regular semisimple and, using \cite{K}, we conclude that $|C_M(g)|\leq 945$. Now, as before, we find that $|C_{G_0}(g)|>945$ and we are done.

\medskip 

\noindent \textbf{(3) } Suppose now that $M_0\cong{^2\!E_6}(3)$. From \cite[Table 5.1]{LSS}, we see that $M_0$ has a maximal subgroup $\SU_3(27).3$. Write $L$ for the simple subgroup $\SU_3(27)$ of this.  Let $x\in L$ be an element of order $26$ which, written with respect to a hyperbolic basis $\{e_1, f_1, x\}$ of the corresponding unitary 3-space, is diagonal with entries $(t,t^{-1},1)$. We proceed in a series of steps.

\smallskip 

\textbf{Claim 1: } If $X$ is an almost simple group with socle $\SU_3(27)$, and $Y<X$ is a core-free subgroup such that $|X:Y|$ is not divisible by $3^2$, then the action of $X$ on $(X:Y)$ is not binary.

\textbf{Proof of Claim 1:} This is a \magma computation.

\medskip 

\textbf{Claim 2: } Let $x\in L$ be the element defined above, and suppose $x \in H<M$ with $H$ a core-free subgroup of $M$. Then  the action of $M$ on $(M:H)$ is not binary.

\textbf{Proof of Claim 2:} Assume that the action of $M$ on $(M:H)$ is binary. We shall repeatedly use Lemma~\ref{altsec} which asserts that $M$ does not contain a section isomorphic to $\mathrm{Alt}(d)$ for any $d\geq 12$. 

Let $L_1 \cong \SOr_3(27)\cong \PGL_2(27)$ be a subfield subgroup of $L$ containing $x$. Then
there are two Sylow 3-subgroups $U_1,U_2$ of $L_1$ such that $\langle x\rangle$ normalizes and acts transitively on the set of non-trivial elements of each of them. This implies (using Lemma~\ref{l: 2t}) that either $U_1$ is in $H$, or else $HU_1$ is a subset of $(M:H)$ on which the set-wise stabilizer acts 2-transitively. But in the latter case we obtain a beautiful subset, a contradiction to Lemma~\ref{l: beautiful}. Thus $H$ contains $U_1$ and, similarly, $U_2$. Thus $H\geq \langle U_1, U_2, x\rangle=L_1$.

Now let $g$ be a diagonal element of $L=\SU_3(27)$ of order $27^2-1$ and such that $g^{28}=x$. Then $\langle x\rangle$ acts transitively on the set of non-trivial elements of $U_1^g$ and $U_2^g$. We deduce, as in the previous paragraph, that $H$ contains $U_1^g$ and $U_2^g$, and we conclude that $H\geq L = \SU_3(27)$. From the information on the maximal subgroups of $M$ given by Theorem \ref{MAXSUB}, it follows that $L \le H \le N_M(L) \le L.6$. 

Write $\Omega = (M:H)$. Now $N_M(L)$ acts transitively on ${\rm fix}_{\Omega}(L)$, so the number of fixed points $f =  
|{\rm fix}_{\Omega}(L)|$ divides 6. Also $|\Omega| = f+\sum_i|\Delta_i|$, where the $\Delta_i$ are the faithful $H$-orbits on $\Omega$. Since $|\Omega|$ is divisible by $3^2$, it follows that there is an $H$-orbit $\Delta_i$ of size not divisible 
by $3^2$. Hence the action of $H$ on $\Delta_i$ is not binary, by Claim 1. However the action of $M$ on $(M:H)$ is binary by assumption, so this is a contradiction to Lemma~\ref{l: point stabilizer}. This establishes Claim 2.

\smallskip 
We now show that  Claim 2 implies the conclusion of the lemma. 
 Just as before, we need to show that there exists $x\in C_G(x)\setminus M$. 

We start by computing the order of $C_M(x)$. The subgroup $L=\mathrm{SU}_3(27)$ arises as the fixed point group of a Frobenius endomorphism of the algebraic group $E_6$ acting on a subsystem $A_2^3$ (see \cite{LSS}).  By \cite[Prop. 2.1]{LS96}, the Lie algebra $L(E_6)$ restricts to $A_2^3$ as the sum of $L(A_2^3)$ together with the tensor product $V_1 \otimes V_2\otimes V_2$ and its dual, where each $V_i$ is a natural 3-dimensional module for one of the $A_2$ factors. The element $x$ acts on this tensor product as $(t,t^{-1},1) \otimes (t^3,t^{-3},1) \otimes (t^9,t^{-9},1)$, so has fixed point space 0.  Hence the fixed point space of $x$ on $L(E_6)$ has dimension 6, and it follows that $x$ is regular semisimple in $M_0$, with centralizer $C_{M_0}(x)$ of order $27^2-1$. 

On the other hand, for the classical group $G_0 = \Cl_n(q)$, $q=p^e$, we have $n \ge 27$ if $p=3$ by \cite{Lueb}, and $n\ge 3^{11}-3^9$ if $p\ne 3$ by Proposition~\ref{centbd2}. Hence $|C_G(x)|$ is far greater than $|C_M(x)|$ by Proposition~\ref{centbd}. This final contradiction completes the proof.
\end{proof}

In light of the preceding two results, to prove Proposition~\ref{p: S} when $M_0$ of Lie type, we may exclude the following list of possibilities for $M_0$:
\begin{equation}\label{listm0}
\begin{array}{l}
\PSL_2(r)\,(r\le 31),\,\PSL_3(r)\,(r\le 5),\PSL_4(2), \\
\PSU_3(r)\,(r\le 5),\,\PSU_4(r)\,(r\le 5),\,\PSU_5(2),\,\PSU_6(2),\\
\PSp_4(r)\,(r\le 7),\,\PSp_6(r)\,(r\le 3),\,\PSp_8(2), \\
\Omega_7(r)\,(r\le 9),\,\POmega_8^-(r)\,(r\le 9),\,\Omega_{10}^\pm(2),\,\POmega_{10}^-(3),\\
G_2(r)\,(r\le 5),\,{^3\!D_4}(r)\,(r\le 5),\,F_4(r)\,(r\le 3),\,{{^2\!E_6}}(r)\,(r\le 3),\,{^2\!F_4}(2)'.
\end{array}
\end{equation}

For convenience, we restate Proposition \ref{p: S} for this case:

\begin{lem}\label{sgrouplie}
 Let $G$ be an almost simple group with socle a classical group $G_0 = \Cl_n(q)$, and suppose $M$ is a maximal subgroup of $G$ in the family ${\mathcal S}$, with socle $M_0$ a group of Lie type not in the list~\eqref{listm0}. Then the action of $G$ on $(G:M)$ is not binary.
\end{lem}

Note that the list (\ref{listm0}) includes all the exceptions in the conclusions of Lemmas \ref{l: psl element 2}, \ref{l: classical element}, \ref{l: classical element 2}, \ref{l: psu3 element}, \ref{l: psu4} and \ref{l: chev exce element}.

Now let $x\in M_0$ to be the element defined in these propositions, as detailed in Table \ref{xdef}.

\begin{table}[!ht]
\caption{Definition of the element $x \in M_0$}
\label{xdef}
\[
\begin{array}{|c|c|}
\hline
M_0 & x \hbox{ as in Lemma} \\
\hline
\PSL_d(r) & \ref{l: psl element 2} \\
\PSU_d(r) & \ref{l: psu3 element} \;(d=3) \\
               & \ref{l: psu4} \;(d= 4) \\
               & \ref{l: classical element} \;(d\ge 5) \\
\PSp_d(r) & \ref{l: classical element} \\
\POmega^\e_d(r) & \ref{l: classical element}\; (r\hbox{ even}) \\
                            & \ref{l: classical element 2}\; (r\hbox{ odd}) \\
\hbox{exceptional} &  \ref{l: chev exce element}, \,\ref{l: exc rank 2 element} \\
\hline
\end{array}
\]
\end{table}

We shall need upper bounds for the order of the centralizer of $x$ in $M$, given in the next result.

\begin{lem}\label{uppercent}
Let $M$ be almost simple, with socle $M_0$ of Lie type over $\F_r$, and let $x \in M_0$ be as defined in Table~$\ref{xdef}$.
\begin{itemize}
\item[(i)] For $M_0$ classical, we have $|C_M(x)| < N$, where $N$ is as in Table~$\ref{upcent}$.
\item[(ii)] For $M_0$ exceptional, upper bounds for $|C_M(x)|$ are given in Lemmas~$\ref{l: chev exce element}$  and ~$\ref{l: exc rank 2 element}$.
\end{itemize}
\end{lem}

\begin{proof} 
The argument for (i) is very similar for all types of classical groups. For $M_0 = \PSL_d(r)$ with $r>2$, the element $x$ has preimage in $\SL_d(r)$ of the form ${\rm diag}(1,A,a^{-1})$ and the centralizer in $\GL_d(r)$ of this element has order $(r-1)^2|\GL_1(r^{d-2})|$. Moreover $x$ is not centralized by any non-trivial field automorphism, and can only be centralized by a graph automorphism if $d=3$. It follows that for $d\ne 3$, $|C_M(x)| \le C_{\PGL_d(r)}(x) \le (r-1)(r^{d-2}-1)$, and for $d=3$ there is an extra factor of 2.

For $M_0=\PSU_d(r)$ with $d=2j+\d$ (where $\d = 1$ or 2), the centralizer in $\GU_d(r)$ of a preimage of $x$ is contained in $\GU_{2}(r)\times \GU_\d(r) \times \GL_1(r^{2(j-1)})$, and there are no outer automorphisms centralizing $x$ unless $d=3$ or 4 (in which case there is a graph automorphism); this leads to the bound in Table \ref{upcent}.  Similarly for $M_0=\PSp_{2k}(r)$, the centralizer is $\Sp_2(r) \times \GL_1(r^{k-1})$.

Next consider $M_0= \POmega^-_{2k}(r)$. If $r$ is odd, then $x$ has preimage ${\rm diag}(I_4,\zeta,\zeta^{-1},A,A^{-T})$, and the centralizer of this in  $\Or^-_{2k}(r)$ is $\Or_4^-(r) \times \GL_1(r) \times \GL_1(r^{k-3})$, leading to the required bound. Similar considerations give the result for $ \POmega^+_{2k}(r)$ and $ \POmega_{2k+1}(r)$.
\end{proof}

\begin{table}[!ht]
\caption{Upper bounds for $|C_M(x)|$}
\label{upcent}
\[
\begin{array}{|c|c|}
\hline
M_0 & N \\
\hline
\PSL_d(r) & r^{d-1}(1+\d_{d,3}) \\
\PSU_d(r)\,(d\ge 3) & 2r^{d+3},\,d\ge 6 \hbox{ even} \\
                                & 2r^{d+1},\,d\ge 5 \hbox{ odd} \\
                              & 2r^5,\,d=4 \\
                                & 2r^2,\,d=3 \\
\PSp_{2k}(r)\,(k\ge 2) & r^{k+2} \\
\POmega^+_{2k}(r)\,(k\ge 4) & 2r^k \\
\POmega^-_{2k}(r)\,(k\ge 4) & 2r^{k+4} \\
\POmega_{2k+1}(r)\,(k\ge 3) & 2r^{k+2} \\
\hline
\end{array}
\]
\end{table}

For the proof of Lemma~\ref{sgrouplie}, we now adopt the following assumptions:
\begin{itemize}
\item[(1)] $G$ is an almost simple group with socle $G_0 = \Cl_n(q)$ ($q=p^a$), a classical group.
\item[(2)] $M$ is a maximal subgroup of $G$ in the family ${\mathcal S}$ with socle $M_0$, a group of Lie type over $\F_r$; moreover, $M_0$ is not one of the groups in the list (\ref{listm0}).
\item[(3)] The action of $G$ on $(G:M)$ is binary.
\end{itemize}
We aim for a contradiction. This will prove Lemma~\ref{sgrouplie}.

\begin{lem}\label{cgxcmx} 
Adopt the above assumptions $(1)$--$(3)$, and let $x \in M_0$ be as defined in Table~$\ref{xdef}$. Then $C_G(x) = C_M(x)$.
\end{lem}

\begin{proof} Suppose there exists $g \in C_G(x)\setminus M$. Then $x \in M\cap M^g$. If $M_0\le M\cap M^g$, then $g \in N_G(M_0) = M$ which is not the case; hence $M\cap M^g$ is a core-free subgroup of $M$ containing $x$. It now follows from the results listed in the last column of Table \ref{xdef} that the action of $M$ on $(M:M\cap M^g)$ is not binary. But then $(G,(G:M))$ is also not binary by Lemma \ref{l: point stabilizer}, a contradiction.
\end{proof} 

Recall that the classical group $G_0 = \Cl_n(q)$ is defined over the field $\F_q$ of characteristic $p$, while the subgroup $M_0$ is a group of Lie type over $\F_r$. 
At this point we divide the analysis into two cases: the cross-characteristic case (where $p\nmid r$) and the defining characteristic case (where $p\mid r$). 

\begin{lem}\label{cross} Under the assumptions $(1)$--$(3)$, the cross-characteristic case $p\nmid r$ does not occur.
\end{lem}

\begin{proof}
Suppose $p\nmid r$. Then the following hold:
\begin{itemize}
\item[(a)] $n \ge R(M_0)$, as given in Table \ref{lanseibd}.
\item[(b)] By Lemma \ref{centbd}, we have 
\[
|C_G(x)| > \frac{q^{\lceil (n-1)/2 \rceil}}{4}\left(\frac{q-1}{2qe(\log_q(2n) +4)}\right)^{1/2}.
\]
\item[(c)] By Lemma \ref{uppercent} we also have $|C_M(x)|\le N$, where $N$ is as defined in Tables \ref{upcent} for $M_0$ classical, and in Table \ref{tab: m lower} and Lemma \ref{l: exc rank 2 element} for $M_0$ exceptional.
\end{itemize}
By Lemma \ref{cgxcmx}, it follows that $N$ is greater than the right hand side of the inequality in (b). However, when combined with the inequality  $n \ge R(M_0)$, it is routine to check that this leads to a contradiction.
\end{proof}

\medskip
It remains to handle the  defining characteristic case, where $p\mid r$. Recall that $G_0 = \Cl_n(q)$ ($q=p^a$), and  $M_0$, the socle of the maximal subgroup $M$, is a group of Lie type over $\F_r$. Let $V$ be the natural  $n$-dimensional module for $G_0$. According to \cite[Cor. 6]{Seitz}  together with \cite{Sch}, there are two possibilities:
\begin{itemize}
\item[(A)] $\F_r \supset \F_q$: in this case $r=q^k$ with $k\ge 2$, and the embedding $M_0 < G_0$ is as in \cite[Table 1B]{Sch}, and takes the form $\Cl_d(q^k) < \Cl_{d^k}(q)$;
\item[(B)] $\F_r \subseteq \F_q$: in this case the representation of $M_0$ on $V$ corresponds to a restricted representation of the overlying simple algebraic group over $\bar \F_p$.
\end{itemize}

First we deal with Case~(B).

\begin{lem}\label{defineb} Under the assumptions $(1)$--$(3)$, the  defining characteristic case $(B)$ above does not occur.
\end{lem}

\begin{proof}
Assume we are in case (B), so that $M_0 = M_0(r) < G_0 = \Cl_n(q)$ with $\F_r\subseteq \F_q$. We shall use the lower bounds for the dimensions of restricted representations of simple algebraic groups given by \cite{Lueb}. For an integer $d$, define $\e_{p,d}$ to be 1 if 
$p\mid d$, and 0 otherwise. 

\vspace{2mm}
\noindent $\boldsymbol{(\alpha)}$ Assume first that $M_0 = \PSL^\e_d(r)$. If the restriction of $V$ to $M_0$ is self-dual, then $G_0$ is symplectic or orthogonal; otherwise, $G_0 = \PSL_n^\e(q)$. Hence using \cite{Lueb}, we see that one of the following holds:
\begin{itemize}
\item[(i)] $G_0 = \PSp_n(q)$ or $\POmega_n(q)$: 
\[
\begin{array}{l}
d=2,\, n\ge 4, \hbox{ or }\\
d\ge 3, n\ge d^2-1-\e_{p,d};
\end{array}
\]
\item[(ii)] $G_0 = \PSL^\e_n(q)$: 
\[
\begin{array}{l}
d=3,\, n\ge 6, \hbox{ or }\\
d=4, n\ge 10, \hbox{ or }\\
d=5, n=10 \hbox{ or } n\ge 15, \hbox{ or }\\
d\ge 6, n\ge \frac{1}{2}d(d-1).
\end{array}
\]
\end{itemize}
Consider the element $x \in M_0$ defined in Table~\ref{xdef}. By Lemma~\ref{uppercent}, we have 
$|C_M(x)| < N$, where $N$ is as in Table \ref{upcent}; and by Proposition \ref{centbd} we have 
$|C_G(g)| > f(n,q)$, where $f(n,q)$ is as in Table \ref{fnq}.  Hence Lemma \ref{cgxcmx} gives
\[
N > f(n,q).
\]
Combined with the lower bounds on $n$ in (i) and (ii) above, this gives a contradiction except for the following cases:
\begin{itemize}
\item[(1)] $d=3, p=3,n=7$: here $M_0 = \PSL_3^\e(q) < G_0 = \Omega_7(q)$, $q = 3^a$,
\item[(2)] $d=5,n=10, \e=-$: here $M_0 = \PSU_5(q) < G_0 = \PSU_{10}(q)$.
\end{itemize}

In case (1) the element $x\in M_0$ has preimage ${\rm diag}(1,t,t^{-1})$ in $\SL_3^\e(q)$, where $t \in \F_q$ has order $q-1$. Moreover, the natural 7-dimensional module $V$ for $G_0$ is a constituent of the adjoint module for $M_0$, and hence the action of $x$ on $V$ is ${\rm diag}(1,t,t,t^{-1},t^{-1},t^2,t^{-2})$. Clearly then, $C_M(x) \ne C_G(x)$, contradicting Lemma \ref{cgxcmx}.

In  case (2) above, $x$ has preimage ${\rm diag}(1,1,1,t,t^{-1})$ in $\SU_5(q)$, and $V$ is the exterior square of the 5-dimensional natural module for $M_0$. Hence $x$ acts on $V$ as ${\rm diag}(1^4,t,t,t,t^{-1},t^{-1},t^{-1})$, and again $C_M(x) \ne C_G(x)$. This completes the proof for the case where $M_0 = \PSL^\e_d(r)$.

\vspace{2mm}

\noindent $\boldsymbol{(\beta)}$ Next assume that $M_0 = \PSp_{2k}(r)$ with $k\ge 2$. In this case \cite{Lueb} gives
\[
\begin{array}{l}
k=2,\, n=10 \hbox{ or }n\ge 12, \hbox{ or }\\
k=3,\, n=8\,(p=2) \hbox{ or }n=14-\d_{p,3} \hbox{ or } n\ge 21, \hbox{ or }\\
k\ge 4, n=2^k\,(p=2) \hbox{ or }n\ge k(2k-1)-1-\e_{p,k}.
\end{array}
\]
Again we have $|C_M(x)| < N$ with $N$ is as in Table \ref{upcent}, and also $|C_G(g)| > f(n,q)$ with $f(n,q)$ as in  Proposition \ref{centbd}(ii).  Hence Lemma \ref{cgxcmx} gives $N > f(n,q)$, and combined with the above lower bounds for $n$, this yields a contradiction apart from the following cases:
\begin{itemize}
\item[(1)] $k=2,n=10$,
\item[(2)] $k=3, \,n=8\,(p=2)$ or $n=14-\d_{p,3}$,
\item[(3)] $k=4, n=16\,(p=2)$.
\end{itemize}

In case (1), the element $x\in M_0$ has preimage ${\rm diag}(1,1,t,t^{-1})$ in $\Sp_4(q)$, where $t \in \F_q$ has order $q-1$; also $p\ne 2$ and $V$ is the symmetric square of the natural 4-dimensional module for $M_0$. Hence 
$x$ acts on $V$ as ${\rm diag}(1^4,t,t,t,t^{-1},t^{-1},t^{-1})$, and $C_M(x) \ne C_G(x)$, contradicting Lemma \ref{cgxcmx}.

In case (2), $x$ has preimage ${\rm diag}(1,1,A,A^{-T})$ in $\Sp_6(q)$, where $A \in \GL_2(q)$ has order $q^2-1$.
If $n=14-\d_{p,3}$, then $V$ is a constituent of the exterior square of the natural 6-dimensional module, and so the action of $x$ has diagonal blocks ${\rm diag}(A,A,A^{-T},A^{-T},\ldots)$. But this implies that $C_G(x)$ contains a subgroup $\SL_2(q^2)$, so again $C_M(x) \ne C_G(x)$. And if $n=8$ with $p=2$, then $V$ is a spin module for $M_0 = \Sp_6(q)$. 
Observe that $x$ lies in a subgroup $\Sp_2(q) \times \Sp_4(q)$, and on a spin module this acts as $\Sp_2(q) \otimes \Sp_4(q)$. Hence $x$ acts as $I_2 \otimes {\rm diag}(A,A^{-T})$, and so as before, $C_G(x)$ contains a subgroup $\SL_2(q^2)$.

A similar argument applies in case (3), where $V$ is a spin module for $M_0 = \Sp_8(q)$. Here $x ={\rm diag}(1,1,A,A^{-T}) \in M_0$, where $A\in \mathrm{GL}_3(q)$ has order $q^3-1$. This lies in a subgroup $\Sp_2(q) \times \Sp_6(q)$, hence as above, acts on a spin module as $(I_4,A,,A,A^{-T},A^{-T})$. Then $C_G(x)$ contains a subgroup $\SL_2(q^3)$, so again 
 $C_M(x) \ne C_G(x)$.

\vspace{2mm}

\noindent $\boldsymbol{(\gamma)}$ Now consider the case where $M_0$ is an orthogonal group $\POmega^\e_d(r)$ with $d\ge 7$. In this case the dimension bounds are:
\[
\begin{array}{l}
d=7,\, n=8 \hbox{ or }n\ge 21, \hbox{ or }\\
d=8,\, n=8 \hbox{ or } n\ge 26, \hbox{ or }\\
d\ge 9,\, n=2^{\lfloor \frac{d-1}{2} \rfloor} \hbox{ or }n\ge \frac{1}{2}d(d-1)-2.
\end{array}
\]
As above, the inequality  $N > f(n,q)$ now gives a contradiction apart from the following cases:
\begin{itemize}
\item[(1)] $d = 7$, $n=8$,
\item[(2)] $d = 8$, $n=8$,
\item[(3)] $d = 9$ or 10, $n=16$.
\end{itemize}

Consider (1). Here $M_0 = \Omega_7(q) < G_0 = \POmega_8^+(q)$, and the element $x = (1^3,\zeta,\zeta^{-1},a,a^{-1}) \in M_0$ with $a,\zeta \in \F_q^\times$ of order $q-1$ and $\zeta \ne a^{\pm 1}$. 
Write $x = (I_3,X) \in \Omega_3(q) \times \Omega_4^+(q) < M_0$, where $X = (\zeta^{-1},a,a^{-1})$.
In $\SL_2(q) \otimes \SL_2(q) \cong \Omega_4^+(q)$, $X$ takes the form $(\a,\a^{-1}) \otimes (\b,\b^{-1})$, where $\a\b = a$, $\a\b^{-1} = \zeta$. Hence in the spin representation on $V$, $x$ acts as $(I_2\otimes (\a,\a^{-1}),  
I_2\otimes (\b,\b^{-1}))$. It follows that $C_G(x)$ contains a subgroup $(\SL_2(q))^2$, so $C_M(x) \ne C_G(x)$, giving the usual contradiction.

Now consider (2). In this case $M_0 = \POmega_8^-(q^{1/2}) < G_0 = \POmega_8^+(q)$, where $M_0$ is the image of a subfield subgroup under a triality automorphism of $G_0$. We can write the element $x$ as $(I_2,X)$ with $X \in \Omega_6^+(q^{1/2})$ given by Lemmas \ref{l: classical element} (for $p=2$) and \ref{l: classical element 2} (for $p$ odd). Arguing in similar fashion to the previous paragraph, we see that $C_G(x)$ contains a subgroup $\SL_2(q^2)$ ($p=2$) or $(\SL_2(q))^2$ ($p$ odd). Hence again $C_M(x) \ne C_G(x)$.

Finally, consider case (3). For $d=9$ we have $M_0 = \Omega_9(q) < G_0 = \POmega_{16}^+(q)$ with $q$ odd, and $V$ is a spin module for $M_0$. We have $x = (1^3,\zeta,\zeta^{-1},A,A^{-T}) \in M_0$, where $\zeta \in \F_q^\times$ has order $q-1$ and $A \in \GL_2(q)$ has order $q^2-1$. Then $x \in \Omega_3(q) \times \Omega_6^+(q) < M_0$, and this subgroup acts on the spin module $V$ as $(V_2\otimes V_4) \oplus (V_2\otimes V_4^*)$, where the action of $x$ on $V_4$ is computed via the isomorphism $\Omega_6^+(q) \cong \SL_4(q)/\langle -I\rangle$. It follows that $C_G(x)$ contains 
$(\SL_2(q^2))^2$, hence $C_M(x) \ne C_G(x)$. A very similar computation applies in the case where $d=10$. 

\vspace{2mm}
\noindent $\boldsymbol{(\delta)}$ To complete the proof of the lemma, it remains to handle the case where $M_0$ is an exceptional group of Lie type over $\F_r$ with $\F_r\subseteq \F_q$. 
From the bounds for the dimensions of restricted representations of groups of Lie type given in \cite{Lueb}, it follows that either $n=R_0$, or $n \ge R$, where $R_0,R$ are as in the following table:
\[
\begin{array}{c|ccccccccc}
M_0 & E_8(r) & E_7(r) & E_6^\e(r) & F_4(r) & G_2(r) & {^2\!F_4}(r) & {^2\!G_2}(r) & {^2\!B_2}(r) & {^3\!D_4}(r) \\
\hline
R_0 & 248 & 56 & 27&26-\d_{p,3} & 7-\d_{p,2} & 26 & 7 & 4 & 8\\
\hline
R & 248 & 132 & 77 & 52 & 14 & 246 & 26 & 16 & 26 \\
\end{array}
\]
We have $|C_M(x)| < N$ with $N$ as given in Lemmas \ref{l: chev exce element}, \ref{l: exc rank 2 element}, 
and also $|C_G(g)| > f(n,q)$ with $f(n,q)$ as in  Proposition \ref{centbd}(ii).  Hence as usual, Lemma \ref{cgxcmx} gives 
$N > f(n,q)$, and combined with the above bounds for $n$, this yields a contradiction apart from the cases where $n = R_0$ and $$M_0 =  E_6^\e(q), \,\,F_4(q),\,\, G_2(q),\,\,{^2\!B_2}(q)\, \hbox{or}\,\, {^3\!D_4}(q^{1/3})$$ (note that $r=q$ in all but the last case, by the maximality of $M$).

If $M_0 =  E_6^\e(q)$, then $n=27$, the module $V$ has highest weight $\l_1$ in the usual notation, and is not self-dual, so that $G_0 = \PSL_{27}^\e(q)$. Now the inequality $N > f(n,q)$, with $f(n,q)$ as in Table \ref{fnq}, gives a contradiction.

Next consider $M_0 = F_4(q)$ with $n = 26-\d_{p,3}$. By the exclusions in the list (\ref{listm0}), we have $q \ge 4$. 
The element $x \in M_0$ is as defined in Lemma \ref{l: chev exce element}: it lies in a subsystem subgroup $A_3 \cong \SL_4(q)$ and takes the form ${\rm diag}(1,a,A)$ where $A$ is a $2\times 2$ matrix of order $q^2-1$ and determinant $a^{-1}$. The restriction $V\downarrow A_3$ is given in \cite[Table 8.7]{LS96}: in terms of highest weight modules, the composition factors are $V(\l_1)^2/V(\l_3)^2/V(\l_2)/0^{4-\d_{p,3}}$. Here $W:=V(\l_1)$ is the natural 4-dimensional $A_3$-module, $V(\l_3) = W^*$, $V(\l_2) = \wedge^2W$ and $0$ is the trivial module. Hence we compute that $\dim C_V(x) = 8-\d_{p,3}$, and so $C_G(x)$ has a subgroup $\Omega_7(q)$. However $C_M(x)$ has no such subgroup by Lemma \ref{l: chev exce element}, a contradiction. 

Now let $M_0 = G_2(q)$ with $n = 7-\d_{p,2}$ (and $q \ge 7$ by the exclusions in (\ref{listm0})). Here $G_0$ is $\Sp_6(q)$ if $q$ is even , and $\Omega_7(q)$ if $q$ is odd. We have $x = {\rm diag}(1,a,a^{-1})$ in a subsystem subgroup $\SL_3(q)$, where $a \in \F_q^{\times}$ has order $q-1$. Hence $x$ acts on $V$ as ${\rm diag}(a,a,a^{-1},a^{-1},1^{3-\d_{p,2}})$, and it follows that $C_G(x)$ has a subgroup $\Sp_2(q)\times \SL_2(q)$ or $\Omega_3(q)\times \SL_2(q)$, whereas $C_M(x)$ has no such subgroup. 

If $M_0 =  {^2\!B_2}(q)$ with $n=4$, then $G_0 = \mathrm{Sp}_4(q)$ and we have $|C_M(x)| = q-1$ by Lemma \ref{l: exc rank 2 element}, whereas $|C_G(x)| = (q-1)^2$. Finally, if $M_0 = {^3\!D_4}(q^{1/3})$ with $n=8$, then $G_0 = \POmega^+_8(q)$ and $x$ acts on $V$ as  ${\rm diag}(a,a,a^{-1},a^{-1},1^4)$; hence $C_G(x)$ contains $\Omega_4^+(q)$, so once again $C_M(x) \ne C_G(x)$. This completes the proof. 
\end{proof}

\begin{table}[ht!]
\caption{Embeddings $M_0 =  \Cl_d(q^k) < G_0 = \Cl_{d^k}(q)$ ($k$ prime)}
\label{schaff}
\[
\begin{array}{|c|c|c|}
\hline
M_0 & G_0 & \hbox{conditions} \\
\hline
\PSL_d^\e(q^k) & \PSL_{d^k}^\e(q) & d\ge 3,\,(k,\e) \ne (2,-) \\
\PSL_d(q^2) & \PSU_{d^2}(q) & d\ge 3 \\
\PSp_d(q^k) & \PSp_{d^k}(q) & kq \hbox{ odd} \\
\PSp_d(q^k) & \POmega^+_{d^k}(q) & k \hbox{ odd}, q \hbox{ even}\\
\PSp_d(q^2) & \POmega^\e_{d^2}(q) & d\ge 4, \,\e = (-)^{d/2} \\
\POmega^{\pm}_d(q^k) & \POmega^\e_{d^k}(q) & d\ge 6 \hbox{ even}, q\hbox{ odd} \\
\POmega_d(q^k) & \POmega_{d^k}(q) & dq \hbox{ odd}, \,d\ge 3 \\
\hline
\end{array}
\]
\end{table}

\begin{lem}\label{definea} Under the assumptions $(1)$--$(3)$, the  defining characteristic case $(A)$ above does not occur.
\end{lem}

\begin{proof}
Assume we are in case (A), so that $M_0 =  \Cl_d(q^k) < G_0 = \Cl_{d^k}(q)$ with $k\ge 2$. Specifically, the embeddings $M_0 < G_0$ are as given by \cite[Table 1B]{Sch}, and are as in Table \ref{schaff}, with $k$ prime. With one exception, 
the natural module for $G_0$ is of the form $V = W \otimes W^{(q)} \otimes \cdots \otimes W^{(q^{k-1})}$, where $W$ is the natural $d$-dimensional module for $M_0$; the exception is for the embedding $\PSL_d(q^2) \le \PSU_{d^2}(q)$ in the second row of the table, where $V = W \otimes W^{*(q)}$.

The argument is very similar for all entries in the table: we have $x \in M_0$, a semisimple element with centralizer as described in the proof of Lemma \ref{uppercent}. Then $C_G(x)$ contains a maximal torus of $G$, whereas we argue that $C_M(x)$ cannot contain such a torus: in most cases this is obvious, as $G$ has much larger rank than $M$, but nevertheless we shall give a sketch for each case below.

Consider the first row of the table, $\PSL_d^\e(q^k) \le \PSL_{d^k}^\e(q)$ with $d\ge 3,\,(k,\e) \ne (2,-)$. 
For $\e=+$ the element $x \in M_0$ has preimage of the form ${\rm diag}(1,a,A)$ where $A \in \GL_{d-2}(q^k)$ has order $q^{k(d-2)}-1$. This acts on $V$ as $(1,a,A) \otimes (1,a,A)^{(q)} \otimes \cdots \otimes (1,a,A)^{(q^{k-1})}$, and hence we see that $C_G(x)$ has order divisible by $(q^{k(d-2)}-1)^2$ if $d\ge 4$, and by $(q^k-1)^3/(q-1)$ if $d=3$. Hence $C_G(x) \ne C_M(x)$. Now consider $\e=-$. Here the semisimple element $x \in M_0$ has at least two eigenvalues 1 if $d\ge 4$, and is ${\rm diag}(1,a,a^{-1})$ if $d=3$, where $a$ generates $\F_{q^k}^\times$. From the action of $x$ on $V$, we see that $C_G(x)$ contains $\SU_{2^k}(q)$  if $d\ge 4$, and contains $(\GL_1(q^{2k}))^2$ if $d=3$. Hence again 
$C_G(x) \ne C_M(x)$. 

The argument for the second row of Table \ref{schaff} is entirely similar: here $C_G(x)$ has order divisible by $(q^{2(d-2)}-1)^2$ if $d\ge 4$, and by $(q^2-1)^3$ if $d=3$.

Next consider $M_0 = \PSp_d(q^k)$, with embedding as in rows 3-5 of Table \ref{schaff}. Suppose first that $d=2$, so that $k$ is odd and $G_0$ is $\PSp_{2^k}(q)$ or $\POmega^+_{2^k}(q)$, according as $q$ is odd or even, respectively. Also $q^k>31$, by the exclusions of (\ref{listm0}). The element $x = {\rm diag}(a,a^{-1}) \in M_0$ has centralizer in $M$ of order dividing $q^k-1$. If $k=3$, then $q>3$ and the action of $x$ on $V$ has eigenvalues $\mu^{\pm 1},\mu^{\pm q}, \mu^{\pm q^2}, \l,\l^{-1}$, where $\mu = a^{q^2+q-1}$, $\l = a^{q^2+q+1}$; hence $C_G(x)$ has order divisible by $(q^3-1)(q-1)/(2,q-1)$, and so $C_G(x)\ne C_M(x)$. And if $k \ge 5$, then we reach the same contradiction as $|C_G(x)|$ is divisible by $(q^k-1)^{(2^{k-1}-1)/k}$. 

This deals with $d=2$, so suppose now that $M_0 = \PSp_d(q^k)$ with $d\ge 4$. Here $x = (I_2,A,A^{-T})$, where $A \in \mathrm{GL}_{\frac{d}{2}-1}(q^k)$ has order $q^{k(\frac{d}{2}-1)}-1$. If $k\ge 3$, then the fixed point space of $x$ on $V$ has dimension $2^k$ and $C_G(x)$ contains $\Cl_{2^k}(q)$. Hence $k=2$ and $x$ acts on $V$ as 
$(I_2,A,A^{-T}) \otimes (I_2,A,A^{-T})^{(q)}$. This has centralizer in $G$ containing $\Omega_4^\e(q) \times 
\mathrm{SL}_2(q^{k(\frac{d}{2}-1)})$, so once again $C_G(x)\ne C_M(x)$.

Finally, consider $M_0 = \POmega_d^\e(q^k)$, as in the last two rows of Table \ref{schaff}. 

Because of exceptional isomorphisms of low-dimensional orthogonal groups, we need to consider separately the cases $d=3,5$ and 6. If $d=3$ then $x \in M_0 = \Omega_3(q^k) \cong \PSL_2(q^k)$ has the form ${\rm diag}(1,a,a^{-1})$, where $a \in \F_{q^k}$ has order $(q^k-1)/2$, and we argue in the usual way that $|C_G(x)|$ is divisible by $(q^k-1)^2/2$, so $C_G(x)\ne C_M(x)$.

If $d=5$ then $x \in M_0 = \Omega_5(q^k) \cong \PSp_4(q^k)$; in $\PSp_4(q^k)$, $x$ takes the form $(I_2,a,a^{-1})$, so in $\Omega_5(q^k)$, we have $x = (1,aI_2,a^{-1}I_2)$. Now we argue that $C_G(x)$ contains $(\SL_2(q^k))^2$ if $k\ge 3$, and contains $\SL_2(q^2)\times \SU_2(q)$ if $k=2$. In both cases, $C_G(x)\ne C_M(x)$.

Next, if $d=6$ then  $x \in M_0 = \POmega_6^\e (q^k) \cong \PSL_4^\e(q^k)$. For $\e=+$, $x$ takes the form $(1,a,A)$ in $\PSL_4(q)$, hence $x = (a,a^{-1}A,A^{-T}) \in \POmega_6^+ (q^k)$; and for $e=-$, $x$ is $(I_2,a,a^{-1} \in \PSU_4(q)$, hence $x = (I_2,aI_2,a^{-1}I_2) \in \POmega_6^- (q^k)$. Now argue in the usual way that $C_G(x) \ne C_M(x)$.

Finally, if $d \ge 7$ then $x = (I_{2+y},\zeta,\zeta^{-1},A,A^{-T}) \in M_0 = \POmega_d^\e(q^k)$ (where $y \in \{0,1,2\}$), and so $C_G(x)$ contains $\Omega_{(2+y)^k}(q) \times \SL_{2^{k-1}}(q^{k(d-4-y)/2})$, and once again we have the contradiction $C_G(x)\ne C_M(x)$. This completes the proof. 
\end{proof}

\medskip

This completes the proof of Lemma~\ref{sgrouplie}, and hence also the proof of Proposition \ref{p: S}.

\section{Exceptional automorphisms}\label{s: missing}

As explained in Section \ref{s: assumption}, in our proof of Theorem \ref{t: classical}, we have so far been assuming that our almost simple group $G$ contains no graph automorphisms when $G$ has socle $\Sp_4(2^a)$, and no triality automorphisms when $G$ has socle $\POmega_8^+(q)$ (except when $M$ is in the Aschbacher class $\mathcal{S}$). In this final section, we complete the proof of Theorem \ref{t: classical} by handling these cases.

Thus we assume in this section that $G$ is an almost simple group such that one of the following holds:
\begin{enumerate}
 \item the socle of $G$ is isomorphic to $\Sp_4(2^a)$ with $a>1$, and $G$ contains a graph automorphism;
 \item the socle of $G$ is isomorphic to $\POmega_8^+(q)$ and $G$ contains a triality automorphism.
\end{enumerate}

Note that we omit the case $\Sp_4(2)$, as the theorem is already proved for groups with alternating socle in \cite{gs_binary}. 

We slightly adjust terminology for this final section: we use $S$ to denote the socle of $G$.

\begin{lem}\label{l: sp4}
 Let $S=\Sp_4(q)$ where $q=2^{a}$ with $a>1$, and suppose that $S\leq G \leq \Aut(S)$. Let $M$ be a core-free maximal subgroup of $G$. Then the action of $G$ on $(G:M)$ is not binary.
\end{lem}

%Note that, although we have already proved Lemma~\ref{l: sp4} in the case where $G$ does not contain a graph automorphism, we will not use this directly in the proof.

\begin{proof}
For $q\in \{4,8,16\}$, we refer to Lemma~\ref{l: beautifulsetssmall}. Assume that $q\geq 32$. We refer to \cite[Table 8.14]{bhr} for the maximal subgroups of $G$. One checks that with three exceptions, all of them contain an element $g$ as defined in Lemma~\ref{l: sp4 element}; hence these can be excluded. The exceptions are
\[
M\cap S=\Sp_4(q_0), \;\;(q^2+1):4 \hbox{ or } (q+1)^2:D_8.
\]
 
 If $M\cap S=\Sp_4(q_0)$, then $q=q_0^r$, where $r$ is prime, and the argument of Lemma~\ref{l: c5 spn} gives the conclusion.

In the remaining two cases, $M$ is a torus normalizer, and we use arguments similar to those in \S\ref{tori}. Suppose that $M\cap S=(q^2+1):4$. Then $N=M\cap S$ is a Frobenius group with $T=q^2+1$, the Frobenius kernel. Let $g \in M\cap S$ be of order $4$; again we check that there exists $c\in C_G(g)\setminus N_G(T)$. Then the action of $N$ on $(N:N\cap N^x)$ is a Frobenius action and, since $N\cap N^x=N\cap M \cap M^x$, Lemma~\ref{l: frobenius cyclic kernel} implies that the action of $M$ on $(M:M\cap M^x)$ is not binary; hence the action of $G$ on $(G:M)$ is not binary by Lemma~\ref{l: point stabilizer}.

Suppose finally that $M\cap S=(q+1)^2:D_8$. We apply Lemma~\ref{torlem} with $A\cong D \cong \SL_2(q)$, and $T_0\cong T_1 \cong q+1$. The listed conditions are all easy to verify; in particular, item (vii) of the lemma is verified using \cite{dgs_binary}, which asserts that the action of a group with socle $\SL_2(q)$ on the set of cosets of the normalizer of a non-split torus is not binary.
\end{proof}

\begin{lem}\label{l: omega8}
 Let $S=\POmega_8^+(q)$, suppose that $S\leq G \leq \Aut(S)$, and suppose that $G$ contains an element in the coset of a  triality automorphism of $S$. If $M$ is a  maximal core-free subgroup of $G$, then the action of $G$ on $\Omega=(G:M)$ is not binary.
\end{lem}

\begin{proof}
This is covered by Lemma \ref{l: beautifulsetssmall} when $q\le 4$, so assume that $q\ge 5$.

We refer to \cite{kleidman_orthogonal} for a list of the maximal subgroups of $G$. Following \cite{kleidman_orthogonal} we set $d=\gcd(2,q-1)$ and, when giving the isomorphism type of a subgroup, we prefix a circumflex symbol to indicate that we are giving the structure of the group in $\Omega_8^+(q)$, rather than its projective image in $\POmega_8^+(q)$. 
%Note that Proposition~\ref{p: S} applies in the current case, so we can exclude the possibility that $M\cap \mathrm{P \Gamma O}_8^+(q)$ is a member of the $\mathcal{S}$-class.

Suppose, first, that $M\cap S$ is a maximal subgroup of $S$ in the $\mathcal{C}_1$ family. Then \cite[Proposition~4.6]{gs_binary} implies that $\Omega$ contains a beautiful subset, and the result follows immediately. Suppose, next, that $M$ is a novelty maximal subgroup of $G$ such that $M\cap S$ is a proper subgroup of a maximal $\mathcal{C}_1$ subgroup of $S$. There are four possibilities for $M$, and we list them in Table~\ref{o8} together with an integer $r$.  The integer $r$ indicates the presence of a subgroup $A=\SL_r(q)$ in $M$, together with a subgroup of $S$ that is isomorphic to a central quotient of $\SL_{r+1}(q)$ satisfying the conditions of Lemma~\ref{aff}. The lemma then implies that there is a subset $\Delta$ of $\Omega$ of size $q^r$ on which $G^\Delta$ acts 2-transitively. Now Lemma \ref{l: alt sections classical} implies that $\POmega_8^+(q)$ does not contain a section isomorphic to $\Alt(q^r)$, and the result follows.

\begin{table}[!ht]
\[
\begin{array}{|c|c|}
\hline
M\cap S  & r \\
\hline
 \hat\!\,[q^{11}:\left[\frac{q-1}{d}\right]:\frac1d\GL_2(q).d^2 & 2 \\
 G_2(q) & 3 \\
 \hat\!\left(\frac{q-1}{d} \times \frac1d\GL_3(q)\right).[2d] & 3 \\
 \hat\!\left(\frac{q+1}{d} \times \frac1d\GU_3(q)\right).[2d] & 2 \\
\hline
\end{array}
\]
\caption{Novelty $\mathcal{C}_1$-subgroups in $\POmega_8^+(q)$}
\label{o8}
\end{table}

Next suppose that $M\cap S$ is a maximal subgroup of $S$ in the $\mathcal{C}_2$ family, stabilizing a decomposition of $V = V_8(q)$ as a direct sum of $m$-spaces. Lemma~\ref{l: c2 omegaeven} implies that either
\begin{itemize}
\item  there is a beautiful subset (and we are done), or 
\item the parameter $m=1$ (and \cite{kleidman_orthogonal} implies that we can exclude this case, since such groups are not maximal given our assumption that $G$ contains a triality automorphism), or
\item  $M\cap S$ is of type $\Or_2^-(q) \wr \Sym(4)$.
\end{itemize} In this last case $M$ is the normalizer of a torus, with $M\cap S\cong \hat\!\left(\frac{q+1}{d}\right)^4.d^3.2^3.\Sym(4)$. Now, as in the previous lemma, we appeal to Lemma~\ref{torlem} with $A=A_1(q)$ and $D=A_1(q)^3$, and we conclude that the action of $G$ on $(G:M)$ is not binary.

Suppose now that $M$ is a novelty maximal subgroup of $G$ such that $M\cap S$ is a proper subgroup of a maximal subgroup of $S$ in the $\mathcal{C}_2$ class. Then \cite{kleidman_orthogonal} implies that there are two possibilities: $M\cap S\cong [2^9].\PSL_3(2)$ or $(D_{2(q^2+1)/d})^2.[2^2]$. For the first we use \magma to check that any transitive action of $M$ of degree $k$ with $k\not\equiv 0\pmod 4$ and with $M/K$ non-solvable, where $K$ is the kernel of the action, is not binary. Now if $q\equiv 1,7\pmod 8$, then $|G:M|$ is even and hence $M$ must have a non-trivial suborbit of odd degree; so we are done in this case (note that we can ignore actions where $M/K$ is solvable by Lemma~\ref{l: higman}). If $q\equiv 3,5\pmod 8$, then $|G:M|\equiv 3\pmod 4$. Therefore $|G:M|-1\equiv2\pmod 4$ and hence $M$ cannot have all suborbits of cardinality a multiple of $4$. Again, we are done. For the second group $(D_{2(q^2+1)/d})^2.[2^2]$, we use Lemma~\ref{torlem} with $A=A_1(q^2)$ and $D=A_1(q^2)$ and we conclude that the action of $G$ on $(G:M)$ is not binary.

Assume next that $M\cap S$ is a maximal subgroup of $S$ in the $\mathcal{C}_5$ class. Then Lemma~\ref{l: c5 omegan} implies that either there is a beautiful subset (and we are done), or 
$F^*(M\cap S)=\Omega_8^-(q_0)$ with $q_0\in \{2,3\}$ and $S=\Omega_8^+(q_0^2)$ (but this case does not occur when $G$ contains a triality automorphism), or $F^*(M\cap S)=\Omega_8^+(2)$ and $q=2^r$ with $r$ odd. For this last case $M$ equals either $L$ or $L\times r$ where $L$ is almost simple with socle $\Omega_8^+(2)$; now we obtain the result arguing in exactly the same way as in Lemma~\ref{l: c5 omegan 2}. 
%(\textbf{For this we need Magma check for the group $L\times 3 = (\Omega_8^+(2).X) \times 3$, where $X = 3$ or $\Sym(3)$ -- can you do this????})

The final case to consider is that in which $M\cap S$ is a subgroup in the family $\mathcal{S}$ of subgroups of $S$. However this case has already been dealt with in Proposition \ref{p: S}, thanks to our relaxation of the triality assumption at the beginning of Section \ref{s: S}.
\end{proof}

This completes our consideration of the exceptional automorphisms. The proof of Theorem \ref{t: classical} is now complete.

\end{document}